\documentclass[11pt]{amsart}

\usepackage{amssymb,amsmath,amsthm,mathtools}
\usepackage[margin=1.15in]{geometry}
\usepackage{booktabs,graphicx,array}
\usepackage{enumitem}
\usepackage{cite}
\usepackage[hidelinks]{hyperref}
\newcommand{\topdim}{18}
\newcommand{\topdimword}{eighteen}
\hypersetup{pdftitle={Convex counterexamples to the Schiffer and Pompeiu conjectures in dimensions two to \topdimword},pdfauthor={Jizhou Guo}}

\newtheorem{theorem}{Theorem}[section]
\newtheorem{proposition}[theorem]{Proposition}
\newtheorem{lemma}[theorem]{Lemma}
\newtheorem{corollary}[theorem]{Corollary}
\newtheorem{conjecture}{Conjecture}
\theoremstyle{definition}

\theoremstyle{remark}
\newtheorem{remark}[theorem]{Remark}

\numberwithin{equation}{section}

\newcommand{\R}{\mathbb{R}}
\newcommand{\C}{\mathbb{C}}
\newcommand{\Z}{\mathbb{Z}}
\newcommand{\D}{\mathbb{D}}
\newcommand{\Sc}{\mathcal{S}}
\newcommand{\Gc}{\mathcal{G}}
\newcommand{\Fc}{\mathcal{F}}
\newcommand{\Nc}{\mathcal{N}}
\newcommand{\Bc}{\mathcal{B}}
\newcommand{\one}{\mathbf{1}}
\newcommand{\eps}{\varepsilon}
\newcommand{\bw}{\overline{w}}
\newcommand{\jst}{j_*}
\newcommand{\sU}{\mathsf{U}}
\DeclareMathOperator{\re}{Re}
\DeclareMathOperator{\im}{Im}
\DeclareMathOperator{\dist}{dist}
\newcommand{\nrm}[1]{\lVert #1\rVert}
\newcommand{\Ad}{\operatorname{Ad}}
\newcommand{\kf}{\mathfrak{k}}
\newcommand{\tf}{\mathfrak{t}}
\newcommand{\Hc}{\mathcal{H}}
\newcommand{\Yc}{\mathcal Y}
\newcommand{\Ecal}{\mathcal E}
\newcommand{\PP}{\mathsf P}
\newcommand{\QQ}{\mathsf Q}
\DeclareMathOperator{\tr}{tr}
\DeclareMathOperator{\diag}{diag}
\DeclareMathOperator{\acosh}{acosh}

\newcommand{\gdDimensions}{\{5,\allowbreak9,\allowbreak11,\allowbreak12,\allowbreak13,\allowbreak15,\allowbreak16,\allowbreak17,\allowbreak18,\allowbreak20,\allowbreak21\}}
\newcommand{\gdData}{%
\gdRow{5}
  {$481$ & $56.49556626181198$ & $1.05$ & $10^{21}$ & $10^{-30}$}
  {$\tfrac{3}{2}$ & $<9.69898\cdot10^{1}$ & $<1.39334\cdot10^{0}$ & $<2.1392\cdot10^{0}$ & $96$}
  {$165$ & $<1.57384\cdot10^{0}$ & $<7.64322\cdot10^{9}$ & $<4.45681\cdot10^{14}$ & $<9.58269\cdot10^{13}$}
  {$<1.54023\cdot10^{32}$ & $<5.5497\cdot10^{-40}$ & $<6.72503\cdot10^{-55}$ & $<1.35527\cdot10^{-22}$ & $<1.22103\cdot10^{-13}$}
  {$<3.60318\cdot10^{-72}$ & $<6.10514\cdot10^{16}$ & $<9.55874\cdot10^{-4}$ & $<4.62346\cdot10^{-25}$ & $<-9.99999\cdot10^{-31}$}
  {$>9.72125\cdot10^{-1}$ & $>8.25031\cdot10^{-1}$ & $>4.45696\cdot10^{-4}$}
  {$1.22$ & $2048$ & $1800$ & $2048$ & $511.006$}
\gdRow{9}
  {$385$ & $169.58703412984173$ & $1.02$ & $10^{34}$ & $10^{-43}$}
  {$\tfrac{7}{2}$ & $<9.1983\cdot10^{5}$ & $<8.14061\cdot10^{-1}$ & $<1.23237\cdot10^{0}$ & $204$}
  {$496$ & $<9.55464\cdot10^{0}$ & $<2.10086\cdot10^{18}$ & $<2.12299\cdot10^{23}$ & $<1.20829\cdot10^{23}$}
  {$<6.28334\cdot10^{52}$ & $<2.43517\cdot10^{-59}$ & $<1.64019\cdot10^{-83}$ & $<1.54869\cdot10^{-29}$ & $<3.17983\cdot10^{-14}$}
  {$<3.8756\cdot10^{-112}$ & $<1.58992\cdot10^{29}$ & $<1.70017\cdot10^{0}$ & $<8.49417\cdot10^{-35}$ & $<-9.99999\cdot10^{-44}$}
  {$>9.99203\cdot10^{-1}$ & $>9.93044\cdot10^{-1}$ & $>1.45326\cdot10^{-5}$}
  {$1.2$ & $2048$ & $1800$ & $2048$ & $601.903$}
\gdRow{11}
  {$321$ & $57.86006297284511$ & $1.1$ & $10^{32}$ & $10^{-38}$}
  {$\tfrac{9}{2}$ & $<3.98269\cdot10^{4}$ & $<7.17949\cdot10^{-1}$ & $<1.12425\cdot10^{0}$ & $110$}
  {$159$ & $<1.40122\cdot10^{0}$ & $<3.93706\cdot10^{19}$ & $<9.78801\cdot10^{24}$ & $<4.76229\cdot10^{24}$}
  {$<5.78969\cdot10^{53}$ & $<2.16886\cdot10^{-46}$ & $<6.87444\cdot10^{-70}$ & $<3.78775\cdot10^{-16}$ & $<8.73215\cdot10^{-9}$}
  {$<3.74606\cdot10^{-100}$ & $<4.36608\cdot10^{29}$ & $<2.86311\cdot10^{1}$ & $<1.3967\cdot10^{-31}$ & $<-9.99999\cdot10^{-39}$}
  {$>9.76159\cdot10^{-1}$ & $>9.10902\cdot10^{-1}$ & $>1.02429\cdot10^{-3}$}
  {$1.25$ & $2048$ & $1800$ & $2048$ & $349.695$}
\gdRow{12}
  {$257$ & $42.77848161319951$ & $1.1$ & $10^{31}$ & $10^{-50}$}
  {$5$ & $<1.29082\cdot10^{5}$ & $<6.62546\cdot10^{-1}$ & $<1.03741\cdot10^{0}$ & $50$}
  {$111$ & $<5.99436\cdot10^{0}$ & $<2.81456\cdot10^{15}$ & $<2.60965\cdot10^{19}$ & $<1.75305\cdot10^{19}$}
  {$<1.93402\cdot10^{47}$ & $<7.70138\cdot10^{-76}$ & $<4.65561\cdot10^{-98}$ & $<1.12753\cdot10^{-50}$ & $<7.32031\cdot10^{-26}$}
  {$<3.98207\cdot10^{-123}$ & $<3.66016\cdot10^{24}$ & $<4.94961\cdot10^{-3}$ & $<2.45652\cdot10^{-34}$ & $<-9.9999\cdot10^{-51}$}
  {$>9.93062\cdot10^{-1}$ & $>9.75281\cdot10^{-1}$ & $>4.53906\cdot10^{-4}$}
  {$1.25$ & $2048$ & $1800$ & $2048$ & $269.463$}
\gdRow{13}
  {$481$ & $84.57459591630585$ & $1.02$ & $10^{44}$ & $10^{-50}$}
  {$\tfrac{11}{2}$ & $<2.33676\cdot10^{9}$ & $<5.80748\cdot10^{-1}$ & $<8.7881\cdot10^{-1}$ & $168$}
  {$236$ & $<1.58733\cdot10^{0}$ & $<1.76013\cdot10^{21}$ & $<1.24785\cdot10^{27}$ & $<7.23999\cdot10^{26}$}
  {$<3.97845\cdot10^{67}$ & $<1.87769\cdot10^{-67}$ & $<4.3153\cdot10^{-98}$ & $<1.8351\cdot10^{-29}$ & $<4.88659\cdot10^{-13}$}
  {$<4.71966\cdot10^{-135}$ & $<2.4433\cdot10^{37}$ & $<1.20781\cdot10^{1}$ & $<5.97467\cdot10^{-44}$ & $<-9.99999\cdot10^{-51}$}
  {$>9.8983\cdot10^{-1}$ & $>9.37264\cdot10^{-1}$ & $>5.67768\cdot10^{-4}$}
  {$1.2$ & $2048$ & $1800$ & $2048$ & $568.835$}
\gdRow{15}
  {$417$ & $32.1111962396826$ & $1.03$ & $10^{40}$ & $10^{-49}$}
  {$\tfrac{13}{2}$ & $<8.73889\cdot10^{9}$ & $<5.11755\cdot10^{-1}$ & $<7.77598\cdot10^{-1}$ & $114$}
  {$74$ & $<1.09998\cdot10^{0}$ & $<6.99886\cdot10^{20}$ & $<4.52519\cdot10^{28}$ & $<2.82705\cdot10^{28}$}
  {$<4.68743\cdot10^{66}$ & $<5.77465\cdot10^{-54}$ & $<5.15043\cdot10^{-80}$ & $<1.14824\cdot10^{-12}$ & $<9.94207\cdot10^{-9}$}
  {$<1.23195\cdot10^{-120}$ & $<4.97104\cdot10^{40}$ & $<2.54826\cdot10^{8}$ & $<1.23371\cdot10^{-32}$ & $<-9.99942\cdot10^{-50}$}
  {$>9.61839\cdot10^{-1}$ & $>8.32915\cdot10^{-1}$ & $>4.37252\cdot10^{-3}$}
  {$1.35$ & $1024$ & $900$ & $2048$ & $209.550$}
\gdRow{16}
  {$641$ & $369.83634515753704$ & $1.02$ & $10^{58}$ & $10^{-72}$}
  {$7$ & $<8.38181\cdot10^{11}$ & $<4.78064\cdot10^{-1}$ & $<7.23328\cdot10^{-1}$ & $622$}
  {$1078$ & $<2.9373\cdot10^{1}$ & $<6.16265\cdot10^{38}$ & $<2.23123\cdot10^{45}$ & $<1.28403\cdot10^{45}$}
  {$<1.10148\cdot10^{98}$ & $<2.71173\cdot10^{-87}$ & $<2.91643\cdot10^{-129}$ & $<4.3328\cdot10^{-30}$ & $<7.98901\cdot10^{-16}$}
  {$<2.46191\cdot10^{-185}$ & $<3.99451\cdot10^{56}$ & $<6.61634\cdot10^{9}$ & $<3.30765\cdot10^{-49}$ & $<-9.99999\cdot10^{-73}$}
  {$>9.99839\cdot10^{-1}$ & $>9.98507\cdot10^{-1}$ & $>2.66517\cdot10^{-6}$}
  {$1.45$ & $1024$ & $900$ & $2048$ & $1243.897$}
\gdRow{17}
  {$257$ & $74.91722119388515$ & $1.03$ & $10^{48}$ & $10^{-55}$}
  {$\tfrac{15}{2}$ & $<2.95691\cdot10^{11}$ & $<4.54799\cdot10^{-1}$ & $<6.91009\cdot10^{-1}$ & $78$}
  {$198$ & $<2.11596\cdot10^{1}$ & $<7.69351\cdot10^{21}$ & $<9.94338\cdot10^{25}$ & $<9.00711\cdot10^{25}$}
  {$<4.89228\cdot10^{70}$ & $<5.21583\cdot10^{-71}$ & $<6.23668\cdot10^{-104}$ & $<3.8058\cdot10^{-32}$ & $<3.64832\cdot10^{-17}$}
  {$<1.06614\cdot10^{-141}$ & $<1.82416\cdot10^{38}$ & $<1.41374\cdot10^{0}$ & $<7.06073\cdot10^{-49}$ & $<-9.99999\cdot10^{-56}$}
  {$>9.98964\cdot10^{-1}$ & $>9.94623\cdot10^{-1}$ & $>7.9917\cdot10^{-5}$}
  {$1.45$ & $1024$ & $900$ & $2048$ & $170.385$}
\gdRow{18}
  {$321$ & $20.807047789264114$ & $1.05$ & $10^{40}$ & $10^{-52}$}
  {$8$ & $<3.13399\cdot10^{10}$ & $<4.3715\cdot10^{-1}$ & $<6.69826\cdot10^{-1}$ & $40$}
  {$33$ & $<1.45968\cdot10^{0}$ & $<8.91113\cdot10^{15}$ & $<3.33056\cdot10^{21}$ & $<3.03179\cdot10^{21}$}
  {$<1.45567\cdot10^{59}$ & $<1.45364\cdot10^{-70}$ & $<3.45744\cdot10^{-96}$ & $<1.25466\cdot10^{-36}$ & $<4.31556\cdot10^{-19}$}
  {$<9.98606\cdot10^{-130}$ & $<2.15778\cdot10^{33}$ & $<3.60003\cdot10^{1}$ & $<1.72899\cdot10^{-39}$ & $<-9.99999\cdot10^{-53}$}
  {$>9.4782\cdot10^{-1}$ & $>8.25454\cdot10^{-1}$ & $>7.7897\cdot10^{-3}$}
  {$1.4$ & $1024$ & $900$ & $2048$ & $150.032$}
\gdRow{20}
  {$513$ & $83.43918979610575$ & $1.03$ & $10^{58}$ & $10^{-71}$}
  {$9$ & $<5.82046\cdot10^{13}$ & $<3.89747\cdot10^{-1}$ & $<5.92132\cdot10^{-1}$ & $168$}
  {$217$ & $<3.13955\cdot10^{0}$ & $<1.49392\cdot10^{30}$ & $<2.26518\cdot10^{36}$ & $<1.92607\cdot10^{36}$}
  {$<6.37761\cdot10^{90}$ & $<1.05248\cdot10^{-91}$ & $<6.35578\cdot10^{-132}$ & $<2.07659\cdot10^{-40}$ & $<9.86528\cdot10^{-21}$}
  {$<1.65027\cdot10^{-182}$ & $<4.93264\cdot10^{50}$ & $<7.57525\cdot10^{4}$ & $<3.76177\cdot10^{-54}$ & $<-9.99999\cdot10^{-72}$}
  {$>9.92479\cdot10^{-1}$ & $>9.64444\cdot10^{-1}$ & $>5.79555\cdot10^{-4}$}
  {$1.45$ & $1024$ & $900$ & $2048$ & $320.510$}
\gdRow{21}
  {$641$ & $112.60858355309749$ & $1.03$ & $10^{64}$ & $10^{-82}$}
  {$\tfrac{19}{2}$ & $<3.38598\cdot10^{14}$ & $<3.72013\cdot10^{-1}$ & $<5.65179\cdot10^{-1}$ & $436$}
  {$300$ & $<2.30008\cdot10^{0}$ & $<3.26875\cdot10^{42}$ & $<3.05724\cdot10^{52}$ & $<1.59458\cdot10^{52}$}
  {$<7.92495\cdot10^{112}$ & $<1.6709\cdot10^{-96}$ & $<8.05097\cdot10^{-142}$ & $<2.77832\cdot10^{-28}$ & $<8.31387\cdot10^{-15}$}
  {$<2.1084\cdot10^{-209}$ & $<4.15694\cdot10^{67}$ & $<5.53845\cdot10^{16}$ & $<2.74234\cdot10^{-48}$ & $<-9.99999\cdot10^{-83}$}
  {$>9.90823\cdot10^{-1}$ & $>9.483\cdot10^{-1}$ & $>5.76413\cdot10^{-4}$}
  {$1.5$ & $1024$ & $900$ & $2048$ & $426.939$}
}

\newcommand{\gdReplayData}{%
\gdReplayRow{5}{1.22}{2048}{1800}{2048}{528.104}{535.312}
\gdReplayRow{9}{1.2}{2048}{1800}{2048}{637.539}{645.653}
\gdReplayRow{11}{1.25}{2048}{1800}{2048}{370.256}{378.099}
\gdReplayRow{12}{1.25}{2048}{1800}{2048}{286.531}{294.107}
\gdReplayRow{13}{1.2}{2048}{1800}{2048}{581.188}{588.951}
\gdReplayRow{15}{1.35}{1024}{900}{2048}{209.550}{217.013}
\gdReplayRow{16}{1.45}{1024}{900}{2048}{1243.897}{1252.022}
\gdReplayRow{17}{1.45}{1024}{900}{2048}{170.385}{177.179}
\gdReplayRow{18}{1.4}{1024}{900}{2048}{150.032}{157.956}
\gdReplayRow{20}{1.45}{1024}{900}{2048}{320.510}{328.003}
\gdReplayRow{21}{1.5}{1024}{900}{2048}{426.939}{434.442}
}

\makeatletter
\def\gdUnmath#1{\gdUnmathAux#1\gdStop}
\def\gdUnmathAux#1$#2$#3\gdStop{#2}
\def\gdIntroRadius#1&#2&#3&#4&#5\gdStop{\gdUnmath{#5}}
\def\gdIntroDefect#1&#2&#3&#4&#5\gdStop{\gdUnmath{#3}}
\def\gdIntroCurvature#1&#2&#3\gdStop{\gdUnmath{#2}}
\def\gdRow#1#2#3#4#5#6#7#8{%
  \expandafter\protected@edef\csname gdIntroRow#1\endcsname{%
    $#1$ & $O(\number\numexpr#1-1\relax)\times\noexpand\Z_2$ &
    Meridian pull-back; exact inverse &
    $r=\gdIntroRadius#2\gdStop$, $Z\gdIntroDefect#5\gdStop$;
    curvature numerator $\gdIntroCurvature#7\gdStop$ &
    \noexpand\S\noexpand\S\noexpand\ref{gd:sec:theorem}--\noexpand\ref{gd:sec:conclusion}\noexpand\\}}
\gdData
\makeatother

\begin{document}

\title[Convex Schiffer and Pompeiu counterexamples in dimensions 2--\topdim]{Convex counterexamples to the Schiffer and Pompeiu conjectures in dimensions two to \MakeUppercase{\topdimword}}

\author{Jizhou Guo}
\address{Dots Studio, Rednote}
\email{mitsuha2021b@gmail.com, sjtu18640985163@sjtu.edu.cn}

\date{October 6, 2026}


\begin{abstract}
We construct the first convex planar counterexample to the Schiffer and Pompeiu conjectures: a bounded strictly convex non-disc domain with real-analytic boundary carrying a nonconstant solution of $\Delta u+u=0$ with $u=1$ and $\partial_\nu u=0$ on the boundary. Together with the higher-dimensional constructions, this gives convex non-ball counterexamples in every dimension from two to \topdimword, and in dimensions twenty and twenty-one, with real-analytic boundaries diffeomorphic to spheres and indicator Fourier transforms vanishing on the unit sphere. In dimensions $4$, $6$, $8$, $10$ and $14$, planar reductions through compact Lie algebras of rank two use Harish-Chandra's radial part formula and Kostant's convexity theorem. In dimensions $3$, $5$ and $7$, an axisymmetric formulation on the unit ball gives a quartic equation in weighted coefficient spaces. The exact two-sided inverse of the residual-corrected linearisation combines a Dirichlet Helmholtz inverse, a holomorphic boundary real-part problem and the material-derivative identity; it gives the planar Schiffer domain and, by one dimension-parametrised argument, the domains in dimensions $5$, $9$, $11$ to $13$, $15$ to $18$, $20$ and $21$. We also construct the first convex planar counterexample to Berenstein's conjecture: a strictly convex non-disc domain with real-analytic boundary carrying a sign-changing solution of $\Delta u+u=0$ with $u=0$ and $\partial_\nu u=1$ on the boundary. Planar existence for this problem is due to Colbrook, Sadeghi and Stepaniants. All existence proofs in this paper are computer-assisted.
\end{abstract}

\maketitle

\tableofcontents
\clearpage

\section{Introduction}\label{sec:intro}\label{r2:sec:intro}

A bounded domain has the \emph{Pompeiu property} if a continuous function whose integral over every rigid motion of the domain is zero must vanish identically. Schiffer's conjecture asserts that a bounded smooth domain with spherical boundary carrying a nonconstant Helmholtz eigenfunction with constant Dirichlet data and zero Neumann data is a ball. We construct convex counterexamples, beginning with the plane: in every dimension from two to \topdimword, and in dimensions twenty and twenty-one, we find a bounded convex domain $\Omega\subset\R^n$, not a ball, and a nonconstant function $u$ such that
\begin{equation}\label{eq:main}
\Delta u+u=0\ \text{ in }\Omega,\qquad u=1\ \text{ and }\ \nabla u=0\ \text{ on }\partial\Omega .
\end{equation}

Let $D_m$ be the group of order $2m$ generated by the rotation through $2\pi/m$ and the reflection in the real axis.

Write $x=(x_1,x')\in\R\times\R^{n-1}$ for $n\in\{3,4,7\}\cup\gdDimensions$, and $x=(x',x'')\in\R^3\times\R^3$ for points of $\R^6$.

For $n\in\{8,10,14\}$ let $(\mathfrak k,G)$ be $(\mathfrak{su}(3),SU(3))$, $(\mathfrak{so}(5),SO(5))$ and $(\mathfrak g_2,G_2)$, respectively. We identify $\mathfrak k$ with $\R^n$ by a linear isometry for an $\mathrm{Ad}(G)$-invariant inner product.

\begin{theorem}[Main theorem: convex counterexamples]\label{thm:main}
In every dimension from two to \topdimword, and in dimensions twenty and twenty-one, there exist a bounded convex domain $\Omega\subset\R^n$ and a nonconstant real-valued function $u$, smooth on $\overline\Omega$, such that \eqref{eq:main} holds. For $n\geq3$, $u$ is real analytic on an open neighbourhood of $\overline\Omega$. Moreover:
\begin{enumerate}[label=\textup{(\roman*)}]
\item $\Omega$ is not a ball;
\item $\Omega$ is strictly star-shaped with respect to the origin, and $\partial\Omega$ is a compact real-analytic hypersurface diffeomorphic to $S^{n-1}$;
\item for $n=2$, $\Omega$ is strictly convex and invariant under $D_{182}$; for $n\in\{3,7\}\cup\gdDimensions$, $\Omega$ is invariant under $O(n-1)$ on $x'$ and under $x_1\mapsto-x_1$; for $n=4$, $\Omega$ is invariant under the action of $O(3)$ on $x'$ and under $x_1\mapsto-x_1$; for $n=6$, $\Omega$ is invariant under the action of $O(3)\times O(3)$ on $(x',x'')$ and under the exchange $(x',x'')\mapsto(x'',x')$; for $n\in\{8,10,14\}$, $\Omega\subset\mathfrak k$ is invariant under the adjoint action of $G$ and under $x\mapsto-x$.
\end{enumerate}
In particular, Schiffer's conjecture (Conjecture~\ref{conj:schiffer} below) fails in every dimension $2\le n\le\topdim$ and in dimensions $20$ and $21$, even within the class of convex domains.
\end{theorem}

\begin{corollary}[Pompeiu]\label{cor:pompeiu}
Let $n$ satisfy $2\le n\le\topdim$ or $n\in\{20,21\}$ and let $\Omega\subset\R^n$ be the corresponding domain of Theorem~\ref{thm:main}. The Fourier transform of $\one_\Omega$ vanishes on the unit sphere of $\R^n$. Consequently $\int_{\sigma(\Omega)}\cos x_1\,dx=0$ for every rigid motion $\sigma$ of $\R^n$, the domain $\Omega$ fails the Pompeiu property, and the Pompeiu conjecture fails in $\R^n$, even within the class of convex domains.
\end{corollary}

\subsection{The planar domain}
Write $\D=\{w\in\C:|w|<1\}$. Our normalisation fixes the physical eigenvalue to be one; dilating a domain changes the eigenvalue by the inverse square of the dilation factor.
Here $\nu$ denotes the outward unit normal.

\begin{theorem}\label{r2:thm:main}
There exist a bounded strictly convex domain $\Omega\subset\R^2$, strictly star-shaped with respect to the origin and different from a disc, with real-analytic boundary, and a nonconstant real function $u$, smooth on $\overline\Omega$, such that
\begin{equation}\label{r2:eq:main}
 \Delta u+u=0\quad\text{in }\Omega,\qquad
 u=1,\quad\partial_\nu u=0\quad\text{on }\partial\Omega.
\end{equation}
The domain is invariant under $D_{182}$. It is the image of $\D$ under a conformal map $\psi$ normalised by $\psi(0)=0$, $\psi'(0)>0$, with
\begin{align}
 814.45&<\psi'(0)<814.46,\label{r2:eq:radius-main}\\
 \re\left(1+\frac{w\psi''(w)}{\psi'(w)}\right)&>0.2998879
 \quad (|w|=1).\label{r2:eq:convex-main}
\end{align}
More precisely, $\psi$ is holomorphic and univalent on $|w|<R$, where $R=2^{1/182}>1.0038$, and satisfies $\re\psi'(w)>808$ there. On the closed unit disc $|\psi'(w)|>811$, and the coefficient of $w^{183}$ in $\psi$ has absolute value greater than $0.01645$.
The Fourier transform of $\one_\Omega$ vanishes on $|\xi|=1$, so $\Omega$ fails the Pompeiu property.
\end{theorem}

The exact solution is specified by a contraction ball rather than a finite Taylor polynomial. With the spaces of Subsection~\ref{r2:sec:spaces} and the exactly clamped centre $(g_0,p_0)$ of Subsection~\ref{r2:sec:centre}, there is a unique zero of
\begin{equation}\label{r2:eq:F-intro}
 F(g,p)=g+|p|^2(1+Kg)
\end{equation}
(Here $F$ denotes the cubic map of the plane; from Subsection~\ref{subsec:lie} on, $F$ is the planar field on a Cartan section.) The zero lies in
\[
 \nrm{g-g_0}_Y+\nrm{p-p_0}_W\le10^{-32},\qquad
 \psi(w)=\int_0^w p(z)\,dz.
\]
This description, together with the decimal input in the ancillary files, determines the domain of Theorem~\ref{r2:thm:main}. 

The domain lies within relative Hausdorff distance about $2\cdot10^{-5}$ of a centred disc, but its curvature varies substantially. The centre has conformal radius approximately $814.457$ and $182$ small ripples. A binary64 evaluation of the centre gives a peak-to-trough radial variation of about $4\cdot10^{-5}$ relative to the conformal radius, and curvature between about $0.35$ and $1.69$ times the curvature of the circle of that radius; see Figure~\ref{r2:fig:domain}. These values describe the centre and are not interval enclosures. The strict lower bound~\eqref{r2:eq:convex-main} comes from the existence-ball estimates of Section~\ref{r2:sec:existence}.

\begin{figure}[htbp]
 \centering
 \includegraphics[width=\textwidth]{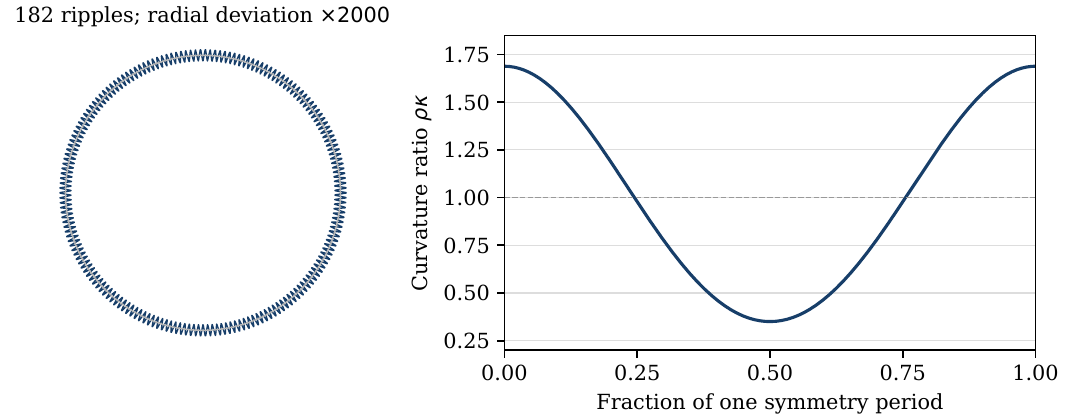}
 \caption{The forty-mode centre $\psi_0$ of Subsection~\ref{r2:sec:centre}, with conformal radius $\rho=\psi_0'(0)$. Left: Schematic boundary, in which the radial deviation from the circle of radius $\rho$ (grey) is multiplied by $2000$; this exaggerated curve is not the domain of Theorem~\ref{r2:thm:main}. Right: the curvature of the centre, multiplied by $\rho$, over one symmetry period of the conformal angle. Both panels are binary64 evaluations of the centre, not interval enclosures.}\label{r2:fig:domain}
\end{figure}

\subsection{A convex domain for Berenstein's problem}\label{br:intro}
Berenstein's Dirichlet overdetermined problem prescribes $u=0$ and $\partial_\nu u=\mathrm{const}\ne0$ for a Helmholtz eigenfunction on a bounded domain; it is conjectured to be solvable only on balls. Colbrook, Sadeghi and Stepaniants~\cite{CSS26} constructed a planar counterexample. Theorem~\ref{br:thm:main} gives the first convex planar counterexample: a bounded strictly convex non-disc domain with real-analytic boundary and a sign-changing solution with eigenvalue and outward flux both equal to one. Equivalently, the Fourier transform of its boundary arc-length measure vanishes on the unit circle. Part~\ref{br:part} proves the result by a Robin inverse and an exact two-sided inverse of the residual-corrected fixed-disc linearisation.

\subsection{The Pompeiu problem and Schiffer's conjecture}

Let $n\geq2$. A bounded domain $\Omega\subset\R^n$ has the \emph{Pompeiu property} if the only continuous function $f\colon\R^n\to\C$ such that
\[
\int_{\sigma(\Omega)}f(x)\,dx=0\qquad\text{for every rigid motion }\sigma\text{ of }\R^n
\]
is $f\equiv0$. The question of which domains have this property goes back to Pompeiu \cite{Pom29a,Pom29b,Pom29c}. Balls do not have it: if $J_{n/2}(kR)=0$ for some $k>0$, then the Fourier transform of $\one_{B_R}$ vanishes on the sphere $\{|\xi|=k\}$, and $f(x)=e^{ikx_1}$ integrates to zero over every ball of radius $R$. Brown, Schreiber and Taylor \cite{BST73} characterised the failure of the Pompeiu property through the zero set of the Fourier--Laplace transform of the indicator function; in particular, a bounded domain fails the property whenever $\widehat{\one_\Omega}$ vanishes identically on a sphere $\{|\xi|=k\}$ with $k>0$ (this direction is elementary, see Section~\ref{sec:pompeiu}). The \emph{Pompeiu conjecture} asserts that a bounded domain with Lipschitz boundary homeomorphic to $S^{n-1}$ fails the Pompeiu property only if it is a ball. Some topological assumption is necessary: for every $k>0$ the function $R\mapsto R^{n/2}J_{n/2}(kR)$ takes equal values at suitable distinct radii $R_1<R_2$, so the spherical shell $\{R_1<|x|<R_2\}$ fails the Pompeiu property, and such shells are simply connected when $n\geq3$.

\begin{conjecture}[Schiffer]\label{conj:schiffer}
Let $\Omega\subset\R^n$ be a bounded domain with smooth boundary homeomorphic to $S^{n-1}$. If, for some $\mu>0$, there is a nonconstant solution of
\begin{equation}\label{eq:schiffer}
\Delta u+\mu u=0\ \text{ in }\Omega,\qquad u=\text{const}\ \text{ and }\ \partial_\nu u=0\ \text{ on }\partial\Omega,
\end{equation}
then $\Omega$ is a ball.
\end{conjecture}

On a ball $B_R$, the radial Neumann eigenfunctions $u=r^{1-n/2}J_{n/2-1}(\sqrt\mu\,r)$ solve \eqref{eq:schiffer} precisely when $J_{n/2}(\sqrt\mu\,R)=0$. Formulations of the conjecture in the literature differ in their regularity and topological hypotheses (smooth boundary, simply connected domain, connected boundary); the domains constructed below satisfy all of them. Williams \cite{Wil76} showed that, for bounded domains with Lipschitz boundary homeomorphic to a sphere, the failure of the Pompeiu property is equivalent to the solvability of \eqref{eq:schiffer} for some $\mu>0$, building on \cite{BST73}; the direction from \eqref{eq:schiffer} to the failure of the Pompeiu property is a consequence of Green's identity. In the planar case Schiffer's problem appears as Problem~80 in Yau's list \cite{Yau82}. We refer to Zalcman's survey \cite{Zal92} for the classical bibliography.

A large literature proves the conjectured rigidity under additional hypotheses: for domains carrying solutions for infinitely many eigenvalues \cite{BY87}, for low eigenvalues \cite{Den12}, for perturbations of balls \cite{Agr93,Kob93,Can14}, for planar convex domains \cite{BK82} and centrally symmetric planar domains \cite{Mon26}, in particular at high frequency \cite{DSWZ25}, and under further conditions \cite{Dal99,KL20}. Section~\ref{sec:rigidity} compares our examples with the results of \cite{BY87,Den12,Kob93,BK82,Mon26,Dal00,DSWZ25} and with the statements of \cite{Ber80,Avi86,DLS26} quoted there from secondary sources.

\subsection{Previous constructions}

In August 2026 the planar conjectures were disproved independently by two groups. Colbrook and Stepaniants \cite{CS26} constructed a single explicit ten-fold symmetric noncircular domain $\Omega\subset\R^2$ with real-analytic boundary for which \eqref{eq:schiffer} holds with $\sqrt\mu\in(31.967007261,31.967007293)$. Their proof pulls the problem back to the unit disc by a conformal map, which turns it into a cubic equation on a space of disk-polynomial coefficients, and establishes the existence of an exact solution near an explicit approximation by a computer-assisted contraction argument. Cao-Labora and de Dios Pont \cite{CLdDP26} constructed infinitely many $N$-fold symmetric planar counterexamples, for $N$ sufficiently large, by a bifurcation argument for a relaxed problem in which $N$ is a real parameter, with a branch length that is uniform in $N$. Colbrook, Sadeghi and Stepaniants \cite{CSS26} adapted the conformal method to disprove the planar Berenstein conjecture. Dolha \cite{Dol26} gave a further computer-assisted planar counterexample, with seven-fold dihedral symmetry, in a public software repository. None of \cite{CS26}, \cite{CLdDP26} and \cite{CSS26} asserts convexity; the published centres of the examples of \cite{CS26,Dol26} have negative boundary curvature at some points. Singh \cite{Sin26} used the counterexample of \cite{CS26} to obtain noncircular homogeneous dielectric inclusions that are exactly non-scattering for an explicit incident wave in the two-dimensional scalar Maxwell reduction. Earlier constructions of Schiffer-type domains in related settings include domains in the flat cylinder $\R^N\times\R/2\pi\Z$ and on $S^2$ \cite{FMW25}, contractible domains on the half-sphere \cite{CLF25}, doubly connected planar domains on whose two boundary components the eigenfunction takes different constants \cite{EFRS25}, and noncircular planar solutions of the related problem with nonzero constant Neumann data \cite{Whe25}. Schiffer-type domains in round spheres bounded by isoparametric hypersurfaces were found earlier \cite{Shk00,PS21}; their boundaries are in general not topological spheres. Apart from our earlier preprint \cite{Guo26} and version~1 of this paper (see the end of this section), to the best of our knowledge no counterexample to either conjecture was previously known for bounded domains in $\R^n$, $n\geq3$, with boundary homeomorphic to a sphere.

Rigidity for convex planar domains is known under several additional hypotheses, including width and central-symmetry conditions, and at high frequency above a domain-dependent threshold \cite{BK82,Mon26,Dal00,DSWZ25}. To our knowledge, no convex planar counterexample was known. Theorem~\ref{r2:thm:main} supplies such an example.

\subsection{The higher-dimensional domains}

\begin{theorem}[Convex example in dimension three]\label{thm:r3}
There exist a bounded convex domain $\Omega\subset\R^3$ and a nonconstant real-valued function $u$, real analytic on an open neighbourhood of $\overline\Omega$, such that
\eqref{eq:main} holds.
The domain is not a ball, is strictly star-shaped with respect to the origin, and has compact real-analytic boundary diffeomorphic to $S^2$. It is invariant under $O(2)$ acting on $x'$ and the reflection $x_1\mapsto-x_1$.
\end{theorem}

The case $n=2$ is Theorem~\ref{r2:thm:main}, the case $n=3$ is Theorem~\ref{thm:r3}, the cases $n=5,7$ are proved in Sections~\ref{sec:r5} and~\ref{r7:sec}, the cases $n\in\gdDimensions\setminus\{5\}$ and a second proof for $n=5$ are given in Part~\ref{gd:part}, and the cases $n=4,6,8,10,14$ are proved in Part~\ref{part:planar}. Table~\ref{tab:dimensions} lists for each dimension the symmetry reduction, the method, the certified radius, linear defect and geometric bound, and the place of the proof.

Rescaling $\Omega$ and $u$ gives, for every $\mu>0$, a counterexample with eigenvalue $\mu$ in \eqref{eq:schiffer}.

The convex three-dimensional example is distinct from the nonconvex example of Proposition~\ref{r3:prop-AX3}, which is retained below. The comparison with rigidity results and other claims is in Section~\ref{sec:rigidity}.

\begin{table}[htbp]
\centering\footnotesize
\begin{tabular}{@{}c>{\raggedright\arraybackslash}p{2.2cm}>{\raggedright\arraybackslash}p{3.2cm}>{\raggedright\arraybackslash}p{5.5cm}>{\raggedright\arraybackslash}p{1.5cm}@{}}
\toprule
$n$ & Symmetry reduction & Method & Certified bounds & Proof\\
\midrule
$2$ & $D_{182}$ & Exact coupled inverse & $r=10^{-32}$, $Z<2.122\cdot10^{-38}$; convexity factor $>0.2998879$ & \S\S\ref{r2:sec:inverse}--\ref{r2:sec:existence}\\
$3$ & $O(2)\times\Z_2$ & Unit-ball pull-back & $r=10^{-7}$, $Z<0.776701$; convexity factor $>0.7878$ & \S\ref{sec:r3}\\
$4$ & $\mathfrak u(2)$ & Cartan section & $r=1/50000$, $Z<0.360262$; convexity factor $>0.414074$ & \S\S\ref{sec:reduction}--\ref{sec:recon}\\
$5$ & $O(4)\times\Z_2$ & Degree-weighted pull-back & $r=10^{-10}$, $Z<0.887235$; convexity factor $>0.839460$ & \S\ref{sec:r5}; \S\S\ref{gd:sec:theorem}--\ref{gd:sec:conclusion}\\
$6$ & $\mathfrak{so}(4)$ & Cartan section & $r=1/5000$, $Z<0.464977$; convexity factor $>0.71569200$ & \S\ref{sec:r6}\\
$7$ & $O(6)\times\Z_2$ & Squared-degree-weighted pull-back & $r=10^{-9}$, $Z_0<0.898934481$; convexity factor $>0.951916263887$ & \S\ref{r7:sec}\\
$8$ & $\mathfrak{su}(3)$ & Cartan section & $r=1/30000$, $Z<0.418572$; convexity factor $>0.83153153$ & \S\ref{lie:sec}\\
\csname gdIntroRow9\endcsname
$10$ & $\mathfrak{so}(5)$ & Cartan section & $r=10^{-7}$, $Z<0.552897$; convexity factor $>0.97305365$ & \S\ref{lie:sec}\\
\csname gdIntroRow11\endcsname
\csname gdIntroRow12\endcsname
\csname gdIntroRow13\endcsname
$14$ & $\mathfrak g_2$ & Cartan section & $r=1/6000$, $Z<0.604948$; convexity factor $>0.89406714$ & \S\ref{lie:sec}\\
\csname gdIntroRow15\endcsname
\csname gdIntroRow16\endcsname
\csname gdIntroRow17\endcsname
\csname gdIntroRow18\endcsname
\csname gdIntroRow20\endcsname
\csname gdIntroRow21\endcsname
\bottomrule
\end{tabular}
\caption{The convex examples. The radius and linear defect use the norms of the indicated proof; the row $n=5$ records the proof of Section~\ref{sec:r5}, and the second proof in Part~\ref{gd:part} has the bounds of Tables~\ref{gd:tab:inputs}--\ref{gd:tab:geometry}; for the rows $n\in\gdDimensions\setminus\{5\}$, the curvature numerator is $\chi_*$ of \eqref{gd:eq:geometry-sums}, which controls both meridian and transverse curvatures; for the other entries, the convexity factor is a lower bound for $\re(1+w\psi''/\psi')$ on the unit circle of the planar or meridian map. The equation in Part~\ref{gd:part} is quartic. The earlier equations are cubic for $n=2,4$, quartic for $n=3,5,6,7$, and of degree $m+2$ for the Cartan reductions $n=2m+2$.}\label{tab:dimensions}
\end{table}

\begin{remark}[The domains]\label{rem:domain}
In dimensions $4$, $6$, $8$, $10$ and $14$ the domains are obtained from a planar domain $D=\psi(\D)\subset\R^2\cong\C$, the image of the unit disc under a map $\psi$ with real Taylor coefficients; the axisymmetric domains of Part~\ref{part:r3} and Theorem~\ref{gd:thm:main} are images of unit balls under axisymmetric extensions of such maps.

\emph{Dimension three.} Here $\Omega=\Theta(B^3)$, where $B^3$ is the unit ball and
\[
\Theta(x_1,x')=\Big(\re\psi(x_1+i|x'|),\ \frac{\im\psi(x_1+i|x'|)}{|x'|}\,x'\Big),\qquad\psi(w)=\sum_{j\ \mathrm{odd}}c_jw^j .
\]
The exact coefficient vector lies in the ball
\[
\sum_{j\ \mathrm{odd}}8b\,j\Big(\frac{21}{20}\Big)^j|c_j-c_j^\circ|\leq10^{-7},\qquad b=7626062690497999\cdot2^{-46},
\]
about the dyadic centre in \texttt{anc/r3\_convex/candidate.json}. Its linear coefficient is approximately $108.372869$; the nearby radial Schiffer ball has radius $j_{3/2,34}\approx108.37572$, the thirty-fourth positive root of $\tan x=x$. The non-ball witness is $c_{19}\ne0$. An independent example near $j_{3/2,15}\approx48.67414$ is strictly star-shaped but nonconvex (Proposition~\ref{r3:prop-AX3}).

\emph{Dimension five.} Here $\Omega=\rho\,\Theta(B^5)$, where $B^5$ is the unit ball, $\rho=7951044098899425\cdot2^{-47}\approx56.4955663$, and, for $x'\in\R^4$,
\[
\Theta(x_1,x')=\Big(\re\psi(x_1+i|x'|),\ \frac{\im\psi(x_1+i|x'|)}{|x'|}\,x'\Big),\qquad\psi(w)=\sum_{j\ \mathrm{odd}}c_jw^j .
\]
The exact coefficient vector lies in the ball
\[
\sum_{j\ \mathrm{odd}}30\rho^2\,j\Big(\frac{21}{20}\Big)^j\Big(1+\frac{j-1}{10^6}\Big)|c_j-c_j^\circ|\leq10^{-10}
\]
about the dyadic centre embedded in \texttt{anc/r5/verify\_r5.py}; one has $c^\circ_1\approx1.0002252$ and $c^\circ_3\approx4.457\cdot10^{-4}$, and the non-ball witness is $c_3\neq0$. Numerically, the distance of $\partial\Omega$ from the origin lies between $56.4763$ and $56.6242$. The ball of $\R^5$ that carries a radial solution of \eqref{eq:main} with the closest radius has radius $j_{5/2,17}\approx56.495566$, the seventeenth positive zero of $J_{5/2}$, and the relative deviation of $\partial\Omega$ from that sphere lies between $-0.034\%$ and $+0.228\%$. Expanding the boundary radius in zonal spherical harmonics on $S^4$, the dominant nonconstant term has degree twelve, with relative amplitude about $2.5\cdot10^{-3}$.

\emph{Dimension seven.} Here $\Omega=\rho\Theta(B^7)$ with the same lift, $x'\in\R^6$ and $\rho=731659610854495\cdot2^{-43}\approx83.1800674$. The shape coefficients satisfy
\[
\sum_{j\ \mathrm{odd}}25\rho^2j\Big(\frac65\Big)^j\Big(1+\frac{j-1}{10^6}\Big)^2|c_j-c_j^\circ|\leq10^{-9},
\]
about the exact dyadic centre embedded in \texttt{anc/r7/verify\_r7.py}, with $c^\circ_1\approx1.0000971$ and $c^\circ_3\approx1.899137\cdot10^{-4}$. The non-ball witness is $c_3\ne0$. Numerically, the unscaled meridian radius lies between $0.9999409$ and $1.0009305$; the nearest radial Schiffer ball has radius $j_{7/2,25}\approx83.1800674$, and the dominant nonconstant zonal harmonic has degree fourteen. Section~\ref{r7:sec} certifies strict convexity on the whole existence ball.

\emph{Dimension four.} Here $\Omega=\{(x_1,x'):(x_1,|x'|)\in D\}$ and $\psi(w)=\sum_{j\ \mathrm{odd}}c_jw^j$. The coefficients satisfy
\[
\sum_{j\ \mathrm{odd}}(c_1^\circ)^2(j+1)\Big(\frac{21}{20}\Big)^{j}\,|c_j-c_j^\circ|\leq\frac1{50000},
\]
where $c^\circ_1,c^\circ_3,\ldots,c^\circ_{41}$ are explicit dyadic rationals listed in the verification script (Section~\ref{sec:cap}) and $c^\circ_j=0$ for $j>41$; one has $c_1^\circ\approx14.840043$ and $c^\circ_3\approx0.0912$, $c^\circ_5\approx0.0942$, $c^\circ_7\approx0.0964$. For orientation we record some numerical observations about the centre $\psi^\circ$, which are not used in the proofs. The boundary of $D$ is a radial graph whose radius lies between $14.7213$ and $15.1343$. The ball of $\R^4$ that carries a radial solution of \eqref{eq:main} with the same eigenvalue and closest radius has radius $j_{2,4}\approx14.79595$, the fourth positive zero of $J_2$; thus the relative deviation of $\partial\Omega$ from that sphere lies between $-0.5\%$ and $+2.3\%$. Expanding the boundary radius of $\Omega$ in zonal spherical harmonics on $S^3$, the dominant nonconstant term is the zonal harmonic of degree six, with relative amplitude about $2.2\times10^{-2}$ at the poles.

\emph{Dimension six.} Here $\Omega=\{(x',x''):(|x'|,|x''|)\in D\}$ and $\psi(w)=\sum_{j\equiv1\ (\mathrm{mod}\ 4)}c_jw^j$. The coefficients satisfy
\[
\sum_{j\equiv1\ (\mathrm{mod}\ 4)}\frac{(c_1^\circ)^3}{2}(j+2)\Big(\frac{11}{10}\Big)^{j+1}\,|c_j-c_j^\circ|\leq\frac1{5000},
\]
where $c^\circ_1,c^\circ_5,\ldots,c^\circ_{57}$ are explicit dyadic rationals listed in the ancillary data file \texttt{center\_r6.json} (Section~\ref{subsec:cap6}) and $c^\circ_j=0$ for $j>57$; one has $c_1^\circ\approx28.864751$, $c^\circ_5\approx-0.0852970$ and $c^\circ_9\approx-0.0837734$, while $|c^\circ_j|<4\cdot10^{-4}$ for $j\geq13$. Numerically, the boundary of $D$ is a radial graph whose radius lies between $28.6958$ and $28.9596$, and the ball of $\R^6$ with the same volume as $\Omega$ has radius $28.908339$. The ball of $\R^6$ that carries a radial solution of \eqref{eq:main} with the same eigenvalue and closest radius has radius $j_{3,8}\approx28.908351$, the eighth positive zero of $J_3$; the relative deviation of $\partial\Omega$ from that sphere lies between $-0.74\%$ and $+0.18\%$, so that the radius of $\partial\Omega$ varies by about $0.92\%$. Expanding the boundary radius in the $O(3)\times O(3)$-invariant spherical harmonics on $S^5$, which are the functions $U_k(|\omega'|^2-|\omega''|^2)$ of degree $2k$ (Section~\ref{subsec:mechanism}), the dominant nonconstant term has degree eight, with relative amplitude about $7.2\times10^{-3}$.

\emph{Dimensions eight, ten and fourteen.} Let $m=3,4,6$, respectively. Here $\Omega=\mathrm{Ad}(G)D$, where $D=\Omega\cap\mathfrak t$ is the section of $\Omega$ by a Cartan subalgebra $\mathfrak t\cong\C$ (Section~\ref{subsec:lie}), and $\psi(w)=\sum_{j\equiv1\ (\mathrm{mod}\ 2m)}c_jw^j$ lies in an explicit weighted $\ell^1$ ball around explicit dyadic coefficients $c^\circ_j$ (Section~\ref{lie:sec}). One has $c_1^\circ\approx20.787029$, $38.152926$, $88.497305$ and $c^\circ_{2m+1}\approx-0.0774840$, $-0.0138325$, $0.0463087$, respectively. Numerically, the boundary of $D$ is a radial graph whose radius lies between $20.7086$ and $20.8635$, between $38.1390$ and $38.1667$, and between $88.4477$ and $88.5402$. The balls that carry a radial solution of \eqref{eq:main} with the closest radii have radii $j_{4,5}\approx20.826933$, $j_{5,10}\approx38.159869$ and $j_{7,25}\approx88.474363$, and the relative deviations of $\partial\Omega$ from these spheres lie between $-0.57\%$ and $+0.18\%$, between $-0.055\%$ and $+0.018\%$, and between $-0.030\%$ and $+0.074\%$. In each case the dominant nonconstant term of the radius of $D$ is a multiple of $\cos(2m\theta)$, which corresponds to the $\mathrm{Ad}(G)$-invariant spherical harmonic of degree $2m$; its restriction to the unit circle of $\mathfrak t$ is a multiple of $1+2\cos(2m\theta)$.

\emph{Dimension nine.} For the centre of Part~\ref{gd:part} (Table~\ref{gd:tab:inputs}), $\rho\approx169.58703$, $c_1^\circ\approx0.9999926$ and $c_3^\circ\approx-1.4533\cdot10^{-5}$. Numerically, the boundary distance from the origin is approximately between $169.5727$ and $169.5875$. The nearest radial Schiffer ball has radius $j_{9/2,52}\approx169.587034$, with relative boundary deviations approximately $-0.00846\%$ to $+0.00027\%$. The dominant nonconstant zonal harmonic has degree $20$, with relative polar amplitude about $8.5\cdot10^{-5}$.

\emph{Dimension eleven.} For the centre of Part~\ref{gd:part} (Table~\ref{gd:tab:inputs}), $\rho\approx57.86006$, $c_1^\circ\approx1.0005497$ and $c_3^\circ\approx1.0243\cdot10^{-3}$. Numerically, the boundary distance from the origin is approximately between $57.8560$ and $58.0511$. The nearest radial Schiffer ball has radius $j_{11/2,16}\approx57.860063$, with relative boundary deviations approximately $-0.00701\%$ to $+0.33013\%$. The dominant nonconstant zonal harmonic has degree $10$, with relative polar amplitude about $3.4\cdot10^{-3}$.

\emph{Dimension twelve.} For the centre of Part~\ref{gd:part} (Table~\ref{gd:tab:inputs}), $\rho\approx42.77848$, $c_1^\circ\approx1.0002524$ and $c_3^\circ\approx4.5391\cdot10^{-4}$. Numerically, the boundary distance from the origin is approximately between $42.7775$ and $42.8315$. The nearest radial Schiffer ball has radius $j_{6,11}\approx42.778482$, with relative boundary deviations approximately $-0.00241\%$ to $+0.12399\%$. The dominant nonconstant zonal harmonic has degree $8$, with relative polar amplitude about $1.3\cdot10^{-3}$.

\emph{Dimension thirteen.} For the centre of Part~\ref{gd:part} (Table~\ref{gd:tab:inputs}), $\rho\approx84.57460$, $c_1^\circ\approx1.0003023$ and $c_3^\circ\approx5.6777\cdot10^{-4}$. Numerically, the boundary distance from the origin is approximately between $84.5726$ and $84.7411$. The nearest radial Schiffer ball has radius $j_{13/2,24}\approx84.574596$, with relative boundary deviations approximately $-0.00241\%$ to $+0.19682\%$. The dominant nonconstant zonal harmonic has degree $12$, with relative polar amplitude about $2.0\cdot10^{-3}$.

\emph{Dimension fifteen.} For the centre of Part~\ref{gd:part} (Table~\ref{gd:tab:inputs}), $\rho\approx32.11120$, $c_1^\circ\approx0.9973022$ and $c_3^\circ\approx-4.3725\cdot10^{-3}$. Numerically, the boundary distance from the origin is approximately between $31.8317$ and $32.1169$. The nearest radial Schiffer ball has radius $j_{15/2,7}\approx32.111196$, with relative boundary deviations approximately $-0.87031\%$ to $+0.01780\%$. The dominant nonconstant zonal harmonic has degree $6$, with relative polar amplitude about $1.3\cdot10^{-2}$.

\emph{Dimension sixteen.} For the centre of Part~\ref{gd:part} (Table~\ref{gd:tab:inputs}), $\rho\approx369.83635$, $c_1^\circ\approx0.9999986$ and $c_3^\circ\approx-2.6652\cdot10^{-6}$. Numerically, the boundary distance from the origin is approximately between $369.8305$ and $369.8364$. The nearest radial Schiffer ball has radius $j_{8,114}\approx369.836345$, with relative boundary deviations approximately $-0.00159\%$ to $+0.00001\%$. The dominant nonconstant zonal harmonic has degree $28$, with relative polar amplitude about $1.6\cdot10^{-5}$. The centre is extremely close to a ball: its smallest and largest boundary distances $r_{\min},r_{\max}$ give the relative radial range $(r_{\max}-r_{\min})/(r_{\max}+r_{\min})\approx8.0\cdot10^{-6}$.

\emph{Dimension seventeen.} For the centre of Part~\ref{gd:part} (Table~\ref{gd:tab:inputs}), $\rho\approx74.91722$, $c_1^\circ\approx1.0000441$ and $c_3^\circ\approx7.9917\cdot10^{-5}$. Numerically, the boundary distance from the origin is approximately between $74.9171$ and $74.9348$. The nearest radial Schiffer ball has radius $j_{17/2,20}\approx74.917221$, with relative boundary deviations approximately $-0.00014\%$ to $+0.02340\%$. The dominant nonconstant zonal harmonic has degree $10$, with relative polar amplitude about $2.3\cdot10^{-4}$.

\emph{Dimension eighteen.} For the centre of Part~\ref{gd:part} (Table~\ref{gd:tab:inputs}), $\rho\approx20.80705$, $c_1^\circ\approx0.9944513$ and $c_3^\circ\approx-7.7897\cdot10^{-3}$. Numerically, the boundary distance from the origin is approximately between $20.5106$ and $20.8138$. The nearest radial Schiffer ball has radius $j_{9,3}\approx20.807048$, with relative boundary deviations approximately $-1.42481\%$ to $+0.03229\%$. The dominant nonconstant zonal harmonic has degree $4$, with relative polar amplitude about $2.0\cdot10^{-2}$.

\emph{Dimension twenty.} For the centre of Part~\ref{gd:part} (Table~\ref{gd:tab:inputs}), $\rho\approx83.43919$, $c_1^\circ\approx0.9996765$ and $c_3^\circ\approx-5.7956\cdot10^{-4}$. Numerically, the boundary distance from the origin is approximately between $83.3038$ and $83.4397$. The nearest radial Schiffer ball has radius $j_{10,22}\approx83.439190$, with relative boundary deviations approximately $-0.16231\%$ to $+0.00061\%$. The dominant nonconstant zonal harmonic has degree $10$, with relative polar amplitude about $1.7\cdot10^{-3}$.

\emph{Dimension twenty-one.} For the centre of Part~\ref{gd:part} (Table~\ref{gd:tab:inputs}), $\rho\approx112.60858$, $c_1^\circ\approx1.0003139$ and $c_3^\circ\approx5.7641\cdot10^{-4}$. Numerically, the boundary distance from the origin is approximately between $112.6081$ and $112.8144$. The nearest radial Schiffer ball has radius $j_{21/2,31}\approx112.608584$, with relative boundary deviations approximately $-0.00045\%$ to $+0.18275\%$. The dominant nonconstant zonal harmonic has degree $12$, with relative polar amplitude about $1.8\cdot10^{-3}$.
\end{remark}

\subsection{Compact Lie algebras of rank two}\label{subsec:lie}

The domains in dimensions $4$, $6$, $8$, $10$ and $14$ are instances of a single construction. Let $\mathfrak k$ be a compact Lie algebra of rank two, with an $\mathrm{Ad}$-invariant inner product, let $\mathfrak t\subset\mathfrak k$ be a Cartan subalgebra, identified with $\C$ by a linear isometry, and let $W$ be the Weyl group. If $\mathfrak k$ is not abelian, then $W$ is the dihedral group of order $2m$ with $m\in\{1,2,3,4,6\}$, there are $m$ positive roots, $\dim\mathfrak k=2m+2$, and, up to isomorphism, $\mathfrak k$ is $\mathfrak u(2)$, $\mathfrak{so}(4)\cong\mathfrak{su}(2)\oplus\mathfrak{su}(2)$, $\mathfrak{su}(3)$, $\mathfrak{so}(5)\cong\mathfrak{sp}(2)$ or $\mathfrak g_2$. The product $\varpi$ of the positive roots is a $W$-anti-invariant harmonic polynomial of degree $m$; in suitable coordinates $\varpi(z)$ is a nonzero multiple of $\im z^m$. For an $\mathrm{Ad}$-invariant $C^2$ function $u$ with restriction $f=u|_{\mathfrak t}$, a classical formula of Harish-Chandra \cite{HC57} (see \cite[Ch.~II]{Hel84}) states that
\begin{equation}\label{eq:radial-part}
(\Delta_{\mathfrak k}u)|_{\mathfrak t}=\varpi^{-1}\Delta_{\mathfrak t}(\varpi f)
\end{equation}
at the regular points of $\mathfrak t$. Consequently, $\mathrm{Ad}$-invariant solutions of \eqref{eq:main} on $\Omega\subset\mathfrak k$ correspond to $W$-anti-invariant solutions $F=\varpi f$ of the planar problem
\begin{equation}\label{eq:intro-Pm}
\Delta F+F=0\ \text{ in }D,\qquad F=\varpi\ \text{ and }\ \nabla F=\nabla\varpi\ \text{ on }\partial D,
\end{equation}
on the Cartan section $D=\Omega\cap\mathfrak t$. There is no potential term, so the conformal method of \cite{CS26} applies to \eqref{eq:intro-Pm}. Conversely, if $D$ is $W$-invariant, contains $0$ and has a positive real-analytic radial function, then $\Omega=\mathrm{Ad}(G)D$ is a bounded domain whose boundary is a real-analytic hypersurface diffeomorphic to $S^{2m+1}$, and an analytic solution of \eqref{eq:intro-Pm} lifts to a solution of \eqref{eq:main} on $\Omega$ (Section~\ref{lie:sec}); some hypothesis on $D$ is needed here, since for an annulus $D$ the lifted boundary has two components. By Kostant's convexity theorem \cite{Kos73}, in the form of \cite[Theorems~2.2 and~3.6]{Lew00}, $\Omega$ is convex if and only if $D$ is. For $\mathfrak k=\mathfrak u(2)\cong\R\times\R^3$ the adjoint group acts as $SO(3)$ on the second factor and $\varpi$ is a multiple of $y$; for $\mathfrak k=\mathfrak{so}(4)\cong\R^3\times\R^3$ it acts as $SO(3)\times SO(3)$ and $\varpi$ is a multiple of $y_1y_2$. These are the reductions of Sections~\ref{sec:reduction} and~\ref{sec:r6}, which we present there in elementary form. For the abelian algebra $\R^2$ one has $\varpi=1$, and \eqref{eq:intro-Pm} is the planar Schiffer problem itself. Among splittings $\R^n=\R^{d_1}\times\R^{d_2}$ with symmetry group $O(d_1)\times O(d_2)$, the meridian problem has no potential term exactly for $n\in\{2,4,6\}$ (Remark~\ref{rem:why4}), which are the cases $m\in\{0,1,2\}$ above; in dimension three no such reduction is available, and Theorem~\ref{thm:r3} is proved by a different formulation.

The ball $B_\rho\subset\mathfrak k$ carries the radial solution $u=r^{-m}J_m(r)$ of \eqref{eq:main} exactly when $J_{m+1}(\rho)=0$; in $\mathfrak t$ it corresponds to $F=J_m(r)\sin m\theta$. An $\mathrm{Ad}$-invariant spherical harmonic of degree $\ell$, which exists for $\ell\in m\Z_{\geq0}$, corresponds to the planar mode $J_{m+\ell}(r)\sin(m+\ell)\theta$, and the linearisation of the Schiffer problem at $B_\rho$ in this direction is degenerate exactly when $J_{m+\ell}(\rho)=0$. By Siegel's theorem \cite{Sie29} this never happens for $\ell\geq2$; the case $\ell=1$, which occurs only for $m=1$, corresponds to translations and is excluded by the symmetry $x_1\mapsto-x_1$. So there is no bifurcation from balls in these symmetry classes; our solutions are found instead at near-resonances, where $J_{m+\ell}(\rho)$ is small (Section~\ref{subsec:mechanism} and Table~\ref{tab:resonances}).

\subsection{Strategy of the proof}

\emph{Near-resonance.} Numerically, the centres lie close to Schiffer balls at near-resonant eigenvalues. Let $\rho$ be a zero of $J_{n/2}$ and $\ell\geq2$ the degree of a harmonic in the symmetry class. The linearisation in that direction is degenerate exactly when $J_{\ell+n/2-1}(\rho)=0$, which never happens. In the higher-dimensional classes considered here, a formal reduced equation has a nonzero quadratic coefficient; the amplitude is then of the order of the small divisor. In the plane the quadratic coefficient vanishes, and the cubic term gives an amplitude of the order of the square root of the small divisor. A mode of angular frequency $m$ changes curvature on the scale $m^2$ times its relative amplitude, so planar near-resonance competes directly with convexity. The planar proof uses the full equation, independently of this formal balance; Section~\ref{subsec:mechanism} describes the higher-dimensional mechanism.

\begin{table}[ht]
\centering
\begin{tabular}{@{}ccccc@{}}
\toprule
$n$ & symmetry & $\rho$ & $\ell$ & $J_{\ell+n/2-1}(\rho)$\\
\midrule
$2$ & $D_{182}$ & $j_{1,259}\approx814.45744$ & $182$ & $J_{182}(\rho)\approx1.6\cdot10^{-6}$\\
$3$ & $O(2)\times\Z_2$ & $j_{3/2,34}\approx108.37572$ & $18$ & $J_{37/2}(\rho)\approx1.25\cdot10^{-4}$\\
$4$ & $\mathfrak u(2)$ & $j_{2,4}\approx14.79595$ & $6$ & $J_{7}(\rho)\approx-4.9\cdot10^{-3}$\\
$5$ & $O(4)\times\Z_2$ & $j_{5/2,17}\approx56.49557$ & $12$ & $J_{27/2}(\rho)\approx-6.1\cdot10^{-4}$\\
$6$ & $\mathfrak{so}(4)$ & $j_{3,8}\approx28.90835$ & $8$ & $J_{10}(\rho)\approx-3.0\cdot10^{-3}$\\
$7$ & $O(6)\times\Z_2$ & $j_{7/2,25}\approx83.18007$ & $14$ & $J_{33/2}(\rho)\approx2.3\cdot10^{-4}$\\
$8$ & $\mathfrak{su}(3)$ & $j_{4,5}\approx20.82693$ & $6$ & $J_{9}(\rho)\approx-3.3\cdot10^{-3}$\\
$9$ & $O(8)\times\Z_2$ & $j_{9/2,52}\approx169.58703$ & $20$ & $J_{47/2}(\rho)\approx-1.9\cdot10^{-5}$\\
$10$ & $\mathfrak{so}(5)$ & $j_{5,10}\approx38.15987$ & $8$ & $J_{12}(\rho)\approx-4.4\cdot10^{-4}$\\
$11$ & $O(10)\times\Z_2$ & $j_{11/2,16}\approx57.86006$ & $10$ & $J_{29/2}(\rho)\approx-6.5\cdot10^{-4}$\\
$12$ & $O(11)\times\Z_2$ & $j_{6,11}\approx42.77848$ & $8$ & $J_{13}(\rho)\approx-2.3\cdot10^{-4}$\\
$13$ & $O(12)\times\Z_2$ & $j_{13/2,24}\approx84.57460$ & $12$ & $J_{35/2}(\rho)\approx3.3\cdot10^{-4}$\\
$14$ & $\mathfrak g_2$ & $j_{7,25}\approx88.47436$ & $12$ & $J_{18}(\rho)\approx-9.0\cdot10^{-4}$\\
$15$ & $O(14)\times\Z_2$ & $j_{15/2,7}\approx32.11120$ & $6$ & $J_{25/2}(\rho)\approx-2.0\cdot10^{-3}$\\
$16$ & $O(15)\times\Z_2$ & $j_{8,114}\approx369.83635$ & $28$ & $J_{35}(\rho)\approx-1.9\cdot10^{-6}$\\
$17$ & $O(16)\times\Z_2$ & $j_{17/2,20}\approx74.91722$ & $10$ & $J_{35/2}(\rho)\approx-2.9\cdot10^{-5}$\\
$18$ & $O(17)\times\Z_2$ & $j_{9,3}\approx20.80705$ & $4$ & $J_{12}(\rho)\approx2.7\cdot10^{-3}$\\
$20$ & $O(19)\times\Z_2$ & $j_{10,22}\approx83.43919$ & $10$ & $J_{19}(\rho)\approx1.6\cdot10^{-4}$\\
$21$ & $O(20)\times\Z_2$ & $j_{21/2,31}\approx112.60858$ & $12$ & $J_{43/2}(\rho)\approx-1.5\cdot10^{-4}$\\
\bottomrule
\end{tabular}
\caption{Near-resonances: the Schiffer ball $B_\rho$ of $\R^n$ at eigenvalue one and the degree $\ell$ of the dominant spherical harmonic in the symmetry class. Here $\rho$ is the Schiffer-ball radius; for $n=2$ it exceeds the conformal radius of Theorem~\ref{r2:thm:main} by approximately $3\cdot10^{-5}$. The Bessel values are numerical evaluations.}\label{tab:resonances}
\end{table}

\emph{The plane: an exact inverse.} In the fixed-disc formulation of Colbrook and Stepaniants~\cite{CS26}, the equation is $F(g,p)=g+|p|^2(1+Kg)=0$. We use angularly weighted radial supremum norms. The additional ingredient is an exact two-sided inverse of a residual-corrected linearisation on the entire space: a symmetric-sector Dirichlet Helmholtz inverse gives a normal trace, a holomorphic boundary real-part problem gives the shape variation, and the material-derivative identity gives the field variation. A complete spectral window and the Gram identity for its full complement bound the Dirichlet inverse; analytic boundary tails bound the residual of the exactly clamped centre. The shape-derivative identity is the classical Hadamard mechanism~\cite{SZ92,HP18}, which is not specific to two dimensions; conformal coordinates make the coupled inverse explicit here. In~\cite{CS26} the inverse is approximate, with a finite matrix completed by an explicit map on the tail coefficients.

\emph{Dimensions four, six, eight, ten and fourteen: conformal fixed-disc formulation.} Following \cite{CS26} we write $D=\psi(\D)$ and $V=F\circ\psi$, so that $\Delta V+|\psi'|^2V=0$ in $\D$, and $V-\varpi\circ\psi$ has vanishing value and normal derivative on $\partial\D$. We write $V=Kg+\varpi\circ\psi$, where $K$ is an explicit inverse of the Laplacian with vanishing Cauchy data on its natural range. For $n=4$, where $\varpi=y$, this gives the cubic equation
\begin{equation}\label{eq:intro-cubic}
\Fc(g,\psi)=g+|\psi'|^2\,(Kg+\im\psi)=0,
\end{equation}
and in general the polynomial equation
\begin{equation}\label{eq:intro-poly}
g+|\psi'|^2\,\big(Kg+\varpi\circ\psi\big)=0
\end{equation}
of degree $m+2$ in the unknowns. Compared with the equation $g+|p|^2(1+Kg)=0$ of \cite{CS26}, the constant boundary datum is replaced by $\varpi\circ\psi$, so that the unknown map enters the data as well; the eigenvalue is absorbed into the scale of $\psi$. The symmetry class consists of maps $\psi(w)=\sum_{j\equiv1\,(\mathrm{mod}\ 2m)}c_jw^j$ with real coefficients, which is slightly more symmetric than the Weyl group requires, and $g$ lies in the sine sector of angular frequencies $\equiv m\pmod{2m}$. The part of the linearisation acting on high shape modes has a harmonic part which is a difference operator of step $2m$ (upper bidiagonal for $m=1$), and we invert it exactly. We prove that \eqref{eq:intro-poly} has an exact zero in an explicit ball around a numerically computed centre, using a Newton--Kantorovich (radii polynomial) argument in weighted $\ell^1$ spaces of disk-polynomial coefficients. The approximate inverse consists of a finite matrix, of size $525$, $465$, $200$, $259$ and $610$ for $n=4,6,8,10,14$, inverted and checked in ball arithmetic, together with an exact inverse on the tail. All infinitely many columns of the defect operator are covered: finitely many are evaluated without truncation in interval arithmetic, and the remaining ones are bounded by analytic estimates. From the zero we obtain a univalent map $\psi$, a strictly convex planar domain $D$ (by a lower bound for $\re(1+w\psi''/\psi')$ on the closed disc, uniform over the existence ball), a solution of \eqref{eq:intro-Pm} that extends analytically across $\partial D$, and finally the domain $\Omega$ and the function $u$, by the rotations of Sections~\ref{sec:reduction} and~\ref{sec:r6} for $n=4,6$ and by the adjoint lift of Section~\ref{lie:sec} for $n=8,10,14$.

\emph{Dimension three.} Here there is no reduction to the planar Laplacian (Remark~\ref{rem:why4}), and we pull the three-dimensional problem back to the unit ball $B^3$. Let $\psi$ be a map as in Remark~\ref{rem:domain}, with odd powers and real coefficients, and $\Theta$ its axisymmetric extension. The differential of $\Theta$ is a similarity with factor $|\psi'|$ in the meridian directions and multiplication by $q=\im\psi(w)/\im w$ in the azimuthal direction, where $w=x_1+i|x'|$. Hence, for an axisymmetric function $U=u\circ\Theta$ on $B^3$, problem \eqref{eq:main} becomes
\begin{equation}\label{eq:intro-r3}
\operatorname{div}(q\nabla U)+q|\psi'|^2U=0\ \text{ in }B^3,\qquad U=1\ \text{ and }\ \nabla U=0\ \text{ on }\partial B^3 .
\end{equation}
We write $U=1+K_3g$, where $K_3$ is an explicit inverse of the Laplacian of $\R^3$ with vanishing Cauchy data on $\partial B^3$, acting on the axisymmetric ball polynomials $r^\ell P_\ell(x_1/r)P^{(0,\ell+1/2)}_s(2r^2-1)$ with $\ell$ even and $s\geq1$. Since $q$ is a real-analytic function of $x_1$ and $|x'|^2$ that depends linearly on the coefficients of $\psi$, \eqref{eq:intro-r3} becomes a quartic polynomial equation in the unknown field coefficients and shape coefficients, with no division by the unknown $q$. The disk-polynomial algebra of \cite{CS26} is replaced by weighted $\ell^1$ norms adapted to the Legendre--Chebyshev connection coefficients; multiplication estimates follow from the positivity of the Gaunt coefficients and of Jacobi connection coefficients. The approximate inverse consists of a finite block of dimension $7371$, with $7280$ field and $91$ shape unknowns, and an explicit inverse of the principal boundary part on high shape modes, given by a convergent Neumann series. The Newton--Kantorovich argument again covers all columns, finitely many by interval evaluation and all others by analytic estimates. Univalence and positivity of $q$ follow from a lower bound for $\re\psi'$ on a disc of radius $101/100$, uniform over the existence ball, which also shows that $\Theta$ is a real-analytic diffeomorphism near $\overline{B^3}$; strict star-shapedness follows from a lower bound for $\re(w\psi'/\psi)$ on the unit circle, and strict convexity from a positive lower bound for $\re(1+w\psi''/\psi')$ on the whole closed disc and throughout the projected shape ball. The field $u=U\circ\Theta^{-1}$ extends analytically across the axis and across $\partial\Omega$ by elliptic regularity.

\emph{Dimension five.} The same pull-back works in $\R^5$ with $q^3$ in place of $q$: for an axisymmetric function $U=u\circ\Theta$ on $B^5$, problem \eqref{eq:main} becomes $\operatorname{div}(q^3\nabla U)+q^3|\psi'|^2U=0$ with $U=1$ and $\nabla U=0$ on $\partial B^5$, and after removal of the positive factor $q^2$ the equation is again a quartic polynomial equation in the field and shape coefficients. The new difficulty is functional-analytic: for $n\geq5$ the Jacobi coefficient norm used for $n=3$ does not control $K_ng$ in $C^1$, since the gradients of $K_n$ applied to normalised ball polynomials of radial degree $s$ grow like $s^{(n-4)/2}$. We therefore grade all coefficient norms by the weight $1+D/10^6$ of the polynomial degree $D$. In these norms $K_5$ is bounded into $C^1$ with a large but finite constant, which is used only qualitatively; the multiplication constants are unchanged, and the stencil, inverse and tail bounds acquire factors of at most $1.001$. The approximate inverse has a finite block of dimension $4698$, and a new analytic bound covers all shape columns beyond $j=599$. The resulting domain is convex, and its non-ball witness is the coefficient $c_3$ (Section~\ref{sec:r5}).

\emph{Dimension seven.} The same axisymmetric construction uses $q^5$ in the divergence equation and the squared weight $(1+D/10^6)^2$. The finite block has dimension $9010$; Section~\ref{r7:sec} gives the seven-dimensional stencils and all field and shape tail estimates.

\emph{The exact inverse in general dimension.}
Part~\ref{gd:part} applies the material-derivative construction of the planar proof to the axisymmetric quartic equation on $B^n$.
The field norm sums analytically weighted Gegenbauer angular coefficients with a supremum norm in the radial variable, replacing the graded Jacobi coefficient norms of Sections~\ref{r5:sec} and~\ref{r7:sec}.
A complete Bessel spectral window and the Gram identity for its full complement certify a symmetric-sector Dirichlet inverse; heat-kernel smoothing of the low angular degrees and the radial maximum principle on the high degrees transfer it to the working norm.
The Poisson normal trace gains the derivative needed by the holomorphic boundary real-part problem, and the material identity gives an exact two-sided inverse of the residual-corrected linearisation, including its regular axis term.
An exactly clamped centre, quartic remainder bounds and a radii polynomial then give the domains of Theorem~\ref{gd:thm:main}, with univalence, both principal-curvature estimates and a nonzero shape coefficient holding on the entire existence ball.

\subsection*{Relation with earlier versions}
Version~1 of this paper proved Theorem~\ref{thm:main} in dimensions $3,4,6,8,10,14$. This version adds convex counterexamples in dimensions $2$, $5$, $7$, $9$, $11$--$13$, $15$--$18$, $20$ and $21$, a second proof of the five-dimensional case by the exact-inverse method, and the first convex planar counterexample to Berenstein's conjecture (Theorem~\ref{br:thm:main}). The earlier preprint and its verification files are archived in~\cite{Guo26}.

\subsection*{Organisation}
Part~\ref{r2:part} treats the plane first: Section~\ref{r2:sec:inverse} gives the spaces and Dirichlet inverse, Section~\ref{r2:sec:coupled} constructs the exactly clamped centre and the coupled inverse, and Section~\ref{r2:sec:existence} proves existence and geometry. Part~\ref{br:part} treats Berenstein's problem: Sections~\ref{br:sec:formulation}--\ref{br:sec:transfer} give the fixed-disc equation and Robin inverse, Section~\ref{br:sec:inverse} constructs the residual-corrected coupled inverse, and Section~\ref{br:sec:existence} proves existence, geometry and sign change. Part~\ref{part:planar} gives the Cartan-plane reductions: Sections~\ref{sec:reduction}--\ref{sec:recon} prove the four-dimensional case, Section~\ref{sec:r6} the six-dimensional case, and Section~\ref{lie:sec} the cases $n=8,10,14$. Part~\ref{part:r3} proves the axisymmetric cases $n=3,5,7$ in Sections~\ref{sec:r3}, \ref{sec:r5} and~\ref{r7:sec}, retaining the additional nonconvex three-dimensional example. Part~\ref{gd:part} proves Theorem~\ref{gd:thm:main} by one dimension-parametrised exact-inverse argument. Part~\ref{part:further} contains the Pompeiu implication, the computer-assisted verification, the mechanism and further examples, the comparison with rigidity results, and the open problems.

\part{The plane}\label{r2:part}
\section{Function spaces and the Dirichlet inverse}\label{r2:sec:inverse}
In this part $m=182$, $p=\psi'$, and $W$ is the holomorphic shape space; $Y$ is the real invariant field space and $K$ is the zero-Dirichlet Poisson inverse, restricted to its clamped range when acting on $X_g$. These symbols are local: in later parts $W$ may denote a Weyl group, $Y$ a residual bound, and $K$ a compatible inverse in a coefficient space. In the spectral subsection $\rho$ is the radius of the reference shape; thereafter it is the radius of the forty-mode centre, rather than the angular coefficient weight used in the Cartan reductions.

\subsection{Conformal formulation and Banach spaces}\label{r2:sec:spaces}

\subsubsection{Pull-back to the disc}
Identify $\R^2$ with $\C$. For a holomorphic univalent map $\psi$ with derivative $p$, conformal covariance of the Laplacian gives
\begin{equation}\label{r2:eq:covariance}
 \Delta_w(u\circ\psi)=|p|^2(\Delta_z u)\circ\psi.
\end{equation}
Along the unit circle the vector $w p(w)$ points in the outward normal direction to the image boundary and has length $|p(w)|$. Hence, for $U=u\circ\psi$,
\[
 U=1,\quad\partial_rU=0\quad\text{on }\partial\D
 \quad\Longleftrightarrow\quad
 u=1,\quad\partial_\nu u=0\quad\text{on }\partial\Omega.
\]
The pulled-back equation is $\Delta U+|p|^2U=0$. We encode both boundary conditions in the field space before solving the equation.

Throughout the construction set
\begin{equation}\label{r2:eq:weight}
 m=182,\qquad R=2^{1/m},\qquad R^{jm}=2^j.
\end{equation}
For a continuous function on $\overline\D$ with angular coefficients $f_n(r)$, define
\begin{equation}\label{r2:eq:Ycal}
 f(r,\theta)=\sum_{n\in\Z}f_n(r)e^{in\theta},\qquad
 \nrm{f}_{\Yc}=\sum_{n\in\Z}R^{|n|}\sup_{0\le r\le1}|f_n(r)|.
\end{equation}
The coefficients must have the origin compatibility of a continuous Cartesian function. In particular, $f_n(0)=0$ for $n\ne0$. The space $\Yc$ consists of those functions for which~\eqref{r2:eq:Ycal} is finite. Let $Y$ be its real $D_m$-invariant subspace. Its only frequencies are $n\in m\Z$, and its coefficients satisfy $f_{-n}=f_n\in\R$.

\begin{lemma}\label{r2:lem:algebra}
The spaces $\Yc$ and $Y$ are Banach spaces, and
\begin{equation}\label{r2:eq:algebra}
 \nrm{fh}_{\Yc}\le\nrm{f}_{\Yc}\nrm{h}_{\Yc}.
\end{equation}
The boundary trace has norm at most one into the Laurent coefficient space with norm $\sum_nR^{|n|}|f_n(1)|$.
\end{lemma}
\begin{proof}
The weighted sum in~\eqref{r2:eq:Ycal} gives uniform convergence of the angular series. A Cauchy sequence converges in each radial coefficient in $C([0,1])$ and in the weighted sum of these coefficient norms. The limit is uniform on the closed disc, so it preserves Cartesian continuity and origin compatibility. The reality and symmetry conditions are closed. For products, absolute convergence gives
\[
 (fh)_n=\sum_k f_kh_{n-k},\qquad
 R^{|n|}\le R^{|k|}R^{|n-k|}.
\]
Taking radial suprema and summing proves~\eqref{r2:eq:algebra}. The trace estimate follows coefficient by coefficient.
\end{proof}

The shape derivative belongs to the real-coefficient holomorphic space
\begin{equation}\label{r2:eq:W}
 W=\left\{p(w)=\sum_{j\ge0}p_jw^{jm}:p_j\in\R,
 \quad\nrm{p}_W=\sum_{j\ge0}2^j|p_j|<\infty\right\}.
\end{equation}
Every $p\in W$ is holomorphic on $|w|<R$ and continuous on its closure. As an element of $\Yc$ on the unit disc it has norm exactly $\nrm{p}_W$. We normalise its primitive by
\begin{equation}\label{r2:eq:primitive}
 \psi(w)=\int_0^w p(z)\,dz
 =\sum_{j\ge0}c_jw^{jm+1},\qquad
 c_j=\frac{p_j}{jm+1}.
\end{equation}
The shape space fixes the origin, the orientation, and the reflection symmetry.

We also need one angular derivative of boundary data. For a Laurent series set
\begin{equation}\label{r2:eq:Wone}
 \nrm{h}_{W_1}=\sum_{n\in\Z}R^{|n|}(1+|n|)|h_n|.
\end{equation}
The holomorphic version uses only nonnegative indices. Both retain the analytic weight. If $W_0$ denotes the same coefficient norm without $1+|n|$, then
\[
 \nrm{h}_{W_1}=\nrm{h}_{W_0}+\nrm{\partial_\theta h}_{W_0}.
\]
The inequalities $R^{|n+k|}\le R^{|n|}R^{|k|}$ and $1+|n+k|\le(1+|n|)(1+|k|)$ show that $W_1$ is a Banach algebra. We use the same notation for a holomorphic series and its boundary values when the norm is identical.

\subsubsection{The clamped range of the Laplacian}
Let $K$ denote the zero-Dirichlet inverse of $\Delta$ on the disc. Define the closed real subspace
\begin{equation}\label{r2:eq:Xg}
 X_g=\left\{g\in Y:
 \int_0^1 r^{|n|+1}g_n(r)\,dr=0
 \quad(n\in m\Z)\right\}.
\end{equation}
Each moment is a bounded functional, so this common kernel is closed.

\begin{lemma}\label{r2:lem:Poisson}
For $g\in Y$, the zero-Dirichlet solution $v=Kg$ satisfies
\begin{equation}\label{r2:eq:Poisson}
 \nrm{Kg}_{\Yc}\le\frac14\nrm{g}_Y,\qquad
 \partial_rv_n(1)=\int_0^1r^{|n|+1}g_n(r)\,dr.
\end{equation}
Consequently $g\in X_g$ if and only if $Kg$ has zero Neumann trace.
\end{lemma}
\begin{proof}
For $n\ge0$, the positive Green kernel of $-\Delta_n$, with radial measure $s\,ds$, is
\[
 G_n(r,s)=\frac{(a/b)^n-(ab)^n}{2n}\quad(n>0),\qquad
 G_0(r,s)=-\log b,
 \quad a=\min(r,s),\ b=\max(r,s).
\]
The formula at $r=0$ is understood by continuity. Since
\[
 0\le G_n(r,s)\le\frac{1-b^{2n}}{2n}\le-\log b,
 \qquad\int_0^1G_0(r,s)s\,ds=\frac{1-r^2}{4},
\]
the supremum norm of each radial inverse is at most $1/4$. Summing with the angular weights proves the norm bound. The $n$th radial equation is
\[
 v_n''+\frac1r v_n'-\frac{n^2}{r^2}v_n=g_n.
\]
After multiplication by $r^{|n|+1}$ the left side is the derivative of $r^{|n|+1}v_n'-|n|r^{|n|}v_n$. This function tends to zero as $r\to0$ for the regular solution, and $v_n(1)=0$. Integration over $[0,1]$ therefore gives
\[
 \int_0^1r^{|n|+1}g_n(r)\,dr
 =v_n'(1)-|n|v_n(1)=v_n'(1).
\]
Continuous right sides give $v\in W^{2,s}(\D)$ for every finite $s$, hence $v\in C^1(\overline\D)$ by taking $s>2$. Consequently $v_n'(1)$ is the $n$th Fourier coefficient of the Neumann trace $\partial_rv|_{r=1}$, and the identity just proved gives its Fourier expansion. Since $|v_n'(1)|\le\sup|g_n|/(|n|+2)$, the trace series is absolutely convergent. The Neumann trace therefore vanishes if and only if every moment in~\eqref{r2:eq:Xg} vanishes, which is the last assertion.
\end{proof}

\subsubsection{The cubic equation}
We work on real Banach spaces
\begin{equation}\label{r2:eq:X}
 X=X_g\times W,\qquad
 \nrm{(h,k)}_X=\nrm{h}_Y+\nrm{k}_W,
 \qquad F\colon X\longrightarrow Y,
 \quad F(g,p)=g+|p|^2(1+Kg).
\end{equation}
The algebra estimate makes $F$ a continuous cubic polynomial map. For any zero set $U=1+Kg$. Lemma~\ref{r2:lem:Poisson} gives its exact boundary data, while $F=0$ gives
\begin{equation}\label{r2:eq:fixed-disc}
 \Delta U+|p|^2U=0\quad\text{in }\D,
 \qquad U=1,\quad\partial_rU=0\quad\text{on }\partial\D.
\end{equation}
Univalence and strict convexity are proved for every shape in the final ball, after existence has been established in $X$.

We distinguish two fixed shapes. The eight-mode reference $p_{\rm ref}$ is used only to compute the Dirichlet spectral inverse in Subsection~\ref{r2:sec:spectrum}. The forty-mode shape $p_0$ defines the clamped centre and every subsequent nonlinear constant. Their proximity transfers the reference inverse to $p_0$ at the end of Subsection~\ref{r2:sec:spectrum}. The exact shape $p$ supplied by the contraction is a third object. The finite arrays specify the first two objects, while the contraction specifies the last.

\subsection{The Dirichlet operator in the symmetric sector}\label{r2:sec:spectrum}

This section constructs an inverse of $L_{\rm ref}=\Delta+|p_{\rm ref}|^2$ in the real even $D_m$ sector and transfers it to $L=\Delta+|p_0|^2$. In the spectral construction only, abbreviate
\[
 p=p_{\rm ref},\qquad \rho=p_{{\rm ref},0},\qquad
 q=|p|^2,\qquad V=q-\rho^2.
\]
The $L^2$ measure is $r\,dr\,d\theta/(2\pi)$. All projections and operators in this section act inside that invariant sector.

\subsubsection{The spectral window}
Let $a_{j,k}$ denote the positive zeros of $J_{jm}$ in increasing order. An orthonormal Dirichlet basis is
\begin{equation}\label{r2:eq:basis}
 b_{j,k}(r)\chi_j(\theta),\qquad
 b_{j,k}(r)=\frac{\sqrt2 J_{jm}(a_{j,k}r)}{|J_{jm+1}(a_{j,k})|},
 \quad \chi_0=1,\quad\chi_j=\sqrt2\cos(jm\theta)\ (j>0).
\end{equation}
Indeed, the angular factors are orthonormal, and integration of the Bessel equation and its parameter derivative gives
\[
 \int_0^1rJ_n(ar)^2\,dr=\frac{J_{n+1}(a)^2}{2}\quad(J_n(a)=0).
\]
The usual Dirichlet spectral decomposition on the disc establishes completeness. The potential $V$ preserves this sector.

Let $\PP$ be the projection onto all basis states with
\begin{equation}\label{r2:eq:window}
 |a_{j,k}-\rho|<160,
\end{equation}
and set $\QQ=I-\PP$. Complete zero enumeration gives the counts in Table~\ref{r2:tab:zeros}; their sum is $408$.
\begin{table}[htbp]
\centering
\begin{tabular}{c rrrrrrr}
\toprule
Angular order & $0$ & $182$ & $364$ & $546$ & $728$ & $910$ & $1092$\\
States in the window & $102$ & $99$ & $91$ & $74$ & $37$ & $5$ & $0$\\
\bottomrule
\end{tabular}
\caption{The complete symmetric-sector spectral window. All higher angular orders have no state in the window.}\label{r2:tab:zeros}
\end{table}

\begin{lemma}\label{r2:lem:zero-coverage}
The interval sign and spacing tests described below enumerate every state of~\eqref{r2:eq:window}. The independent quarter-step sign grid gives the same enumeration without using the root proposals.
\end{lemma}
\begin{proof}
For integer $n\ge0$, put $y(x)=\sqrt{x}J_n(x)$. The Bessel equation becomes
\begin{equation}\label{r2:eq:Sturm}
 y''+Q_n(x)y=0,\qquad Q_n(x)=1-\frac{n^2-1/4}{x^2}.
\end{equation}
If $y$ vanishes at $a<b$ and $Q_n\le Q_{\max}$ on $[a,b]$, integration by parts and Wirtinger's inequality imply
\[
 \frac{\pi^2}{(b-a)^2}\int_a^b y^2
 \le\int_a^b(y')^2=\int_a^bQ_ny^2
 \le Q_{\max}\int_a^b y^2.
\]
Hence two distinct zeros of $y$ in an interval on which $Q_n\le Q_{\max}$ are separated by at least $\pi/\sqrt{Q_{\max}}$. Since $Q_n(x)\le5/4$ for $x\ge1$ and every integer $n\ge0$, zeros in $x\ge1$ are separated by more than $2.8$. Each proposed root is enclosed by opposite signs at exact offsets $\pm2^{-220}$ from its exact midpoint. Its left endpoint is checked to exceed $1$ and its width to be less than $1/10$, so the bracket contains exactly one root. For successive brackets, the distance between the outer endpoints is checked to be less than $2\pi/\sqrt{Q_{\max}}$, where $Q_{\max}$ is the maximum of $Q_n$ between those endpoints. A further root between the two enclosed roots would require the opposite inequality. Guard roots on both sides cover the endpoints of the window.

There is a separate initial check when the angular order lies above the lower endpoint of the window. For $n>0$, the positive leading power of $J_n$ gives $J_n,J_n'>0$ near zero. A first zero of $J_n'$ in $(0,n)$ while $J_n>0$ would satisfy
\[
 J_n''=(n^2/x^2-1)J_n>0,
\]
which is incompatible with a first crossing from positive derivative. A first zero of $J_n$ is then also impossible. The limiting values at $n$ give $y(n),y'(n)>0$. Comparison in~\eqref{r2:eq:Sturm} with the constant coefficient $Q_{\max}$ places the first subsequent zero at distance at least $\pi/(2\sqrt{Q_{\max}})$ from $n$. A hidden first zero before the proposed one would place the latter at distance at least $3\pi/(2\sqrt{Q_{\max}})$; the reverse strict inequality is checked. For $n\ge1092>\rho+160$, positivity up to $n$ excludes the entire window.

For the second enumeration, use again the universal bound
\[
 Q_n(x)\le\frac54\quad(x\ge1),\qquad
 \frac{\pi}{\sqrt{5/4}}>2.8.
\]
Every cell of length $1/4$ in $x\ge1$ therefore contains at most one zero. Positive zeros are simple by uniqueness for the Bessel ODE. Nonzero endpoint signs determine whether the cell contains one zero or none. Evaluate the signs at the $1281$ exact points
\begin{equation}\label{r2:eq:sign-grid}
 x_k=\rho-160+k/4,\qquad 0\le k\le1280,
\end{equation}
for $n=0,182,364,546,728,910,1092$. Increasing precision until the sign is resolved yields Table~\ref{r2:tab:zeros}. The positivity argument excludes all remaining orders. Finally, each root bracket from the first method must lie in a distinct sign-changing cell of the second method. Equality of the reconstructed state lists establishes agreement at the level of individual roots, as well as total and sector counts.
\end{proof}

\subsubsection{Matrix entries and certified quadrature}
On the $408$-dimensional range of $\PP$, form
\begin{equation}\label{r2:eq:matrices}
 B=\PP V\PP,\qquad B_2=\PP V^2\PP,\qquad
 H=\diag(a_{j,k}^2-\rho^2)-B.
\end{equation}
These are exact matrices, evaluated through enclosures. If $p(w)=\sum p_jw^{jm}$ is the eight-mode reference, write
\begin{align*}
 q_h(r)&=\sum_{k\ge0}p_{k+h}p_kr^{(2k+h)m}\quad(h>0),\\
 q_0(r)&=\sum_{k>0}p_k^2r^{2km}.
\end{align*}
Then $V=q_0+\sum_{h>0}2q_h\cos(hm\theta)$; absent coefficients are zero. The Fourier coefficients of $V^2$ are the full finite convolution over both positive and negative indices. With the angular normalisation in~\eqref{r2:eq:basis}, the angular integrals of multiplication by $V$ are
\[
 q_0\quad(0,0),\qquad
 \sqrt2q_j\quad(0,j),\qquad
 q_{|i-j|}+q_{i+j}\quad(i,j>0).
\]
The same formulas applied to the convolved coefficients give the entries of $B_2$. Thus the remaining integrations are one-dimensional radial integrals.

Use $N_G=1536$ Gauss--Legendre nodes on $[-1,1]$. Each is enclosed as a sign-changing bracket of the Legendre polynomial $P_{N_G}$; the disjoint brackets and the degree prove that all nodes are covered. At a true node $x_i$, the weight is
\[
 w_i=\frac{2}{(1-x_i^2)P_{N_G}'(x_i)^2}.
\]
The identities
\[
 (x^2-1)P_{N_G}'=N_G(xP_{N_G}-P_{N_G-1}),\qquad
 \max_{[-1,1]}|P_{N_G}''|
 =\frac{(N_G-1)N_G(N_G+1)(N_G+2)}8
\]
bound evaluation at the unknown true root from evaluation at its enclosing midpoint. The transformed radial node is $r_i=(1+x_i)/2$, with radial measure factor $w_i r_i/2$.

The Bessel values at exact midpoint arguments are enclosed at $2048$ bits. For real arguments and integer order, the integral representation implies $|J_n'|\le1$. It therefore bounds the additional uncertainty in arguments arising from enclosed roots and nodes. A binary64 value exported for matrix assembly is interpreted as an exact dyadic number, accompanied by a verified row error; the resulting matrix enclosure includes that error and the quadrature error.

For completeness, the quadrature majorant is explicit. On the Bernstein ellipse of parameter $2$ in $x=2r-1$,
\[
 |r|\le9/8,\qquad |\im r|\le3/8,
 \qquad |J_n(ar)|\le e^{a|\im r|}.
\]
Set
\[
 P_e=\sum_j|p_j|(9/8)^{jm},\qquad V_e=P_e^2+\rho^2,
 \qquad b_* =\min|J_{jm+1}(a_{j,k})|,\qquad a_{\max}=\max a_{j,k},
\]
where the minimum and the maximum are taken over the $408$ states of the window~\eqref{r2:eq:window}; in particular $a_{\max}<\rho+160$. The common bound
\begin{equation}\label{r2:eq:ellipse-majorant}
 M_e=8(1+9/8)(1+V_e^2)e^{3a_{\max}/4}/b_*^2
\end{equation}
dominates every integrand for $B$ and $B_2$, including the radial measure, both basis normalisations and angular factors. Its analytic Chebyshev coefficients have absolute value at most $2M_e2^{-k}$. Truncate at degree $2N_G-1$. The Gauss rule is exact on that polynomial, and both the integral and the positive quadrature rule have total mass two. The error is therefore at most
\begin{equation}\label{r2:eq:Gauss-error}
 4\sum_{k\ge2N_G}2M_e2^{-k}=16M_e2^{-2N_G}.
\end{equation}
This error is added to every matrix entry, before any inverse estimate is made.

\subsubsection{The entire complement}
Let $s=\sum_{j>0}|p_j|$ and define
\begin{equation}\label{r2:eq:complement-gap}
 v_* =2\rho s+s^2,\qquad
 g_*=160(2\rho-160),\qquad\gamma=g_*-v_*.
\end{equation}
Here $\nrm{V}_\infty\le v_*$. On the complementary Dirichlet operator domain set
\[
 D_Q=\QQ(-\Delta-\rho^2)\QQ,\qquad
 M_Q=D_Q-\QQ V\QQ.
\]
The operator $D_Q$ is generally indefinite. Its spectrum nonetheless has a two-sided gap of size $g_*$. For a root $a\le\rho-160$, put $t=\rho-a$; then $t\in[160,\rho)$ and
\[
 |a^2-\rho^2|=t(2\rho-t)\ge160(2\rho-160).
\]
For $a\ge\rho+160$ the corresponding bound is $160(2\rho+160)$, which is larger. Thus $D_Q^{-1}$ is bounded with norm at most $1/g_*$. If $\gamma>0$, the factorization
\begin{equation}\label{r2:eq:unbounded-factor}
 M_Q=(I-\QQ V\QQ D_Q^{-1})D_Q
\end{equation}
on the domain of $D_Q$ has an invertible bounded first factor, by its Neumann series. Consequently
\begin{equation}\label{r2:eq:M-inverse}
 \nrm{M_Q^{-1}}\le1/\gamma.
\end{equation}
This argument uses a gap in absolute spectral value, not positivity of $M_Q$.

The following identity is the reason no further complementary radial enumeration is required:
\begin{equation}\label{r2:eq:Gram}
 G:=B_2-B^2
  =\PP V(\PP+\QQ)V\PP-(\PP V\PP)^2
  =(\QQ V\PP)^*(\QQ V\PP)\succeq0.
\end{equation}
It holds for the full orthogonal complement $\QQ$. In particular, positivity is an analytic property of the exact matrix $G$, even when its numerical enclosure has small entrywise uncertainty.

\begin{proposition}[Schur complement estimate]\label{r2:prop:Schur}
Assume that $\gamma>0$ in~\eqref{r2:eq:complement-gap}. Let $C$ be a real $408\times408$ matrix and $D$ a real diagonal matrix of the same size with nonzero diagonal entries $D_i$, and suppose that the numbers $\eps_C$, $\eps$ and $\eta$ satisfy
\begin{align}
 \nrm{C^TC-I}_\infty&\le\eps_C<1,\label{r2:eq:C-defect}\\
 \nrm{|D|^{-1/2}(C^THC-D)|D|^{-1/2}}_\infty&\le\eps,\label{r2:eq:H-defect}\\
 \frac{\tr(|D|^{-1}C^TGC)}{\gamma}&\le\eta,
 \qquad\eps+\eta<1.\label{r2:eq:eta}
\end{align}
Then $L_{\rm ref}$, with its Dirichlet domain in the symmetric sector, has a bounded inverse on $L^2$, and
\begin{align}
 \nrm{L_{\rm ref}^{-1}}_{L^2\to L^2}
 &\le B_{\rm ref}:=\left(1+(v_*/\gamma)^2\right)S_*+1/\gamma,
 \label{r2:eq:BLref}\\
 S_*&:=\frac{1+\eps_C}{\min_i|D_i|(1-\eps-\eta)}.\label{r2:eq:S-bound}
\end{align}
\end{proposition}
\begin{proof}
\emph{The block form.}
The range of $\PP$ is spanned by finitely many Dirichlet eigenfunctions, so it lies in the operator domain of $L_{\rm ref}$, and $\PP$ commutes with $\Delta$ on that domain. Consequently $\QQ$ maps the operator domain into itself, and the operator domain is the direct sum of the range of $\PP$ and the domain of $D_Q$. For $u=u_P+u_Q$ in this decomposition, $\PP(-\Delta-\rho^2)u_Q=0$ and $\QQ(-\Delta-\rho^2)u_P=0$. Hence
\[
 -L_{\rm ref}=\begin{pmatrix}H&b^*\\b&M_Q\end{pmatrix},
 \qquad b=-\QQ V\PP,\quad b^*=-\PP V\QQ,
\]
where $\nrm b\le\nrm V_\infty\le v_*$. The operator $D_Q$ is self-adjoint and diagonal in the basis~\eqref{r2:eq:basis}, and $\QQ V\QQ$ is bounded and symmetric. Thus $M_Q$ is self-adjoint on the domain of $D_Q$, and its inverse, which exists by~\eqref{r2:eq:unbounded-factor}, is bounded and self-adjoint. The Schur complement
\[
 S=H-b^*M_Q^{-1}b
\]
is a real symmetric matrix on the range of $\PP$.

\emph{The Schur complement.}
Both defect matrices in~\eqref{r2:eq:C-defect} and~\eqref{r2:eq:H-defect} are symmetric. Their $\ell^2$ operator norms are therefore bounded by their maximum row sums. The first defect makes $C^TC$, and hence the square matrix $C$, invertible, and gives $\nrm C^2\le1+\eps_C$. Set $X_Q=bC|D|^{-1/2}$. By~\eqref{r2:eq:Gram}, $X_Q^*X_Q=|D|^{-1/2}C^TGC|D|^{-1/2}$ is positive semidefinite, so its norm is at most its trace $\tr(|D|^{-1}C^TGC)$. With~\eqref{r2:eq:M-inverse} and~\eqref{r2:eq:eta},
\[
 \nrm{X_Q^*M_Q^{-1}X_Q}
 \le\gamma^{-1}\nrm{X_Q^*X_Q}
 \le\eta.
\]
The normalised Schur matrix is
\[
 |D|^{-1/2}C^TSC|D|^{-1/2}
 =\operatorname{sgn}(D)
 +|D|^{-1/2}(C^THC-D)|D|^{-1/2}-X_Q^*M_Q^{-1}X_Q,
\]
that is, $\operatorname{sgn}(D)$ plus a perturbation of norm at most $\eps+\eta<1$. Since $\operatorname{sgn}(D)$ is an isometry, the normalised Schur matrix is invertible and its inverse has norm at most $(1-\eps-\eta)^{-1}$. Undoing the congruence,
\[
 S^{-1}=C|D|^{-1/2}\bigl(|D|^{-1/2}C^TSC|D|^{-1/2}\bigr)^{-1}|D|^{-1/2}C^T,
 \qquad\nrm{S^{-1}}\le S_*,
\]
with $S_*$ as in~\eqref{r2:eq:S-bound}. The signs of $D$ are retained throughout; no definiteness of $H$ or of $S$ is used.

\emph{The block inverse.}
For $f=f_P+f_Q$ in $L^2$, put
\[
 u_P=S^{-1}\bigl(f_P-b^*M_Q^{-1}f_Q\bigr),\qquad
 u_Q=M_Q^{-1}(f_Q-bu_P).
\]
Then $u_Q$ lies in the domain of $D_Q$, so $u=u_P+u_Q$ lies in the operator domain, and
\[
 Hu_P+b^*u_Q=Su_P+b^*M_Q^{-1}f_Q=f_P,\qquad bu_P+M_Qu_Q=f_Q.
\]
Thus $-L_{\rm ref}u=f$. Conversely, if $-L_{\rm ref}u=0$ for some $u$ in the operator domain, the second block row gives $u_Q=-M_Q^{-1}bu_P$, the first gives $Su_P=0$, and therefore $u=0$. Hence $-L_{\rm ref}$ is a bijection of its domain onto $L^2$, with the bounded inverse
\[
 (-L_{\rm ref})^{-1}
 =\begin{pmatrix}I\\-M_Q^{-1}b\end{pmatrix}
 S^{-1}\begin{pmatrix}I&-b^*M_Q^{-1}\end{pmatrix}
 +\begin{pmatrix}0&0\\0&M_Q^{-1}\end{pmatrix}.
\]
Since $M_Q^{-1}$ is self-adjoint, the row operator is the adjoint of the column. The squared norm of the column is at most $1+(v_*/\gamma)^2$, because $\nrm{M_Q^{-1}b}\le\nrm b/\gamma\le v_*/\gamma$. Together with $\nrm{S^{-1}}\le S_*$ and~\eqref{r2:eq:M-inverse}, this gives~\eqref{r2:eq:BLref}.
\end{proof}

For the eight-mode reference, $C$ and $D$ are binary64 proposals read as exact dyadic matrices, and the enclosed calculation gives
\begin{equation}\label{r2:eq:spectral-numbers}
 \begin{gathered}
 v_*<5004.03,\qquad\gamma>230022.34,\qquad\min_i|D_i|>0.1953,\\
 \eps_C<10^{-12},\qquad\eps<10^{-9},\qquad\eta<0.240915,
 \end{gathered}
\end{equation}
and the quadrature error~\eqref{r2:eq:Gauss-error} is below $1.35\cdot10^{-360}$. The hypotheses of Proposition~\ref{r2:prop:Schur} therefore hold, and $B_{\rm ref}<6.748$. Floating-point root and congruence proposals enter this conclusion only through the root coverage of Lemma~\ref{r2:lem:zero-coverage}, the enclosed matrix entries and the defects~\eqref{r2:eq:C-defect}--\eqref{r2:eq:eta}.

\subsubsection{Transfer to the forty-mode centre}
Return now to $p_0$. The pointwise difference on $\overline\D$ satisfies
\begin{equation}\label{r2:eq:transfer-difference}
 \nrm{|p_0|^2-|p_{\rm ref}|^2}_\infty\le\delta_q,
 \qquad
 \delta_q=(\nrm{p_0}_\infty+\nrm{p_{\rm ref}}_\infty)
 \nrm{p_0-p_{\rm ref}}_\infty.
\end{equation}
Bounding the three supremum norms by unweighted sums of coefficients gives $\delta_q<8.376\cdot10^{-14}$. On the Dirichlet domain,
\[
 \Delta+|p_0|^2
 =\Bigl(I+\bigl(|p_0|^2-|p_{\rm ref}|^2\bigr)L_{\rm ref}^{-1}\Bigr)L_{\rm ref},
\]
and the first factor is invertible by its Neumann series because $B_{\rm ref}\delta_q<1$. Hence
\begin{equation}\label{r2:eq:transfer}
 \nrm{(\Delta+|p_0|^2)^{-1}}_{L^2\to L^2}\le
 B_{L^2}:=\frac{B_{\rm ref}}{1-B_{\rm ref}\delta_q}<6.748.
\end{equation}
All constants from Subsection~\ref{r2:sec:analytic} onward use $p_0$ and its conformal radius $\rho=p_0(0)$.

\begin{remark}
The inverse in~\eqref{r2:eq:transfer} is confined to the symmetric sector. For a genuine Schiffer solution, physical derivatives of $u$ solve a homogeneous Dirichlet problem at the same eigenvalue. Their angular sectors are different from the invariant even sector, so they do not contradict the inverse proved here.
\end{remark}

\subsection{From the spectral inverse to the analytic norm}\label{r2:sec:analytic}

In this subsection and Section~\ref{r2:sec:coupled}, put $p=p_0$, $q=|p_0|^2$, $\rho=p_0(0)$, and $L=\Delta+q$; write $p(w)=\sum_{j\ge0}p_jw^{jm}$. Define
\begin{equation}\label{r2:eq:P-bounds}
 P=\sum_{j\ge0}2^j|p_j|,\qquad
 P_1=\sum_{j\ge0}2^j jm|p_j|,\qquad V_*=P^2-\rho^2.
\end{equation}
The finite centre makes these sums finite. The product estimate shows
\begin{equation}\label{r2:eq:V-analytic}
 \nrm{q}_Y\le P^2,
 \qquad\nrm{q-\rho^2}_Y\le V_*.
\end{equation}
For the second bound, expand $p\bar p$ and remove the $(j,k)=(0,0)$ term; $2^{|j-k|}\le2^{j+k}$ bounds the remaining terms.

\begin{proposition}\label{r2:prop:Yinverse}
Suppose the symmetric-sector Dirichlet inverse has $L^2$ bound $B_{L^2}$. Let $h$ be the least positive integer with $hm>\rho$, and set
\begin{equation}\label{r2:eq:high-constants}
 d_h=(hm)^2-\rho^2,\qquad\beta=V_*/d_h,
 \qquad W_\ell=1+2\sum_{j=1}^{h-1}2^j.
\end{equation}
If $d_h>0$ and $0\le\beta<1$, then the zero-Dirichlet solution of $Lv=f$, for $f\in Y$, belongs to $Y$ and satisfies
\begin{equation}\label{r2:eq:BY}
 \nrm v_Y\le B_Y\nrm f_Y,
 \qquad B_Y=\frac{W_\ell(1+P^2B_{L^2})/(2\sqrt2)+1/d_h}{1-\beta}.
\end{equation}
Its normal trace $N=\partial_rv|_{r=1}$ belongs to $W_1$, with
\begin{equation}\label{r2:eq:BN}
 \nrm N_{W_1}\le B_N\nrm f_Y,
 \qquad B_N=1+P^2B_Y.
\end{equation}
\end{proposition}
\begin{proof}
Parseval's identity and the radial measure give
\[
 \nrm f_{L^2}^2=\sum_n\int_0^1|f_n(r)|^2r\,dr
 \le\frac12\sum_n\sup_r|f_n(r)|^2\le\frac12\nrm f_Y^2.
\]
The spectral inverse therefore gives $\nrm v_{L^2}\le B_{L^2}\nrm f_Y/\sqrt2$. For every angular coefficient, apply the radial Poisson inverse to $(f-qv)_n$. Its kernel is dominated by $-\log\max(r,s)\le-\log s$, and
\[
 \int_0^1(\log s)^2s\,ds=\frac14.
\]
Cauchy--Schwarz yields
\begin{equation}\label{r2:eq:low-mode}
 \sup_r|v_n(r)|\le\frac12\nrm{(f-qv)_n}_{L^2(r\,dr)}
 \le\frac{1+P^2B_{L^2}}{2\sqrt2}\nrm f_Y.
\end{equation}
This estimates all coefficients, but we use it only for $|n|<hm$. Their total weighted norm is bounded by $W_\ell$ times the right side of~\eqref{r2:eq:low-mode}.

For $|n|\ge hm$, the positive potential of $-\Delta_n-\rho^2$ is
\[
 \frac{n^2}{r^2}-\rho^2\ge n^2-\rho^2\ge d_h>0.
\]
The radial maximum principle bounds its zero-Dirichlet inverse in supremum norm by $1/(n^2-\rho^2)$. The direct sum of these inverses is consequently bounded by $1/d_h$ in the weighted radial supremum norm on the high sector $|n|\ge hm$ of $Y$. Let $\Pi_H$ be the angular projection onto that sector; it has norm one on $Y$, and multiplication by $q-\rho^2$ has norm at most $V_*$ by~\eqref{r2:eq:algebra} and~\eqref{r2:eq:V-analytic}. Let $v_\ell=v-\Pi_Hv$ be the low part already estimated, which belongs to $Y$, and solve the high equations
\[
 (\Delta+\rho^2)\tilde v_H=\Pi_H\bigl(f-(q-\rho^2)(v_\ell+\tilde v_H)\bigr),
 \qquad\tilde v_H=0\ \text{on }\partial\D,
\]
by the contraction mapping theorem in the high sector of $Y$. The contraction constant is at most $\beta<1$, so there is a unique solution $\tilde v_H$, and its weighted norm is finite. This construction establishes finiteness of the weighted high norm before any inequality involving that norm is rearranged.

We identify $\tilde v_H$ with the high part $v_H=\Pi_Hv$ of the original $L^2$ solution. Put
\[
 F_H=\Pi_H\bigl(f-(q-\rho^2)(v_\ell+\tilde v_H)\bigr)-\rho^2\tilde v_H\in Y.
\]
In every angular mode, $\tilde v_{H,n}$ is the regular solution of the radial equation $\Delta_n\tilde v_{H,n}=F_{H,n}$ with $\tilde v_{H,n}(1)=0$; hence $\tilde v_H=KF_H$. Since $F_H$ is continuous, $\tilde v_H\in W^{2,s}(\D)$ for every finite $s$, as in the proof of Lemma~\ref{r2:lem:Poisson}. In particular $\tilde v_H\in H^2\cap H^1_0$, and the high equation holds almost everywhere. Applying $\Pi_H$ to $(\Delta+\rho^2)v=f-(q-\rho^2)v$ shows that $v_H\in H^2\cap H^1_0$ satisfies the same equation with $\tilde v_H$ replaced by $v_H$. The difference $z=\tilde v_H-v_H$ therefore lies in the high sector and solves
\[
 (\Delta+\rho^2)z=-\Pi_H\bigl((q-\rho^2)z\bigr).
\]
On this sector, $r\le1$ and angular orthogonality give
\[
 \int_\D|\nabla z|^2\ge\int_\D\frac{|\partial_\theta z|^2}{r^2}
 \ge(hm)^2\int_\D|z|^2.
\]
Multiplying the equation for $z$ by $-z$ and integrating over $\D$, and using that $\Pi_H$ is self-adjoint in $L^2$ with $\Pi_Hz=z$, gives
\[
 d_h\int_\D|z|^2\le\int_\D|\nabla z|^2-\rho^2\int_\D|z|^2
 =\int_\D(q-\rho^2)|z|^2
 \le V_*\int_\D|z|^2,
\]
because $\nrm{q-\rho^2}_\infty\le\nrm{q-\rho^2}_Y\le V_*$. Since $d_h>V_*$, $z=0$. Thus $v_H=\tilde v_H$ belongs to $Y$.

The high fixed-point equation now gives
\[
 \nrm{v_H}_Y\le d_h^{-1}\nrm f_Y
 +\beta\bigl(\nrm{v_\ell}_Y+\nrm{v_H}_Y\bigr).
\]
Add the low norm and solve for the total norm $\nrm v_Y=\nrm{v_\ell}_Y+\nrm{v_H}_Y$ to obtain~\eqref{r2:eq:BY}. Finally apply the trace formula from Lemma~\ref{r2:lem:Poisson} to $\Delta v=f-qv$:
\begin{equation}\label{r2:eq:normal-moment}
 N_n=\int_0^1r^{|n|+1}(f-qv)_n(r)\,dr.
\end{equation}
It follows that
\[
 R^{|n|}(1+|n|)|N_n|
 \le R^{|n|}\frac{1+|n|}{2+|n|}\sup_r|(f-qv)_n(r)|.
\]
Summation and~\eqref{r2:eq:V-analytic} prove~\eqref{r2:eq:BN}.
\end{proof}

For the centre in this paper, $h=5$ and $W_\ell=61$. The numerical hypotheses and consequences are
\begin{equation}\label{r2:eq:Y-numbers}
 \begin{gathered}
 d_h>164759.13,\qquad\beta<0.061958,\\
 B_Y<1.045\cdot10^8,\qquad B_N<7.039\cdot10^{13}.
 \end{gathered}
\end{equation}
These are bounds in the same angular weight used for the shape space. The normal trace gains the factor $(|n|+2)^{-1}$, which pays for one derivative when the shape variation is reconstructed in Subsection~\ref{r2:sec:linearization}. No reduction of the analytic radius is needed.

\section{The centre and the exact linearisation inverse}\label{r2:sec:coupled}

\subsection{The exactly clamped approximate solution}\label{r2:sec:centre}

\subsubsection{The finite input and the analytic lift}
The frozen centre contains $41$ exact decimal entries in each of the arrays $c_j$, $a_j$ and $s_j$, indexed by $0\le j\le40$. Every $s_j$ is positive. The file key \texttt{p} stores $c_j$, not the derivative coefficient $p_j=(jm+1)c_j$. Define
\begin{equation}\label{r2:eq:finite-centre}
 \psi_0(w)=\sum_{j=0}^{40}c_jw^{jm+1},\qquad p_0=\psi_0',
 \qquad\rho=c_0.
\end{equation}
The exact decimal value of $\rho$ lies in $(814.45,814.46)$. The finite physical field is
\begin{equation}\label{r2:eq:physical-centre}
 u_*(z)=\sum_{j=0}^{40}\frac{a_j}{s_j}J_{jm}(|z|)\cos(jm\arg z),
 \qquad U_*(w)=u_*(\psi_0(w)).
\end{equation}
The normalisation divisors $s_j$ are part of the exact input; no identity with a Bessel value is assumed. Each summand extends smoothly through $z=0$ and is an entire real-analytic solution of $(\Delta+1)u=0$. Thus $\Delta U_*+qU_*=0$, where $q=|p_0|^2$.

Let $D_j$ and $N_j$ be the full cosine coefficients in
\[
 U_*(1,\theta)-1=\sum_{j\ge0}D_j\cos(jm\theta),\qquad
 \partial_rU_*(1,\theta)=\sum_{j\ge0}N_j\cos(jm\theta).
\]
For $n=jm$ introduce
\begin{equation}\label{r2:eq:lifts}
 h_n^D(r)=r^n\left(1-\frac n2(r^2-1)\right),\qquad
 h_n^N(r)=\frac12r^n(r^2-1).
\end{equation}
Their Dirichlet and radial derivative traces at $r=1$ are $(1,0)$ and $(0,1)$, respectively. Set
\begin{equation}\label{r2:eq:clamped}
 \begin{split}
 \ell(r,\theta)&=\sum_{j\ge0}
  \bigl(D_jh_{jm}^D(r)+N_jh_{jm}^N(r)\bigr)\cos(jm\theta),\\
 U_0&=U_*-\ell,\qquad g_0=\Delta(U_0-1).
 \end{split}
\end{equation}
The bounds below imply convergence of this series with all derivatives needed in the proof, on a neighbourhood of the closed disc. Consequently
\begin{equation}\label{r2:eq:exact-clamp}
 U_0=1,\quad\partial_rU_0=0\quad\text{on }\partial\D,
 \qquad g_0\in X_g,\quad Kg_0=U_0-1.
\end{equation}
The inclusion follows by Green's identity against the harmonic monomials; the last equality follows from Dirichlet uniqueness for $\Delta$. Thus $(g_0,p_0)$ is an exactly specified point of $X$, even though $g_0$ is defined by an infinite analytic lift.

\subsubsection{Uniform strip bounds for the field}
Complexify $\phi=m\theta$ to $|\im\phi|\le\log4$. Define
\begin{equation}\label{r2:eq:strip-shape}
 \begin{gathered}
 S=\sum_{j>0}4^j|c_j|,\qquad
 d=\frac{2\rho S+S^2}{\rho},\qquad
 \eps_s=\frac12\log\frac{\rho+S}{\rho-S},\\
 P_4=\sum_j4^j(jm+1)|c_j|.
 \end{gathered}
\end{equation}
For fixed $0\le r\le1$, write
\[
 A_\pm=\rho+\sum_{j>0}c_jr^{jm}e^{\pm ij\phi},\qquad
 B_s=(A_+A_-)^{1/2}.
\]
Both $A_\pm$ lie within $S$ of $\rho$. For the centre $S<0.07$, so $\rho>2S$ and each of $A_\pm$ has argument of modulus less than $\pi/6$; this places $A_+A_-$ in the right half-plane and selects the square root $B_s$ with positive real part. Since $|B_s^2-\rho^2|=|A_+A_--\rho^2|\le2\rho S+S^2$ and $|B_s+\rho|\ge\rho$, and since $|A_\pm/B_s|^2=|A_\pm/A_\mp|\le(\rho+S)/(\rho-S)$,
\begin{equation}\label{r2:eq:strip-B}
 |B_s-\rho|\le d,\qquad
 |A_+/B_s|\le e^{\eps_s},\qquad
 |A_-/B_s|\le e^{\eps_s}.
\end{equation}
The continuation of the positive angular branch in~\eqref{r2:eq:physical-centre} is
\[
 J_n(rB_s)e^{in\theta}(A_+/B_s)^n,
 \qquad n=jm,
\]
and the negative branch uses $A_-$. This representation also follows from the entire power series in the two independent variables $z,\bar z$, so the apparent polar coordinates introduce no singularity at the origin.

For integer $n\ge0$, shift the contour in the Bessel integral by $i\alpha$, $\alpha\ge0$. The integrand is entire and periodic, so
\begin{equation}\label{r2:eq:Bessel-contour}
 |J_n(z)|\le
 \exp\bigl(-n\alpha+|\re z|\sinh\alpha+|\im z|\cosh\alpha\bigr).
\end{equation}
Choose $\alpha_j=0$ if $jm\le\rho+d$, and otherwise
\[
 \alpha_j=\acosh\frac{jm}{\rho+d}.
\]
Combining~\eqref{r2:eq:strip-B} and~\eqref{r2:eq:Bessel-contour}, define
\begin{equation}\label{r2:eq:bj}
 b_j=\frac{|a_j|}{s_j}4^j
 \exp\bigl(jm\eps_s-jm\alpha_j
       +(\rho+d)\sinh\alpha_j+d\cosh\alpha_j\bigr).
\end{equation}
Then $M_U=\sum_j b_j$ bounds $|U_*|$ on the complex strip, uniformly in $r$. Cauchy's estimate in the periodic variable gives a Fourier coefficient bound $M_U4^{-|k|}$; summing the weight $2^{|k|}$ yields
\begin{equation}\label{r2:eq:Ub}
 \nrm{U_*}_Y\le U_b:=3M_U.
\end{equation}

We need a bound for a Cartesian derivative in the same weight, rather than only a bound for an angular derivative. The Wirtinger identities are
\begin{align*}
 \partial_z\bigl(J_n(|z|)e^{in\arg z}\bigr)
 &=\tfrac12J_{n-1}(|z|)e^{i(n-1)\arg z},\\
 \partial_z\bigl(J_n(|z|)e^{-in\arg z}\bigr)
 &=-\tfrac12J_{n+1}(|z|)e^{-i(n+1)\arg z}.
\end{align*}
Hence, for $n=jm$,
\[
 \partial_zu_*=\sum_{j=0}^{40}\frac{a_j}{4s_j}
 \Bigl(J_{n-1}(|z|)e^{i(n-1)\arg z}-J_{n+1}(|z|)e^{-i(n+1)\arg z}\Bigr),
\]
and $\partial_wU_*=p_0\,(\partial_zu_*)\circ\psi_0$. We bound the continuation of each summand to the strip as before; the two angular factors continue to $e^{i(n-1)\theta}(A_+/B_s)^{n-1}$ and $e^{-i(n+1)\theta}(A_-/B_s)^{n+1}$. Let $j\ge1$. With the same $\alpha_j$, the contour bound~\eqref{r2:eq:Bessel-contour} for the orders $n\pm1$ exceeds the bound for the order $n$ by at most the factor $e^{\alpha_j}$. On the strip $|e^{\pm i\theta}|\le4^{1/m}$ and $|A_\pm/B_s|\le e^{\eps_s}$, so the angular factor of order $n+1$ is at most $4^{1/m}e^{\eps_s}$ times the bound $4^je^{n\eps_s}$ used in~\eqref{r2:eq:bj}, while the factor of order $n-1$ does not exceed that bound. For $j=0$ the summand is $-(a_0/2s_0)J_1(|z|)e^{-i\arg z}$, and $\alpha_0=0$. Finally $|p_0|\le P_4$ on the strip. The $j$th summand of $\partial_wU_*$ is therefore at most $\frac12\,4^{1/m}P_4e^{\eps_s}b_je^{\alpha_j}$ for every $j\ge0$, and the same holds for $\partial_{\bar w}U_*$ by reality. Dropping the factor $1/2$, both Wirtinger derivatives satisfy, on the strip and uniformly in $r$,
\begin{equation}\label{r2:eq:strip-derivative}
 |\partial_wU_*|,\ |\partial_{\bar w}U_*|\le
 M_\partial:=4^{1/m}P_4e^{\eps_s}\sum_jb_je^{\alpha_j}.
\end{equation}
The derivative $\partial_wU_*$ has angular frequencies $km-1$, $k\in\Z$. The function $e^{i\theta}\partial_wU_*$ is therefore periodic in $\phi$, and it is bounded by $4^{1/m}M_\partial$ on the strip. Cauchy's estimate bounds its Fourier coefficients by $4^{1/m}M_\partial4^{-|k|}$, while $R^{|km-1|}\le R2^{|k|}$. Summation over $k$ gives
\begin{equation}\label{r2:eq:Db}
 \nrm{\partial_wU_*}_{\Yc}\le
 D_b:=3R4^{1/m}M_\partial=3R4^{2/m}P_4e^{\eps_s}\sum_jb_je^{\alpha_j}.
\end{equation}
Reality gives the same bound for $\partial_{\bar w}U_*$. The resulting enclosures are $U_b<131.610$ and $D_b<110996$.

Both boundary defects have a common strip bound. Since $2M_U\ge1$ for the centre (indeed $M_U=U_b/3>43$),
\[
 |U_*-1|\le M_U+1\le3M_U=U_b.
\]
The radial derivative is $\partial_rU_*=e^{i\theta}\partial_wU_*+e^{-i\theta}\partial_{\bar w}U_*$, and $|e^{\pm i\theta}|\le4^{1/m}$ on the strip. By~\eqref{r2:eq:strip-derivative},
\[
 |\partial_rU_*|\le2\cdot4^{1/m}M_\partial=\frac{2D_b}{3R}.
\]
We can therefore use
\begin{equation}\label{r2:eq:trace-M}
 M=U_b+2D_b
\end{equation}
as a common majorant for the two boundary functions.

\subsubsection{All boundary modes and the residual}
The boundary computation encloses $1024$ equispaced values at $4096$ bits and their cosine coefficients through $j=400$. For a periodic strip-analytic function bounded by $M$, the discrete coefficient differs from the continuum coefficient by the sum of all aliased coefficients. Their geometric bound gives, uniformly for $0\le j\le400$,
\begin{equation}\label{r2:eq:alias}
 e_Q=\frac{8M4^{-(1024-400)}}{1-4^{-1024}}.
\end{equation}
For $j>400$, Cauchy's estimate gives $|D_j|,|N_j|\le2M4^{-j}$. The constant mode is included with its own coefficient convention; in particular no defect is discarded because it belongs to a finite equation used to choose the centre.

For $0\le k\le3$, form the full moments
\begin{equation}\label{r2:eq:moments}
 D^{[k]}=\sum_{j\ge0}2^j(jm+1)^k|D_j|,
 \qquad N^{[k]}=\sum_{j\ge0}2^j(jm+1)^k|N_j|.
\end{equation}
The tail begins at $j=401$. Expanding $(m(401+t)+1)^k$ by the binomial theorem reduces its bound to
\begin{equation}\label{r2:eq:tail-sums}
 \sum_{t\ge0}t^\ell2^{-t}=2,\ 2,\ 6,\ 26\quad\text{for }\ell=0,1,2,3\text{ respectively},
\end{equation}
where $t^0=1$, also at $t=0$. Thus every moment in~\eqref{r2:eq:moments} includes the infinite tail. The stronger coefficient decay $4^{-j}$ also proves smooth convergence of the lift on any closed disc strictly inside radius $4^{1/m}$.

Direct differentiation gives
\[
 \Delta(h_n^D\cos n\theta)=-2n(n+1)r^n\cos n\theta,
 \qquad
 \Delta(h_n^N\cos n\theta)=2(n+1)r^n\cos n\theta.
\]
On $0\le r\le1$, $|h_n^D|\le1$ and $|h_n^N|\le1/(n+2)$. For the first inequality, $h_n^D$ increases to one, since its derivative is $n(n+2)r^{n-1}(1-r^2)/2$ for $n>0$; the $n=0$ case is constant. For the second, maximize $r^n(1-r^2)/2$. The following enlarged constants then suffice:
\begin{equation}\label{r2:eq:lift-bounds}
 \begin{aligned}
 L_0&=D^{[0]}+N^{[0]}/2,&
 L_1&=2R(D^{[2]}+N^{[1]}),\\
 L_\Delta&=2(D^{[2]}+N^{[1]}),&
 L_{\Delta,1}&=2R(D^{[3]}+N^{[2]}).
 \end{aligned}
\end{equation}
They bound $\ell$, $\partial_w\ell$, $\Delta\ell$ and $\partial_w\Delta\ell$, respectively, in $\Yc$. To see the derivative bounds, express $r^n\cos n\theta=(w^n+\bar w^n)/2$ and multiply by $w\bar w$ for the second term of each lift. A Wirtinger derivative costs at most the degree and one factor $R$ in the angular weight. The coefficients in~\eqref{r2:eq:lift-bounds} dominate all resulting terms, including $n=0$.

The residual at the clamped centre is
\begin{equation}\label{r2:eq:E}
 E=F(g_0,p_0)=\Delta U_0+qU_0=-(\Delta+q)\ell.
\end{equation}
Since $\nrm{\partial_wq}_{\Yc}\le RPP_1$, the residual and its derivative have bounds
\begin{equation}\label{r2:eq:E-bounds}
 \nrm E_Y\le E_0=L_\Delta+P^2L_0,
 \qquad
 \nrm{\partial_wE}_{\Yc}\le E_1=L_{\Delta,1}+RPP_1L_0+P^2L_1.
\end{equation}
Since $\Delta U_0=-qU_*-\Delta\ell$, we also have the fixed-centre derivative bounds
\begin{equation}\label{r2:eq:D-bounds}
 \begin{split}
 \nrm{\Delta U_0}_Y&\le\Lambda_0:=P^2U_b+L_\Delta,\\
 \nrm{\partial_w\Delta U_0}_{\Yc}&\le
 \Lambda_1:=RPP_1U_b+P^2D_b+L_{\Delta,1}.
 \end{split}
\end{equation}
The resulting enclosures are
\begin{equation}\label{r2:eq:E-numbers}
 E_0<9.643\cdot10^{-61},\qquad E_1<5.418\cdot10^{-56},
 \qquad\Lambda_0<8.865\cdot10^7,\quad\Lambda_1<7.489\cdot10^{10}.
\end{equation}
The small residual is a bound for the clamped centre in the Banach-space equation, including all modes. The existence argument uses this residual and the inverse proved next.

\subsection{An exact inverse of the corrected linearisation}\label{r2:sec:linearization}

All quantities are evaluated at $(g_0,p_0)$; abbreviate $p=p_0$, $q=|p|^2$. The centre $U_0$ is smooth on a neighbourhood of $\overline\D$ and has the exact Cauchy data~\eqref{r2:eq:exact-clamp}. Its residual is $E=\Delta U_0+qU_0$. For a variation $(\delta g,\delta p)\in X$, set
\begin{equation}\label{r2:eq:variation}
 \eta(w)=\int_0^w\delta p(z)\,dz,\qquad
 v=\eta/p,
 \qquad T=v\partial_wU_0+\bar v\partial_{\bar w}U_0.
\end{equation}
Here $p$ has no zero on $|w|\le R$ as soon as $j_*>0$ in~\eqref{r2:eq:reciprocal-denominators} below, because $|p(w)-\rho|\le P-\rho$ there; $v$ is then holomorphic on $|w|<R$. The real vector field represented by the complex number $v$ acts on $U_0$ as $T$. Define the bounded correction
\begin{equation}\label{r2:eq:residual-correction}
 \Ecal(\delta p)=v\partial_wE+\bar v\partial_{\bar w}E+(v'+\bar v')E,
 \qquad \widetilde D=DF(g_0,p_0)-\Ecal,
\end{equation}
where $\Ecal$ on $X$ acts only on the second component.

\begin{proposition}[The coupled inverse]\label{r2:prop:coupled}
Suppose $L=\Delta+q$ has the symmetric-sector Dirichlet inverse and trace bound of Proposition~\ref{r2:prop:Yinverse}, and the centre satisfies
\begin{equation}\label{r2:eq:reciprocal-denominators}
 d_*:=\rho^2-(P^2-\rho^2)-E_0>0,
 \qquad j_*:=\rho-(P-\rho)>0.
\end{equation}
Then $\widetilde D\colon X\to Y$ has a bounded two-sided inverse $A\colon Y\to X$. Define
\begin{align}
 I_0&=d_*^{-1},& I_1&=I_0+I_0^2(2PP_1+2RE_1),\label{r2:eq:I}\\
 \mathcal J_0&=j_*^{-1},& \mathcal J_1&=\mathcal J_0+P_1\mathcal J_0^2,\label{r2:eq:J}\\
 B_p&=(P+P_1)I_1B_N,&
 B_g&=B_N+2(R\Lambda_1+\Lambda_0)I_1B_N.\label{r2:eq:BgBp}
\end{align}
Then, with $A_*=B_g+B_p$,
\begin{equation}\label{r2:eq:A-defects}
 \begin{gathered}
 \nrm A_{Y\to X}\le A_*,\qquad
 \nrm{AF(g_0,p_0)}_X\le Y_0:=A_*E_0,\\
 \nrm{I_X-A DF(g_0,p_0)}_{X\to X}
 \le Z:=2A_*(RE_1+E_0)\mathcal J_1.
 \end{gathered}
\end{equation}
\end{proposition}

\begin{proof}
\emph{The material-derivative identity.}
Use $\Delta=4\partial_w\partial_{\bar w}$ and holomorphicity of $v$. Direct differentiation gives
\begin{equation}\label{r2:eq:DeltaT}
 \Delta T=v\partial_w\Delta U_0+\bar v\partial_{\bar w}\Delta U_0
          +(v'+\bar v')\Delta U_0.
\end{equation}
There is no second derivative of $v$ in this formula. For example,
\[
 4\partial_w\partial_{\bar w}(v\partial_wU_0)
 =v'\Delta U_0+v\partial_w\Delta U_0;
\]
the conjugate term has the analogous expression. Since $\eta=pv$,
\begin{equation}\label{r2:eq:deltaq}
 \begin{split}
 \delta q&=\bar p\eta'+p\bar\eta'\\
 &=v\partial_wq+\bar v\partial_{\bar w}q+q(v'+\bar v').
 \end{split}
\end{equation}
Insert $\Delta U_0=E-qU_0$ into~\eqref{r2:eq:DeltaT}, apply the product rule, and use~\eqref{r2:eq:deltaq}. The terms involving derivatives of $q$ combine with $\delta qU_0$, and the terms involving derivatives of $U_0$ combine with $qT$. The result is the exact identity
\begin{equation}\label{r2:eq:material}
 (\Delta+q)T+\delta qU_0
 =v\partial_wE+\bar v\partial_{\bar w}E+(v'+\bar v')E
 =\Ecal(\delta p).
\end{equation}
For $\delta U=K\delta g$, differentiation of $F$ gives
\begin{equation}\label{r2:eq:DF}
 DF(g_0,p_0)(\delta g,\delta p)
 =(\Delta+q)\delta U+\delta qU_0.
\end{equation}
Thus $\widetilde D(\delta g,\delta p)=L(\delta U-T)$.

\emph{The boundary Hessian.}
The constant Dirichlet data and zero normal derivative imply $\nabla U_0=0$ on the circle. Differentiate this vector identity along the circle: the Hessian annihilates its tangent vector. Since the Hessian is symmetric, it has only a normal--normal component, and its trace is that component. Therefore
\[
 \operatorname{Hess}U_0=(\Delta U_0)\nu\otimes\nu
 \quad\text{on }\partial\D.
\]
The derivative of the vector field $v$ in the normal derivative of $T$ multiplies $\nabla U_0$ and vanishes. With $\nu=e^{i\theta}$ and $U_0=1$ this proves
\begin{equation}\label{r2:eq:T-boundary}
 T=0,\qquad
 \partial_rT=(E-q)\re(ve^{-i\theta})
 \quad\text{on }\partial\D.
\end{equation}

\emph{The boundary reciprocal.}
Put $d_\partial=(q-E)|_{\partial\D}$. In the weighted boundary algebra $W_0$,
\[
 \nrm{d_\partial-\rho^2}_{W_0}
 \le P^2-\rho^2+E_0<\rho^2.
\]
A Neumann series in $W_0$ gives $d_\partial^{-1}\in W_0$ and norm at most $I_0$. The angular derivative satisfies
\[
 \nrm{\partial_\theta q}_{W_0}\le2PP_1,
 \qquad\nrm{\partial_\theta E}_{W_0}\le2RE_1.
\]
For the second inequality use $\partial_\theta E=i(w\partial_wE-\bar w\partial_{\bar w}E)$ and reality. Differentiate the reciprocal to obtain
\[
 \nrm{d_\partial^{-1}}_{W_1}
 \le I_0+I_0^2\nrm{\partial_\theta d_\partial}_{W_0}\le I_1.
\]
This argument uses a smallness condition only in $W_0$, and the derivative identity in $W_1$; it does not require the stronger $W_1$ perturbation to be small.

\emph{Construction of the right inverse.}
For $f\in Y$, solve
\begin{equation}\label{r2:eq:Vproblem}
 L\Phi=f\quad\text{in }\D,\qquad\Phi=0\quad\text{on }\partial\D,
 \qquad N=\partial_r\Phi|_{r=1}.
\end{equation}
Set $H_\partial=N/d_\partial$. Proposition~\ref{r2:prop:Yinverse} and the reciprocal bound give $\nrm{H_\partial}_{W_1}\le I_1B_N\nrm f_Y$. This function is real and even, with frequencies in $m\Z$. If its complex Fourier coefficients are $H_n$, define
\begin{equation}\label{r2:eq:realpart}
 a(w)=H_0+2\sum_{j>0}H_{jm}w^{jm}.
\end{equation}
Then $a$ has real coefficients, $\re a=H_\partial$ on the circle, and
\begin{equation}\label{r2:eq:a-norm}
 \nrm a_{W_1}=\nrm{H_\partial}_{W_1}\le I_1B_N\nrm f_Y.
\end{equation}
The equality uses $H_{-n}=H_n$ and the factor two on positive frequencies. This real-part problem has a unique solution in the real-coefficient class: the usual imaginary constant is excluded by that class. The constant mode $H_0$ is allowed and determines the dilation component of the shape.

Define
\begin{equation}\label{r2:eq:A-recipe}
 v=wa,\qquad\eta=pv,\qquad\delta p=(pv)',
 \qquad\delta U=\Phi+T,\qquad\delta g=\Delta\delta U.
\end{equation}
Now $\re(ve^{-i\theta})=\re a=N/d_\partial$. Equations~\eqref{r2:eq:T-boundary} and~\eqref{r2:eq:Vproblem} give zero Dirichlet and Neumann traces of $\delta U$. Consequently $\delta g\in X_g$ and $K\delta g=\delta U$. Its shape variation lies in $W$, with the required real coefficients and frequencies. Formula~\eqref{r2:eq:material} then gives $\widetilde D(\delta g,\delta p)=L\Phi=f$. The construction is linear in $f$; denote it by $Af$.

\emph{Bounds and regularity of the construction.}
Multiplication and differentiation of the holomorphic series give
\begin{equation}\label{r2:eq:v-estimates}
 \nrm v_{\Yc}\le R\nrm a_{W_1},\qquad
 \nrm{v'}_{\Yc}\le\nrm a_{W_1},\qquad
 \nrm{\delta p}_W=\nrm{pa}_{W_1}
 \le(P+P_1)\nrm a_{W_1}.
\end{equation}
In the middle inequality, writing $a(w)=\sum_{j\ge0}\hat a_jw^{jm}$, one has $v'=\sum_j(jm+1)\hat a_jw^{jm}$, so equality actually holds. The last identity follows from differentiating $wpa$. For the field, $\Delta\Phi=f-q\Phi$ and~\eqref{r2:eq:DeltaT} imply
\[
 \nrm{\delta g}_Y
 \le B_N\nrm f_Y+2(R\Lambda_1+\Lambda_0)\nrm a_{W_1}.
\]
Together with~\eqref{r2:eq:a-norm}, these are $B_p$ and $B_g$ in~\eqref{r2:eq:BgBp} and prove $\nrm A\le A_*$. They also prove that the recipe maps all of $Y$ into $X$.

There is no loss of angular radius in these estimates. Although a general input field is only continuous in $Y$, its Poisson potential and the Dirichlet solution are in $W^{2,s}$ for finite $s$ and have $C^1$ traces. The function $v$ is holomorphic on the larger disc $|w|<R$; hence $T$ is smooth near $\overline\D$. The identities hold distributionally for these general inputs, and the boundary identities are classical for $T$. Thus no density assertion in a stronger differentiability norm is required.

\emph{The left inverse.}
Take any $(\delta g,\delta p)\in X$ and define $\eta$, $v$, $T$ as in~\eqref{r2:eq:variation}. For $a=\eta/(wp)$, the apparent singularity at zero is removable. Moreover,
\begin{equation}\label{r2:eq:primitive-norm}
 \nrm{\eta/w}_{W_1}=\nrm{\delta p}_W.
\end{equation}
The $W_0$ Neumann series for $1/p$, followed by differentiation of the reciprocal, gives $\nrm{1/p}_{W_1}\le\mathcal J_1$ from~\eqref{r2:eq:J}. Thus $a\in W_1$, with real coefficients and frequencies $jm$.

Put $\Phi=K\delta g-T$. It has zero Dirichlet trace and satisfies $L\Phi=\widetilde D(\delta g,\delta p)$. Dirichlet uniqueness in the symmetric sector recovers exactly the function in~\eqref{r2:eq:Vproblem} when the recipe is applied to this right side. Its normal trace is, by~\eqref{r2:eq:T-boundary},
\[
 \partial_r\Phi=-\partial_rT=d_\partial\re a.
\]
Dividing by $d_\partial$ and solving the real-part problem recovers $a$ uniquely. The recipe therefore recovers $v$, $\eta$, $\delta p$ and $\delta g$. Hence $A\widetilde D=I_X$, as well as $\widetilde DA=I_Y$. In particular $A$ is injective.

\emph{The Newton defect.}
Equation~\eqref{r2:eq:primitive-norm} gives $\nrm a_{W_1}\le\mathcal J_1\nrm{\delta p}_W$. From~\eqref{r2:eq:residual-correction} and~\eqref{r2:eq:v-estimates},
\[
 \nrm{\Ecal(\delta p)}_Y
 \le2(RE_1+E_0)\mathcal J_1\nrm{\delta p}_W.
\]
Since $A\widetilde D=I_X$,
\[
 I_X-A DF(g_0,p_0)=-A\Ecal.
\]
This proves the left defect estimate in~\eqref{r2:eq:A-defects}. The bound for $AF(g_0,p_0)$ follows from $\nrm E_Y\le E_0$.
\end{proof}

\begin{remark}[Shape differentiation]
The material derivative $T=v\cdot\nabla U_0$ and the use of the boundary normal displacement are the classical Hadamard shape-derivative mechanism; see, for example,~\cite{SZ92,HP18}. This mechanism is not specific to two dimensions. The planar conformal coordinates and holomorphic real-part problem make the resulting inverse explicit here. The residual term in~\eqref{r2:eq:material} records that the centre solves the equation only approximately.
\end{remark}

For the forty-mode centre, exact rational arithmetic gives $P<820.7003$, $P_1<1176.5538$ and $V_*=P^2-\rho^2<10208.097$. With the bound for $E_0$ in~\eqref{r2:eq:E-numbers},
\begin{equation}\label{r2:eq:reciprocal-numbers}
 d_*>6.5313\cdot10^5,\qquad j_*>808.21,
\end{equation}
so~\eqref{r2:eq:reciprocal-denominators} holds. The enclosed constants are $I_1<6.059\cdot10^{-6}$, $\mathcal J_1<3.039\cdot10^{-3}$, $B_p<8.517\cdot10^{11}$ and $B_g<6.419\cdot10^{19}$, and the inverse bound is $A_*<6.419\cdot10^{19}$. Its size is compatible with the contraction because the residual bounds in~\eqref{r2:eq:E-numbers} are much smaller. The two-sided construction is needed because the Newton map acts on $X$: its contraction estimate uses $I_X-A\,DF(g_0,p_0)$, a left-inverse defect on $X$.

\section{Existence, geometry and the Pompeiu property}\label{r2:sec:existence}

\subsection{The radii polynomial}
Let $x_0=(g_0,p_0)$ and use the fixed bounded inverse $A$ of Proposition~\ref{r2:prop:coupled}. Write $\kappa_K=1/4$ and $U_c=U_b+L_0$. For $(h,k)\in X$, expansion of the cubic map gives
\begin{equation}\label{r2:eq:cubic-expansion}
 \begin{split}
 F(g_0+h,p_0+k)={}&E+DF(g_0,p_0)(h,k)\\
 &+(p_0\bar k+\bar p_0k)Kh+|k|^2U_0+|k|^2Kh.
 \end{split}
\end{equation}
For $s=\nrm h_Y+\nrm k_W$, the algebra estimate and Lemma~\ref{r2:lem:Poisson} bound the quadratic terms by $(2P\kappa_K+U_c)s^2$ and the cubic term by $\kappa_Ks^3$. Define
\begin{equation}\label{r2:eq:C23}
 C_2=A_*(2P\kappa_K+U_c),\qquad C_3=A_*\kappa_K.
\end{equation}
These constants also bound the corresponding polarised bilinear and trilinear maps. For example, polarising the first quadratic term gives the average of $2P\kappa_K\nrm{k_1}\nrm{h_2}$ and $2P\kappa_K\nrm{k_2}\nrm{h_1}$; each is bounded by the product of the two $X$ norms. Polarizing $|k|^2U_0$ uses $\re(k_1\bar k_2)U_0$ with the same bound. The symmetric trilinear form for $|k|^2Kh$ has bound $\kappa_K$ in the product of three $X$ norms. Consequently the derivative majorants are $2C_2s$ and $3C_3s^2$.

\begin{lemma}\label{r2:lem:contraction}
If a radius $r_*>0$ satisfies
\begin{equation}\label{r2:eq:radii}
 Y_0+(Z-1)r_*+C_2r_*^2+C_3r_*^3<0,
 \qquad Z+2C_2r_*+3C_3r_*^2<1,
\end{equation}
then $F$ has a unique zero in the closed $X$ ball $\nrm{x-x_0}_X\le r_*$.
\end{lemma}
\begin{proof}
For the Newton map $\mathcal T(x)=x-AF(x)$, Proposition~\ref{r2:prop:coupled} and~\eqref{r2:eq:cubic-expansion} imply
\[
 \nrm{\mathcal T(x_0+(h,k))-x_0}_X
 \le Y_0+Zs+C_2s^2+C_3s^3.
\]
The first inequality of~\eqref{r2:eq:radii} makes this ball invariant. The polarised derivative bounds give
\[
 \nrm{D\mathcal T(x)}_{X\to X}\le Z+2C_2r_*+3C_3r_*^2<1
\]
on the ball. Completeness gives a unique fixed point. At that point $AF(x)=0$, and injectivity of $A$ gives $F(x)=0$. Every zero in the ball is a fixed point, so the uniqueness is uniqueness of a zero of $F$.
\end{proof}

For the forty-mode centre take $r_*=10^{-32}$. The constants obtained from the exact interval chain are recorded in Table~\ref{r2:tab:constants}. Upper bounds are rounded upward; negative upper bounds are rounded toward zero. The two inequalities of Lemma~\ref{r2:lem:contraction} are strict.
\begin{table}[htbp]
\centering
\begin{tabular}{ll}
\toprule
Quantity & Enclosed bound\\
\midrule
$B_{L^2}$ & $<6.748$\\
$B_Y$ & $<1.045\cdot10^8$\\
$B_N$ & $<7.039\cdot10^{13}$\\
$E_0$ & $<9.643\cdot10^{-61}$\\
$E_1$ & $<5.418\cdot10^{-56}$\\
$A_*$ & $<6.419\cdot10^{19}$\\
$Y_0$ & $<6.189\cdot10^{-41}$\\
$Z$ & $<2.122\cdot10^{-38}$\\
$C_2$ & $<3.479\cdot10^{22}$\\
$C_3$ & $<1.605\cdot10^{19}$\\
$Y_0+(Z-1)r_*+C_2r_*^2+C_3r_*^3$ & $<-9.99999993\cdot10^{-33}$\\
$Z+2C_2r_*+3C_3r_*^2$ & $<6.957\cdot10^{-10}$\\
$4C_2Y_0$ & $<8.611\cdot10^{-18}$\\
\bottomrule
\end{tabular}
\caption{The forty-mode certificate, with $r_*=10^{-32}$. Every row is rounded from its own interval enclosure. In particular the row for $Y_0$ and the last three rows are not recomputed from the rounded rows above; for example, the rounded bounds for $A_*$ and $E_0$ alone give only $Y_0<6.190\cdot10^{-41}$.}\label{r2:tab:constants}
\end{table}

The inequalities~\eqref{r2:eq:radii} remain valid at $r_*=10^{-32}$ when $A_*$, and hence $Y_0$, $Z$, $C_2$ and $C_3$, are multiplied by $10^8$: the first expression is then less than $-3.46\cdot10^{-33}$ and the second less than $0.0696$ (ancillary file \path{ROBUSTNESS_CHECK.json}).

\subsection{Geometry on the entire ball}
Let $(g,p)$ be the zero obtained above, and write $p(w)=\sum_jp_jw^{jm}$ and $\psi(w)=\sum_jc_jw^{jm+1}$ as in~\eqref{r2:eq:primitive}. For the centre write $p_0(w)=\sum_jp^{(0)}_jw^{jm}$ and $\psi_0(w)=\sum_jc^{(0)}_jw^{jm+1}$, so that $p^{(0)}_0=c^{(0)}_0=\rho$ and $P=\sum_j2^j|p^{(0)}_j|$. The estimates of this subsection use only
\[
 \nrm{p-p_0}_W=\sum_{j\ge0}2^j|p_j-p^{(0)}_j|\le r_*,
\]
so they hold for every shape in the ball. On $|w|\le R$ the weighted coefficient bound gives
\begin{equation}\label{r2:eq:univalence-ball}
 \re p(w)\ge\rho-(P-\rho)-r_*>808.
\end{equation}
If $w_1\ne w_2$ are in the disc $|w|<R$, then
\[
 \frac{\psi(w_2)-\psi(w_1)}{w_2-w_1}
 =\int_0^1p\bigl(w_1+t(w_2-w_1)\bigr)\,dt
\]
has positive real part. Thus $\psi$ is univalent on that larger disc. The same estimate makes $p$ nonzero there, and $\psi$ is analytic on a neighbourhood of $\overline\D$.

For unit-disc convexity, use the sharper unweighted centre sums
\begin{equation}\label{r2:eq:geometry-LU}
 L_* =\rho-\sum_{j>0}|p^{(0)}_j|-r_*,
 \qquad\Theta_*=\sum_{j>0}jm|p^{(0)}_j|+\frac{m+1}{2}r_*.
\end{equation}
Here $\sum_j|p_j-p^{(0)}_j|\le r_*$, and $\sum_jjm|p_j-p^{(0)}_j|\le\frac{m+1}2r_*$ because $jm/2^j\le(m+1)/2$ for every $j\ge1$. Throughout the closed unit disc therefore
\[
 |p|\ge L_*>811,
 \qquad |wp'|\le\Theta_*.
\]
On the boundary,
\begin{equation}\label{r2:eq:convexity-ball}
 \re\left(1+\frac{wp'}p\right)
 \ge1-\frac{\Theta_*}{L_*}>0.2998879.
\end{equation}
For $z(\theta)=\psi(e^{i\theta})$, the derivative of the tangent angle is $\re(1+wp'/p)$ and the speed is $|p|$. Thus the outwardly oriented curvature is
\begin{equation}\label{r2:eq:curvature-formula}
 \kappa(\theta)=\frac{\re(1+wp'/p)}{|p|}>0.
\end{equation}
Univalence makes the boundary a simple closed curve. Its positive curvature and total turning $2\pi$ make its enclosed domain strictly convex. Equivalently, this is the analytic convexity criterion for conformal maps~\cite[\S2.5]{Dur83}. The same coefficient estimates give the convexity inequality in additive form: for every shape in the ball and every $|w|\le1$,
\begin{equation}\label{r2:eq:joint-margin}
 |p(w)|-|wp'(w)|\ge
 \rho-\sum_{j>0}(jm+1)|p^{(0)}_j|-\frac{m+1}{2}r_*>243.3264,
\end{equation}
because $jm+1\le\frac{m+1}22^j$ for every $j\ge0$.

Since $\psi(0)=0$, the origin lies in $\Omega$. Every tangent line to the boundary supports this strictly convex domain and therefore does not pass through the origin. Thus every ray from the origin meets the boundary exactly once and transversally. The coefficient bounds also give, for $|w|=1$,
\[
 \left|p(w)-\frac{\psi(w)}w\right|\le P-\rho+r_*,\qquad
 \left|\frac{\psi(w)}w\right|\ge\re\int_0^1p(tw)\,dt>808,
\]
and hence
\[
 \re\frac{w\psi'(w)}{\psi(w)}
 \ge1-\frac{P-\rho+r_*}{808}>0,
\]
because $P<820.7003$, $\rho>814.457$ and $r_*=10^{-32}$. This proves strict star-shapedness with respect to the origin. The real-analytic parametrisation $\theta\mapsto\psi(e^{i\theta})$, with nonvanishing derivative, identifies the boundary diffeomorphically with $S^1$.

The first nonlinear coefficient of $\psi$ obeys
\begin{equation}\label{r2:eq:non-disc}
 |c_1|\ge|c^{(0)}_1|-\frac{r_*}{2(m+1)}>0.01645,
\end{equation}
because $2|p_1-p^{(0)}_1|\le\nrm{p-p_0}_W\le r_*$ and $c_1=p_1/(m+1)$. The coefficient restrictions give
\[
 \psi(e^{2\pi i/m}w)=e^{2\pi i/m}\psi(w),\qquad
 \psi(\bar w)=\overline{\psi(w)}.
\]
The domain is therefore $D_m$-invariant. If it were a disc, its centre would be fixed by rotation through $2\pi/m$ and hence would be zero. Its normalised conformal map would then be linear, contrary to~\eqref{r2:eq:non-disc}. Hence $\Omega$ is not a disc. Since $\psi'(0)=p(0)$ and $|p(0)-\rho|\le r_*$, while the exact decimal $\rho$ satisfies $814.457<\rho<814.458$, the bound~\eqref{r2:eq:radius-main} follows. Analyticity and nonvanishing of $\psi'$ near the circle give a real-analytic boundary.

\subsection{Transfer of the field}
At the zero of $F$, $U=1+Kg$ satisfies~\eqref{r2:eq:fixed-disc} initially in the distributional sense, with its $C^1$ boundary traces. The coefficient $|p|^2$ is smooth near $\overline\D$. Dirichlet elliptic regularity, applied successively to $\Delta U=-|p|^2U$ with smooth constant Dirichlet data, gives smoothness up to the boundary. The previously established zero normal trace is therefore the classical normal condition.

Set $\Omega=\psi(\D)$ and $u=U\circ\psi^{-1}$. Formula~\eqref{r2:eq:covariance} proves $(\Delta+1)u=0$. The normal transformation gives $\partial_\nu u=|p|^{-1}\partial_rU=0$, while $u=1$ on the boundary. The function cannot be constant: its boundary value would force that constant to be one, and the constant one does not solve $(\Delta+1)u=0$. This establishes the differential-equation and geometric assertions of Theorem~\ref{r2:thm:main}.

The Green's identity argument in Section~\ref{sec:pompeiu} proves that $\widehat{\one_\Omega}$ vanishes on the unit circle and that $\Omega$ fails the Pompeiu property, completing Theorem~\ref{r2:thm:main}.

\par
\newcommand{\brM}{77}
\newcommand{\brModes}{72}
\newcommand{\brWindow}{160}
\newcommand{\brStates}{300}
\newcommand{\brGaussNodes}{1536}
\newcommand{\brDGrid}{512}
\newcommand{\brTraceGrid}{1024}
\newcommand{\brTraceCutoff}{400}
\newcommand{\brSpectralBits}{2048}
\newcommand{\brTraceBits}{4096}
\newcommand{\brRecordedSeconds}{173.042}
\newcommand{\brWorkers}{4}
\newcommand{\brCentreHash}{272abaf4b002d8b63cf68344616ca501eb1a3c1f2636b27280af4960713c4472}
\newcommand{\brVerifierHash}{375ea2e68225d3419309d5359c7d5d2a589c07bb0985e71f2cd0ac9bc3fc95fb}
\newcommand{\brRadius}{10^{-36}}
\newcommand{\brRhoLo}{237.97}
\newcommand{\brRhoHi}{237.98}
\newcommand{\brAnalyticRadius}{1.00904}
\newcommand{\brDerivativeLo}{237.2785}
\newcommand{\brUnivalentLo}{235.7747}
\newcommand{\brConvexLo}{0.64407}
\newcommand{\brNonDiscLo}{0.0038627}
\newcommand{\brOriginLo}{19.3341}
\newcommand{\brMu}{0.011202}
\newcommand{\brEpsC}{5.276\cdot10^{-13}}
\newcommand{\brEps}{2.580\cdot10^{-11}}
\newcommand{\brSchur}{0.013407}
\newcommand{\brDinf}{0.355923}
\newcommand{\brDiagLo}{0.5561235}
\newcommand{\brGaussError}{9.734\cdot10^{-51}}
\newcommand{\brBLtwo}{1.823}
\newcommand{\brBY}{1.117\cdot10^6}
\newcommand{\brBTone}{6.443\cdot10^{10}}
\newcommand{\brBTtwo}{1.187\cdot10^{13}}
\newcommand{\brHighFactor}{0.033901}
\newcommand{\brP}{240.1776}
\newcommand{\brPone}{295.0400}
\newcommand{\brPtwo}{42522.831}
\newcommand{\brJtwo}{0.793080}
\newcommand{\brBzeroone}{541.556}
\newcommand{\brUb}{68.459}
\newcommand{\brDUb}{17609.075}
\newcommand{\brEzero}{9.699\cdot10^{-67}}
\newcommand{\brEone}{2.728\cdot10^{-62}}
\newcommand{\brA}{1.947\cdot10^{22}}
\newcommand{\brYzero}{1.888\cdot10^{-44}}
\newcommand{\brZ}{2.981\cdot10^{-39}}
\newcommand{\brMtwo}{65827.163}
\newcommand{\brCtwo}{6.406\cdot10^{26}}
\newcommand{\brPolynomial}{-9.99999980\cdot10^{-37}}
\newcommand{\brContraction}{1.282\cdot10^{-9}}
\newcommand{\brFourCY}{4.836\cdot10^{-17}}
\newcommand{\brCoverageRows}{%
Angular order & $0$ & $77$ & $154$ & $231$ & $308$ & $385$ & $462$\\
\midrule
States & $102$ & $91$ & $60$ & $33$ & $13$ & $1$ & $0$\\}
\newcommand{\brCurvatureDiagnostic}{[2.719\cdot10^{-3},\,5.283\cdot10^{-3}]}
\newcommand{\brCSSConvexDiagnostic}{-2.5506}
\newcommand{\brCSSCurvatureDiagnostic}{-2.968}
\newcommand{\brAlphaDiagnostic}{56632.65}
\newcommand{\brScaleDiagnostic}{84.137}
\newcommand{\brRemainderDiagnostic}{8467}
\newcommand{\brThresholdDiagnostic}{92.6\text{--}228.8}
\newcommand{\brRatioDiagnostic}{37}
\newcommand{\brRemainderRows}{%
$5000$ & $3953$\\
$23309.55191881634$ & $7096$\\
$100000$ & $8467$\\}

\part{A convex planar domain for Berenstein's problem}\label{br:part}
\section{The domain and the fixed-disc equation}\label{br:sec:formulation}

The Dirichlet counterpart of Schiffer's problem prescribes zero boundary
value and a nonzero constant normal derivative. Colbrook, Sadeghi and
Stepaniants~\cite{CSS26} constructed a planar counterexample to
Berenstein's conjecture. The construction below gives the first convex
planar counterexample.

\begin{theorem}\label{br:thm:main}
There exist a bounded strictly convex domain $\Omega\subset\R^2$, different
from a disc, with real-analytic boundary and $D_{\brM}$ symmetry, and a
real-valued sign-changing function $u\in C^\infty(\overline\Omega)$ such that
\begin{equation}\label{br:eq:physical}
 \Delta u+u=0\quad\text{in }\Omega,\qquad
 u=0,\quad\partial_\nu u=1\quad\text{on }\partial\Omega.
\end{equation}
Here $\nu$ is the outward unit normal.
The domain is $\psi(\D)$, with $\psi(0)=0$, $\psi'(0)>0$, and
\begin{equation}\label{br:eq:main-geometry}
 \brRhoLo<\psi'(0)<\brRhoHi,\qquad
 \re\left(1+\frac{w\psi''(w)}{\psi'(w)}\right)>\brConvexLo
 \quad(|w|=1).
\end{equation}
The map is holomorphic and univalent on $|w|<R=2^{1/\brM}>\brAnalyticRadius$,
with $\re\psi'(w)>\brUnivalentLo$ there. On $\overline\D$,
$|\psi'|>\brDerivativeLo$. Writing
$\psi(w)=\sum_{j\ge0}c_jw^{jm+1}$, $m=\brM$, one has
$|c_1|>\brNonDiscLo$ and $u(0)>\brOriginLo$.
There is a unique zero of the fixed-disc equation~\eqref{br:eq:F} in
\begin{equation}\label{br:eq:solution-ball}
 \nrm{g-g_0}_Y+\nrm{p-p_0}_{W_1}\le\brRadius,
 \qquad p=\psi'.
\end{equation}
Here $Y$, $W_1$ and $X$ are the spaces of
Subsection~\ref{br:sec:spaces}, and $(g_0,p_0)$ is the clamped centre of
Subsection~\ref{br:sec:centre}.
\end{theorem}

The second quantity in~\eqref{br:eq:main-geometry} is the conformal convexity
factor. The physical curvature is this factor divided by $|\psi'|$.
For the finite centre, a floating-point evaluation gives physical curvature
in the approximate range $\brCurvatureDiagnostic$.

\subsection{The arc-length Fourier transform}\label{br:sec:fourier}
Let $\sigma_{\partial\Omega}$ denote arc-length measure and use the convention
$\widehat\sigma(\xi)=\int_{\partial\Omega}e^{-i\xi\cdot x}\,ds$.
For a smooth bounded
domain with connected Jordan boundary, the existence of a real solution
$u$ of~\eqref{br:eq:physical} is equivalent to
\begin{equation}\label{br:eq:fourier}
 \widehat\sigma_{\partial\Omega}(\xi)=0\qquad(|\xi|=1).
\end{equation}
Indeed, if $v=e^{-i\xi\cdot x}$, Green's identity gives
$0=\int_\Omega(u\Delta v-v\Delta u)
=\int_{\partial\Omega}(u\partial_\nu v-v\partial_\nu u)
=-\widehat\sigma(\xi)$. Conversely, let $\Phi$ be the outgoing fundamental
solution with $(\Delta+1)\Phi=-\delta$ and put
$S(x)=\int_{\partial\Omega}\Phi(x-y)\,ds_y$.
Its far-field amplitude is a nonzero constant multiple of
$\widehat\sigma(\widehat x)$, so~\eqref{br:eq:fourier} and Rellich uniqueness
give $S=0$ in the connected exterior. The single layer is continuous across
the boundary, and its normal derivatives obey
$\partial_\nu S|_{\rm ext}-\partial_\nu S|_{\rm int}=-1$.
Thus $u=\re S|_\Omega$ solves~\eqref{br:eq:physical}; smooth boundary
regularity gives the classical traces. This is the Fourier-circle
equivalence of~\cite[Proposition~2.1]{CSS26}. In particular, the theorem produces a strictly
convex noncircular curve whose arc-length Fourier transform vanishes on
the unit circle.

\subsection{Spaces and lifts}\label{br:sec:spaces}
In this part $m=\brM$, $R=2^{1/m}$, and the symbols $Y$, $\Yc$, $K$, $X_g$
have the meanings of Subsection~\ref{r2:sec:spaces}, with this value of $m$.
The field is real, reflection-even and invariant under rotation through
$2\pi/m$. In particular,
\begin{equation}\label{br:eq:norms}
 \nrm f_Y=\sum_{n\in m\Z}R^{|n|}\sup_{0\le r\le1}|f_n(r)|,
 \qquad
 \nrm b_{W_s}=\sum_{n\in m\Z}R^{|n|}(1+|n|)^s|b_n|.
\end{equation}
The full angular space $\Yc$ permits all integer frequencies. Continuity
is Cartesian, so nonconstant angular coefficients vanish at the origin.
Holomorphic versions of $W_s$ use nonnegative frequencies and real
coefficients. The algebra proof of Lemma~\ref{r2:lem:algebra} applies for
every $R>1$ and every $m$; the additional weight satisfies
$1+|n+k|\le(1+|n|)(1+|k|)$, making $W_s$ a Banach algebra for
$s=0,1,2$. A holomorphic real-coefficient series $h$ satisfies
\begin{equation}\label{br:eq:realpart-norm}
 \nrm{\re h}_{W_s}=\nrm h_{W_s}.
\end{equation}
For each positive frequency, taking the real part puts half the Taylor
coefficient at each of the two opposite Fourier frequencies; their weights
are equal. The constant coefficient is real and retains its weight.

The zero-Dirichlet Poisson inverse $K$ and its clamped range are
\begin{equation}\label{br:eq:clamped-space}
 X_g=\left\{g\in Y:\int_0^1r^{|n|+1}g_n(r)\,dr=0\ (n\in m\Z)\right\},
 \qquad \nrm K_{Y\to Y}\le\tfrac14.
\end{equation}
Lemma~\ref{r2:lem:Poisson} uses neither the Schiffer boundary value nor the
symmetry order: its radial Green formula and moment identity apply
verbatim. Thus $Kg$ has zero Cauchy data precisely when $g\in X_g$.
For a boundary series introduce the Neumann lift
\begin{equation}\label{br:eq:neumann-lift}
 \Bc b=\frac12\sum_n b_nr^{|n|}(r^2-1)e^{in\theta}.
\end{equation}
Its Cauchy traces are $(0,b)$, and
\begin{equation}\label{br:eq:neumann-lift-bounds}
 \nrm{\Bc b}_Y\le\tfrac12\nrm b_{W_0},\qquad
 \nrm{\Delta\Bc b}_Y\le2\nrm b_{W_1}.
\end{equation}
Indeed $\Delta(r^{|n|}(r^2-1)e^{in\theta}/2)
=2(|n|+1)r^{|n|}e^{in\theta}$, and the radial factor has modulus at most
$1/2$ on $[0,1]$.

Write $p=\psi'=\sum_jp_jw^{jm}$ and
$\psi(w)=\int_0^wp(z)\,dz$, so $c_j=p_j/(jm+1)$.
Conformal covariance~\eqref{r2:eq:covariance} and the normal transformation
give the following fixed-disc problem for $U=u\circ\psi$:
\begin{equation}\label{br:eq:pullback}
 \Delta U+|p|^2U=0\quad\text{in }\D,\qquad
 U=0,\quad U_r=|p|\quad\text{on }\partial\D.
\end{equation}
The boundary vector $wp(w)$ is outward normal to the image and has length
$|p(w)|$, so the normal derivative in physical coordinates is
$|p|^{-1}U_r$.
We work on the real Banach space
\begin{equation}\label{br:eq:X}
 X=X_g\times W_1,\qquad
 \nrm{(g,p)}_X=\nrm g_Y+\nrm p_{W_1},\qquad
 U(g,p)=\Bc(|p|)+Kg.
\end{equation}
The shape component $p$ of $X$ is a holomorphic real-coefficient series
in $W_1$: $p=\sum_{j\ge0}p_jw^{jm}$ with $p_j\in\R$. Its reflection
symmetry $p(\bar w)=\overline{p(w)}$ is equivariance; it does not make
$p$ real on $\partial\D$. The field is real-valued and reflection-even
in $Y$. On the ball~\eqref{br:eq:solution-ball}, the boundary function
$|p|$ lies in the full two-sided space $W_1$, and every shape satisfies
$\re p\ge2\rho-P-\brRadius>0$ on the closed weighted disc $|w|\le R$,
with $P$ of~\eqref{br:eq:P}; see~\eqref{br:eq:univalence}. Thus $p$ is
nonzero throughout that disc.
We solve
\begin{equation}\label{br:eq:F}
 F(g,p)=g+\Delta\Bc(|p|)+|p|^2\bigl(\Bc(|p|)+Kg\bigr)=0.
\end{equation}
On the ball $\nrm{p-p_0}_{W_1}<1/J_1$, with $J_1$ as
in~\eqref{br:eq:reciprocals}, write $p=p_0(1+k/p_0)$. The binomial
series for the square roots of $1+k/p_0$ and its reflection converge
in $W_1$, and $|p_0|\in W_1$ by~\eqref{br:eq:b-bound}. Thus $F$
is a real-analytic map from this neighbourhood of the centre in $X$ to $Y$.
The term $\Delta\Bc(|p|)$ requires the $W_1$ shape norm.
For every $(g,p)\in X$ in this neighbourhood, $U(g,p)$ has the exact
Cauchy data $(0,|p|)$, and $F=\Delta U+|p|^2U$ distributionally.
The Schiffer cubic equation and its Dirichlet linear inverse are therefore
replaced by~\eqref{br:eq:F} and a Robin inverse.

\subsection{The exactly clamped centre}\label{br:sec:centre}
The input consists of exact decimal coefficients
$c_j^{(0)}$, $a_j$ and positive divisors $s_j$, $0\le j\le\brModes$.
The divisors are specified positive numbers; no identity with a Bessel
value is assumed. The decimal strings are the entries $0\le j\le72$
of \texttt{p}, \texttt{a} and \texttt{scales} in
\path{anc/berenstein/centre_frozen.json}; the array \texttt{p} stores
$c_j^{(0)}$, while $p_j^{(0)}=(jm+1)c_j^{(0)}$. Put
\begin{equation}\label{br:eq:input}
 \begin{split}
 \psi_0(w)&=\sum_{j=0}^{\brModes}c_j^{(0)}w^{jm+1},\quad
 p_0=\psi_0',\quad \rho=c_0^{(0)},\quad q=|p_0|^2,\\
 u_*(z)&=\sum_{j=0}^{\brModes}\frac{a_j}{s_j}
 J_{jm}(|z|)\cos(jm\arg z),\quad U_*=u_*\circ\psi_0.
 \end{split}
\end{equation}
Each physical summand extends through the origin as an entire real-analytic
Helmholtz solution, so $(\Delta+q)U_*=0$ identically.
The finite conformal polynomial and finite field specify an infinite-dimensional
clamped centre as follows. Let $D_j$ and $N_j$ be the complete cosine
coefficients of
\begin{equation}\label{br:eq:defects}
 U_*(1,\theta)=\sum_{j\ge0}D_j\cos(jm\theta),\qquad
 \partial_rU_*(1,\theta)-|p_0(e^{i\theta})|
 =\sum_{j\ge0}N_j\cos(jm\theta).
\end{equation}
For $n\ge0$ use the polynomials of~\eqref{r2:eq:lifts},
\[
 h_n^D=r^n\bigl(1-n(r^2-1)/2\bigr),\qquad
 h_n^N=r^n(r^2-1)/2,
\]
whose traces $(1,0)$ and $(0,1)$ do not depend on the prescribed Cauchy
data. Define
\begin{equation}\label{br:eq:centre}
 \begin{split}
 \ell&=\sum_{j\ge0}(D_jh_{jm}^D+N_jh_{jm}^N)\cos(jm\theta),\\
 U_0&=U_*-\ell,\qquad
 g_0=\Delta\bigl(U_0-\Bc(|p_0|)\bigr).
 \end{split}
\end{equation}
The strip estimates of Subsection~\ref{br:sec:residual} give convergence with
all derivatives needed near $\overline\D$. Thus $U_0$ has traces
$(0,|p_0|)$, the quantity inside the Laplacian has zero Cauchy data,
$g_0\in X_g$, and Poisson uniqueness gives
$Kg_0=U_0-\Bc(|p_0|)$. In particular $U(g_0,p_0)=U_0$ and
\begin{equation}\label{br:eq:E}
 E:=F(g_0,p_0)=-(\Delta+q)\ell.
\end{equation}
All boundary harmonics enter this definition. The exact zero enclosed
in~\eqref{br:eq:solution-ball} need not have a finite conformal or field
series.
\section{The Robin inverse in the symmetric sector}\label{br:sec:spectrum}
At the centre define
\begin{equation}\label{br:eq:Robin}
 LV=(\Delta+q)V=f,\qquad V_r+\beta V=0\quad\text{on }\partial\D,
 \qquad \beta=1+d,\quad d=\re\frac{wp_0'}{p_0}.
\end{equation}
All operators in this section act in the real even $D_m$ sector, with
interior measure $r\,dr\,d\theta/(2\pi)$ and boundary measure
$d\theta/(2\pi)$. By~\eqref{r2:eq:curvature-formula}, $\beta=|p_0|\kappa_0$,
where $\kappa_0$ is the curvature of $\partial\psi_0(\D)$. The symmetric form associated with $-L$ is
\begin{equation}\label{br:eq:Robin-form}
 \mathfrak a(z,y)=\int_\D\nabla z\cdot\nabla y
 +\int_{\partial\D}\beta zy-\int_\D qzy,
 \qquad z,y\in H^1(\D).
\end{equation}
The form is closed and bounded below after addition of a sufficiently large
multiple of the $L^2$ inner product. Its self-adjoint operator is the Robin
realisation of $-L$. Boundary multiplication is treated as a form on $H^1$.

\subsection{The Dini basis and complete spectral window}\label{br:sec:coverage}
Let $\alpha_{j,k}>0$ run over the roots
\begin{equation}\label{br:eq:Dini-roots}
 \alpha J_n'(\alpha)+J_n(\alpha)=0,\qquad n=jm.
\end{equation}
The Robin-one Laplacian $-\Delta_{R,1}$ has the complete orthonormal Dini
basis
\begin{equation}\label{br:eq:basis}
 e_{j,k}=b_{j,k}(r)\chi_j(\theta),\qquad
 b_{j,k}(r)=\frac{\sqrt2J_n(\alpha_{j,k}r)}
 {|J_n(\alpha_{j,k})|\sqrt{1+(1-n^2)/\alpha_{j,k}^2}},
\end{equation}
where $\chi_0=1$ and $\chi_j=\sqrt2\cos(jm\theta)$ for $j>0$.
The Dini parameter is $1$, with $n+1>0$, so no exceptional nonpositive
radial eigenvalue is present. Orthogonality and completeness also follow
from the separated self-adjoint Robin Sturm--Liouville problems and the
complete angular cosine basis. The identity
\[
 \int_0^1rJ_n(\alpha r)^2\,dr
 =\tfrac12\bigl(J_n'(\alpha)^2+(1-n^2/\alpha^2)J_n(\alpha)^2\bigr)
\]
follows by integrating the Bessel equation and its parameter derivative.
At~\eqref{br:eq:Dini-roots} it gives
$\tfrac12J_n(\alpha)^2(1+(1-n^2)/\alpha^2)$, which is positive:
every root exceeds $n$ for $n\ge1$, and the $n=0$ expression is positive.
In particular the basis functions have nonzero boundary values.

Let $\PP$ select every root satisfying
\begin{equation}\label{br:eq:window}
 |\alpha_{j,k}-\rho|<\omega,\qquad \omega=\brWindow,
\end{equation}
and put $\QQ=I-\PP$. The window contains $\brStates$ states, with the
sector counts of Table~\ref{br:tab:coverage}.
Enumerate these states by $i\leftrightarrow(j,k)$ and write
$n_i=jm$, $\alpha_i=\alpha_{j,k}$, $\chi_i=\chi_j$ and
$\tau_i=b_{j,k}(1)$.
\begin{table}[htbp]
\centering
\begin{tabular}{c rrrrrrr}
\toprule
\brCoverageRows
\bottomrule
\end{tabular}
\caption{Complete Robin-one spectral window. All higher angular orders
have no state in the window.}\label{br:tab:coverage}
\end{table}

\begin{lemma}[Two coverage arguments]\label{br:lem:coverage}
The Dirichlet interlacing and the quarter-step Robin sign grid described
below enumerate all states of~\eqref{br:eq:window}.
\end{lemma}
\begin{proof}
For $y(x)=\sqrt xJ_n(x)$ the Bessel equation gives
\[
 y''+Q_n(x)y=0,\qquad Q_n(x)=1-(n^2-1/4)/x^2.
\]
On an interval with $Q_n\le Q_{\max}$, integration by parts between two
zeros and Wirtinger's inequality give their spacing at least
$\pi/\sqrt{Q_{\max}}$, exactly as in
Lemma~\ref{r2:lem:zero-coverage}. Enclose proposed Dirichlet zeros by
opposite signs in disjoint brackets of width less than $1/10$.
For successive brackets the distance between their outer endpoints is
less than $2\pi/\sqrt{Q_{\max}}$, where $Q_{\max}$ is the maximum
between those endpoints. Three zeros there would require the reverse
inequality, so each bracket contains exactly one zero and none lies
between them. Guard zeros beyond both endpoints of the window
complete this argument whenever a lower Dirichlet zero is present.

For every positive angular order whose first Dirichlet zero lies above
the lower window endpoint, the initial interval requires a separate guard.
This includes $n=m,2m,\ldots,5m$ in the present window. When $n>0$, $J_n$ and $J_n'$ are
positive on $(0,n]$: at a first derivative zero below $n$ the Bessel ODE
would give $J_n''=(n^2/x^2-1)J_n>0$, which cannot be a first downward
crossing. This also excludes a first zero of $J_n$ there. The identity
$(xJ_n')'=(n^2/x-x)J_n>0$ on $(0,n)$ gives $J_n'(n)>0$,
and hence $y(n),y'(n)>0$. Constant-coefficient comparison puts the
first zero after $n$ at distance at least $\pi/(2\sqrt{Q_{\max}})$.
A hidden first zero before the proposed bracket would put the bracketed
zero at distance at least $3\pi/(2\sqrt{Q_{\max}})$. The strict opposite
inequality is imposed on its upper endpoint, with $Q_{\max}$ taken on
the entire interval from $n$ to that endpoint. This proves the initial
Dirichlet guard.

Between consecutive Dirichlet zeros consider $z=xJ_n'/J_n$. It satisfies
$xz'=n^2-x^2-z^2$. Thus $1+z$ runs from $+\infty$ to $-\infty$ and, at any
zero, has derivative $-(x^2+1-n^2)/x<0$. There is exactly one Robin root
in each such gap. In the initial gap it is likewise unique between $n$
and the first Dirichlet zero, since $1+z>0$ at $n$. Each proposed Robin
bracket lies strictly inside its corresponding certified gap. The
location relative to both window endpoints is resolved by strict
inequalities. For all remaining angular orders $n>\rho+\omega$,
positivity up to $n$ excludes any root in the window.

For the second argument put $\mathfrak g_n=J_n+xJ_n'$ and
$\mathfrak d_n=x^2+1-n^2$. On $x>\max(n,1)$, direct differentiation gives
\begin{equation}\label{br:eq:Robin-ode}
 \left(\mathfrak g_n\sqrt{x/\mathfrak d_n}\right)''+
 \left(1+\frac{1/4-n^2}{x^2}+\frac4{\mathfrak d_n}
              -\frac{3x^2}{\mathfrak d_n^2}\right)
 \mathfrak g_n\sqrt{x/\mathfrak d_n}=0.
\end{equation}
The coefficient is at most $21/4$: $\mathfrak d_n>1$, the second term
is at most $1/4$, and the last term is nonpositive. The same spacing
argument gives separation greater than $\pi/\sqrt{21/4}>1$.
Roots are simple because at a root
$\mathfrak g_n'=-(\mathfrak d_n/x)J_n\ne0$; simultaneous vanishing of $\mathfrak g_n$ and
$J_n$ would also force $J_n'=0$. Below $n$ the sign of $\mathfrak g_n$ is positive.
The nonzero endpoint signs at
$x_k=\rho-\omega+k/4$, $0\le k\le8\omega$, therefore determine the number
of roots in every quarter-step cell: one for opposite signs, zero for
equal signs. In a cell crossing $n$ the portion below $n$ contains no
root and the portion above it contains at most one. Each bracket from
the first argument belongs to a distinct sign-changing cell, and all
such cells are represented. This establishes a bijection of the state
lists, including the empty next angular sector.
\end{proof}

\subsection{Matrices and integration}\label{br:sec:matrices}
Write $\mathcal V=q-\rho^2$ and form the exact matrices on $\operatorname{Ran}\PP$
\begin{equation}\label{br:eq:matrices}
 \begin{gathered}
 \mathsf B=\PP\mathcal V\PP,\qquad
 \mathsf B_2=\PP\mathcal V^2\PP,\\
 \mathsf D_{ij}=\int_{\partial\D}d\,\tau_i\chi_i \tau_j\chi_j,
 \qquad \mathsf D_{2,ij}=\int_{\partial\D}d^2\tau_i\chi_i \tau_j\chi_j,\\
 \mathsf H=\diag(\alpha_i^2-\rho^2)-\mathsf B+\mathsf D.
 \end{gathered}
\end{equation}
The plus sign in $\mathsf D$ follows from~\eqref{br:eq:Robin-form}.
The two Gram identities are
\begin{equation}\label{br:eq:grams}
 \mathsf G:=\mathsf B_2-\mathsf B^2
 = (\QQ\mathcal V\PP)^*(\QQ\mathcal V\PP),\qquad
 x^T\mathsf D_2x=\nrm{d\operatorname{Tr}(\PP x)}_2^2.
\end{equation}
Both include all complementary modes; $\mathsf D_2$ is not a truncated
square of $\mathsf D$.

Writing $p_0(w)=\sum_jp_j^{(0)}w^{jm}$, the complex Fourier coefficients
of $q$ are
\[
 q_h(r)=\sum_kp_{k+h}^{(0)}p_k^{(0)}r^{(2k+h)m}\qquad(h\ge0).
\]
Subtract $\rho^2$ from $q_0$ to obtain those of $\mathcal V$, and take
the full two-sided convolution for $\mathcal V^2$. For complex coefficients
$v_h$, the angular factors in the cosine basis are $v_0$ in position
$(0,0)$, $\sqrt2v_j$ in $(0,j)$, and $v_{|i-j|}+v_{i+j}$ for $i,j>0$.
The same rule applies to $d$ and $d^2$.

The radial integrations use $N_G=\brGaussNodes$ Gauss--Legendre nodes.
Disjoint sign brackets of the Legendre polynomial in $(-1,1)$, equal in
number to its degree, enclose all its roots. The weights
$2/[(1-x_i^2)\mathrm{Leg}_{N_G}'(x_i)^2]$ include root uncertainty using
\[
 \max_{[-1,1]}|\mathrm{Leg}_{N_G}''|
 = (N_G-1)N_G(N_G+1)(N_G+2)/8.
\]
Real Bessel arguments are evaluated at exact midpoint arguments and
enlarged by their argument radii, since the integer-order integral
representation gives $|J_n'|\le1$ on the real axis. An exported dyadic
basis value carries its interval evaluation error into every matrix entry.

For the analytic quadrature error use the Bernstein ellipse of parameter
two in $x=2r-1$. It has $|r|\le9/8$ and $|\im r|\le3/8$. Set
\begin{equation}\label{br:eq:Gauss-majorant}
 \begin{split}
 P_e&=\sum_j|p_j^{(0)}|(9/8)^{jm},\qquad V_e=P_e^2+\rho^2,\\
 M_e&=8(1+9/8)(1+V_e^2)e^{3\alpha_{\max}/4}/b_*^2,
 \end{split}
\end{equation}
where $b_*$ is the smallest denominator in~\eqref{br:eq:basis} and
$\alpha_{\max}$ is the largest root in the window. The Bessel integral
gives $|J_n(z)|\le e^{|\im z|}$, so $M_e$ bounds every integrand, including
the angular factors and normalisations. The Chebyshev argument
of~\eqref{r2:eq:Gauss-error} depends only on this majorant and Gauss
exactness; hence the error in each entry is at most
\begin{equation}\label{br:eq:Gauss-error}
 16M_e2^{-2N_G}<\brGaussError.
\end{equation}

For the boundary entries complexify $\phi=m\theta$. On
$|\im\phi|\le\log4$ define
$S_4=\sum_{j>0}4^j|p_j^{(0)}|$ and
$P_{1,4}=\sum_j4^jjm|p_j^{(0)}|$.
The Laurent continuation of $d$ is the half-sum of $wp_0'/p_0$ and its
reflection. For the centre, $S_4<7.762<\rho$; this excludes its poles and gives strip
bound $M_d=P_{1,4}/(\rho-S_4)$. A trapezoidal sum with
$N_d=\brDGrid$ nodes adds to every required Fourier coefficient of $d$
and $d^2$ the error
\begin{equation}\label{br:eq:boundary-matrix-error}
 \frac{4(1+M_d^2)4^{-(N_d-h_{\max})}}{1-4^{-N_d}}.
\end{equation}
Here $h_{\max}$ is the largest Fourier index used by the matrices.
This follows by summing the aliased coefficients, each bounded by its
strip supremum times $4^{-|h|}$.

\subsection{The complement as a quadratic form}\label{br:sec:complement}
Let
\begin{equation}\label{br:eq:complement-constants}
 \gamma_0=\omega(2\rho-\omega),\qquad
 t_Q=\frac{2(1+\sqrt{\gamma_0+\rho^2})}{\gamma_0},\qquad
 v_*=2\rho\sigma_0+\sigma_0^2,\quad d_*=\sigma_1/(\rho-\sigma_0).
\end{equation}
Here $\sigma_0=\sum_{j>0}|p_j^{(0)}|$ and
$\sigma_1=\sum_{j>0}jm|p_j^{(0)}|$. On $\overline\D$,
$\rho-\sigma_0\le|p_0|\le\rho+\sigma_0$ and
$|wp_0'|\le\sigma_1$, so $\nrm{\mathcal V}_\infty\le v_*$ and
$\nrm d_\infty\le d_*$.
On $\operatorname{Ran}\QQ$ set
$D_Q=\QQ(-\Delta_{R,1}-\rho^2)\QQ$. Every omitted state obeys
$|\alpha^2-\rho^2|\ge\gamma_0$: below the window this is
$(\rho-\alpha)(\rho+\alpha)\ge\omega(2\rho-\omega)$, and above the window
the bound is larger. Thus $|D_Q|\ge\gamma_0$.
For $z\in H^1(\D)$ the divergence theorem applied to $x|z|^2$ gives
\begin{equation}\label{br:eq:trace-divergence}
 \nrm{\operatorname{Tr}z}_2^2
 \le2\nrm z_2^2+2\nrm z_2\nrm{\nabla z}_2.
\end{equation}
The estimate extends from smooth functions by $H^1$ approximation.
For $z\in\QQ H^1$, the reference Robin-one form and
$\alpha^2\le|\alpha^2-\rho^2|+\rho^2$ give
\[
 \nrm{\nabla z}_2^2\le
 (1+\rho^2/\gamma_0)\langle|D_Q|z,z\rangle.
\]
Dropping its positive boundary term and inserting this inequality
in~\eqref{br:eq:trace-divergence} proves
\begin{equation}\label{br:eq:trace-energy}
 \nrm{\operatorname{Tr}\QQ|D_Q|^{-1/2}}^2\le t_Q.
\end{equation}
Consequently the complementary form, conjugated by $|D_Q|^{-1/2}$,
is $\operatorname{sgn}D_Q$ plus a bounded self-adjoint perturbation of norm
at most
\begin{equation}\label{br:eq:mu}
 \mu_Q=v_*/\gamma_0+d_*t_Q.
\end{equation}
If $\mu_Q<1$, its inverse in this energy pairing has norm at most
$(1-\mu_Q)^{-1}$. The sign operator is a self-adjoint isometry;
no positivity of $D_Q$ is required.

\begin{proposition}[Robin form Schur bound]\label{br:prop:Schur}
Let $C$ be a real square matrix of dimension $\brStates$ and $\Theta$
a real diagonal matrix with nonzero entries. Suppose
\begin{equation}\label{br:eq:congruence}
 \begin{split}
 \nrm{C^TC-I}_{\rm row}&\le\epsilon_C<1,\\
 \nrm{|\Theta|^{-1/2}(C^T\mathsf HC-\Theta)|\Theta|^{-1/2}}_{\rm row}
 &\le\epsilon.
 \end{split}
\end{equation}
Define
\begin{equation}\label{br:eq:Schur-data}
 \begin{split}
 g_V&=\tr(|\Theta|^{-1}C^T\mathsf GC),\qquad
 g_d=\tr(|\Theta|^{-1}C^T\mathsf D_2C),\\
 \eta_Q&=\frac{(\sqrt{g_V/\gamma_0}+\sqrt{t_Qg_d})^2}{1-\mu_Q},\\
 S_*&=\frac{1+\epsilon_C}{\min_i|\Theta_i|(1-\epsilon-\eta_Q)},\\
 t_P&=\max_j\sum_{i:n_i=jm}\tau_i^2,\qquad
 k_*=v_*/\sqrt{\gamma_0}+d_*\sqrt{t_Qt_P},\\
 \delta_Q&=\frac{k_*}{\sqrt{\gamma_0}(1-\mu_Q)}.
 \end{split}
\end{equation}
If $\mu_Q<1$ and $\epsilon+\eta_Q<1$, the problem~\eqref{br:eq:Robin}
has an $L^2$ inverse with
\begin{equation}\label{br:eq:BLtwo}
 \nrm{L^{-1}}_{2\to2}\le B_{L^2}
 :=(1+\delta_Q^2)S_*+\frac1{\gamma_0(1-\mu_Q)}.
\end{equation}
\end{proposition}
\begin{proof}
The finite spectral projection preserves $H^1$, and the energy space
of $|D_Q|^{1/2}$ is $\QQ H^1$ with an equivalent norm. Thus the
block decomposition of~\eqref{br:eq:Robin-form} is legitimate on
$\operatorname{Ran}\PP\oplus\QQ H^1$. Its finite diagonal block is
$\mathsf H$ and its complementary block is the form just estimated.
Denote the coupling from the finite block into the complementary dual
space by $\mathfrak b$. Its volume part is $-\QQ\mathcal V\PP$; its
boundary part is the pairing with $d\operatorname{Tr}\PP$.
By~\eqref{br:eq:grams} and~\eqref{br:eq:trace-energy}, after multiplication
on the right by $C|\Theta|^{-1/2}$, their energy Hilbert--Schmidt norms
are at most $\sqrt{g_V/\gamma_0}$ and $\sqrt{t_Qg_d}$.
The triangle inequality bounds their sum. Pairing twice with the
complementary inverse yields the relative Schur correction at most
$\eta_Q$.

The matrices in~\eqref{br:eq:congruence} are symmetric, so their maximum
row sums bound their Euclidean operator norms. The first makes $C$
invertible and gives $\nrm C^2\le1+\epsilon_C$. If $\mathsf S$ is the
finite Schur complement, then
$|\Theta|^{-1/2}C^T\mathsf SC|\Theta|^{-1/2}$ is
$\operatorname{sgn}\Theta$ plus a symmetric perturbation of norm at
most $\epsilon+\eta_Q<1$. Its inverse norm is at most
$(1-\epsilon-\eta_Q)^{-1}$, and undoing the congruence gives
$\nrm{\mathsf S^{-1}}\le S_*$.

The finite trace map has squared norm at most $t_P$: angular
orthogonality gives
$\nrm{\operatorname{Tr}\PP x}_2^2
=\sum_j|\sum_{i:n_i=jm}\tau_ix_i|^2\le t_P\nrm x_2^2$.
The unweighted coupling into the complementary energy space therefore
has norm at most $k_*$. The complementary inverse, followed by inclusion
into $L^2$, bounds the off-diagonal solve by $\delta_Q$ and its
$L^2$ diagonal solve by $1/[\gamma_0(1-\mu_Q)]$.
The symmetric block inverse of $-L$ has the form
\[
 \begin{pmatrix}I\\-X_Q\end{pmatrix}\mathsf S^{-1}
 \begin{pmatrix}I&-X_Q^*\end{pmatrix}
 +\begin{pmatrix}0&0\\0&\mathfrak M_Q^{-1}\end{pmatrix},
 \qquad \nrm{X_Q}\le\delta_Q,
\]
where $\mathfrak M_Q^{-1}$ is the complementary form inverse restricted
to $L^2$. This proves~\eqref{br:eq:BLtwo}. The formula constructs a weak
solution for every $L^2$ right side; the invertible complementary and
Schur blocks also prove uniqueness. Smooth coefficients and boundary
give the usual Robin $H^2$ regularity, so this form inverse is the
operator inverse asserted in the proposition.
\end{proof}

For the centre~\eqref{br:eq:input}, the enclosures give
\begin{equation}\label{br:eq:spectral-numbers}
 \begin{gathered}
 d_*<\brDinf<1,\qquad \mu_Q<\brMu,\qquad
 \epsilon_C<\brEpsC,\qquad\epsilon<\brEps,\\
 \min_i|\Theta_i|>\brDiagLo,\qquad
 \eta_Q<\brSchur,\qquad B_{L^2}<\brBLtwo.
 \end{gathered}
\end{equation}
In particular $\beta\ge1-d_*>0$.

\section{Transfer to the working norms}\label{br:sec:transfer}
Define the centre coefficient sums
\begin{equation}\label{br:eq:P}
 \begin{gathered}
 P=\sum_j2^j|p_j^{(0)}|,\qquad
 P_1=\sum_j2^jjm|p_j^{(0)}|,\qquad
 P_2=\sum_j2^j(jm)^2|p_j^{(0)}|,\\
 J_0=(2\rho-P)^{-1},\qquad
 d_0=J_0P_1,\qquad
 d_1=J_0(P_1+P_2)+J_0^2P_1^2.
 \end{gathered}
\end{equation}
The reciprocal is well defined because $P<\brP<2\rho$.
The Neumann series for $1/p_0$ in $W_0$ gives norm at most $J_0$.
For $\mathcal Q_0=wp_0'/p_0$,~\eqref{br:eq:realpart-norm} and differentiation give
\begin{equation}\label{br:eq:d-norms}
 \nrm d_{W_0}\le d_0,\qquad \nrm d_{W_1}\le d_1.
\end{equation}
Indeed $\nrm{wp_0'}_{W_0}=P_1$, while
$\nrm{\partial_\theta \mathcal Q_0}_{W_0}\le J_0P_2+J_0^2P_1^2$.
Also multiplication by $q$ on $Y$ has norm at most $P^2$, and
\begin{equation}\label{br:eq:potential-norm}
 \nrm{q-\rho^2}_Y\le P^2-\rho^2.
\end{equation}
For the latter, remove the constant--constant product in $p_0\overline{p_0}$
and bound $2^{|j-k|}$ by $2^{j+k}$ on every remaining term.

\begin{proposition}[Robin inverse and two trace derivatives]\label{br:prop:transfer}
Let $j_{\rm h}$ be the least positive integer with $j_{\rm h}m>\rho$. Put
\begin{equation}\label{br:eq:high-data}
 \gamma_{\rm h}=(j_{\rm h}m)^2-\rho^2,\qquad
 b_{\rm h}=\frac{P^2-\rho^2}{\gamma_{\rm h}}+\frac{d_0}{1+\sqrt{\gamma_{\rm h}}}.
\end{equation}
Suppose that~\eqref{br:eq:Robin} has an inverse on $L^2$ of the
symmetric sector with norm at most $B_{L^2}$, and that $\beta>0$
and $b_{\rm h}<1$. For $f\in Y$ the Robin solution of
\eqref{br:eq:Robin} belongs to $Y$, its boundary trace belongs to $W_2$,
and
\begin{equation}\label{br:eq:transfer-bounds}
 \nrm V_Y\le B_Y\nrm f_Y,\qquad
 \nrm{\operatorname{Tr}V}_{W_s}\le B_{Ts}\nrm f_Y\quad(s=1,2),
\end{equation}
where, with $P_{\rm un}=\sum_j|p_j^{(0)}|$,
\begin{equation}\label{br:eq:transfer-constants}
 \begin{split}
 \tau_*^2&=2B_{L^2}^2+
 2B_{L^2}\sqrt{P_{\rm un}^2B_{L^2}^2+B_{L^2}},\\
 b_{\rm low}&=\bigl[(1+P^2B_{L^2})/2+\tau_*\bigr]/\sqrt2,\\
 B_Y&=\frac{(1+2\sum_{j=1}^{j_{\rm h}-1}2^j)b_{\rm low}+1/\gamma_{\rm h}}{1-b_{\rm h}},\\
 B_{T1}&=1+P^2B_Y+d_0B_Y,\\
 B_{T2}&=1+P^2B_Y+d_1B_{T1}.
 \end{split}
\end{equation}
\end{proposition}
\begin{proof}
\emph{Energy and low modes.}
The weak equation~\eqref{br:eq:Robin-form}, tested with $V$, gives
\[
 \nrm{\nabla V}_2^2+\int_{\partial\D}\beta V^2
 =\int_\D qV^2-\int_\D fV.
\]
Dropping the positive boundary term and using the $L^2$ inverse yields
\[
 \nrm{\nabla V}_2^2
 \le(P_{\rm un}^2B_{L^2}^2+B_{L^2})\nrm f_2^2.
\]
The divergence estimate~\eqref{br:eq:trace-divergence} then gives
$\nrm{\operatorname{Tr}V}_2\le\tau_*\nrm f_2$.
For each angular coefficient write
$V_n=K_n(f-qV)_n+t_nr^{|n|}$, where $K_n$ is the zero-Dirichlet inverse of the radial operator in
Lemma~\ref{r2:lem:Poisson} and $t_n$ is its boundary trace.
The radial Green kernel in Lemma~\ref{r2:lem:Poisson} is bounded by
$-\log\max(r,s)$ and
$\int_0^1(\log s)^2s\,ds=1/4$. Therefore
\[
 \sup_r|K_n(f-qV)_n|\le\tfrac12\nrm{(f-qV)_n}_{L^2(r\,dr)}.
\]
The three comparisons completing the estimate are
\[
 \nrm f_2\le\nrm f_Y/\sqrt2,\qquad
 \nrm{(f-qV)_n}_{L^2(r\,dr)}\le\nrm{f-qV}_2,
 \qquad |t_n|\le\nrm{\operatorname{Tr}V}_2.
\]
The first follows from Parseval and $\int_0^1r\,dr=1/2$, the second
from orthogonality, and the third from the boundary Fourier coefficient
formula. Thus every coefficient has supremum at most
$b_{\rm low}\nrm f_Y$. The low part $V_{\rm low}$, consisting of $|n|<j_{\rm h}m$,
has weighted norm at most
$(1+2\sum_{j=1}^{j_{\rm h}-1}2^j)b_{\rm low}\nrm f_Y$.

\emph{Construction of the high modes.}
For $|n|\ge j_{\rm h}m$, the operator $-\Delta_n-\rho^2$ has positive radial
potential $n^2/r^2-\rho^2\ge\gamma_{\rm h}$. The maximum principle with
Robin-one boundary data gives interior-forcing norm at most
$1/\gamma_{\rm h}$. This may also be seen by comparing with the constant
supersolution $\nrm{\mathfrak f}_\infty/\gamma_{\rm h}$ for the interior forcing
$\mathfrak f$ of~\eqref{br:eq:high-equations}. The supersolution has
positive Robin boundary datum; apply the comparison to both signs of the real solution.
For a pure boundary source the regular homogeneous solution is a multiple
of $J_{|n|}(\rho r)$, which is positive and increasing because $|n|>\rho$.
Its logarithmic derivative
$z_n=r\rho J_{|n|}'(\rho r)/J_{|n|}(\rho r)$ obeys
\[
 rz_n'=n^2-\rho^2r^2-z_n^2.
\]
The power series at zero gives
$z_n=|n|-\rho^2r^2/[2(|n|+1)]+O(r^4)$, greater than
$\sqrt{n^2-\rho^2r^2}$ for small positive $r$. At a first contact with
this decreasing comparison function, $z_n'=0$ and the difference has
positive derivative, excluding the contact. Consequently the solution
with boundary source $\varsigma_n$ has supremum at most
$|\varsigma_n|/(1+\sqrt{\gamma_{\rm h}})$.

Let $\Pi_H$ be the angular high projection. In the high subspace of $Y$
solve by contraction
\begin{equation}\label{br:eq:high-equations}
 \begin{split}
 (\Delta+\rho^2)\widetilde V_H
 &=\Pi_H\bigl(f-(q-\rho^2)(V_{\rm low}+\widetilde V_H)\bigr),\\
 (\widetilde V_H)_r+\widetilde V_H
 &=-\Pi_H\bigl(d\operatorname{Tr}(V_{\rm low}+\widetilde V_H)\bigr).
 \end{split}
\end{equation}
The interior and boundary estimates, the trace norm $Y\to W_0$ at most
one, and~\eqref{br:eq:d-norms} give Lipschitz constant $b_{\rm h}<1$.
This constructs $\widetilde V_H$ with finite $Y$ norm before any total
analytic norm is assumed.

For its identification with $\Pi_HV$ we first justify the energy
equation. In each mode the constructed solution is regular at zero and
solves its radial ODE with Robin-one boundary source. If $\mathfrak f_n$ and $\varsigma_n$
are the right sides of~\eqref{br:eq:high-equations}, the mode energy
identity is
\[
 \int_0^1\bigl(|(\widetilde V_{H,n})'|^2+
 (n^2/r^2-\rho^2)|\widetilde V_{H,n}|^2\bigr)r\,dr
 +|\widetilde V_{H,n}(1)|^2
 =-\re\int_0^1\mathfrak f_n\overline{\widetilde V_{H,n}}r\,dr
 +\re(\varsigma_n\overline{\widetilde V_{H,n}(1)}).
\]
The weighted supremum sums for the solution and $\mathfrak f_n$, and the weighted
boundary sum for $\varsigma_n$, make the sum of the absolute right sides finite.
Adding the finite $\rho^2L^2$ term shows convergence of the mode series
in $H^1$. The mode equations, first tested against finite angular sums
and then by $H^1$ density, give the weak high-subspace equation and its
boundary pairing. In particular the boundary source is an actual weak
Robin datum.

The original $L^2$ inverse gives $V\in H^2$; testing its weak equation
against a radial test function times $\cos(n\theta)$ gives
\begin{equation}\label{br:eq:mode-Robin}
 V_n'(1)=-(\beta\operatorname{Tr}V)_n.
\end{equation}
Here $V_n$ is $C^2$ on $(0,1]$ since $(f-qV)_n$ is continuous.
Thus $V_H=\Pi_HV$ satisfies the same weak projected equation as
$\widetilde V_H$, with its own occurrence on the right side. Their
difference $z\in\Pi_HH^1$ satisfies
\[
 \nrm{\nabla z}_2^2+\int_{\partial\D}\beta z^2=\int_\D qz^2.
\]
Angular orthogonality gives
$\nrm{\nabla z}_2^2\ge(j_{\rm h}m)^2\nrm z_2^2$.
The inequality $b_{\rm h}<1$ implies $(j_{\rm h}m)^2>P^2\ge\nrm q_\infty$.
Since $\beta>0$, $z=0$. This identifies the analytic high solution
with the original Robin solution.

The high fixed-point equation now gives
$\nrm{V_H}_Y\le\nrm f_Y/\gamma_{\rm h}+b_{\rm h}\nrm V_Y$.
Add the low bound and rearrange to obtain $B_Y$.

\emph{Two boundary derivatives.}
Integrate the radial equation $\Delta_nV_n=(f-qV)_n$ against
$r^{|n|+1}$. The moment identity underlying
Lemma~\ref{r2:lem:Poisson}, now with nonzero Dirichlet trace, reads
\[
 \int_0^1r^{|n|+1}(f-qV)_n\,dr=V_n'(1)-|n|t_n.
\]
Using~\eqref{br:eq:mode-Robin} gives
\begin{equation}\label{br:eq:trace-identity}
 (|n|+1)t_n=-\int_0^1r^{|n|+1}(f-qV)_n\,dr-(dt)_n.
\end{equation}
Since $t\in W_0$ already, summing this formula with weights $R^{|n|}$
first proves $t\in W_1$. Its norm is at most
$\tfrac12\nrm{f-qV}_Y+d_0\nrm t_{W_0}$, bounded by the stated $B_{T1}$.
Multiply the same identity by $|n|+1$ and use
$(|n|+1)/(|n|+2)\le1$. The $W_1$ algebra bound on $dt$ then proves
$t\in W_2$ and the stated $B_{T2}$.
\end{proof}

The centre enclosures give $b_{\rm h}<\brHighFactor<1$ and
\begin{equation}\label{br:eq:transfer-numbers}
 B_Y<\brBY,\qquad B_{T1}<\brBTone,\qquad B_{T2}<\brBTtwo.
\end{equation}
The second boundary derivative is needed because reconstruction of $p$
from its primitive consumes one derivative while the unknown shape
belongs to $W_1$.
\section{All-mode residual and the exact linearisation inverse}\label{br:sec:inverse}
\subsection{Strip bounds and the clamped residual}\label{br:sec:residual}
Complexify $\phi=m\theta$ to $|\im\phi|\le\log4$. Set
\begin{equation}\label{br:eq:strip-data}
 \begin{gathered}
 S_c=\sum_{j>0}4^j|c_j^{(0)}|,\qquad
 \delta_s=(2\rho S_c+S_c^2)/\rho,\qquad
 \varepsilon_s=\tfrac12\log\frac{\rho+S_c}{\rho-S_c},\\
 P_4=\sum_j4^j(jm+1)|c_j^{(0)}|.
 \end{gathered}
\end{equation}
For $0\le r\le1$ put
$A_\pm=\rho+\sum_{j>0}c_j^{(0)}r^{jm}e^{\pm ij\phi}$.
For the centre, $S_c<0.0571$, so $\rho>2S_c$; this places each $A_\pm$ in the disc
$|A-\rho|\le S_c<\rho/2$. Define
\begin{equation}\label{br:eq:strip-root}
 B_s=\sqrt{A_+}\sqrt{A_-}
\end{equation}
using the principal roots of the individual factors. Both roots are
analytic, and $|\arg B_s|<\pi/6$, hence $\re B_s>0$ and
$|B_s+\rho|\ge\rho$. Therefore
\begin{equation}\label{br:eq:strip-geometry}
 |B_s-\rho|=\frac{|A_+A_--\rho^2|}{|B_s+\rho|}\le\delta_s,
 \qquad |A_\pm/B_s|\le e^{\varepsilon_s}.
\end{equation}

The complexified physical radius is $rB_s$. The positive angular branch
of order $n=jm$ is
$J_n(rB_s)e^{ij\phi}(A_+/B_s)^n$, with the reflected expression for
the negative branch. The entire Cartesian Bessel power series interprets
this formula also at $r=0$. For integer order the contour shift
in~\eqref{r2:eq:Bessel-contour} uses only the Bessel integral and gives
\[
 |J_n(z)|\le
 \exp(-n\vartheta+|\re z|\sinh\vartheta+|\im z|\cosh\vartheta)
 \quad(\vartheta\ge0).
\]
Choose $\vartheta_j=0$ for $jm\le\rho+\delta_s$ and otherwise
$\vartheta_j=\acosh(jm/(\rho+\delta_s))$. Define
\begin{equation}\label{br:eq:strip-summands}
 \mathfrak u_j=\frac{|a_j|}{s_j}4^j
 \exp\bigl(jm\varepsilon_s-jm\vartheta_j
       +(\rho+\delta_s)\sinh\vartheta_j+\delta_s\cosh\vartheta_j\bigr).
\end{equation}
Since $|\re(rB_s)|\le\rho+\delta_s$ and $|\im(rB_s)|\le\delta_s$,
$\sum_j\mathfrak u_j$ bounds $U_*$ uniformly on the strip. The coefficient
bound $M4^{-|j|}$ yields the weighted sum factor
$1+2\sum_{j\ge1}2^{-j}=3$. Consequently
\begin{equation}\label{br:eq:field-bounds}
 \begin{split}
 \nrm{U_*}_Y&\le U_b:=3\sum_j\mathfrak u_j,\\
 \nrm{\partial_wU_*}_{\Yc},\ \nrm{\partial_{\bar w}U_*}_{\Yc}
 &\le \mathfrak D_U:=3R4^{2/m}P_4e^{\varepsilon_s}
                      \sum_j\mathfrak u_je^{\vartheta_j}.
 \end{split}
\end{equation}
For the derivative bound use the Wirtinger recurrences displayed in
Subsection~\ref{r2:sec:centre}: a physical derivative changes Bessel
order by one, costs at most $e^{\vartheta_j}$ in the contour bound, and
changes the angular factor by at most $4^{1/m}e^{\varepsilon_s}$.
The chain rule contributes $P_4$. The resulting angular frequencies are
$km\pm1$; multiplying by $e^{\mp i\theta}$ makes them periodic in
$\phi$, and $R^{|km\pm1|}\le R2^{|k|}$ costs the further factors
$R4^{1/m}$. These estimates give the enlarged $\mathfrak D_U$ in
\eqref{br:eq:field-bounds}, including the order-zero term via
the Bessel identity $J_{-1}(z)=-J_1(z)$.

The common strip majorant for both defects in~\eqref{br:eq:defects} is
\begin{equation}\label{br:eq:trace-majorant}
 M_{\rm tr}=U_b+2\mathfrak D_U+P_4.
\end{equation}
Indeed the radial derivative is
$e^{i\theta}U_w+e^{-i\theta}U_{\bar w}$, whose strip bound is at most
$2\mathfrak D_U$ by the estimates just derived. The complexification of $|p_0|$
is the product of the principal square roots of $p_0(e^{i\theta})$ and
its reflected series. The inequality $S_4<\rho$ of Subsection~\ref{br:sec:matrices} makes this product analytic on the strip, with modulus at most
$P_4$. This accounts for the additional term in~\eqref{br:eq:trace-majorant}.

Use $N_b=\brTraceGrid$ interval trapezoidal nodes and retain the cosine
coefficients through $L_b=\brTraceCutoff<N_b/2$. Summing the aliases gives
the common retained-coefficient error
\begin{equation}\label{br:eq:trace-alias}
 e_b=\frac{8M_{\rm tr}4^{-(N_b-L_b)}}{1-4^{-N_b}}.
\end{equation}
For $j>L_b$, $|D_j|,|N_j|\le2M_{\rm tr}4^{-j}$.
Cauchy's estimate gives the same bound for every $j\ge1$, and the
cosine coefficients of $|p_0|$ are bounded by $2P_44^{-j}$. Thus the
series~\eqref{br:eq:centre} for $\ell$ and the series for $\Bc(|p_0|)$
converge with all derivatives on every closed disc of radius less than
$4^{1/m}$. The complete moments are
\begin{equation}\label{br:eq:trace-moments}
 D^{[k]}=\sum_{j\ge0}2^j(jm+1)^k|D_j|,\qquad
 N^{[k]}=\sum_{j\ge0}2^j(jm+1)^k|N_j|\quad(0\le k\le3).
\end{equation}
Complex Fourier coefficients and one-sided cosine coefficients give the
same norm here, since the two half-coefficients at opposite frequencies
sum to the full cosine coefficient. With $J=L_b+1$, the tail added to
each moment is at most
\begin{equation}\label{br:eq:trace-tail}
 2M_{\rm tr}2^{-J}\sum_{i=0}^k\binom{k}{i}
 (mJ+1)^{k-i}m^i\mathfrak m_i,
 \qquad (\mathfrak m_0,\mathfrak m_1,\mathfrak m_2,\mathfrak m_3)
 =(2,2,6,26).
\end{equation}
These four integers are the exact sums
$\sum_{t\ge0}t^i2^{-t}$, as in~\eqref{r2:eq:tail-sums}. Expanding
$(m(J+t)+1)^k$ proves~\eqref{br:eq:trace-tail} and includes every
uncomputed harmonic.

The lift polynomials are identical to those of the Schiffer centre;
their norm estimates depend only on the defect coefficients. In the
present notation they give
\begin{equation}\label{br:eq:lift-constants}
 \begin{aligned}
 L_0&=D^{[0]}+N^{[0]}/2,&
 L_1&=2R(D^{[2]}+N^{[1]}),\\
 L_\Delta&=2(D^{[2]}+N^{[1]}),&
 L_{\Delta,1}&=2R(D^{[3]}+N^{[2]}).
 \end{aligned}
\end{equation}
They bound $\ell$, each Wirtinger derivative of $\ell$, $\Delta\ell$,
and each Wirtinger derivative of $\Delta\ell$, respectively.
For example,
\[
 \begin{split}
 \Delta(h_n^D\cos n\theta)&=-2n(n+1)r^n\cos n\theta,\\
 \Delta(h_n^N\cos n\theta)&=2(n+1)r^n\cos n\theta.
 \end{split}
\]
Writing the lifts in $w,\bar w$ proves the
derivative bounds with one factor $R$ for the shifted angular frequency;
the displayed enlargements include $n=0$.
Since $\nrm{\partial_wq}_{\Yc}\le RPP_1$ and multiplication by $q$
has norm at most $P^2$, equation~\eqref{br:eq:E} gives
\begin{equation}\label{br:eq:residual-bounds}
 \begin{split}
 \nrm E_Y&\le E_0:=L_\Delta+P^2L_0,\\
 \nrm{\partial_wE}_{\Yc},\ \nrm{\partial_{\bar w}E}_{\Yc}
 &\le E_1:=L_{\Delta,1}+RPP_1L_0+P^2L_1,\\
 \nrm{\operatorname{Tr}E}_{W_1}&\le E_{W1}:=E_0+2RE_1.
 \end{split}
\end{equation}
The last bound uses $\partial_\theta E=i(wE_w-\bar wE_{\bar w})$.
Likewise $\Delta U_0=-qU_*-\Delta\ell$ implies
\begin{equation}\label{br:eq:Lambda}
 \nrm{\Delta U_0}_Y\le\Lambda_0:=P^2U_b+L_\Delta,
 \qquad
 \nrm{\partial_w\Delta U_0}_{\Yc}\le
 \Lambda_1:=RPP_1U_b+P^2\mathfrak D_U+L_{\Delta,1},
\end{equation}
with the same bound for the conjugate derivative. In particular,
\begin{equation}\label{br:eq:residual-numbers}
 U_b<\brUb,\qquad \mathfrak D_U<\brDUb,\qquad
 E_0<\brEzero,\qquad E_1<\brEone.
\end{equation}

\subsection{Material derivative and boundary identities}\label{br:sec:linearization}
All quantities below are evaluated at $x_0=(g_0,p_0)$. For a variation
$(\delta g,k)\in X$ define
\begin{equation}\label{br:eq:variation}
 \eta(w)=\int_0^wk(z)\,dz,\qquad
 v=\eta/p_0=wa,\qquad
 T=v\partial_wU_0+\bar v\partial_{\bar w}U_0.
\end{equation}
The symbol $a$ denotes only the holomorphic real-coefficient shape
function in this formula; the $a_j$ in~\eqref{br:eq:input} are the fixed
field coefficients. Put $b=|p_0|$ on the circle and denote its real
derivative by $\mathrm Db[k]=b\re(k/p_0)$.
Differentiation of~\eqref{br:eq:F} yields
\begin{equation}\label{br:eq:DF}
 DF(x_0)(\delta g,k)=LH+\delta q\,U_0,
 \qquad H=K\delta g+\Bc(\mathrm Db[k]),\quad
 \delta q=2\re(\bar p_0k).
\end{equation}
The material identity of the Schiffer part uses only holomorphicity
and the equation $\Delta U_0+qU_0=E$, so it gives here
\begin{equation}\label{br:eq:material}
 LT+\delta q\,U_0=\Ecal(k),\qquad
 \Ecal(k)=vE_w+\bar vE_{\bar w}+(v'+\bar v')E.
\end{equation}
To see all cancellations, $\Delta(vU_w)=v\partial_w\Delta U+v'\Delta U$,
and the conjugate formula gives
\begin{equation}\label{br:eq:DeltaT}
 \Delta T=v\partial_w\Delta U_0+\bar v\partial_{\bar w}\Delta U_0
                  +(v'+\bar v')\Delta U_0.
\end{equation}
The identity $k=(p_0v)'$ gives
$\delta q=vq_w+\bar vq_{\bar w}+(v'+\bar v')q$.
Substitute $\Delta U_0=E-qU_0$ in~\eqref{br:eq:DeltaT}; the terms involving
derivatives of $q$ cancel with $\delta qU_0$, and the other terms involving
$U_0$ cancel with $qT$. This proves~\eqref{br:eq:material}.

The boundary Hessian calculation differs from the Schiffer case.
Exact clamping gives $U_0=0$, $(U_0)_r=b$, $(U_0)_{r\theta}=b_\theta$,
and $(U_0)_{rr}=E-b$ on the circle. Since
$T=r\re a\,(U_0)_r+\im a\,(U_0)_\theta$, radial differentiation gives
\begin{equation}\label{br:eq:T-boundary}
 T=b\re a,\qquad
 T_r=b\re(a+wa')+(E-b)\re a+b_\theta\im a.
\end{equation}
For $\mathcal Q_0=wp_0'/p_0$ one has
\[
 \mathrm Db[k]=b\re\bigl((1+\mathcal Q_0)a+wa'\bigr),\qquad
 b_\theta=-b\im \mathcal Q_0.
\]
The tangential terms cancel in the difference, leaving
\begin{equation}\label{br:eq:boundary-difference}
 \mathrm Db[k]-T_r=(\beta-E/b)T,
 \qquad \beta=1+\re \mathcal Q_0.
\end{equation}
Thus the residual introduces a boundary term as well as the interior
term $\Ecal$. Define its Neumann lift and the corrected linearisation by
\begin{equation}\label{br:eq:corrected}
 \mathcal C(k)=\Bc(E\re a),\qquad
 \widetilde D=DF(x_0)-\Ecal+L\mathcal C.
\end{equation}
Equations~\eqref{br:eq:DF} and~\eqref{br:eq:material} then give
$\widetilde D(\delta g,k)=L(H-T+\mathcal C(k))$.

\subsection{The exact two-sided inverse}\label{br:sec:two-sided}
\begin{proposition}\label{br:prop:inverse}
Suppose that~\eqref{br:eq:Robin} has the inverse and trace bounds of
Proposition~\ref{br:prop:transfer}, and that $P<2\rho$. Then
$\widetilde D\colon X\to Y$ of~\eqref{br:eq:corrected} has a bounded
two-sided inverse $A\colon Y\to X$ with $\nrm A\le A_*$, where $A_*$
is defined in~\eqref{br:eq:inverse-norm-data}.
\end{proposition}
\begin{proof}
\emph{Right inverse.}
Given $f\in Y$, solve~\eqref{br:eq:Robin} by the Robin inverse and let
$t=\operatorname{Tr}V\in W_2$. Since $b$ does not vanish, prescribe
\begin{equation}\label{br:eq:shape-recovery}
 \re a=-t/b.
\end{equation}
The reciprocal estimates in Subsection~\ref{br:sec:inverse-norm} below
give $1/b\in W_2$. The right side is real, even and has frequencies in
$m\Z$. If its cosine coefficients are $\widehat a_j$, define
$a(w)=\sum_{j\ge0}\widehat a_jw^{jm}$. Thus the cosine coefficients of
$-t/b$ are the Taylor coefficients of $a$. Equivalently retain the
constant complex Fourier coefficient and double every positive one.
This is the unique solution in the holomorphic real-coefficient class;
the imaginary constant of the general real-part problem is excluded.
Equation~\eqref{br:eq:realpart-norm} gives equality of the $W_2$ norms.
Define, in this order,
\begin{equation}\label{br:eq:inverse-recipe}
 v=wa,\quad k=(p_0wa)',\quad
 H=V+T-\mathcal C(k),\quad
 \delta g=\Delta\bigl(H-\Bc(\mathrm Db[k])\bigr).
\end{equation}
The boundary value of $H$ is $t+b\re a=0$. Using
$V_r=-\beta t=\beta T$ and~\eqref{br:eq:boundary-difference},
\[
 H_r=\beta T+T_r-E\re a=\mathrm Db[k].
\]
These equalities hold mode by mode by~\eqref{br:eq:mode-Robin}.
The function $H-\Bc(\mathrm Db[k])$ has zero Cauchy data and a Laplacian
in $Y$. The radial moment identity therefore gives $\delta g\in X_g$,
and Poisson uniqueness gives
$K\delta g=H-\Bc(\mathrm Db[k])$.
The interior identity now gives
$\widetilde D(\delta g,k)=LV=f$. The recipe is linear in $f$.

\emph{Derivative count and regularity.}
For an arbitrary $k\in W_1$,
\begin{equation}\label{br:eq:derivative-count}
 \nrm{\eta/w}_{W_2}=\nrm k_{W_1},\qquad
 \nrm{(p_0wa)'}_{W_1}=\nrm{p_0a}_{W_2}.
\end{equation}
Both identities follow coefficient by coefficient from multiplication
or division by $jm+1$. The Robin solve gains two boundary derivatives,
division by $b$ preserves $W_2$, and reconstruction of $k$ spends one.
The Laplacian of the Neumann lift spends one boundary derivative,
which is available from $k\in W_1$. All holomorphic shape functions
are smooth near $\overline\D$ because $R>1$. The Robin solution is a
weak $H^2$ solution, with continuous right side; standard elliptic
regularity gives $W^{2,s}$ for finite $s$ and $C^1$ boundary traces.
The mode-wise boundary identity proved above justifies the zero
Neumann trace of $H-\Bc(\mathrm Db[k])$. The estimates below show that
its Laplacian belongs to $Y$, and hence that the recipe maps $Y$ to $X$.

\emph{Left inverse.}
Start with any $(\delta g,k)\in X$. Form $\eta,v,a,T$ by
\eqref{br:eq:variation}, $H$ by~\eqref{br:eq:DF}, and put
$V=H-T+\mathcal C(k)$. The apparent singularity in
$a=\eta/(wp_0)$ at zero is removable, and
\eqref{br:eq:derivative-count} gives $a\in W_2$.
Its Dirichlet trace is $-b\re a$, and
\eqref{br:eq:boundary-difference} gives $V_r+\beta V=0$.
Also $LV=\widetilde D(\delta g,k)$ distributionally.
The terms involving the fixed smooth centre and the holomorphic shape
are smooth near the circle; $K\delta g$ has $C^1$ traces. Thus this
function is the weak Robin solution. Uniqueness recovers the same $V$
when the right-inverse recipe is applied to $\widetilde D(\delta g,k)$.
Its Dirichlet trace divided by $-b$ recovers $\re a$, and the unique
real-coefficient real-part solution recovers $a$. Then $k$, $T$,
$\mathcal C(k)$, $H$ and $\delta g$ are recovered in turn.
Consequently $A\widetilde D=I_X$ and $\widetilde DA=I_Y$.
\end{proof}

\subsection{Inverse norm and left Newton defect}\label{br:sec:inverse-norm}
The reciprocal bounds required above are
\begin{equation}\label{br:eq:reciprocals}
 J_1=J_0+P_1J_0^2,\qquad
 J_2=J_0+(2P_1+P_2)J_0^2+2P_1^2J_0^3.
\end{equation}
They bound $1/p_0$ in $W_1$ and $W_2$, respectively; $J_2$ also bounds
$1/b$ in $W_2$. To prove the holomorphic bounds differentiate the
reciprocal once and twice in the angular variable and use
$\nrm h_{W_2}=\nrm h_{W_0}+2\nrm{\partial_\theta h}_{W_0}
+\nrm{\partial_\theta^2h}_{W_0}$. For the boundary reciprocal,
$\log b=(\log p_0+\log\overline{p_0})/2$ has first derivative norm at
most $J_0P_1$ and second derivative norm at most
$J_0P_2+J_0^2P_1^2$. The product of the individual square-root reciprocal
series gives $\nrm{1/b}_{W_0}\le J_0$; differentiating
$\exp(-\log b)$ gives exactly the stated $J_2$ enlargement.

With $\delta_p=(P-\rho)/\rho<1$, set
\begin{equation}\label{br:eq:b-bound}
 B_{01}=\rho(1+\delta_p)^2+
                 \frac{1+\delta_p}{\sqrt{1-\delta_p}}P_1.
\end{equation}
This bounds $\nrm b_{W_1}$. Indeed the sum of absolute nonconstant
coefficients in the series for $\sqrt{1+z}$ at $|z|=\delta_p$ is
$1-\sqrt{1-\delta_p}\le\delta_p$. Apply this to $\sqrt{p_0/\rho}$ and
its reflection and multiply. Differentiating a square root adds at
most $P_1/[2\sqrt\rho\sqrt{1-\delta_p}]$; the two product derivatives
give the second term of~\eqref{br:eq:b-bound}. In particular,
\begin{equation}\label{br:eq:Db-bound}
 \nrm{\mathrm Db[k]}_{W_1}\le B_{01}J_1\nrm k_{W_1},
\end{equation}
by~\eqref{br:eq:realpart-norm} applied to $k/p_0$.

Define
\begin{equation}\label{br:eq:inverse-norm-data}
 \begin{split}
 B_a&=J_2B_{T2},\qquad B_p=(P+2P_1+P_2)B_a,\\
 B_g&=1+P^2B_Y+2(R\Lambda_1+\Lambda_0+E_{W1})B_a
                         +2B_{01}J_1B_p,\\
 A_*&=B_g+B_p.
 \end{split}
\end{equation}
For $Af=(\delta g,k)$, equation~\eqref{br:eq:shape-recovery} bounds
$\nrm a_{W_2}\le B_a\nrm f_Y$, and the exact second identity
of~\eqref{br:eq:derivative-count} gives
$\nrm k_{W_1}\le B_p\nrm f_Y$ because
$\nrm{p_0}_{W_2}=P+2P_1+P_2$.
As in~\eqref{r2:eq:v-estimates}, $\nrm v_{\Yc}\le R\nrm a_{W_1}$
and $\nrm{v'}_{\Yc}=\nrm a_{W_1}$.
The product identity~\eqref{br:eq:DeltaT} therefore gives
\[
 \nrm{\Delta T}_Y\le2(R\Lambda_1+\Lambda_0)\nrm a_{W_1}.
\]
Also $\nrm{\Delta\mathcal C(k)}_Y\le2E_{W1}\nrm a_{W_1}$.
Applying these estimates and~\eqref{br:eq:Db-bound} to
\eqref{br:eq:inverse-recipe}, with $\Delta V=f-qV$, proves
$\nrm{\delta g}_Y\le B_g\nrm f_Y$ and $\nrm A\le A_*$.

For an arbitrary input shape,
$\nrm a_{W_2}\le J_2\nrm k_{W_1}$ by the first identity of
\eqref{br:eq:derivative-count}. Hence
\begin{equation}\label{br:eq:Newton-data}
 \begin{split}
 \nrm{AF(x_0)}_X&\le Y_0:=A_*E_0,\\
 \nrm{I_X-A DF(x_0)}_{X\to X}&\le Z,\\
 Z&:=A_*J_2\bigl[2(RE_1+E_0)+2E_{W1}+P^2E_0/2\bigr].
 \end{split}
\end{equation}
Indeed $I_X-A DF=A(-\Ecal+L\mathcal C)$, the material correction has
norm at most $2(RE_1+E_0)\nrm a_{W_1}$, and
$\nrm{L\mathcal C}_Y\le(2E_{W1}+P^2E_0/2)\nrm a_{W_1}$.
The last term includes $q\mathcal C$ as well as its Laplacian.
Both identities of Proposition~\ref{br:prop:inverse} are used: the left
identity gives this defect on $X$, and the right identity makes $A$
injective, so a fixed point of the Newton map is a zero of $F$.

\subsection{Symmetry directions and the fixed parameter}\label{br:sec:kernels}
At a true solution, the Eulerian derivative $\dot u$ associated with
normal displacement $h_\nu$ has boundary data
$\dot u=-h_\nu$ and $\partial_\nu\dot u=\kappa h_\nu$.
Thus $\partial_\nu\dot u+\kappa\dot u=0$, agreeing with
\eqref{br:eq:Robin}. Translations have angular character $\pm1$ and
are absent from the $D_m$ invariant sector. The rotation derivative is
reflection-odd and is absent from its real even subspace.
The real conformal coefficients and $\psi(0)=0$, $\psi'(0)>0$ fix the
conformal normalisation.
The constant coefficient $\widehat a_0$ of $a$ generates the radial
reparametrisation $w\mapsto(1+\varepsilon\widehat a_0)w$ of the parameter
disc. It is retained in the inverse and equals the mean of $-t/b$
in~\eqref{br:eq:shape-recovery}. Euclidean dilation is not a homogeneous kernel with
eigenvalue and flux fixed. The flux-preserving family
$u_s(x)=s\,u(x/s)$ has parameter $s^{-2}$; its derivative
$u-x\cdot\nabla u$ satisfies $(\Delta+1)(u-x\cdot\nabla u)=2u$.
Multiplying the field by a constant changes its prescribed flux.
The absence of a kernel within the working sector follows from the
Robin inverse and the two-sided construction above.
\section{Existence, geometry and earlier results}\label{br:sec:existence}
\subsection{The full nonlinear bound}\label{br:sec:nonlinear}
Fix $r_*>0$ with $J_1r_*<1/4$ and consider the closed $X$ ball about
$x_0$ of that radius. Put
\begin{equation}\label{br:eq:ball-data}
 \begin{gathered}
 P_b=P+r_*,\qquad B_b=4B_{01},\qquad J_b=2J_1,\qquad U_c=U_b+L_0,\\
 U_{\rm ball}=U_c+r_*/4+B_bJ_br_*/2.
 \end{gathered}
\end{equation}
On the ball, $\nrm p_{W_0}\le P_b$, $\nrm{|p|}_{W_1}\le B_b$,
$\nrm{1/p}_{W_1}\le J_b$, and $\nrm{U(g,p)}_Y\le U_{\rm ball}$.
For the reciprocal and square root write $p=p_0(1+k/p_0)$.
The $W_1$ algebra gives $\nrm{k/p_0}_{W_1}<1/4$; its Neumann series
bounds the reciprocal by $(4/3)J_1\le J_b$. Multiplying the square-root
series of $1+k/p_0$ and its reflection bounds the modulus by
$4B_{01}=B_b$. Along the segment from $p_0$ to $p$, the formula
$\mathrm D|p|[k]=|p|\re(k/p)$ bounds its change in $W_1$ by
$B_bJ_b\nrm k_{W_1}$. The lift bound and $\nrm K\le1/4$ give the last
bound in~\eqref{br:eq:ball-data}.

\begin{lemma}\label{br:lem:nonlinear}
Throughout the ball, $F$ has second derivative bound
\begin{equation}\label{br:eq:Mtwo}
 \nrm{F''}_{X\times X\to Y}\le M_2
 :=(4+P_b^2)B_bJ_b^2+2P_bB_bJ_b+P_b+2U_{\rm ball}.
\end{equation}
\end{lemma}
\begin{proof}
Differentiating the modulus derivative once more and applying the
$W_1$ product inequality gives
$\nrm{\mathrm D^2|p|[k_1,k_2]}_{W_1}
\le2B_bJ_b^2\nrm{k_1}_{W_1}\nrm{k_2}_{W_1}$.
Let $q_p=|p|^2$ for the varying shape; $q$ continues to denote the fixed
centre potential $|p_0|^2$. For $x_i=(g_i,k_i)$ the full second derivative is
\begin{equation}\label{br:eq:Fsecond}
 \begin{split}
 F''[x_1,x_2]={}&(\Delta+|p|^2)\Bc(\mathrm D^2|p|[k_1,k_2])\\
 &+\mathrm Dq_p[k_1]\bigl(\Bc(\mathrm D|p|[k_2])+Kg_2\bigr)\\
 &+\mathrm Dq_p[k_2]\bigl(\Bc(\mathrm D|p|[k_1])+Kg_1\bigr)
   +\mathrm D^2q_p[k_1,k_2]U.
 \end{split}
\end{equation}
Here $\nrm{\mathrm Dq_p[k]}_Y\le2P_b\nrm k_{W_1}$ and
$\nrm{\mathrm D^2q_p[k_1,k_2]}_Y\le2\nrm{k_1}_{W_1}\nrm{k_2}_{W_1}$.
The first line of~\eqref{br:eq:Fsecond} is bounded by
$(4+P_b^2)B_bJ_b^2$, using both bounds
in~\eqref{br:eq:neumann-lift-bounds}. The two shape/lift products give
$2P_bB_bJ_b$, and the two shape/Poisson products are bounded by $P_b$
using $\nrm K\le1/4$. The final product gives $2U_{\rm ball}$.
For unit $X$ directions these enlarged bounds sum to~\eqref{br:eq:Mtwo}.
\end{proof}

\subsection{The radii polynomial}\label{br:sec:contraction}
Set $C_2=A_*M_2/2$. The Newton map $\mathcal T(x)=x-AF(x)$ satisfies,
for $\nrm{x-x_0}_X\le r_*$,
\begin{equation}\label{br:eq:Newton-map-bounds}
 \nrm{\mathcal T(x)-x_0}_X\le Y_0+Zr_*+C_2r_*^2,
 \qquad \nrm{D\mathcal T(x)}\le Z+2C_2r_*.
\end{equation}
These follow from Taylor's integral remainder,~\eqref{br:eq:Newton-data}
and Lemma~\ref{br:lem:nonlinear}. Thus
\begin{equation}\label{br:eq:radii}
 Y_0+(Z-1)r_*+C_2r_*^2<0,
 \qquad Z+2C_2r_*<1
\end{equation}
make the ball invariant and the map a contraction. The fixed-point
argument in Lemma~\ref{r2:lem:contraction} applies with the quadratic
Taylor majorant~\eqref{br:eq:Newton-map-bounds} in place of the cubic
expansion. The contraction has a unique fixed point;
there $AF(x)=0$. The identity $\widetilde DA=I_Y$ implies $F(x)=0$.
Every zero in the ball is a fixed point, so uniqueness is uniqueness of
the zero of~\eqref{br:eq:F}.

Take $r_*=\brRadius$. The bounds in Table~\ref{br:tab:constants} give
both strict inequalities in~\eqref{br:eq:radii}.
\begin{table}[htbp]
\centering
\begin{tabular}{ll}
\toprule
Quantity & Enclosed bound\\
\midrule
$B_{L^2}$ & $<\brBLtwo$\\
$B_Y$ & $<\brBY$\\
$B_{T1}$ & $<\brBTone$\\
$B_{T2}$ & $<\brBTtwo$\\
$E_0$ & $<\brEzero$\\
$E_1$ & $<\brEone$\\
$A_*$ & $<\brA$\\
$Y_0$ & $<\brYzero$\\
$Z$ & $<\brZ$\\
$M_2$ & $<\brMtwo$\\
$C_2$ & $<\brCtwo$\\
$Y_0+(Z-1)r_*+C_2r_*^2$ & $<\brPolynomial$\\
$Z+2C_2r_*$ & $<\brContraction$\\
$4C_2Y_0$ & $<\brFourCY$\\
\bottomrule
\end{tabular}
\caption{The certificate at $r_*=\brRadius$. The last row gives
$4C_2Y_0<(1-Z)^2$, the discriminant condition for the radii polynomial.}
\label{br:tab:constants}
\end{table}

\subsection{Geometry on the whole ball}\label{br:sec:geometry}
Let $\sigma_0,\sigma_1$ be the sums defined
after~\eqref{br:eq:complement-constants}.
For every shape in the ball, the $W_1$ variation norm bounds both
$\sum_j|p_j-p_j^{(0)}|$ and
$\sum_jjm|p_j-p_j^{(0)}|$ by $r_*$. Consequently, on $\overline\D$,
\begin{equation}\label{br:eq:whole-ball-geometry}
 |p|\ge\rho-\sigma_0-r_*>\brDerivativeLo,
 \qquad
 \re(1+wp'/p)\ge1-\frac{\sigma_1+r_*}{\rho-\sigma_0-r_*}
 >\brConvexLo.
\end{equation}
On the larger disc $|w|\le R$,
\begin{equation}\label{br:eq:univalence}
 \re p(w)\ge2\rho-P-r_*>\brUnivalentLo>0.
\end{equation}
For distinct $w_1,w_2$ inside that disc, the difference quotient
\[
 \frac{\psi(w_2)-\psi(w_1)}{w_2-w_1}
 =\int_0^1p(w_1+t(w_2-w_1))\,dt
\]
has positive real part, proving injectivity. Thus $\psi$ is univalent
and $p$ nonzero on a neighbourhood of $\overline\D$; $R>\brAnalyticRadius$.
The boundary is a real-analytic Jordan curve. The tangent-angle
calculation~\eqref{r2:eq:curvature-formula} is a purely conformal
identity and applies to the present boundary:
\begin{equation}\label{br:eq:curvature}
 \kappa(\theta)=\frac{\re(1+wp'/p)}{|p|}>0.
\end{equation}
Its positive curvature and total turning $2\pi$ make the bounded image
strictly convex.

For its first nonlinear coefficient,
\begin{equation}\label{br:eq:nondisc}
 |c_1|\ge|c_1^{(0)}|-\frac{r_*}{2(m+1)^2}>\brNonDiscLo.
\end{equation}
Indeed the $j=1$ term in $\nrm{p-p_0}_{W_1}$ is
$2(m+1)|p_1-p_1^{(0)}|$, and $c_1=p_1/(m+1)$.
The real coefficient and frequency restrictions give
\[
 \psi(e^{2\pi i/m}w)=e^{2\pi i/m}\psi(w),\qquad
 \psi(\bar w)=\overline{\psi(w)}.
\]
Thus $\Omega$ has $D_m$ symmetry. A disc invariant under this rotation
would be centred at zero, and its normalised Riemann map would be
linear, contradicting~\eqref{br:eq:nondisc}. Finally
$|\psi'(0)-\rho|\le r_*$ and the exact centre coefficient give
$\brRhoLo<\psi'(0)<\brRhoHi$.

\subsection{The field and its signs}\label{br:sec:sign}
At the zero, $U=\Bc(|p|)+Kg$ solves~\eqref{br:eq:pullback}
distributionally. Initially $U$ is continuous, with the Cauchy traces
given by the lift and Lemma~\ref{r2:lem:Poisson}. The coefficient
$|p|^2$ is real analytic near $\overline\D$. Dirichlet elliptic
regularity upgrades $U$ to a smooth solution on the closed disc,
and the Cauchy traces are classical. With
$u=U\circ\psi^{-1}$, conformal covariance and the normal transformation
give exactly~\eqref{br:eq:physical}.

At the origin the finite field has value $a_0/s_0$. Here $s_0$ is
the positive field divisor from~\eqref{br:eq:input}, while the shape
sums are $\sigma_0,\sigma_1$. The lift and ball bounds give
\begin{equation}\label{br:eq:origin}
 U(0)\ge\frac{a_0}{s_0}-L_0-r_*/4-B_bJ_br_*/2>\brOriginLo>0.
\end{equation}
For any boundary point $x$, smoothness and the outward derivative one
give $u(x-t\nu(x))=-t+O(t^2)<0$ for small positive $t$.
Thus $u$ changes sign. This completes the proof of
Theorem~\ref{br:thm:main} and, by the Green identity of
Subsection~\ref{br:sec:fourier}, its arc-length Fourier formulation.

\subsection{Relation with earlier work}\label{br:sec:literature}
\subsubsection{Sign, spectral sequences and low eigenvalues}
Serrin's symmetry theorem~\cite{Ser71}, in the form recorded
in~\cite[Theorem~2]{Dal14}, concerns a bounded connected open set
$\Omega$ with $C^2$ boundary and $u\in C^2(\overline\Omega)$ satisfying
$\Delta u+\lambda u+\varepsilon=0$, $\lambda>0$,
$\varepsilon\in\{0,\pm1\}$, and $u=0$, $\partial_\nu u=c$ on
$\partial\Omega$: if $u>0$ in $\Omega$, then $\Omega$ is a ball. Changing the sign of the solution gives the negative case.
The field in~\eqref{br:eq:origin} has both signs.
Berenstein's planar spectral-sequence theorem~\cite{Ber80}, as stated
in~\cite[Proposition~1.1]{Dal10} for simply connected bounded planar
domains with $C^{2,\varepsilon}$ boundary and in~\cite[\S0]{BY87}, forces a disc
when the Dirichlet overdetermined problem is solvable for infinitely
many distinct spectral parameters. Berenstein and Yang~\cite[Theorem 1]{BY87}
give the corresponding result for bounded Lipschitz domains with
connected boundary in all dimensions, for either the Neumann or
Dirichlet overdetermined problem. The present construction supplies
one parameter. These sequence results imply that the domain of
Theorem~\ref{br:thm:main} has only finitely many overdetermined Dirichlet
eigenvalues.

Dalmasso~\cite[Theorem~1]{Dal14} proves that a bounded connected domain
with $C^1$ boundary is a ball if $\Delta u+\lambda u=0$ has a
nonconstant solution $u\in C^2(\overline\Omega)$ with constant
Dirichlet and Neumann data and $\lambda\le\lambda_1$. If the domain
is also convex and symmetric about a hyperplane, the same conclusion
holds for $\lambda\le\lambda_2$.
The geometric and symmetry hypotheses of the latter statement hold
here. Its spectral inequality does not: the trial functions
$1-r^2$ and $r^n(1-r^2)\cos n\theta$, $r^n(1-r^2)\sin n\theta$,
$1\le n\le6$, span a thirteen-dimensional zero-Dirichlet space on
the disc. As in~\eqref{r2:eq:dirichlet-minmax}, their Rayleigh quotients
are $2(n+1)(n+3)\le126$. Conformal invariance of energy and
$|p|>\brDerivativeLo$ therefore give
\begin{equation}\label{br:eq:low-index}
 \lambda_2(\Omega)\le\lambda_{13}(\Omega)
 \le126/(\brDerivativeLo)^2<1.
\end{equation}
For the finite centre, comparison with the inscribed and circumscribed
origin-centred discs, using binary64 Bessel zeros, shows that between
$14037$ and $14040$ Dirichlet eigenvalues do not exceed one.
The second-Dirichlet-eigenvalue case of the Schiffer problem is recorded
in~\cite[\S0]{BY87} as a consequence of the Payne--Weinberger
isoperimetric inequality and attributed to Berenstein~\cite{Ber80}
in~\cite[Remark~1]{Dal14}. The corresponding application of Dai, Liu
and Sun~\cite{DLS26} is quoted in~\cite[\S1]{DSWZ25}. These statements
concern the Schiffer boundary data.

Dalmasso's finiteness theorem~\cite[Theorem 1.1]{Dal10} concerns
$\Delta u+\lambda u+\mu=0$ with $u=0$, $\partial_\nu u=c\ne0$.
For a bounded convex planar domain with smooth positive-curvature
boundary, the Schiffer property and nonconstant width, it allows at
most finitely many pairs $(\lambda,\mu)$ at fixed $c$.
It is a finiteness statement and permits the single pair $(1,0)$.
Liu's criterion~\cite{Liu07}, as recorded in~\cite{CSS26,DSWZ25},
involves constancy of the second interior normal derivative.
Here the boundary decomposition of the Laplacian gives directly
\begin{equation}\label{br:eq:second-normal}
 u_{\nu\nu}=-\kappa u_\nu=-\kappa\quad\text{on }\partial\Omega:
\end{equation}
the tangential second derivative of the zero boundary trace vanishes,
and $\Delta u=-u=0$ there. Constant curvature would make the boundary
a circle, so~\eqref{br:eq:nondisc} also excludes constancy of
$u_{\nu\nu}$.

\subsubsection{Dai--Sun--Wei--Zhang and the numerical remainder comparison}
Dai, Sun, Wei and Zhang~\cite[Theorem 1.1, version 2]{DSWZ25} prove
Dirichlet rigidity above a threshold $\alpha_0(\Omega)$ for bounded
uniformly convex planar domains with connected $C^{2,\varepsilon}$
boundary, $0<\varepsilon<1$. The geometric hypotheses hold here.
The threshold depends on the full geometry of the fixed domain.
In support coordinates in Steiner position, let $h(\theta)$ be the
support function and $\varrho=h+h''$ the radius of curvature.
Their amplitude defect is
\begin{equation}\label{br:eq:DSWZ-defect}
 \mathfrak B(\Omega)=\max_{\theta,\,\upsilon\in\{-1,1\}}
 \left|\sqrt{\varrho(\theta)}e^{\upsilon h'(\theta)}
 -\sqrt{\varrho(\theta+\pi)}e^{-\upsilon h'(\theta+\pi)}\right|.
\end{equation}
This is the quantity $\mathcal B(\Omega)$ of~\eqref{r2:eq:DSWZ-defect}.

For a controlled $C^5$ class, put
\[
 I_D(\lambda,\upsilon,\theta)=\int_{\partial\Omega}
 e^{\upsilon\eta_\perp\cdot x+i\lambda\xi\cdot x}\,ds,
 \qquad \xi=(\cos\theta,\sin\theta),\quad\eta_\perp\perp\xi.
\]
Its two-point leading term is
\begin{equation}\label{br:eq:DSWZ-leading}
 \sqrt{\frac{2\pi}{\lambda}}\left[
 \sqrt{\varrho(\theta)}e^{\upsilon h'(\theta)}e^{i\lambda h(\theta)-i\pi/4}
 +\sqrt{\varrho(\theta+\pi)}e^{-\upsilon h'(\theta+\pi)}
                          e^{-i\lambda h(\theta+\pi)+i\pi/4}\right].
\end{equation}
Writing $R_D$ for the difference, their expansion gives
$|R_D|\le C_*\lambda^{-3/2}$ for $\lambda\ge3$, $|\upsilon|\le1$.
The constant $C_*$ depends on the class and is not supplied numerically.
Their exclusion lemma, for $\mathfrak B\ge\delta>0$, has threshold
\begin{equation}\label{br:eq:DSWZ-threshold}
 \alpha>\left[\max\left\{3,\frac{C_*}{\sqrt{2\pi}\delta}\right\}\right]^2-1.
\end{equation}
The qualitative theorem chooses $\delta=\mathfrak B/2$.

The following comparison uses floating-point quadrature for the finite centre. Let $\Omega_1$ be its image rescaled to conformal radius
one; its parameter is $\alpha_u=\rho^2\approx\brAlphaDiagnostic$.
Any constant admissible for its stationary-phase expansion must
dominate every evaluated remainder. In the unit scale define
\[
 G_u(\zeta)=\zeta^{3/2}\max_\theta|R_D^{\Omega_1}(\zeta,0,\theta)|.
\]
The computed values in Table~\ref{br:tab:remainder} use trapezoidal
integration in the conformal angle and the support coordinates of the
two stationary points.
\begin{table}[htbp]
\centering
\begin{tabular}{rr}
\toprule
$\zeta$ & Sampled $G_u(\zeta)$\\
\midrule
\brRemainderRows
\bottomrule
\end{tabular}
\caption{Numerical stationary-phase remainders for the normalised
Berenstein centre in binary64 arithmetic. Each sampled angular maximum
estimates a lower bound for $G_u(\zeta)$.}
\label{br:tab:remainder}
\end{table}

The comparison also includes dilation. For $s\Omega_1$ the parameter
is $\alpha_u/s^2$, and
\[
 R_D^{s\Omega_1}(\lambda,0,\theta)
 =sR_D^{\Omega_1}(s\lambda,0,\theta),\qquad
 \mathfrak B(s\Omega_1)=\sqrt s\,\mathfrak B_u(s),
\]
where
$\mathfrak B_u(s)=\max_{\theta,\upsilon=\pm1}
|\sqrt\varrho\,e^{\upsilon sh'}-\sqrt{\varrho(\theta+\pi)}e^{-\upsilon sh'(\theta+\pi)}|$
uses the unit-scale support data. An admissible $C_*$ for a class
containing $s\Omega_1$ satisfies
\[
 C_*\ge\lambda'^{3/2}|R_D^{s\Omega_1}(\lambda',0,\theta)|,
 \qquad \lambda'\ge3.
\]
Taking the angular maximum gives
$C_*\ge s^{-1/2}G_u(s\lambda')$. With the most favourable choice
$\delta=\mathfrak B(s\Omega_1)$, exclusion of $\alpha_u/s^2$
by~\eqref{br:eq:DSWZ-threshold} requires
$\lambda_s=\sqrt{\alpha_u/s^2+1}>3$ and
\[
 C_*<\sqrt{2\pi}\lambda_s\mathfrak B(s\Omega_1)
 =\sqrt{2\pi(\alpha_u+s^2)}s^{-1/2}\mathfrak B_u(s).
\]
Hence exclusion requires
\begin{equation}\label{br:eq:DSWZ-rescaled}
 G_u(\zeta)<T_D(s):=\sqrt{2\pi(\alpha_u+s^2)}\,\mathfrak B_u(s)
 \quad\text{for every }\zeta\ge3s.
\end{equation}
For $s>\sqrt{\alpha_u/8}$ the parameter is below eight and
cannot exceed~\eqref{br:eq:DSWZ-threshold}. On the remaining scales
$0<s\le\sqrt{\alpha_u/8}\approx\brScaleDiagnostic$, the numerical
calculation gives $T_D(s)$ approximately in the range
$\brThresholdDiagnostic$. The sampled remainder at the largest
frequency in Table~\ref{br:tab:remainder} is approximately
$\brRemainderDiagnostic$, about $\brRatioDiagnostic$ times the
largest computed threshold. In this computation the condition
\eqref{br:eq:DSWZ-rescaled} fails at every admissible scale, so the
stationary-phase exclusion does not remove the finite centre.

Version~1 of~\cite{DSWZ25} required central symmetry. The present
domain is not centrally symmetric. Any centre of symmetry would be
fixed by its nontrivial rotational symmetry and hence be the origin.
Then normalised Riemann-map uniqueness would imply $\psi(-w)=-\psi(w)$.
Since $m=\brM$ is odd, its nonzero coefficient $c_1$ multiplies the
even power $w^{m+1}$, which contradicts that identity.

\subsubsection{The planar counterexample of Colbrook--Sadeghi--Stepaniants}
Colbrook, Sadeghi and Stepaniants~\cite{CSS26} proved existence on
a bounded simply connected non-disc domain with real-analytic Jordan
boundary, $D_{13}$ symmetry and a sign-changing eigenfunction.
For their numerical centre
$\varphi^\circ(w)=w+\sum_{j=1}^{20}\varphi_j^\circ w^{13j+1}$,
the binary64 evaluation of its appendix coefficients in the present paper gives
\begin{equation}\label{br:eq:CSS-computation}
 \min_{|w|=1}\re(1+w(\varphi^\circ)''/(\varphi^\circ)')
       \approx\brCSSConvexDiagnostic,\qquad
 \min\kappa^\circ\approx\brCSSCurvatureDiagnostic.
\end{equation}
Here $\kappa^\circ$ is the curvature of
$\varphi^\circ(\partial\D)$ in the normalisation
$(\varphi^\circ)'(0)=1$. The centre curve is not convex: one of its
points lies at distance about $1.0\cdot10^{-2}$ inside the boundary of
its convex hull. The exact boundary of~\cite{CSS26} differs from the
centre curve pointwise by less than $4.67\cdot10^{-10}$~\cite[Subsection~5.1]{CSS26},
so their domain is not convex. The interval triangle witness is given in
\path{anc/diagnostics/berenstein/csb_nonconvex.json}; its construction is
described in Subsection~\ref{br:sec:verification}.
Theorem~\ref{br:thm:main} supplies a convex domain.
The conformal fixed-disc method is due to~\cite{CS26},
and~\cite{CSS26} adapted it to the Dirichlet problem. The coupled inverse here uses the Dirichlet trace of
a Robin solution to recover the shape, together with the explicit
residual Neumann lift~\eqref{br:eq:corrected}.

\part{Planar reductions: convex counterexamples in dimensions four, six, eight, ten and fourteen}\label{part:planar}

\section{Reduction to a planar problem}\label{sec:reduction}

Sections~\ref{sec:reduction}--\ref{sec:recon} contain the proof for $n=4$; the proof for $n=6$ is given in Section~\ref{sec:r6}. Write points of the meridian plane as $(x_1,y)\in\R^2$, identified with $z=x_1+iy\in\C$, and let $\Pi\colon\R^4\to\R^2$, $\Pi(x_1,x')=(x_1,|x'|)$. The map $\Pi$ is continuous and its image is the closed half-plane $\{y\geq0\}$.

\begin{lemma}\label{lem:laplace-identity}
Let $f$ be a $C^2$ function on an open subset of $\{y>0\}$ and $F(x_1,y)=yf(x_1,y)$. Then
\[
F_{x_1x_1}+F_{yy}=y\Big(f_{x_1x_1}+f_{yy}+\frac2yf_y\Big).
\]
The expression in parentheses, evaluated at $(x_1,|x'|)$, is the Laplacian of the function $x\mapsto f(x_1,|x'|)$ at every point with $x'\neq0$.
\end{lemma}

\begin{proof}
$(yf)_{yy}=yf_{yy}+2f_y$. The second statement is the formula $\partial_y^2+\frac2y\partial_y$ for the Laplacian of $\R^3$ acting on radial functions.
\end{proof}

\begin{lemma}\label{lem:odd-division}
Let $N\subset\R^2$ be open and symmetric under $(x_1,y)\mapsto(x_1,-y)$, and let $F$ be real analytic on $N$ and odd in $y$. Then there is a real-analytic function $G$, defined on an open set containing $\{(x_1,y^2):(x_1,y)\in N\}$, such that $F(x_1,y)=y\,G(x_1,y^2)$ on $N$. Consequently $x\mapsto G(x_1,|x'|^2)$ is real analytic on the open set $\Pi^{-1}(N)$, and it equals $F(x_1,|x'|)/|x'|$ where $x'\neq0$.
\end{lemma}

\begin{proof}
For $t>0$ with $(x_1,\sqrt t)\in N$ put $G(x_1,t)=F(x_1,\sqrt t)/\sqrt t$; this is real analytic. Let $(a,0)\in N$. For some $\delta>0$, $F(x_1,y)=\sum_{i,k\geq0}c_{ik}(x_1-a)^iy^k$ converges absolutely for $|x_1-a|<\delta$, $|y|<\delta$, and $c_{ik}=0$ for even $k$ by oddness. Then $G(x_1,t)=\sum_{i,j}c_{i,2j+1}(x_1-a)^it^j$ converges for $|x_1-a|<\delta$, $|t|<\delta^2$ and agrees with the previous definition where $0<t<\delta^2$. Two such local definitions agree on their overlap, which is a connected product of intervals containing points with $t>0$. This defines $G$. The last assertion follows because $x\mapsto(x_1,|x'|^2)$ is polynomial.
\end{proof}

\begin{proposition}[From the plane to $\R^4$]\label{prop:rotation}
Let $D\subset\R^2$ be a bounded domain, symmetric under $y\mapsto-y$, whose boundary is a Jordan curve meeting the axis $\{y=0\}$ in exactly two points. Let $F$ be real analytic on an open set $N\supset\overline D$ symmetric under $y\mapsto-y$, odd in $y$, and such that
\begin{equation}\label{eq:P}
\Delta F+F=0\ \text{ in }D,\qquad F=y\ \text{ and }\ \nabla F=e_y\ \text{ on }\partial D .
\end{equation}
Put $\Omega=\Pi^{-1}(D)$ and let $u$ be the function $x\mapsto G(x_1,|x'|^2)$ of Lemma~\ref{lem:odd-division}. Then $\Omega$ is a bounded open set, $u$ is real analytic on the neighbourhood $\Pi^{-1}(N)$ of $\overline\Omega$, $u$ is nonconstant, and \eqref{eq:main} holds.
\end{proposition}

\begin{proof}
$\Omega$ is open and bounded because $\Pi$ is continuous and proper. By continuity $\overline\Omega\subset\Pi^{-1}(\overline D)\subset\Pi^{-1}(N)$, and $\partial\Omega\subset\Pi^{-1}(\partial D)$, since a point $x\in\overline\Omega$ with $\Pi(x)\in D$ lies in the open set $\Omega$.

Off the axis $\{x'=0\}$, $u(x)=f(x_1,|x'|)$ with $f=F/y$, so Lemma~\ref{lem:laplace-identity} gives $|x'|(\Delta u+u)(x)=(\Delta F+F)(\Pi(x))=0$ in $\Omega\setminus\{x'=0\}$. By continuity of $\Delta u+u$ the equation holds in $\Omega$.

Let $x\in\partial\Omega$ with $y=|x'|>0$. Then $\Pi(x)\in\partial D$, so $f=F/y=1$ there. Since $\nabla F=e_y$ on $\partial D$, we have $f_{x_1}=F_{x_1}/y=0$ and $f_y=(F_y-f)/y=0$, hence $\nabla u(x)=(f_{x_1},f_y\,x'/|x'|)=0$. The set $\partial\Omega\cap\{x'\neq0\}=\Pi^{-1}(\partial D\cap\{y>0\})$ is dense in $\partial\Omega$, since $\partial D$ meets the axis in only two points and each of them is a limit of points of $\partial D\cap\{y>0\}$. By continuity $u=1$ and $\nabla u=0$ on all of $\partial\Omega$. Finally, a constant solution of $\Delta u+u=0$ vanishes identically, which contradicts $u=1$ on $\partial\Omega$.
\end{proof}

The reduction is reversible in the analytic class, which explains why nothing is lost by working in the plane.

\begin{proposition}[From $\R^4$ to the plane]\label{prop:converse}
Let $\Omega\subset\R^4$ be a bounded domain invariant under $O(3)$ acting on $x'$, and let $u$ be real analytic on a neighbourhood of $\overline\Omega$ and solve \eqref{eq:main}. Put $D=\{(x_1,y):(x_1,|y|,0,0)\in\Omega\}$ and $\tilde u(x)=\int_{O(3)}u(x_1,Qx')\,dQ$ \textup{(}Haar probability measure\textup{)}. Then $\tilde u$ also solves \eqref{eq:main}, it has the form $\tilde u(x)=f(x_1,|x'|)$ with $f$ real analytic and even in $y$, and $F=yf$ solves \eqref{eq:P}.
\end{proposition}

\begin{proof}
The Laplacian commutes with rotations, and the boundary conditions are rotation invariant; hence $\tilde u$ solves \eqref{eq:main}. The function $f(x_1,y)=\tilde u(x_1,y,0,0)$ is real analytic and even in $y$ (use the rotation $x'\mapsto-x'$), and $\tilde u(x)=f(x_1,|x'|)$ by invariance. Lemma~\ref{lem:laplace-identity} gives $\Delta F+F=0$ in $D\cap\{y\neq0\}$, and by analyticity in $D$. On $\partial D\cap\{y>0\}$ we have $f=1$ and $\nabla f=0$, hence $F=y$ and $\nabla F=f e_y+y\nabla f=e_y$; continuity extends this to the points of $\partial D$ on the axis.
\end{proof}

\begin{remark}[Splittings without potential]\label{rem:why4}
For a function on $\R^d$ that depends only on $y=|\xi|$, $\xi\in\R^d$, the Laplacian is $\partial_y^2+(d-1)y^{-1}\partial_y$, and for $F=y^{(d-1)/2}f$ one has
\[
\partial_y^2F=y^{(d-1)/2}\Big(f_{yy}+\frac{d-1}{y}f_y\Big)+\frac{(d-1)(d-3)}{4}\,\frac{F}{y^2}.
\]
Let $\R^n=\R^{d_1}\times\R^{d_2}$ and let $u(x)=f(y_1,y_2)$ be invariant under $O(d_1)\times O(d_2)$, where $y_i$ is the modulus of the $i$-th component of $x$ (for $d_i=1$ one may use the coordinate itself). Then $F=y_1^{(d_1-1)/2}y_2^{(d_2-1)/2}f$ satisfies
\[
F_{y_1y_1}+F_{y_2y_2}=y_1^{(d_1-1)/2}y_2^{(d_2-1)/2}\,\Delta u+\Big(\frac{(d_1-1)(d_1-3)}{4y_1^2}+\frac{(d_2-1)(d_2-3)}{4y_2^2}\Big)F .
\]
The inverse-square potential vanishes exactly when $d_1,d_2\in\{1,3\}$, that is, for $n=d_1+d_2\in\{2,4,6\}$. In these cases the weight $\varphi=y_1^{(d_1-1)/2}y_2^{(d_2-1)/2}$ is one of the harmonic polynomials $1$, $y$, $y_1y_2$, and the boundary conditions $u=1$, $\nabla u=0$ become $F=\varphi$, $\nabla F=\nabla\varphi$. For $n=4$ ($d_1=1$, $d_2=3$) this is Lemma~\ref{lem:laplace-identity}, and for $n=6$ ($d_1=d_2=3$) it is Lemma~\ref{lem:laplace6} below. The conformal method uses in an essential way that the planar operator is the Laplacian, whose conformal pull-back is a multiple of the Laplacian. For $n=3,5,7$ every splitting into two blocks contains a block of dimension different from $1$ and $3$, and the meridian problem contains a first-order drift term, or equivalently an inverse-square potential; a different fixed-domain formulation is needed, and it is given in Section~\ref{sec:r3} for $n=3$, in Section~\ref{sec:r5} for $n=5$ and in Section~\ref{r7:sec} for $n=7$. Symmetry classes of adjoint type, which give in addition the dimensions $8$, $10$ and $14$, are treated in Section~\ref{lie:sec}.
\end{remark}

Let $x_1+iy=re^{i\phi}$ be planar polar coordinates, and let $U_\ell$ denote the Chebyshev polynomial of the second kind, so that $U_\ell(\cos\phi)$ is the zonal spherical harmonic of degree $\ell$ on $S^3$. Under $u=F/y$, the planar Helmholtz solution $F=J_m(r)\sin m\phi$, with $m$ odd, corresponds to the four-dimensional Helmholtz solution
\[
u=r^{-1}J_m(r)\,U_{m-1}(\cos\phi).
\]
For instance, $F=J_1(r)\sin\phi$ gives the radial solution $r^{-1}J_1(r)$, whose derivative is $-r^{-1}J_2(r)$; this is why the Schiffer balls of $\R^4$ at eigenvalue one have radii $j_{2,k}$, the positive zeros of $J_2$.

\section{The conformal fixed-disc formulation}\label{sec:conformal}

Let $\D=\{w\in\C:|w|<1\}$, $w=re^{i\theta}$. The following lemma pulls \eqref{eq:P} back to the disc; it is the analogue of the reformulation in \cite{CS26}.

\begin{lemma}\label{lem:pullback}
Let $\psi$ be holomorphic and injective on a neighbourhood of $\overline\D$, with $\psi'\neq0$ there, and let $D=\psi(\D)$. Let $V\in C^1(\overline\D)$ be real valued and satisfy
\begin{equation}\label{eq:V}
\Delta V+|\psi'|^2V=0\ \text{ in }\mathcal D'(\D),\qquad V-\im\psi=0\ \text{ and }\ \nabla(V-\im\psi)=0\ \text{ on }\partial\D .
\end{equation}
Then $F=V\circ\psi^{-1}\in C^1(\overline D)$ satisfies $\int_D(\nabla F\cdot\nabla\varphi-F\varphi)\,dA=0$ for all $\varphi\in C^\infty_c(D)$, and $F=y$, $\nabla F=e_y$ on $\partial D$. If moreover $\psi$ is odd with real Taylor coefficients, $V(\overline w)=-V(w)$ and $V(-\overline w)=V(w)$, then $D$ is symmetric under both coordinate reflections, $F$ is odd in $y$ and even in $x_1$.
\end{lemma}

\begin{proof}
For $\varphi\in C_c^\infty(D)$, the function $\varphi\circ\psi$ belongs to $C_c^\infty(\D)$. Conformal invariance of the Dirichlet integral and the Jacobian $|\psi'|^2$ give
\[
\int_D\nabla F\cdot\nabla\varphi\,dA=\int_\D\nabla V\cdot\nabla(\varphi\circ\psi)\,dA=\int_\D|\psi'|^2V\,(\varphi\circ\psi)\,dA=\int_DF\varphi\,dA,
\]
where the middle equality is the distributional equation for $V\in C^1$. The function $V-\im\psi=(F-y)\circ\psi$ has zero value and gradient on $\partial\D$; since the Jacobian matrix of $\psi$ is invertible on $\partial\D$, the chain rule gives $F-y=0$ and $\nabla(F-y)=0$ on $\partial D$. If $\psi$ is odd with real coefficients, then $\psi(\overline w)=\overline{\psi(w)}$ and $\psi(-\overline w)=-\overline{\psi(w)}$, which gives the symmetries of $D$ and, with the stated parities of $V$, those of $F$.
\end{proof}

We shall look for $V$ of the form $V=Kg+\im\psi$, where $g$ is an unknown function on $\D$ and $K$ is an inverse of the Laplacian with zero Cauchy data on $\partial\D$ (Section~\ref{sec:K}). Since $\im\psi$ is harmonic, \eqref{eq:V} becomes $g+|\psi'|^2(Kg+\im\psi)=0$, which is \eqref{eq:intro-cubic}. We normalise the eigenvalue to one; the size of $D$ is then an unknown, encoded in the coefficient $c_1$ of $\psi$. The symmetry class removes the degeneracies of the problem: the scaling is fixed by the eigenvalue, translations along the axis are excluded by the symmetry $x_1\mapsto-x_1$, and the M\"obius automorphisms of the disc are excluded by requiring $\psi(0)=0$ with real Taylor coefficients.

\section{Disk polynomials and coefficient spaces}\label{sec:disk}

\subsection{Disk polynomials}
Write $w=re^{i\theta}$ and $t=r^2$. For $m\in\Z$ and an integer $s\geq0$ let
\begin{equation}\label{eq:Phi}
\Phi_{m,s}(w)=r^{|m|}P_s^{(0,|m|)}(2r^2-1)\,e^{im\theta},
\end{equation}
where $P^{(\alpha,\beta)}_s$ is the Jacobi polynomial in the standard normalisation, so that $P^{(0,n)}_s(1)=1$ and $\Phi_{m,s}=e^{im\theta}$ on $\partial\D$. Explicitly,
\begin{equation}\label{eq:jacobi}
P_s^{(0,n)}(2t-1)=\sum_{k=0}^sa^{(n)}_{s,k}t^k,\qquad a^{(n)}_{s,k}=(-1)^{s-k}\frac{(n+s+k)!}{k!\,(s-k)!\,(n+k)!},
\end{equation}
and, equivalently, $P_s^{(0,n)}(2t-1)=\frac1{s!}t^{-n}\frac{d^s}{dt^s}\{t^{n+s}(t-1)^s\}$. We set $a^{(n)}_{s,k}=0$ for $k\notin\{0,\ldots,s\}$ and $\Phi_{m,-1}=0$. Each $\Phi_{m,s}$ is a polynomial in $w,\bw$ of degree $|m|+2s$; these are the disk (Zernike) polynomials \cite{Koo78}. Complex conjugation gives $\overline{\Phi_{m,s}}=\Phi_{-m,s}$.

\begin{lemma}\label{lem:orth}
The disk polynomials have the following properties.
\begin{enumerate}[label=\textup{(\alph*)}]
\item For $n,s\geq0$, the polynomial $P^{(0,n)}_s(2t-1)$ is orthogonal in $L^2([0,1],t^n\,dt)$ to all polynomials of degree less than $s$.
\item The functions $\Phi_{m,s}$ are pairwise orthogonal in $L^2(\D)$, and $\int_\D|\Phi_{m,s}|^2\,dA=\pi/(|m|+2s+1)$.
\item $|\Phi_{m,s}(w)|\leq1$ for $|w|\leq1$.
\end{enumerate}
\end{lemma}

\begin{proof}
(a) Integrate the Rodrigues formula by parts $s$ times against $t^n\cdot t^i$, $i<s$; all boundary terms vanish because the derivatives of order less than $s$ of $t^{n+s}(t-1)^s$ vanish at $t=0$ and $t=1$.

(b) Angular integration separates different $m$. For equal $m$ with $n=|m|$, $\int_\D\Phi_{m,s}\overline{\Phi_{m,s'}}\,dA=\pi\int_0^1t^nP_s^{(0,n)}P_{s'}^{(0,n)}\,dt$, which vanishes for $s\neq s'$ by (a). For $s=s'$, by (a) only the leading coefficient $a^{(n)}_{s,s}=\frac{(n+2s)!}{s!(n+s)!}$ of one factor contributes, and $s$ integrations by parts in the Rodrigues formula give $\int_0^1t^{n+s}P^{(0,n)}_s\,dt=\int_0^1t^{n+s}(1-t)^s\,dt=\frac{(n+s)!\,s!}{(n+2s+1)!}$. The product is $1/(n+2s+1)$.

(c) Let $|w|\leq1$, $h=\sqrt{1-|w|^2}$, and consider the unitary matrix $U=\left(\begin{smallmatrix}w&-h\\h&\bw\end{smallmatrix}\right)$, acting on the space of homogeneous polynomials of degree $N=n+2s$ in two variables $e_1,e_2$, with the $U(2)$-invariant inner product in which the monomials are orthogonal. The induced operator is unitary, so its diagonal matrix coefficients with respect to the normalised monomials have modulus at most one. The coefficient of $e_1^{n+s}e_2^{s}$ in $(we_1+he_2)^{n+s}(-he_1+\bw e_2)^{s}$ is
\[
\sum_{k=0}^s\binom{n+s}{k}\binom{s}{k}w^{n+s-k}\bw^{s-k}(-h^2)^k=e^{in\theta}r^n\sum_{k=0}^s\binom{n+s}{k}\binom sk t^{s-k}(t-1)^k .
\]
By Leibniz's rule applied to the Rodrigues formula, the last sum equals $P^{(0,n)}_s(2t-1)$. This proves (c) for $m=n\geq0$; conjugation gives $m<0$.
\end{proof}

Fix $\rho>1$ (we shall take $\rho=21/20$). Let $\Bc_\rho$ be the space of complex sequences $f=(f_{m,s})_{m\in\Z,s\geq0}$ with
\[
\nrm f_\rho=\sum_{m,s}\rho^{|m|}|f_{m,s}|<\infty .
\]
By Lemma~\ref{lem:orth}(c), $\sum f_{m,s}\Phi_{m,s}$ converges absolutely and uniformly on $\overline\D$, with $\sup_{\overline\D}|f|\leq\nrm f_\rho$; by Lemma~\ref{lem:orth}(b) the coefficients are determined by the function. We identify $f$ with this continuous function. $\Bc_\rho$ is a Banach space, and every bounded linear operator on it (or between such weighted $\ell^1$ spaces) is determined by its values on the basis vectors.

\begin{lemma}[Multiplication by $w$ and $\bw$]\label{lem:mult}
Let $n=|m|$. Then
\begin{equation}\label{eq:rec}
w\,\Phi_{m,s}=
\begin{cases}
\dfrac{n+s+1}{n+2s+1}\Phi_{m+1,s}+\dfrac{s}{n+2s+1}\Phi_{m+1,s-1}, & m\geq0,\\[8pt]
\dfrac{n+s}{n+2s+1}\Phi_{m+1,s}+\dfrac{s+1}{n+2s+1}\Phi_{m+1,s+1}, & m<0,
\end{cases}
\end{equation}
and $\bw\,\Phi_{m,s}=\overline{w\,\Phi_{-m,s}}$. In both cases the coefficients are nonnegative and sum to one. Consequently $\nrm{wf}_\rho\leq\rho\nrm f_\rho$ and $\nrm{\bw f}_\rho\leq\rho\nrm f_\rho$, and for every holomorphic $p(w)=\sum_{a\geq0}p_aw^a$ with $N_\rho(p):=\sum_a\rho^a|p_a|<\infty$,
\begin{equation}\label{eq:algebra}
\nrm{pf}_\rho\leq N_\rho(p)\nrm f_\rho,\qquad \nrm{\overline pf}_\rho\leq N_\rho(p)\nrm f_\rho .
\end{equation}
Moreover, if $m\geq0$ then $w^N\Phi_{m,s}$ is a convex combination of $\Phi_{m+N,s'}$ with $s'\leq s$.
\end{lemma}

\begin{proof}
For $m=n\geq0$ both sides of the first formula are $r^{n+1}e^{i(n+1)\theta}$ times a polynomial in $t$, and the claim is the identity $a^{(n)}_{s,k}=\frac{n+s+1}{n+2s+1}a^{(n+1)}_{s,k}+\frac{s}{n+2s+1}a^{(n+1)}_{s-1,k}$. Dividing by $a^{(n+1)}_{s,k}$ (which is nonzero for $0\leq k\leq s$) and using \eqref{eq:jacobi}, it reduces to $(n+s+1)(n+s+k+1)-s(s-k)=(n+k+1)(n+2s+1)$. For $m=-n<0$ the claim is $a^{(n)}_{s,k-1}=\frac{n+s}{n+2s+1}a^{(n-1)}_{s,k}+\frac{s+1}{n+2s+1}a^{(n-1)}_{s+1,k}$; dividing by $a^{(n-1)}_{s,k}$ for $0\leq k\leq s$ it reduces to $(n+s)(s+1-k)-(s+1)(n+s+k)=-k(n+2s+1)$, and the case $k=s+1$ is checked directly. The formula for $\bw$ follows from $\overline{\Phi_{m,s}}=\Phi_{-m,s}$. Since $|m\pm1|\leq|m|+1$, the weight grows by at most a factor $\rho$, which gives the bounds for $w$ and $\bw$; \eqref{eq:algebra} follows by iteration and absolute summation. The last assertion follows from the case $m\geq0$ of \eqref{eq:rec}.
\end{proof}

\begin{lemma}[Radial index under monomials]\label{lem:radial}
Let $\alpha,\beta\geq0$, $m\in\Z$, $s\geq0$. Then $w^\alpha\bw^\beta\Phi_{m,s}$ is a finite linear combination of $\Phi_{m+\alpha-\beta,s'}$ with $s'\geq s-\max(\alpha,\beta)$.
\end{lemma}

\begin{proof}
Put $m'=m+\alpha-\beta$. The product equals $r^{\alpha+\beta+|m|}P^{(0,|m|)}_s(2t-1)\,e^{im'\theta}$, so, by Lemma~\ref{lem:orth}(b) its coefficient on $\Phi_{m',s'}$ is a multiple of
\[
\int_0^1t^{|m|}P^{(0,|m|)}_s(2t-1)\,t^{e}P^{(0,|m'|)}_{s'}(2t-1)\,dt,\qquad e=\frac{\alpha+\beta+|m'|-|m|}2 .
\]
Since $|m'|-|m|\equiv\alpha+\beta\pmod 2$ and $\big||m'|-|m|\big|\leq|\alpha-\beta|$, $e$ is an integer with $0\leq\min(\alpha,\beta)\leq e\leq\max(\alpha,\beta)$. The second factor is a polynomial of degree $e+s'$, so the integral vanishes if $e+s'<s$ by Lemma~\ref{lem:orth}(a).
\end{proof}

\begin{lemma}[Monomials]\label{lem:monomial}
Let $\alpha\geq\beta\geq0$. Then $w^\alpha\bw^\beta=\sum_{s=0}^{\beta}\kappa_s\Phi_{\alpha-\beta,s}$ with $\kappa_s\geq0$, $\sum_s\kappa_s=1$ and $\kappa_0=\dfrac{\alpha-\beta+1}{\alpha+1}$.
\end{lemma}

\begin{proof}
$\bw^\beta=\Phi_{-\beta,0}$. Applying \eqref{eq:rec} $\alpha$ times preserves nonnegativity and total mass one, raises the radial index by at most one in each of the first $\beta$ steps and never afterwards. The harmonic coefficient is the $L^2$ projection onto $r^{\alpha-\beta}e^{i(\alpha-\beta)\theta}$, namely $\int_0^1r^{2\alpha+1}dr\big/\int_0^1r^{2(\alpha-\beta)+1}dr$.
\end{proof}

\subsection{The sine sector}
For odd $n\geq1$ and $s\geq0$ let
\[
S_{n,s}=\im\Phi_{n,s}=r^nP_s^{(0,n)}(2r^2-1)\sin n\theta=\frac{\Phi_{n,s}-\Phi_{-n,s}}{2i}.
\]
Let $\Sc$ be the real Banach space of $f=\sum_{n\ \mathrm{odd},\,s\geq0}f_{n,s}S_{n,s}$ with $\nrm f_\rho=\sum\rho^n|f_{n,s}|<\infty$; this is also its norm as an element of $\Bc_\rho$. The elements of $\Sc$ are real-valued, odd under $w\mapsto\bw$ and even under $w\mapsto-\bw$. Conversely, since $\Phi_{m,s}(\bw)=\Phi_{-m,s}(w)$ and $\Phi_{m,s}(-\bw)=(-1)^m\Phi_{-m,s}(w)$, every real-valued $f\in\Bc_\rho$ with these two symmetries belongs to $\Sc$. The functions $S_{n,0}=r^n\sin n\theta$ are harmonic; we call $f_{n,0}$ the harmonic coefficients and put
\[
\Gc=\{g\in\Sc:\ g_{n,0}=0\text{ for all }n\}.
\]

\section{The compatible inverse of the Laplacian}\label{sec:K}

For odd $n\geq1$ and $s\geq1$ put $q=n+2s$ and define
\begin{equation}\label{eq:K}
KS_{n,s}=\frac{S_{n,s-1}}{4q(q+1)}-\frac{S_{n,s}}{2q(q+2)}+\frac{S_{n,s+1}}{4(q+1)(q+2)} .
\end{equation}

\begin{lemma}\label{lem:K}
For odd $n\geq1$ and $s\geq1$:
\begin{enumerate}[label=\textup{(\alph*)}]
\item $\Delta KS_{n,s}=S_{n,s}$;
\item $KS_{n,s}=0$ and $\nabla KS_{n,s}=0$ on $\partial\D$;
\item $KS_{n,s}=\dfrac{(1-r^2)^2}{4s(s+1)}\,r^nP_{s-1}^{(2,n)}(2r^2-1)\sin n\theta$;
\item the absolute sum of the coefficients in \eqref{eq:K} is $1/(q(q+2))$. Hence $K$ extends to a bounded operator $K\colon\Gc\to\Sc$ with $\nrm K=1/15$.
\end{enumerate}
\end{lemma}

\begin{proof}
Write $\mathsf A,\mathsf B,\mathsf C$ for the three coefficients in \eqref{eq:K} and $a_{s,k}=a^{(n)}_{s,k}$.

(a) Since $\Delta(r^{n+2k}\sin n\theta)=4k(n+k)r^{n+2k-2}\sin n\theta$, the claim is equivalent to
$\mathsf Aa_{s-1,k+1}+\mathsf Ba_{s,k+1}+\mathsf Ca_{s+1,k+1}=a_{s,k}/(4(k+1)(n+k+1))$ for $0\leq k\leq s$. By \eqref{eq:jacobi} the three ratios $a_{s-1,k+1}/a_{s,k}$, $a_{s,k+1}/a_{s,k}$, $a_{s+1,k+1}/a_{s,k}$ equal $(s-k)(s-k-1)$, $-(n+s+k+1)(s-k)$ and $(n+s+k+2)(n+s+k+1)$, each divided by $(k+1)(n+k+1)$, and a direct computation gives
$\mathsf A(s-k)(s-k-1)-\mathsf B(n+s+k+1)(s-k)+\mathsf C(n+s+k+2)(n+s+k+1)=\tfrac14$.

(b) At $r=1$ each $r^nP_t^{(0,n)}(2r^2-1)$ equals one, and $\mathsf A+\mathsf B+\mathsf C=0$. Its radial derivative at $r=1$ is $\beta_t=n+2t(t+n+1)$, because $\frac{d}{dx}P^{(0,n)}_t(1)=t(t+n+1)/2$. Since $\beta_s-\beta_{s-1}=2q$ and $\beta_{s+1}-\beta_s=2(q+2)$, we get $\mathsf A\beta_{s-1}+\mathsf B\beta_s+\mathsf C\beta_{s+1}=-2q\mathsf A+2(q+2)\mathsf C=0$. The tangential derivative vanishes because the boundary value vanishes identically.

(c) By (b), $KS_{n,s}=r^nR(t)\sin n\theta$ where $R$ is a polynomial of degree $s+1$ with $R(1)=0$ and $nR(1)+2R'(1)=0$, so $R(t)=(1-t)^2Q(t)$ with $\deg Q=s-1$. As a combination of $P^{(0,n)}_{s-1},P^{(0,n)}_s,P^{(0,n)}_{s+1}$, the polynomial $R$ is orthogonal in $L^2(t^ndt)$ to polynomials of degree less than $s-1$ (Lemma~\ref{lem:orth}(a)); hence $Q$ is orthogonal in $L^2(t^n(1-t)^2dt)$ to them and is a multiple of $P^{(2,n)}_{s-1}(2t-1)$. Comparing leading coefficients, $\mathsf C\,a_{s+1,s+1}$ for $R$ and $\frac{(n+2s)!}{(s-1)!(n+s+1)!}$ for $P^{(2,n)}_{s-1}(2t-1)$, gives the factor $1/(4s(s+1))$.

(d) $\mathsf A-\mathsf B+\mathsf C=\frac{(q+2)+2(q+1)+q}{4q(q+1)(q+2)}=\frac1{q(q+2)}$. The operator $K$ does not change the angular index, hence not the weight. Since $q\geq3$, the supremum of $1/(q(q+2))$ is $1/15$, attained at $(n,s)=(1,1)$.
\end{proof}

\begin{lemma}\label{lem:traces}
For $g\in\Gc$, the function $W=Kg$ belongs to $C^1(\overline\D)$, satisfies $\Delta W=g$ in $\mathcal D'(\D)$, and $W=0$, $\nabla W=0$ on $\partial\D$.
\end{lemma}

\begin{proof}
Let $g_N$ be the truncations of $g$ to finitely many coefficients and $W_N=Kg_N$, a polynomial with $\Delta W_N=g_N$ and zero Cauchy data by Lemma~\ref{lem:K}. The Dirichlet Green function of the disc is $G(w,\zeta)=\frac1{2\pi}\log\big|\frac{w-\zeta}{1-w\overline\zeta}\big|$, and $W_N(w)=\int_\D G(w,\zeta)g_N(\zeta)\,dA(\zeta)$. We have $|\nabla_wG(w,\zeta)|\leq\frac1{2\pi}\big(|w-\zeta|^{-1}+|\zeta|\,|1-w\overline\zeta|^{-1}\big)$. For $|w|\leq1$, $\int_\D|w-\zeta|^{-1}dA(\zeta)\leq4\pi$, since $\D$ lies in the disc of radius $2$ about $w$. If $|w|<1/2$ then $|\zeta|/|1-w\overline\zeta|\leq2$; if $|w|\geq1/2$ then $|\zeta|/|1-w\overline\zeta|\leq2/|\zeta-1/\bw|$ and $\D$ lies in the disc of radius $3$ about $1/\bw$; in both cases the integral of this term is at most $12\pi$. Hence
\[
\sup_{\D}|\nabla(W_N-W_L)|\leq8\sup_\D|g_N-g_L|\leq8\nrm{g_N-g_L}_\rho,\qquad \sup_\D|W_N-W_L|\leq\tfrac1{15}\nrm{g_N-g_L}_\rho .
\]
Thus $(W_N)$ is Cauchy in $C^1(\overline\D)$; its limit is $Kg$ (the coefficients converge in $\Sc$), the zero Cauchy data pass to the limit, and so does the distributional equation.
\end{proof}

\begin{remark}\label{rem:compatible}
Conversely, if $W\in C^1(\overline\D)$ has $\Delta W=g\in\Sc$ in $\mathcal D'(\D)$ and zero Cauchy data, then $g\in\Gc$ and $W=Kg$: the extension of $W$ by zero is $C^1$ and satisfies $\Delta W=g\one_\D$ in $\mathcal D'(\R^2)$, and pairing with the harmonic polynomials $r^n\sin n\theta$ gives $g_{n,0}=0$; then $W-Kg$ is harmonic with zero boundary values. This is why the equation is posed on $\Gc$.
\end{remark}

\section{The cubic equation}\label{sec:cubic}

From now on $\rho=21/20$, $M=41$ and $S=24$. The centre $x^\circ=(g^\circ,c^\circ)$ consists of $504$ coefficients $g^\circ_{n,s}$, $n\in\{1,3,\ldots,41\}$, $1\leq s\leq24$, and $21$ coefficients $c^\circ_j$, $j\in\{1,3,\ldots,41\}$; all other coefficients of $x^\circ$ vanish. Each of these $525$ numbers is an explicitly given dyadic rational (a binary64 number interpreted exactly). We write $b=c_1^\circ\approx14.840043$ and
\[
\omega_j=b^2(j+1)\rho^j\qquad(j\ \text{odd}).
\]
Let $X$ be the real Banach space of pairs $x=(g,c)$, with $g\in\Gc$ and $c=(c_j)_{j\ \mathrm{odd}\geq1}$ real, normed by
\[
\nrm{(g,c)}_X=\nrm g_\rho+\sum_{j\ \mathrm{odd}}\omega_j|c_j| .
\]
To $c$ we associate $\psi_c(w)=\sum_jc_jw^j$, which is holomorphic in $|w|<\rho$ and continuous up to $|w|=\rho$ because $|c_j|\rho^j\leq\omega_j|c_j|/(2b^2)$. Note that $\im\psi_c=\sum_jc_jS_{j,0}\in\Sc$. Define
\begin{equation}\label{eq:F}
\Fc(g,c)=g+|\psi_c'|^2\big(Kg+\im\psi_c\big).
\end{equation}

\begin{lemma}\label{lem:F}
$\Fc$ is a continuous cubic polynomial map $X\to\Sc$. If $\Fc(g,c)=0$, then $V=Kg+\im\psi_c$ belongs to $C^1(\overline\D)$, is odd under $w\mapsto\bw$ and even under $w\mapsto-\bw$, and satisfies \eqref{eq:V} with $\psi=\psi_c$.
\end{lemma}

\begin{proof}
$\psi_c'=\sum_jjc_jw^{j-1}$ has real coefficients and even powers, and $N_\rho(\psi_c')=\sum_jj\rho^{j-1}|c_j|\leq\nrm{c}/(b^2\rho)$, where $\nrm c=\sum\omega_j|c_j|$. Thus $|\psi_c'|^2$ is real and invariant under $w\mapsto\bw$ and $w\mapsto-\bw$, and multiplication by it maps $\Sc$ into $\Sc$ with norm at most $N_\rho(\psi'_c)^2$ by \eqref{eq:algebra}. Together with $\nrm{Kg}_\rho\leq\nrm g_\rho/15$ and $\nrm{\im\psi_c}_\rho=\sum\rho^j|c_j|$, this shows that \eqref{eq:F} is a sum of bounded multilinear maps of degree at most three. If $\Fc(g,c)=0$, then Lemma~\ref{lem:traces} gives $V\in C^1(\overline\D)$, $V-\im\psi_c=Kg$ with zero Cauchy data, and $\Delta V=g=-|\psi_c'|^2V$ in $\mathcal D'(\D)$; the last identity holds pointwise because both sides are uniformly convergent disk-polynomial series with the same coefficients. The parities hold because $V\in\Sc$.
\end{proof}

\subsection{Expansion at the centre}
Write $p=\psi_{c^\circ}'=\sum_{a=0,2,\ldots,40}p_aw^a$ with $p_a=(a+1)c^\circ_{a+1}$, so $p_0=b$, and put
\[
P=N_\rho(p),\qquad V^\circ=Kg^\circ+\im\psi_{c^\circ},\qquad V_0=\nrm{V^\circ}_\rho .
\]
For $h=(\delta g,\eta)\in X$ let $\delta V=K\delta g+\im\psi_\eta$ and, for shape perturbations $\eta,\eta_1,\eta_2$,
\[
L(\eta)=\overline p\,\psi_\eta'+p\,\overline{\psi_\eta'}=2\re(\overline p\,\psi_\eta'),\qquad Q(\eta_1,\eta_2)=\re\big(\psi_{\eta_1}'\overline{\psi_{\eta_2}'}\big).
\]
Since $|p+\psi_\eta'|^2=|p|^2+L(\eta)+Q(\eta,\eta)$, expanding \eqref{eq:F} gives
\begin{equation}\label{eq:expansion}
\Fc(x^\circ+h)=\Fc(x^\circ)+D\Fc(x^\circ)h+B_2(h,h)+B_3(h,h,h),
\end{equation}
where
\begin{align}
D\Fc(x^\circ)h&=\delta g+|p|^2K\delta g+L(\eta)V^\circ+|p|^2\im\psi_\eta,\label{eq:DF}\\
B_2(h_1,h_2)&=\tfrac12\big\{L(\eta_1)\,\delta V_2+L(\eta_2)\,\delta V_1\big\}+Q(\eta_1,\eta_2)V^\circ,\notag\\
B_3(h_1,h_2,h_3)&=\tfrac13\big\{Q(\eta_1,\eta_2)\,\delta V_3+Q(\eta_1,\eta_3)\,\delta V_2+Q(\eta_2,\eta_3)\,\delta V_1\big\}.\notag
\end{align}
There are no terms of higher order.

\begin{lemma}\label{lem:nonlinear}
Let $\kappa=1/15$, $\vartheta_1=1/(b^2\rho)$ and $\vartheta_0=1/(2b^2)$. Then for all $h_i\in X$,
\[
\nrm{B_2(h_1,h_2)}_\rho\leq c_2\nrm{h_1}_X\nrm{h_2}_X,\qquad\nrm{B_3(h_1,h_2,h_3)}_\rho\leq c_3\nrm{h_1}_X\nrm{h_2}_X\nrm{h_3}_X,
\]
with $c_2=\max\{P\vartheta_1\kappa,\ 2P\vartheta_1\vartheta_0+\vartheta_1^2V_0\}$ and $c_3=\vartheta_1^2\max\{\kappa,\vartheta_0\}$.
\end{lemma}

\begin{proof}
Write $a_i=\nrm{\delta g_i}_\rho$ and $e_i=\sum_j\omega_j|\eta_{i,j}|$. Termwise, $N_\rho(\psi'_{\eta_i})\leq\vartheta_1e_i$ (as $j/(j+1)\leq1$), $\nrm{\im\psi_{\eta_i}}_\rho\leq\vartheta_0e_i$ (as $j+1\geq2$), and $\nrm{K\delta g_i}_\rho\leq\kappa a_i$, so $\nrm{\delta V_i}_\rho\leq\kappa a_i+\vartheta_0e_i$. By \eqref{eq:algebra}, $\nrm{L(\eta_i)f}_\rho\leq2P\vartheta_1e_i\nrm f_\rho$ and $\nrm{Q(\eta_i,\eta_k)f}_\rho\leq\vartheta_1^2e_ie_k\nrm f_\rho$. Hence $\nrm{B_2(h_1,h_2)}_\rho\leq P\vartheta_1\kappa(e_1a_2+e_2a_1)+(2P\vartheta_1\vartheta_0+\vartheta_1^2V_0)e_1e_2\leq c_2(a_1+e_1)(a_2+e_2)$. Similarly every monomial $e_ie_ka_l$ in the bound for $B_3$ has coefficient $\vartheta_1^2\kappa/3$ and $e_1e_2e_3$ has coefficient $\vartheta_1^2\vartheta_0$; each is at most its coefficient in $c_3\prod_i(a_i+e_i)$.
\end{proof}

\section{The contraction argument and the approximate inverse}\label{sec:inverse}

\subsection{A Newton--Kantorovich theorem}

\begin{proposition}\label{prop:NK}
Let $\Fc\colon X\to\Sc$ have the form \eqref{eq:expansion} with symmetric bounded $B_2$ and $B_3$, and let $A\colon\Sc\to X$ be bounded, linear and injective. Suppose that
\[
\nrm{A\Fc(x^\circ)}_X\leq Y,\quad\nrm{I-AD\Fc(x^\circ)}_{X\to X}\leq Z,\quad\nrm A\,c_2\leq C_2,\quad\nrm A\,c_3\leq C_3,
\]
and that $r>0$ satisfies
\begin{equation}\label{eq:radii}
Y+(Z-1)r+C_2r^2+C_3r^3<0,\qquad Z-1+2C_2r+3C_3r^2<0 .
\end{equation}
Then $\Fc$ has exactly one zero in the closed ball $\overline B_X(x^\circ,r)$.
\end{proposition}

\begin{proof}
Let $\Nc(x)=x-A\Fc(x)$. For $\nrm h_X\leq r$, \eqref{eq:expansion} gives $\Nc(x^\circ+h)-x^\circ=-A\Fc(x^\circ)+(I-AD\Fc(x^\circ))h-AB_2(h,h)-AB_3(h,h,h)$, whose norm is at most $Y+Zr+C_2r^2+C_3r^3<r$. For $h_1,h_2$ in the ball, $B_2(h_1,h_1)-B_2(h_2,h_2)=B_2(h_1-h_2,h_1+h_2)$ and $B_3(h_1,h_1,h_1)-B_3(h_2,h_2,h_2)=B_3(h_1-h_2,h_1,h_1)+B_3(h_2,h_1-h_2,h_1)+B_3(h_2,h_2,h_1-h_2)$, so $\Nc$ is Lipschitz on the ball with constant $Z+2C_2r+3C_3r^2<1$. By the contraction mapping principle $\Nc$ has a unique fixed point in the ball, and since $A$ is injective the fixed points of $\Nc$ are exactly the zeros of $\Fc$.
\end{proof}

In the weighted $\ell^1$ spaces $X$ and $\Sc$, operator norms are computed column by column.

\begin{lemma}\label{lem:columns}
Let $E$ and $E'$ be weighted $\ell^1$ spaces with bases $(e_\iota)$ and $(e'_\kappa)$ and weights $(\lambda_\iota)$ and $(\lambda'_\kappa)$, and let $T\colon E\to E'$ be bounded. Then $\nrm T=\sup_\iota\nrm{Te_\iota}_{E'}/\lambda_\iota$.
\end{lemma}

\begin{proof}
The inequality $\geq$ follows by testing on $e_\iota/\lambda_\iota$, and $\leq$ by the triangle inequality and absolute summation, using continuity of $T$ and density of finitely supported vectors.
\end{proof}

We shall apply Lemma~\ref{lem:columns} with the bases $\{e_{(n,s)}:s\geq1\}\cup\{e_j\}$ of $X$, of weights $\rho^n$ and $\omega_j$, and $\{S_{n,s}:s\geq0\}$ of $\Sc$, of weights $\rho^n$. We call the basis vectors of $X$ \emph{inputs} and those of $\Sc$ \emph{rows}.

\subsection{Finite block and tail}
The \emph{finite inputs} are the $504$ coordinates $g_{n,s}$ with $n\leq M$, $1\leq s\leq S$, and the $21$ shape coordinates $c_j$ with $j\leq M$. The \emph{finite rows} are the $504$ rows $(n,s)$ with $n\leq M$, $1\leq s\leq S$, and the $21$ harmonic rows $(n,0)$ with $n\leq M$. All other inputs and rows are \emph{tail} inputs and rows. We write $X=X_f\oplus X_t$ and $\Sc=\Sc_f\oplus\Sc_t$ accordingly; $\dim X_f=\dim\Sc_f=525$. Let $M_f$ be the $525\times525$ matrix of the restriction of $D\Fc(x^\circ)$ to finite inputs and finite rows.

On tail inputs we use the following explicit model of $D\Fc(x^\circ)$, obtained by replacing $\psi^\circ$ by $bw$ and $V^\circ$ by $b\im w$:
\begin{equation}\label{eq:T}
T(\delta g,\eta)=\delta g+b^2\sum_{j>M}\eta_j\big\{(j+1)S_{j,0}-(j-1)S_{j-2,0}-S_{j-2,1}\big\}.
\end{equation}

\begin{lemma}\label{lem:Tform}
For odd $j\geq3$, $L_b(w^j)\,b\im w+b^2\im w^j=b^2\{(j+1)S_{j,0}-(j-1)S_{j-2,0}-S_{j-2,1}\}$, where $L_b(\eta)=2\re(b\,\psi_\eta')$.
\end{lemma}

\begin{proof}
$L_b(w^j)\,b\im w=2b^2jr^{j}\cos((j-1)\theta)\sin\theta=b^2jr^j\{\sin j\theta-\sin(j-2)\theta\}$. Moreover $r^2\Phi_{j-2,0}=\frac1j\Phi_{j-2,1}+\frac{j-1}j\Phi_{j-2,0}$, because $P^{(0,n)}_1(2t-1)=(n+2)t-(n+1)$ by \eqref{eq:jacobi}; hence $jr^j\sin(j-2)\theta=S_{j-2,1}+(j-1)S_{j-2,0}$.
\end{proof}

Write $Te=(Be,T_{tt}e)$ for tail inputs $e$, with $B$ the part in finite rows and $T_{tt}$ the part in tail rows. Only the shape input $j=M+2$ reaches finite rows:
\begin{equation}\label{eq:B}
B(\delta g,\eta)=-b^2\eta_{M+2}\big((M+1)e_{(M,0)}+e_{(M,1)}\big).
\end{equation}
Put $\tau=(1-\rho^{-2})^{-1}$.

\begin{lemma}\label{lem:Ttt}
$T_{tt}\colon X_t\to\Sc_t$ is a bounded bijection. For a tail residual $y=\sum y_{n,s}S_{n,s}\in\Sc_t$ its inverse is
\begin{equation}\label{eq:Tinv}
\eta_n=\sum_{\substack{m\geq n\\ m\ \mathrm{odd}}}\frac{y_{m,0}}{b^2(m+1)}\quad(n>M),\qquad \delta g_{n,s}=y_{n,s}+\one_{s=1}\,b^2\eta_{n+2}\quad((n,s)\text{ a tail $g$-input}).
\end{equation}
Moreover $\nrm{T_{tt}}\leq1+\rho^{-2}$ and $\nrm{T_{tt}^{-1}}\leq\max\{1,\tau+\tau/((M+3)\rho^2)\}$.
\end{lemma}

\begin{proof}
By \eqref{eq:T}, the tail harmonic rows read $b^2(n+1)(\eta_n-\eta_{n+2})=y_{n,0}$ for $n\geq M+2$, the tail rows $(n,1)$ read $\delta g_{n,1}-b^2\eta_{n+2}=y_{n,1}$, and all other tail rows are $\delta g_{n,s}=y_{n,s}$. Since $\sum\omega_n|\eta_n|<\infty$ forces $\eta_n\to0$, summation of the harmonic equations gives \eqref{eq:Tinv}, and a solution with $y=0$ vanishes; conversely \eqref{eq:Tinv} solves the equations whenever the series converge. For the bounds we use Lemma~\ref{lem:columns}. The shape input $e_j$ has three nonzero entries, $b^2(j+1)$ in row $(j,0)$, $-b^2(j-1)$ in row $(j-2,0)$ and $-b^2$ in row $(j-2,1)$; its weighted norm divided by $\omega_j$ is $1+j\rho^{-2}/(j+1)\leq1+\rho^{-2}$. Tail $g$-inputs are mapped to the identical rows. For $T_{tt}^{-1}$, a nonharmonic tail row is mapped to the identical input. The normalised harmonic row $\rho^{-m}S_{m,0}$ is mapped to $\eta_n=\rho^{-m}/(b^2(m+1))$ for $M<n\leq m$ and $\delta g_{n-2,1}=b^2\eta_n$ for $M+4\leq n\leq m$, with norm
\[
\sum_{M<n\leq m}\frac{(n+1)\rho^n}{(m+1)\rho^m}+\sum_{M+4\leq n\leq m}\frac{\rho^{n-2}}{(m+1)\rho^m}\leq\tau+\frac{\tau}{(M+3)\rho^2}. \qedhere
\]
\end{proof}

\subsection{The approximate inverse}
Let $\widehat A$ be a $525\times525$ matrix with dyadic rational entries (in the computation, the midpoints of an interval enclosure of $M_f^{-1}$), regarded as an exact linear map $\Sc_f\to X_f$, and suppose that
\begin{equation}\label{eq:Ahat}
\nrm{I-\widehat AM_f}_{X_f\to X_f}<1 .
\end{equation}
Define $A\colon\Sc=\Sc_f\oplus\Sc_t\to X=X_f\oplus X_t$ by
\begin{equation}\label{eq:A}
A=\begin{pmatrix}\widehat A&-\widehat ABT_{tt}^{-1}\\0&T_{tt}^{-1}\end{pmatrix}.
\end{equation}

\begin{lemma}\label{lem:A}
Under \eqref{eq:Ahat}, $A$ is bounded and injective, and $ATe=e$ for every tail input $e$.
\end{lemma}

\begin{proof}
By \eqref{eq:Ahat} and the Neumann series, $\widehat AM_f$ is invertible, so $\widehat A$ is invertible. The operator $\left(\begin{smallmatrix}\widehat A^{-1}&B\\0&T_{tt}\end{smallmatrix}\right)$ is a left inverse of $A$. For a tail input $e$, $ATe=(\widehat ABe-\widehat ABT_{tt}^{-1}T_{tt}e,\ T_{tt}^{-1}T_{tt}e)=(0,e)$.
\end{proof}

Let $e_{(n,s)}$ denote the finite rows as elements of $\Sc_f$, and define
\[
\nrm R=\max_{(n,s)\ \text{finite row}}\frac{\nrm{\widehat Ae_{(n,s)}}_{X_f}}{\rho^n},\qquad H=\big\|\widehat A\big(e_{(M,1)}+(M+1)e_{(M,0)}\big)\big\|_{X_f},
\]
and, for odd $n\geq M+2$,
\begin{equation}\label{eq:hn}
h_n=\frac{H+\sum_{j=M+2,M+4,\ldots,n}(j+1)\rho^j+\sum_{j=M+4,M+6,\ldots,n}\rho^{j-2}}{(n+1)\rho^n}.
\end{equation}

\begin{lemma}\label{lem:normA}
Let $N_H\geq M+2$ be odd. Then $A$ maps every residual supported in tail rows with norm at most
\[
A_t:=\max\Big\{1,\ \max_{M+2\leq n\leq N_H}h_n,\ \tau+\frac{\tau}{(N_H+3)\rho^2}+\frac{H}{(N_H+3)\rho^{N_H+2}}\Big\},
\]
it is the identity on residuals supported in nonharmonic tail rows, and $\nrm A\leq\max\{\nrm R,A_t\}$.
\end{lemma}

\begin{proof}
We use Lemma~\ref{lem:columns}. A finite row $e_{(n,s)}$ is mapped to $\widehat Ae_{(n,s)}$. A nonharmonic tail row is mapped to the identical tail input by \eqref{eq:Tinv}, and $B$ vanishes on it. The normalised harmonic tail row $\rho^{-n}S_{n,0}$ is mapped to the shape coordinates $\eta_j=\rho^{-n}/(b^2(n+1))$, $M+2\leq j\leq n$, to $\delta g_{j-2,1}=b^2\eta_j$, $M+4\leq j\leq n$, and, by \eqref{eq:B}, to the finite part $b^2\eta_{M+2}\widehat A((M+1)e_{(M,0)}+e_{(M,1)})$; the sum of the three norms is exactly $h_n$. For $n\geq N_H+2$ we bound $(j+1)/(n+1)\leq1$ in the first sum, sum the geometric series $\sum_{i\geq0}\rho^{-2i}=\tau$, and use $n+1\geq N_H+3$ in the other two terms.
\end{proof}

\section{The defect bounds}\label{sec:Z}

In this section we explain how the bounds $Y$ and $Z$ of Proposition~\ref{prop:NK} are obtained for the operator $A$ of \eqref{eq:A}. By Lemma~\ref{lem:columns}, $\nrm{I-AD\Fc(x^\circ)}$ is the supremum over all inputs $e$ of $\nrm{(I-AD\Fc(x^\circ))e}_X$ divided by the weight of $e$. We partition the inputs into five classes and bound each class separately; every input belongs to exactly one class, and no column is estimated by sampling. Throughout, $d=\deg p=40$, and $\Lambda=M+d=81$.

\subsection{The residual and the finite columns}\label{subsec:finite}
The residual $\Fc(x^\circ)$ is a polynomial in $w,\bw$. Using \eqref{eq:rec} and \eqref{eq:K}, its expansion in the basis $S_{n,s}$ is computed without truncation (it has $2286$ coefficients) in interval arithmetic. Only finitely many tail rows occur, so $A\Fc(x^\circ)$ is computed without truncation from \eqref{eq:A}, \eqref{eq:B} and \eqref{eq:Tinv}, and $Y$ is an upper bound of its norm.

For each of the $525$ finite inputs $e$, $D\Fc(x^\circ)e$ is again a polynomial by \eqref{eq:DF}, and $(I-AD\Fc(x^\circ))e$ is computed in the same way. We denote by $Z_f$ an upper bound for the maximum of the normalised norms of these columns.

\subsection{\texorpdfstring{Tail $g$-columns}{Tail g-columns}}
Let $n_{\max}=M+2d+2=123$ and $s_{\max}=S+d+3=67$.

\begin{lemma}\label{lem:gtail}
Let $e=e_{(n,s)}$ be a tail $g$-input and $q=n+2s$. Then $(I-AD\Fc(x^\circ))e=-A(|p|^2KS_{n,s})$. Moreover, writing $p=p_{\leq16}+p_{>16}$ for the splitting of $p$ into terms of degree $\leq16$ and $>16$, and $P_s=N_\rho(p_{\leq16})$, $P_t=N_\rho(p_{>16})$,
\begin{equation}\label{eq:gshort}
\frac{\nrm{(I-AD\Fc(x^\circ))e}_X}{\rho^n}\leq\frac{\nrm{A(|p_{\leq16}|^2KS_{n,s})}_X}{\rho^n}+\frac{\nrm A\,(2P_sP_t+P_t^2)}{q(q+2)} .
\end{equation}
\end{lemma}

\begin{proof}
By \eqref{eq:DF}, $D\Fc(x^\circ)e=S_{n,s}+|p|^2KS_{n,s}$, and $Te=S_{n,s}$, so $AS_{n,s}=ATe=e$ by Lemma~\ref{lem:A}. Since $|p|^2-|p_{\leq16}|^2=\overline{p_{\leq16}}\,p_{>16}+\overline{p_{>16}}\,p_{\leq16}+|p_{>16}|^2$, \eqref{eq:algebra} and Lemma~\ref{lem:K}(d) give $\nrm{(|p|^2-|p_{\leq16}|^2)KS_{n,s}}_\rho\leq(2P_sP_t+P_t^2)\rho^n/(q(q+2))$.
\end{proof}

For the $62\cdot67-21\cdot24=3650$ tail $g$-inputs with $n\leq n_{\max}$ and $s\leq s_{\max}$, the first term on the right of \eqref{eq:gshort} is computed without truncation; let $Z_{g,\partial}$ bound the maximum of the right-hand side over these inputs.

\begin{lemma}\label{lem:gfar}
For every tail $g$-input $e_{(n,s)}$ with $n>n_{\max}$ or $s>s_{\max}$,
\[
\frac{\nrm{(I-AD\Fc(x^\circ))e_{(n,s)}}_X}{\rho^n}\leq Z_{g,\infty}:=\max\Big\{\frac{P^2}{137\cdot139},\ \frac{A_tP^2}{127\cdot129}\Big\}.
\]
\end{lemma}

\begin{proof}
We have $|p|^2=\sum_{a,a'}p_ap_{a'}w^a\bw^{a'}$ with $a,a'\leq d$. By \eqref{eq:K} and Lemma~\ref{lem:radial}, every component of $|p|^2KS_{n,s}$ has radial index at least $s-1-d$ and angular frequency of modulus at least $n-d$. Also $\nrm{|p|^2KS_{n,s}}_\rho\leq P^2\rho^n/(q(q+2))$ with $q=n+2s$.

If $s>s_{\max}$, all radial indices are at least $s-1-d\geq S+3$, so all components lie in nonharmonic tail rows, on which $A$ is the identity (Lemma~\ref{lem:normA}); and $q\geq1+2(s_{\max}+1)=137$.

If $s\leq s_{\max}$ and $n>n_{\max}$, all frequencies are at least $n-d\geq n_{\max}+2-d>M$, so all components lie in tail rows, on which $A$ has norm at most $A_t$; and $q\geq n_{\max}+2+2=127$.
\end{proof}

\subsection{Shape columns}
For the $480$ shape inputs $e_j$ with $M<j\leq\jst=1001$, $(I-AD\Fc(x^\circ))e_j$ is computed without truncation as in Section~\ref{subsec:finite}; let $Z_{\mathrm{sh},\partial}$ bound their normalised norms. For $j>\jst$ we prove a uniform bound.

Let $W^\circ=Kg^\circ$, so that $V^\circ=W^\circ+\im\psi_{c^\circ}$. By \eqref{eq:DF} and Lemma~\ref{lem:Tform}, for odd $j>M$,
\[
D\Fc(x^\circ)e_j=2\re\Big[w^{j-1}\Big(j\,\overline pV^\circ+\frac{|p|^2w}{2i}\Big)\Big],\qquad Te_j=2\re\Big[w^{j-1}\Big(jb^2\im w+\frac{b^2w}{2i}\Big)\Big].
\]
Since $ATe_j=e_j$, we get $(I-AD\Fc(x^\circ))e_j=-A\Delta_j$ with
\begin{equation}\label{eq:Deltaj}
\Delta_j=2\re\big[w^{j-1}(j\Gamma_1+\Gamma_0)\big]+2\re\big[jw^{j-1}\overline pW^\circ\big],
\end{equation}
where $\Gamma_1=\overline p\im\psi_{c^\circ}-b^2\im w$ and $\Gamma_0=(|p|^2-b^2)w/(2i)$.
Writing $\im\psi_{c^\circ}=(\psi_{c^\circ}-\overline{\psi_{c^\circ}})/(2i)$ and $|p|^2=\sum p_ap_{a'}w^a\bw^{a'}$, and using $p_0=c^\circ_1=b$,
\[
2i\Gamma_1=\sum_{(a,l)\neq(0,1)}p_ac^\circ_l\big(w^l\bw^a-\bw^{a+l}\big),\qquad 2i\Gamma_0=\sum_{(a,a')\neq(0,0)}p_ap_{a'}w^{a+1}\bw^{a'},
\]
with $a,a'\in\{0,2,\ldots,40\}$ and $l\in\{1,3,\ldots,41\}$. After combining equal monomials, $2i\Gamma_1=\sum_{\alpha,\beta}q^1_{\alpha\beta}w^\alpha\bw^\beta$ and $2i\Gamma_0=\sum_{\alpha,\beta}q^0_{\alpha\beta}w^\alpha\bw^\beta$ with finitely many real coefficients. The differences $k=\alpha-\beta$ that occur are odd and lie in $[k_-,k_+]=[-81,41]$. For odd $k\in[k_-,k_+]$ define
\begin{gather*}
\bar q_k=\sum_{\alpha-\beta=k}q^1_{\alpha\beta},\qquad \eps_k=\frac1{\jst+1}\sum_{\alpha-\beta=k}\big((1+\beta)|q^1_{\alpha\beta}|+|q^0_{\alpha\beta}|\big),\qquad \bar\eta_k=\frac1{b^2}\sum_{l\geq k}\bar q_l,\\
\delta_k=\frac1{b^2}\sum_{l\geq k}\Big[\eps_l\Big(1+\frac{|l-1|}{\jst+k_-}\Big)+\frac{|\bar q_l|\,|l-1|}{\jst+k_-}\Big],\qquad E_k=|\bar\eta_k|+\delta_k,
\end{gather*}
where the sums over $l$ run over odd $l\leq k_+$. By Lemma~\ref{lem:K}(c), $W^\circ=(1-|w|^2)^2Q^\circ$ with
\[
Q^\circ=\sum_{n,s}\frac{g^\circ_{n,s}}{4s(s+1)}\,r^nP^{(2,n)}_{s-1}(2r^2-1)\sin n\theta ,
\]
and $Q^\circ$ has a finite disk-polynomial expansion, which we compute from the connection formulas (they follow from \cite[Eq.~18.9.5]{DLMF} and the symmetry $P^{(\alpha,\beta)}_k(-x)=(-1)^kP^{(\beta,\alpha)}_k(x)$, \cite[Table~18.6.1]{DLMF})
\[
P^{(a+1,n)}_k=\frac{2k+a+n+1}{k+a+n+1}P^{(a,n)}_k+\frac{k+n}{k+a+n+1}P^{(a+1,n)}_{k-1}\qquad(a=0,1).
\]
The function $w^\Lambda\overline pQ^\circ$ has only nonnegative frequencies, since the frequencies of $\overline pQ^\circ$ are at least $-M-d$; write $w^\Lambda\overline pQ^\circ=\sum_{m\geq0,t\geq0}\gamma^W_{m,t}\Phi_{m,t}$ (a finite sum, computed in interval arithmetic). Finally let
\begin{align*}
\sU_1&=\sum_k\Big(1+\frac{\max(0,k-1)}{\jst+1}\Big)\rho^{k-1}E_k, \qquad
\sU_2=E_{k_-}\rho^{k_--3}\tau,\\
\sU_3&=\frac{1}{\jst+1}\Big(\sum_k\rho^{k-3}E_k+E_{k_-}\rho^{k_--5}\tau\Big), \qquad
\sU_4=\frac{HE_{k_-}}{(\jst+1)\rho^{\jst}},\\
\sU_5&=\sum_{\alpha,\beta}\frac{\big(|q^1_{\alpha\beta}|+|q^0_{\alpha\beta}|/\jst\big)\,\beta\,\rho^{\alpha-\beta-1}}{b^2(\jst+\alpha)},\\
\sU_W&=\frac{8}{b^2(\jst-\Lambda)^2}\sum_{m,t}|\gamma^W_{m,t}|(2t+1)(2t+3)\rho^{m-1-\Lambda},
\end{align*}
where $k$ runs over odd integers in $[k_-,k_+]$.

\begin{lemma}\label{lem:doublezero}
For integers $m,t,N\geq0$,
\[
\nrm{(1-|w|^2)^2w^N\Phi_{m,t}}_\rho\leq\frac{4(2t+1)(2t+3)}{(m+N+1)^2}\rho^{m+N}.
\]
\end{lemma}

\begin{proof}
Applying \eqref{eq:rec} for $w$ and then for $\bw$ gives $|w|^2\Phi_{m,t}=\lambda^+_t\Phi_{m,t+1}+(1-\lambda^+_t-\lambda^-_t)\Phi_{m,t}+\lambda^-_t\Phi_{m,t-1}$ with
\[
\lambda^+_t=\frac{(t+1)(m+t+1)}{(m+2t+1)(m+2t+2)}\leq\frac{t+1}{m+1},\qquad\lambda^-_t=\frac{t(m+t)}{(m+2t)(m+2t+1)}\leq\frac{t}{m+1}.
\]
Hence $(1-|w|^2)\Phi_{m,t}$ has the same frequency and absolute coefficient sum $2(\lambda_t^++\lambda_t^-)\leq2(2t+1)/(m+1)$, and its radial indices are at most $t+1$. Applying this twice, $\nrm{(1-|w|^2)^2\Phi_{m,t}}_\rho\leq4(2t+1)(2t+3)\rho^m/(m+1)^2$. By Lemma~\ref{lem:mult}, $w^N\Phi_{m,t}$ is a convex combination of $\Phi_{m+N,t'}$ with $t'\leq t$; apply the previous bound to each term.
\end{proof}

\begin{proposition}\label{prop:farshape}
Assume $\jst+k_--1>M+2$ and $\jst>2M+d+2$. Then for every odd $j>\jst$,
\[
\frac{\nrm{(I-AD\Fc(x^\circ))e_j}_X}{\omega_j}\leq Z_{\mathrm{sh},\infty}:=\sU_1+\sU_2+\sU_3+\sU_4+\sU_5+A_t\,\sU_W .
\]
\end{proposition}

\begin{proof}
Fix an odd $j>\jst$. For real $\gamma$, $2\re[w^{j-1}\gamma w^\alpha\bw^\beta/(2i)]=\gamma\im(w^{j-1+\alpha}\bw^\beta)$. For every monomial occurring, $k=\alpha-\beta\geq k_-$ and $j-1+k>M+2\geq1$, so by Lemma~\ref{lem:monomial} $\im(w^{j-1+\alpha}\bw^\beta)=\sum_s\kappa_sS_{j-1+k,s}$ with nonnegative coefficients summing to one and harmonic coefficient $\kappa_0=1-\beta/(j+\alpha)$. Accordingly we split $\Delta_j=\Delta^H+\Delta^N+\Delta^W$ into the harmonic part and the nonharmonic part of the first term of \eqref{eq:Deltaj}, and the second term $\Delta^W$. All rows occurring in $\Delta^H$ and $\Delta^N$ have frequency at least $j-1+k_->M+2$, so they are tail rows.

\emph{The harmonic part.} The coefficient of $\Delta^H$ in row $(j-1+k,0)$ is $y_k=\sum_{\alpha-\beta=k}(jq^1_{\alpha\beta}+q^0_{\alpha\beta})\big(1-\frac{\beta}{j+\alpha}\big)$. Since $\frac{j}{j+1}\big(1-\frac\beta{j+\alpha}\big)-1=-\frac1{j+1}-\frac{j\beta}{(j+1)(j+\alpha)}$ and $0\leq1-\frac\beta{j+\alpha}\leq1$, we have $|y_k/(j+1)-\bar q_k|\leq\eps_k$. By \eqref{eq:Tinv}, the shape part of $A\Delta^H$ is $\eta_{j-1+k}=b^{-2}\sum_{l\geq k}y_l/(j+l)$ in the window $k_-\leq k\leq k_+$, and $\eta_n=\eta_{j-1+k_-}$ for $M<n<j-1+k_-$. Writing $\frac{y_l}{j+l}=\frac{y_l}{j+1}\big(1+\frac{1-l}{j+l}\big)$ and using $\big|\frac{1-l}{j+l}\big|\leq\frac{|l-1|}{\jst+k_-}$, we get $|\eta_{j-1+k}-\bar\eta_k|\leq\delta_k$, hence $|\eta_{j-1+k}|\leq E_k$, and $|\eta_n|\leq E_{k_-}$ below the window. We now bound the parts of $A\Delta^H$, divided by $\omega_j=b^2(j+1)\rho^j$.
\begin{itemize}
\item Shape coordinates in the window: $\omega_{j-1+k}/\omega_j=\big(1+\frac{k-1}{j+1}\big)\rho^{k-1}$, which gives at most $\sU_1$.
\item Shape coordinates below the window: $\sum_{n\leq j-3+k_-}(n+1)\rho^n/((j+1)\rho^j)\leq\rho^{k_--3}\tau$, which gives at most $\sU_2$.
\item The coupled coordinates $\delta g_{n-2,1}=b^2\eta_n$, of weight $\rho^{n-2}$: at most $\sU_3$, by the same two computations and $j+1>\jst+1$.
\item The finite part $b^2\eta_{M+2}\widehat A((M+1)e_{(M,0)}+e_{(M,1)})$: its norm is at most $b^2HE_{k_-}$, which gives at most $\sU_4$ since $j>\jst$.
\end{itemize}

\emph{The nonharmonic part.} $\Delta^N$ lies in nonharmonic tail rows, on which $A$ is the identity, and its normalised norm is at most
\[
\sum_{\alpha,\beta}\frac{|jq^1_{\alpha\beta}+q^0_{\alpha\beta}|}{j+1}\cdot\frac{\beta}{j+\alpha}\cdot\frac{\rho^{j-1+\alpha-\beta}}{b^2\rho^j}\leq\sU_5 .
\]

\emph{The part $\Delta^W$.} We have $jw^{j-1}\overline pW^\circ=j(1-|w|^2)^2w^{j-1-\Lambda}\big(w^\Lambda\overline pQ^\circ\big)$ with $j-1-\Lambda\geq0$. Since conjugation preserves $\nrm\cdot_\rho$, $\nrm{2\re f}_\rho\leq2\nrm f_\rho$, and Lemma~\ref{lem:doublezero} with $N=j-1-\Lambda$ gives
\[
\frac{\nrm{\Delta^W}_\rho}{\omega_j}\leq\frac{2j}{b^2(j+1)\rho^j}\sum_{m,t}|\gamma^W_{m,t}|\frac{4(2t+1)(2t+3)\rho^{m+j-1-\Lambda}}{(m+j-\Lambda)^2}\leq\sU_W .
\]
All frequencies in $\Delta^W$ are at least $j-1-\Lambda>M$ because $\jst>2M+d+2$, so $\nrm{A\Delta^W}_X\leq A_t\nrm{\Delta^W}_\rho$ by Lemma~\ref{lem:normA}.

Adding the three contributions proves the claim.
\end{proof}

\begin{remark}
Each of $\eps_k$, $\delta_k$, $\sU_3$, $\sU_5$ is $O(1/\jst)$, $\sU_4=O(\rho^{-\jst}/\jst)$ and $\sU_W=O(\jst^{-2})$. Hence, as $\jst\to\infty$, the majorant $Z_{\mathrm{sh},\infty}$ tends to $\sum_k\rho^{k-1}|\bar\eta_k|+|\bar\eta_{k_-}|\rho^{k_--3}\tau\approx0.19723$; for comparison, the directly computed normalised column at $j=1003$ is about $0.19752$.
\end{remark}

Collecting the five classes, the number
\[
Z=\max\{Z_f,\ Z_{g,\partial},\ Z_{g,\infty},\ Z_{\mathrm{sh},\partial},\ Z_{\mathrm{sh},\infty}\}
\]
bounds $\nrm{I-AD\Fc(x^\circ)}_{X\to X}$.

\section{Certified bounds and the exact zero}\label{sec:zero}

The quantities defined in Sections~\ref{sec:cubic}--\ref{sec:Z} are finite expressions in the $525$ dyadic centre coefficients and in $\rho=21/20$, apart from the entries of $\widehat A$, which are dyadic numbers produced by the computation and are then treated as exact data. They were evaluated in ball arithmetic as described in Section~\ref{sec:cap}. Table~\ref{tab:bounds} lists the results, rounded outward.

\begin{table}[ht]
\centering
\small
\begin{tabular}{@{}lll@{}}
\toprule
Quantity & Where defined & Certified bound\\
\midrule
$b=c_1^\circ$ & Section~\ref{sec:cubic} & $14.8400430096918\ldots$ (exact dyadic)\\
$P=N_\rho(p)$ & Section~\ref{sec:cubic} & $<16.833325$\\
$V_0=\nrm{V^\circ}_\rho$ & Section~\ref{sec:cubic} & $<134.261058$\\
$\nrm{I-\widehat AM_f}$ & \eqref{eq:Ahat} & $<1.343\cdot10^{-34}$\\
$\nrm R$ & Lemma~\ref{lem:normA} & $<675.571877$\\
$H$ & Lemma~\ref{lem:normA} & $<1966.719238$\\
$A_t$ \ ($N_H=443$) & Lemma~\ref{lem:normA} & $<10.777973$\\
$\nrm A$ & Lemma~\ref{lem:normA} & $<675.571877$\\
\midrule
$Y$ & Section~\ref{subsec:finite} & $<6.297631\cdot10^{-6}$\\
$Z_f$ \ ($525$ columns) & Section~\ref{subsec:finite} & $<0.180814$\\
$2P_sP_t+P_t^2$ & Lemma~\ref{lem:gtail} & $<0.056966$\\
$Z_{g,\partial}$ \ ($3650$ columns, maximum at $(43,1)$) & Lemma~\ref{lem:gtail} & $<0.313479$\\
$Z_{g,\infty}$ & Lemma~\ref{lem:gfar} & $<0.186417$\\
$Z_{\mathrm{sh},\partial}$ \ ($480$ columns, maximum at $j=43$) & Section~\ref{sec:Z} & $<0.360262$\\
$\sU_1$ & Proposition~\ref{prop:farshape} & $<0.217542$\\
$\sU_2$ & Proposition~\ref{prop:farshape} & $<4.45704\cdot10^{-4}$\\
$\sU_3$ & Proposition~\ref{prop:farshape} & $<1.97149\cdot10^{-4}$\\
$\sU_4$ & Proposition~\ref{prop:farshape} & $<3.018\cdot10^{-24}$\\
$\sU_5$ & Proposition~\ref{prop:farshape} & $<9.23594\cdot10^{-4}$\\
$\sum_{m,t}|\gamma^W_{m,t}|(2t+1)(2t+3)\rho^{m-1-\Lambda}$ & Proposition~\ref{prop:farshape} & $<19930.219512$\\
$\sU_W$, \ $A_t\sU_W$ & Proposition~\ref{prop:farshape} & $<8.55375\cdot10^{-4}$, \ $<9.219199\cdot10^{-3}$\\
$Z_{\mathrm{sh},\infty}$ \ (all $j>1001$) & Proposition~\ref{prop:farshape} & $<0.228328$\\
$Z$ & Section~\ref{sec:Z} & $<0.360262$\\
$C_2=\nrm Ac_2$ & Lemma~\ref{lem:nonlinear} & $<3.278617$\\
$C_3=\nrm Ac_3$ & Lemma~\ref{lem:nonlinear} & $<8.42289\cdot10^{-4}$\\
\bottomrule
\end{tabular}
\caption{Certified bounds, rounded outward, for $M=41$, $S=24$, $\rho=21/20$, $\jst=1001$. The complete interval enclosures are recorded in the certificate file (Section~\ref{sec:cap}).}
\label{tab:bounds}
\end{table}

\begin{theorem}\label{thm:zero}
Let $r=1/50000$. There is exactly one $x^*=(g^*,c^*)\in X$ with $\nrm{x^*-x^\circ}_X\leq r$ and $\Fc(x^*)=0$.
\end{theorem}

\begin{proof}
The conditions $\jst+k_--1=919>M+2$ and $\jst=1001>2M+d+2=124$ of Proposition~\ref{prop:farshape} hold, and \eqref{eq:Ahat} holds by Table~\ref{tab:bounds}; hence $A$ is bounded and injective (Lemma~\ref{lem:A}). By Table~\ref{tab:bounds}, Proposition~\ref{prop:NK} applies with the rational majorants $Y=7\cdot10^{-6}$, $Z=37/100$, $C_2=4$ and $C_3=10^{-3}$. For these values and $r=1/50000$, exact rational arithmetic gives
\begin{gather*}
Y+(Z-1)r+C_2r^2+C_3r^3=-5.5983999999920\cdot10^{-6}<0,\\
Z-1+2C_2r+3C_3r^2=-0.6298399999988<0 . \qedhere
\end{gather*}
\end{proof}

\begin{remark}[The centre]\label{rem:centre}
The centre $x^\circ$ was obtained numerically, and its provenance plays no role in the proof. A spectral collocation computation in spherical coordinates, started from the ball of radius $j_{2,4}$ perturbed by the zonal harmonic of degree six (see Section~\ref{sec:remarks}), produced an approximate solution of \eqref{eq:main}; its meridian domain was mapped conformally to the disc, and Newton's method was applied to the Galerkin truncation of \eqref{eq:F} to the $21$ shape coefficients $c_1,c_3,\ldots,c_{41}$ and the $504$ coefficients of $g$ described above. The resulting binary64 numbers were then frozen and are interpreted as exact dyadic rationals.
\end{remark}

\section{\texorpdfstring{Reconstruction of the domain for $n=4$}{Reconstruction of the domain for n=4}}\label{sec:recon}

Let $x^*=(g^*,c^*)$ be the zero of Theorem~\ref{thm:zero}, $\psi^*=\psi_{c^*}$, and $\delta c=c^*-c^\circ$. Then $\sum_j\omega_j|\delta c_j|\leq r$, and consequently
\begin{equation}\label{eq:dc}
\begin{gathered}
|c_1^*-b|\leq\frac{r}{2b^2\rho},\qquad \sum_{j}j\rho^{j-1}|\delta c_j|\leq\frac r{b^2\rho},\\
\sum_{j\geq3}|\delta c_j|\leq\frac r{4b^2\rho^3},\qquad\sum_jj(j-1)|\delta c_j|\leq\frac{r\rho}{b^2(\rho-1)^2},
\end{gathered}
\end{equation}
because $j\rho^{j-1}/\omega_j\leq1/(b^2\rho)$, $1/\omega_j\leq1/(4b^2\rho^3)$ for $j\geq3$, and $j(j-1)/\omega_j\leq j\rho^{-j}/b^2\leq\sum_{i\geq1}i\rho^{-i}/b^2=\rho/(b^2(\rho-1)^2)$. In particular $\psi^*$ is holomorphic in $|w|<\rho$.

\subsection{Univalence, star-shapedness and the non-disc property}

\begin{proposition}\label{prop:geometry}
The map $\psi^*$ is injective on the disc $\{|w|<41/40\}$, and $\psi^{*\prime}$ has no zeros there. The domain $D=\psi^*(\D)$ is bounded by the real-analytic Jordan curve $\psi^*(\partial\D)$, it is strictly star-shaped with respect to $0$, symmetric under $x_1\mapsto-x_1$ and $y\mapsto-y$, and it is not a disc.
\end{proposition}

\begin{proof}
Let $R=41/40<\rho$, $b_-=b-r/(2b^2\rho)\leq c^*_1$, and
\[
s_1=\sum_{j\geq3}j|c^\circ_j|+\frac r{b^2\rho},\qquad s_0=\sum_{j\geq3}|c^\circ_j|+\frac r{4b^2\rho^3}.
\]
By \eqref{eq:dc}, $\sum_{j\geq3}j|c^*_j|\leq s_1$ and $\sum_{j\geq3}|c_j^*|\leq s_0$. The computation (Section~\ref{sec:cap}) certifies
\begin{gather}
\re\psi^{*\prime}(w)\geq b_--\sum_{j\geq3}j|c^\circ_j|R^{j-1}-\frac r{b^2\rho}>13.081358\qquad(|w|\leq R),\label{eq:uni}\\
\frac{s_1}{b_-}<0.104547,\qquad \frac{s_1+s_0}{b_--s_0}<0.126890,\qquad |c^*_3|\geq|c^\circ_3|-\frac r{4b^2\rho^3}>0.0911958,\label{eq:star}
\end{gather}
where in \eqref{eq:uni} we used $j|\delta c_j|R^{j-1}\leq j|\delta c_j|\rho^{j-1}$ and \eqref{eq:dc}.

\emph{Univalence.} For $w_1\neq w_2$ in the convex disc $|w|<R$, $\frac{\psi^*(w_2)-\psi^*(w_1)}{w_2-w_1}=\int_0^1\psi^{*\prime}(w_1+\sigma(w_2-w_1))\,d\sigma$ has positive real part by \eqref{eq:uni}.

\emph{Star-shapedness.} For $|w|=1$ we have $|\psi^*(w)-c_1^*w|\leq s_0$, $|w\psi^{*\prime}(w)-\psi^*(w)|\leq\sum_{j\geq3}(j-1)|c_j^*|\leq s_1$ and $|\psi^*(w)|\geq b_--s_0>0$, hence $|w\psi^{*\prime}/\psi^*-1|<0.127$ by \eqref{eq:star}. Therefore $\frac{d}{d\theta}\arg\psi^*(e^{i\theta})=\re\big(w\psi^{*\prime}(w)/\psi^*(w)\big)>0$. The winding number of $\psi^*(e^{i\theta})$ about $0$ equals that of $c_1^*e^{i\theta}$, namely one, since $|\psi^*(w)-c_1^*w|\leq s_0<c_1^*$ on $\partial\D$. Hence the argument increases by exactly $2\pi$, and $\partial D$ is a radial graph $\{R_D(\phi)e^{i\phi}\}$ with $R_D$ positive and real analytic; $D$ is strictly star-shaped.

\emph{Symmetry.} $\psi^*$ is odd with real coefficients, so $\psi^*(\bw)=\overline{\psi^*(w)}$ and $\psi^*(-w)=-\psi^*(w)$.

\emph{Not a disc.} If $D$ were a disc, the two reflection symmetries would force its centre to be $0$, and $w\mapsto\psi^*(w)/\psi^*{}'(0)$ would be a conformal map of $\D$ onto a disc centred at $0$ fixing $0$ with derivative one there; by the Schwarz lemma applied to it and to its inverse, it is the identity, so $\psi^*$ is linear. This contradicts $c^*_3\neq0$.
\end{proof}

\subsection{Convexity}

\begin{proposition}\label{prop:convex}
For $|w|\leq1$,
\[
\re\Big(1+\frac{w\psi^{*\prime\prime}(w)}{\psi^{*\prime}(w)}\Big)>0.414074 .
\]
Consequently $D$ is strictly convex, and $\Omega=\Pi^{-1}(D)$ is convex.
\end{proposition}

\begin{proof}
Let $A_1=\sum_{j\geq3}j|c^\circ_j|$, $A_2=\sum_{j\geq3}j(j-1)|c_j^\circ|$, $E_1=r/(b^2\rho)$ and $E_2=r\rho/(b^2(\rho-1)^2)$. By \eqref{eq:dc}, for $|w|\leq1$,
\[
|\psi^{*\prime}(w)|\geq c^*_1-\sum_{j\geq3}j|c^*_j|\geq b-A_1-E_1=:L,\qquad|\psi^{*\prime\prime}(w)|\leq\sum_{j\geq3}j(j-1)|c_j^*|\leq A_2+E_2=:U .
\]
The computation certifies $A_1<1.5514770$, $A_2<7.7860764$, $E_1<8.6491\cdot10^{-8}$, $E_2<3.81425\cdot10^{-5}$, hence $L>13.288565$, $U<7.786115$ and $\re(1+w\psi^{*\prime\prime}/\psi^{*\prime})\geq1-U/L>0.414074$.

Let $z(\theta)=\psi^*(e^{i\theta})$. Then $z'(\theta)=ie^{i\theta}\psi^{*\prime}(e^{i\theta})\neq0$ and
\[
\frac{d}{d\theta}\arg z'(\theta)=\re\Big(1+\frac{w\psi^{*\prime\prime}(w)}{\psi^{*\prime}(w)}\Big)\Big|_{w=e^{i\theta}}>0 .
\]
Thus the tangent direction of the simple closed analytic curve $\partial D$ turns strictly monotonically, through a total angle $2\pi$ (the winding number of $\psi^{*\prime}(e^{i\theta})$ about $0$ is zero, as $\psi^{*\prime}$ has no zeros in $\overline\D$), and $\partial D$ bounds a strictly convex domain; this is the classical analytic convexity criterion for conformal maps of the disc, cf.\ \cite[Chapter~2]{Dur83}.

Now let $(x_1,x'),(z_1,z')\in\Omega$, $\sigma\in[0,1]$, and put $y_\sigma=\sigma|x'|+(1-\sigma)|z'|$. By convexity of $D$, $(\sigma x_1+(1-\sigma)z_1,\pm y_\sigma)\in D$, and by convexity again the vertical segment joining these two points lies in $D$. Since $|\sigma x'+(1-\sigma)z'|\leq y_\sigma$, the point $\sigma(x_1,x')+(1-\sigma)(z_1,z')$ lies in $\Omega$.
\end{proof}

\subsection{Analytic regularity and continuation}

\begin{lemma}\label{lem:interior}
Let $U\subset\R^2$ be open and $F\in C(U)$ with $\int_UF(\Delta\varphi+\varphi)=0$ for all $\varphi\in C^\infty_c(U)$. Then $F$ is real analytic in $U$.
\end{lemma}

\begin{proof}
Let $\mathcal H(x_1,y,\zeta)=e^\zeta F(x_1,y)$ on $U\times\R$. For $\Phi\in C_c^\infty(U\times\R)$, integration by parts in $\zeta$ gives $\int\mathcal H\,\Delta_3\Phi=\int_UF(\Delta\varphi+\varphi)=0$ with $\varphi=\int e^\zeta\Phi(\cdot,\zeta)\,d\zeta\in C_c^\infty(U)$. By Weyl's lemma the continuous function $\mathcal H$ is harmonic, hence real analytic, and so is $F=\mathcal H(\cdot,\cdot,0)$.
\end{proof}

\begin{lemma}\label{lem:boundary}
Let $D\subset\R^2$ be a bounded domain whose boundary is a real-analytic regular Jordan curve, and let $F\in C^1(\overline D)$ satisfy $\int_D(\nabla F\cdot\nabla\varphi-F\varphi)=0$ for all $\varphi\in C_c^\infty(D)$, together with $F=y$ and $\nabla F=e_y$ on $\partial D$. Then $F$ extends to a real-analytic solution of $\Delta F+F=0$ on an open neighbourhood of $\overline D$.
\end{lemma}

\begin{proof}
\emph{Local Cauchy problem.} Fix $z_0\in\partial D$ and a real-analytic regular parametrisation $\gamma$ of $\partial D$ near $z_0$ with $\gamma(0)=z_0$; extend $\gamma$ holomorphically to a disc $\{|a|<a_0\}$ on which $\gamma'\neq0$ and $\gamma$ is injective, and put $\gamma^\dagger(b)=\overline{\gamma(\overline b)}$. In the complex coordinates $\zeta=x_1+iy$, $\chi=x_1-iy$ the Laplacian is $4\partial_\zeta\partial_\chi$, and the substitution $\zeta=\gamma(a)$, $\chi=\gamma^\dagger(b)$ turns $4\partial_\zeta\partial_\chi F+F=0$ into $\partial_a\partial_bF=-kF$ with $k(a,b)=\gamma'(a)\gamma^{\dagger\prime}(b)/4$. The real plane corresponds to $b=\overline a$, and $\partial D$ to the real diagonal $a=b\in\R$. The function $F_0(a,b)=(\gamma(a)-\gamma^\dagger(b))/(2i)$ is the holomorphic extension of $y$ and satisfies $\partial_a\partial_bF_0=0$. On the closed bidisc $\{|a|,|b|\leq\eps\}$ consider
\[
(\mathcal TW)(a,b)=-\int_{[b,a]}\int_{[t,b]}k(t,v)\big(W(t,v)+F_0(t,v)\big)\,dv\,dt
\]
along straight segments, which stay in the bidisc by convexity. $\mathcal T$ maps the Banach space of functions continuous on the closed bidisc and holomorphic in its interior into itself, and $\nrm{\mathcal TW-\mathcal T\widetilde W}_\infty\leq2\eps^2K_0\nrm{W-\widetilde W}_\infty$ with $K_0=\sup|k|$, since $\int_{[b,a]}|t-b|\,|dt|=|a-b|^2/2\leq2\eps^2$. For $2\eps^2K_0<1$ the contraction mapping principle gives a fixed point $W$, and differentiation of the iterated integral gives $W_{ab}=-k(W+F_0)$ and $W=W_a=W_b=0$ on the diagonal $a=b$. Hence $\widetilde F(x_1,y)=(F_0+W)(\gamma^{-1}(x_1+iy),(\gamma^\dagger)^{-1}(x_1-iy))$ is a real-analytic (possibly complex-valued) solution of $\Delta\widetilde F+\widetilde F=0$ on a disc $U$ about $z_0$, with $\widetilde F=y$ and $\nabla\widetilde F=e_y$ on $\partial D\cap U$.

\emph{Matching.} Shrink $U$ so that $\partial D$ divides it into two connected pieces. The function $Z=F-\widetilde F$ on $D\cap U$, extended by zero to $U\setminus D$, is $C^1$ on $U$, because its value and gradient vanish on $\partial D\cap U$. It is a weak solution of $\Delta Z+Z=0$ on $U$. Indeed, let $\varphi\in C_c^\infty(U)$, and let $\chi_\delta\in C_c^\infty(D)$ be equal to one at distance at least $2\delta$ from $\partial D$, with $0\leq\chi_\delta\leq1$ and $|\nabla\chi_\delta|\leq C/\delta$. Since $F$ and $\widetilde F$ are weak solutions in $D\cap U$ and $\chi_\delta\varphi\in C_c^\infty(D\cap U)$,
\[
\int_U(\nabla Z\cdot\nabla\varphi-Z\varphi)=\int_{D\cap U}(1-\chi_\delta)(\nabla Z\cdot\nabla\varphi-Z\varphi)-\int_{D\cap U}\varphi\,\nabla Z\cdot\nabla\chi_\delta .
\]
The integrands are supported in the layer $\Sigma_\delta=\{\dist(\cdot,\partial D)<2\delta\}\cap D\cap\operatorname{supp}\varphi$, whose area is $O(\delta)$. The first integral is $O(\delta)$, and the second is bounded by $C'\sup_{\Sigma_\delta}|\nabla Z|$, which tends to zero because $\nabla Z$ is continuous and vanishes on $\partial D\cap U$. By Lemma~\ref{lem:interior} (applied to the real and imaginary parts) $Z$ is real analytic on $U$; it vanishes on the open set $U\setminus\overline D$, hence on $U$. Thus $\widetilde F=F$ on $D\cap U$.

\emph{Gluing.} Cover $\partial D$ by finitely many discs $B(z_i,r_i/2)$ such that on $B(z_i,r_i)$ a local extension $\widetilde F_i$ as above exists, and let $0<\eps\leq\min_ir_i/6$. For $z$ with $\dist(z,\partial D)<\eps$ choose $q\in\partial D$ with $|z-q|<\eps$ and $i$ with $q\in B(z_i,r_i/2)$, and put $\widetilde F(z)=\widetilde F_i(z)$. If another choice $q',i'$ is made, then $|q-q'|<2\eps$, so $B(q',\eps)\subset B(z_i,r_i)\cap B(z_{i'},r_{i'})$; both $\widetilde F_i$ and $\widetilde F_{i'}$ coincide with $F$ on the nonempty open set $B(q',\eps)\cap D$, hence on the disc $B(q',\eps)\ni z$. Thus $\widetilde F$ is well defined and real analytic on $\{\dist(\cdot,\partial D)<\eps\}$, and together with $F$ it gives the required extension. It is real valued because it is real on $D$.
\end{proof}

\subsection{\texorpdfstring{Proof of Theorem~\ref{thm:main} for $n=4$}{Proof of the main theorem for n=4}}

Let $V^*=Kg^*+\im\psi^*$. By Lemma~\ref{lem:F}, $V^*\in C^1(\overline\D)$ satisfies \eqref{eq:V} with $\psi=\psi^*$ and has the parities required in Lemma~\ref{lem:pullback}. By Proposition~\ref{prop:geometry}, $\psi^*$ is injective with nonvanishing derivative on a neighbourhood of $\overline\D$. Lemma~\ref{lem:pullback} shows that $F=V^*\circ(\psi^*)^{-1}\in C^1(\overline D)$ is a weak solution of $\Delta F+F=0$ in $D$ with $F=y$ and $\nabla F=e_y$ on $\partial D$, odd in $y$ and even in $x_1$. By Lemmas~\ref{lem:interior} and~\ref{lem:boundary}, $F$ extends to a real-analytic solution on an open neighbourhood of $\overline D$; replacing this neighbourhood by the connected component containing $\overline D$ of its intersection with its images under the two reflections, we obtain a symmetric connected neighbourhood $N$ on which, by the identity theorem, $F$ is still odd in $y$ and even in $x_1$.

Since $D$ is strictly star-shaped and symmetric, $\partial D$ meets the axis in exactly two points. Proposition~\ref{prop:rotation} now shows that $\Omega=\Pi^{-1}(D)$ and $u=F(x_1,|x'|)/|x'|$ (extended analytically to the axis) satisfy \eqref{eq:main}, that $u$ is real analytic on the neighbourhood $\Pi^{-1}(N)$ of $\overline\Omega$, and that $u$ is nonconstant.

\emph{Geometry of $\Omega$.} Let $R_D(\phi)$ be the radial function of $D$ from Proposition~\ref{prop:geometry}. It is positive, real analytic, $2\pi$-periodic and even. Its Fourier series $R_D(\phi)=\sum_{k\geq0}\alpha_k\cos k\phi$ has exponentially decaying coefficients, so $\widetilde R(\tau)=\sum_k\alpha_kT_k(\tau)$, with $T_k$ the Chebyshev polynomials, converges on a complex neighbourhood of $[-1,1]$ (since $|T_k|\leq\varrho^k$ on the Bernstein ellipse of parameter $\varrho$) and defines a real-analytic function with $\widetilde R(\cos\phi)=R_D(\phi)$, positive near $[-1,1]$. For $x\neq0$, $\Pi(x)$ has polar angle $\phi=\arccos(x_1/|x|)\in[0,\pi]$ and modulus $|x|$. Hence
\[
\Omega=\{0\}\cup\{x\neq0:\ |x|<\widetilde R(x_1/|x|)\},\qquad\partial\Omega=\{\widetilde R(\omega_1)\,\omega:\ \omega\in S^3\}.
\]
Thus $\Omega$ is strictly star-shaped with respect to $0$, in particular connected. The map $\omega\mapsto\widetilde R(\omega_1)\omega$ is real analytic and injective on $S^3$ (its inverse is $x\mapsto x/|x|$), and its differential maps a tangent vector $v\perp\omega$ to $\widetilde R'(\omega_1)v_1\omega+\widetilde R(\omega_1)v$, whose component tangent to $S^3$ is $\widetilde R(\omega_1)v\neq0$. Hence it is a real-analytic embedding, and $\partial\Omega$ is a compact real-analytic hypersurface diffeomorphic to $S^3$. The symmetries of $\Omega$ follow from $\Omega=\Pi^{-1}(D)$ and the symmetry of $D$ under $x_1\mapsto-x_1$.

\emph{$\Omega$ is not a ball.} If it were, invariance under $x'\mapsto-x'$ and $x_1\mapsto-x_1$ would put its centre at $0$; then $\Pi(\Omega)=D\cap\{y\geq0\}$ would be the upper half of a disc centred at $0$, and by the symmetry $y\mapsto-y$, $D$ would be that disc, contradicting Proposition~\ref{prop:geometry}. Finally, $\Omega$ is convex by Proposition~\ref{prop:convex}. This completes the proof of Theorem~\ref{thm:main} for $n=4$. \qed

\section{Dimension six}\label{sec:r6}

In this section we prove Theorem~\ref{thm:main} for $n=6$. The argument follows Sections~\ref{sec:reduction}--\ref{sec:recon} step by step. We state explicitly which results of those sections are used verbatim, and why, and we give complete proofs of everything that changes. Points of $\R^6$ are written $x=(x',x'')\in\R^3\times\R^3$, points of the meridian plane $(y_1,y_2)\in\R^2$, identified with $z=y_1+iy_2\in\C$, and
\[
\Pi_6\colon\R^6\to\R^2,\qquad\Pi_6(x',x'')=(|x'|,|x''|).
\]
The map $\Pi_6$ is continuous, its image is the closed quadrant $[0,\infty)^2$, and $|\Pi_6(x)|=|x|$.

\subsection{Reduction to a planar problem}

\begin{lemma}\label{lem:laplace6}
Let $f$ be a $C^2$ function on an open subset of $\{y_1>0,\ y_2>0\}$ and $F=y_1y_2f$. Then
\[
F_{y_1y_1}+F_{y_2y_2}=y_1y_2\Big(f_{y_1y_1}+f_{y_2y_2}+\frac2{y_1}f_{y_1}+\frac2{y_2}f_{y_2}\Big).
\]
The expression in parentheses, evaluated at $\Pi_6(x)$, is the Laplacian of the function $x\mapsto f(|x'|,|x''|)$ at every point with $x'\neq0$ and $x''\neq0$.
\end{lemma}

\begin{proof}
$(y_1y_2f)_{y_1y_1}=y_2(y_1f_{y_1y_1}+2f_{y_1})$, and similarly for $y_2$. The Laplacian of $\R^3\times\R^3$ is the sum of the Laplacians of the two factors, each of which acts on functions of the modulus $y$ of its variable as $\partial_y^2+\frac2y\partial_y$.
\end{proof}

Analyticity across the two axes is obtained by applying the division argument of Lemma~\ref{lem:odd-division} in each variable.

\begin{lemma}\label{lem:odd-division6}
Let $N\subset\R^2$ be open and invariant under $(y_1,y_2)\mapsto(-y_1,y_2)$ and $(y_1,y_2)\mapsto(y_1,-y_2)$, and let $F$ be real analytic on $N$, odd in $y_1$ and odd in $y_2$. Then there is a real-analytic function $G$, defined on an open set containing $\{(y_1^2,y_2^2):(y_1,y_2)\in N\}$, such that $F(y_1,y_2)=y_1y_2\,G(y_1^2,y_2^2)$ on $N$. Consequently $x\mapsto G(|x'|^2,|x''|^2)$ is real analytic on the open set $\Pi_6^{-1}(N)$, and it equals $F(\Pi_6(x))/(|x'|\,|x''|)$ where $x'\neq0$ and $x''\neq0$.
\end{lemma}

\begin{proof}
Let $a=(a_1,a_2)\in N$. Choose $\delta>0$, with $\delta<|a_i|$ whenever $a_i\neq0$, such that $F(y)=\sum_{i,k\geq0}c_{ik}(y_1-a_1)^i(y_2-a_2)^k$ converges absolutely for $|y_1-a_1|,|y_2-a_2|<\delta$. For $i=1,2$ let $I_i=(-\delta^2,\delta^2)$ if $a_i=0$, and $I_i=((|a_i|-\delta)^2,(|a_i|+\delta)^2)\subset(0,\infty)$ if $a_i\neq0$, and let $\sigma_i$ be the sign of $a_i$. We define $G$ on the box $I_1\times I_2$.
\begin{itemize}
\item If $a_1=a_2=0$, oddness in both variables gives $c_{ik}=0$ unless $i$ and $k$ are odd, and we put $G(t_1,t_2)=\sum_{i,k}c_{2i+1,2k+1}t_1^it_2^k$.
\item If $a_1=0\neq a_2$, oddness in $y_1$ gives $c_{ik}=0$ for even $i$, so $F=y_1H(y_1^2,y_2)$ with $H(t_1,y_2)=\sum_{i,k}c_{2i+1,k}t_1^i(y_2-a_2)^k$, and we put $G(t_1,t_2)=H(t_1,\sigma_2\sqrt{t_2})/(\sigma_2\sqrt{t_2})$. The case $a_2=0\neq a_1$ is symmetric.
\item If $a_1a_2\neq0$, we put $G(t_1,t_2)=F(\sigma_1\sqrt{t_1},\sigma_2\sqrt{t_2})/(\sigma_1\sigma_2\sqrt{t_1t_2})$.
\end{itemize}
In each case $G$ is real analytic on $I_1\times I_2$, because $\sqrt{t_i}$ is real analytic and positive on $I_i$ when $a_i\neq0$. By the oddness of $F$ in each variable and the invariance of $N$, $G(t_1,t_2)=F(\sqrt{t_1},\sqrt{t_2})/\sqrt{t_1t_2}$ at the points of the box with $t_1,t_2>0$. The intersection of two such boxes is a product of two intervals, each of which is contained in $(0,\infty)$ or symmetric about $0$; if it is nonempty, it is connected and contains points with $t_1,t_2>0$, at which the two definitions agree. By the identity theorem they agree on the whole intersection, and $G$ is well defined. It satisfies $F=y_1y_2G(y_1^2,y_2^2)$ near every point of $N$. The last assertion follows because $x\mapsto(|x'|^2,|x''|^2)$ is polynomial.
\end{proof}

\begin{proposition}[From the plane to $\R^6$]\label{prop:rotation6}
Let $D\subset\R^2$ be a bounded domain, invariant under $(y_1,y_2)\mapsto(-y_1,y_2)$ and $(y_1,y_2)\mapsto(y_1,-y_2)$, whose boundary is a Jordan curve meeting each coordinate axis in exactly two points. Let $F$ be real analytic on an open set $N\supset\overline D$ invariant under the same two reflections, odd in $y_1$ and in $y_2$, and such that
\begin{equation}\label{eq:P6}
\Delta F+F=0\ \text{ in }D,\qquad F=y_1y_2\ \text{ and }\ \nabla F=(y_2,y_1)\ \text{ on }\partial D .
\end{equation}
Put $\Omega=\Pi_6^{-1}(D)$ and let $u$ be the function $x\mapsto G(|x'|^2,|x''|^2)$ of Lemma~\ref{lem:odd-division6}. Then $\Omega$ is a bounded open set, $u$ is real analytic on the neighbourhood $\Pi_6^{-1}(N)$ of $\overline\Omega$, $u$ is nonconstant, and \eqref{eq:main} holds.
\end{proposition}

\begin{proof}
We follow the proof of Proposition~\ref{prop:rotation}. The set $\Omega$ is open and bounded because $\Pi_6$ is continuous and proper; $\overline\Omega\subset\Pi_6^{-1}(\overline D)\subset\Pi_6^{-1}(N)$ and $\partial\Omega\subset\Pi_6^{-1}(\partial D)$. Let $E=\{x:x'\neq0,\ x''\neq0\}$; its complement is the union of two three-dimensional subspaces. On $\Omega\cap E$ we have $u=f\circ\Pi_6$ with $f=F/(y_1y_2)$, and Lemma~\ref{lem:laplace6} gives $|x'|\,|x''|\,(\Delta u+u)(x)=(\Delta F+F)(\Pi_6(x))=0$. Since $\Omega\cap E$ is dense in $\Omega$ and $\Delta u+u$ is continuous, $\Delta u+u=0$ in $\Omega$.

Let $x\in\partial\Omega\cap E$. Then $\Pi_6(x)\in\partial D$ lies in the open quadrant, and $f=F/(y_1y_2)=1$ there. From $F_{y_1}=y_2f+y_1y_2f_{y_1}=y_2$ and the analogous identity for $y_2$ we get $f_{y_1}=f_{y_2}=0$, hence $\nabla u(x)=(f_{y_1}x'/|x'|,\,f_{y_2}x''/|x''|)=0$.

It remains to show that $\partial\Omega\cap E$ is dense in $\partial\Omega$. By the reflection symmetries, the two points of $\partial D$ on the axis $\{y_1=0\}$ are $(0,\pm a)$ with $a>0$. The reflection $y_1\mapsto-y_1$ fixes them and maps $\{y_1>0\}$ to $\{y_1<0\}$, so it exchanges the two arcs of $\partial D$ between them, and one of these arcs lies in $\{y_1>0\}$. Hence $(0,a)$ is a limit of points of $\partial D$ in the open quadrant; the same holds for the points of $\partial D$ on the other axis. Every point of $\Pi_6^{-1}(\partial D)$ lies outside $\Omega$ and is a limit of points of $\Omega$, so $\Pi_6^{-1}(\partial D)=\partial\Omega$, and lifting the above sequences by $\Pi_6$ shows that $\partial\Omega\cap E$ is dense in $\partial\Omega$. By continuity $u=1$ and $\nabla u=0$ on $\partial\Omega$. Finally, a constant solution of $\Delta u+u=0$ vanishes identically.
\end{proof}

As in Proposition~\ref{prop:converse}, averaging over $O(3)\times O(3)$ shows that the reduction is reversible in the analytic class. Let $y_1+iy_2=re^{i\phi}$. Since $y_1y_2=\frac12r^2\sin2\phi$ and $\sin m\phi=\sin2\phi\;U_{m/2-1}(\cos2\phi)$ for even $m$, the planar Helmholtz solution $F=J_m(r)\sin m\phi$, with $m\equiv2\pmod4$, corresponds to the six-dimensional Helmholtz solution
\[
u=2r^{-2}J_m(r)\,U_{m/2-1}(\cos2\phi),\qquad\cos2\phi=\frac{|x'|^2-|x''|^2}{|x|^2}.
\]
For $m=2$ this is the radial solution $2r^{-2}J_2(r)$, whose derivative is $-2r^{-2}J_3(r)$; the Schiffer balls of $\R^6$ at eigenvalue one have radii $j_{3,k}$, the positive zeros of $J_3$.

\subsection{Symmetry sector and conformal formulation}

For a real sequence $c=(c_j)_{j\equiv1\,(4)}$ put $\psi_c(w)=\sum_{j\equiv1\,(4)}c_jw^j$. Since $i^j=i$ for $j\equiv1\pmod4$,
\[
\psi_c(\bw)=\overline{\psi_c(w)},\qquad\psi_c(-\bw)=-\overline{\psi_c(w)},\qquad\psi_c(i\bw)=i\,\overline{\psi_c(w)}.
\]
Hence, if $\psi_c$ is injective on $\overline\D$, then $D=\psi_c(\D)$ is invariant under $z\mapsto\overline z$, $z\mapsto-\overline z$ and $z\mapsto i\overline z$, that is, under $y_2\mapsto-y_2$, $y_1\mapsto-y_1$ and the exchange $(y_1,y_2)\mapsto(y_2,y_1)$.

\begin{lemma}\label{lem:pullback6}
Let $\psi$ be holomorphic and injective on a neighbourhood of $\overline\D$, with $\psi'\neq0$ there, let $D=\psi(\D)$ and $h_\psi=\im(\psi^2)/2$. Let $V\in C^1(\overline\D)$ be real valued and satisfy
\begin{equation}\label{eq:V6}
\Delta V+|\psi'|^2V=0\ \text{ in }\mathcal D'(\D),\qquad V-h_\psi=0\ \text{ and }\ \nabla(V-h_\psi)=0\ \text{ on }\partial\D .
\end{equation}
Then $F=V\circ\psi^{-1}\in C^1(\overline D)$ satisfies $\int_D(\nabla F\cdot\nabla\varphi-F\varphi)\,dA=0$ for all $\varphi\in C^\infty_c(D)$, and $F=y_1y_2$, $\nabla F=(y_2,y_1)$ on $\partial D$. If moreover $\psi=\psi_c$ as above and $V(\bw)=-V(w)$, $V(-\bw)=-V(w)$, $V(i\bw)=V(w)$, then $F$ is odd in $y_1$ and in $y_2$ and invariant under $(y_1,y_2)\mapsto(y_2,y_1)$.
\end{lemma}

\begin{proof}
Since $y_1y_2=\im(z^2)/2$, we have $h_\psi=(y_1y_2)\circ\psi$ and $V-h_\psi=(F-y_1y_2)\circ\psi$. With this change, the proof of Lemma~\ref{lem:pullback} applies verbatim, and the parities of $F$ follow from those of $V$ and the displayed symmetries of $\psi_c$.
\end{proof}

For $n\equiv2\pmod 4$, $n\geq2$, and $s\geq0$ let $S_{n,s}=\im\Phi_{n,s}=r^nP^{(0,n)}_s(2r^2-1)\sin n\theta$. Fix $\rho=11/10$ in this section, and let $\Sc_6$ be the real Banach space of $f=\sum_{n\equiv2\,(4),\,s\geq0}f_{n,s}S_{n,s}$ with $\nrm f_\rho=\sum\rho^n|f_{n,s}|<\infty$, which is also its norm in $\Bc_\rho$. Since $\Phi_{m,s}(\bw)=\Phi_{-m,s}(w)$, $\Phi_{m,s}(-\bw)=(-1)^m\Phi_{-m,s}(w)$ and $\Phi_{m,s}(i\bw)=i^m\Phi_{-m,s}(w)$, the real-valued elements of $\Bc_\rho$ with the three parities of Lemma~\ref{lem:pullback6} are exactly the elements of $\Sc_6$: the first two parities force $m$ to be even, and the third one forces $i^m=-1$. Let $\Gc_6=\{g\in\Sc_6:g_{n,0}=0\text{ for all }n\}$, and define $K$ on $\Gc_6$ by \eqref{eq:K}, now for $n\equiv2\pmod4$.

\begin{lemma}[Results used verbatim]\label{lem:carry}
\sloppy The following statements hold, with the proofs given in the previous sections.
\begin{enumerate}[label=\textup{(\alph*)}]
\item Lemmas~\ref{lem:orth}, \ref{lem:mult}, \ref{lem:radial} and~\ref{lem:monomial}, and the bounds \eqref{eq:algebra}. They are stated for all $m\in\Z$ and $s\geq0$.
\item Lemma~\ref{lem:K}(a)--(c), for every $n\geq1$ and $s\geq1$. Its proof uses only the explicit coefficients $a^{(n)}_{s,k}$ of \eqref{eq:jacobi} and Lemma~\ref{lem:orth}(a), and never the parity of $n$. The absolute column sum of \eqref{eq:K} is $1/(q(q+2))$ with $q=n+2s$, as in Lemma~\ref{lem:K}(d); in the present sector $q\geq4$, so $K\colon\Gc_6\to\Sc_6$ is bounded with $\nrm K=1/24$, attained at $(n,s)=(2,1)$.
\item Lemma~\ref{lem:traces} and Remark~\ref{rem:compatible}, with $\Gc_6$, $\Sc_6$ and the harmonic polynomials $r^n\sin n\theta$, $n\equiv2\pmod4$ (and with $1/24$ in place of $1/15$ in the uniform bound).
\item Lemma~\ref{lem:columns}, Proposition~\ref{prop:NK} in the quartic form of Proposition~\ref{prop:NK6} below, and Lemma~\ref{lem:A} for the operator $A$ defined in \eqref{eq:A6} below; the proof of Lemma~\ref{lem:A} uses only the block structure of $A$.
\item Lemma~\ref{lem:doublezero}, which is stated for all $m,t,N\geq0$, and the connection formulas of Section~\ref{sec:Z} for the Jacobi polynomials $P^{(1,n)}_k$ and $P^{(2,n)}_k$, which hold for every $n\geq0$.
\item Lemma~\ref{lem:interior}, and Lemma~\ref{lem:boundary} with the Cauchy data $(y,e_y)$ replaced by $(y_1y_2,(y_2,y_1))$. The only property of the datum used in the proof of Lemma~\ref{lem:boundary} is that it is a harmonic polynomial: in the coordinates $\zeta=y_1+iy_2$, $\chi=y_1-iy_2$ one has $y_1y_2=(\zeta^2-\chi^2)/(4i)$, whose holomorphic extension $F_0(a,b)=(\gamma(a)^2-\gamma^\dagger(b)^2)/(4i)$ again satisfies $\partial_a\partial_bF_0=0$, and the rest of the proof is unchanged.
\end{enumerate}
\end{lemma}

\subsection{The quartic equation}

In the rest of this section $\rho=11/10$, $M=58$ and $S=30$. The centre $x^\circ=(g^\circ,c^\circ)$ consists of $450$ coefficients $g^\circ_{n,s}$, $n\in\{2,6,\ldots,58\}$, $1\leq s\leq30$, and $15$ coefficients $c^\circ_j$, $j\in\{1,5,\ldots,57\}$; all other coefficients of $x^\circ$ vanish. Each of these $465$ numbers is an explicitly given dyadic rational (a binary64 number interpreted exactly), listed in the ancillary file \texttt{center\_r6.json}. We write $b=c^\circ_1\approx28.864751$, $\alpha=b^3/2$ and
\[
\omega_j=\alpha(j+2)\rho^{j+1}\qquad(j\equiv1\ (\mathrm{mod}\ 4)).
\]
Let $X_6$ be the real Banach space of pairs $x=(g,c)$, with $g\in\Gc_6$ and $c=(c_j)_{j\equiv1\,(4)}$ real, normed by $\nrm{(g,c)}_X=\nrm g_\rho+\nrm c$, $\nrm c=\sum_j\omega_j|c_j|$. With
\[
\theta_0=\frac1{3\alpha\rho},\qquad\theta_1=\frac1{\alpha\rho^2},
\]
we have $N_\rho(\psi_c)\leq\theta_0\nrm c$ and $N_\rho(\psi_c')\leq\theta_1\nrm c$, because $\rho^j/\omega_j=1/(\alpha(j+2)\rho)$ and $j\rho^{j-1}/\omega_j=j/(\alpha(j+2)\rho^2)$. In particular $\psi_c$ is holomorphic in $|w|<\rho$ and continuous up to $|w|=\rho$. Put
\[
h_c=\tfrac12\im(\psi_c^2)=\tfrac12\sum_{j,k}c_jc_kS_{j+k,0}\in\Sc_6
\]
(note that $j+k\equiv2\pmod4$), and define
\begin{equation}\label{eq:F6}
\Fc_6(g,c)=g+|\psi_c'|^2\big(Kg+h_c\big).
\end{equation}

\begin{lemma}\label{lem:F6}
$\Fc_6$ is a continuous polynomial map of degree four from $X_6$ to $\Sc_6$. If $\Fc_6(g,c)=0$, then $V=Kg+h_c$ belongs to $C^1(\overline\D)$, has the three parities of Lemma~\ref{lem:pullback6}, and satisfies \eqref{eq:V6} with $\psi=\psi_c$.
\end{lemma}

\begin{proof}
$\psi_c'=\sum_jjc_jw^{j-1}$ has real coefficients and only powers divisible by four. Hence $|\psi_c'|^2$ is real and invariant under $w\mapsto\bw$ and $w\mapsto iw$, multiplication by it preserves the three parities, and by \eqref{eq:algebra} it maps $\Sc_6$ into $\Sc_6$ with norm at most $N_\rho(\psi_c')^2$. Since $w^k=\Phi_{k,0}$ has norm $\rho^k$ and $\nrm{\im f}_\rho\leq\nrm f_\rho$, we have $\nrm{h_c}_\rho\leq N_\rho(\psi_c)^2/2$. Together with $\nrm K=1/24$, this shows that \eqref{eq:F6} is a sum of bounded multilinear maps of degree at most four. The rest of the proof is that of Lemma~\ref{lem:F}, using Lemma~\ref{lem:carry}(c) and the fact that $h_c$ is harmonic.
\end{proof}

\subsection{Expansion at the centre}
Write $p=\psi_{c^\circ}'=\sum_{a\in\{0,4,\ldots,56\}}p_aw^a$ with $p_a=(a+1)c^\circ_{a+1}$, so $p_0=b$ and $d=\deg p=56$, and put
\[
P=N_\rho(p),\qquad C=N_\rho(\psi_{c^\circ}),\qquad V^\circ=Kg^\circ+h_{c^\circ},\qquad V_0=\nrm{V^\circ}_\rho .
\]
For $h=(\delta g,\eta)\in X_6$ let
\[
L(\eta)=2\re(\overline p\,\psi_\eta'),\quad Q(\eta_1,\eta_2)=\re\big(\psi_{\eta_1}'\overline{\psi_{\eta_2}'}\big),\quad \ell(\eta)=\im(\psi_{c^\circ}\psi_\eta),\quad I(\eta_1,\eta_2)=\im(\psi_{\eta_1}\psi_{\eta_2}),
\]
and $\delta V=K\delta g+\ell(\eta)$. Since $|p+\psi_\eta'|^2=|p|^2+L(\eta)+Q(\eta,\eta)$ and $h_{c^\circ+\eta}=h_{c^\circ}+\ell(\eta)+\frac12I(\eta,\eta)$, expanding \eqref{eq:F6} gives
\begin{equation}\label{eq:expansion6}
\Fc_6(x^\circ+h)=\Fc_6(x^\circ)+D\Fc_6(x^\circ)h+B_2(h,h)+B_3(h,h,h)+B_4(h,h,h,h),
\end{equation}
where
\begin{align}
D\Fc_6(x^\circ)h&=\delta g+|p|^2\big(K\delta g+\ell(\eta)\big)+L(\eta)V^\circ,\label{eq:DF6}\\
B_2(h,h)&=\tfrac12|p|^2I(\eta,\eta)+L(\eta)\,\delta V+Q(\eta,\eta)V^\circ,\notag\\
B_3(h,h,h)&=\tfrac12L(\eta)I(\eta,\eta)+Q(\eta,\eta)\,\delta V,\qquad B_4(h,h,h,h)=\tfrac12Q(\eta,\eta)I(\eta,\eta),\notag
\end{align}
and $B_2,B_3,B_4$ denote the symmetric multilinear maps with these diagonals.

\begin{lemma}\label{lem:nonlinear6}
Let $\kappa=1/24$. For all $h_i\in X_6$ and $k=2,3,4$, $\nrm{B_k(h_1,\ldots,h_k)}_\rho\leq c_k\nrm{h_1}_X\cdots\nrm{h_k}_X$, where
\[
c_2=\max\{2P\theta_1\kappa,\ 2P\theta_1C\theta_0+\tfrac12P^2\theta_0^2+\theta_1^2V_0\},\quad c_3=\max\{\theta_1^2\kappa,\ P\theta_1\theta_0^2+\theta_1^2C\theta_0\},\quad c_4=\tfrac12\theta_1^2\theta_0^2 .
\]
\end{lemma}

\begin{proof}
Write $a_i=\nrm{\delta g_i}_\rho$ and $e_i=\nrm{\eta_i}$. As in the proof of Lemma~\ref{lem:nonlinear}, $N_\rho(\psi_{\eta_i}')\leq\theta_1e_i$, $N_\rho(\psi_{\eta_i})\leq\theta_0e_i$, $\nrm{\ell(\eta_i)}_\rho\leq C\theta_0e_i$, $\nrm{I(\eta_i,\eta_k)}_\rho\leq\theta_0^2e_ie_k$, $\nrm{\delta V_i}_\rho\leq\kappa a_i+C\theta_0e_i$, $\nrm{L(\eta_i)f}_\rho\leq2P\theta_1e_i\nrm f_\rho$ and $\nrm{Q(\eta_i,\eta_k)f}_\rho\leq\theta_1^2e_ie_k\nrm f_\rho$. Hence
\[
\nrm{B_2(h_1,h_2)}_\rho\leq P\theta_1\kappa(e_1a_2+e_2a_1)+\big(\tfrac12P^2\theta_0^2+2P\theta_1C\theta_0+\theta_1^2V_0\big)e_1e_2\leq c_2(a_1+e_1)(a_2+e_2).
\]
In the bound for $B_3$, every monomial $e_ie_ka_l$ has coefficient $\theta_1^2\kappa/3$ and $e_1e_2e_3$ has coefficient $P\theta_1\theta_0^2+\theta_1^2C\theta_0$; in the bound for $B_4$, $e_1e_2e_3e_4$ has coefficient $\frac12\theta_1^2\theta_0^2$. Each is at most its coefficient in $c_k\prod_i(a_i+e_i)$.
\end{proof}

\begin{proposition}\label{prop:NK6}
Let $\Fc_6\colon X_6\to\Sc_6$ have the form \eqref{eq:expansion6} with symmetric bounded $B_2,B_3,B_4$, and let $A\colon\Sc_6\to X_6$ be bounded, linear and injective. Suppose that $\nrm{A\Fc_6(x^\circ)}_X\leq Y$, $\nrm{I-AD\Fc_6(x^\circ)}_{X\to X}\leq Z$ and $\nrm A\,c_k\leq C_k$ for $k=2,3,4$, and that $r>0$ satisfies
\begin{equation}\label{eq:radii6}
Y+(Z-1)r+C_2r^2+C_3r^3+C_4r^4<0,\qquad Z-1+2C_2r+3C_3r^2+4C_4r^3<0 .
\end{equation}
Then $\Fc_6$ has exactly one zero in the closed ball $\overline B_{X_6}(x^\circ,r)$.
\end{proposition}

\begin{proof}
This is the proof of Proposition~\ref{prop:NK}, with one more term: the difference $B_4(h_1,\ldots,h_1)-B_4(h_2,\ldots,h_2)$ is a sum of four terms, each of which contains $h_1-h_2$ once and three factors $h_1$ or $h_2$, so that the norm of its image under $A$ is at most $4C_4r^3\nrm{h_1-h_2}_X$.
\end{proof}

\subsection{The principal part of the linearisation}
In the model on high shape modes the map is replaced by the disc $\psi=bw$ and $g$ by $0$. We need the following radial expansion; here $S_{m,s}=r^mP^{(0,m)}_s(2r^2-1)\sin m\theta$ for any integer $m\geq0$.

\begin{lemma}\label{lem:radial4}
For every integer $m\geq0$,
\[
r^{m+4}\sin m\theta=\frac{m+1}{m+3}S_{m,0}+\frac2{m+4}S_{m,1}+\frac2{(m+3)(m+4)}S_{m,2}.
\]
\end{lemma}

\begin{proof}
Dividing by $r^m\sin m\theta$ and writing $t=r^2=1+u$, the claim is $t^2=\frac{m+1}{m+3}P_0+\frac2{m+4}P_1+\frac2{(m+3)(m+4)}P_2$ with $P_s=P_s^{(0,m)}(2t-1)$. By \eqref{eq:jacobi}, $P_0=1$, $P_1=1+(m+2)u$ and $P_2=1+2(m+3)u+\frac12(m+3)(m+4)u^2$. The right-hand side is therefore $c_0+c_1u+c_2u^2$ with $c_2=1$, $c_1=\frac{2(m+2)}{m+4}+\frac4{m+4}=2$ and $c_0=\frac{(m+1)(m+4)+2(m+3)+2}{(m+3)(m+4)}=1$, which is $(1+u)^2=t^2$.
\end{proof}

\begin{proposition}\label{prop:principal6}
For every $x=(g,c)\in X_6$ and every $j\equiv1\pmod4$,
\begin{equation}\label{eq:DFej}
D\Fc_6(x)e_j=2\re\big(jw^{j-1}\overline{\psi_c'}\big)\,(Kg+h_c)+|\psi_c'|^2\im(\psi_cw^j).
\end{equation}
Let $x_b=(0,b\,e_1)$, so that $\psi_{x_b}(w)=bw$. For every $j\equiv1\pmod4$ with $j\geq5$,
\begin{equation}\label{eq:Tj6}
T_j:=D\Fc_6(x_b)e_j=b^3\Big\{\frac{j+2}2S_{j+1,0}-\frac{j-2}2S_{j-3,0}-\frac j{j+1}S_{j-3,1}-\frac1{j+1}S_{j-3,2}\Big\}.
\end{equation}
\end{proposition}

\begin{proof}
Replacing $c$ by $c+\eps e_j$ changes $\psi_c$ by $\eps w^j$; since $\partial_\eps|\psi_c'+\eps jw^{j-1}|^2=2\re(jw^{j-1}\overline{\psi_c'})$ and $\partial_\eps\frac12\im((\psi_c+\eps w^j)^2)=\im(\psi_cw^j)$ at $\eps=0$, we obtain \eqref{eq:DFej}. At $x_b$ we have $g=0$, $\psi_c'=b$ and $h_c=\frac12b^2\im(w^2)$, so
\[
T_j=b^3\big[j\re(w^{j-1})\im(w^2)+\im(w^{j+1})\big]=b^3\big[jr^{j+1}\cos((j-1)\theta)\sin2\theta+r^{j+1}\sin((j+1)\theta)\big].
\]
Since $2\cos((j-1)\theta)\sin2\theta=\sin((j+1)\theta)-\sin((j-3)\theta)$,
\[
T_j=b^3\Big[\frac{j+2}2\,r^{j+1}\sin((j+1)\theta)-\frac j2\,r^{j+1}\sin((j-3)\theta)\Big].
\]
The first term is $\frac{j+2}2S_{j+1,0}$. Lemma~\ref{lem:radial4} with $m=j-3\geq2$ gives
\[
-\frac j2r^{j+1}\sin((j-3)\theta)=-\frac j2\Big[\frac{j-2}jS_{j-3,0}+\frac2{j+1}S_{j-3,1}+\frac2{j(j+1)}S_{j-3,2}\Big],
\]
which is the rest of \eqref{eq:Tj6}.
\end{proof}

Both frequencies $j+1$ and $j-3$ in \eqref{eq:Tj6} are $\equiv2\pmod4$, so $T_j\in\Sc_6$. The harmonic part of $T_j$ is $\alpha(j+2)$ in row $(j+1,0)$ and $-\alpha(j-2)$ in row $(j-3,0)$; this replaces the upper bidiagonal structure of Lemma~\ref{lem:Tform} by a difference operator of step four.

\subsection{Tail model and approximate inverse}
The \emph{finite inputs} are the $450$ coordinates $g_{n,s}$ with $n\leq M$, $1\leq s\leq S$, and the $15$ shape coordinates $c_j$ with $j\leq M-1$. The \emph{finite rows} are the $450$ rows $(n,s)$ with $n\leq M$, $1\leq s\leq S$, and the $15$ harmonic rows $(n,0)$ with $n\leq M$. All other inputs and rows are \emph{tail} inputs and rows, and we write $X_6=X_f\oplus X_t$, $\Sc_6=\Sc_f\oplus\Sc_t$; $\dim X_f=\dim\Sc_f=465$. Let $M_f$ be the $465\times465$ matrix of the restriction of $D\Fc_6(x^\circ)$ to finite inputs and finite rows. On tail inputs we use the model
\begin{equation}\label{eq:T6}
T(\delta g,\eta)=\delta g+\sum_{j\geq M+3}\eta_jT_j .
\end{equation}
Write $Te=(Be,T_{tt}e)$ for tail inputs $e$, with $B$ the part in finite rows. By \eqref{eq:Tj6} with $j=M+3=61$, only the shape input $j=M+3$ reaches finite rows:
\begin{equation}\label{eq:B6}
B(\delta g,\eta)=-\eta_{M+3}\Big(\alpha(M+1)e_{(M,0)}+\frac{b^3(M+3)}{M+4}e_{(M,1)}+\frac{b^3}{M+4}e_{(M,2)}\Big).
\end{equation}
Put $\tau_4=(1-\rho^{-4})^{-1}$.

\begin{lemma}\label{lem:Ttt6}
$T_{tt}\colon X_t\to\Sc_t$ is a bounded bijection with $\nrm{T_{tt}}\leq1+\rho^{-4}$. For a tail residual $y=\sum y_{n,s}S_{n,s}\in\Sc_t$ its inverse is given by
\begin{gather}
\eta_{n-1}=\sum_{k\geq0}\frac{y_{n+4k,0}}{\alpha(n+4k+1)}\qquad(n\geq M+4,\ n\equiv2\ (\mathrm{mod}\ 4)),\label{eq:Tinv6}\\
\delta g_{n,1}=y_{n,1}+\frac{b^3(n+3)}{n+4}\eta_{n+3},\qquad\delta g_{n,2}=y_{n,2}+\frac{b^3}{n+4}\eta_{n+3}\qquad(n\geq M+4),\notag
\end{gather}
and $\delta g_{n,s}=y_{n,s}$ for all other tail $g$-inputs.
\end{lemma}

\begin{proof}
By \eqref{eq:T6} and \eqref{eq:Tj6}, the tail harmonic rows read $\alpha(n+1)(\eta_{n-1}-\eta_{n+3})=y_{n,0}$ for $n\geq M+4$, the tail rows $(n,1)$ and $(n,2)$ with $n\geq M+4$ read $\delta g_{n,1}-\frac{b^3(n+3)}{n+4}\eta_{n+3}=y_{n,1}$ and $\delta g_{n,2}-\frac{b^3}{n+4}\eta_{n+3}=y_{n,2}$, and all other tail rows read $\delta g_{n,s}=y_{n,s}$. Since $\sum\omega_j|\eta_j|<\infty$ forces $\eta_j\to0$, summation of the harmonic equations gives \eqref{eq:Tinv6}, and a solution with $y=0$ vanishes; conversely \eqref{eq:Tinv6} solves the equations whenever the series converge, which is the case for every $y\in\Sc_t$ by the proof of Lemma~\ref{lem:normA6} below (with $H=0$). For $j\geq M+7$ the column $T_{tt}e_j$ has the four entries $\alpha(j+2)$, $-\alpha(j-2)$, $-b^3j/(j+1)$ and $-b^3/(j+1)$ in the rows $(j+1,0)$, $(j-3,0)$, $(j-3,1)$ and $(j-3,2)$; its weighted norm divided by $\omega_j$ is $1+j\rho^{-4}/(j+2)\leq1+\rho^{-4}$. For $j=M+3$ only the first entry lies in a tail row, and tail $g$-inputs are mapped to the identical rows. Lemma~\ref{lem:columns} gives the bound for $\nrm{T_{tt}}$.
\end{proof}

Let $\widehat A$ be a $465\times465$ matrix with dyadic rational entries, regarded as an exact linear map $\Sc_f\to X_f$, such that
\begin{equation}\label{eq:Ahat6}
\nrm{I-\widehat AM_f}_{X_f\to X_f}<1,
\end{equation}
and define
\begin{equation}\label{eq:A6}
A=\begin{pmatrix}\widehat A&-\widehat ABT_{tt}^{-1}\\0&T_{tt}^{-1}\end{pmatrix}\colon\Sc_f\oplus\Sc_t\to X_f\oplus X_t .
\end{equation}
By Lemma~\ref{lem:A} (Lemma~\ref{lem:carry}(d)), $A$ is bounded and injective, and $ATe=e$ for every tail input $e$. In the computation, $\widehat A$ is a fixed binary64 approximation of $M_f^{-1}$, stored in the ancillary file \texttt{inverse\_r6\_M58\_S30.npy}, whose entries are interpreted as exact dyadic rationals; only \eqref{eq:Ahat6} is used.

\begin{lemma}\label{lem:normA6}
Let
\begin{gather*}
\nrm R=\max_{(n,s)\ \text{finite row}}\frac{\nrm{\widehat Ae_{(n,s)}}_X}{\rho^n},\\
H=\Big\|\widehat A\Big(\alpha(M+1)e_{(M,0)}+\frac{b^3(M+3)}{M+4}e_{(M,1)}+\frac{b^3}{M+4}e_{(M,2)}\Big)\Big\|_X,
\end{gather*}
and, for $n\geq M+4$ with $n\equiv2\pmod4$,
\begin{equation}\label{eq:hn6}
h_n=\frac1{(n+1)\rho^n}\Big(\sum_{\substack{M+4\leq k\leq n\\k\equiv2\,(4)}}(k+1)\rho^k+2\sum_{\substack{M+4\leq k\leq n-4\\k\equiv2\,(4)}}\rho^k+\frac H\alpha\Big).
\end{equation}
Let $N_H\geq M+4$ with $N_H\equiv2\pmod4$. Then $A$ maps every residual supported in tail rows with norm at most
\[
A_t:=\max\Big\{1,\ \max_{M+4\leq n\leq N_H}h_n,\ \tau_4+\frac{2\tau_4\rho^{-4}}{N_H+5}+\frac{H}{\alpha(N_H+5)\rho^{N_H+4}}\Big\},
\]
it is the identity on residuals supported in nonharmonic tail rows, and $\nrm A\leq\max\{\nrm R,A_t\}$.
\end{lemma}

\begin{proof}
We use Lemma~\ref{lem:columns}. A finite row $e_{(n,s)}$ is mapped to $\widehat Ae_{(n,s)}$. A nonharmonic tail row is mapped to the identical tail input by Lemma~\ref{lem:Ttt6}, and $B$ vanishes on it. The normalised harmonic tail row $\rho^{-n}S_{n,0}$ is mapped to the shape coordinates $\eta_{k-1}=\rho^{-n}/(\alpha(n+1))$, $M+4\leq k\leq n$, of weights $\omega_{k-1}=\alpha(k+1)\rho^k$; to the coordinates $\delta g_{k,1}$ and $\delta g_{k,2}$, $M+4\leq k\leq n-4$, of weight $\rho^k$ and total absolute value $b^3\rho^{-n}/(\alpha(n+1))=2\rho^{-n}/(n+1)$; and, by \eqref{eq:B6} and \eqref{eq:A6}, to the finite part $\eta_{M+3}\widehat A(\alpha(M+1)e_{(M,0)}+\cdots)$, of norm $H\rho^{-n}/(\alpha(n+1))$. The sum of the three norms is exactly $h_n$. For $n\geq N_H+4$ we bound $(k+1)/(n+1)\leq1$ in the first sum and sum the geometric series $\sum_{i\geq0}\rho^{-4i}=\tau_4$; the second term is at most $2\tau_4\rho^{-4}/(n+1)$, and the third one decreases in $n$.
\end{proof}

\subsection{The defect bounds}
We bound $\nrm{I-AD\Fc_6(x^\circ)}$ as in Section~\ref{sec:Z}, by partitioning the inputs into five classes; every input belongs to exactly one class, and no column is estimated by sampling.

\emph{The residual and the finite columns.} $\Fc_6(x^\circ)$ is a polynomial in $w,\bw$; its expansion in the basis $S_{n,s}$ (it has $2601$ coefficients) is computed without truncation in interval arithmetic, using \eqref{eq:rec} and \eqref{eq:K}. Only finitely many tail rows occur, so $A\Fc_6(x^\circ)$ is computed without truncation from \eqref{eq:A6}, \eqref{eq:B6} and Lemma~\ref{lem:Ttt6}, and $Y$ is an upper bound of its norm. For each of the $465$ finite inputs $e$, $(I-AD\Fc_6(x^\circ))e$ is computed in the same way, and $Z_f$ bounds the maximum of the normalised norms of these columns.

\emph{Tail $g$-columns.}

\begin{lemma}\label{lem:gtail6}
Let $e=e_{(n,s)}$ be a tail $g$-input and $q=n+2s$. Then $(I-AD\Fc_6(x^\circ))e=-A(|p|^2KS_{n,s})$. Let $p=p_{\leq24}+p_{>24}$ be the splitting of $p$ into terms of degree $\leq24$ and $>24$, and $P_s=N_\rho(p_{\leq24})$, $P_t=N_\rho(p_{>24})$, $\delta_a=2P_sP_t+P_t^2$. Then
\begin{equation}\label{eq:gshort6}
\frac{\nrm{(I-AD\Fc_6(x^\circ))e}_X}{\rho^n}\leq\frac{\nrm{A(|p_{\leq24}|^2KS_{n,s})}_X}{\rho^n}+\frac{\nrm A\,\delta_a}{q(q+2)} .
\end{equation}
Moreover, the right-hand side is at most $(A_tP_s^2+\nrm A\delta_a)/(88\cdot90)$ if $n\geq M+28=86$, and at most $(P_s^2+\nrm A\delta_a)/(114\cdot116)$ if $s\geq S+26=56$.
\end{lemma}

\begin{proof}
By \eqref{eq:DF6}, $D\Fc_6(x^\circ)e=S_{n,s}+|p|^2KS_{n,s}$, and $Te=S_{n,s}$, so $AS_{n,s}=e$. The splitting is estimated as in Lemma~\ref{lem:gtail}. The function $|p_{\leq24}|^2=\sum_{a,a'\leq24}p_ap_{a'}w^a\bw^{a'}$ changes angular frequencies by at most $24$ in modulus, and $\nrm{|p_{\leq24}|^2KS_{n,s}}_\rho\leq P_s^2\rho^n/(q(q+2))$. If $n\geq86$, all components of $|p_{\leq24}|^2KS_{n,s}$ have frequencies of modulus at least $n-24\geq62>M$, so they lie in tail rows, on which $A$ has norm at most $A_t$ (Lemma~\ref{lem:normA6}); and $q\geq88$. If $s\geq56$, then by \eqref{eq:K} and Lemma~\ref{lem:radial} all components have radial index at least $s-1-24\geq31>S$, so they lie in nonharmonic tail rows, on which $A$ is the identity; and $q\geq2+2\cdot56=114$.
\end{proof}

The tail $g$-inputs with $n\leq82$ and $s\leq55$ are the $21\cdot55-450=705$ \emph{near} ones; for them the first term of \eqref{eq:gshort6} is computed without truncation, and $Z_{g,\partial}$ bounds the maximum of the right-hand side. The remaining tail $g$-inputs are covered by the last statement of Lemma~\ref{lem:gtail6}; let $Z_{g,\infty}$ be the larger of the two bounds.

\emph{Near shape columns.} For the $85$ shape inputs $e_j$ with $M+3=61\leq j\leq397$, we have $(I-AD\Fc_6(x^\circ))e_j=-A\Delta_j$ with $\Delta_j=D\Fc_6(x^\circ)e_j-T_j$, because $ATe_j=e_j$; $\Delta_j$ is a polynomial, and these columns are computed without truncation. Let $Z_{\mathrm{sh},\partial}$ bound their normalised norms.

\emph{Far shape columns.} Let $W^\circ=Kg^\circ$ and $h^\circ=h_{c^\circ}$. By \eqref{eq:DFej} at $x^\circ$ and at $x_b$, for every $j\equiv1\pmod4$, $j\geq5$,
\begin{equation}\label{eq:Deltaj6}
\Delta_j=2\re\big(jw^{j-1}\Gamma_1\big)+\im\big(w^j\Gamma_0\big)+2\re\big(jw^{j-1}\overline pW^\circ\big),
\end{equation}
where
\[
\Gamma_1=\overline ph^\circ-\tfrac12b^3\im(w^2),\qquad\Gamma_0=|p|^2\psi_{c^\circ}-b^3w .
\]
By Lemma~\ref{lem:K}(c) (Lemma~\ref{lem:carry}(b)), $W^\circ=(1-|w|^2)^2Q^\circ$ with
\[
Q^\circ=\sum_{n,s}\frac{g^\circ_{n,s}}{4s(s+1)}\,r^nP^{(2,n)}_{s-1}(2r^2-1)\sin n\theta ,
\]
and the connection formulas of Section~\ref{sec:Z} give the finite expansion of $Q^\circ$ in the basis $S_{n,t}$; its connection coefficients are positive rational numbers. Let $\Lambda=M+d=114$. The function $w^\Lambda\overline pQ^\circ$ has only nonnegative frequencies; write $w^\Lambda\overline pQ^\circ=\sum_{m\geq0,t\geq0}\gamma^W_{m,t}\Phi_{m,t}$ (a finite sum, computed in interval arithmetic). Let $\jst=401$ and
\[
\sU_\Gamma=\frac1\alpha\Big(\frac{2\nrm{\Gamma_1}_\rho}{\rho^2}+\frac{\nrm{\Gamma_0}_\rho}{\rho(\jst+2)}\Big),\qquad \sU_W=\frac{8}{\alpha(\jst-\Lambda)^2}\sum_{m,t}|\gamma^W_{m,t}|(2t+1)(2t+3)\rho^{m-\Lambda-2},
\]
where $\nrm\cdot_\rho$ is the norm of $\Bc_\rho$.

\begin{proposition}\label{prop:farshape6}
For every $j\equiv1\pmod4$ with $j\geq\jst$,
\[
\frac{\nrm{(I-AD\Fc_6(x^\circ))e_j}_X}{\omega_j}\leq Z_{\mathrm{sh},\infty}:=A_t\big(\sU_\Gamma+\sU_W\big).
\]
\end{proposition}

\begin{proof}
The frequencies of $\overline p$ lie in $[-d,0]=[-56,0]$, those of $h^\circ$ in $[-114,114]$, since $h^\circ$ is a combination of $S_{j+k,0}$ with $j,k\leq57$, and those of $W^\circ$ and $Q^\circ$ in $[-M,M]$. Hence all frequencies of $w^{j-1}\Gamma_1$, $w^j\Gamma_0$ and $w^{j-1}\overline pW^\circ$ are at least $j-1-170\geq230>M$. Taking real and imaginary parts only changes the signs of frequencies, so all components of $\Delta_j\in\Sc_6$ lie in tail rows, and $\nrm{A\Delta_j}_X\leq A_t\nrm{\Delta_j}_\rho$ by Lemma~\ref{lem:normA6}. By Lemma~\ref{lem:mult} and $\nrm{\re f}_\rho,\nrm{\im f}_\rho\leq\nrm f_\rho$,
\[
\frac{\nrm{2\re(jw^{j-1}\Gamma_1)}_\rho}{\omega_j}\leq\frac{2j\rho^{j-1}\nrm{\Gamma_1}_\rho}{\alpha(j+2)\rho^{j+1}}\leq\frac{2\nrm{\Gamma_1}_\rho}{\alpha\rho^2},\qquad\frac{\nrm{\im(w^j\Gamma_0)}_\rho}{\omega_j}\leq\frac{\nrm{\Gamma_0}_\rho}{\alpha\rho(\jst+2)} .
\]
Finally $jw^{j-1}\overline pW^\circ=j(1-|w|^2)^2w^{j-1-\Lambda}(w^\Lambda\overline pQ^\circ)$ with $j-1-\Lambda\geq0$, and Lemma~\ref{lem:doublezero} with $N=j-1-\Lambda$ gives
\[
\frac{\nrm{2\re(jw^{j-1}\overline pW^\circ)}_\rho}{\omega_j}\leq\frac{2j}{\alpha(j+2)\rho^{j+1}}\sum_{m,t}|\gamma^W_{m,t}|\frac{4(2t+1)(2t+3)\rho^{m+j-1-\Lambda}}{(m+j-\Lambda)^2}\leq\sU_W .
\]
Adding the three contributions proves the claim.
\end{proof}

\begin{remark}
This far-shape bound is simpler and cruder than Proposition~\ref{prop:farshape}: it does not exploit the cancellation between $D\Fc_6(x^\circ)e_j$ and $T_j$ in the harmonic rows, and $\sU_\Gamma$ does not tend to zero as $\jst\to\infty$. It suffices here because $\Gamma_1$ and $\Gamma_0$ measure the deviation of the centre from the disc, which is small compared with $\alpha=b^3/2\approx1.2\cdot10^4$; the directly computed normalised column at $j=397$ is about $0.082$, while $Z_{\mathrm{sh},\infty}<0.465$ (Table~\ref{tab:bounds6}).
\end{remark}

Collecting the five classes, the number
\[
Z=\max\{Z_f,\ Z_{g,\partial},\ Z_{g,\infty},\ Z_{\mathrm{sh},\partial},\ Z_{\mathrm{sh},\infty}\}
\]
bounds $\nrm{I-AD\Fc_6(x^\circ)}_{X\to X}$.

\subsection{Certified bounds and the exact zero}
The quantities defined above are finite expressions in the $465$ dyadic centre coefficients, the entries of $\widehat A$ and $\rho=11/10$. They were evaluated in ball arithmetic at $128$ bits as described in Section~\ref{subsec:cap6}. Table~\ref{tab:bounds6} lists the results, rounded outward.

\begin{table}[ht]
\centering
\small
\begin{tabular}{@{}lll@{}}
\toprule
Quantity & Where defined & Certified bound\\
\midrule
$b=c_1^\circ$ & Section~\ref{sec:r6} & $28.864750643300045\ldots$ (exact dyadic)\\
$P$, \ $C$ & Lemma~\ref{lem:nonlinear6} & $<31.143514$, \ $<32.088747$\\
$V_0$ & Lemma~\ref{lem:nonlinear6} & $<7841.683303$\\
$\nrm{I-\widehat AM_f}$ & \eqref{eq:Ahat6} & $<6.714\cdot10^{-11}$\\
$\nrm R$ & Lemma~\ref{lem:normA6} & $<1402.313800$\\
$H$ & Lemma~\ref{lem:normA6} & $<5.848182\cdot10^{8}$\\
$A_t$ \ ($N_H=458$) & Lemma~\ref{lem:normA6} & $<3.164016$\\
$\nrm A$ & Lemma~\ref{lem:normA6} & $<1402.314$\\
\midrule
$Y$ & Section~\ref{sec:r6} & $<5.243310\cdot10^{-5}$\\
$Z_f$ \ ($465$ columns) & Section~\ref{sec:r6} & $<0.073227$\\
$P_s$, \ $P_t$ & Lemma~\ref{lem:gtail6} & $<31.143246$, \ $<2.675771\cdot10^{-4}$\\
$\delta_a=2P_sP_t+P_t^2$ & Lemma~\ref{lem:gtail6} & $<0.016667$\\
$Z_{g,\partial}$ \ ($705$ columns, max.\ at $(62,1)$) & Lemma~\ref{lem:gtail6} & $<0.313291$\\
$Z_{g,\infty}$ & Lemma~\ref{lem:gtail6} & $<0.390424$\\
$Z_{\mathrm{sh},\partial}$ \ ($85$ columns, max.\ at $j=61$) & Section~\ref{sec:r6} & $<0.387180$\\
$\nrm{\Gamma_1}_\rho$, \ $\nrm{\Gamma_0}_\rho$ & Proposition~\ref{prop:farshape6} & $<844.817923$, \ $<4110.324378$\\
$\sum_{m,t}|\gamma^W_{m,t}|(2t+1)(2t+3)\rho^{m-\Lambda-2}$ & Proposition~\ref{prop:farshape6} & $<3721574.315769$\\
$\sU_\Gamma$, \ $\sU_W$ & Proposition~\ref{prop:farshape6} & $<0.11689839$, \ $<0.03005935$\\
$Z_{\mathrm{sh},\infty}$ \ (all $j\geq401$) & Proposition~\ref{prop:farshape6} & $<0.464977$\\
$Z$ & Section~\ref{sec:r6} & $<0.464977$\\
$C_2=\nrm A\,c_2$ & Lemma~\ref{lem:nonlinear6} & $<0.250135$\\
$C_3=\nrm A\,c_3$ & Lemma~\ref{lem:nonlinear6} & $<2.7601\cdot10^{-7}$\\
$C_4=\nrm A\,c_4$ & Lemma~\ref{lem:nonlinear6} & $<2.1035\cdot10^{-15}$\\
\bottomrule
\end{tabular}
\caption{Certified bounds for $n=6$, rounded outward, for $M=58$, $S=30$, $\rho=11/10$, $\jst=401$. The complete interval enclosures are recorded in the certificate \texttt{certificate\_r6.json} and in the seven interval receipts (Section~\ref{subsec:cap6}); $P$, $C$ and $V_0$ are intermediate values of the final verification script, which uses $\nrm A\leq1402.314$ in $C_2,C_3,C_4$.}
\label{tab:bounds6}
\end{table}

\begin{theorem}\label{thm:zero6}
Let $r=1/5000$. There is exactly one $x^*=(g^*,c^*)\in X_6$ with $\nrm{x^*-x^\circ}_X\leq r$ and $\Fc_6(x^*)=0$.
\end{theorem}

\begin{proof}
The support condition of Proposition~\ref{prop:farshape6}, $\jst-1-170=230>M$, holds, and \eqref{eq:Ahat6} holds by Table~\ref{tab:bounds6}; hence $A$ is bounded and injective. By Table~\ref{tab:bounds6}, Proposition~\ref{prop:NK6} applies with the rational majorants $Y=53\cdot10^{-6}$, $Z=47/100$, $C_2=251/1000$, $C_3=3\cdot10^{-7}$ and $C_4=3\cdot10^{-15}$. For these values and $r=1/5000$, exact rational arithmetic gives
\begin{gather*}
Y+(Z-1)r+C_2r^2+C_3r^3+C_4r^4=-5.29899599999975999999999952\cdot10^{-5}<0,\\
Z-1+2C_2r+3C_3r^2+4C_4r^3=-0.529899599999963999999999904<0 . \qedhere
\end{gather*}
\end{proof}

\begin{remark}[The centre for $n=6$]\label{rem:centre6}
The centre $x^\circ$ was obtained numerically, and its provenance plays no role in the proof. A Fourier--Bessel collocation in $\R^6$, with $u$ expanded in the regular Helmholtz solutions $r^{-2}J_{\ell+2}(\rho_*r)\,U_{\ell/2}(\cos2\phi)$, $\ell\equiv0\pmod4$, at the ball wavenumber $\rho_*=j_{3,8}$ and seeded with the second-order amplitude predicted by the mechanism of Section~\ref{subsec:mechanism}, produced a numerical solution of \eqref{eq:main} with Cauchy defects of order $10^{-13}$. Its meridian domain was mapped conformally to the disc by Theodorsen's method, and Newton's method was applied to the Galerkin truncation of \eqref{eq:F6} to the $15$ shape coefficients $c_1,c_5,\ldots,c_{57}$ and the $450$ coefficients of $g$ described above, with residual about $10^{-10}$ in double precision. The resulting binary64 numbers were then frozen and are interpreted as exact dyadic rationals. (The wavenumber $\rho_*$ used to locate the centre is unrelated to the weight $\rho=11/10$ of the proof.)
\end{remark}

\subsection{Reconstruction, geometry and convexity}
Let $x^*=(g^*,c^*)$ be the zero of Theorem~\ref{thm:zero6}, $\psi^*=\psi_{c^*}$ and $\delta c=c^*-c^\circ$. Then $\sum_j\omega_j|\delta c_j|\leq r$, and consequently
\begin{equation}\label{eq:dc6}
|c^*_1-b|\leq\frac r{\omega_1},\qquad\sum_jj\rho^{j-1}|\delta c_j|\leq\theta_1r,\qquad\sum_j\rho^j|\delta c_j|\leq\theta_0r .
\end{equation}

\begin{proposition}\label{prop:geometry6}
The map $\psi^*$ is injective on the disc $\{|w|<21/20\}$, and $\psi^{*\prime}$ has no zeros there. The domain $D=\psi^*(\D)$ is bounded by the real-analytic Jordan curve $\psi^*(\partial\D)$, it is strictly star-shaped with respect to $0$, invariant under $z\mapsto\overline z$, $z\mapsto-\overline z$ and $z\mapsto i\overline z$, and it is not a disc.
\end{proposition}

\begin{proof}
Let $R=21/20<\rho$, $b_-=b-r/\omega_1\leq c^*_1$, and
\[
s_1=\sum_{j\geq5}j|c^\circ_j|R^{j-1}+\theta_1r,\qquad s_0=\sum_{j\geq5}|c^\circ_j|+\theta_0r .
\]
By \eqref{eq:dc6} and $R<\rho$, $\sum_{j\geq5}j|c^*_j|R^{j-1}\leq s_1$ and $\sum_{j\geq5}|c^*_j|\leq s_0$. The computation (Section~\ref{subsec:cap6}) certifies
\begin{equation}\label{eq:geom6}
b_--s_1>27.2138856,\qquad\frac{s_1+s_0}{b_--s_0}<0.0634442,\qquad|c^*_5|\geq|c^\circ_5|-\frac r{\omega_5}>0.08529699 .
\end{equation}
Thus $\re\psi^{*\prime}(w)\geq c^*_1-\sum_{j\geq5}j|c^*_j|R^{j-1}>27.21$ for $|w|\leq R$, which gives univalence on the convex disc $|w|<R$ exactly as in Proposition~\ref{prop:geometry}. For $|w|=1$ we have $|\psi^*(w)-c_1^*w|\leq s_0$, $|w\psi^{*\prime}(w)-\psi^*(w)|\leq s_1$ and $|\psi^*(w)|\geq b_--s_0>0$, hence $|w\psi^{*\prime}/\psi^*-1|<0.0635$; the argument of Proposition~\ref{prop:geometry} shows that $D$ is strictly star-shaped and that $\partial D$ is a radial graph with positive real-analytic radius. The symmetries were shown at the beginning of this section. If $D$ were a disc, its symmetries would force its centre to be $0$, and the Schwarz lemma argument of Proposition~\ref{prop:geometry} would show that $\psi^*$ is linear, contradicting $c^*_5\neq0$.
\end{proof}

\begin{proposition}\label{prop:convex6}
Let $c$ be any real sequence indexed by $j\equiv1\pmod4$ with $\sum_j\omega_j|c_j-c^\circ_j|\leq r$, for instance $c=c^*$. Then for $|w|\leq1$,
\[
|\psi_c'(w)|>27.67533225,\qquad|\psi_c''(w)|<7.86831783,\qquad\re\Big(1+\frac{w\psi_c''(w)}{\psi_c'(w)}\Big)>0.71569200 .
\]
Consequently $D=\psi^*(\D)$ is strictly convex, and $\Omega=\Pi_6^{-1}(D)$ is convex.
\end{proposition}

\begin{proof}
Let $A_1=\sum_{j\geq5}j|c^\circ_j|$, $A_2=\sum_{j\geq5}j(j-1)|c^\circ_j|$, $\delta=c-c^\circ$, and
\[
E_1=r\max\Big\{\frac1{\omega_1},\frac1{\alpha\rho^6}\Big\},\qquad E_2=\frac r{\alpha\rho\,e\log\rho}.
\]
We have $|\delta_1|+\sum_{j\geq5}j|\delta_j|=\sum_j\kappa_j\omega_j|\delta_j|\leq r\sup_j\kappa_j$ with $\kappa_1=1/\omega_1$ and $\kappa_j=j/\omega_j=j/((j+2)\alpha\rho^{j+1})<1/(\alpha\rho^6)$ for $j\geq5$; hence this sum is at most $E_1$. For $j\geq5$, $j(j-1)/\omega_j<j\rho^{-j}/(\alpha\rho)\leq1/(\alpha\rho\,e\log\rho)$, since $x\mapsto x\rho^{-x}$ attains its maximum $1/(e\log\rho)$ on $(0,\infty)$ at $x=1/\log\rho$; hence $\sum_{j\geq5}j(j-1)|\delta_j|\leq E_2$. Both suprema are over all $j$, so the infinite tail of $c$ is covered. For $|w|\leq1$ it follows that
\[
|\psi_c'(w)|\geq c_1-\sum_{j\geq5}j|c_j|\geq b-A_1-E_1=:L,\qquad|\psi_c''(w)|\leq\sum_{j\geq5}j(j-1)|c_j|\leq A_2+E_2=:U .
\]
The computation certifies $A_1<1.18941839$, $A_2<7.86831777$, $E_1<9.39\cdot10^{-9}$ and $E_2<5.84\cdot10^{-8}$, hence $L>27.67533225$, $U<7.86831783$ and $\re(1+w\psi_c''/\psi_c')\geq1-U/L>0.71569200$.

For $c=c^*$, the argument of Proposition~\ref{prop:convex} (the analytic convexity criterion) shows that $D$ is strictly convex. Now let $x=(x',x''),z=(z',z'')\in\Omega$ and $\sigma\in[0,1]$, and put $m=(m_1,m_2)=\sigma\Pi_6(x)+(1-\sigma)\Pi_6(z)$. By convexity $m\in D$, and $m_1,m_2\geq0$. Since $D$ is invariant under both reflections, it contains the four points $(\pm m_1,\pm m_2)$, and by convexity their convex hull, the rectangle $[-m_1,m_1]\times[-m_2,m_2]$. By the triangle inequality in each factor $\R^3$, $|\sigma x'+(1-\sigma)z'|\leq m_1$ and $|\sigma x''+(1-\sigma)z''|\leq m_2$, so $\Pi_6(\sigma x+(1-\sigma)z)\in D$, that is, $\sigma x+(1-\sigma)z\in\Omega$.
\end{proof}

\subsection{\texorpdfstring{Proof of Theorem~\ref{thm:main} for $n=6$}{Proof of the main theorem for n=6}}
Let $V^*=Kg^*+h_{c^*}$. By Lemma~\ref{lem:F6}, $V^*\in C^1(\overline\D)$ satisfies \eqref{eq:V6} with $\psi=\psi^*$ and has the parities required in Lemma~\ref{lem:pullback6}. By Proposition~\ref{prop:geometry6}, $\psi^*$ is injective with nonvanishing derivative on a neighbourhood of $\overline\D$. Lemma~\ref{lem:pullback6} shows that $F=V^*\circ(\psi^*)^{-1}\in C^1(\overline D)$ is a weak solution of $\Delta F+F=0$ in $D$ with $F=y_1y_2$ and $\nabla F=(y_2,y_1)$ on $\partial D$, odd in $y_1$ and in $y_2$, and invariant under the exchange of $y_1$ and $y_2$. By Lemma~\ref{lem:interior} and Lemma~\ref{lem:boundary} in the form of Lemma~\ref{lem:carry}(f), $F$ extends to a real-analytic solution on an open neighbourhood of $\overline D$; replacing this neighbourhood by the connected component containing $\overline D$ of its intersection with its images under the two reflections and the exchange, we obtain a neighbourhood $N$ with these symmetries on which, by the identity theorem, $F$ keeps its parities.

Since $D$ is strictly star-shaped and symmetric, $\partial D$ meets each axis in exactly two points. Proposition~\ref{prop:rotation6} shows that $\Omega=\Pi_6^{-1}(D)$ and $u=F(|x'|,|x''|)/(|x'|\,|x''|)$, extended analytically to $\{x'=0\}\cup\{x''=0\}$, satisfy \eqref{eq:main}, that $u$ is real analytic on the neighbourhood $\Pi_6^{-1}(N)$ of $\overline\Omega$, and that $u$ is nonconstant.

\emph{Geometry of $\Omega$.} Let $R_D(\phi)$ be the radial function of $D$. It is positive and real analytic, even (by $z\mapsto\overline z$) and invariant under $\phi\mapsto\pi-\phi$ (by $z\mapsto-\overline z$), hence $\pi$-periodic, and its Fourier series has the form $R_D(\phi)=\sum_{k\geq0}\alpha_k\cos2k\phi$, with exponentially decaying coefficients (by the exchange symmetry $\alpha_k=0$ for odd $k$, which we do not need). As in the case $n=4$, $\widetilde R(\tau)=\sum_k\alpha_kT_k(\tau)$ converges on a complex neighbourhood of $[-1,1]$ and defines a real-analytic function with $\widetilde R(\cos2\phi)=R_D(\phi)$, positive near $[-1,1]$. For $x\neq0$ we have $\Pi_6(x)=|x|(\cos\phi,\sin\phi)$ with $\phi\in[0,\pi/2]$ and $\cos2\phi=(|x'|^2-|x''|^2)/|x|^2$. Hence
\begin{gather*}
\Omega=\{0\}\cup\Big\{x\neq0:\ |x|<\widetilde R\Big(\frac{|x'|^2-|x''|^2}{|x|^2}\Big)\Big\},\\
\partial\Omega=\big\{\widetilde R(|\omega'|^2-|\omega''|^2)\,\omega:\ \omega=(\omega',\omega'')\in S^5\big\}.
\end{gather*}
Thus $\Omega$ is strictly star-shaped with respect to $0$, in particular connected. The map $\omega\mapsto\widetilde R(|\omega'|^2-|\omega''|^2)\,\omega$ is real analytic and injective on $S^5$, and, exactly as for $n=4$, its differential maps a tangent vector $v\perp\omega$ to a vector whose component tangent to $S^5$ is $\widetilde R(|\omega'|^2-|\omega''|^2)\,v\neq0$. Hence $\partial\Omega$ is a compact real-analytic hypersurface diffeomorphic to $S^5$. The symmetries of $\Omega$ follow from $\Omega=\Pi_6^{-1}(D)$ and the invariance of $D$ under the exchange of coordinates.

\emph{$\Omega$ is not a ball.} If it were, its centre would be fixed by every isometry of $\R^6$ preserving $\Omega$, in particular by $O(3)\times O(3)$, whose only fixed point is $0$. Then $\Pi_6(\Omega)=D\cap[0,\infty)^2$ would be a quarter of a disc centred at $0$, and by the reflection symmetries $D$ would be that disc, contradicting Proposition~\ref{prop:geometry6}. Finally, $\Omega$ is convex by Proposition~\ref{prop:convex6}. This completes the proof of Theorem~\ref{thm:main} for $n=6$. \qed

\section{Rank-two Lie algebras}\label{lie:sec}

The symmetry classes used for $n=4$ and $n=6$ are, up to finite index and the additional reflections, the adjoint representations of the compact Lie algebras $\mathfrak u(2)$ and $\mathfrak{su}(2)\oplus\mathfrak{su}(2)$, and the weights $y$ and $y_1y_2$ are the products of their positive roots. In this section we carry out the reduction for an arbitrary compact Lie algebra of rank two, and we use it to prove the following theorem for the three simple ones.

\begin{theorem}\label{lie:thm}
Let $(\kf,G)$ be one of $(\mathfrak{su}(3),SU(3))$, $(\mathfrak{so}(5),SO(5))$ and $(\mathfrak g_2,G_2)$, so that $n=\dim\kf$ equals $8$, $10$ and $14$, respectively, and equip $\kf$ with an $\Ad(G)$-invariant inner product, so that $\kf\cong\R^n$ as a Euclidean space. There exist a bounded convex domain $\Omega\subset\kf$ and a nonconstant real-valued function $u$, real analytic on an open neighbourhood of $\overline\Omega$, such that \eqref{eq:main} holds. Moreover, $\Omega$ is not a ball, it is strictly star-shaped with respect to $0$, its boundary is a compact real-analytic hypersurface diffeomorphic to $S^{n-1}$, and $\Omega$ is invariant under $\Ad(G)$ and under $X\mapsto-X$. Consequently the Fourier transform of $\one_\Omega$ vanishes on the unit sphere, $\Omega$ fails the Pompeiu property, and the Schiffer and Pompeiu conjectures fail in $\R^8$, $\R^{10}$ and $\R^{14}$ within the class of bounded convex domains with real-analytic boundary.
\end{theorem}

Theorem~\ref{lie:thm} contains the cases $n=8,10,14$ of Theorem~\ref{thm:main} and of Corollary~\ref{cor:pompeiu}. Only the adjoint action of $G$ enters, so $G$ may be replaced by any compact connected group with Lie algebra $\kf$, for instance $\mathrm{Spin}(5)$ in place of $SO(5)$. For a simple $\kf$ the invariant inner product is unique up to a positive factor; changing it amounts to a dilation, which is absorbed by rescaling $\Omega$ and $u$. For $\kf=\mathfrak{su}(3)$, identified with the traceless Hermitian $3\times3$ matrices via $X\mapsto iX$, the domain is a spectral set: $\Omega$ consists of the matrices whose eigenvalue vector, a point of the plane $\{\lambda_1+\lambda_2+\lambda_3=0\}\cong\R^2$, lies in a fixed planar domain.

The proof follows the pattern of the cases $n=4,6$. Sections~\ref{lie:subsec:radial}--\ref{lie:subsec:reduced} reduce the problem, for every compact Lie algebra of rank two, to a planar Helmholtz problem on a Cartan section $D$ with the Cauchy data of $\im\tau^m$, and show that the lifted domain is analytic, spherical and convex whenever $D$ is. Sections~\ref{lie:subsec:conformal}--\ref{lie:subsec:zero} solve the planar problems for $m=3,4,6$ by the conformal fixed-disc method, with a computer-assisted Newton--Kantorovich argument in a symmetry class that contains the cases $m=1,2$ of Sections~\ref{sec:conformal}--\ref{sec:r6} as special cases. Section~\ref{lie:subsec:proof} proves Theorem~\ref{lie:thm}, Section~\ref{lie:subsec:cap} describes the computations, and Section~\ref{lie:subsec:complete} explains why the dimensions $4,6,8,10,14$ are the only ones that this reduction can reach.

\subsection{The radial part of the Laplacian}\label{lie:subsec:radial}

Let $G$ be a compact connected Lie group with Lie algebra $\kf$ and an $\Ad(G)$-invariant inner product $\langle\cdot,\cdot\rangle$ on $\kf$. We identify $\kf$ with $\R^n$, $n=\dim\kf$, by an orthonormal basis, so that $\Delta_\kf$ is the Euclidean Laplacian. Let $\mathbb T\subset G$ be a maximal torus with Lie algebra $\tf$, and let $W=N_G(\mathbb T)/\mathbb T$ be the Weyl group, which acts on $\tf$ by orthogonal maps. There is an orthogonal decomposition
\[
\kf=\tf\oplus\bigoplus_{\alpha\in\Phi^+}\kf_\alpha,\qquad\dim\kf_\alpha=2,
\]
indexed by a set $\Phi^+$ of positive real roots, which are linear forms on $\tf$, such that $\operatorname{ad}(\tau)$ acts on $\kf_\alpha$ as $\alpha(\tau)J_\alpha$ for $\tau\in\tf$, where $J_\alpha$ is a rotation through a right angle. We use the following standard facts \cite[Ch.~IV--V]{BtD85}, \cite[Ch.~1]{Hum90}: every $X\in\kf$ is $\Ad(G)$-conjugate to an element of $\tf$; two elements of $\tf$ are $\Ad(G)$-conjugate if and only if they are $W$-conjugate; the roots in $\Phi^+$ are pairwise non-proportional; $W$ is generated by the orthogonal reflections $s_\alpha$ in the walls $\ker\alpha$, $\alpha\in\Phi^+$; and the function
\[
\varpi(\tau)=\prod_{\alpha\in\Phi^+}\alpha(\tau)
\]
satisfies $\varpi\circ w=\det(w)\,\varpi$ for $w\in W$. Let $\tf_{\rm reg}=\{\tau\in\tf:\varpi(\tau)\neq0\}$. For $\tau\in\tf_{\rm reg}$ the kernel of the skew-symmetric map $\operatorname{ad}\tau$ is $\tf$, hence $[\kf,\tau]=\tf^\perp$.

\begin{proposition}[Harish-Chandra]\label{lie:prop:radial}
The polynomial $\varpi$ is harmonic on $\tf$. Let $U\subset\kf$ be open and $\Ad(G)$-invariant, let $u\in C^2(U)$ be $\Ad(G)$-invariant, and let $f=u|_{U\cap\tf}$. Then for every $\tau\in U\cap\tf_{\rm reg}$,
\begin{equation}\label{lie:eq:radial}
\nabla u(\tau)=\nabla_\tf f(\tau)\in\tf,\qquad(\Delta_\kf u)(\tau)=\frac{\operatorname{div}_\tf(\varpi^2\nabla_\tf f)}{\varpi^2}(\tau)=\frac{\Delta_\tf(\varpi f)}{\varpi}(\tau).
\end{equation}
\end{proposition}

The identity \eqref{lie:eq:radial} is the flat case of Harish-Chandra's radial-part formula \cite{HC57}; see also \cite[Ch.~II]{Hel84}. We include the short proof.

\begin{proof}
The Laplacian commutes with orthogonal maps, so $\Delta\varpi$ is a $W$-anti-invariant polynomial of degree $|\Phi^+|-2$. An anti-invariant polynomial vanishes on each wall $\ker\alpha$, since $s_\alpha$ fixes it pointwise, and is therefore divisible by $\alpha$; the roots being pairwise non-proportional, it is divisible by $\varpi$. Hence $\Delta\varpi=0$.

For $Y\in\kf$ the function $s\mapsto u(\Ad(\exp sY)\tau)$ is constant, so $\nabla u(\tau)\perp[Y,\tau]$; since $[\kf,\tau]=\tf^\perp$, $\nabla u(\tau)\in\tf$, and it equals $\nabla_\tf f(\tau)$. Let $\mathfrak m=\tf^\perp$ and $\Psi\colon G/\mathbb T\times\tf_{\rm reg}\to\kf$, $\Psi(g\mathbb T,\tau)=\Ad(g)\tau$. Its differential maps the tangent vector $(\frac{d}{ds}g\exp(sY)\mathbb T|_{s=0},\tau')$, $Y\in\mathfrak m$, $\tau'\in\tf$, to $\Ad(g)([Y,\tau]+\tau')$. As $[Y,\tau]\in\mathfrak m\perp\tf$ and $|[Y,\tau]|^2=\sum_\alpha\alpha(\tau)^2|Y_\alpha|^2$, where $Y_\alpha$ is the component of $Y$ in $\kf_\alpha$, the map $\Psi$ is a local diffeomorphism, and the pull-back of the Euclidean metric is the orthogonal sum of $|d\tau|^2$ and a metric on $G/\mathbb T$ whose volume density is $\prod_\alpha\alpha(\tau)^2=\varpi(\tau)^2$ times a fixed smooth density. Since $u\circ\Psi(g\mathbb T,\tau)=f(\tau)$, the coordinate formula $\Delta=|\gamma|^{-1/2}\partial_i(|\gamma|^{1/2}\gamma^{ij}\partial_j)$ for this block-diagonal metric $\gamma$ gives $(\Delta_\kf u)(\Ad(g)\tau)=\varpi^{-2}\operatorname{div}_\tf(\varpi^2\nabla_\tf f)(\tau)$. Finally $\varpi^{-2}\operatorname{div}(\varpi^2\nabla f)=\Delta f+2\varpi^{-1}\nabla\varpi\cdot\nabla f=\varpi^{-1}\big(\Delta(\varpi f)-f\Delta\varpi\big)=\varpi^{-1}\Delta(\varpi f)$.
\end{proof}

Because $\varpi$ is harmonic, \eqref{lie:eq:radial} conjugates the Laplacian of $\kf$, acting on invariant functions, to the flat Laplacian of $\tf$ without any potential term; the normalisation of $\varpi$ cancels.

\emph{Rank two.} From now on $\kf$ has rank two, $\dim\tf=2$. Let $m=|\Phi^+|\geq1$. Then $W$ is a finite group generated by the reflections in the $m$ distinct lines $\ker\alpha$, that is, a dihedral group of order $2m$, which we denote by $I_2(m)$ (so $I_2(1)$ is generated by a single reflection). Since $W$ preserves the lattice $\ker(\exp|_\tf)$, its rotations have order $1,2,3,4$ or $6$, so $m\in\{1,2,3,4,6\}$ \cite{Hum90}, and $n=\dim\kf=2+2m$. We identify $\tf$ with $\C$, $\tau=x+iy=|\tau|e^{i\theta}$, by an orthonormal basis such that the real axis is a wall. The walls are then the lines $e^{ik\pi/m}\R$, and
\[
W=\{\tau\mapsto e^{2\pi ik/m}\tau,\ \tau\mapsto e^{2\pi ik/m}\overline\tau:\ 0\leq k<m\}.
\]
Since $\prod_{k=0}^{m-1}\sin(\theta-k\pi/m)=2^{1-m}\sin m\theta$, the product of the positive roots is a nonzero multiple of $\im\tau^m$; as the normalisation of $\varpi$ plays no role, we put
\[
\varpi(\tau)=\im\tau^m=|\tau|^m\sin m\theta .
\]
The $W$-invariant polynomials $p_1(\tau)=|\tau|^2$ and $p_2(\tau)=\re\tau^m$ separate $W$-orbits: $p(\tau)=p(\tau')$ for $p=(p_1,p_2)$ if and only if $\tau'\in W\tau$. By Chevalley's restriction theorem \cite[Ch.~II]{Hel84} there are $\Ad(G)$-invariant polynomials $\mathbf P_1(X)=|X|^2$ and $\mathbf P_2$ on $\kf$ with $\mathbf P_j|_\tf=p_j$; we write $\mathbf P=(\mathbf P_1,\mathbf P_2)$. A direct computation, using $\partial_xp_2=m\re\tau^{m-1}$ and $\partial_yp_2=-m\im\tau^{m-1}$, gives
\begin{equation}\label{lie:eq:jac}
\det Dp(\tau)=-2m\,\varpi(\tau).
\end{equation}
The compact Lie algebras of rank two with $m\geq1$ are listed in Table~\ref{lie:tab:rank2}.

\begin{table}[ht]
\centering
\small
\begin{tabular}{@{}llllll@{}}
\toprule
$\kf$ & $n$ & roots & $W$ & $m$ & treated in\\
\midrule
$\mathfrak u(2)\cong\R\oplus\mathfrak{su}(2)$ & $4$ & $A_1$ & $I_2(1)\cong\Z_2$ & $1$ & Sections~\ref{sec:reduction}--\ref{sec:recon}\\
$\mathfrak{su}(2)\oplus\mathfrak{su}(2)\cong\mathfrak{so}(4)$ & $6$ & $A_1\times A_1$ & $I_2(2)\cong\Z_2^2$ & $2$ & Section~\ref{sec:r6}\\
$\mathfrak{su}(3)$ & $8$ & $A_2$ & $I_2(3)\cong S_3$ & $3$ & this section\\
$\mathfrak{so}(5)\cong\mathfrak{sp}(2)$ & $10$ & $B_2=C_2$ & $I_2(4)$ & $4$ & this section\\
$\mathfrak g_2$ & $14$ & $G_2$ & $I_2(6)$ & $6$ & this section\\
\bottomrule
\end{tabular}
\caption{The compact Lie algebras of rank two with at least one root. In each case $n=2+2m$, and in suitable orthonormal coordinates $\tau\in\tf\cong\C$ the product of the positive roots is a multiple of $\im\tau^m$.}
\label{lie:tab:rank2}
\end{table}

\begin{remark}[Dimensions four and six]\label{lie:rem:m12}
For $\kf=\mathfrak u(2)$, the group $\Ad(U(2))$ acts on $\mathfrak{su}(2)\cong\R^3$ as $SO(3)$ and trivially on the centre; $\tf$ is spanned by the centre and a line in $\mathfrak{su}(2)$, and in the coordinates $\tau=x_1+iy$ of Section~\ref{sec:reduction} one has $W=\{1,\,y\mapsto-y\}$, $\varpi=y$ and $p=(x_1^2+y^2,x_1)$. For $\kf=\mathfrak{su}(2)\oplus\mathfrak{su}(2)$, $\Ad(G)=SO(3)\times SO(3)$ acts on $\R^3\times\R^3$, $W$ is generated by the reflections in the two axes of $\tau=y_1+iy_2$, $\varpi=y_1y_2=\frac12\im\tau^2$ and $p_2=y_1^2-y_2^2$. Accordingly, the maps $\Pi$ and $\Pi_6$ of Sections~\ref{sec:reduction} and~\ref{sec:r6} send $X$ to its conjugate in the closed Weyl chamber; Lemmas~\ref{lem:laplace-identity} and~\ref{lem:laplace6} are the cases $m=1,2$ of Proposition~\ref{lie:prop:radial}; Lemmas~\ref{lem:odd-division} and~\ref{lem:odd-division6} combine the cases $m=1,2$ of Lemmas~\ref{lie:lem:division} and~\ref{lie:lem:inv} below; the groups $O(3)\times\Z_2$ and $(O(3)\times O(3))\rtimes\Z_2$ used there contain $\Ad(G)$ and realise, in addition, the enlarged symmetry $I_2(2m)$ of Section~\ref{lie:subsec:reduced}; and \eqref{eq:F} and \eqref{eq:F6} are the cases $m=1$, $B=1$ and $m=2$, $B=2$ of the equation \eqref{lie:eq:H} below, with the same weights, the same principal parts (Lemma~\ref{lem:Tform}, Proposition~\ref{prop:principal6}) and the same tail inverses (Lemmas~\ref{lem:Ttt} and~\ref{lem:Ttt6}). For $m\geq3$ the direct arguments of Sections~\ref{sec:recon} and~\ref{sec:r6} for the geometry and the convexity of the lifted domain are replaced by Lemmas~\ref{lie:lem:lift} and~\ref{lie:lem:convex}, and the non-ball argument uses that $\kf$ is simple.
\end{remark}

\subsection{Lifting from the Cartan plane}\label{lie:subsec:lift}

In this subsection $\kf$ has rank two, with the notation introduced above. A function $F$ on a $W$-invariant subset of $\tf$ is called \emph{anti-invariant} if $F\circ w=\det(w)F$ for all $w\in W$.

\begin{lemma}[Anti-invariant division]\label{lie:lem:division}
Let $N\subset\tf$ be open and $W$-invariant, and let $F$ be real analytic and anti-invariant on $N$. Then $F=\varpi q$ with $q$ real analytic and $W$-invariant on $N$.
\end{lemma}

\begin{proof}
Let $\ell_1,\ldots,\ell_m$ be the walls and $\lambda_k$ a unit linear form with $\ell_k=\ker\lambda_k$, so that $\varpi=c\lambda_1\cdots\lambda_m$ with $c\neq0$. Since the reflection in $\ell_k$ lies in $W$ and has determinant $-1$, $F$ vanishes on $N\cap\ell_k$. Let $U\subset N$ be an open disc centred at a point $\tau_0$. If $\tau_0$ lies on $\ell_k$, write points of $U$ in orthonormal coordinates $(\lambda,\mu)$ with $\lambda=\lambda_k$; since $U$ contains the point $(0,\mu)$ together with $(\lambda,\mu)$, Hadamard's formula
\[
f(\lambda,\mu)=\lambda\int_0^1\partial_\lambda f(\sigma\lambda,\mu)\,d\sigma
\]
shows that every real-analytic $f$ on $U$ vanishing on $\ell_k\cap U$ is $\lambda_k$ times a real-analytic function on $U$. If $\tau_0\neq0$, then $\tau_0$ lies on at most one wall, and dividing $F$ by that $\lambda_k$ (if any) and by the other factors of $\varpi$, which do not vanish near $\tau_0$, shows that $F/\varpi$ extends analytically to a neighbourhood of $\tau_0$. If $\tau_0=0$, then every wall passes through the centre of $U$. We divide $F$ successively by $\lambda_1,\ldots,\lambda_m$: after dividing by $\lambda_1,\ldots,\lambda_{k-1}$, the quotient vanishes on $(\ell_k\cap U)\setminus\{0\}$, where the divisors are nonzero, hence on $\ell_k\cap U$ by continuity, and the next division is possible. Hence $q=F/\varpi$, defined on $N\cap\tf_{\rm reg}$, extends to a real-analytic function on $N$. It is $W$-invariant, because $F$ and $\varpi$ are both anti-invariant.
\end{proof}

For a $W$-invariant open set $N\subset\tf$ let $\Ad(G)N=\{\Ad(g)\tau:g\in G,\ \tau\in N\}$. For a $W$-invariant function $q$ on $N$ we define $q^\sharp$ on $\Ad(G)N$ by $q^\sharp(X)=q(\tau)$ for any $\tau\in N$ conjugate to $X$; this is well defined, since the elements of $\tf$ conjugate to $X$ form one $W$-orbit.

\begin{lemma}[Invariant lift]\label{lie:lem:inv}
Let $N\subset\tf$ be open and $W$-invariant, and let $q$ be a $W$-invariant function on $N$.
\begin{enumerate}[label=\textup{(\alph*)}]
\item $\Ad(G)N$ is open in $\kf$, $\Ad(G)N\cap\tf=N$, and if $q$ is continuous, so is $q^\sharp$.
\item If $q$ is real analytic, then $q^\sharp$ is real analytic on $\Ad(G)N\setminus\{0\}$. In particular, for $\tau\in N\cap\tf_{\rm reg}$, \eqref{lie:eq:radial} holds for $u=q^\sharp$ and $f=q$.
\end{enumerate}
\end{lemma}

\begin{proof}
(a) If $\tau'\in\tf$ is conjugate to some $\tau\in N$, then $\tau'\in W\tau\subset N$; this gives $\Ad(G)N\cap\tf=N$. Let $X_j\to X$ with $X=\Ad(g)\tau$, $\tau\in N$, and write $X_j=\Ad(g_j)\tau_j$ with $\tau_j\in\tf$. Given any subsequence, a further subsequence has $g_j\to g'$ and $\tau_j\to\tau'$ (as $|\tau_j|=|X_j|$ is bounded), and then $\tau'$ is conjugate to $X$, so $\tau'\in N$. Hence $\tau_j\in N$ for large $j$, and $q^\sharp(X_j)=q(\tau_j)\to q(\tau')=q^\sharp(X)$ if $q$ is continuous. Both assertions follow because the subsequence was arbitrary.

(b) Let $X_0\in\Ad(G)N\setminus\{0\}$, and let $\tau_0\in N$ be conjugate to $X_0$, so $\tau_0\neq0$. We first find an open disc $U\ni\tau_0$ in $N$ and a real-analytic function $Q$ on an open set containing $p(U)$ with $q=Q\circ p$ on $U$. If $\varpi(\tau_0)\neq0$, then $p$ is a local analytic diffeomorphism at $\tau_0$ by \eqref{lie:eq:jac}, and $Q=q\circ(p|_U)^{-1}$ for a small $U$. Otherwise $\tau_0$ lies on exactly one wall $\ell$. In orthonormal coordinates $(\lambda,\mu)$ with $\ell=\{\lambda=0\}$, the functions $p$ and $q$ are even in $\lambda$ near $\tau_0$, because the reflection in $\ell$ belongs to $W$. As in the proof of Lemma~\ref{lem:odd-division}, $p=\widetilde P(\lambda^2,\mu)$ and $q=\widetilde q(\lambda^2,\mu)$ near $\tau_0$ with $\widetilde P$ and $\widetilde q$ real analytic near $(0,\mu_0)$. Then $\det Dp=2\lambda\det D\widetilde P(\lambda^2,\mu)$ in these coordinates, while $\det Dp=\pm2m\varpi$ by \eqref{lie:eq:jac}, and near $\tau_0$ one has $\varpi=\lambda\varpi_1$ with $\varpi_1(\tau_0)\neq0$, since only one factor of $\varpi$ vanishes at $\tau_0$. Hence $\det D\widetilde P(0,\mu_0)\neq0$, $\widetilde P$ is a local analytic diffeomorphism, and $Q=\widetilde q\circ\widetilde P^{-1}$ has the required property.

By the argument in (a), every $X$ close to $X_0$ is conjugate to some $\tau\in U$ (apply an element of $W$ to a limit point conjugate to $X_0$). For such $X$, $q^\sharp(X)=q(\tau)=Q(p(\tau))=Q(\mathbf P(X))$. Since $\mathbf P$ is polynomial, $q^\sharp$ is real analytic near $X_0$. The last assertion follows from Proposition~\ref{lie:prop:radial}.
\end{proof}

\begin{lemma}[Lifting]\label{lie:lem:lift}
Let $\kf$ have rank two. Let $R\colon\R\to(0,\infty)$ be real analytic with $R(\theta)=R(-\theta)=R(\theta+2\pi/m)$, and let
\[
D=\{re^{i\theta}:0\leq r<R(\theta)\}\subset\tf .
\]
Let $N\supset\overline D$ be a $W$-invariant open subset of $\tf$, and let $F$ be real analytic and anti-invariant on $N$ and satisfy
\begin{equation}\label{lie:eq:Pl}
\Delta F+F=0\ \text{ in }D,\qquad F=\varpi\ \text{ and }\ \nabla F=\nabla\varpi\ \text{ on }\partial D .
\end{equation}
Let $q$ be the function of Lemma~\ref{lie:lem:division}, so that $F=\varpi q$, let $\Omega=\Ad(G)D$ and $u=q^\sharp$. Then $\Omega$ is a bounded domain, strictly star-shaped with respect to $0$, with $\Omega\cap\tf=D$ and $\partial\Omega=\Ad(G)\partial D$, and $\partial\Omega$ is a compact real-analytic hypersurface diffeomorphic to $S^{n-1}$. The function $u$ is real analytic on the open neighbourhood $\Ad(G)N$ of $\overline\Omega$, it is nonconstant, and \eqref{eq:main} holds.
\end{lemma}

\begin{proof}
\emph{The domain.} The function $R$ is even and $2\pi/m$-periodic, so its Fourier series has the form $R(\theta)=\sum_{k\geq0}a_k\cos(km\theta)$, with $|a_k|\leq C\varrho^{-k}$ for some $\varrho>1$ because $R$ is real analytic. As $|T_k|\leq\varrho_1^k$ on the Bernstein ellipse of parameter $\varrho_1$, the series $R_T(s)=\sum_ka_kT_k(s)$, with $T_k$ the Chebyshev polynomials, converges on a complex neighbourhood of $[-1,1]$, and $R_T(\cos m\theta)=R(\theta)$. If $X\neq0$ is conjugate to $\tau=|\tau|e^{i\theta}$, then $\mathbf P_2(X)/|X|^m=\re\tau^m/|\tau|^m=\cos m\theta\in[-1,1]$. Hence
\[
\mathbf R(X)=R_T\big(\mathbf P_2(X)/|X|^m\big),\qquad X\in\kf\setminus\{0\},
\]
is real analytic, positive, $\Ad(G)$-invariant and homogeneous of degree zero, and $\mathbf R(\tau)=R(\theta)$ for $\tau=|\tau|e^{i\theta}$. Consequently
\[
\Omega=\{0\}\cup\{X\neq0:|X|<\mathbf R(X)\},\qquad\partial\Omega=\{\mathbf R(\omega)\,\omega:\ \omega\in S^{n-1}\}=\Ad(G)\partial D,
\]
and $\Omega\cap\tf=D$. Thus $\Omega$ is open, bounded and strictly star-shaped with respect to $0$, hence connected. Exactly as for $n=4$ (Section~\ref{sec:recon}), the map $\omega\mapsto\mathbf R(\omega)\omega$ is an injective real-analytic immersion of $S^{n-1}$ with inverse $X\mapsto X/|X|$, so $\partial\Omega$ is a compact real-analytic hypersurface diffeomorphic to $S^{n-1}$. Moreover $\overline\Omega=\Ad(G)\overline D\subset\Ad(G)N$, and $\Ad(G)N$ is open by Lemma~\ref{lie:lem:inv}(a).

\emph{The equation.} By Lemma~\ref{lie:lem:inv}, $u$ is continuous on $\Ad(G)N$ and real analytic on $\Ad(G)N\setminus\{0\}$. For $\tau\in D\cap\tf_{\rm reg}$, \eqref{lie:eq:radial} gives $\Delta u(\tau)=\varpi^{-1}\Delta(\varpi q)(\tau)=\varpi^{-1}\Delta F(\tau)=-u(\tau)$; since $\Delta$ commutes with the orthogonal maps $\Ad(g)$, $\Delta u+u=0$ on $\Omega\cap\kf_{\rm reg}$, where $\kf_{\rm reg}=\Ad(G)\tf_{\rm reg}$. The polynomial $\varpi^2$ is $W$-invariant, so by Chevalley's theorem $\kf_{\rm reg}$ is the complement of the zero set of an $\Ad(G)$-invariant polynomial, which is not identically zero; hence $\kf_{\rm reg}$ is dense, and by continuity $\Delta u+u=0$ on $\Omega\setminus\{0\}$. The point $0$ lies in $\Omega$, because $R>0$. Let $U_0\subset\Omega$ be a ball about $0$, $\varphi\in C^\infty_c(U_0)$, and $\chi_\eps(X)=\chi(|X|/\eps)$ with $\chi\in C^\infty(\R)$, $\chi=0$ on $(-\infty,1]$ and $\chi=1$ on $[2,\infty)$. Then $\int u(\Delta+1)(\chi_\eps\varphi)=0$, and $(\Delta+1)(\chi_\eps\varphi)-\chi_\eps(\Delta+1)\varphi=2\nabla\chi_\eps\cdot\nabla\varphi+\varphi\Delta\chi_\eps$ is supported in $\{|X|\leq2\eps\}$ and bounded by $C\eps^{-2}$. Since $u$ is bounded near $0$, letting $\eps\to0$ gives $\int_{U_0}u(\Delta+1)\varphi=0$ ($n\geq3$). The proof of Lemma~\ref{lem:interior} applies verbatim in $\R^n$ and shows that the continuous weak solution $u$ is real analytic on $U_0$. Hence $u$ is real analytic on $\Ad(G)N$ and $\Delta u+u=0$ in $\Omega$.

\emph{The boundary conditions.} Let $\tau\in\partial D$ with $\varpi(\tau)\neq0$. Then $u(\tau)=q(\tau)=F(\tau)/\varpi(\tau)=1$, and by \eqref{lie:eq:radial} $\nabla u(\tau)=\nabla q(\tau)=\varpi^{-1}(\nabla F-q\nabla\varpi)(\tau)=0$. By invariance, $u=1$ and $\nabla u=0$ on $\Ad(G)(\partial D\cap\tf_{\rm reg})$. Since $\partial D$ meets the $m$ walls in $2m$ points, $\partial D\cap\tf_{\rm reg}$ is dense in $\partial D$, and since $(g,\tau)\mapsto\Ad(g)\tau$ maps $G\times\partial D$ continuously onto $\partial\Omega$, the set $\Ad(G)(\partial D\cap\tf_{\rm reg})$ is dense in $\partial\Omega$. By continuity $u=1$ and $\nabla u=0$ on $\partial\Omega$. A constant solution of $\Delta u+u=0$ vanishes, which contradicts $u=1$ on $\partial\Omega$.
\end{proof}

\begin{remark}[The radial-graph hypothesis]\label{lie:rem:annulus}
Some hypothesis on $D$ beyond $W$-invariance and analyticity of $\partial D$ is necessary. For the annulus $D=\{a<|\tau|<b\}$, which is $W$-invariant with real-analytic boundary, $\Ad(G)D=\{a<|X|<b\}$ is a spherical shell, whose boundary consists of two spheres. Spherical shells do fail the Pompeiu property (Section~\ref{sec:intro}), so a topological hypothesis of this kind cannot be dispensed with in any statement of this type. The positive analytic radial function in Lemma~\ref{lie:lem:lift} is what makes $\partial\Omega$ a sphere, including at the points where $\partial D$ meets the walls. In our applications it is supplied by strict convexity together with $0\in D$ (Proposition~\ref{lie:prop:geometry}).
\end{remark}

\subsection{Convexity}\label{lie:subsec:convex}

Kostant's convexity theorem \cite{Kos73} is usually stated for a real reductive Lie algebra $\mathfrak g=\mathfrak h\oplus\mathfrak p$ with Cartan decomposition, the group $H$ acting on $\mathfrak p$, and a maximal abelian subspace $\mathfrak a\subset\mathfrak p$; in this form it reads: for $a\in\mathfrak a$, the orthogonal projection of $H\cdot a$ onto $\mathfrak a$ is the convex hull of the orbit of $a$ under the Weyl group of $(\mathfrak g,\mathfrak a)$ \cite[Theorem~2.2]{Lew00}. The adjoint representation of a compact $\kf$ is of this type. Let $\mathfrak g=\kf_\C=\kf\oplus i\kf$, regarded as a real Lie algebra, with the Cartan involution $X+iY\mapsto X-iY$, so that $\mathfrak h=\kf$ and $\mathfrak p=i\kf$. The map $X\mapsto iX$ is a $G$-equivariant linear isometry of $\kf$ onto $\mathfrak p$ (with $\langle iX,iY\rangle:=\langle X,Y\rangle$), because $\Ad(g)$ is complex linear on $\kf_\C$. It maps $\tf$ onto the maximal abelian subspace $\mathfrak a=i\tf$, and the Weyl group $N_G(\mathfrak a)/Z_G(\mathfrak a)=N_G(\mathbb T)/\mathbb T$ of $(\mathfrak g,\mathfrak a)$ acts on $\mathfrak a\cong\tf$ as $W$ \cite[Ch.~VI]{Kna02}. Hence, for a compact semisimple $\kf$ and $\tau\in\tf$, the orthogonal projection $\operatorname{pr}\colon\kf\to\tf$ satisfies
\begin{equation}\label{lie:eq:kostant}
\operatorname{pr}\big(\Ad(G)\tau\big)=\operatorname{conv}(W\tau).
\end{equation}

\begin{lemma}[Convexity transfer]\label{lie:lem:convex}
Let $\kf$ be compact and semisimple, of any rank, let $D\subset\tf$ be $W$-invariant, and $\Omega=\Ad(G)D$. Then $\Omega\cap\tf=D$, and $\Omega$ is convex if and only if $D$ is convex.
\end{lemma}

\begin{proof}
If $\tau'\in\tf$ is conjugate to $\tau\in D$, then $\tau'\in W\tau\subset D$; hence $\Omega\cap\tf=D$, and if $\Omega$ is convex, so is $D$. Conversely, let $D$ be convex. By \eqref{lie:eq:kostant}, $\operatorname{pr}(\Omega)=\bigcup_{\tau\in D}\operatorname{conv}(W\tau)=D$, since $W\tau\subset D$ and $\tau\in\operatorname{conv}(W\tau)$. The convex hull $\operatorname{conv}\Omega$ is $\Ad(G)$-invariant, as each $\Ad(g)$ is linear, and $\operatorname{pr}(\operatorname{conv}\Omega)=\operatorname{conv}(\operatorname{pr}\Omega)=D$. As $\operatorname{pr}$ is the identity on $\tf$, $\operatorname{conv}\Omega\cap\tf\subset D$. Every element of $\operatorname{conv}\Omega$ is conjugate to an element of $\operatorname{conv}\Omega\cap\tf$, hence $\operatorname{conv}\Omega\subset\Ad(G)D=\Omega$, and $\Omega$ is convex.
\end{proof}

This is the case of adjoint representations of \cite[Theorem~3.6]{Lew00}, which holds for arbitrary invariant sets. No closedness or openness of $D$ is needed.

\subsection{The reduced problem and the balls}\label{lie:subsec:reduced}

By Lemmas~\ref{lie:lem:lift} and~\ref{lie:lem:convex}, Theorem~\ref{lie:thm} follows once we find a $W$-invariant domain $D$ with a positive real-analytic radial function, which is convex and not a disc, and an anti-invariant real-analytic $F$ on a $W$-invariant neighbourhood of $\overline D$ that satisfies \eqref{lie:eq:Pl}. The reduction is reversible in the analytic class: if $\Omega=\Ad(G)D$ and an $\Ad(G)$-invariant $u$, real analytic on a neighbourhood of $\overline\Omega$, satisfy \eqref{eq:main}, then $F=\varpi\cdot u|_\tf$ is anti-invariant and satisfies \eqref{lie:eq:Pl}, by \eqref{lie:eq:radial}; any solution can be made invariant by averaging over $G$, as in Proposition~\ref{prop:converse}.

\emph{Balls.} Let $D=\{|\tau|<R\}$, so that $\Omega$ is the ball of radius $R$. The function $F=J_m(|\tau|)\sin m\theta$ solves $\Delta F+F=0$ and is anti-invariant, and $F/\varpi=|\tau|^{-m}J_m(|\tau|)$ is the regular radial Helmholtz solution in $\R^{2m+2}$. Since $J_m'(r)=\frac mrJ_m(r)-J_{m+1}(r)$, a multiple of $F$ satisfies \eqref{lie:eq:Pl} if and only if $J_{m+1}(R)=0$ (then $J_m(R)\neq0$, and the multiple is $R^mF/J_m(R)$), in agreement with the Schiffer balls of $\R^n$ at eigenvalue one, whose radii are the zeros of $J_{n/2}=J_{m+1}$.

\emph{The enlarged symmetry.} We shall look for $D$ invariant under the dihedral group
\[
I_2(2m)=\langle\tau\mapsto\overline\tau,\ \tau\mapsto e^{i\pi/m}\tau\rangle\supset W,
\]
and for $F$ with $F(\overline\tau)=-F(\tau)$ and $F(e^{i\pi/m}\tau)=-F(\tau)$; then $F\circ w=\det(w)F$ for $w\in W$, and $\varpi$ transforms in the same way. The regular Helmholtz solutions in this class are the series $\sum_{k\geq0}a_kJ_{(2k+1)m}(r)\sin((2k+1)m\theta)$. Since $-1\in I_2(2m)$, the lifted domain is centrally symmetric; for $m=3$ the rotation through $\pi/3$ is the composition of $-1$ with an element of $W$, while for $m=4,6$ one has $-1\in W$ already. For $m=1,2$ the enlarged symmetry is that of Sections~\ref{sec:reduction} and~\ref{sec:r6}.

\emph{Near-resonances.} By Proposition~\ref{lie:prop:radial}, an $\Ad(G)$-invariant harmonic polynomial $h$ of degree $\ell$ on $\kf$ corresponds to the anti-invariant harmonic polynomial $\varpi\cdot h|_\tf$ of degree $\ell+m$. Anti-invariance under $\tau\mapsto\overline\tau$ and invariance under the rotation through $2\pi/m$ force this polynomial to be a multiple of $\im\tau^{\ell+m}$ with $m\mid\ell$. Thus the invariant spherical harmonics restrict on $\tf$ to multiples of $U_{\ell/m}(\cos m\theta)$, where $U_k$ is the Chebyshev polynomial of the second kind, and the planar mode $J_{m+\ell}(r)\sin((m+\ell)\theta)$ corresponds to the Helmholtz solution $r^{-m}J_{m+\ell}(r)U_{\ell/m}(\cos m\theta)$ in $\R^{2m+2}$; the enlarged symmetry allows $\ell\in2m\Z$. The linearisation of the Schiffer problem at the ball of radius $R$, $J_{m+1}(R)=0$, in the direction of such a harmonic is degenerate if and only if $J_{m+\ell}(R)=0$, which is impossible for $\ell\geq2$ by Siegel's theorem \cite{Sie29}, as in Section~\ref{subsec:mechanism}. (The case $\ell=1$ occurs only for $m=1$, where it corresponds to translations along the centre of $\mathfrak u(2)$.) The centres used below were located at the near-resonances
\[
(m,R,\ell)=(3,j_{4,5},6),\qquad(4,j_{5,10},8),\qquad(6,j_{7,25},12),
\]
with $j_{4,5}\approx20.826933$, $j_{5,10}\approx38.159869$, $j_{7,25}\approx88.474363$ and small divisors $J_9(j_{4,5})\approx-3.3\cdot10^{-3}$, $J_{12}(j_{5,10})\approx-4.4\cdot10^{-4}$, $J_{18}(j_{7,25})\approx-9.0\cdot10^{-4}$ (Remark~\ref{lie:rem:centres}).

\subsection{The conformal formulation for general $m$}\label{lie:subsec:conformal}

Fix $m\geq1$ and $\rho>1$. For a real sequence $c=(c_j)_{j\equiv1\,(2m)}$, $j\geq1$, put $\psi_c(w)=\sum_jc_jw^j$. Since $e^{i\pi j/m}=e^{i\pi/m}$ for $j\equiv1\pmod{2m}$,
\begin{equation}\label{lie:eq:sym}
\psi_c(\bw)=\overline{\psi_c(w)},\qquad\psi_c(e^{i\pi/m}w)=e^{i\pi/m}\psi_c(w);
\end{equation}
hence, if $\psi_c$ is injective on $\overline\D$, then $D=\psi_c(\D)$ is invariant under $I_2(2m)$.

For $n\equiv m\pmod{2m}$, $n\geq m$, and $s\geq0$ let $S_{n,s}=\im\Phi_{n,s}=r^nP^{(0,n)}_s(2r^2-1)\sin n\theta$, and let $\Sc_m$ be the real Banach space of $f=\sum_{n\equiv m\,(2m),\,s\geq0}f_{n,s}S_{n,s}$ with $\nrm f_\rho=\sum\rho^n|f_{n,s}|<\infty$, which is also its norm in $\Bc_\rho$. Since $\Phi_{k,s}(\bw)=\Phi_{-k,s}(w)$ and $\Phi_{k,s}(e^{i\pi/m}w)=e^{ik\pi/m}\Phi_{k,s}(w)$, and $e^{ik\pi/m}=-1$ exactly when $k\equiv m\pmod{2m}$, the real-valued elements $f\in\Bc_\rho$ with $f(\bw)=-f(w)$ and $f(e^{i\pi/m}w)=-f(w)$ are exactly the elements of $\Sc_m$. Let $\Gc_m=\{g\in\Sc_m:g_{n,0}=0\text{ for all }n\}$, and define $K$ on $\Gc_m$ by \eqref{eq:K}. For $m=1$ and $m=2$ these are the spaces $\Sc,\Gc$ of Section~\ref{sec:disk} and $\Sc_6,\Gc_6$ of Section~\ref{sec:r6}. Put
\[
\kappa_m=\frac1{(m+2)(m+4)} .
\]

\begin{lemma}[Results used verbatim]\label{lie:lem:carry}
The statements of Lemma~\ref{lem:carry}\textup{(a)--(f)} hold with $\Sc_6$, $\Gc_6$ and the condition $n\equiv2\pmod4$ replaced by $\Sc_m$, $\Gc_m$ and $n\equiv m\pmod{2m}$, with $1/24$ replaced by $\kappa_m$, and with the Cauchy datum $(y_1y_2,(y_2,y_1))$ in \textup{(f)} replaced by $(\varpi,\nabla\varpi)$. In particular $K\colon\Gc_m\to\Sc_m$ is bounded with $\nrm K=\kappa_m$, the supremum of $1/(q(q+2))$ over $q=n+2s\geq m+2$ being attained at $(n,s)=(m,1)$.
\end{lemma}

\begin{proof}
Only (f) needs a comment. In the coordinates $\zeta=x+iy$, $\chi=x-iy$ one has $\varpi=(\zeta^m-\chi^m)/(2i)$, whose holomorphic extension $F_0(a,b)=(\gamma(a)^m-\gamma^\dagger(b)^m)/(2i)$ in the proof of Lemma~\ref{lem:boundary} satisfies $\partial_a\partial_bF_0=0$; the rest of that proof is unchanged.
\end{proof}

Let $B>0$ (a fixed \emph{field scale}) and $b>0$ (in the application, $b=c^\circ_1$), and put
\[
\alpha=\frac{b^{m+1}}B,\qquad\omega_j=\alpha(j+m)\rho^{j+m-1},\qquad\theta_0=\frac1{\alpha(m+1)\rho^{m-1}},\qquad\theta_1=\frac1{\alpha\rho^m}.
\]
Let $X_m$ be the real Banach space of pairs $x=(g,c)$, with $g\in\Gc_m$ and $c=(c_j)_{j\equiv1\,(2m)}$ real, normed by $\nrm{(g,c)}_X=\nrm g_\rho+\nrm c$, $\nrm c=\sum_j\omega_j|c_j|$. Since $\rho^j/\omega_j=1/(\alpha(j+m)\rho^{m-1})$ and $j\rho^{j-1}/\omega_j=j/(\alpha(j+m)\rho^m)$,
\begin{equation}\label{lie:eq:theta}
N_\rho(\psi_c)\leq\theta_0\nrm c,\qquad N_\rho(\psi_c')\leq\theta_1\nrm c .
\end{equation}
The powers of $w$ in $\psi_c^m$ are $\equiv m\pmod{2m}$, so $h_c=\im(\psi_c^m)/B$ is a harmonic element of $\Sc_m$. Define
\begin{equation}\label{lie:eq:H}
\Hc_m(g,c)=g+|\psi_c'|^2\big(Kg+h_c\big),\qquad h_c=\frac{\im(\psi_c^m)}B .
\end{equation}
For $m=1$, $B=1$ this is \eqref{eq:F}, with $\omega_j=b^2(j+1)\rho^j$, $\theta_0=\vartheta_0$ and $\theta_1=\vartheta_1$; for $m=2$, $B=2$ it is \eqref{eq:F6}, with the weights of Section~\ref{sec:r6}.

\begin{lemma}\label{lie:lem:F}
$\Hc_m$ is a continuous polynomial map of degree $m+2$ from $X_m$ to $\Sc_m$. If $\Hc_m(g,c)=0$, then $V=Kg+h_c$ belongs to $C^1(\overline\D)\cap\Sc_m$, and it satisfies $\Delta V+|\psi_c'|^2V=0$ in $\mathcal D'(\D)$ and $V-h_c=0$, $\nabla(V-h_c)=0$ on $\partial\D$.
\end{lemma}

\begin{proof}
$\psi_c'$ has real coefficients and only powers divisible by $2m$, so $|\psi_c'|^2$ is real, invariant under $w\mapsto\bw$ and $w\mapsto e^{i\pi/m}w$, and multiplication by it maps $\Sc_m$ into itself with norm at most $N_\rho(\psi_c')^2$ by \eqref{eq:algebra}. Moreover $\nrm{h_c}_\rho\leq N_\rho(\psi_c)^m/B$, since $w^k=\Phi_{k,0}$ has norm $\rho^k$ and $N_\rho$ is submultiplicative. With \eqref{lie:eq:theta} and $\nrm K=\kappa_m$, \eqref{lie:eq:H} is a sum of bounded multilinear maps of degree at most $m+2$. The rest is the proof of Lemma~\ref{lem:F}, using Lemma~\ref{lie:lem:carry}.
\end{proof}

\begin{lemma}\label{lie:lem:pullback}
Let $\psi$ be holomorphic and injective on a neighbourhood of $\overline\D$ with $\psi'\neq0$ there, $D=\psi(\D)$, and let $V\in C^1(\overline\D)$ be real valued with $\Delta V+|\psi'|^2V=0$ in $\mathcal D'(\D)$ and $V-h=0$, $\nabla(V-h)=0$ on $\partial\D$, where $h=\im(\psi^m)/B$. Then $F=B\,V\circ\psi^{-1}\in C^1(\overline D)$ satisfies $\int_D(\nabla F\cdot\nabla\varphi-F\varphi)\,dA=0$ for all $\varphi\in C^\infty_c(D)$, and $F=\varpi$, $\nabla F=\nabla\varpi$ on $\partial D$. If moreover $\psi=\psi_c$ and $V\in\Sc_m$, then $F(\overline\tau)=-F(\tau)$ and $F(e^{i\pi/m}\tau)=-F(\tau)$.
\end{lemma}

\begin{proof}
Since $\varpi\circ\psi=\im\psi^m=Bh$, we have $B(V-h)=(F-\varpi)\circ\psi$, and the proof of Lemma~\ref{lem:pullback} applies. The parities follow from \eqref{lie:eq:sym} and the characterisation of $\Sc_m$.
\end{proof}

\emph{Expansion at a centre.} Let $x^\circ=(g^\circ,c^\circ)\in X_m$ have finitely many nonzero coefficients, $b=c^\circ_1$, and put
\[
p=\psi_{c^\circ}',\qquad P=N_\rho(p),\qquad C=N_\rho(\psi_{c^\circ}),\qquad V^\circ=Kg^\circ+h_{c^\circ},\qquad V_0=\nrm{V^\circ}_\rho .
\]
For $h=(\delta g,\eta)\in X_m$, $|\psi'_{c^\circ+\eta}|^2=|p|^2+L(\eta)+Q(\eta,\eta)$ with $L(\eta)=2\re(\overline p\,\psi_\eta')$ and $Q(\eta,\eta')=\re(\psi_\eta'\overline{\psi_{\eta'}'})$, and $h_{c^\circ+\eta}=h_{c^\circ}+\sum_{k=1}^mH_k(\eta)$ with $H_k(\eta)=\binom mk\im(\psi_{c^\circ}^{m-k}\psi_\eta^k)/B$. Hence
\begin{equation}\label{lie:eq:DH}
D\Hc_m(x^\circ)h=\delta g+|p|^2\big(K\delta g+H_1(\eta)\big)+L(\eta)V^\circ,
\end{equation}
and $\Hc_m(x^\circ+h)=\Hc_m(x^\circ)+D\Hc_m(x^\circ)h+\sum_{q=2}^{m+2}B_q(h,\ldots,h)$, where the degree-$q$ part consists of the terms $|p|^2H_q(\eta)$ ($q\leq m$), $L(\eta)H_{q-1}(\eta)$ ($2\leq q\leq m+1$), $Q(\eta,\eta)H_{q-2}(\eta)$ ($3\leq q\leq m+2$), $L(\eta)K\delta g+Q(\eta,\eta)V^\circ$ ($q=2$) and $Q(\eta,\eta)K\delta g$ ($q=3$), and $B_q$ is the symmetric $q$-linear map with this diagonal.

\begin{lemma}\label{lie:lem:nonlinear}
Let $h_k=\binom mkC^{m-k}\theta_0^k/B$ for $1\leq k\leq m$, and $h_k=0$ otherwise. For $2\leq q\leq m+2$ and all $h_i\in X_m$, $\nrm{B_q(h_1,\ldots,h_q)}_\rho\leq c_q\nrm{h_1}_X\cdots\nrm{h_q}_X$, where
\begin{equation}\label{lie:eq:cq}
c_q=P^2h_q+2P\theta_1h_{q-1}+\theta_1^2h_{q-2}+\one_{q=2}\big(2P\theta_1\kappa_m+\theta_1^2V_0\big)+\one_{q=3}\,\theta_1^2\kappa_m .
\end{equation}
\end{lemma}

\begin{proof}
By \eqref{lie:eq:theta}, \eqref{eq:algebra}, $\nrm{\im f}_\rho\leq\nrm f_\rho$ and submultiplicativity of $N_\rho$,
\[
\nrm{L(\eta)f}_\rho\leq2P\theta_1\nrm\eta\nrm f_\rho,\qquad\nrm{Q(\eta,\eta')f}_\rho\leq\theta_1^2\nrm\eta\nrm{\eta'}\nrm f_\rho,\qquad\nrm{K\delta g}_\rho\leq\kappa_m\nrm{\delta g}_\rho,
\]
and the symmetric $k$-linear map with diagonal $H_k$ is bounded by $h_k$, because it is a symmetrised product of $k$ multiplications by $\psi_\eta$, each bounded by $\theta_0$, and of the fixed factor $\binom mk\psi_{c^\circ}^{m-k}/B$. Each degree-$q$ term listed above is a product of $q$ such linear factors, and the symmetrisation of a product of linear maps is bounded by the product of their norms. Adding the bounds of the terms of degree $q$ gives $c_q$.
\end{proof}

\begin{proposition}\label{lie:prop:NK}
Let $A\colon\Sc_m\to X_m$ be bounded, linear and injective. Suppose that $\nrm{A\Hc_m(x^\circ)}_X\leq Y$, $\nrm{I-AD\Hc_m(x^\circ)}_{X\to X}\leq Z$ and $\nrm A\,c_q\leq C_q$ for $2\leq q\leq m+2$, and that $r>0$ satisfies
\begin{equation}\label{lie:eq:radii}
Y+(Z-1)r+\sum_{q=2}^{m+2}C_qr^q<0,\qquad Z-1+\sum_{q=2}^{m+2}qC_qr^{q-1}<0 .
\end{equation}
Then $\Hc_m$ has exactly one zero in the closed ball $\overline B_{X_m}(x^\circ,r)$.
\end{proposition}

\begin{proof}
This is the proof of Proposition~\ref{prop:NK6}: $B_q(h_1,\ldots,h_1)-B_q(h_2,\ldots,h_2)$ is a sum of $q$ terms, each containing $h_1-h_2$ once, so its image under $A$ has norm at most $qC_qr^{q-1}\nrm{h_1-h_2}_X$ on the ball.
\end{proof}

\subsection{The principal part and the approximate inverse}\label{lie:subsec:inverse}

\begin{proposition}\label{lie:prop:principal}
For every $x=(g,c)\in X_m$ and $j\equiv1\pmod{2m}$,
\[
D\Hc_m(x)e_j=2\re\big(jw^{j-1}\overline{\psi_c'}\big)(Kg+h_c)+\frac mB|\psi_c'|^2\im\big(\psi_c^{m-1}w^j\big).
\]
Let $x_b=(0,b\,e_1)$, so that $\psi_{x_b}(w)=bw$. For $n\geq0$ write
\begin{equation}\label{lie:eq:qs}
r^{n+2m}\sin n\theta=\sum_{s=0}^mq_s(n)S_{n,s};
\end{equation}
then $q_s(n)\geq0$, $\sum_sq_s(n)=1$ and $q_0(n)=(n+1)/(n+m+1)$. For every $j\equiv1\pmod{2m}$ with $j\geq2m+1$,
\begin{equation}\label{lie:eq:Tj}
T_j:=D\Hc_m(x_b)e_j=\alpha\Big\{(j+m)S_{j+m-1,0}-j\sum_{s=0}^mq_s(j-m-1)S_{j-m-1,s}\Big\}.
\end{equation}
\end{proposition}

\begin{proof}
The first formula follows by differentiating $|\psi_c'+\eps jw^{j-1}|^2$ and $\im((\psi_c+\eps w^j)^m)/B$ at $\eps=0$. The statements on $q_s$ are Lemma~\ref{lem:monomial} for $w^{n+m}\bw^m=r^{n+2m}e^{in\theta}$, after taking imaginary parts. At $x_b$ we have $Kg+h_c=b^m\im(w^m)/B$ and $\psi_c'=b$, so
\begin{align*}
T_j&=\alpha\big[2j\re(w^{j-1})\im(w^m)+m\im(w^{j+m-1})\big]\\
&=\alpha r^{j+m-1}\big[(j+m)\sin((j+m-1)\theta)-j\sin((j-m-1)\theta)\big],
\end{align*}
using $2\cos((j-1)\theta)\sin(m\theta)=\sin((j+m-1)\theta)-\sin((j-m-1)\theta)$. With $n=j-m-1\geq m$, $n\equiv m\pmod{2m}$, \eqref{lie:eq:qs} gives \eqref{lie:eq:Tj}.
\end{proof}

The harmonic part of $T_j$ is $\alpha(j+m)$ in row $(j+m-1,0)$ and $-\alpha(j-m)$ in row $(j-m-1,0)$, a difference operator of step $2m$; the nonharmonic part lies in the rows $(j-m-1,s)$, $1\leq s\leq m$, with total mass $\alpha j\sum_{s\geq1}q_s(j-m-1)=\alpha m$. For $m=1,2$, \eqref{lie:eq:Tj} reduces to Lemma~\ref{lem:Tform} and to \eqref{eq:Tj6} (Lemma~\ref{lem:radial4} gives $q_s$ for $m=2$).

\emph{Finite block and tail.} Let $M\equiv m\pmod{2m}$ and $S\geq m$. The \emph{finite inputs} are the coordinates $g_{n,s}$ with $n\leq M$, $1\leq s\leq S$, and $c_j$ with $j\leq M-m+1$; the \emph{finite rows} are $(n,s)$ with $n\leq M$, $1\leq s\leq S$, and $(n,0)$ with $n\leq M$. All others are \emph{tail} inputs and rows, $X_m=X_f\oplus X_t$, $\Sc_m=\Sc_f\oplus\Sc_t$, and $\dim X_f=\dim\Sc_f=(S+1)(M+m)/(2m)$. Let $M_f$ be the matrix of $D\Hc_m(x^\circ)$ on finite inputs and finite rows. On tail inputs we use the model $T(\delta g,\eta)=\delta g+\sum_{j\geq M+m+1}\eta_jT_j$. By \eqref{lie:eq:Tj}, only $\eta_{M+m+1}$ reaches finite rows, since $S\geq m$:
\begin{equation}\label{lie:eq:B}
B_f(\delta g,\eta)=-\eta_{M+m+1}\,\alpha\Big\{(M+1)e_{(M,0)}+(M+m+1)\sum_{s=1}^mq_s(M)e_{(M,s)}\Big\}.
\end{equation}
Write $Te=(B_fe,T_{tt}e)$ for tail inputs $e$, and put $\tau_m=(1-\rho^{-2m})^{-1}$.

\begin{lemma}\label{lie:lem:Ttt}
$T_{tt}\colon X_t\to\Sc_t$ is a bounded bijection with $\nrm{T_{tt}}\leq1+\rho^{-2m}$. For a tail residual $y=\sum y_{n,s}S_{n,s}\in\Sc_t$ its inverse is
\begin{gather}
\eta_{n-m+1}=\sum_{k\geq0}\frac{y_{n+2mk,0}}{\alpha(n+2mk+1)},\qquad\delta g_{n,s}=y_{n,s}+\alpha(n+m+1)q_s(n)\,\eta_{n+m+1}\quad(1\leq s\leq m),\label{lie:eq:Tinv}
\end{gather}
for $n\geq M+2m$, $n\equiv m\pmod{2m}$, and $\delta g_{n,s}=y_{n,s}$ for all other tail $g$-inputs.
\end{lemma}

\begin{proof}
By \eqref{lie:eq:Tj}, the tail harmonic rows read $\alpha(n+1)(\eta_{n-m+1}-\eta_{n+m+1})=y_{n,0}$ for $n\geq M+2m$, the rows $(n,s)$, $1\leq s\leq m$, $n\geq M+2m$, read $\delta g_{n,s}-\alpha(n+m+1)q_s(n)\eta_{n+m+1}=y_{n,s}$, and all other tail rows read $\delta g_{n,s}=y_{n,s}$. The rest is the proof of Lemma~\ref{lem:Ttt6}: $\eta_j\to0$ gives uniqueness and \eqref{lie:eq:Tinv}, and the column $T_{tt}e_j$, $j\geq M+3m+1$, has weighted norm $\omega_j(1+j\rho^{-2m}/(j+m))\leq(1+\rho^{-2m})\omega_j$.
\end{proof}

Let $\widehat A$ be a matrix with dyadic entries, regarded as an exact linear map $\Sc_f\to X_f$, with
\begin{equation}\label{lie:eq:Ahat}
\nrm{I-\widehat AM_f}_{X_f\to X_f}<1,
\end{equation}
and define $A=\left(\begin{smallmatrix}\widehat A&-\widehat AB_fT_{tt}^{-1}\\0&T_{tt}^{-1}\end{smallmatrix}\right)\colon\Sc_f\oplus\Sc_t\to X_f\oplus X_t$. By Lemma~\ref{lem:A} (Lemma~\ref{lie:lem:carry}), $A$ is bounded and injective, and $ATe=e$ for every tail input $e$.

\begin{lemma}\label{lie:lem:normA}
Let $\nrm R=\max_{(n,s)\ \text{finite row}}\nrm{\widehat Ae_{(n,s)}}_X/\rho^n$,
\[
H=\Big\|\widehat A\,\alpha\Big((M+1)e_{(M,0)}+(M+m+1)\sum_{s=1}^mq_s(M)e_{(M,s)}\Big)\Big\|_X,
\]
and, for $n\geq M+2m$, $n\equiv m\pmod{2m}$,
\begin{equation}\label{lie:eq:hn}
h_n=\frac1{(n+1)\rho^n}\Big(\sum_{\substack{M+2m\leq k\leq n\\k\equiv m\,(2m)}}(k+1)\rho^k+m\sum_{\substack{M+2m\leq k\leq n-2m\\k\equiv m\,(2m)}}\rho^k+\frac H\alpha\Big).
\end{equation}
Let $N_H\geq M+2m$, $N_H\equiv m\pmod{2m}$. Then $A$ maps every residual supported in tail rows with norm at most
\[
A_t:=\max\Big\{1,\ \max_{M+2m\leq n\leq N_H}h_n,\ \tau_m+\frac{m\tau_m\rho^{-2m}}{N_H+2m+1}+\frac H{\alpha(N_H+2m+1)\rho^{N_H+2m}}\Big\},
\]
it is the identity on residuals supported in nonharmonic tail rows, and $\nrm A\leq\max\{\nrm R,A_t\}$.
\end{lemma}

\begin{proof}
As in Lemma~\ref{lem:normA6}: the normalised harmonic tail row $\rho^{-n}S_{n,0}$ is mapped to $\eta_{k-m+1}=\rho^{-n}/(\alpha(n+1))$ for $M+2m\leq k\leq n$, of weights $\omega_{k-m+1}=\alpha(k+1)\rho^k$; to the coordinates $\delta g_{k,s}$, $1\leq s\leq m$, $M+2m\leq k\leq n-2m$, of weight $\rho^k$ and total absolute value $\alpha m\rho^{-n}/(\alpha(n+1))$, because $\sum_{s\geq1}(k+m+1)q_s(k)=m$; and to the finite part $\eta_{M+m+1}\widehat A\alpha(\ldots)$, of norm $H\rho^{-n}/(\alpha(n+1))$. The sum is $h_n$. For $n\geq N_H+2m$ we use $(k+1)/(n+1)\leq1$ and sum the geometric series.
\end{proof}

\subsection{The defect bounds}\label{lie:subsec:Z}

We bound $\nrm{I-AD\Hc_m(x^\circ)}$ by Lemma~\ref{lem:columns}, splitting the inputs into five classes; every input belongs to exactly one class, and no column is estimated by sampling. Let $d=\deg p=M-m$ and $L=M+d$.

\emph{Residual and finite columns.} $\Hc_m(x^\circ)$ and the columns $D\Hc_m(x^\circ)e$, $e$ a finite input, are polynomials in $w,\bw$; they are expanded in the basis $S_{n,s}$ without truncation, $A$ is applied exactly using \eqref{lie:eq:B} and \eqref{lie:eq:Tinv}, and $Y$ and $Z_f$ bound $\nrm{A\Hc_m(x^\circ)}_X$ and the normalised norms of the finite columns of $I-AD\Hc_m(x^\circ)$.

\begin{lemma}\label{lie:lem:gtail}
Let $e=e_{(n,s)}$ be a tail $g$-input and $q=n+2s$. Then $(I-AD\Hc_m(x^\circ))e=-A(|p|^2KS_{n,s})$. Let $p=p_{\leq d_s}+p_{>d_s}$ be the splitting of $p$ into the terms of degree $\leq d_s$ and $>d_s$, where $d_s\equiv0\pmod{2m}$ is fixed, $P_s=N_\rho(p_{\leq d_s})$, $P_t=N_\rho(p_{>d_s})$ and $\delta_a=2P_sP_t+P_t^2$. Then
\begin{equation}\label{lie:eq:gshort}
\frac{\nrm{(I-AD\Hc_m(x^\circ))e}_X}{\rho^n}\leq\frac{\nrm{A(|p_{\leq d_s}|^2KS_{n,s})}_X}{\rho^n}+\frac{\nrm A\,\delta_a}{q(q+2)} .
\end{equation}
Moreover, the right-hand side is at most $(A_tP_s^2+\nrm A\,\delta_a)/((n_f+2)(n_f+4))$ if $n\geq n_f:=M+d_s+2m$, and at most $(P_s^2+\nrm A\,\delta_a)/((m+2s_f)(m+2s_f+2))$ if $s\geq s_f:=S+d_s+2$.
\end{lemma}

\begin{proof}
This is the proof of Lemma~\ref{lem:gtail6}. Multiplication by $|p_{\leq d_s}|^2$ changes frequencies by at most $d_s$, so for $n\geq n_f$ all components lie in rows of frequency at least $M+2m$, which are tail rows; and by \eqref{eq:K} and Lemma~\ref{lem:radial}, for $s\geq s_f$ all radial indices are at least $s-1-d_s>S$, so all components lie in nonharmonic tail rows, on which $A$ is the identity. Finally $q\geq n_f+2$, respectively $q\geq m+2s_f$.
\end{proof}

The tail $g$-inputs with $n\leq M+d_s$ and $s\leq S+d_s+1$ are the \emph{near} ones; for them the first term of \eqref{lie:eq:gshort} is computed without truncation, and $Z_{g,\partial}$ bounds the maximum of the right-hand side. The remaining tail $g$-inputs are covered by the last statement of Lemma~\ref{lie:lem:gtail}, and $Z_{g,\infty}$ is the larger of the two bounds.

\emph{Shape columns.} For a tail shape input $e_j$, $j\geq M+m+1$, we have $(I-AD\Hc_m(x^\circ))e_j=-A\Delta_j$ with $\Delta_j=D\Hc_m(x^\circ)e_j-T_j$. For the \emph{near} shape inputs $M+m+1\leq j\leq\jst-2m$, $\Delta_j$ is a polynomial, the columns are computed without truncation, and $Z_{\mathrm{sh},\partial}$ bounds their normalised norms. Let $W^\circ=Kg^\circ$, $h^\circ=h_{c^\circ}$ and $h_{be_1}=b^m\im(w^m)/B$. By Proposition~\ref{lie:prop:principal},
\begin{equation}\label{lie:eq:Deltaj}
\Delta_j=2\re\big(jw^{j-1}\Gamma_1\big)+m\im\big(w^j\Gamma_0\big)+2\re\big(jw^{j-1}\overline pW^\circ\big),
\end{equation}
where
\[
\Gamma_1=\overline ph^\circ-bh_{be_1},\qquad\Gamma_0=\frac{|p|^2\psi_{c^\circ}^{m-1}}B-\alpha w^{m-1}.
\]
By Lemma~\ref{lem:K}(c), $W^\circ=(1-|w|^2)^2Q^\circ$ with $Q^\circ=\sum_{n,s}\frac{g^\circ_{n,s}}{4s(s+1)}r^nP^{(2,n)}_{s-1}(2r^2-1)\sin n\theta$, whose finite expansion in the basis $S_{n,t}$ is given by the connection formulas of Section~\ref{sec:Z}, with positive rational coefficients. The function $w^L\overline pQ^\circ$ has only nonnegative frequencies; write $w^L\overline pQ^\circ=\sum_{k,t\geq0}\gamma^W_{k,t}\Phi_{k,t}$, $\Gamma_1=\sum\gamma^1_{k,t}\Phi_{k,t}$ and $\Gamma_0=\sum\gamma^0_{k,t}\Phi_{k,t}$ (finite sums, computed in interval arithmetic). For $\jst\equiv1\pmod{2m}$ with $\jst>L$ let
\begin{gather*}
N_1^+=\sum_{k+\jst-1>M}\rho^{|k|}|\gamma^1_{k,t}|,\qquad N_0^+=\sum_{k+\jst>M}\rho^{|k|}|\gamma^0_{k,t}|,\\
\Sigma_W^+=\sum_{k+\jst-1-L>M}\frac{|\gamma^W_{k,t}|(2t+1)(2t+3)\rho^{k-L-m}}{(k+\jst-L)^2},
\end{gather*}
and let $N_1^-$, $N_0^-$, $\Sigma_W^-$ be the corresponding sums over the remaining indices.

\begin{proposition}\label{lie:prop:farshape}
For every $j\equiv1\pmod{2m}$ with $j\geq\jst>L$,
\[
\frac{\nrm{(I-AD\Hc_m(x^\circ))e_j}_X}{\omega_j}\leq Z_{\mathrm{sh},\infty}:=\sU'_\Gamma+\sU'_W,
\]
where
\[
\sU'_\Gamma=\frac1\alpha\Big(\frac{2(A_tN_1^++\nrm AN_1^-)}{\rho^m}+\frac{m(A_tN_0^++\nrm AN_0^-)}{\rho^{m-1}(\jst+m)}\Big),\qquad\sU'_W=\frac8\alpha\big(A_t\Sigma_W^++\nrm A\,\Sigma_W^-\big).
\]
If moreover $\jst-1-\big((m+1)M-m^2\big)>M$, then $N_1^-=N_0^-=\Sigma_W^-=0$ and $Z_{\mathrm{sh},\infty}\leq A_t(\sU_\Gamma+\sU_W)$ with
\[
\sU_\Gamma=\frac1\alpha\Big(\frac{2\nrm{\Gamma_1}_\rho}{\rho^m}+\frac{m\nrm{\Gamma_0}_\rho}{\rho^{m-1}(\jst+m)}\Big),\qquad\sU_W=\frac{8}{\alpha(\jst-L)^2}\sum_{k,t}|\gamma^W_{k,t}|(2t+1)(2t+3)\rho^{k-L-m}.
\]
\end{proposition}

\begin{proof}
Fix $j\geq\jst$. By Lemma~\ref{lem:mult}, $w^{j-1}\Phi_{k,t}$ is a combination of $\Phi_{k+j-1,t'}$ of norm at most $\rho^{j-1+|k|}$, and taking real or imaginary parts only changes the signs of frequencies. Hence the contribution of $\gamma^1_{k,t}$ to $2\re(jw^{j-1}\Gamma_1)$ has norm at most $2j\rho^{j-1}\rho^{|k|}|\gamma^1_{k,t}|$ and lies in rows of frequency $|k+j-1|$; if $k+\jst-1>M$ these are tail rows, on which $A$ has norm at most $A_t$ (Lemma~\ref{lie:lem:normA}), and otherwise we use $\nrm A\geq A_t$. Dividing by $\omega_j=\alpha(j+m)\rho^{j+m-1}$ and using $j/(j+m)\leq1$ gives the first term of $\sU'_\Gamma$. The term $m\im(w^j\Gamma_0)$ is treated in the same way, with $m\rho^j/\omega_j=m/(\alpha(j+m)\rho^{m-1})\leq m/(\alpha(\jst+m)\rho^{m-1})$. Finally $jw^{j-1}\overline pW^\circ=j(1-|w|^2)^2w^{j-1-L}(w^L\overline pQ^\circ)$ with $j-1-L\geq0$, and by Lemma~\ref{lem:doublezero} with $N=j-1-L$ the contribution of $\gamma^W_{k,t}$ has frequency $k+j-1-L$ and normalised norm at most
\[
\frac{2j}{\alpha(j+m)\rho^{j+m-1}}\,|\gamma^W_{k,t}|\frac{4(2t+1)(2t+3)\rho^{k+j-1-L}}{(k+j-L)^2}\leq\frac{8|\gamma^W_{k,t}|(2t+1)(2t+3)\rho^{k-L-m}}{\alpha(k+j-L)^2}.
\]
As $j$ increases, every output frequency increases, so the sets of indices counted with $A_t$ can only grow, and every factor above decreases; hence the bound at $\jst$ covers all $j\geq\jst$. For the last statement: the frequencies of $\overline p$, $h^\circ$, $\psi_{c^\circ}$ and $Q^\circ$ lie in $[-d,0]$, $[-m(M-m+1),m(M-m+1)]$, $[0,M-m+1]$ and $[-M,M]$, so the frequencies of $\Gamma_1$, $\Gamma_0$ and $\overline pQ^\circ$ are all at least $-(m(M-m+1)+d)=-((m+1)M-m^2)$, and $L=2M-m\leq(m+1)M-m^2$. Under the stated condition all output frequencies therefore exceed $M$, so that $N_1^-=N_0^-=\Sigma_W^-=0$, and $(k+\jst-L)^2\geq(\jst-L)^2$ since $k\geq0$.
\end{proof}

Collecting the five classes, the number
\[
Z=\max\{Z_f,\ Z_{g,\partial},\ Z_{g,\infty},\ Z_{\mathrm{sh},\partial},\ Z_{\mathrm{sh},\infty}\}
\]
bounds $\nrm{I-AD\Hc_m(x^\circ)}_{X\to X}$.

\subsection{Certified bounds and the exact zeros}\label{lie:subsec:zero}

For $m=3,4,6$ we take $\rho=11/10$ and the data of Table~\ref{lie:tab:param}. Each centre $x^\circ$ consists of the finite inputs listed there, given as binary64 numbers interpreted as exact dyadic rationals. The field scale $B$ is such a number as well: $B=1$ for $m=3$, and $B\approx j_{m+1,k}^m$ is the binary64 number with hexadecimal representation \texttt{0x1.02d7fa4b541bbp+21} for $m=4$ and \texttt{0x1.beb0441320216p+38} for $m=6$. The finite inverse $\widehat A$ is a fixed binary64 approximation of $M_f^{-1}$ whose entries are interpreted as exact dyadic rationals; only \eqref{lie:eq:Ahat} is used. All quantities of Sections~\ref{lie:subsec:conformal}--\ref{lie:subsec:Z} were evaluated in ball arithmetic at $128$ bits (Section~\ref{lie:subsec:cap}). Table~\ref{lie:tab:bounds} lists the results, rounded outward.

\begin{table}[ht]
\centering
\small
\begin{tabular}{@{}llll@{}}
\toprule
 & $n=8$ & $n=10$ & $n=14$\\
\midrule
$(\kf,m)$ & $(\mathfrak{su}(3),3)$ & $(\mathfrak{so}(5),4)$ & $(\mathfrak g_2,6)$\\
$(M,S)$ & $(45,24)$ & $(52,36)$ & $(114,60)$\\
finite $g_{n,s}$ & $8\times24$ & $7\times36$ & $10\times60$\\
finite $c_j$ & $c_1,c_7,\ldots,c_{43}$ & $c_1,c_9,\ldots,c_{49}$ & $c_1,c_{13},\ldots,c_{109}$\\
$\dim X_f$ & $200$ & $259$ & $610$\\
$B$ & $1$ & $\approx2120447.2868$ & $\approx4.7962817\cdot10^{11}$\\
$b=c^\circ_1$ & $20.78702921089377\ldots$ & $38.15292567121707\ldots$ & $88.49730472771703\ldots$\\
$d_s$; $n_f$, $s_f$ & $24$; $75$, $50$ & $24$; $84$, $62$ & $60$; $186$, $122$\\
$\jst$; $L$ & $223$; $87$ & $305$; $100$ & $241$; $222$\\
$N_H$ & $645$ & $852$ & $1314$\\
$r$ & $1/30000$ & $10^{-7}$ & $1/6000$\\
\bottomrule
\end{tabular}
\caption{Parameters of the three computer-assisted proofs ($\rho=11/10$). The finite shape inputs are $c_j$, $j\equiv1\pmod{2m}$, $j\leq M-m+1$; the finite $g$-inputs are $g_{n,s}$, $n\equiv m\pmod{2m}$, $n\leq M$, $1\leq s\leq S$. $B$ and $b$ are exact dyadic rationals.}
\label{lie:tab:param}
\end{table}

\begin{table}[tp]
\centering
\footnotesize
\begin{tabular}{@{}llll@{}}
\toprule
Quantity & $n=8$ & $n=10$ & $n=14$\\
\midrule
$P$ & $<21.788460$ & $<38.428146$ & $<91.291342$\\
$C$ & $<23.020116$ & $<42.001369$ & $<97.545293$\\
$V_0$ & $<131547.591$ & $<28.746248$ & $<79.018365$\\
$\nrm{I-\widehat AM_f}$ & $<2.2858\cdot10^{-14}$ & $<5.7259\cdot10^{-12}$ & $<5.0709\cdot10^{-9}$\\
$\nrm R$ & $<963.7779$ & $<10194.7512$ & $<11766.8786$\\
$H$ & $<1.4540263\cdot10^9$ & $<718348.086$ & $<1.2223739\cdot10^9$\\
$A_t$ & $<2.302038$ & $<2.014454$ & $<1.661175$\\
$\nrm A$ (ceiling) & $964$ & $10195$ & $11767$\\
\midrule
$Y$ & $<1.519444\cdot10^{-5}$ & $<9.591861\cdot10^{-9}$ & $<4.592200\cdot10^{-5}$\\
$Z_f$ & $<0.052262$ & $<0.060605$ & $<0.120542$\\
$P_s$ & $<21.788460$ & $<38.428143$ & $<91.291236$\\
$P_t$ & $<1.4978\cdot10^{-7}$ & $<2.6609\cdot10^{-6}$ & $<1.05815\cdot10^{-4}$\\
$\delta_a$ & $<6.5267\cdot10^{-6}$ & $<2.04502\cdot10^{-4}$ & $<0.0193200$\\
near $g$-columns; max.\ at & $396$; $(51,1)$ & $358$; $(60,1)$ & $1215$; $(126,1)$\\
$Z_{g,\partial}$ & $<0.194306$ & $<0.462762$ & $<0.576826$\\
far bound, $n\geq n_f$ & $<0.179735$ & $<0.393351$ & $<0.393946$\\
far bound, $s\geq s_f$ & $<0.043897$ & $<0.088871$ & $<0.135896$\\
near shape columns; max.\ at & $29$; $j=49$ & $31$; $j=57$ & $10$; $j=121$\\
$Z_{\mathrm{sh},\partial}$ & $<0.284628$ & $<0.552897$ & $<0.600896$\\
$\nrm{\Gamma_1}_\rho$ resp.\ $N_1^+$ & $<14093.696$ & $<0.466786$ & $<5.147419$\\
$N_1^-$ & --- & --- & $<3.3948\cdot10^{-8}$\\
$\nrm{\Gamma_0}_\rho$ resp.\ $N_0^+$ & $<21422.774$ & $<0.692754$ & $<7.696736$\\
$N_0^-$ & --- & --- & $0$\\
$\Sigma_W^+$ & --- & --- & $<3.293822$\\
$\Sigma_W^-$ & --- & --- & $<1.0394\cdot10^{-17}$\\
$\sU_\Gamma$ resp.\ $\sU'_\Gamma$ & $<0.114684$ & $<0.016902$ & $<0.111093$\\
$\sU_W$ resp.\ $\sU'_W$ & $<0.067068$ & $<0.088733$ & $<0.493856$\\
$Z_{\mathrm{sh},\infty}$ (all $j\geq\jst$) & $<0.418572$ & $<0.212795$ & $<0.604948$\\
$Z$ & $<0.418572$ & $<0.552897$ & $<0.604948$\\
\midrule
$C_2$ & $<0.0072192$ & $<395.40418$ & $<210.53260$\\
$C_3$ & $<4.8836\cdot10^{-10}$ & $<0.071133$ & $<0.0060610$\\
$C_4$ & $<1.5491\cdot10^{-18}$ & $<3.2353\cdot10^{-7}$ & $<1.9135\cdot10^{-9}$\\
$C_5$ & $<2.1152\cdot10^{-26}$ & $<1.7470\cdot10^{-11}$ & $<2.2718\cdot10^{-14}$\\
$C_6$ & --- & $<3.7236\cdot10^{-16}$ & $<1.5942\cdot10^{-19}$\\
$C_7$ & --- & --- & $<6.1326\cdot10^{-25}$\\
$C_8$ & --- & --- & $<9.9963\cdot10^{-31}$\\
\bottomrule
\end{tabular}
\caption{Certified bounds, rounded outward. The row $\nrm A$ gives the rational ceiling for $\max\{\nrm R,A_t\}$ used in $C_q=\nrm A\,c_q$; the tail estimates use this ceiling ($n=8$) or the certified enclosure of $\nrm R$ ($n=10,14$). For $n=8,10$ the far shape bound is the second one of Proposition~\ref{lie:prop:farshape}, $Z_{\mathrm{sh},\infty}=A_t(\sU_\Gamma+\sU_W)$ (for $n=8$ with the ceiling $A_t\leq2.303$); for $n=14$ it is the first one, $Z_{\mathrm{sh},\infty}=\sU'_\Gamma+\sU'_W$. The maximum defining $A_t$ is attained by the remainder term for $n=8$, and at $n=M+2m$ for $n=10,14$. $P$, $C$ and $V_0$ are intermediate values of the final verification scripts; the complete enclosures of the other quantities are recorded in the certificates and receipts (Section~\ref{lie:subsec:cap}).}
\label{lie:tab:bounds}
\end{table}

\begin{theorem}\label{lie:thm:zero}
Let $m\in\{3,4,6\}$, and let $M$, $S$, $B$, $x^\circ$ and $r$ be as in Table~\ref{lie:tab:param}. There is exactly one $x^*=(g^*,c^*)\in X_m$ with $\nrm{x^*-x^\circ}_X\leq r$ and $\Hc_m(x^*)=0$.
\end{theorem}

\begin{proof}
The condition \eqref{lie:eq:Ahat} holds by Table~\ref{lie:tab:bounds}, so $A$ is bounded and injective. The hypotheses of Proposition~\ref{lie:prop:farshape} hold: for $m=3$ and $m=4$, $\jst-1-((m+1)M-m^2)$ equals $222-171=51>45$ and $304-244=60>52$, and for $m=6$, $\jst=241>L=222$. By Table~\ref{lie:tab:bounds}, Proposition~\ref{lie:prop:NK} applies with the following rational majorants:
\begin{itemize}
\item $m=3$: $Y=153\cdot10^{-7}$, $Z=43/100$, $C_2=1/125$, $C_3=5\cdot10^{-10}$, $C_4=2\cdot10^{-18}$, $C_5=3\cdot10^{-26}$;
\item $m=4$: $Y=10^{-8}$, $Z=56/100$, $C_2=396$, $C_3=72/1000$, $C_4=4\cdot10^{-7}$, $C_5=2\cdot10^{-11}$, $C_6=4\cdot10^{-16}$;
\item $m=6$: $Y=46\cdot10^{-6}$, $Z=61/100$, $C_2=211$, $C_3=61\cdot10^{-4}$, $C_4=2\cdot10^{-9}$, $C_5=3\cdot10^{-14}$, $C_6=2\cdot10^{-19}$, $C_7=7\cdot10^{-25}$, $C_8=2\cdot10^{-30}$.
\end{itemize}
For these values and $r=1/30000$, $10^{-7}$, $1/6000$, respectively, exact rational arithmetic gives for the two expressions in \eqref{lie:eq:radii}
\begin{gather*}
-3.699991111111111\ldots\cdot10^{-6}\ \text{ and }\ -0.5699994666666666\ldots\qquad(m=3),\\
-3.399603999999992\ldots\cdot10^{-8}\ \text{ and }\ -0.4399207999999978\ldots\qquad(m=4),\\
-1.313888886064814\ldots\cdot10^{-5}\ \text{ and }\ -0.3196666661583333\ldots\qquad(m=6),
\end{gather*}
all negative.
\end{proof}

\begin{remark}[The centres]\label{lie:rem:centres}
The centres were obtained numerically, and their provenance plays no role in the proof. A Fourier--Bessel collocation of \eqref{lie:eq:Pl} in the enlarged symmetry class, at the ball wavenumber $j_{m+1,k}$ of Section~\ref{lie:subsec:reduced} and seeded with the near-resonant mode $\ell=2m$, produced numerical solutions with Cauchy defects of order $10^{-12}$ or smaller. Their domains were mapped conformally to the disc by Theodorsen's method, and Newton's method was applied to the Galerkin truncation of \eqref{lie:eq:H}. For $m=4,6$ the unscaled field has boundary data of size about $j_{m+1,k}^m$, and double-precision Newton iterations stagnated at residuals of about $10^{-6}$ and $2$; dividing the field by the frozen constant $B$, which does not change the Cauchy data of $F=B\,V\circ\psi^{-1}$, removed this obstruction. For $m=6$ a first centre with $M=54$ gave $Y\approx1.7$, the defect being concentrated in the omitted angular frequency $66$; this led to $M=114$. For orientation we record some numerical observations about the centres, not used in the proofs. The boundary radius of $D$ lies between $20.7086$ and $20.8635$ ($m=3$), $38.1390$ and $38.1667$ ($m=4$), $88.4477$ and $88.5402$ ($m=6$), so that $\partial\Omega$ deviates from the sphere of radius $j_{m+1,k}$ by between $-0.57\%$ and $+0.18\%$, $-0.055\%$ and $+0.018\%$, $-0.030\%$ and $+0.074\%$, respectively; the minimum of $\re(1+w\psi''/\psi')$ on $\partial\D$ is about $0.832$, $0.973$ and $0.896$, close to the certified lower bounds of Table~\ref{lie:tab:geom}.
\end{remark}

\subsection{Geometry and proof of Theorem~\ref{lie:thm}}\label{lie:subsec:proof}

\begin{proposition}\label{lie:prop:geometry}
Let $m\in\{3,4,6\}$, let $c$ be a real sequence indexed by $j\equiv1\pmod{2m}$ with $\sum_j\omega_j|c_j-c^\circ_j|\leq r$ \textup{(}for instance $c=c^*$\textup{)}, and put $R_e=21/20$, $E_1=r\theta_1$, $E_2=r/(\alpha\rho^{m-1}e\log\rho)$,
\[
\Lambda_e=b-\sum_{j>1}j|c^\circ_j|R_e^{j-1}-E_1,\qquad L_1=b-\sum_{j>1}j|c^\circ_j|-E_1,\qquad U_2=\sum_{j>1}j(j-1)|c^\circ_j|+E_2 .
\]
Then $\re\psi_c'\geq\Lambda_e$ on $\{|w|\leq R_e\}$, $|\psi_c'|\geq L_1$ and $|\psi_c''|\leq U_2$ on $\overline\D$, and $|c_{2m+1}|\geq|c^\circ_{2m+1}|-r/\omega_{2m+1}$. The computation certifies the bounds of Table~\ref{lie:tab:geom}; in particular, for $|w|\leq1$,
\[
\re\Big(1+\frac{w\psi_c''(w)}{\psi_c'(w)}\Big)\geq1-\frac{U_2}{L_1}>0 .
\]
Consequently $\psi^*=\psi_{c^*}$ is injective with nonvanishing derivative on $\{|w|<21/20\}$, and $D=\psi^*(\D)$ is a strictly convex domain bounded by a real-analytic Jordan curve, invariant under $I_2(2m)$, not a disc, with $0\in D$ and $\partial D=\{R(\theta)e^{i\theta}\}$ for a positive, real-analytic, $W$-invariant function $R$.
\end{proposition}

\begin{table}[ht]
\centering
\small
\begin{tabular}{@{}llll@{}}
\toprule
Quantity & $n=8$ & $n=10$ & $n=14$\\
\midrule
$\Lambda_e$ & $>20.037195$ & $>37.965074$ & $>87.134069$\\
$L_1$ & $>20.231942$ & $>38.026654$ & $>87.809995$\\
$U_2$ & $<3.408445$ & $<1.024680$ & $<9.301964$\\
$1-U_2/L_1$ & $>0.83153153$ & $>0.97305365$ & $>0.89406714$\\
$|c_{2m+1}|$ & $|c_7|>0.07748402$ & $|c_9|>0.01383247$ & $|c_{13}|>0.04630866$\\
\bottomrule
\end{tabular}
\caption{Certified geometric bounds of Proposition~\ref{lie:prop:geometry}, valid for every $c$ in the ball of Theorem~\ref{lie:thm:zero}.}
\label{lie:tab:geom}
\end{table}

\begin{proof}
Let $\delta=c-c^\circ$. Since $j\rho^{j-1}/\omega_j\leq\theta_1$ for all $j$, we have $\sum_jj\rho^{j-1}|\delta_j|\leq E_1$; since, for $j\geq2$, $j(j-1)/\omega_j\leq j\rho^{-j}/(\alpha\rho^{m-1})\leq1/(\alpha\rho^{m-1}e\log\rho)$, as $x\mapsto x\rho^{-x}$ has maximum $1/(e\log\rho)$ on $(0,\infty)$, we have $\sum_{j>1}j(j-1)|\delta_j|\leq E_2$; and $|\delta_{2m+1}|\leq r/\omega_{2m+1}$. Both suprema are over all $j$, so the infinite tail of $c$ is covered. For $|w|\leq R_e<\rho$ this gives $\re\psi_c'(w)\geq c_1-\sum_{j>1}j|c_j|R_e^{j-1}\geq\Lambda_e$, and similarly $|\psi_c'|\geq L_1$ and $|\psi_c''|\leq U_2$ on $\overline\D$, so $\re(1+w\psi_c''/\psi_c')\geq1-|\psi_c''|/|\psi_c'|\geq1-U_2/L_1$.

For $c=c^*$, univalence on the convex disc $\{|w|<R_e\}$ follows from $\re\psi^{*\prime}>0$ as in Proposition~\ref{prop:geometry}, and strict convexity of $D$ from the analytic convexity criterion, as in Proposition~\ref{prop:convex}. The symmetries follow from \eqref{lie:eq:sym}, and $0=\psi^*(0)\in D$. A line tangent to $\partial D$ at some point is a supporting line of the convex set $D$, so it cannot pass through the interior point $0$; hence $\frac{d}{d\phi}\arg\psi^*(e^{i\phi})\neq0$, every ray from $0$ meets $\partial D$ exactly once and transversally, and $\partial D$ is the graph of a positive real-analytic radial function $R$, which is $W$-invariant because $D$ is. If $D$ were a disc, its centre would be fixed by the rotation through $\pi/m$, hence equal to $0$, and the Schwarz lemma argument of Proposition~\ref{prop:geometry} would show that $\psi^*$ is linear, contradicting $c^*_{2m+1}\neq0$.
\end{proof}

\begin{proof}[Proof of Theorem~\ref{lie:thm}]
Let $m=3,4,6$ for $\kf=\mathfrak{su}(3),\mathfrak{so}(5),\mathfrak g_2$, respectively, let $x^*=(g^*,c^*)$ be the zero of Theorem~\ref{lie:thm:zero}, $\psi^*=\psi_{c^*}$ and $V^*=Kg^*+h_{c^*}$. By Lemma~\ref{lie:lem:F}, $V^*$ satisfies the hypotheses of Lemma~\ref{lie:lem:pullback} with $\psi=\psi^*$, and by Proposition~\ref{lie:prop:geometry} $\psi^*$ is injective with nonvanishing derivative on a neighbourhood of $\overline\D$. Hence $F=B\,V^*\circ(\psi^*)^{-1}\in C^1(\overline D)$ is a weak solution of $\Delta F+F=0$ in $D$ with $F=\varpi$ and $\nabla F=\nabla\varpi$ on $\partial D$, and $F(\overline\tau)=-F(\tau)$, $F(e^{i\pi/m}\tau)=-F(\tau)$. By Lemma~\ref{lem:interior} and Lemma~\ref{lem:boundary} in the form of Lemma~\ref{lie:lem:carry}, $F$ extends to a real-analytic solution on an open neighbourhood of $\overline D$. Replacing this neighbourhood by the connected component containing $\overline D$ of its intersection with its images under $I_2(2m)$, we obtain an $I_2(2m)$-invariant neighbourhood $N$ on which, by the identity theorem, $F$ keeps its parities; in particular $F$ is anti-invariant under $W$.

By Proposition~\ref{lie:prop:geometry}, $D=\{re^{i\theta}:r<R(\theta)\}$ with $R$ positive, real analytic and $W$-invariant. Lemma~\ref{lie:lem:lift} shows that $\Omega=\Ad(G)D$ is a bounded domain, strictly star-shaped with respect to $0$, whose boundary is a compact real-analytic hypersurface diffeomorphic to $S^{n-1}$, and that $u=(F/\varpi)^\sharp$ is a nonconstant real-analytic function on the neighbourhood $\Ad(G)N$ of $\overline\Omega$ satisfying \eqref{eq:main}. Since $D$ is convex, $\Omega$ is convex by Lemma~\ref{lie:lem:convex}. It is $\Ad(G)$-invariant by construction, and $-\Omega=\Ad(G)(-D)=\Omega$ because $-1\in I_2(2m)$.

If $\Omega$ were a ball, its centre would be fixed by $\Ad(G)$. As $\kf$ is simple, the only $\Ad(G)$-fixed vector is $0$, so $\Omega$ would be a ball centred at $0$ and $D=\Omega\cap\tf$ a disc, contradicting Proposition~\ref{lie:prop:geometry}. Finally, the proof of Corollary~\ref{cor:pompeiu} in Section~\ref{sec:pompeiu} uses only \eqref{eq:main} and the regularity of $u$ and $\partial\Omega$, and applies verbatim: $\widehat{\one_\Omega}$ vanishes on the unit sphere, and $\int_{\sigma(\Omega)}\cos x_1\,dx=0$ for every rigid motion $\sigma$ of $\kf\cong\R^n$.
\end{proof}

\subsection{Which dimensions arise}\label{lie:subsec:complete}

\begin{remark}\label{lie:rem:complete}
The reduction of Proposition~\ref{lie:prop:radial} lowers the dimension to the rank of $\kf$ and never produces a potential term, because $\varpi$ is harmonic; it is planar exactly when $\kf$ has rank two. Since $\kf=\mathfrak z(\kf)\oplus[\kf,\kf]$ with $[\kf,\kf]$ semisimple, a compact Lie algebra of rank two is isomorphic to $\R^2$, to $\R\oplus\mathfrak{su}(2)\cong\mathfrak u(2)$, to $\mathfrak{su}(2)\oplus\mathfrak{su}(2)$, or to one of the simple algebras $\mathfrak{su}(3)$, $\mathfrak{so}(5)$, $\mathfrak g_2$ of types $A_2$, $B_2=C_2$, $G_2$ \cite{BtD85,Kna02}. Leaving aside the abelian case, in which there is nothing to reduce, the dimensions $n=2+2m$ are $4,6,8,10,14$, corresponding to $m\in\{1,2,3,4,6\}$, the orders of rotations allowed by the crystallographic restriction. Thus $\{4,6,8,10,14\}$ is the complete list of dimensions obtainable by exact planar, potential-free reductions of adjoint type, and Theorems~\ref{thm:main} and~\ref{lie:thm} cover all of them. The classes $O(d_1)\times O(d_2)$ of Remark~\ref{rem:why4} give no further dimensions: the admissible ones, $(d_1,d_2)=(1,3)$ and $(3,3)$, are the adjoint classes of $\mathfrak u(2)$ and $\mathfrak{su}(2)\oplus\mathfrak{su}(2)$. For $m=5$ or $m\geq7$, the planar problem \eqref{lie:eq:Pl} with $\varpi=\im\tau^m$ still makes sense and could be treated by the method of this section, but the dihedral group $I_2(m)$ is not crystallographic, and the problem is not the reduction of a problem in a Euclidean space of this type. More generally, for the isotropy representation of a Riemannian symmetric space of rank two whose restricted roots all have the same multiplicity $k$, the radial part of the Laplacian is $f\mapsto J^{-1}\operatorname{div}(J\nabla f)$ with $J=|\varpi|^k$ \cite[Ch.~II]{Hel84}, and conjugation by $J^{1/2}$ leaves the potential $\frac k2\big(\frac k2-1\big)|\nabla\varpi|^2/\varpi^2$, which vanishes only for $k=2$, the adjoint case.
\end{remark}

\part{Axisymmetric counterexamples in dimensions three, five and seven}\label{part:r3}

\section{Dimension three}\label{r3:sec}\label{sec:r3}

In this section we prove Theorem~\ref{thm:r3}. Points of the physical space $\R^3$ are written $X=(X_1,X')\in\R\times\R^2$, so that $X_1$ and $X'$ are the coordinates $x_1$ and $x'$ of Theorem~\ref{thm:r3}; the letter $x$ is used below for points of a reference ball. The convex domain and its dyadic centre are described in Remark~\ref{rem:domain}. The independent earlier example is stated at the end of this section.

\subsection*{Strategy}
For $n=3$ there is no reduction to the planar Laplacian of the kind used for $n=4$ and $n=6$ (Remark~\ref{rem:why4}): for $O(2)$-invariant functions the meridian problem contains the drift term $y^{-1}\partial_y$. We therefore keep the three-dimensional Laplacian and pull back only the meridian plane by a conformal map $\psi$ of the unit disc, which we lift to a diffeomorphism $\Theta$ of the unit ball $B\subset\R^3$ (Section~\ref{r3:sec-pullback}). With $p=\psi'$ and $q=\im\psi/y$, the pulled-back problem is
\[
\operatorname{div}(q\nabla U)+q|p|^2U=0\ \text{ in }B,\qquad U=1\ \text{ and }\ \nabla U=0\ \text{ on }\partial B .
\]
We write $U=1+K_3g$, where $K_3$ is an explicit inverse of the Laplacian of $\R^3$ on axisymmetric functions with vanishing Cauchy data on $\partial B$ (Section~\ref{r3:sec-K}), expand $g$ in zonal solid harmonics times Jacobi polynomials (Section~\ref{r3:sec-spaces}), and obtain a quartic equation $\Fc(g,c)=0$ for $g$ and the Taylor coefficients $c$ of $\psi$ (Section~\ref{r3:sec-quartic}). The unknown coefficient $q$ is not divided out, so that the equation is polynomial. The approximate inverse of the linearisation consists of a $7371\times7371$ matrix, the identity on high field modes, and the exact inverse of a Toeplitz operator that describes the linearisation on high shape modes (Section~\ref{r3:sec-inverse}). The defect bounds (Section~\ref{r3:sec-Z}) cover all columns of the linearisation: the finite parts of $7280$ columns and the complete columns of $457+310$ further inputs are evaluated in interval arithmetic, while the output tails of the former and all remaining columns are covered by explicit envelopes ($91$ radial classes, fourteen near angular classes, one uniform angular class and one far shape class), whose validity is proved there. Compared with $n=4$ and $n=6$, the equation has the variable coefficient $q$ in its principal part, so that the model of the linearisation on high shape modes is a Toeplitz operator whose symbol is built from $q$ and $\psi'$ rather than a bidiagonal operator; the output space measures harmonic components through their boundary traces; and the computation is much larger, taking about $1.8$ hours with five processes (Section~\ref{r3:sec-cap}).

\subsection{Axisymmetric functions and the pull-back}\label{r3:sec-pullback}

Throughout this section $R=21/20$; the notation is independent of that of Sections~\ref{sec:reduction}--\ref{lie:sec}. Let $B\subset\R^3$ be the open unit ball and $B_R=\{|x|<R\}$. We write $x=(x_1,x')\in\R\times\R^2$ for points of the reference space, $X$ for points of the physical space, and
\[
y=|x'|,\qquad w=x_1+iy,\qquad r=|x|=|w|,\qquad t=r^2,\qquad \xi=x_1/r .
\]
A function on a rotation-invariant subset of $\R^3$ is \emph{axisymmetric} if it is invariant under the action of $O(2)$ on $x'$; it is then a function of $(x_1,y)$, and we use both descriptions. Polynomials in $x_1$ and $y^2$ are polynomials on $\R^3$. We write $T_n$, $U_n$, $P_\ell$ and $C^{(2)}_n$ for the Chebyshev polynomials of the first and second kind, the Legendre polynomials and the Gegenbauer polynomials of index two, with $U_{-1}=0$.

Let $\mathcal E$ be the Banach space of real sequences $c=(c_j)_{j\ \mathrm{odd}\geq1}$ with $\nrm c_{\mathcal E}=\sum_jjR^j|c_j|<\infty$. For $c\in\mathcal E$ let $\psi_c(w)=\sum_jc_jw^j$; the series for $\psi_c$ and $\psi_c'$ converge absolutely and uniformly on $|w|\leq R$. Put
\[
p_c=\psi_c'(x_1+iy),\qquad q_c=\frac{\im\psi_c(x_1+iy)}{y}\ \ (y\neq0),\qquad q_c(x_1,0)=\re\psi_c'(x_1).
\]
Since $\im(w^j)=r^j\sin j\theta$ for $w=re^{i\theta}$, $y=r\sin\theta$ and $U_{j-1}(\cos\theta)=\sin j\theta/\sin\theta$,
\begin{equation}\label{r3:eq-qseries}
q_c=\sum_jc_jr^{j-1}U_{j-1}(\xi),\qquad \partial_{x_1}q_c=\sum_{j\geq3}jc_jr^{j-2}U_{j-2}(\xi),\qquad \frac{\partial_yq_c}y=-2\sum_{j\geq3}c_jr^{j-3}C^{(2)}_{j-3}(\xi).
\end{equation}
The second identity follows from $\partial_{x_1}\im(w^j)=\im(jw^{j-1})$. For the third, a computation in polar coordinates gives
\[
\frac1y\,\partial_y\frac{\im w^j}{y}=r^{j-3}\,\frac{(j-1)\sin\theta\cos(j-1)\theta-\sin(j-1)\theta\cos\theta}{\sin^3\theta}=-r^{j-3}U_{j-2}'(\cos\theta),
\]
and $U_{j-2}'=2C^{(2)}_{j-3}$. Moreover, with $p_a=(a+1)c_{a+1}$,
\begin{equation}\label{r3:eq-p2}
|p_c|^2=\sum_{\alpha,\beta\ \mathrm{even}}p_\alpha p_\beta\,\re(w^\alpha\bw^\beta)=\sum_{\alpha,\beta\ \mathrm{even}}p_\alpha p_\beta\,r^{\alpha+\beta}T_{|\alpha-\beta|}(\xi).
\end{equation}
Since $\psi_c$ is odd with real coefficients, $\psi_c(\bw)=\overline{\psi_c(w)}$ and $\psi_c(-\bw)=-\overline{\psi_c(w)}$; hence $q_c$, $y^{-1}\partial_yq_c$ and $|p_c|^2$ are even functions of $x_1$, and $\partial_{x_1}q_c$ is odd.

For $c\in\mathcal E$ define $\Theta_c\colon B_R\to\R^3$ by
\begin{equation}\label{r3:eq-Theta}
\Theta_c(x)=\big(\re\psi_c(x_1+i|x'|),\ q_c(x)\,x'\big).
\end{equation}

\begin{lemma}\label{r3:lem-lift}
Let $c\in\mathcal E$ and suppose that $\re\psi_c'>0$ on a disc $\{|w|<R_0\}$ with $1<R_0<R$. Then:
\begin{enumerate}[label=\textup{(\alph*)}]
\item $q_c$, $\partial_{x_1}q_c$, $y^{-1}\partial_yq_c$, $|p_c|^2$ and $\re\psi_c(x_1+iy)$, regarded as functions of $x$, are real analytic on $B_{R_0}$, and $q_c>0$ there;
\item $\Theta_c$ is a real-analytic diffeomorphism of $B_{R_0}$ onto an open subset of $\R^3$, with $\det D\Theta_c=q_c|p_c|^2$; it commutes with the action of $O(2)$ on $x'$, and $\Theta_c(-x_1,x')=(-X_1,X')$ if $\Theta_c(x_1,x')=(X_1,X')$;
\item for every axisymmetric $C^1$ function $U$ on $B_{R_0}$, $(\det D\Theta_c)(D\Theta_c^{\mathsf T}D\Theta_c)^{-1}\nabla U=q_c\nabla U$.
\end{enumerate}
\end{lemma}

\begin{proof}
(a) The function $F(x_1,y)=\im\psi_c(x_1+iy)$ is real analytic on the disc $\{|w|<R_0\}$ and odd in $y$. By Lemma~\ref{lem:odd-division} there is a real-analytic $G$ with $F=yG(x_1,y^2)$, and $x\mapsto G(x_1,|x'|^2)$ is real analytic on $B_{R_0}$; it equals $q_c$ off the axis, and on the axis $G(x_1,0)=\partial_yF(x_1,0)=\re\psi_c'(x_1)$ by the Cauchy--Riemann equations. Hence $q_c$ is real analytic, and so are $\partial_{x_1}q_c=(\partial_1G)(x_1,y^2)$ and $y^{-1}\partial_yq_c=2(\partial_2G)(x_1,y^2)$. Functions of $(x_1,y)$ that are real analytic and even in $y$, such as $\re\psi_c$ and $|\psi_c'|^2$, are treated by applying Lemma~\ref{lem:odd-division} to their product with $y$. For $y\neq0$, the Cauchy--Riemann equations give $q_c=y^{-1}\int_0^y\re\psi_c'(x_1+i\varsigma)\,d\varsigma>0$, and $q_c=\re\psi_c'>0$ on the axis.

(b) Both components of $\Theta_c$ are real analytic by (a). For $w_1\neq w_2$ in the convex disc $\{|w|<R_0\}$, $(\psi_c(w_2)-\psi_c(w_1))/(w_2-w_1)=\int_0^1\psi_c'(w_1+\varsigma(w_2-w_1))\,d\varsigma$ has positive real part, so $\psi_c$ is injective there. Since $q_c>0$, the second component of $\Theta_c(x)$ has modulus $\im\psi_c(x_1+i|x'|)$. Hence $\Theta_c(x)=\Theta_c(z)$ implies $\psi_c(x_1+i|x'|)=\psi_c(z_1+i|z'|)$, so $x_1=z_1$ and $|x'|=|z'|$, and then $q_c(x)x'=q_c(x)z'$ gives $x'=z'$. At a point with $x'\neq0$ let $e_\rho=x'/|x'|$ and let $e_\phi$ be the unit vector in the $x'$-plane orthogonal to $e_\rho$. On the half-plane spanned by $e_1$ and $e_\rho$, $\Theta_c$ acts as $\psi_c$ in the coordinates $(x_1,y)$, and $D\Theta_c$ maps this plane to itself as $|p_c|$ times a rotation; it maps $e_\phi$ to $q_ce_\phi$, since $\Theta_c$ maps the circle of radius $y$ about the axis onto the circle of radius $q_cy$. Hence $\det D\Theta_c=q_c|p_c|^2>0$ off the axis, and by continuity on $B_{R_0}$. By the inverse function theorem $\Theta_c$ is a real-analytic diffeomorphism onto its open image. The symmetries follow from those of $\psi_c$.

(c) In the frame $(e_1,e_\rho,e_\phi)$ of (b), $D\Theta_c^{\mathsf T}D\Theta_c=\mathrm{diag}(|p_c|^2,|p_c|^2,q_c^2)$, and $\nabla U$ lies in the span of $e_1$ and $e_\rho$. Off the axis the identity follows, and on the axis by continuity.
\end{proof}

\begin{lemma}\label{r3:lem-pullback}
Let $c$ be as in Lemma~\ref{r3:lem-lift}, $\Omega=\Theta_c(B)$, and let $U\in C^1(\overline B)$ be axisymmetric and satisfy
\begin{equation}\label{r3:eq-weakU}
\int_B\big(q_c\nabla U\cdot\nabla\varphi-q_c|p_c|^2U\varphi\big)\,dx=0\qquad(\varphi\in C^\infty_c(B)).
\end{equation}
Then $u=U\circ\Theta_c^{-1}\in C^1(\overline\Omega)$ satisfies $\int_\Omega(\nabla u\cdot\nabla\varphi-u\varphi)\,dX=0$ for all $\varphi\in C^\infty_c(\Omega)$. If $U=1$ and $\nabla U=0$ on $\partial B$, then $u=1$ and $\nabla u=0$ on $\partial\Omega$.
\end{lemma}

\begin{proof}
Let $\varphi\in C_c^\infty(\Omega)$ and $\tilde\varphi=\varphi\circ\Theta_c\in C_c^\infty(B)$. Since $(\nabla u)\circ\Theta_c=D\Theta_c^{-\mathsf T}\nabla U$, the change of variables $X=\Theta_c(x)$ and Lemma~\ref{r3:lem-lift}(b),(c) give
\[
\int_\Omega(\nabla u\cdot\nabla\varphi-u\varphi)\,dX=\int_B\Big((\det D\Theta_c)(D\Theta_c^{\mathsf T}D\Theta_c)^{-1}\nabla U\cdot\nabla\tilde\varphi-q_c|p_c|^2U\tilde\varphi\Big)dx,
\]
which vanishes by Lemma~\ref{r3:lem-lift}(c) and \eqref{r3:eq-weakU}. The boundary values follow from $(\nabla u)\circ\Theta_c=D\Theta_c^{-\mathsf T}\nabla U$.
\end{proof}

\subsection{Zonal harmonics, Jacobi polynomials and coefficient spaces}\label{r3:sec-spaces}

For $\ell\geq0$ let $H_\ell(x)=r^\ell P_\ell(\xi)$, the axisymmetric harmonic polynomial on $\R^3$, homogeneous of degree $\ell$, with $H_\ell(1,0,0)=1$; axisymmetric harmonic homogeneous polynomials of degree $\ell$ are multiples of $H_\ell$ \cite[Ch.~IV]{SW71}. We have $|H_\ell|\leq1$ on $\overline B$, $H_\ell=P_\ell(\xi)$ on $\partial B$, and, with $H_{-1}=0$,
\begin{equation}\label{r3:eq-Hrec}
x_1H_\ell=\frac{(\ell+1)H_{\ell+1}+\ell\,tH_{\ell-1}}{2\ell+1},\qquad \partial_{x_1}H_\ell=\ell H_{\ell-1},\qquad (x\cdot\nabla)H_\ell=\ell H_\ell .
\end{equation}
The first identity is Bonnet's recurrence for $P_\ell$ multiplied by $r^{\ell+1}$; for the second, $\partial_{x_1}H_\ell$ is an axisymmetric harmonic homogeneous polynomial of degree $\ell-1$, hence a multiple of $H_{\ell-1}$, and the multiple is $\ell$ because $H_\ell(x_1,0,0)=x_1^\ell$.

For real $\beta>-1$ and integers $s\geq0$ let $J^\beta_s(t)=P^{(0,\beta)}_s(2t-1)$, so that
\begin{equation}\label{r3:eq-Jcoef}
J^\beta_s(t)=\sum_{k=0}^s(-1)^{s-k}\frac{\Gamma(\beta+s+k+1)}{k!\,(s-k)!\,\Gamma(\beta+k+1)}\,t^k=\frac{t^{-\beta}}{s!}\frac{d^s}{dt^s}\big\{t^{\beta+s}(t-1)^s\big\},\qquad J^\beta_s(1)=1,
\end{equation}
and $J^\beta_{-1}=0$. We write $[t^k]f$ for the coefficient of $t^k$ in a polynomial $f$, and $\nrm f_{J^\beta,1}=\sum_s|f_s|$ if $f=\sum_sf_sJ^\beta_s$. Define the axisymmetric polynomials
\[
\Psi_{\ell,s}=H_\ell\,J^{\ell+1/2}_s(t)\qquad(\ell,s\geq0),
\]
of degree $\ell+2s$; thus $\Psi_{\ell,0}=H_\ell$.

\begin{lemma}\label{r3:lem-jacobi}
Let $\beta>0$ and $s,s'\geq0$.
\begin{enumerate}[label=\textup{(\alph*)}]
\item For every integer $m\geq0$,
\[
\int_0^1t^{\beta+m}J^\beta_s(t)\,dt=\frac{m(m-1)\cdots(m-s+1)}{(\beta+m+1)(\beta+m+2)\cdots(\beta+m+s+1)} ;
\]
in particular $J^\beta_s$ is orthogonal in $L^2([0,1],t^\beta dt)$ to all polynomials of degree less than $s$, and $\int_0^1t^\beta J^\beta_sJ^\beta_{s'}\,dt=\delta_{ss'}/(\beta+2s+1)$.
\item $J^\beta_s=\dfrac{\beta+s+1}{\beta+2s+1}J^{\beta+1}_s+\dfrac{s}{\beta+2s+1}J^{\beta+1}_{s-1}$, which holds also for $-1<\beta\leq0$, and $tJ^\beta_s=\dfrac{\beta+s}{\beta+2s+1}J^{\beta-1}_s+\dfrac{s+1}{\beta+2s+1}J^{\beta-1}_{s+1}$. The coefficients are nonnegative and sum to one, and in the first identity the coefficient of the term with the lower index is at most $1/2$.
\item The coefficient of $J^\beta_s$ in the expansion of $(1-t)^m$ is $(-1)^s$ times a positive number for $0\leq s\leq m$, and vanishes for $s>m$. Moreover $\nrm{(1-t)^m}_{J^\beta,1}=m!/(\beta/2+1)_m$, which is decreasing in $\beta$.
\item For $M\geq0$, the coefficients of $t^M$ in the basis $(J^\beta_s)_s$ are nonnegative and sum to one, the coefficient of $J^\beta_0$ is $(\beta+1)/(\beta+M+1)$, and $\nrm{t^M-1}_{J^\beta,1}=2M/(\beta+M+1)$.
\end{enumerate}
\end{lemma}

\begin{proof}
(a) Insert the Rodrigues formula \eqref{r3:eq-Jcoef} and integrate by parts $s$ times. The boundary terms vanish: at $t=1$ the derivatives of order less than $s$ of $t^{\beta+s}(t-1)^s$ vanish, and at $t=0$ they are $O(t^{\beta+1})$. This gives $\frac{m(m-1)\cdots(m-s+1)}{s!}\int_0^1t^{\beta+m}(1-t)^sdt$, which is the stated value. For $m<s$ the product vanishes. The leading coefficient of $J^\beta_s$ is $\Gamma(\beta+2s+1)/(s!\,\Gamma(\beta+s+1))$, and multiplying it by the value for $m=s$ gives $1/(\beta+2s+1)$.

(b) Divide the coefficient of $t^k$ of each side by $[t^k]J^{\beta+1}_s$, respectively by $[t^k]J^{\beta-1}_s$, using \eqref{r3:eq-Jcoef}. For $0\leq k\leq s$ the first identity reduces to $(\beta+k+1)(\beta+2s+1)=(\beta+s+1)(\beta+s+k+1)-s(s-k)$, and for $1\leq k\leq s$ the second to $-k(\beta+2s+1)=(\beta+s)(s-k+1)-(s+1)(\beta+s+k)$; both are verified by expansion. In the second identity the cases $k=0$ and $k=s+1$ are checked directly. The computation for the first identity uses only $\beta>-1$. The remaining assertions are clear, and $s\leq(\beta+2s+1)/2$ for $\beta\geq-1$.

(c) By the orthogonality in (a), the coefficient is $(\beta+2s+1)\int_0^1t^\beta(1-t)^mJ^\beta_s\,dt$. Integrating by parts $s$ times in the Rodrigues formula gives $\frac{(-1)^s}{s!}\,m(m-1)\cdots(m-s+1)\int_0^1t^{\beta+s}(1-t)^m\,dt$ times $\beta+2s+1$, which proves the sign pattern. Let $\Lambda(f)=\frac\beta2\int_0^1t^{\beta/2-1}f(t)\,dt$. The same integration by parts, now against $t^{-\beta/2-1}$, gives $\Lambda(J^\beta_s)=(-1)^s\frac\beta2(\frac\beta2+1)_s\mathrm B(\frac\beta2,s+1)/s!=(-1)^s$. Hence $\nrm{(1-t)^m}_{J^\beta,1}=\Lambda((1-t)^m)=\frac\beta2\mathrm B(\frac\beta2,m+1)=m!/(\beta/2+1)_m$.

(d) By (b), multiplication by $t$ is the lowering step followed by the raising step (with $\beta-1$ in place of $\beta$); both are given by nonnegative matrices with column sums one. Hence $t^M=t^MJ^\beta_0$ has nonnegative coefficients summing to $t^M|_{t=1}=1$. The coefficient of $J^\beta_0=1$ is $(\beta+1)\int_0^1t^{\beta+M}dt$. Therefore $\nrm{t^M-1}_{J^\beta,1}$ equals $\big(1-\frac{\beta+1}{\beta+M+1}\big)$ plus the sum of the other coefficients, which is the same number.
\end{proof}

\subsubsection*{Chebyshev and Legendre coordinates}
Let $a_j=\binom{2j}j4^{-j}$. Then $a_0=1$, $a_{j+1}/a_j=(2j+1)/(2j+2)$, so $(a_j)$ is decreasing, and $a_k\sqrt{k+1/4}$ is nondecreasing in $k$, because $(2k+1)^2(4k+5)-(2k+2)^2(4k+1)=1$. Consequently
\begin{equation}\label{r3:eq-aratio}
\frac{a_\ell}{a_m}\leq\Big(\frac{m+1/4}{\ell+1/4}\Big)^{1/2}\quad(\ell\leq m),\qquad a_{\ell-j}\leq\sqrt2\,a_\ell\quad(0\leq j\leq\ell/2).
\end{equation}
Factoring the generating function, $(1-2\xi z+z^2)^{-1/2}=(1-e^{i\theta}z)^{-1/2}(1-e^{-i\theta}z)^{-1/2}$ for $\xi=\cos\theta$, gives $P_\ell(\cos\theta)=\sum_{j=0}^\ell a_ja_{\ell-j}e^{i(\ell-2j)\theta}$, that is,
\[
P_\ell=\sum_nD_{n\ell}T_n,\qquad D_{\ell-2j,\ell}=(2-\delta_{\ell,2j})\,a_ja_{\ell-j}\geq0 .
\]
Conversely write $T_n=\sum_\ell C_{\ell n}P_\ell$ and $U_n=\sum_\ell\upsilon_{\ell n}P_\ell$. Let $\langle\cdot\rangle_R$ be the linear functional on polynomials in $\xi$ with $\langle T_n\rangle_R=R^n$, and put
\[
h_\ell=\langle P_\ell\rangle_R=\sum_nD_{n\ell}R^n,\qquad \kappa_R=\frac4{\sqrt{1-R^{-2}}}-1,\qquad B_0=\frac{2\sqrt2}{\sqrt{1-R^{-2}}} .
\]
Numerically $\kappa_R\approx12.118596$ and $B_0\approx9.276248$.

\begin{lemma}\label{r3:lem-connection}
\begin{enumerate}[label=\textup{(\alph*)}]
\item $C_{00}=1$, $C_{nn}=1/(2a_n)$ for $n\geq1$, and for $n=\ell+2j$ with $j\geq1$,
\[
C_{\ell n}=-\frac{n(2\ell+1)\,a_j}{2(\ell+j)(2\ell+2j+1)\,a_{\ell+j}\,(2j-1)}<0 .
\]
\item For $n=\ell+2j$, $\upsilon_{\ell n}=(2\ell+1)a_j/((2\ell+2j+1)a_{\ell+j})>0$. The Legendre coefficients of $C^{(2)}_n$ are positive.
\end{enumerate}
\end{lemma}

\begin{proof}
We use the connection formula for ultraspherical polynomials \cite[18.18.16]{DLMF},
\[
C^{(\lambda)}_n=\sum_{j=0}^{\lfloor n/2\rfloor}\frac{(\lambda-\mu)_j(\lambda)_{n-j}}{(\mu+1)_{n-j}\,j!}\,\frac{\mu+n-2j}{\mu}\,C^{(\mu)}_{n-2j},
\]
with $\mu=\frac12$, so that $C^{(\mu)}_\ell=P_\ell$ and $(\mu+n-2j)/\mu=2\ell+1$ for $\ell=n-2j$. We use $(\frac12)_j=j!\,a_j$, $(-\frac12)_j=-j!\,a_j/(2j-1)$ and $(\frac32)_m=(2m+1)\,m!\,a_m$. For $\lambda=1$, $C^{(1)}_n=U_n$, and the coefficient is $(\frac12)_j(n-j)!(2\ell+1)/((\frac32)_{n-j}j!)$, which is the stated $\upsilon_{\ell n}$. For $\lambda\to0$, $C^{(\lambda)}_n/\lambda\to(2/n)T_n$ for $n\geq1$ \cite[\S18.7(iii)]{DLMF} and $(\lambda)_{n-j}/\lambda\to(n-j-1)!$, which gives (a). For $\lambda=2$ all coefficients are positive, since $(\frac32)_j>0$.
\end{proof}

\begin{lemma}\label{r3:lem-weights}
For all $\ell,n\geq0$:
\begin{enumerate}[label=\textup{(\alph*)}]
\item $1\leq h_\ell\leq B_0a_\ell R^\ell$;
\item $h_{\ell+1}\leq Rh_\ell$, and $h_{\ell-1}\leq R^{-1}\frac{\ell}{\ell-1}h_\ell$ for $\ell\geq2$;
\item $\sum_\ell|C_{\ell n}|h_\ell\leq\kappa_RR^n$;
\item $C_{nn}h_n\leq\frac{B_0}2R^n$; for $j\geq1$ and $\ell=n-2j\geq0$, $|C_{\ell n}|h_\ell\leq B_0\frac{a_jR^{-2j}}{2j-1}R^n$ and $\upsilon_{\ell n}h_\ell\leq B_0a_jR^{-2j}R^n$;
\item $D_{n\ell}R^n\leq\sqrt2\,R^{n-\ell}h_\ell$.
\end{enumerate}
\end{lemma}

\begin{proof}
(a) $h_\ell\geq\sum_nD_{n\ell}=P_\ell(1)=1$. By \eqref{r3:eq-aratio},
\[
h_\ell\leq2\sum_{j\leq\ell/2}a_ja_{\ell-j}R^{\ell-2j}\leq2\sqrt2\,a_\ell R^\ell\sum_ja_jR^{-2j},
\]
and $\sum_ja_jz^j=(1-z)^{-1/2}$.

(b) Write $Rh_\ell-h_{\ell+1}=\sum_e\delta_eR^e$, where $e\equiv\ell+1\pmod2$ runs over $0\leq e\leq\ell+1$. Comparing coefficients, $\delta_e=2a_i(a_{\ell-i}-a_{\ell+1-i})\geq0$ for $e\geq2$, where $i=(\ell+1-e)/2$. The lowest exponent $e_0\in\{0,1\}$ has a negative coefficient: $\delta_1=a_{\ell/2}(a_{\ell/2}-2a_{\ell/2+1})<0$ if $\ell$ is even, and $\delta_0=-a_{(\ell+1)/2}^2$ if $\ell$ is odd. For $R=1$ all $h_k$ equal one, so $\sum_e\delta_e=0$. Hence $Rh_\ell-h_{\ell+1}=\sum_{e>e_0}\delta_e(R^e-R^{e_0})\geq0$ for $R\geq1$. For the second inequality we compare the coefficients of $R^e$ in $Rh_{\ell-1}$ and $h_\ell$. For $e=\ell-2i\geq2$ they are $2a_ia_{\ell-1-i}$ and $2a_ia_{\ell-i}$, whose ratio is $2m/(2m-1)$ with $m=\ell-i\geq(\ell+2)/2$, hence at most $(\ell+2)/(\ell+1)\leq\ell/(\ell-1)$. For $e=1$ ($\ell$ odd, $i=(\ell-1)/2$) they are $a_i^2$ and $2a_ia_{i+1}$, with ratio $(\ell+1)/(2\ell)\leq\ell/(\ell-1)$. For $\ell$ even, $h_\ell$ has an additional nonnegative coefficient at $e=0$.

(c) For $n=0$ the sum is $1\leq\kappa_R$. For $n\geq1$, since $h_\ell=\langle P_\ell\rangle_R$, we have $\sum_\ell C_{\ell n}h_\ell=\langle T_n\rangle_R=R^n$. By Lemma~\ref{r3:lem-connection}(a) all $C_{\ell n}$ with $\ell<n$ are negative, so $\sum_\ell|C_{\ell n}|h_\ell=2C_{nn}h_n-R^n\leq(B_0-1)R^n\leq\kappa_RR^n$ by (a).

(d) The first bound follows from (a). For $j\geq1$ put $m=\ell+j$. By (a) and Lemma~\ref{r3:lem-connection}, the second bound reduces to $n(2\ell+1)a_\ell\leq2m(2m+1)a_m$. By \eqref{r3:eq-aratio} it suffices that $4m^2(2m+1)^2(\ell+\frac14)-n^2(2\ell+1)^2(m+\frac14)\geq0$; with $m=\ell+j$, $n=\ell+2j$ this is a polynomial in $(\ell,j)$ whose coefficients are all nonnegative. The third bound reduces to $(2\ell+1)a_\ell\leq(2m+1)a_m$, which follows from \eqref{r3:eq-aratio} and the monotonicity of $(2x+1)^2/(x+\frac14)$ on $[0,\infty)$.

(e) For $\ell\geq1$, $h_\ell\geq D_{\ell\ell}R^\ell=2a_\ell R^\ell$, and $D_{n\ell}\leq2a_{\ell-j}\leq2\sqrt2\,a_\ell$ by \eqref{r3:eq-aratio}. For $\ell=0$ the claim is trivial.
\end{proof}

\subsubsection*{Coefficient spaces}
Let $\mathcal L$ be the Banach space of real families $f=(f_{\ell,s})_{\ell,s\geq0}$ with $\nrm f_{\mathcal L}=\sum_{\ell,s}h_\ell|f_{\ell,s}|<\infty$, identified with the formal series $\sum f_{\ell,s}\Psi_{\ell,s}$, and let $\mathcal L^{\mathrm{ev}}$ be its subspace with $f_{\ell,s}=0$ for odd $\ell$. The field space is
\[
\Gc=\{g\in\mathcal L^{\mathrm{ev}}:\ g_{\ell,0}=0\text{ for all }\ell\}.
\]
For even $n\geq0$ let $\mathcal T_n=\sum_\ell C_{\ell n}H_\ell$; this is the harmonic polynomial with boundary values $T_n(\xi)$ on $\partial B$, and $|\mathcal T_n|\leq1$ on $\overline B$ by the maximum principle. The output space $\mathcal Y$ is the Banach space of families $f=\big((f_n)_{n\ \mathrm{even}},(f_{\ell,s})_{\ell\ \mathrm{even},\,s\geq1}\big)$, identified with $\sum_nf_n\mathcal T_n+\sum_{s\geq1}f_{\ell,s}\Psi_{\ell,s}$, with
\[
\nrm f_{\mathcal Y}=\sum_nR^n|f_n|+\sum_{\ell,\,s\geq1}h_\ell|f_{\ell,s}| .
\]
Thus the harmonic layer of $\mathcal Y$ is measured through the Chebyshev coefficients of its boundary trace. Since $H_\ell=\sum_nD_{n\ell}\mathcal T_n$ with $D_{n\ell}\geq0$, the inclusion $\mathcal L^{\mathrm{ev}}\to\mathcal Y$ satisfies $\nrm f_{\mathcal Y}\leq\nrm f_{\mathcal L}$. The unknowns lie in $\mathcal X=\Gc\times\mathcal E$, with norm $\nrm{(g,c)}_{\mathcal X}=\nrm g_{\mathcal L}+\sigma\nrm c_{\mathcal E}$ for a constant $\sigma>0$ fixed in Section~\ref{r3:sec-quartic}. All three spaces are weighted $\ell^1$ spaces, and Lemma~\ref{lem:columns} applies to operators between them.

\begin{lemma}\label{r3:lem-embed}
\begin{enumerate}[label=\textup{(\alph*)}]
\item The functions $\Psi_{\ell,s}$ are pairwise orthogonal in $L^2(B)$, $\nrm{\Psi_{\ell,s}}^2_{L^2(B)}=4\pi/((2\ell+1)(2\ell+4s+3))$, and $\Psi_{\ell,s}$ is orthogonal to all polynomials on $\R^3$ of degree less than $\ell+2s$.
\item The series defining the elements of $\mathcal L$ and $\mathcal Y$ converge in $L^2(B)$, and the resulting maps $\iota\colon\mathcal L\to L^2(B)$ and $\iota\colon\mathcal Y\to L^2(B)$ are continuous, injective and compatible with the inclusion $\mathcal L^{\mathrm{ev}}\subset\mathcal Y$. The harmonic part $\sum_nf_n\mathcal T_n$ of $f\in\mathcal Y$ converges uniformly on $\overline B$.
\end{enumerate}
\end{lemma}

\begin{proof}
(a) In spherical coordinates,
\[
\int_B\Psi_{\ell,s}\Psi_{\ell',s'}\,dx=2\pi\int_{-1}^1P_\ell P_{\ell'}\,d\xi\cdot\frac12\int_0^1t^{(\ell+\ell'+1)/2}J^{\ell+1/2}_sJ^{\ell'+1/2}_{s'}\,dt,
\]
and Lemma~\ref{r3:lem-jacobi}(a) gives orthogonality and the norm. A polynomial of degree less than $\ell+2s$ is a sum of terms $|x|^{2m}Y_k$ with $Y_k$ harmonic and homogeneous of degree $k$, $k+2m<\ell+2s$ \cite[Ch.~IV]{SW71}. The angular integral of $\Psi_{\ell,s}|x|^{2m}Y_k$ vanishes unless $k=\ell$, and then $m<s$ and the radial integral $\frac12\int_0^1t^{\ell+1/2+m}J^{\ell+1/2}_s\,dt$ vanishes by Lemma~\ref{r3:lem-jacobi}(a).

(b) Continuity follows from $\nrm{\Psi_{\ell,s}}_{L^2}\leq(4\pi/3)^{1/2}\leq(4\pi/3)^{1/2}h_\ell$ and $|\mathcal T_n|\leq1\leq R^n$. On $\mathcal L$, injectivity follows from (a). Let $f\in\mathcal Y$ with $\iota f=0$. The $\mathcal T_n$ are harmonic polynomials, hence orthogonal to all $\Psi_{\ell,s}$ with $s\geq1$ by (a); pairing with $\Psi_{\ell,s}$ gives $f_{\ell,s}=0$ for $s\geq1$. The harmonic part $h=\sum_nf_n\mathcal T_n$ converges uniformly on $\overline B$, so it is harmonic in $B$, continuous on $\overline B$, and vanishes in $B$; hence its boundary values $\sum_nf_nT_n(\cos\theta)=\sum_nf_n\cos n\theta$ vanish, and all $f_n=0$.
\end{proof}

\subsection{Multiplication}\label{r3:sec-mult}

\begin{lemma}[Adams]\label{r3:lem-gaunt}
For $a,\ell\geq0$, $P_aP_\ell=\sum_kG_{a\ell k}P_k$, where $G_{a\ell k}\geq0$ vanishes unless $|a-\ell|\leq k\leq a+\ell$ and $k\equiv a+\ell\pmod2$, and $\sum_kG_{a\ell k}=1$. Consequently $H_aH_\ell=\sum_kG_{a\ell k}H_kt^{(a+\ell-k)/2}$ and
\[
\sum_kG_{a\ell k}h_k\leq h_ah_\ell .
\]
\end{lemma}

\begin{proof}
Adams' formula \cite{Ada78} gives, with $v=(a+\ell+k)/2$ in the support,
\[
G_{a\ell k}=(2k+1)\Big[\frac{v!}{(v-a)!(v-\ell)!(v-k)!}\Big]^2\frac{(2v-2a)!\,(2v-2\ell)!\,(2v-2k)!}{(2v+1)!}\geq0 ;
\]
the sum is one by evaluation at $\xi=1$, and the identity for $H_aH_\ell$ follows after multiplication by $r^{a+\ell}$. Finally, $\sum_kG_{a\ell k}h_k=\langle P_aP_\ell\rangle_R$, and $P_aP_\ell=\sum_{m,n}D_{ma}D_{n\ell}\frac12(T_{m+n}+T_{|m-n|})$ with $D\geq0$, so $\langle P_aP_\ell\rangle_R\leq\sum_{m,n}D_{ma}D_{n\ell}R^{m+n}=h_ah_\ell$.
\end{proof}

Every axisymmetric polynomial $m$ has unique expansions
\[
m=\sum_{a,k\geq0}m_{a,k}H_at^k=\sum_{n,k\geq0}\tilde m_{n,k}r^{n+2k}T_n(\xi),
\]
the first by the Fischer decomposition, the second because $r^nT_n(\xi)=\re(x_1+iy)^n$. The homogeneous part of degree $d$ of $m$ consists of the terms with $a+2k=d$, respectively $n+2k=d$; we call $H_at^k$ a \emph{solid monomial} of degree $a+2k$. We put
\[
\tau(m)=\sum_{a,k}h_a|m_{a,k}|,\qquad \nrm m_{\mathrm{ch}}=\sum_{n,k}R^n|\tilde m_{n,k}|,
\]
and extend both to series with finite norm; such series converge uniformly on $\overline B$, since $|H_at^k|\leq1$ and $|r^{n+2k}T_n(\xi)|\leq1$ there.

\begin{lemma}\label{r3:lem-mult}
\begin{enumerate}[label=\textup{(\alph*)}]
\item Let $H_at^k$ be a solid monomial and $\ell,s\geq0$. Then
\begin{equation}\label{r3:eq-monomial}
H_at^k\Psi_{\ell,s}=\sum_vG_{a\ell v}\sum_{\varsigma\geq0}\pi^{(v)}_\varsigma\Psi_{v,\varsigma},
\end{equation}
where, for each $v$, $(\pi^{(v)}_\varsigma)_\varsigma$ is a probability vector. Consequently $\nrm{mf}_{\mathcal L}\leq\tau(m)\nrm f_{\mathcal L}$ for every $f\in\mathcal L$ and every $m$ with $\tau(m)<\infty$.
\item $\tau(m)\leq\kappa_R\nrm m_{\mathrm{ch}}$. If $m=r^{n+2k}U_n(\xi)$ or $m=r^{n+2k}C^{(2)}_n(\xi)$, then $\tau(m)=\nrm m_{\mathrm{ch}}$.
\item $\nrm{m_1m_2}_{\mathrm{ch}}\leq\nrm{m_1}_{\mathrm{ch}}\nrm{m_2}_{\mathrm{ch}}$ and $\nrm{\re(w^\alpha\bw^\beta)}_{\mathrm{ch}}\leq R^{\alpha+\beta}$.
\item For $m$ with $\tau(m)<\infty$ and $f\in\mathcal L$, $\iota(mf)=m\,\iota(f)$.
\end{enumerate}
\end{lemma}

\begin{proof}
(a) By Lemma~\ref{r3:lem-gaunt}, $H_at^k\Psi_{\ell,s}=\sum_vG_{a\ell v}H_v\,t^eJ^{\ell+1/2}_s$ with $e=(a+\ell-v)/2+k$. If $v\geq\ell$, we apply the first identity of Lemma~\ref{r3:lem-jacobi}(b) $v-\ell$ times, raising the parameter from $\ell+\frac12$ to $v+\frac12$, and then write the factor $t^e$ as $e$ multiplications by $t$, each a lowering step followed by a raising step. If $v<\ell$, we apply the second identity $\ell-v$ times, which uses $\ell-v$ factors $t$; this is possible because $e\geq\ell-v$, as $\ell-v\leq a$. The remaining $e-(\ell-v)$ factors $t$ are treated as before. Every step is given by a nonnegative matrix with column sums one, which proves \eqref{r3:eq-monomial}. Since $\nrm{\Psi_{v,\varsigma}}_{\mathcal L}=h_v$, Lemma~\ref{r3:lem-gaunt} gives $\nrm{H_at^k\Psi_{\ell,s}}_{\mathcal L}\leq h_ah_\ell$, and the bound for $mf$ follows by absolute summation.

(b) Since $r^{n+2k}T_n(\xi)=\sum_\ell C_{\ell n}H_\ell t^{(n-\ell)/2+k}$, Lemma~\ref{r3:lem-weights}(c) gives $\tau(r^{n+2k}T_n)\leq\kappa_RR^n$. For $U_n$ the Legendre coefficients are nonnegative (Lemma~\ref{r3:lem-connection}(b)), so $\tau(r^{n+2k}U_n)=\sum_\ell\upsilon_{\ell n}h_\ell=\langle U_n\rangle_R$, which is the Chebyshev norm, because $U_n(\cos\theta)=\sum_{|\nu|\leq n,\,\nu\equiv n\ (\mathrm{mod}\ 2)}e^{i\nu\theta}$ has nonnegative Chebyshev coefficients. For $C^{(2)}_n$ the same argument applies, since $C^{(2)}_n(\cos\theta)=\sum_{p+p'=n}(p+1)(p'+1)e^{i(p-p')\theta}$.

(c) $T_mT_n=\frac12(T_{m+n}+T_{|m-n|})$ and $R^{|m-n|}\leq R^{m+n}$; and $\re(w^\alpha\bw^\beta)=r^{\alpha+\beta}T_{|\alpha-\beta|}(\xi)$.

(d) Both sides are continuous in $f\in\mathcal L$ with values in $L^2(B)$, by (a) and Lemma~\ref{r3:lem-embed}, and $m$ is bounded; they agree when $f$ is a finite combination of the $\Psi_{\ell,s}$ and $m$ is a polynomial, and the general case follows by uniform convergence of $m$ and density.
\end{proof}

\begin{lemma}\label{r3:lem-radial}
Let $H_at^k$ be a solid monomial of degree $d=a+2k$, let $\ell,s\geq0$, and let $\pi^{(v)}$ be as in \eqref{r3:eq-monomial}. Then:
\begin{enumerate}[label=\textup{(\alph*)}]
\item $\pi^{(v)}_\varsigma=0$ unless $s-d\leq\varsigma\leq s+d$, and $G_{a\ell v}=0$ unless $|\ell-a|\leq v\leq\ell+a$;
\item $\pi^{(v)}_0\leq\vartheta_d(\ell,s)$, where $\vartheta_d(\ell,s)=\big(d/(\ell+d+\frac52)\big)^s$ if $d\geq s$ (with $0^0=1$) and $\vartheta_d(\ell,s)=0$ if $d<s$;
\item for every integer $j\geq1$, $\sum_{\varsigma\leq s-j}\pi^{(v)}_\varsigma\leq2^{-d}\sum_{i\geq j}\binom di$.
\end{enumerate}
\end{lemma}

\begin{proof}
In the proof of Lemma~\ref{r3:lem-mult}(a), only the raising steps can decrease the radial index, each by at most one, and only the lowering steps can increase it, each by at most one. The number of raising steps is $e+v-\ell=(a-\ell+v)/2+k$ and the number of lowering steps is $e=(a+\ell-v)/2+k$; since $|\ell-v|\leq a$, both are at most $a+k\leq d$. This proves (a), the statement about $v$ being part of Lemma~\ref{r3:lem-gaunt}.

(b) Since $J^{v+1/2}_0=1$, Lemma~\ref{r3:lem-jacobi}(a) gives $\pi^{(v)}_0=(v+\frac32)\int_0^1t^{v+1/2+e}J^{\ell+1/2}_s\,dt$, which is the value in Lemma~\ref{r3:lem-jacobi}(a) with $\beta=\ell+\frac12$ and $m=e+v-\ell=(a-\ell+v)/2+k$, multiplied by $v+\frac32$. Here $0\leq m\leq a+k\leq d$. If $m<s$, in particular if $d<s$, it vanishes. Otherwise we pair the factor $v+\frac32$ with $\ell+m+\frac32$, which is not smaller because $\ell+m=(a+\ell+v)/2+k\geq v$, and each factor $m-i$ ($0\leq i<s$) with $\ell+m+\frac52+i$; since $x\mapsto x/(\ell+x+\frac52)$ is increasing, each quotient $(m-i)/(\ell+m+\frac52+i)$ is at most $d/(\ell+d+\frac52)$.

(c) The expansion coefficients define a Markov chain on the radial index: a raising step at parameter $\beta$ moves $\varsigma$ to $\varsigma-1$ with probability $\varsigma/(\beta+2\varsigma+1)\leq\frac12$ and otherwise leaves it unchanged, and lowering steps never decrease $\varsigma$. Let $N_i$ be the number of decreases among the first $i$ raising steps. Then $\Pr(N_{i+1}\geq j)\leq\Pr(N_i\geq j)+\frac12\Pr(N_i=j-1)=\frac12\big(\Pr(N_i\geq j)+\Pr(N_i\geq j-1)\big)$, and by induction $\Pr(N_i\geq j)\leq\Pr(\mathrm{Bin}(i,\frac12)\geq j)$. Since there are at most $d$ raising steps and the final index is at least $s-N$, the claim follows. (The probabilistic language only organises sums of products of the nonnegative coefficients.)
\end{proof}

\subsection{\texorpdfstring{The operator $K_3$}{The operator K3}}\label{r3:sec-K}

For $\ell\geq0$ and $s\geq1$ put $\beta=\ell+\frac12$, $d=\beta+2s=\ell+2s+\frac12$, and define
\begin{equation}\label{r3:eq-K}
K_3\Psi_{\ell,s}=\frac{\Psi_{\ell,s-1}}{4d(d+1)}-\frac{\Psi_{\ell,s}}{2d(d+2)}+\frac{\Psi_{\ell,s+1}}{4(d+1)(d+2)}=H_\ell\,k_{\ell,s}(t),
\end{equation}
and let $E=x\cdot\nabla=r\partial_r$ be the Euler operator.

\begin{lemma}\label{r3:lem-K}
For $\ell\geq0$ and $s\geq1$, with $\beta,d$ as above:
\begin{enumerate}[label=\textup{(\alph*)}]
\item $k_{\ell,s}'=\big(J^{\beta+1}_s-J^{\beta+1}_{s-1}\big)/(4(d+1))$;
\item $\Delta K_3\Psi_{\ell,s}=\Psi_{\ell,s}$;
\item $K_3\Psi_{\ell,s}=0$ and $\nabla K_3\Psi_{\ell,s}=0$ on $\partial B$;
\item with $\Psi_{-1,\cdot}=0$,
\begin{align*}
\partial_{x_1}K_3\Psi_{\ell,s}&=\frac{\ell+1}{(2\ell+1)2(d+1)}\big(\Psi_{\ell+1,s}-\Psi_{\ell+1,s-1}\big)+\frac{\ell}{(2\ell+1)2(d+1)}\big(\Psi_{\ell-1,s+1}-\Psi_{\ell-1,s}\big),\\
EK_3\Psi_{\ell,s}&=-\frac{d+\frac12}{4d(d+1)}\Psi_{\ell,s-1}+\frac{\Psi_{\ell,s}}{4d(d+2)}+\frac{d+\frac32}{4(d+1)(d+2)}\Psi_{\ell,s+1},
\end{align*}
and $y\partial_yK_3\Psi_{\ell,s}=EK_3\Psi_{\ell,s}-x_1\partial_{x_1}K_3\Psi_{\ell,s}$;
\item $K_3$, $\partial_{x_1}K_3$, $EK_3$ and $y\partial_yK_3$ extend to bounded operators $\Gc\to\mathcal L$ with $\nrm{K_3}=4/45$, $\nrm{\partial_{x_1}K_3}\leq2R/7$, $\nrm{EK_3}\leq6/35$ and $\nrm{y\partial_yK_3}\leq6/35+2R^2/7$.
\end{enumerate}
\end{lemma}

\begin{proof}
Write $\mathsf A=\frac1{4d(d+1)}$, $\mathsf B=-\frac1{2d(d+2)}$, $\mathsf C=\frac1{4(d+1)(d+2)}$ for the coefficients in \eqref{r3:eq-K}, and $J_\varsigma=J^\beta_\varsigma$. For $0\leq k\leq s$ put $u=s-k$; then $\beta+s+k+1=d-u+1$, and \eqref{r3:eq-Jcoef} gives
\begin{gather*}
\frac{[t^{k+1}]J_{s-1}}{[t^k]J_s}=\frac{u(u-1)}{(k+1)(\beta+k+1)},\qquad\frac{[t^{k+1}]J_s}{[t^k]J_s}=-\frac{u(d-u+1)}{(k+1)(\beta+k+1)},\\
\frac{[t^{k+1}]J_{s+1}}{[t^k]J_s}=\frac{(d-u+1)(d-u+2)}{(k+1)(\beta+k+1)} .
\end{gather*}
Since $u(u-1)(d+2)+2u(d+1)(d-u+1)+d(d-u+1)(d-u+2)=d(d+1)(d+2)$ identically in $u$ and $d$,
\begin{equation}\label{r3:eq-coreid}
[t^{k+1}]k_{\ell,s}=\frac{[t^k]J_s}{4(k+1)(\beta+k+1)}\qquad(0\leq k\leq s).
\end{equation}

(a) The coefficient of $t^k$ in $k_{\ell,s}'$ is $(k+1)[t^{k+1}]k_{\ell,s}$. On the other hand, by \eqref{r3:eq-Jcoef},
\[
\frac{[t^k]J^{\beta+1}_s}{[t^k]J_s}=\frac{\beta+s+k+1}{\beta+k+1},\qquad\frac{[t^k]J^{\beta+1}_{s-1}}{[t^k]J_s}=-\frac{s-k}{\beta+k+1},
\]
so the coefficient of $t^k$ on the right-hand side is
\[
[t^k]J_s\,\frac{\beta+2s+1}{4(d+1)(\beta+k+1)}=\frac{[t^k]J_s}{4(\beta+k+1)} .
\]
Both sides have degree $s$.

(b) For a polynomial $v(t)$, $\Delta(H_\ell v)=v\Delta H_\ell+2\nabla H_\ell\cdot\nabla v+H_\ell\Delta v=4H_\ell\big(tv''+(\ell+\frac32)v'\big)$, by \eqref{r3:eq-Hrec} and $\Delta v(|x|^2)=4tv''+6v'$. The coefficient of $t^k$ in $4(tk''+(\beta+1)k')$ is $4(k+1)(\beta+k+1)[t^{k+1}]k_{\ell,s}=[t^k]J_s$ by \eqref{r3:eq-coreid}.

(c) Since $J_\varsigma(1)=1$, $k_{\ell,s}(1)=\mathsf A+\mathsf B+\mathsf C=0$, and $k_{\ell,s}'(1)=0$ by (a). Hence $K_3\Psi_{\ell,s}=0$ and $\nabla(H_\ell k_{\ell,s})=k_{\ell,s}\nabla H_\ell+2k_{\ell,s}'H_\ell x=0$ on $\partial B$.

(d) By \eqref{r3:eq-Hrec}, $\partial_{x_1}(H_\ell k)=\frac{2(\ell+1)}{2\ell+1}H_{\ell+1}k'+\frac{2\ell}{2\ell+1}H_{\ell-1}(\beta k+tk')$ and $E(H_\ell k)=H_\ell(\ell k+2tk')$, with $k=k_{\ell,s}$. The $H_{\ell+1}$ term is given by (a). For the others we expand in the basis $J_\varsigma=J^\beta_\varsigma$: by (a) and the second identity of Lemma~\ref{r3:lem-jacobi}(b) (with $\beta+1$ in place of $\beta$),
\[
tk'=\frac1{4(d+1)}\Big(-\frac{\beta+s}dJ_{s-1}+\frac{\beta(d+1)}{d(d+2)}J_s+\frac{s+1}{d+2}J_{s+1}\Big),
\]
where we used $(\beta+s+1)d-s(d+2)=\beta(d+1)$. Adding $\beta k$, respectively $\ell k$, gives
\begin{align*}
\beta k+tk'&=-\frac{s}{4d(d+1)}J_{s-1}-\frac{\beta}{4d(d+2)}J_s+\frac{\beta+s+1}{4(d+1)(d+2)}J_{s+1},\\
\ell k+2tk'&=-\frac{d+\frac12}{4d(d+1)}J_{s-1}+\frac1{4d(d+2)}J_s+\frac{d+\frac32}{4(d+1)(d+2)}J_{s+1} .
\end{align*}
The second line is the formula for $EK_3$. By the first identity of Lemma~\ref{r3:lem-jacobi}(b) (with $\beta-1$ in place of $\beta$), $(J^{\beta-1}_{s+1}-J^{\beta-1}_s)/(4(d+1))$ has the same expansion as the first line, which gives the formula for $\partial_{x_1}K_3$. The last identity holds because $E=x_1\partial_{x_1}+y\partial_y$ on axisymmetric functions.

(e) We use Lemma~\ref{lem:columns}. The operator $K_3$ preserves $\ell$, and its column sum is $|\mathsf A|+|\mathsf B|+|\mathsf C|=1/(d(d+2))$, which is largest, $4/45$, for $\ell=0$, $s=1$. For $EK_3$ the column sum is $(2d+1)/(4d(d+1))$, decreasing in $d$, with value $6/35$ at $d=\frac52$. For $\partial_{x_1}K_3$ and even $\ell$, the normalised column norm is at most $\big((\ell+1)h_{\ell+1}+\ell h_{\ell-1}\big)/\big((2\ell+1)(d+1)h_\ell\big)$. For $\ell=0$ this is at most $R/(d+1)\leq2R/7$. For $\ell\geq2$, Lemma~\ref{r3:lem-weights}(b), $\ell^2/(\ell-1)\leq\ell+2$ and $d+1\geq\ell+\frac72$ bound it by $R(2\ell+3)/((2\ell+1)(\ell+\frac72))\leq7R/27.5<2R/7$. Finally $\nrm{x_1f}_{\mathcal L}\leq h_1\nrm f_{\mathcal L}=R\nrm f_{\mathcal L}$ by Lemma~\ref{r3:lem-mult}(a).
\end{proof}

The following lemma shows that $K_3g$ is a genuine $C^1$ function with vanishing Cauchy data. Let $\mathcal P_n$ be the space of real polynomials of degree at most $n$ on $\R^3$, with the inner product of $L^2(B)$, and let $V_n$ be the orthogonal complement of $\mathcal P_{n-1}$ in $\mathcal P_n$.

\begin{lemma}\label{r3:lem-C1}
\begin{enumerate}[label=\textup{(\alph*)}]
\item For $P\in V_n$ and $x\in\overline B$, $|P(x)|^2\leq\frac{(n+1)(n+2)(2n+3)}{8\pi}\nrm P^2_{L^2(B)}$.
\item For $\ell\geq0$ and $s\geq1$, $\sup_{\overline B}|\nabla K_3\Psi_{\ell,s}|\leq1$ and $\sup_{\overline B}|K_3\Psi_{\ell,s}|\leq1$.
\item For $g\in\Gc$, the series $W=\sum g_{\ell,s}K_3\Psi_{\ell,s}$ converges in $C^1(\overline B)$; $W=0$ and $\nabla W=0$ on $\partial B$; $\Delta W=\iota(g)$ in $\mathcal D'(B)$; and $\iota(K_3g)=W$, $\iota(\partial_{x_1}K_3g)=\partial_{x_1}W$, $\iota(y\partial_yK_3g)=y\partial_yW$.
\end{enumerate}
\end{lemma}

\begin{proof}
(a) Let $S^4\subset\R^5=\R^3\times\R^2$, with points $(x,z)$ and surface measure $d\varsigma$. Parametrising $S^4$ by $(x,\phi)\mapsto(x,\sqrt{1-|x|^2}(\cos\phi,\sin\phi))$ one finds $d\varsigma=dx\,d\phi$, so $\int_{S^4}f(x)\,d\varsigma=2\pi\int_Bf\,dx$ for integrable $f$. Let $Q(x,z)=P(x)$. Its restriction to $S^4$ is a sum $\sum_{k\leq n}Y_k$ of spherical harmonics of degree $k$ \cite[Ch.~IV]{SW71}. Since $Q$ is invariant under rotations of $z$ and the projections onto the spaces of spherical harmonics commute with rotations, each $Y_k$ is the restriction of a harmonic homogeneous polynomial invariant under these rotations, that is, of a polynomial in $x$ and $|z|^2$; as $|z|^2=1-|x|^2$ on $S^4$, $Y_k(x,z)=P_k(x)$ for some $P_k\in\mathcal P_k$. For $k<n$, $\nrm{Y_k}^2_{L^2(S^4)}=\int_{S^4}QY_k\,d\varsigma=2\pi\int_BPP_k\,dx=0$. Hence $Q|_{S^4}=Y_n$ is a spherical harmonic of degree $n$. For an orthonormal basis $(E_i)$ of the spherical harmonics of degree $n$ on $S^4$, $\sum_i|E_i|^2$ is invariant under rotations, hence equal to $N_n/|S^4|$, where $N_n=\binom{n+4}4-\binom{n+2}4=(n+1)(n+2)(2n+3)/6$ and $|S^4|=8\pi^2/3$. By the Cauchy--Schwarz inequality, $|P(x)|^2=|Y_n(x,z)|^2\leq\frac{N_n}{|S^4|}\nrm{Y_n}^2_{L^2(S^4)}=\frac{2\pi N_n}{|S^4|}\nrm P^2_{L^2(B)}$.

(b) Let $n=\ell+2s$ and $\Xi=K_3\Psi_{\ell,s}$, a polynomial of degree $n+2$ whose homogeneous components have the parity of $n$. By Lemma~\ref{r3:lem-embed}(a), $\Xi$ is orthogonal to $\mathcal P_{n-3}$. For $\varphi\in\mathcal P_{n-2}$ and $i=1,2,3$, $\int_B\partial_i\Xi\,\varphi=-\int_B\Xi\,\partial_i\varphi=0$, because $\Xi=0$ on $\partial B$. Hence $\partial_i\Xi\in V_{n-1}\oplus V_n\oplus V_{n+1}$. The reflection $x\mapsto-x$ preserves each $V_m$, and an element of $V_m$ of parity $(-1)^{m+1}$ has no homogeneous part of degree $m$, hence lies in $\mathcal P_{m-1}\cap V_m=0$; since $\partial_i\Xi$ has parity $(-1)^{n+1}$ and the orthogonal projections onto the $V_m$ commute with the reflection, the component of $\partial_i\Xi$ in $V_n$ vanishes, and $\partial_i\Xi=P_i+Q_i$ with $P_i\in V_{n-1}$, $Q_i\in V_{n+1}$. By (a) and the Cauchy--Schwarz inequality, with $\gamma_m=(m+1)(m+2)(2m+3)/(8\pi)$,
\[
|\nabla\Xi(x)|^2\leq(\gamma_{n-1}+\gamma_{n+1})\sum_i\big(\nrm{P_i}^2+\nrm{Q_i}^2\big)=(\gamma_{n-1}+\gamma_{n+1})\nrm{\nabla\Xi}^2_{L^2(B)} .
\]
By (c) of Lemma~\ref{r3:lem-K}, (b) of the same lemma and orthogonality, $\nrm{\nabla\Xi}^2_{L^2}=-\int_B\Xi\,\Psi_{\ell,s}=\nrm{\Psi_{\ell,s}}^2_{L^2}/(2d(d+2))$ with $d=n+\frac12$. Since $\gamma_{n-1}+\gamma_{n+1}=(2n+3)(n^2+3n+5)/(4\pi)$ and $\nrm{\Psi_{\ell,s}}^2_{L^2}=4\pi/((2\ell+1)(2n+3))$,
\[
|\nabla\Xi(x)|^2\leq\frac{2(n^2+3n+5)}{(2\ell+1)(2n+1)(2n+5)}\leq1,
\]
because $(4n^2+12n+5)-(2n^2+6n+10)=2n^2+6n-5>0$ for $n\geq1$. As $\Xi=0$ on $\partial B$, integration along a radius gives $|\Xi|\leq1$ on $\overline B$.

(c) By (b), $\sum|g_{\ell,s}|\leq\nrm g_{\mathcal L}$ bounds the series in $C^1(\overline B)$. The partial sums $W_N=K_3g_N$ are polynomials with $\Delta W_N=g_N$ and vanishing Cauchy data (Lemma~\ref{r3:lem-K}); since $g_N\to\iota(g)$ in $L^2(B)$ and $W_N\to W$ in $C^1(\overline B)$, the Cauchy data and the distributional equation pass to the limit. The coefficient series $K_3g$, $\partial_{x_1}K_3g$ and $y\partial_yK_3g$ converge in $\mathcal L$ (Lemma~\ref{r3:lem-K}(e)), hence in $L^2(B)$, to the same limits.
\end{proof}

\subsection{The quartic equation}\label{r3:sec-quartic}

Let
\[
b=7626062690497999\cdot2^{-46},\qquad \sigma=8b,
\]
a fixed dyadic rational independent of the varying coefficient $c_1$. For $x=(g,c)\in\mathcal X$ let $W=K_3g$, $W_1=\partial_{x_1}K_3g$ and $W_2=y\partial_yK_3g$, elements of $\mathcal L$ by Lemma~\ref{r3:lem-K}(e), and define
\begin{equation}\label{r3:eq-F}
F(g,c)=q_c\,g+(\partial_{x_1}q_c)\,W_1+\Big(\frac{\partial_yq_c}y\Big)W_2+q_c|p_c|^2(1+W),\qquad \Fc=F/b,
\end{equation}
where the multipliers are given by \eqref{r3:eq-qseries}, \eqref{r3:eq-p2}, the products are those of Lemma~\ref{r3:lem-mult}, and $1=\Psi_{0,0}$.

\begin{lemma}\label{r3:lem-F}
$\Fc$ is a bounded polynomial map of degree four from $\mathcal X$ to $\mathcal Y$. If $\Fc(g,c)=0$ and $\re\psi_c'>0$ on $\{|w|<R_0\}$ for some $1<R_0<R$, then $U=1+\iota(K_3g)\in C^1(\overline B)$ is axisymmetric, satisfies \eqref{r3:eq-weakU}, and $U=1$, $\nabla U=0$ on $\partial B$.
\end{lemma}

\begin{proof}
By \eqref{r3:eq-qseries}, $\nrm{r^{j-1}U_{j-1}}_{\mathrm{ch}}=\sum_{|\nu|\leq j-1,\,\nu\ \mathrm{even}}R^{|\nu|}$, which is at most $jR^{j-1}$ and at most $2R^{j+1}/(R^2-1)$; similarly $\nrm{r^mC^{(2)}_m}_{\mathrm{ch}}=\sum_{p+p'=m}(p+1)(p'+1)R^{|p-p'|}\leq2(m+1)R^{m+4}/(R^2-1)^2$. Hence
\begin{equation}\label{r3:eq-chnorms}
\begin{gathered}
\nrm{q_c}_{\mathrm{ch}}\leq\frac{\nrm c_{\mathcal E}}R,\qquad\nrm{\partial_{x_1}q_c}_{\mathrm{ch}}\leq\frac{2\nrm c_{\mathcal E}}{R^2-1},\\
\Big\|\frac{\partial_yq_c}y\Big\|_{\mathrm{ch}}\leq\frac{4R\nrm c_{\mathcal E}}{(R^2-1)^2},\qquad\nrm{|p_c|^2}_{\mathrm{ch}}\leq\frac{\nrm c^2_{\mathcal E}}{R^2},
\end{gathered}
\end{equation}
the last by \eqref{r3:eq-p2}, Lemma~\ref{r3:lem-mult}(c) and the identity $\sum_a|p_a|R^a=R^{-1}\nrm c_{\mathcal E}$. Using parts (a) and (b) of Lemma~\ref{r3:lem-mult} and Lemma~\ref{r3:lem-K}(e), it follows that \eqref{r3:eq-F} is a sum of bounded multilinear maps of degree at most four. Their values are even in $x_1$, by the parities recorded after \eqref{r3:eq-p2}, so they lie in $\mathcal L^{\mathrm{ev}}$, which is contained in $\mathcal Y$.

Suppose $\Fc(g,c)=0$. By Lemma~\ref{r3:lem-C1}(c), $W=\iota(K_3g)\in C^1(\overline B)$ has vanishing Cauchy data and $\Delta W=\iota(g)$, and by Lemma~\ref{r3:lem-mult}(d) and Lemma~\ref{r3:lem-embed}(b),
\[
0=b\,\iota(\Fc(g,c))=q_c\,\iota(g)+\partial_{x_1}q_c\,\partial_{x_1}W+\frac{\partial_yq_c}y\,y\partial_yW+q_c|p_c|^2U\quad\text{a.e. in }B .
\]
Off the axis, $\nabla q_c\cdot\nabla U=\partial_{x_1}q_c\partial_{x_1}W+\partial_yq_c\partial_yW$. Let $\varphi\in C_c^\infty(B)$. As $q_c$ is smooth on $\overline B$ (Lemma~\ref{r3:lem-lift}(a)), $q_c\varphi\in C_c^\infty(B)$ and
\begin{align*}
\int_B\big(q_c\nabla U\cdot\nabla\varphi-q_c|p_c|^2U\varphi\big)&=\int_B\big(\nabla W\cdot\nabla(q_c\varphi)-\varphi\nabla q_c\cdot\nabla U-q_c|p_c|^2U\varphi\big)\\
&=-\int_B\varphi\,b\,\iota(\Fc(g,c))=0,
\end{align*}
using $\int\nabla W\cdot\nabla(q_c\varphi)=-\int\iota(g)q_c\varphi$.
\end{proof}

\subsubsection*{Nonlinear bounds}
Let $x^\circ=(g^\circ,c^\circ)\in\mathcal X$ and let $G_0\geq\nrm{g^\circ}_{\mathcal L}$, $C_0\geq\nrm{c^\circ}_{\mathcal E}$. Put
\begin{gather*}
L_2=\kappa_R\Big[\frac1R+\frac2{R^2-1}\cdot\frac{2R}7+\frac{4R}{(R^2-1)^2}\Big(\frac6{35}+\frac{2R^2}7\Big)\Big],\qquad L_3=\frac{\kappa_R}{R^3},\qquad L_4=\frac4{45}\cdot\frac{\kappa_R}{R^3},\\
c_2=\frac1b\Big[\frac{L_2+3L_4C_0^2}\sigma+\frac{3C_0(L_3+L_4G_0)}{\sigma^2}\Big],\qquad c_3=\frac1b\Big[\frac{L_3+L_4G_0}{\sigma^3}+\frac{3L_4C_0}{\sigma^2}\Big],\qquad c_4=\frac{L_4}{b\sigma^3}.
\end{gather*}

\begin{lemma}\label{r3:lem-nonlinear}
For $h\in\mathcal X$ with $\nrm h_{\mathcal X}\leq r$,
\begin{gather*}
\nrm{\Fc(x^\circ+h)-\Fc(x^\circ)-D\Fc(x^\circ)h}_{\mathcal Y}\leq c_2r^2+c_3r^3+c_4r^4,\\
\nrm{D\Fc(x^\circ+h)-D\Fc(x^\circ)}_{\mathcal X\to\mathcal Y}\leq2c_2r+3c_3r^2+4c_4r^3 .
\end{gather*}
\end{lemma}

\begin{proof}
Let $\mathsf b(c,g)=q_cg+(\partial_{x_1}q_c)W_1(g)+(y^{-1}\partial_yq_c)W_2(g)$ and let $\mathsf t(c^1,c^2,c^3)$ be the symmetrisation of $q_{c^1}\re(p_{c^2}\overline{p_{c^3}})$, so that $b\Fc(g,c)=\mathsf b(c,g)+\mathsf t(c,c,c)(1+K_3g)$. By Lemma~\ref{r3:lem-mult}, \eqref{r3:eq-chnorms} and Lemma~\ref{r3:lem-K}(e),
\[
\nrm{\mathsf b(c,g)}_{\mathcal Y}\leq L_2\nrm c_{\mathcal E}\nrm g_{\mathcal L},\qquad \nrm{\mathsf t(c^1,c^2,c^3)f}_{\mathcal L}\leq L_3\textstyle\prod_i\nrm{c^i}_{\mathcal E}\,\nrm f_{\mathcal L},
\]
and $\nrm{K_3\gamma}_{\mathcal L}\leq\frac4{45}\nrm\gamma_{\mathcal L}$. Write $h=(\gamma,\eta)$; then $\nrm\gamma_{\mathcal L}\leq r$ and $\nrm\eta_{\mathcal E}\leq r/\sigma$. The terms of degree at least two in $h$ of $b\Fc(x^\circ+h)$ are $\mathsf b(\eta,\gamma)$, $3\mathsf t(c^\circ,\eta,\eta)+\mathsf t(\eta,\eta,\eta)$, and
\[
3\mathsf t(c^\circ,c^\circ,\eta)K_3\gamma+3\mathsf t(c^\circ,\eta,\eta)K_3g^\circ+3\mathsf t(c^\circ,\eta,\eta)K_3\gamma+\mathsf t(\eta,\eta,\eta)K_3g^\circ+\mathsf t(\eta,\eta,\eta)K_3\gamma .
\]
Their norms are at most $L_2\frac r\sigma r$, $3L_3C_0\frac{r^2}{\sigma^2}+L_3\frac{r^3}{\sigma^3}$, and $3L_4C_0^2\frac r\sigma r+3L_4C_0G_0\frac{r^2}{\sigma^2}+3L_4C_0\frac{r^2}{\sigma^2}r+L_4G_0\frac{r^3}{\sigma^3}+L_4\frac{r^3}{\sigma^3}r$, and collecting powers of $r$ gives the first bound. Each of these terms is the restriction to the diagonal of a multilinear map bounded by a product in which every factor $\eta$ contributes $r/\sigma$ and every factor $\gamma$ contributes $r$. Differentiating in the direction $k$ with $\nrm k_{\mathcal X}\leq1$ replaces one factor by a component of $k$, of norm at most $1/\sigma$, respectively $1$; hence the derivative of a term of degree $m$ is bounded by $m/r$ times the bound of the term, which gives the second estimate.
\end{proof}

\begin{proposition}\label{r3:prop-NK}
Let $A\colon\mathcal Y\to\mathcal X$ be bounded, linear and injective, and suppose that $\nrm{A\Fc(x^\circ)}_{\mathcal X}\leq Y$, $\nrm{I-AD\Fc(x^\circ)}_{\mathcal X\to\mathcal X}\leq Z$, $\nrm A\leq\mathsf a$, and that $r>0$ satisfies
\begin{equation}\label{r3:eq-radii}
Y+(Z-1)r+\mathsf a\,(c_2r^2+c_3r^3+c_4r^4)<0,\qquad Z+\mathsf a\,(2c_2r+3c_3r^2+4c_4r^3)<1 .
\end{equation}
Then $\Fc$ has exactly one zero in the closed ball $\overline B_{\mathcal X}(x^\circ,r)$.
\end{proposition}

\begin{proof}
Let $\Nc(x)=x-A\Fc(x)$. For $\nrm h_{\mathcal X}\leq r$, $\Nc(x^\circ+h)-x^\circ=-A\Fc(x^\circ)+(I-AD\Fc(x^\circ))h-A\big(\Fc(x^\circ+h)-\Fc(x^\circ)-D\Fc(x^\circ)h\big)$, whose norm is less than $r$ by Lemma~\ref{r3:lem-nonlinear} and the first inequality. On the convex ball, $D\Nc(x^\circ+h)=(I-AD\Fc(x^\circ))-A(D\Fc(x^\circ+h)-D\Fc(x^\circ))$ has norm less than one by the second inequality, so $\Nc$ is a contraction of the ball into itself. Its unique fixed point is the unique zero of $\Fc$ in the ball, since $A$ is injective.
\end{proof}

\subsection{The approximate inverse}\label{r3:sec-inverse}

From now on $M=180$ and $S=80$. The exact dyadic centre $x^\circ=(g^\circ,c^\circ)$ consists of $7280$ field coefficients with even $\ell\leq180$ and $1\leq s\leq80$, and $91$ odd shape coefficients $c_j^\circ$ with $j\leq181$. These $7371$ coefficients are stored in \url{anc/r3_convex/candidate.json}; their SHA-256 is recorded in the ancillary README. The centre is not assumed to be an exact zero. Write $q^\circ=q_{c^\circ}$, $p^\circ=\sum_ap_aw^a$ with $p_a=(a+1)c^\circ_{a+1}$, and $W^\circ=K_3g^\circ$.

\subsubsection*{Finite and tail coordinates}
The \emph{finite inputs} are the $7280$ field coordinates $(\ell,s)$ with $\ell\leq M$ even and $1\leq s\leq S$, and the $91$ shape coordinates $c_j$ with $j\leq M+1$. The \emph{finite rows} are the same $7280$ field rows $(\ell,s)$ and the $91$ harmonic rows $\mathcal T_n$, $n\leq M$ even. All other inputs and rows are \emph{tail} inputs and rows; tail rows are \emph{harmonic} ($\mathcal T_n$, $n\geq182$) or \emph{nonharmonic} ($(\ell,s)$ with $s\geq1$ and $\ell\geq182$ or $s\geq81$). Accordingly $\mathcal X=\mathcal X_f\oplus\mathcal X_t$ and $\mathcal Y=\mathcal Y_f\oplus\mathcal Y^H_t\oplus\mathcal Y^N_t$. On the shape tail and on the harmonic rows we use normalised coordinates
\[
\hat a_n=b(n+1)R^{n+1}c_{n+1},\qquad\hat y_n=R^ny_n\qquad(n\ \text{even}),
\]
so that the shape part of $\nrm\cdot_{\mathcal X}$ is $8\sum|\hat a_n|$ and the harmonic part of $\nrm\cdot_{\mathcal Y}$ is $\sum|\hat y_n|$.

\subsubsection*{The principal part on high shape modes}
For odd $j$, the derivative of $b^{-1}q_c|p_c|^2$ at $c^\circ$ in the direction $e_j$ contains the \emph{principal polynomial} $\Pi_j=\frac2b\,q^\circ\re(\overline{p^\circ}\,jw^{j-1})$, in which $q$ is frozen. On $\partial B$ we have $w=e^{i\theta}$, $\xi=\cos\theta$, $q^\circ=\sum_kq_ke^{ik\theta}$ with $q_k=\sum_{i>|k|}c^\circ_i$ ($k$ even, $i$ odd), and therefore
\[
\Pi_j|_{\partial B}=\frac{2j}b\sum_{a,k}p_aq_k\,T_{|j-1-a+k|}(\xi) .
\]
Let $\mathsf Me_j=\frac{2j}b\sum_{a,k}p_aq_k\,\mathcal T_{|j-1-a+k|}\in\mathcal Y$ be the harmonic element with these boundary values, and define the symbol and its defect
\[
s_\delta=\frac1{b^2}\sum_{k-a=2\delta}p_aq_kR^{2\delta}\quad(\delta\in\Z),\qquad \eps=\sum_\delta|s_\delta-\delta_{\delta0}| .
\]

\begin{lemma}\label{r3:lem-principal}
Let $n\geq182$ be even. In normalised coordinates, $\mathsf Me_{n+1}/(b(n+1)R^{n+1})$ has the entry $\frac2R\,s_\delta R^{|n+2\delta|-(n+2\delta)}$ in the harmonic row $|n+2\delta|$, summed over $\delta$. Only $-180\leq\delta\leq90$ occur, so the rows $|n+2\delta|\leq M$ occur only for $n\leq540$, and $n+2\delta\geq-178$. Let $T$ be the part of $\mathsf M$ from the shape tail to the harmonic tail rows and $\mathcal B$ the part from the shape tail to the finite harmonic rows. Then, on $\ell^1$ of the normalised coordinates,
\[
T=\frac2R(I+E),\qquad (E\hat a)_m=\sum_{n\geq182}(s_{(m-n)/2}-\delta_{mn})\hat a_n\quad(m\geq182),\qquad\nrm E\leq\eps .
\]
If $\eps<1$, then $T$ is invertible, $T^{-1}=\frac R2\sum_{k\geq0}(-E)^k$, $\nrm{T^{-1}}\leq R/(2(1-\eps))$, and $\big\|T^{-1}\hat y-\frac R2\sum_{k=0}^L(-E)^k\hat y\big\|_1\leq\frac{R\eps^{L+1}}{2(1-\eps)}\nrm{\hat y}_1$ for $L\geq0$.
\end{lemma}

\begin{proof}
The entry of $\mathsf Me_{n+1}$ in the row $\mathcal T_{|n+2\delta|}$ is $\frac{2(n+1)}b\sum_{k-a=2\delta}p_aq_k$; multiplying by $R^{|n+2\delta|}$ and dividing by $b(n+1)R^{n+1}$ gives the stated value, since $R^{2\delta}=R^{(n+2\delta)-n}$. As $0\leq a\leq180$ and $|k|\leq180$, $-180\leq\delta\leq90$. Hence rows $m\geq182$ are reached only with $m=n+2\delta$, and $T$ is $\frac2R$ times the compression to indices $\geq182$ of the Laurent operator with symbol $(s_\delta)$; the compression of a convolution has $\ell^1$ norm at most the $\ell^1$ norm of its kernel. The remaining statements are the Neumann series.
\end{proof}

\subsubsection*{The operator $A$}
Let $J^{\mathrm{fl}}$ be the frozen $7371\times7371$ binary64 Galerkin matrix and let $\widehat A$ be its frozen binary64 approximate inverse, regarded as an exact linear map $\mathcal Y_f\to\mathcal X_f$. They are reconstructed from a nearby numerical seed; the discrepancy from the true finite Jacobian at the full exact centre is enclosed separately. The boundary symbol of the tail inverse uses the full exact centre. Suppose $\eps<1$ and
\begin{equation}\label{r3:eq-eta}
\eta:=\nrm{I-\widehat AJ^{\mathrm{fl}}}_b<1,
\end{equation}
where $\nrm\cdot_b$ is the operator norm for the weights of $\mathcal X$ and $\mathcal Y$ with $\sigma$ replaced by $b$. For $y=y_f+y^H_t+y^N_t\in\mathcal Y$ define
\begin{equation}\label{r3:eq-A}
Ay=\widehat A\big(y_f-\mathcal BT^{-1}y^H_t\big)+T^{-1}y^H_t+y^N_t,
\end{equation}
where $T^{-1}y^H_t$ is a shape tail vector and $y^N_t$ is read as a field tail vector with the same coordinates. For a finite row $\iota$ with weight $w_\iota$ ($h_\ell$ or $R^n$) let $\alpha_\iota=\nrm{\widehat Ae_\iota}_{\mathcal X}/w_\iota$, let $\alpha^H_n=\alpha_{\mathcal T_n}$ for $n\leq M$, $\alpha_{\max}=\max_\iota\alpha_\iota$, and
\[
\rho_{\mathcal B}=\max_{n\geq182}\sum_{m\leq M}|\hat{\mathcal B}_{mn}|\,\alpha^H_m,\qquad A_t=\max\Big\{1,\ \frac{R(8+\rho_{\mathcal B})}{2(1-\eps)}\Big\},
\]
where $\hat{\mathcal B}$ is the matrix of $\mathcal B$ in normalised coordinates.

\begin{lemma}\label{r3:lem-A}
$A\colon\mathcal Y\to\mathcal X$ is a bounded bijection with $\nrm A\leq\max\{\alpha_{\max},A_t\}$. It maps every harmonic tail row, normalised, to a vector of norm at most $A_t$, and every nonharmonic tail row to the identical field tail input. Moreover $A(\mathsf Me_j)=e_j$ for every odd $j\geq183$, and $A\Psi_{\ell,s}=e_{(\ell,s)}$ for every field tail input $(\ell,s)$.
\end{lemma}

\begin{proof}
By \eqref{r3:eq-eta}, $\widehat AJ^{\mathrm{fl}}$ and hence $\widehat A$ are invertible, and $T$ is invertible by Lemma~\ref{r3:lem-principal}. The map $(x_f,\hat a,x^N)\mapsto(\widehat A^{-1}x_f+\mathcal B\hat a)+T\hat a+x^N$ is a two-sided inverse of $A$. We use Lemma~\ref{lem:columns}: a finite row $e_\iota$ is mapped to $\widehat Ae_\iota$; a nonharmonic tail row to itself; the normalised harmonic tail row $\hat y=e_m$ to $T^{-1}e_m$, of shape norm $8\nrm{T^{-1}e_m}_1\leq8R/(2(1-\eps))$, plus $-\widehat A\mathcal BT^{-1}e_m$, of norm at most $\rho_{\mathcal B}\nrm{T^{-1}e_m}_1$. Finally, $\mathsf Me_j=\mathcal Be_j+Te_j$ in normalised coordinates, so $A\mathsf Me_j=\widehat A(\mathcal Be_j-\mathcal BT^{-1}Te_j)+T^{-1}Te_j=e_j$.
\end{proof}

\subsection{The defect bounds}\label{r3:sec-Z}

We bound $Y=\nrm{A\Fc(x^\circ)}_{\mathcal X}$ and $Z=\nrm{I-AD\Fc(x^\circ)}$ by normalised column norms. The disjoint field partition is: $7280$ finite inputs; $443$ exceptional inputs $S<s<S_\ell$ for even $\ell\leq180$; the $91$ radial classes $s\geq S_\ell$; the $14$ explicit inputs $(\ell,1)$ and the $14$ radial classes $s\geq2$ for even $182\leq\ell\leq208$; and the uniform angular tail $\ell\geq210$, $s\geq1$. Thus there are $443+14=457$ explicitly checked exceptional field inputs in total. Here $S_\ell$ equals $93$ for $\ell=0,2,\ldots,12$; $92$ for $14,\ldots,30$; $91$ for $32,\ldots,50$; $90$ for $52,\ldots,62$; $89$ for $64,\ldots,72$; $88$ for $74,76,78$; $87$ for $80,\ldots,90$; $86$ for $\ell=92$; $85$ for $\ell=94$; and $81$ for $96,\ldots,180$. The shape inputs are the $310$ explicit odd indices $j\leq619$ and all odd indices $j\geq621$. These sets exhaust all input columns without overlap.

\subsubsection*{The residual and the explicit columns}
The residual $\Fc(x^\circ)$ is a polynomial. It is expanded in the basis $\Psi_{\ell,s}$ without truncation, in ball arithmetic, using \eqref{r3:eq-Hrec} and Lemma~\ref{r3:lem-jacobi}(b) for the multiplications by $x_1$ and $t$ (which give $x_1\Psi_{\ell,s}$ and $t\Psi_{\ell,s}$ as combinations of at most four, respectively three, $\Psi$'s), the recurrences $r^{n+1}Q_{n+1}=2x_1r^nQ_n-t\,r^{n-1}Q_{n-1}$ for $Q=T,U$ and the three-term recurrence of $C^{(2)}_n$ for the series \eqref{r3:eq-qseries}, \eqref{r3:eq-p2}, and \eqref{r3:eq-K} and Lemma~\ref{r3:lem-K}(d). Its harmonic part is converted to the coordinates of $\mathcal Y$ by the matrix $D$. Since $A$ has norm at most $\alpha_{\max}$ on $\mathcal Y_f$, one on $\mathcal Y^N_t$ and $A_t$ on $\mathcal Y^H_t$ (Lemma~\ref{r3:lem-A}),
\[
Y\leq\alpha_{\max}\nrm{\Fc(x^\circ)_f}_{\mathcal Y}+\nrm{\Fc(x^\circ)^N_t}_{\mathcal Y}+A_t\nrm{\Fc(x^\circ)^H_t}_{\mathcal Y} .
\]
For the $457$ exceptional field inputs and $310$ explicit shape inputs, the column $D\Fc(x^\circ)e$ is polynomial. The shape columns are expanded completely. For exceptional field columns alone, evaluation retains the shape coefficients only through $j=121$ and adds the rigorous omitted-tail bound below. In both cases $A$ uses the full centre and the full boundary symbol. It is applied in the form \eqref{r3:eq-A}, with $T^{-1}$ evaluated by $24$ terms of the Neumann series of Lemma~\ref{r3:lem-principal} plus the stated remainder bound, and $e$ is subtracted. Let $Z_{\mathrm{ex}}$ and $Z_{\mathrm{sh}}$ be the resulting maxima.

Put $c^t=(c_j^\circ\one_{j\leq121})_j$, $C=\nrm{c^\circ}_{\mathcal E}$, $\delta=\nrm{c^\circ-c^t}_{\mathcal E}$, $\kappa_R=4(1-R^{-2})^{-1/2}-1$, and $z=\ell+2s+\frac12$. For a normalised field basis vector the exact three-term inverse and derivative formulas give
\[
K_z=\frac1{z(z+2)},\qquad D_z=\frac{2R}{z+1},\qquad E_z=\frac1z+RD_z
\]
as upper bounds for the norms of $K_3e$, $\partial_{x_1}K_3e$ and $y\partial_yK_3e$, respectively. The linear multiplier inequalities \eqref{r3:eq-chnorms} and the telescoping expansion of the three factors of $q|p|^2$ then give
\[
\nrm{(D_g\Fc(c^\circ)-D_g\Fc(c^t))e}_{\mathcal Y}\leq
\frac{\kappa_R\delta}{b}\left[\frac1R+\frac{2D_z}{R^2-1}+\frac{4RE_z}{(R^2-1)^2}+\frac{3C^2K_z}{R^3}\right].
\]
Multiplication by the certified full $\nrm A$ bounds the omitted contribution to the $X$-defect. For these inputs $z\geq325/2$, giving a common error strictly below $0.007935$; the verifier uses the individual $z$ of each column. No truncation enters $A$ or the residual.

\subsubsection*{Finite field columns}
Let $e=e_{(\ell,s)}$ be a finite field input, and let $Je$ and $(D\Fc(x^\circ)e)_t$ be the components of $D\Fc(x^\circ)e$ in the finite rows and in the tail rows. Then
\[
(I-AD\Fc(x^\circ))e=(I-\widehat AJ^{\mathrm{fl}})e+\widehat A(J^{\mathrm{fl}}-J)e-A(D\Fc(x^\circ)e)_t .
\]
The entries of $Je$ are integrals over $B$ of polynomials, namely $\nrm{\Psi_{\ell',s'}}^{-2}_{L^2}\int_B(D\Fc(x^\circ)e)\Psi_{\ell',s'}$ for the field rows and, for the harmonic rows, $\sum_{\ell'\leq540}D_{n\ell'}\nrm{H_{\ell'}}^{-2}_{L^2}\int_B(D\Fc(x^\circ)e)H_{\ell'}$; all angular degrees of the output (at most $3M=540$, because $q^\circ|p^\circ|^2$ has degree at most $2M$ in $\xi$ at fixed $r$) are integrated before the conversion to Chebyshev coordinates. In spherical coordinates the integrands are even polynomials in $\xi$ of degree at most $1080$ and in $r$ of degree at most $(4M+2S+2)+3M+2=1424$. They are integrated exactly by Gauss--Legendre rules with $542$ nodes in $\xi$ (exact to degree $1083$) and $713$ nodes in $r$ (applied to the even extension to $[-1,1]$, exact to degree $1425$), whose nodes and weights are enclosed in ball arithmetic: each nonnegative root is bracketed by a sign change of the Legendre polynomial, the brackets are disjoint, together with their reflections their number equals the degree, and the weights are computed from the enclosures. Let $\xi_{\mathrm{fin}}$ bound $\max_e\sum_\iota|J^{\mathrm{fl}}_{\iota e}-J_{\iota e}|\,\alpha^b_\iota w_\iota/w_e$, where $\alpha^b_\iota$ is the normalised column norm of $\widehat A$ in the norm $\nrm\cdot_b$. Converting from $\nrm\cdot_b$ to $\nrm\cdot_{\mathcal X}$ costs at most a factor eight, so the first two terms contribute at most $Z_{\mathrm{head}}=8(\eta'+\xi_{\mathrm{fin}})$, where $\eta'$ is $\eta$ increased by explicit allowances for the representation of the weights $h_\ell$, $bjR^j$, $R^n$ by binary64 numbers (relative error below $10^{-12}$, checked) and for the rounding in their application. The last term is bounded with the output-tail envelope below, whose maximum over the finite field inputs we call $Z_{\mathrm{fin},t}$.

\subsubsection*{Envelopes for field columns}
We define step functions of the angular index $L$ and the radial index $\varsigma$. For even $\ell\leq M$ and $1\leq\varsigma\leq S$ let
\[
\mathsf A(\ell,\varsigma)=\max\{1,\ \alpha_{(\ell',\varsigma')}:\ \ell'\geq\ell,\ \varsigma'\geq\varsigma,\ (\ell',\varsigma')\text{ a finite field row}\},
\]
let $\mathsf A(\ell,\varsigma)=1$ if $\ell>M$ or $\varsigma>S$, and $\mathsf A(\ell,\varsigma)=\mathsf A(\ell+1,\varsigma)$ for odd $\ell$. Let $\alpha^H_n=A_t$ for $n\geq182$,
\[
\beta_\ell=\frac1{h_\ell}\sum_nD_{n\ell}R^n\alpha^H_n,\qquad\beta^\infty=A_t+\sqrt2\sum_{n\leq M}(\alpha^H_n-A_t)_+R^{n-1200},
\]
$\mathsf B(L)=\max\big(\{\beta_\ell:\ \ell\ \text{even},\ L\leq\ell\leq1200\}\cup\{\beta^\infty\}\big)$, and $\Gamma(L)=\max\{1,\mathsf A(L,1),\mathsf B(L)\}$. These functions are nonincreasing in $L$ and $\varsigma$.

\begin{lemma}\label{r3:lem-rows}
For $\ell\geq L$ (with $\ell$ even) and $\varsigma\geq1$, $\nrm{AH_\ell}_{\mathcal X}\leq\mathsf B(L)h_\ell$ and $\nrm{A\Psi_{\ell,\varsigma}}_{\mathcal X}\leq\mathsf A(L,\varsigma)h_\ell$.
\end{lemma}

\begin{proof}
$H_\ell=\sum_nD_{n\ell}\mathcal T_n$, and $\nrm{A\mathcal T_n}\leq\alpha^H_nR^n$ by Lemma~\ref{r3:lem-A}; hence $\nrm{AH_\ell}\leq\beta_\ell h_\ell$. For $\ell>1200$, since $\sum_nD_{n\ell}R^n=h_\ell$ and by Lemma~\ref{r3:lem-weights}(e), $\beta_\ell\leq A_t+\sum_{n\leq M}\sqrt2R^{n-\ell}(\alpha^H_n-A_t)_+\leq\beta^\infty$. The second bound holds by definition for finite rows and because $A$ is the identity on nonharmonic tail rows.
\end{proof}

For $d\geq0$, $\ell,s\geq0$ put $\ell_d=\max(0,\ell-d)$, $s_d=\max(1,s-d)$, let $X_d$ be a $\mathrm{Bin}(d,\frac12)$ random variable, and let $\phi_L(\varsigma)=\mathsf A(L,\varsigma)$ for $\varsigma\geq1$ and $\phi_L(\varsigma)=\max\{\mathsf A(L,1),\mathsf B(L)\}$ for $\varsigma\leq0$. Define
\begin{gather*}
\mathfrak B_d(\ell,s)=\min\Big\{\mathsf A(\ell_d,s_d)+\big(\mathsf B(\ell_d)-\mathsf A(\ell_d,s_d)\big)_+\vartheta_d(\ell,s),\ \ \mathbb E\,\phi_{\ell_d}(s-X_d)\Big\},\\
\mathfrak T_d(\ell,s)=\begin{cases}0,&\ell+d\leq M\text{ and }s+d\leq S,\\ 1+(A_t-1)\,\vartheta_d(\ell,s)\,\one_{\ell+d>M},&\text{otherwise},\end{cases}
\end{gather*}
with $\vartheta_d$ as in Lemma~\ref{r3:lem-radial}(b). Both bounds in $\mathfrak B_d$ are nonincreasing in $\ell$ and in $s$: the first because $\mathsf A,\mathsf B,\vartheta_d$ are and the expression is nondecreasing in $\mathsf A$ and in $\vartheta_d$; the second because $\phi_L$ is nonincreasing in $\varsigma$ and in $L$.

\begin{lemma}\label{r3:lem-envelope}
Let $H_at^k$ be a solid monomial of degree $d$ and $\ell,s\geq0$ with $a+\ell$ even. Then $\nrm{A(H_at^k\Psi_{\ell,s})}_{\mathcal X}\leq h_ah_\ell\,\mathfrak B_d(\ell,s)$, and the component $(H_at^k\Psi_{\ell,s})_t$ in the tail rows satisfies $\nrm{A(H_at^k\Psi_{\ell,s})_t}_{\mathcal X}\leq h_ah_\ell\,\mathfrak T_d(\ell,s)$.
\end{lemma}

\begin{proof}
By \eqref{r3:eq-monomial} and Lemma~\ref{r3:lem-gaunt} it suffices to show, for every $v$ with $G_{a\ell v}\neq0$, that
\[
\Sigma_v:=\sum_\varsigma\pi^{(v)}_\varsigma\,\frac{\nrm{A\Psi_{v,\varsigma}}}{h_v}\leq\mathfrak B_d(\ell,s) .
\]
By Lemma~\ref{r3:lem-radial}(a), $v\geq\ell_d$ and $\varsigma\geq s-d$. By Lemma~\ref{r3:lem-rows} and monotonicity, the quotient $\nrm{A\Psi_{v,\varsigma}}/h_v$ is at most $\mathsf A(\ell_d,\varsigma)\leq\mathsf A(\ell_d,s_d)$ for $\varsigma\geq1$, and at most $\mathsf B(\ell_d)$ for $\varsigma=0$. With $\pi^{(v)}_0\leq\vartheta_d(\ell,s)$ this gives the first bound in $\mathfrak B_d$. For the second, the quotient is at most $\phi_{\ell_d}(\varsigma)$, the function $\phi_{\ell_d}$ is nonincreasing, the index is at least $s-N$ with $N$ as in the proof of Lemma~\ref{r3:lem-radial}, and $N$ is stochastically dominated by $X_d$ (Lemma~\ref{r3:lem-radial}(c)); hence $\Sigma_v\leq\mathbb E\,\phi_{\ell_d}(s-X_d)$.

For the tail component: by Lemma~\ref{r3:lem-radial}(a) the outputs have $v\leq\ell+d$ and $\varsigma\leq s+d$. If $\ell+d\leq M$ and $s+d\leq S$, all of them are finite rows, including the harmonic outputs $H_v=\sum_{n\leq v}D_{nv}\mathcal T_n$. Otherwise the tail component consists of nonharmonic tail rows, on which $A$ is the identity, and, only if $\ell+d>M$, of the parts $\sum_{n\geq182}D_{nv}\mathcal T_n$ of harmonic outputs, of norm at most $h_v$, on which $A$ has norm at most $A_t$. Weighting with $\pi^{(v)}_0\leq\vartheta_d$ gives the bound.
\end{proof}

At the centre, let $\chi_1=q^\circ/b-1$, $\chi_2=\partial_{x_1}q^\circ/b$, $\chi_3=y^{-1}\partial_yq^\circ/b$, $\chi_4=x_1\chi_3$ and $\chi_5=q^\circ|p^\circ|^2/b$. For $1\leq i\leq5$ let $\tau_i(d)$ be the $\tau$-norm of the homogeneous part of degree $d$ of $\chi_i$. By Lemma~\ref{r3:lem-mult}(b), $\tau_1,\tau_2,\tau_3$ are the Chebyshev norms of the terms of \eqref{r3:eq-qseries}, with $\tau_1(0)=|c^\circ_1/b-1|$; $\tau_4(d+1)\leq R\,\tau_3(d)$; and $\tau_5$ is computed from the expansion of $\chi_5$ in solid monomials. For $\mathfrak X\in\{\mathfrak B,\mathfrak T\}$ put $\mathfrak X^{(i)}(\ell,s)=\sum_d\tau_i(d)\mathfrak X_d(\ell,s)$ and, with $\bar d=\ell+2s+\frac12$ (the parameter $d$ of Lemma~\ref{r3:lem-K}) and ratios $\rho_\pm$,
\begin{align*}
\mathfrak z_{\mathfrak X}(\ell,s;\rho_+,\rho_-,\lambda_+,\lambda_-)={}&\mathfrak X^{(1)}(\ell,s)+\sum_{i=2,4}\Big[\frac{\lambda_+\rho_+}{2(\bar d+1)}\big(\mathfrak X^{(i)}(\ell+1,s)+\mathfrak X^{(i)}(\ell+1,s-1)\big)\\
&\qquad+\frac{\lambda_-\rho_-}{2(\bar d+1)}\big(\mathfrak X^{(i)}(\ell-1,s+1)+\mathfrak X^{(i)}(\ell-1,s)\big)\Big]\\
&+\frac{(\bar d+\frac12)\mathfrak X^{(3)}(\ell,s-1)}{4\bar d(\bar d+1)}+\frac{\mathfrak X^{(3)}(\ell,s)}{4\bar d(\bar d+2)}+\frac{(\bar d+\frac32)\mathfrak X^{(3)}(\ell,s+1)}{4(\bar d+1)(\bar d+2)}\\
&+\frac{\mathfrak X^{(5)}(\ell,s-1)}{4\bar d(\bar d+1)}+\frac{\mathfrak X^{(5)}(\ell,s)}{2\bar d(\bar d+2)}+\frac{\mathfrak X^{(5)}(\ell,s+1)}{4(\bar d+1)(\bar d+2)} ,
\end{align*}
where terms with $\ell-1<0$ are omitted. The exact values are $\rho_\pm=h_{\ell\pm1}/h_\ell$, $\lambda_+=(\ell+1)/(2\ell+1)$, $\lambda_-=\ell/(2\ell+1)$.

\begin{proposition}\label{r3:prop-gtail}
\begin{enumerate}[label=\textup{(\alph*)}]
\item For a field tail input $(\ell,s)$, $\nrm{(I-AD\Fc(x^\circ))e_{(\ell,s)}}_{\mathcal X}/h_\ell\leq\mathfrak z_{\mathfrak B}(\ell,s)$ (exact parameters), and for fixed $\ell$ this bound is nonincreasing in $s$. Hence the radial class of $\ell$ is bounded by $Z_{\mathrm{rad}}(\ell)=\mathfrak z_{\mathfrak B}(\ell,S_\ell)$.
\item The uniform angular class is bounded by $Z_{\mathrm{ang}}=\mathfrak z_{\mathfrak B}\big(210,1;R,\frac{210}{209R},\frac{211}{421},\frac12\big)$.
\item For a finite field input $(\ell,s)$, $\nrm{A(D\Fc(x^\circ)e_{(\ell,s)})_t}_{\mathcal X}/h_\ell\leq\mathfrak z_{\mathfrak T}(\ell,s)$ (exact parameters).
\end{enumerate}
\end{proposition}

\begin{proof}
For a field input, by \eqref{r3:eq-F} and Lemma~\ref{r3:lem-K}(d),
\[
D\Fc(x^\circ)e_{(\ell,s)}=\Psi_{\ell,s}+\chi_1\Psi_{\ell,s}+\chi_2\partial_{x_1}K_3\Psi_{\ell,s}+\chi_3EK_3\Psi_{\ell,s}-\chi_4\partial_{x_1}K_3\Psi_{\ell,s}+\chi_5K_3\Psi_{\ell,s} .
\]
For a tail input, $A\Psi_{\ell,s}=e_{(\ell,s)}$ (Lemma~\ref{r3:lem-A}), so $(I-AD\Fc(x^\circ))e_{(\ell,s)}$ is $-A$ applied to the sum of the five remaining terms. Expanding them with Lemma~\ref{r3:lem-K}(d) and \eqref{r3:eq-K}, and bounding each product of a homogeneous part of $\chi_i$ with a $\Psi$ by Lemma~\ref{r3:lem-envelope}, gives $\mathfrak z_{\mathfrak B}(\ell,s)h_\ell$; the factors $\rho_\pm$ convert $h_{\ell\pm1}$ to $h_\ell$. Every quantity in $\mathfrak z_{\mathfrak B}$ is nonincreasing in $s$ (the coefficients decrease in $\bar d$, and each $\mathfrak B_d$ is nonincreasing), which proves (a). For (b), let $\ell\geq210$ and $s\geq1$. By Lemma~\ref{r3:lem-weights}(b), $\rho_+\leq R$ and $\rho_-\leq\frac{210}{209R}$; moreover $\lambda_+\leq\frac{211}{421}$, $\lambda_-\leq\frac12$, $\bar d\geq212.5$, and $\mathfrak B_d$ is nonincreasing in both arguments, so every term is bounded by the corresponding term at $\ell=210$, $s=1$. For (c), $\Psi_{\ell,s}$ is a finite row, so the tail component of $D\Fc(x^\circ)e_{(\ell,s)}$ is that of the five remaining terms, and Lemma~\ref{r3:lem-envelope} applies with $\mathfrak T$.
\end{proof}

For the fourteen near angular degrees $182\leq\ell\leq208$, the same bound with the exact $\ell$ and $s=2$, monotone in $s$, covers each infinite radial class; the input $(\ell,1)$ is included among the $457$ explicit exceptional columns. We denote the maximum of these fourteen infinite-class bounds by $Z_{\mathrm{ang,near}}$. Thus the finite field columns are bounded by $Z_{\mathrm{fin}}=Z_{\mathrm{head}}+Z_{\mathrm{fin},t}$, where $Z_{\mathrm{fin},t}=\max\mathfrak z_{\mathfrak T}(\ell,s)$ over the $7280$ finite field inputs.

\subsubsection*{Far shape columns}
For a polynomial $f=\sum_aH_af_a(t)\in\mathcal L$ write $f_a=\sum_sf_{a,s}J^{a+1/2}_s=\sum_mf^\natural_{a,m}(1-t)^m$, and let
\[
V_f(L)=\sum_ah_a\,\Gamma\big((L-a)_+\big)\min\Big\{\sum_s|f_{a,s}|,\ \sum_m|f^\natural_{a,m}|\frac{m!}{\big(\frac{(L-a)_++1/2}2+1\big)_m}\Big\},
\]
which is nonincreasing in $L$. With $\zeta=R^{-2}$, $N\geq0$ and $J=\lfloor N/2\rfloor$ put
\begin{align*}
\mathsf T_f(N)&=\frac{B_0}2V_f(N)+\sum_{j=1}^J\frac{B_0a_j\zeta^j}{2j-1}V_f(N-2j)+\frac{B_0a_{J+1}\zeta^{J+1}}{(2J+1)(1-\zeta)}V_f(0),\\
\mathsf U_f(N)&=\sum_{j=0}^JB_0a_j\zeta^jV_f(N-2j)+\frac{B_0a_{J+1}\zeta^{J+1}}{1-\zeta}V_f(0).
\end{align*}

\begin{lemma}\label{r3:lem-V}
Let $f\in\mathcal L$ be a polynomial, $e\geq0$, and let $L$, $n$ have the parity for which the products below are even in $x_1$. Then $\nrm{A(H_Lt^ef)}_{\mathcal X}\leq h_LV_f(L)$. For $n\geq N$, $\nrm{A(r^nT_n(\xi)t^ef)}_{\mathcal X}\leq R^n\mathsf T_f(N)$ and $\nrm{A(r^nU_n(\xi)t^ef)}_{\mathcal X}\leq R^n\mathsf U_f(N)$. More generally, for every $m$ with $\tau(m)<\infty$ whose homogeneous parts have degree at most $d_m$, $\nrm{A(m\,r^nT_n(\xi)t^ef)}\leq\tau(m)R^n\mathsf T_f(N-d_m)$ for $n\geq N\geq d_m$.
\end{lemma}

\begin{proof}
By Lemma~\ref{r3:lem-gaunt}, $H_Lt^eH_af_a=\sum_kG_{Lak}H_kt^{e+(L+a-k)/2}f_a$ with $k\geq|L-a|$. The $\mathcal L$-coefficients of $H_kt^{e'}f_a$ lie in angular index $k$, and their $\ell^1$ norm in the basis $J^{k+1/2}$ is at most $\sum_s|f_{a,s}|$, by the stochastic steps in the proof of Lemma~\ref{r3:lem-mult}(a), and at most $\sum_m|f^\natural_{a,m}|\nrm{(1-t)^m}_{J^{k+1/2},1}$, by Lemma~\ref{r3:lem-jacobi}(c),(d) (multiplication by $t$ does not increase $\nrm\cdot_{J^\beta,1}$); the latter norm is at most $m!/(((L-a)_++\frac12)/2+1)_m$ since $k\geq(L-a)_+$. On outputs of angular index at least $(L-a)_+$, $A$ has normalised norm at most $\Gamma((L-a)_+)$ by Lemma~\ref{r3:lem-rows}. Summing with $\sum_kG_{Lak}h_k\leq h_Lh_a$ gives the first bound. Next, $r^nT_n(\xi)=\sum_jC_{n-2j,n}H_{n-2j}t^j$, so by the first bound, Lemma~\ref{r3:lem-weights}(d) and the monotonicity of $V_f$,
\begin{align*}
\nrm{A(r^nT_nt^ef)}&\leq\sum_{j\leq n/2}|C_{n-2j,n}|h_{n-2j}V_f(n-2j)\\
&\leq R^n\Big(\frac{B_0}2V_f(N)+\sum_{j\geq1}\frac{B_0a_j\zeta^j}{2j-1}V_f\big((N-2j)_+\big)\Big),
\end{align*}
and the terms with $j>J$ are bounded using $\sum_{j>J}a_j\zeta^j/(2j-1)\leq a_{J+1}\zeta^{J+1}/((2J+1)(1-\zeta))$. The bound for $U_n$ is the same with $\upsilon$. For a multiplier $m$ we write $m=\sum m_{i,k'}H_it^{k'}$ with $i\leq d_m$ and use $H_iH_{n-2j}=\sum_kG_{i,n-2j,k}H_kt^{(\cdots)}$ with $k\geq n-2j-d_m$, together with $\sum_kG_{i,n-2j,k}h_k\leq h_ih_{n-2j}$.
\end{proof}

\begin{lemma}\label{r3:lem-axis}
\begin{enumerate}[label=\textup{(\alph*)}]
\item $y^{-1}\partial_yH_\ell=-\sum_{k=\ell-2,\ell-4,\ldots\geq0}(2k+1)H_kt^{(\ell-2-k)/2}$. Hence, for a polynomial $W=\sum_\ell H_\ell w_\ell(t)$, $Z_W=y^{-1}\partial_yW=\sum_\ell\big((y^{-1}\partial_yH_\ell)w_\ell+2H_\ell w_\ell'\big)$ is a polynomial, and $y\partial_yW=t(1-\xi^2)Z_W$.
\item $2(1-\xi^2)C^{(2)}_{n-2}(\xi)=\frac12\big(U_n(\xi)+U_{n-2}(\xi)\big)-nT_n(\xi)$ for $n\geq2$.
\end{enumerate}
\end{lemma}

\begin{proof}
(a) We have
\[
\partial_y\big(r^\ell P_\ell(x_1/r)\big)=yr^{\ell-2}\big(\ell P_\ell(\xi)-\xi P_\ell'(\xi)\big)=-yr^{\ell-2}P_{\ell-1}'(\xi),
\]
and $P'_{\ell-1}=\sum_{k\equiv\ell-2}(2k+1)P_k$ by the classical identities $\xi P_\ell'-P_{\ell-1}'=\ell P_\ell$ and $P'_{k+1}-P'_{k-1}=(2k+1)P_k$. The rest follows from $\partial_yw(t)=2yw'(t)$ and $y^2=t(1-\xi^2)$.
(b) We have $2C^{(2)}_{n-2}=U'_{n-1}$ and $2\xi U_{n-1}=U_n+U_{n-2}$; differentiating $U_{n-1}(\cos\theta)=\sin n\theta/\sin\theta$ gives $(1-\xi^2)U'_{n-1}(\xi)=\xi U_{n-1}(\xi)-nT_n(\xi)$.
\end{proof}

Let $H^\circ=g^\circ+|p^\circ|^2(1+W^\circ)$, $W^\circ_1=\partial_{x_1}W^\circ$ and $Z^\circ=Z_{W^\circ}$. These are polynomials; their Jacobi coefficients and boundary Taylor coefficients enter $V_f$. By Lemma~\ref{r3:lem-K}(c), every angular component $f_a$ of $W^\circ$ has a double zero at $t=1$, and every angular component of $W^\circ_1$ and of $Z^\circ$ vanishes at $t=1$, so the corresponding Taylor coefficients are exactly zero. Let $q^\circ_{(d)}=c^\circ_{d+1}r^dU_d(\xi)$ be the homogeneous parts of $q^\circ$, with $\tau(q^\circ_{(d)})=|c^\circ_{d+1}|\,\nrm{r^dU_d}_{\mathrm{ch}}$. Finally, for $N\geq360$ let
\[
\mathsf P(N)=\frac2{b^2R}\sum_{a,i,\varsigma}|p_ac^\circ_i|\,R^{\varsigma-a}\Big[\sum_{\nu=0}^{J_K}w^T_\nu\,\Gamma(K-2\nu)\frac{2(m+\nu)}{K-\nu+m+\frac32}+\frac{2B_0a_{J_K+1}\zeta^{J_K+1}}{(2J_K+1)(1-\zeta)}\Gamma(0)\Big],
\]
where $a\in\{0,2,\ldots,180\}$, $i\in\{1,3,\ldots,181\}$, $\varsigma\in\{-(i-1),-(i-3),\ldots,i-1\}$, $K=N-a+\varsigma$, $m=a+(i-1-\varsigma)/2$, $J_K=\lfloor K/2\rfloor$, $w^T_0=B_0/2$ and $w^T_\nu=B_0a_\nu\zeta^\nu/(2\nu-1)$ for $\nu\geq1$.

\begin{proposition}\label{r3:prop-farshape}
Let $N=620$. For every odd $j\geq N+1$,
\begin{align*}
\frac{\nrm{(I-AD\Fc(x^\circ))e_j}_{\mathcal X}}{\sigma jR^j}\leq\frac18\Big[&\frac{\mathsf U_{H^\circ}(N)}{b^2(N+1)R}+\frac{\mathsf U_{W^\circ_1}(N-1)}{b^2R^2}\\
&+\frac{\mathsf T_{Z^\circ}(N)+\big(\mathsf U_{Z^\circ}(N)+R^{-2}\mathsf U_{Z^\circ}(N-2)\big)/(2N+2)}{b^2R}\\
&+\frac2{b^2R}\sum_a|p_a|R^{-a}\sum_d\tau(q^\circ_{(d)})\,\mathsf T_{W^\circ}(N-a-d)+\mathsf P(N)\Big]=:Z_{\mathrm{far}} .
\end{align*}
\end{proposition}

\begin{proof}
Let $j=n+1$ with $n\geq N$ even. Differentiating \eqref{r3:eq-F} in the direction $e_j$, with $\delta q=r^nU_n$, $\delta\partial_{x_1}q=jr^{n-1}U_{n-1}$, $\delta(y^{-1}\partial_yq)=-2r^{n-2}C^{(2)}_{n-2}$ and $\delta p=jw^n$ by \eqref{r3:eq-qseries}, gives
\[
bD\Fc(x^\circ)e_j=r^nU_nH^\circ+j\,r^{n-1}U_{n-1}W^\circ_1-2r^{n-2}C^{(2)}_{n-2}W^\circ_2+2q^\circ\re(\overline{p^\circ}jw^n)(1+W^\circ),
\]
with $W^\circ_2=y\partial_yW^\circ=r^2(1-\xi^2)Z^\circ$ by Lemma~\ref{r3:lem-axis}(a). Since $A\mathsf Me_j=e_j$ (Lemma~\ref{r3:lem-A}), $(I-AD\Fc(x^\circ))e_j=-A(D\Fc(x^\circ)e_j-\mathsf Me_j)$, and we bound five parts, dividing by the input weight $\sigma jR^j=8b(n+1)R^{n+1}$.

The first two terms are bounded by Lemma~\ref{r3:lem-V}, using $n+1\geq N+1$. For the third, Lemma~\ref{r3:lem-axis}(b) gives
\[
-2r^{n-2}C^{(2)}_{n-2}W^\circ_2=-r^n\Big(\frac{U_n+U_{n-2}}2-nT_n\Big)Z^\circ,\qquad r^nU_{n-2}=t\,r^{n-2}U_{n-2},
\]
and Lemma~\ref{r3:lem-V} applies, since $n/(n+1)<1$. For the fourth, we write
\[
2\re(\overline{p^\circ}w^n)=2\sum_ap_at^ar^{n-a}T_{n-a}(\xi),\qquad q^\circ=\sum_dq^\circ_{(d)},
\]
and apply the last statement of Lemma~\ref{r3:lem-V} to each product $q^\circ_{(d)}\,r^{n-a}T_{n-a}\,t^aW^\circ$.

The last part is $\frac2bq^\circ\re(\overline{p^\circ}jw^n)-\mathsf Me_j$. Since $T_{n-a}U_{i-1}=\sum_\varsigma T_{n-a+\varsigma}$, with $\varsigma$ as in the definition of $\mathsf P$ and all indices positive because $n-a-(i-1)\geq N-360>0$,
\[
\frac2bq^\circ\re(\overline{p^\circ}jw^n)=\frac{2j}b\sum_{a,i,\varsigma}p_ac^\circ_i\,r^{K+2m}T_K(\xi),\qquad K=n-a+\varsigma,\quad m=a+\frac{i-1-\varsigma}2,
\]
and $\mathsf Me_j$ is obtained by replacing $r^{K+2m}T_K$ by $\mathcal T_K$, which has the same boundary values. Now $r^{K+2m}T_K-\mathcal T_K=\sum_\nu C_{K-2\nu,K}H_{K-2\nu}(t^{m+\nu}-1)$, and by Lemma~\ref{r3:lem-jacobi}(d) the function $H_\ell(t^M-1)$ has $\mathcal L$-coefficients in angular index $\ell$ only, with $\ell^1$ norm $2M/(\ell+M+\frac32)$ in the basis $J^{\ell+1/2}$. With Lemmas~\ref{r3:lem-rows} and~\ref{r3:lem-weights}(d),
\[
\nrm{A(r^{K+2m}T_K-\mathcal T_K)}\leq R^K\sum_{\nu\leq K/2}w^T_\nu\,\Gamma(K-2\nu)\,\frac{2(m+\nu)}{K-\nu+m+\frac32} .
\]
For fixed $(a,i,\varsigma)$, $R^{K-n}=R^{\varsigma-a}$ and $m$ do not depend on $n$, while increasing $n$ increases $K$; each summand is nonincreasing in $K$, and the new summands with $\nu>J_K$ that appear for larger $K$ are bounded by $2w^T_\nu\Gamma(0)$, whose sum over $\nu>J_K$ is at most the last term of the bracket. Hence the bracket in $\mathsf P(N)$, evaluated at $K=N-a+\varsigma$, bounds the corresponding sum for every $n\geq N$. Collecting the five parts proves the claim.
\end{proof}

The program evaluates an upper bound for $\mathsf P(N)$ by grouping the sum over $\nu$ into blocks and bounding each block by its mass times the values of the factors $\Gamma(K-2\nu)$ and $2(m+\nu)/(K-\nu+m+\frac32)$, both nondecreasing in $\nu$, at the end of the block.

Collecting the finite, exceptional, radial, near-angular, uniform-angular, explicit-shape and far-shape classes,
\[
Z=\max\big\{Z_{\mathrm{fin}},\ Z_{\mathrm{ex}},\ \max_{\ell}Z_{\mathrm{rad}}(\ell),\ Z_{\mathrm{ang}},\ Z_{\mathrm{ang,near}},\ Z_{\mathrm{sh}},\ Z_{\mathrm{far}}\big\}
\]
bounds $\nrm{I-AD\Fc(x^\circ)}_{\mathcal X\to\mathcal X}$.

\subsection{Certified bounds and the exact zero}\label{r3:sec-zero}
The bounds below are outward roundings of the interval enclosures in \texttt{anc/r3\_convex/VERIFICATION\_RECEIPT.json}. The residual uses the entire polynomial support and the fixed dyadic $b$, not a moving value of $c_1$. The certified finite inverse is nonsingular: the exact dyadic product defect, including a rigorous binary64 dot-product error, is less than $0.000173566$. The boundary-symbol defect is less than $0.015579$, hence both blocks of $A$ are invertible.

\begin{table}[ht]
\centering\footnotesize
\begin{tabular}{@{}ll@{}}
\toprule
Quantity & Certified bound\\
\midrule
$\nrm A$ & $<103241.73$\\
$Y=\nrm{A\Fc(x^\circ)}_{\mathcal X}$ & $<9.358\cdot10^{-9}$\\
$Z_{\mathrm{fin}}$, all $457$ exceptional field columns & $<0.283115$, $<0.669070$\\
$Z_{\mathrm{ang,near}}$ (fourteen radial classes) & $<0.559208$\\
$\max_\ell Z_{\mathrm{rad}}(\ell)$, $Z_{\mathrm{ang}}$ & $<0.776701$, $<0.767365$\\
$Z_{\mathrm{sh}}$, $Z_{\mathrm{far}}$ & $<0.547873$, $<0.501461$\\
$Z=\nrm{I-AD\Fc(x^\circ)}_{\mathcal X\to\mathcal X}$ & $<0.776701$\\
$c_2$, $c_3$, $c_4$ & $<3.163938$, $<1.304336\cdot10^{-5}$, $<1.317589\cdot10^{-11}$\\
$r=10^{-7}$: radii polynomial & $<-9.705\cdot10^{-9}$\\
$r=10^{-7}$: contraction bound & $<0.8421$\\
\bottomrule
\end{tabular}
\caption{The convex three-dimensional certificate for $M=180$, $S=80$, $R=21/20$ and $\sigma=8b$. The final $Z$ is the maximum over every field and shape class. The full interval values are in the receipt.}
\label{r3:tab-bounds}
\end{table}

\begin{theorem}\label{r3:thm-zero}
There is exactly one zero $x^*=(g^*,c^*)$ of $\Fc$ in $\overline B_{\mathcal X}(x^\circ,10^{-7})$.
\end{theorem}
\begin{proof}
The finite dyadic product and the Toeplitz symbol bounds give an injective $A$ by Lemmas~\ref{r3:lem-principal} and~\ref{r3:lem-A}. The partition above covers all derivative columns and yields $Z<0.776701$. The verifier evaluates both inequalities \eqref{r3:eq-radii} with the full Arb enclosures of $Y,Z,\nrm A,c_2,c_3,c_4$ at the exact radius $10^{-7}$. Their left sides are respectively below $-9.705\cdot10^{-9}$ and $0.8421<1$, so Proposition~\ref{r3:prop-NK} applies.
\end{proof}

\begin{remark}[Shape scale]\label{r3:rem-centre}
The exact $b$ is held fixed while the shape varies. At the same centre, the scale $\sigma=3b$ cannot certify this fixed-point operator: the normalised defect of shape column $j=183$ is strictly greater than $1.15486483728$, as recorded in \url{anc/r3_convex/SCALE3_OBSTRUCTION.json}. This excludes contraction for that operator in that norm; it says nothing about other operators or centres. The scale $8b$ used above gives $Z<0.776701$ and $c_2<3.163938$.
\end{remark}

\subsection{Reconstruction and convexity}\label{r3:sec-recon}
Let $\psi=\psi_{c^*}$, and put $\varepsilon=\nrm{c^*-c^\circ}_{\mathcal E}\leq r/\sigma$. The existence ball projects inside the shape ball of $\mathcal E$-radius one; in fact the projected radius is $68719476736/595786147695156171875<1$. Write
\[
D_1=\sum_{j\geq3}j|c_j^\circ|,\quad D_2=\sum_{j\geq3}j(j-1)|c_j^\circ|,\quad H=\sup_{j\geq3,\ j\,\mathrm{odd}}(j-1)R^{-j}.
\]
The last supremum is a finite maximum through $j=21$, since the ratio of successive odd-index terms is at most one for $j\geq21$.

\begin{proposition}\label{r3:prop-geometry}
For every $c$ with $\nrm{c-c^\circ}_{\mathcal E}\leq1$, the map $\psi_c$ is univalent on $|w|\leq101/100$, its image $D=\psi_c(\D)$ is strictly convex and strictly star-shaped, and $c_{19}\ne0$. More precisely, on the closed unit disc,
\begin{equation}\label{r3:eq-geom}
\re\left(1+\frac{w\psi_c''(w)}{\psi_c'(w)}\right)>0.7878,
\end{equation}
and $\re\psi_c'>106.2027$ on $|w|\leq101/100$, $\re(w\psi_c'/\psi_c)>0.9717$ on $|w|=1$, and $|c_{19}|>0.000665$.
\end{proposition}
\begin{proof}
For $|w|\leq1$ the triangle inequality gives
\[
\re\left(1+\frac{w\psi_c''}{\psi_c'}\right)
\geq1-\frac{D_2+\varepsilon H}{c_1^\circ-D_1-\varepsilon/R}.
\]
The denominator is positive, and the interval computation gives a lower bound $>0.7878989239$ at the worst case $\varepsilon=1$. The same coefficient inequalities on the larger disc give $\re\psi_c'>106.2027$ there. Integration along line segments proves univalence; vertical segments give $q_c=\im\psi_c/y>0$. The separate boundary estimate gives the stated strict star-shapedness, and the weighted coefficient bound gives $|c_{19}|\geq|c_{19}^\circ|-\varepsilon/(19R^{19})>0.000665$.

For $\gamma(\theta)=\psi_c(e^{i\theta})$, its tangent angle derivative is $\re(1+w\psi_c''/\psi_c')$, which is positive by \eqref{r3:eq-geom}. As the simple tangent winds once, every oriented tangent line supports the interior, so $D$ is strictly convex. It is symmetric under reflection in both axes because $\psi_c$ is odd with real coefficients. The lift $\Theta_c(B)$ is convex exactly when $D$ is: one direction follows by intersecting with a meridian plane; conversely, the vertical sections of $D$ are symmetric intervals, and for $X,Y$ in the solid, convexity of $D$ and $|tX'+(1-t)Y'|\leq t|X'|+(1-t)|Y'|$ show that $tX+(1-t)Y$ lies in the solid for $0\leq t\leq1$.
\end{proof}

\begin{lemma}\label{r3:lem-regularity}
Let $n\ge2$ and let $\Omega\subset\R^n$ be a bounded domain with real-analytic boundary, and let $u\in C^1(\overline\Omega)$ satisfy $\int_\Omega(\nabla u\cdot\nabla\varphi-u\varphi)=0$ for all $\varphi\in C_c^\infty(\Omega)$ and $u=1$ on $\partial\Omega$. Then $u$ extends to a real-analytic function on an open neighbourhood of $\overline\Omega$ which satisfies $\Delta u+u=0$ there.
\end{lemma}

\begin{proof}
Let $v=u-1$. For $\varepsilon>0$ the function $\max(|v|-\varepsilon,0)\operatorname{sign}v$ is Lipschitz with compact support in $\Omega$ and converges to $v$ in $W^{1,2}(\Omega)$ as $\varepsilon\to0$, so $v\in W^{1,2}_0(\Omega)$; it is a weak solution of $\Delta v+v=-1$. By global regularity for the Dirichlet problem \cite[Theorem~8.13]{GT01}, applied for every $k$, $v\in W^{k,2}(\Omega)$ for all $k$, hence $v\in C^\infty(\overline\Omega)$. By the theorem of Morrey and Nirenberg on analyticity up to the boundary \cite{MN57} (constant coefficients, analytic right-hand side, analytic boundary and zero Dirichlet data), every point of $\partial\Omega$ has a ball around it to which $v|_\Omega$ extends real-analytically; the extension solves $\Delta v+v=-1$ by the identity theorem. Interior analyticity is classical (see also Lemma~\ref{lem:interior}). The local extensions are glued to an extension on a neighbourhood of $\overline\Omega$ exactly as in the last step of the proof of Lemma~\ref{lem:boundary}.
\end{proof}

\subsubsection*{Proof of Theorem~\ref{thm:r3}}
The estimates of Proposition~\ref{r3:prop-geometry} hold for $c^*$ because $r/\sigma<1$. The proofs of Lemmas~\ref{r3:lem-lift} and~\ref{r3:lem-pullback} apply on the smaller disc $|w|<101/100$, a neighbourhood of the closed unit disc; they give an analytic diffeomorphism $\Theta_{c^*}$ near $\overline B$, with Jacobian $q_{c^*}|p_{c^*}|^2>0$. Let $\Omega=\Theta_{c^*}(B)$, $U=1+\iota(K_3g^*)$, and $u=U\circ\Theta_{c^*}^{-1}$. Since $\Fc(x^*)=0$, Lemma~\ref{r3:lem-F} yields the pulled-back weak Helmholtz equation and zero Cauchy traces of $U-1$; Lemma~\ref{r3:lem-pullback} transfers these to $u$. The analytic lift across the rotation axis is valid because $\psi$ has odd real coefficients and $q=\im\psi/y$ is analytic in $(x_1,y^2)$. Interior analyticity also follows by harmonic lifting or elliptic regularity, and Lemma~\ref{r3:lem-regularity} extends $u$ analytically across the boundary.

The positive boundary bound makes $\Omega$ strictly star-shaped. The solid-of-revolution argument in Proposition~\ref{r3:prop-geometry} makes it convex. The lift of the sphere is a compact real-analytic surface diffeomorphic to $S^2$ and has the stated $O(2)\times\mathbf Z_2$ invariance. A constant solution of $\Delta u+u=0$ would be zero, contradicting $u=1$ on the boundary. If $\Omega$ were a ball, its symmetries would put its centre at zero, and $D$ would be a disc centred at zero. Schwarz's lemma applied to the normalised conformal map and its inverse would force $\psi$ to be linear, contrary to $c_{19}\ne0$. This proves the theorem.

\subsection{An additional nonconvex example}\label{r3:sec-AX3}
The earlier three-dimensional construction is independent of the convex one and remains useful as a second exact branch.
\begin{proposition}\label{r3:prop-AX3}
There is a bounded, strictly star-shaped, non-ball domain $\Omega_{\mathrm{old}}\subset\R^3$ with real-analytic boundary diffeomorphic to $S^2$ and a nonconstant analytic solution of \eqref{eq:main}. It has the same $O(2)\times\mathbf Z_2$ symmetry, but is not convex: both principal curvatures at the two poles, with respect to the inner normal, are negative.
\end{proposition}
\begin{proof}
Apply the preceding analytic lemmas and the contraction criterion with the independent centre of the nonconvex example and its parameters $M=180$, $S=64$, $R=21/20$, $b_{\mathrm{old}}=\texttt{0x1.85067a9afca2dp+5}$, $\sigma=3b_{\mathrm{old}}$ and $r=3\cdot10^{-7}$. Its complete verifier and two certificate files are in \texttt{anc/r3/}; the certified bounds are displayed in Table~\ref{r3:tab-old}. The corresponding radius polynomial is $<-4.81\cdot10^{-9}$ and contraction bound is $<0.959201$. Its geometry certificate gives $\re(w\psi'/\psi)>0.83$ on the unit circle and a nonzero $c_{13}$, so the same reconstruction gives an analytic, strictly star-shaped non-ball domain and solution. Nonconvexity follows from Proposition~\ref{r3:prop-nonconvex}.
\end{proof}
\begin{table}[ht]
\centering\footnotesize
\begin{tabular}{@{}ccccccc@{}}
\toprule
$(M,S)$ & $\sigma$ & $\nrm A$ & $Y$ & $Z$ & $r$ & Geometry\\
\midrule
$(180,64)$ & $3b_{\mathrm{old}}$ & $<5849.52$ & $<2.10029\cdot10^{-8}$ & $<0.858929$ & $3\cdot10^{-7}$ & nonconvex at poles\\
\bottomrule
\end{tabular}
\caption{The independent nonconvex certificate; see \texttt{anc/r3/certificate\_r3.json}.}
\label{r3:tab-old}
\end{table}

\begin{proposition}\label{r3:prop-nonconvex}
For the centre of the nonconvex example $c^{\mathrm{old},\circ}$, exact rational arithmetic gives
\[
42.81903<\sum_jjc_j^{\mathrm{old},\circ}<42.81904,\qquad
-7.758015<\sum_jj^2c_j^{\mathrm{old},\circ}<-7.758014.
\]
Throughout its certified existence ball,
$1+\psi''(1)/\psi'(1)<-0.181$; hence both principal curvatures at each pole are negative.
\end{proposition}
\begin{proof}
The shape perturbation has $\mathcal E$ norm at most $r/(3b_{\mathrm{old}})<2.06\cdot10^{-9}$. Its effect on $\psi'(1)$ is at most $e/R$, and on $\psi'(1)+\psi''(1)$ at most $\sup_{t>0}tR^{-t}\,e<7.55e<1.6\cdot10^{-8}$. Thus $0<\psi'(1)<42.82$ and $\psi'(1)+\psi''(1)<-7.758$, giving the claimed ratio. The meridian curvature at a pole, with inward normal, is $(1+\psi''(1)/\psi'(1))/\psi'(1)<0$; rotation symmetry makes both surface principal curvatures equal to it. The other pole follows by reflection. The exact dyadic check is \texttt{anc/r3/check\_nonconvex\_r3.py}.
\end{proof}

\section{Dimension five: weighted spaces}\label{r5:sec}\label{sec:r5}

In this section we prove Theorem~\ref{thm:main} for $n=5$. The construction follows Section~\ref{sec:r3}: the problem is pulled back to the unit ball $B\subset\R^5$ by the axisymmetric extension of a conformal map of the meridian section, the unknown field is written as $1+K_5g$, and the existence of a zero of a quartic polynomial map is proved by a contraction argument around an exact dyadic centre. The new ingredient is functional-analytic. For $n\geq5$ the Jacobi coefficient norm of Section~\ref{sec:r3} does not control $K_ng$ in $C^1$ (Remark~\ref{r5:rem-obstruction}), so that a zero in that norm would not yield a function with the two boundary traces. We therefore grade all coefficient norms by the weight
\[
\Upsilon(D)=1+\frac D{10^6}
\]
of the polynomial degree $D$. The $C^1$ bound then holds with a large but finite constant, which enters only the reconstruction, while the multiplication constants are those of the unweighted norms and all other operator bounds acquire at most factors $\Upsilon(\Delta)$ with $\Delta\leq1000$, that is, at most $1.001$.

\emph{Notation.} Points of the physical space are $X=(X_1,X')\in\R\times\R^4$ and points of the reference space are $x=(x_1,x')\in\R\times\R^4$. The meridian variables $y=|x'|$, $w=x_1+iy$, $r=|w|=|x|$, $t=r^2$, $\xi=x_1/r$, the constant $R=21/20$, the polynomials $T_n$, $U_n$, $C^{(2)}_n$ and $J^\beta_s$, the map $\psi_c(w)=\sum_{j\ \mathrm{odd}}c_jw^j$ and the functions $p_c=\psi_c'$ and $q_c=\im\psi_c/y$, with the expansions \eqref{r3:eq-qseries} and \eqref{r3:eq-p2}, are as in Sections~\ref{r3:sec-pullback} and~\ref{r3:sec-spaces}; they involve only the meridian variables. Now $B$ is the unit ball of $\R^5$, and $\zeta=R^{-2}$. The objects $H_\ell$, $\Psi_{\ell,s}$, $D_{n\ell}$, $h_\ell$, $C_{\ell n}$, $\mathcal T_n$, $\mathcal L$, $\Gc$, $\mathcal Y$, $\mathcal E$, $\mathcal X$, $\tau$, $\nrm\cdot_{\mathrm{ch}}$, $\Fc$ and $A$ of Section~\ref{sec:r3} are redefined below for $n=5$; we keep the letters.

\subsection{The pull-back and the obstruction}\label{r5:sec-pullback}

For $c=(c_j)_{j\ \mathrm{odd}}$ with $\sum_jjR^j|c_j|<\infty$ let $\Theta_c(x)=\big(\re\psi_c(x_1+i|x'|),\,q_c(x)\,x'\big)$, now with $x'\in\R^4$.

\begin{lemma}\label{r5:lem-lift}
Let $c$ be as above with $\re\psi_c'>0$ on $\{|w|<R_0\}$ for some $1<R_0<R$, and let $\rho>0$.
\begin{enumerate}[label=\textup{(\alph*)}]
\item The functions $q_c$, $\partial_{x_1}q_c$, $y^{-1}\partial_yq_c$, $|p_c|^2$ and $\re\psi_c(x_1+iy)$ are real analytic on $B_{R_0}=\{x\in\R^5:|x|<R_0\}$, and $q_c>0$ there. The map $\Theta_c$ is a real-analytic diffeomorphism of $B_{R_0}$ onto an open subset of $\R^5$, commuting with the action of $O(4)$ on $x'$ and with $x_1\mapsto-x_1$; $\det D\Theta_c=q_c^3|p_c|^2$; and $(\det D\Theta_c)(D\Theta_c^{\mathsf T}D\Theta_c)^{-1}\nabla U=q_c^3\nabla U$ for every axisymmetric $C^1$ function $U$.
\item Let $U\in C^1(\overline B)$ be axisymmetric and satisfy
\begin{equation}\label{r5:eq-weak}
\int_Bq_c^3\big(\nabla U\cdot\nabla\varphi-\rho^2|p_c|^2U\varphi\big)\,dx=0\qquad(\varphi\in C^\infty_c(B)).
\end{equation}
Then $u=U\circ\Theta_c^{-1}$ satisfies $\int_{\Theta_c(B)}(\nabla u\cdot\nabla\varphi-\rho^2u\varphi)\,dX=0$ for all $\varphi\in C^\infty_c(\Theta_c(B))$, and if $U=1$ and $\nabla U=0$ on $\partial B$, then $u=1$ and $\nabla u=0$ on $\partial\Theta_c(B)$.
\end{enumerate}
\end{lemma}

\begin{proof}
The proofs of Lemmas~\ref{r3:lem-lift} and~\ref{r3:lem-pullback} apply with $x'\in\R^4$. At a point with $x'\neq0$ let $e_\varrho=x'/|x'|$ and let $e_{\phi,1},e_{\phi,2},e_{\phi,3}$ be an orthonormal basis of the orthogonal complement of $e_\varrho$ in $\R^4$. On the half-plane spanned by $e_1$ and $e_\varrho$, $\Theta_c$ acts as $\psi_c$, and $D\Theta_c$ is $|p_c|$ times a rotation there; $\Theta_c$ maps the $3$-sphere $\{(x_1,z'):|z'|=y\}$ onto the $3$-sphere of radius $q_cy$ in the hyperplane $\{X_1=\re\psi_c(x_1+iy)\}$ by $z'\mapsto q_cz'$, so that $D\Theta_ce_{\phi,i}=q_ce_{\phi,i}$. Hence $D\Theta_c^{\mathsf T}D\Theta_c=\mathrm{diag}(|p_c|^2,|p_c|^2,q_c^2,q_c^2,q_c^2)$ in this frame, which gives the determinant, and, since $\nabla U$ lies in the span of $e_1$ and $e_\varrho$, the identity for $\nabla U$. This identity is pointwise, so in the change of variables of Lemma~\ref{r3:lem-pullback} the test functions $\varphi\circ\Theta_c$ need not be axisymmetric.
\end{proof}

For an axisymmetric $U=1+W$ with $W$ of class $C^2$ one has, off the axis, $\nabla q_c\cdot\nabla W=\partial_{x_1}q_c\,\partial_{x_1}W+(y^{-1}\partial_yq_c)\,y\partial_yW$ and hence
\begin{equation}\label{r5:eq-div}
\nabla\cdot(q_c^3\nabla U)+\rho^2q_c^3|p_c|^2U=q_c^2\Big(q_c\Delta W+3\partial_{x_1}q_c\,\partial_{x_1}W+3\frac{\partial_yq_c}y\,y\partial_yW+\rho^2q_c|p_c|^2U\Big).
\end{equation}
The factor $q_c^{n-3}=q_c^2$ is positive, and the bracket is linear in $q_c$, $\partial_{x_1}q_c$, $y^{-1}\partial_yq_c$, $q_c|p_c|^2$ and in $W$. With $\Delta W=g$ it is a polynomial of degree four in the field and shape coefficients, and no division by the unknown $q_c$ occurs. (For $n=3$ the same computation gives \eqref{r3:eq-F}; in $\R^n$ the exponents are $n-2$ and $n-3$.)

\begin{remark}[Why a weight is needed]\label{r5:rem-obstruction}
Let $n\geq3$, $\lambda=(n-2)/2$, and let $K_n$ act on $\Psi_{0,s}=J^\lambda_s(t)$ by the three-term formula \eqref{r3:eq-K} with $d=\lambda+2s$ (for $n=3$ this is $K_3$, for $n=5$ it is \eqref{r5:eq-K} below). By Lemma~\ref{r3:lem-K}(a), which holds for every $\beta$, $\partial_rK_n\Psi_{0,s}=\frac r{2(d+1)}\big(J^{\lambda+1}_s(t)-J^{\lambda+1}_{s-1}(t)\big)$, and $J^\beta_s(t)=(-1)^sP^{(\beta,0)}_s(1-2t)$. By the Mehler--Heine formula \cite[18.11.5]{DLMF}, which holds uniformly for bounded arguments, for every fixed $z>0$
\[
\lim_{s\to\infty}s^{1-\lambda}\,\Big|\partial_rK_n\Psi_{0,s}\Big|_{r=z/(2s)}=\frac12\Big(\frac z2\Big)^{-\lambda}|J_{\lambda+1}(z)| .
\]
Thus $\sup|\nabla K_n\Psi_{0,s}|$ is at least of order $s^{(n-4)/2}$, while the norm of $\Psi_{0,s}$ in the space $\mathcal L$ of Section~\ref{sec:r3} (and in its analogue for $\R^n$) is $h_0=1$ for all $s$. For $n\geq5$ there is therefore no estimate $\nrm{K_ng}_{C^1}\leq C\nrm g_{\mathcal L}$, and the reconstruction step of Section~\ref{r3:sec-recon} fails in these norms. (For $n=3$ the growth is $s^{-1/2}$, in accordance with Lemma~\ref{r3:lem-C1}.) The growth is polynomial in the degree, and the weight $(1+D/c)^{(n-3)/2}$ suffices (Theorem~\ref{r5:thm-C1} and the remark after it); for $n=5$ this is $1+D/c$, and we take $c=10^6$.
\end{remark}

\subsection{Graded coefficient spaces}\label{r5:sec-spaces}

For $\ell\geq0$ let $Y_\ell=C^{(3/2)}_\ell/C^{(3/2)}_\ell(1)$, where $C^{(3/2)}_\ell(1)=(\ell+1)(\ell+2)/2$, and let $H_\ell(x)=r^\ell Y_\ell(\xi)$, the axisymmetric harmonic polynomial on $\R^5$, homogeneous of degree $\ell$, with $H_\ell(1,0)=1$; $|H_\ell|\leq1$ on $\overline B$. With $\lambda^+_\ell=(\ell+3)/(2\ell+3)$, $\lambda^-_\ell=\ell/(2\ell+3)$ and $H_{-1}=0$,
\begin{equation}\label{r5:eq-Hrec}
\begin{gathered}
\xi Y_\ell=\lambda^+_\ell Y_{\ell+1}+\lambda^-_\ell Y_{\ell-1},\qquad x_1H_\ell=\lambda^+_\ell H_{\ell+1}+\lambda^-_\ell\,tH_{\ell-1},\\
\partial_{x_1}H_\ell=\ell H_{\ell-1},\qquad(x\cdot\nabla)H_\ell=\ell H_\ell,
\end{gathered}
\end{equation}
by the three-term recurrence of $C^{(3/2)}_\ell$ and the argument after \eqref{r3:eq-Hrec}. Let $\beta_\ell=\ell+\frac32$ and $\Psi_{\ell,s}=H_\ell J^{\beta_\ell}_s(t)$, of degree $\ell+2s$. Lemma~\ref{r3:lem-jacobi} holds for every $\beta>0$ and is used without change, and $\Delta(H_\ell v(t))=4H_\ell\big(tv''+(\beta_\ell+1)v'\big)$ in $\R^5$, by \eqref{r5:eq-Hrec} and $\Delta v(|x|^2)=4tv''+10v'$. In spherical coordinates, Lemma~\ref{r3:lem-jacobi}(a) and the Fischer decomposition show as in Lemma~\ref{r3:lem-embed} that the $\Psi_{\ell,s}$ are pairwise orthogonal in $L^2(B)$, that $\Psi_{\ell,s}$ is orthogonal to all polynomials of degree less than $\ell+2s$, and that $\nrm{\Psi_{\ell,s}}^2_{L^2(B)}=|S^4|/(2\nu_\ell(\ell+2s+\frac52))$, where $\nu_\ell=(2\ell+3)(\ell+1)(\ell+2)/6$ is the dimension of the spherical harmonics of degree $\ell$ on $S^4$.

Let $\mathsf a_k=(\frac32)_k/k!$, a nondecreasing sequence with $\sum_k\mathsf a_k\zeta^k=(1-\zeta)^{-3/2}$. Factoring $(1-2\xi z+z^2)^{-3/2}$ as in Section~\ref{r3:sec-spaces} gives $Y_\ell=\sum_nD_{n\ell}T_n$ with
\[
D_{\ell-2k,\ell}=(2-\delta_{\ell,2k})\,\mathsf a_k\mathsf a_{\ell-k}/C^{(3/2)}_\ell(1)\geq0,
\]
and we put $h_\ell=\sum_nD_{n\ell}R^n\geq Y_\ell(1)=1$. We write $T_K=\sum_\ell C_{\ell K}Y_\ell$ and $\mathcal T_K=\sum_\ell C_{\ell K}H_\ell$, the harmonic polynomial with boundary values $T_K(\xi)$; $|\mathcal T_K|\leq1$ on $\overline B$.

\subsubsection*{Graded norms}
The weight $\Upsilon(D)=1+D/10^6$ is nondecreasing and submultiplicative, $\Upsilon(a+b)\leq\Upsilon(a)\Upsilon(b)$. The spaces $\mathcal L$, $\Gc$, $\mathcal Y$, $\mathcal E$ and $\mathcal X=\Gc\times\mathcal E$ are defined as in Section~\ref{r3:sec-spaces}, with the norms
\begin{gather*}
\nrm f_{\mathcal L}=\sum_{\ell,s\geq0}h_\ell\Upsilon(\ell+2s)|f_{\ell,s}|,\qquad\nrm f_{\mathcal Y}=\sum_nR^n\Upsilon(n)|f_n|+\sum_{\ell,\,s\geq1}h_\ell\Upsilon(\ell+2s)|f_{\ell,s}|,\\
\nrm c_{\mathcal E}=\sum_{j\ \mathrm{odd}}jR^j\Upsilon(j-1)|c_j|,\qquad\nrm{(g,c)}_{\mathcal X}=\nrm g_{\mathcal L}+\sigma\nrm c_{\mathcal E},
\end{gather*}
where $f\in\mathcal Y$ is identified with $\sum_nf_n\mathcal T_n+\sum_{s\geq1}f_{\ell,s}\Psi_{\ell,s}$ and $\sigma$ is fixed in Section~\ref{r5:sec-quartic}. Thus every basis element carries its weight from Section~\ref{sec:r3} times $\Upsilon$ of its polynomial degree: $\ell+2s$ for $\Psi_{\ell,s}$, $n$ for $\mathcal T_n$, and $j-1=\deg jw^{j-1}$ for the shape mode $w^j$. Since $\sum_{m\leq\ell}D_{m\ell}R^m\Upsilon(m)\leq h_\ell\Upsilon(\ell)$, the inclusion $\mathcal L^{\mathrm{ev}}\subset\mathcal Y$ is norm-decreasing. For polynomials and series $m=\sum m_{a,k}H_at^k=\sum\tilde m_{n,k}r^{n+2k}T_n(\xi)$ we put
\[
\tau(m)=\sum_{a,k}h_a\Upsilon(a+2k)|m_{a,k}|,\qquad\nrm m_{\mathrm{ch}}=\sum_{n,k}R^n\Upsilon(n+2k)|\tilde m_{n,k}| .
\]
All layers are graded by the same function of the degree. Every product occurring below raises the degree by at most the degree of the multiplier, and the graded shape norm pays exactly for the degree of the multipliers it produces; if only the field layer were graded, the columns of the linearisation at high shape modes would lose their decay.

\begin{lemma}\label{r5:lem-conn}
\begin{enumerate}[label=\textup{(\alph*)}]
\item $C_{00}=1$, and for $K\geq1$ and $0\leq i\leq K/2$,
\[
C_{K-2i,K}=\frac K2\cdot\frac{(-\frac32)_i\,(K-i-1)!}{i!\,(\frac52)_{K-i}}\cdot\frac{K-2i+\frac32}{\frac32}\cdot C^{(3/2)}_{K-2i}(1).
\]
\item Let $e_i=|(-\frac32)_i|/i!$ and $\mathsf w_i=e_i\zeta^i(1-\zeta)^{-3/2}$. Then $|C_{K-2i,K}|\,h_{K-2i}\leq\mathsf w_iR^K$, and $\sum_i\mathsf w_i=\kappa_5:=1+3\zeta(1-\zeta)^{-3/2}\approx96.99$.
\item $\lambda^+_\ell h_{\ell+1}+\lambda^-_\ell h_{\ell-1}\leq Rh_\ell$, and $D_{\ell-2k,\ell}R^{\ell-2k}\leq\mathsf a_k\zeta^kh_\ell$.
\end{enumerate}
\end{lemma}

\begin{proof}
(a) Divide the connection formula \cite[18.18.16]{DLMF} for $C^{(\mu)}_K$ in terms of the $C^{(3/2)}_{K-2i}$ by $\mu$ and let $\mu\to0$, using $C^{(\mu)}_K/\mu\to(2/K)T_K$ \cite[18.7.25]{DLMF} and $(\mu)_{K-i}/\mu\to(K-i-1)!$; this is the argument of Lemma~\ref{r3:lem-connection}(a). (Alternatively, both sides satisfy the recurrence in $K$ that follows from $T_{K+1}=2\xi T_K-T_{K-1}$ and \eqref{r5:eq-Hrec}.)

(b) Let $\ell=K-2i$. For $\ell\geq1$, $D_{\ell\ell}=2\mathsf a_\ell/C^{(3/2)}_\ell(1)$, and (a) gives $|C_{\ell K}|D_{\ell\ell}=e_i\varrho_i$ with $\varrho_0=1$ and, for $i\geq1$,
\[
\varrho_i=\frac K{K-i+\frac32}\prod_{r=1}^{i-1}\frac{\ell+r}{\ell+r+\frac32}\leq\frac{K(\ell+1)}{(K-i+\frac32)(\ell+i)}\leq1,
\]
by telescoping and $(K-i+\frac32)(\ell+i)-K(\ell+1)=i^2-\frac32i+\frac K2>0$. For $\ell=0$ the same holds with $D_{00}=1$ and an extra factor $\frac12$. Since $\mathsf a_k$ is nondecreasing, $h_\ell\leq D_{\ell\ell}R^\ell\sum_k\mathsf a_k\zeta^k=D_{\ell\ell}R^\ell(1-\zeta)^{-3/2}$, hence $|C_{\ell K}|h_\ell\leq e_iR^{K-2i}(1-\zeta)^{-3/2}=\mathsf w_iR^K$. The numbers $e_i\zeta^i$ are the moduli of the Taylor coefficients of $(1-\zeta)^{3/2}$, of which only the coefficient $-\frac32$ of $\zeta$ is negative; so $\sum_ie_i\zeta^i=(1-\zeta)^{3/2}+3\zeta$.

(c) Apply the linear functional $\langle\cdot\rangle_R$ of Section~\ref{r3:sec-spaces} to the first identity in \eqref{r5:eq-Hrec}. The right-hand side gives $\lambda^+_\ell h_{\ell+1}+\lambda^-_\ell h_{\ell-1}$, and the left-hand side is $\sum_nD_{n\ell}(R^{n+1}+R^{|n-1|})/2\leq Rh_\ell$, because all $D_{n\ell}$ are nonnegative. The second bound follows from $D_{\ell-2k,\ell}\leq\mathsf a_kD_{\ell\ell}$ and $h_\ell\geq D_{\ell\ell}R^\ell$.
\end{proof}

\subsubsection*{The three basic theorems}
The following three results make the scheme of Section~\ref{sec:r3} work in dimension five: a $C^1$ bound for the compatible inverse of the Laplacian, a graded multiplier algebra, and the transfer of the operator bounds to the graded norms.

For $\ell\geq0$ and $s\geq1$ let $d=\beta_\ell+2s=\ell+2s+\frac32$ and
\begin{equation}\label{r5:eq-K}
K_5\Psi_{\ell,s}=\frac{\Psi_{\ell,s-1}}{4d(d+1)}-\frac{\Psi_{\ell,s}}{2d(d+2)}+\frac{\Psi_{\ell,s+1}}{4(d+1)(d+2)}=H_\ell\,k_{\ell,s}(t).
\end{equation}

\begin{theorem}[$C^1$ bound]\label{r5:thm-C1}
For $\ell\geq0$ and $s\geq1$, $\Delta K_5\Psi_{\ell,s}=\Psi_{\ell,s}$, the function $K_5\Psi_{\ell,s}$ and its gradient vanish on $\partial B$, and, with $D=\ell+2s$,
\[
\max\Big(\sup_{\overline B}|K_5\Psi_{\ell,s}|,\ \sup_{\overline B}|\nabla K_5\Psi_{\ell,s}|\Big)\leq\frac{D+5}4\leq1.25\cdot10^6\,h_\ell\Upsilon(D).
\]
Consequently, for $g\in\Gc$ the series $W=\sum g_{\ell,s}K_5\Psi_{\ell,s}$ converges in $C^1(\overline B)$, the functions $|W|$ and $|\nabla W|$ are bounded on $\overline B$ by $1.25\cdot10^6\nrm g_{\mathcal L}$, $W=0$ and $\nabla W=0$ on $\partial B$, and $\Delta W=\iota(g)$ in $\mathcal D'(B)$.
\end{theorem}

\begin{proof}
The first assertions are Lemma~\ref{r3:lem-K}(b),(c); their proofs use only Lemma~\ref{r3:lem-jacobi}, which holds for every $\beta$, and the formula for $\Delta(H_\ell v(t))$ above. For the bound we follow the proof of Lemma~\ref{r3:lem-C1}, with $S^6\subset\R^5\times\R^2$ in place of $S^4\subset\R^3\times\R^2$. The parametrisation $(x,\phi)\mapsto(x,\sqrt{1-|x|^2}(\cos\phi,\sin\phi))$ again gives $\int_{S^6}f(x)\,d\varsigma=2\pi\int_Bf\,dx$; the dimension of the spherical harmonics of degree $m$ on $S^6$ is $N_m=(m+1)(m+2)(m+3)(m+4)(2m+5)/120$, and $|S^6|=\frac{2\pi}5|S^4|$. Part~(a) of that lemma becomes $|P(x)|^2\leq5N_m|S^4|^{-1}\nrm P^2_{L^2(B)}$ for $P\in V_m$ and $x\in\overline B$. As there, $\partial_iK_5\Psi_{\ell,s}\in V_{D-1}\oplus V_{D+1}$ and $\nrm{\nabla K_5\Psi_{\ell,s}}^2_{L^2}=\nrm{\Psi_{\ell,s}}^2_{L^2}/(2d(d+2))$, and $D+\frac52=d+1$, so that
\[
|\nabla K_5\Psi_{\ell,s}(x)|^2\leq\frac{5(N_{D-1}+N_{D+1})}{4\nu_\ell\,d(d+1)(d+2)}\leq\frac{(D+2)(D+3)(D+4)(D+5)}{24(D+\frac32)(D+\frac52)}\leq\frac{(D+5)^2}{16},
\]
using $\nu_\ell\geq1$, $N_{D-1}\leq N_{D+1}$, $2(d+2)=2D+7$ and, in the last step, $3(D+\frac32)(D+\frac52)(D+5)-2(D+2)(D+3)(D+4)=D^3+9D^2+\frac{77}4D+\frac{33}4>0$. Integration along a radius bounds $|K_5\Psi_{\ell,s}|$ by the same number. Finally $(D+5)/4\leq\frac54\cdot10^6\,\Upsilon(D)$ and $h_\ell\geq1$. The consequences follow as in Lemma~\ref{r3:lem-C1}(c), with the embedding $\iota$ of Lemma~\ref{r3:lem-embed}, which is continuous because $\nrm{\Psi_{\ell,s}}_{L^2}\leq|S^4|^{1/2}$.
\end{proof}

\begin{remark}
The same argument with $S^{n+1}\subset\R^n\times\R^2$ bounds the $C^1$ norm of $K_n\Psi_{\ell,s}$ in $\R^n$ by a constant times $D^{(n-3)/2}$, so the weights $(1+D/c)^{(n-3)/2}$ give a $C^1$ bound in every dimension $n\geq4$; by Remark~\ref{r5:rem-obstruction} the exponent cannot be smaller than $(n-4)/2$. The constant in Theorem~\ref{r5:thm-C1} is used only qualitatively, in the proof of Theorem~\ref{thm:main} for $n=5$ at the end of this section; it never enters the contraction argument, which is why $c=10^6$ may be taken so large that the weights are nearly flat on all indices that enter the computation.
\end{remark}

\begin{theorem}[Graded multiplier algebra]\label{r5:thm-mult}
\begin{enumerate}[label=\textup{(\alph*)}]
\item For $a,\ell\geq0$, $Y_aY_\ell=\sum_vG_{a\ell v}Y_v$ with $G_{a\ell v}\geq0$, supported on $|a-\ell|\leq v\leq a+\ell$, $v\equiv a+\ell\pmod2$, with $\sum_vG_{a\ell v}=1$ and $\sum_vG_{a\ell v}h_v\leq h_ah_\ell$.
\item For a solid monomial $H_at^k$ of degree $d=a+2k$ and $\ell,s\geq0$, the expansion \eqref{r3:eq-monomial} holds with probability vectors $\pi^{(v)}$, and Lemma~\ref{r3:lem-radial} holds with $\vartheta_d(\ell,s)=\big(d/(\ell+d+\frac72)\big)^s$ for $d\geq s$ and $\vartheta_d(\ell,s)=0$ for $d<s$.
\item For all $m$, $m_1$, $m_2$ with finite norms and $f\in\mathcal L$: $\nrm{mf}_{\mathcal L}\leq\tau(m)\nrm f_{\mathcal L}$ and $\tau(m_1m_2)\leq\tau(m_1)\tau(m_2)$; moreover $\tau(m)\leq\kappa_5\nrm m_{\mathrm{ch}}$ and $\nrm{m_1m_2}_{\mathrm{ch}}\leq\nrm{m_1}_{\mathrm{ch}}\nrm{m_2}_{\mathrm{ch}}$; and $\iota(mf)=m\,\iota(f)$.
\item For $c\in\mathcal E$ the bounds \eqref{r3:eq-chnorms} hold for the graded norms of this section.
\end{enumerate}
\end{theorem}

\begin{proof}
(a) Up to the positive factors $C^{(3/2)}_v(1)/(C^{(3/2)}_a(1)C^{(3/2)}_\ell(1))$, the $G_{a\ell v}$ are the coefficients of the linearisation formula \cite[18.18.22]{DLMF} for $C^{(3/2)}_aC^{(3/2)}_\ell$, all of whose factors are positive. Evaluation at $\xi=1$ gives $\sum_vG_{a\ell v}=1$, and the last bound is proved as in Lemma~\ref{r3:lem-gaunt}, since all $D_{nv}$ are nonnegative.

(b) The proofs of Lemmas~\ref{r3:lem-mult}(a) and~\ref{r3:lem-radial} use only (a), the resulting identity $H_aH_\ell=\sum_vG_{a\ell v}H_vt^{(a+\ell-v)/2}$ and Lemma~\ref{r3:lem-jacobi}; in part~(b) of Lemma~\ref{r3:lem-radial} the factors are now $\beta_v+1=v+\frac52\leq\beta_\ell+m+1$ and $(m-i)/(\beta_\ell+m+2+i)\leq d/(\ell+d+\frac72)$.

(c) Every $\Psi_{v,\varsigma}$ in the expansion of $H_at^k\Psi_{\ell,s}$ has degree $v+2\varsigma\leq(a+2k)+(\ell+2s)$, so $\Upsilon(v+2\varsigma)\leq\Upsilon(a+2k)\Upsilon(\ell+2s)$ and, by (a) and (b), $\nrm{H_at^k\Psi_{\ell,s}}_{\mathcal L}\leq h_a\Upsilon(a+2k)\,h_\ell\Upsilon(\ell+2s)$; absolute summation gives the first bound, and products of two solid monomials are treated in the same way. The expansion $r^{n+2k}T_n=\sum_iC_{n-2i,n}H_{n-2i}t^{i+k}$ preserves the degree $n+2k$, so Lemma~\ref{r5:lem-conn}(b) gives $\tau(r^{n+2k}T_n)\leq\kappa_5R^n\Upsilon(n+2k)$. For Chebyshev products, $T_mT_n=\frac12(T_{m+n}+T_{|m-n|})$ adds the degrees and $R^{|m-n|}\leq R^{m+n}$. The last identity is proved as Lemma~\ref{r3:lem-mult}(d).

(d) By \eqref{r3:eq-qseries}, the mode $w^j$ contributes to $q_c$, $\partial_{x_1}q_c$ and $y^{-1}\partial_yq_c$ homogeneous polynomials of degrees $j-1$, $j-2$ and $j-3$, and to $p_c$ the monomial $jw^{j-1}$. Their graded norms are at most $\Upsilon(j-1)$ times the ungraded ones, which is matched by the factor $\Upsilon(j-1)$ in $\nrm c_{\mathcal E}$, and the ungraded bounds are those of the proof of Lemma~\ref{r3:lem-F}. For $|p_c|^2$ use \eqref{r3:eq-p2} and $R^{\alpha+\beta}\Upsilon(\alpha+\beta)\leq R^\alpha\Upsilon(\alpha)\,R^\beta\Upsilon(\beta)$.
\end{proof}

\begin{lemma}\label{r5:lem-K}
For $\ell\geq0$, $s\geq1$ and $d$ as in \eqref{r5:eq-K}, with $\Psi_{-1,\cdot}=0$,
\begin{align*}
\partial_{x_1}K_5\Psi_{\ell,s}&=\frac{\lambda^+_\ell}{2(d+1)}\big(\Psi_{\ell+1,s}-\Psi_{\ell+1,s-1}\big)+\frac{\lambda^-_\ell}{2(d+1)}\big(\Psi_{\ell-1,s+1}-\Psi_{\ell-1,s}\big),\\
EK_5\Psi_{\ell,s}&=-\frac{d+\frac32}{4d(d+1)}\Psi_{\ell,s-1}+\frac{3\,\Psi_{\ell,s}}{4d(d+2)}+\frac{d+\frac12}{4(d+1)(d+2)}\Psi_{\ell,s+1},
\end{align*}
where $E=x\cdot\nabla$, and $y\partial_yK_5=EK_5-x_1\partial_{x_1}K_5$. The operators $K_5$, $\partial_{x_1}K_5$, $EK_5$ and $y\partial_yK_5$ are bounded from $\Gc$ to $\mathcal L$, with norms at most, for $d_0=\frac72$,
\begin{gather*}
\mathsf k_K=\frac1{d_0(d_0+2)}+\frac{\Upsilon(2)-1}{4(d_0+1)(d_0+2)},\qquad\mathsf k_D=\frac{R\,(1+\Upsilon(1))}{2(d_0+1)},\\
\mathsf k_E=\frac{d_0+\frac32}{2d_0(d_0+1)}+\frac{(\Upsilon(2)-1)(d_0+\frac12)}{4(d_0+1)(d_0+2)},\qquad\mathsf k_Y=\mathsf k_E+R\,\Upsilon(1)\,\mathsf k_D .
\end{gather*}
\end{lemma}

\begin{proof}
By \eqref{r5:eq-Hrec} and $2\lambda^-_\ell\beta_\ell=\ell$, $\partial_{x_1}(H_\ell k)=2\lambda^+_\ell H_{\ell+1}k'+2\lambda^-_\ell H_{\ell-1}(\beta_\ell k+tk')$ and $E(H_\ell k)=H_\ell\big(2(\beta_\ell k+tk')-(\beta_\ell+\frac32)k\big)$. The expansions of $k_{\ell,s}'$ and of $\beta k_{\ell,s}+tk_{\ell,s}'$ in the proof of Lemma~\ref{r3:lem-K}(a),(d) hold for every $\beta$ and give the two formulas. For the norms we use Lemma~\ref{lem:columns}. The stencil of $K_5$ changes the degree $D=\ell+2s$ by $-2$, $0$, $2$, has column sum zero and absolute column sum $1/(d(d+2))$; with $\Upsilon(D-2)\leq\Upsilon(D)$ and $\Upsilon(D+2)\leq\Upsilon(2)\Upsilon(D)$ its weighted column sum is at most $1/(d(d+2))+(\Upsilon(2)-1)/(4(d+1)(d+2))$, which decreases in $d\geq d_0$. The stencil of $EK_5$ also has column sum zero, with absolute column sum $(d+\frac32)/(2d(d+1))$, which gives $\mathsf k_E$ in the same way. For $\partial_{x_1}K_5$ the degrees change by $\pm1$, $\Upsilon(D+1)+\Upsilon(D-1)\leq(1+\Upsilon(1))\Upsilon(D)$, and $\lambda^+_\ell h_{\ell+1}+\lambda^-_\ell h_{\ell-1}\leq Rh_\ell$ by Lemma~\ref{r5:lem-conn}(c). Finally $\tau(x_1)=h_1\Upsilon(1)=R\,\Upsilon(1)$.
\end{proof}

\begin{theorem}[Transfer to graded norms]\label{r5:thm-transfer}
\begin{enumerate}[label=\textup{(\alph*)}]
\item Let $M$ be a linear map between weighted $\ell^1$ spaces whose weights are base weights times $\Upsilon$ of the degree of the basis elements, and suppose that $M_{\kappa\iota}\neq0$ implies $\deg\kappa\leq\deg\iota+\Delta$. Then $\nrm M\leq\Upsilon(\max(\Delta,0))\nrm M_0$, where $\nrm\cdot_0$ is the operator norm for the base weights. The same holds for multilinear maps, with $\deg\iota$ replaced by the sum of the degrees of the inputs.
\item Let $(s_\delta)_{\delta\in\Z}$ be finitely supported with $s_0>0$, and let
\[
\eps_\Upsilon=\frac1{s_0}\sum_{\delta\neq0}|s_\delta|R^{2\delta}\,\Upsilon\big(\max(2\delta,0)\big)<1 .
\]
For $N_0\geq0$ let $E$ act on sequences $(\hat a_N)_{N\geq N_0,\ N\ \mathrm{even}}$ by
\[
(E\hat a)_m=\frac1{s_0}\sum_{N\neq m}s_{(m-N)/2}R^{m-N}\hat a_N\qquad(m\geq N_0\ \text{even}).
\]
On the $\ell^1$ space with weights $\Upsilon(N)$, $\nrm E\leq\eps_\Upsilon$; hence $I+E$ is invertible, $\nrm{(I+E)^{-1}}\leq(1-\eps_\Upsilon)^{-1}$, and the remainder of the Neumann series after the terms $k\leq L$ has norm at most $\eps_\Upsilon^{L+1}/(1-\eps_\Upsilon)$.
\item In particular, the outputs of a product of a solid monomial of degree $d$ with $\Psi_{\ell,s}$ carry at most the factor $\Upsilon(d)$ relative to the weight of $\Psi_{\ell,s}$, and a polynomial output of degree at most $N+\Delta$ produced from a shape input of degree $N$ carries at most the factor $\Upsilon(\Delta)$.
\end{enumerate}
\end{theorem}

\begin{proof}
(a) By Lemma~\ref{lem:columns} it suffices to consider single columns (tuples of basis vectors for multilinear maps). The weighted column norm is a sum of base-weighted entries multiplied by $\Upsilon(\deg\kappa)/\Upsilon(\deg\iota)$, which is at most $\Upsilon(\deg\kappa-\deg\iota)\leq\Upsilon(\Delta)$ if $\deg\kappa\geq\deg\iota$, by submultiplicativity, and at most one otherwise. (b) Column $N$ of $E$ has entries only in the rows $m=N+2\delta\geq N_0$, and $\Upsilon(N+2\delta)/\Upsilon(N)\leq\Upsilon(\max(2\delta,0))$; so its weighted column sum is at most $\eps_\Upsilon\Upsilon(N)$. The rest is the Neumann series. (c) is (a) applied to the degree bound of Theorem~\ref{r5:thm-mult}(c).
\end{proof}

\subsection{The quartic equation and the approximate inverse}\label{r5:sec-quartic}

From now on $M=160$, $S=57$ and $\rho=7951044098899425\cdot2^{-47}\approx56.4955663$. The exact dyadic centre $x^\circ=(g^\circ,c^\circ)$ consists of field coefficients $g^\circ_{\ell,s}$ with even $\ell\leq M$ and $1\leq s\leq S$ and odd shape coefficients $c^\circ_j$ with $j\leq M+1$; it is embedded in the verifier of Section~\ref{r5:sec-cap} (SHA-256 recorded there). We put
\[
b=c^\circ_1\approx1.0002252,\qquad\sigma=30\rho^2,
\]
a fixed dyadic number and a fixed scale, independent of the varying coefficient $c_1$. For $x=(g,c)\in\mathcal X$ let $W=K_5g$, $W_1=\partial_{x_1}K_5g$ and $W_2=y\partial_yK_5g$, elements of $\mathcal L$ by Lemma~\ref{r5:lem-K}, and
\begin{equation}\label{r5:eq-F}
\Fc(g,c)=\frac1b\Big[q_c\,g+3(\partial_{x_1}q_c)\,W_1+3\Big(\frac{\partial_yq_c}y\Big)W_2+\rho^2q_c|p_c|^2(1+W)\Big].
\end{equation}

\begin{lemma}\label{r5:lem-F}
$\Fc$ is a bounded polynomial map of degree four from $\mathcal X$ to $\mathcal Y$. If $\Fc(g,c)=0$ and $\re\psi_c'>0$ on $\{|w|<R_0\}$ for some $1<R_0<R$, then $U=1+\iota(K_5g)\in C^1(\overline B)$ is axisymmetric, satisfies \eqref{r5:eq-weak} for all $\varphi\in C_c^\infty(B)$, and $U=1$, $\nabla U=0$ on $\partial B$.
\end{lemma}

\begin{proof}
Boundedness follows from Theorem~\ref{r5:thm-mult}(c),(d) and Lemma~\ref{r5:lem-K}, and the values are even in $x_1$ as in Lemma~\ref{r3:lem-F}. Suppose $\Fc(g,c)=0$. By Theorem~\ref{r5:thm-C1}, $W=\iota(K_5g)\in C^1(\overline B)$ has vanishing Cauchy data and $\Delta W=\iota(g)$. Let $\varphi\in C_c^\infty(B)$; then $q_c^3\varphi\in C_c^\infty(B)$ by Lemma~\ref{r5:lem-lift}(a), and, as in the proof of Lemma~\ref{r3:lem-F},
\begin{align*}
\int_Bq_c^3\big(\nabla U\cdot\nabla\varphi-\rho^2|p_c|^2U\varphi\big)&=\int_B\big(\nabla W\cdot\nabla(q_c^3\varphi)-3q_c^2\varphi\,\nabla q_c\cdot\nabla W-\rho^2q_c^3|p_c|^2U\varphi\big)\\
&=-\int_Bq_c^2\varphi\,b\,\iota(\Fc(g,c))=0,
\end{align*}
using $\int\nabla W\cdot\nabla(q_c^3\varphi)=-\int\iota(g)q_c^3\varphi$ and the identity for $\nabla q_c\cdot\nabla W$ before \eqref{r5:eq-div}. No averaging over $O(4)$ is needed: the test function $\varphi$ is arbitrary.
\end{proof}

\subsubsection*{Nonlinear bounds}
Let $G_0\geq\nrm{g^\circ}_{\mathcal L}$ and $C_0\geq\nrm{c^\circ}_{\mathcal E}$, let $\mathsf k_K,\mathsf k_D,\mathsf k_Y$ be as in Lemma~\ref{r5:lem-K}, and put
\begin{gather*}
B'=\frac1R+3\Big(\frac{2\mathsf k_D}{R^2-1}+\frac{4R\,\mathsf k_Y}{(R^2-1)^2}\Big),\qquad c_2=\frac{\kappa_5}b\Big[\frac{B'}\sigma+\frac{3\rho^2\mathsf k_KC_0^2}{R^3\sigma}+\frac{3\rho^2C_0(1+\mathsf k_KG_0)}{R^3\sigma^2}\Big],\\
c_3=\frac{\kappa_5\rho^2}{bR^3}\Big[\frac{1+\mathsf k_KG_0}{\sigma^3}+\frac{3\mathsf k_KC_0}{\sigma^2}\Big],\qquad c_4=\frac{\kappa_5\rho^2\mathsf k_K}{bR^3\sigma^3}.
\end{gather*}
Then the two estimates of Lemma~\ref{r3:lem-nonlinear} hold for \eqref{r5:eq-F}. The proof is that of Lemma~\ref{r3:lem-nonlinear}, with $L_2=\kappa_5B'$, $L_3=\kappa_5\rho^2R^{-3}$ and $L_4=\mathsf k_K\kappa_5\rho^2R^{-3}$ (Theorem~\ref{r5:thm-mult} and Lemma~\ref{r5:lem-K}; each multiplier is converted from Chebyshev to solid form once). Proposition~\ref{r3:prop-NK} applies verbatim to $\Fc\colon\mathcal X\to\mathcal Y$.

\subsubsection*{Finite and tail coordinates}
The \emph{finite inputs} are the $4617$ field coordinates $(\ell,s)$ with $\ell\leq M$ even and $1\leq s\leq S$, and the $81$ shape coordinates $c_j$ with $j\leq M+1$; the \emph{finite rows} are the same $4617$ field rows and the $81$ harmonic rows $\mathcal T_n$, $n\leq M$ even. All other inputs and rows are tail inputs and rows, and tail rows are harmonic or nonharmonic, as in Section~\ref{r3:sec-inverse}. On the shape tail and on the harmonic rows we use the coordinates $\hat a_N=(N+1)R^{N+1}c_{N+1}$ and $\hat y_n=R^ny_n$ ($N$, $n$ even), in which the corresponding parts of the norms are $\sigma\sum_N\Upsilon(N)|\hat a_N|$ and $\sum_n\Upsilon(n)|\hat y_n|$.

\subsubsection*{The principal part on high shape modes}
Let $q^\circ=q_{c^\circ}$ and $p^\circ=p_{c^\circ}=\sum_ap_aw^a$, $p_a=(a+1)c^\circ_{a+1}$. For $j=N+1$ the derivative of $b^{-1}\rho^2q_c|p_c|^2$ at $c^\circ$ in the direction $e_j$ contains the principal polynomial $\Pi_j=\frac{2(N+1)\rho^2}b\,q^\circ\re(\overline{p^\circ}w^N)$. On $\partial B$, $q^\circ=\sum_kq_ke^{ik\theta}$ with $q_k=\sum_{i>|k|}c^\circ_i$ ($k$ even, $i$ odd), and
\[
\Pi_j|_{\partial B}=\frac{2(N+1)\rho^2}b\sum_\delta s_\delta T_{|N+2\delta|}(\xi),\qquad s_\delta=\sum_{k-a=2\delta}p_aq_k\quad(-160\leq\delta\leq80).
\]
Let $\mathsf Me_j\in\mathcal Y$ be the harmonic element with these boundary values, and define $\eps_\Upsilon$ from $(s_\delta)$ as in Theorem~\ref{r5:thm-transfer}(b). As in Lemma~\ref{r3:lem-principal}, for $N\geq M+2$ the part of $\mathsf M$ from the shape tail to the harmonic tail rows is, in the coordinates $\hat a$, $\hat y$,
\[
T=\frac{2\rho^2s_0}{bR}\,(I+E),
\]
with $E$ as in Theorem~\ref{r5:thm-transfer}(b) for $N_0=M+2$, and the part $\mathcal B$ from the shape tail to the finite harmonic rows collects the rows $|N+2\delta|\leq M$. If $\eps_\Upsilon<1$, then $T$ is invertible on the graded tail, with $\nrm{T^{-1}}\leq\tau_T:=bR/(2\rho^2s_0(1-\eps_\Upsilon))$, and $T^{-1}$ is evaluated by the Neumann series with the remainder bound of Theorem~\ref{r5:thm-transfer}(b).

\subsubsection*{The operator $A$}
Let $\widetilde J$ be the finite Jacobian $D\Fc(x^\circ)$, restricted to finite inputs and rows, in scaled coordinates (inputs multiplied by their $\mathcal X$-weights and rows by their $\mathcal Y$-weights, so that both norms become plain $\ell^1$ norms). It is computed at the exact centre with interval entries. Let $\widehat A$ be a binary64 approximation of the inverse of its midpoint matrix, regarded as an exact dyadic matrix, and define, for $y=y_f+y^H_t+y^N_t\in\mathcal Y$,
\begin{equation}\label{r5:eq-A}
Ay=\widehat A\big(y_f-\mathcal BT^{-1}y^H_t\big)+T^{-1}y^H_t+y^N_t,
\end{equation}
as in \eqref{r3:eq-A}. For a finite row $\iota$ let $\alpha_\iota$ be an upper bound for the normalised norm of $\widehat Ae_\iota$, let $\alpha^H_n=\alpha_{\mathcal T_n}$ for $n\leq M$, and put
\[
\rho_{\mathcal B}=\sup_{N>M}\frac1{\Upsilon(N)}\sum_{n\leq M}\alpha^H_n\Upsilon(n)|\hat{\mathcal B}_{nN}|,\qquad A_t=\tau_T(\sigma+\rho_{\mathcal B}),
\]
where $\hat{\mathcal B}$ is the matrix of $\mathcal B$ in the coordinates $\hat a$, $\hat y$; we set $\alpha^H_n=A_t$ for $n>M$.

\begin{lemma}\label{r5:lem-A}
Suppose that $\eps_\Upsilon<1$ and that the column sums of $|I-\widehat A\widetilde J|$ are less than one. Then $A\colon\mathcal Y\to\mathcal X$ is a bounded bijection with $\nrm A\leq\max\{\max_\iota\alpha_\iota,A_t,1\}$. It maps every normalised harmonic tail row to a vector of norm at most $A_t$ and every nonharmonic tail row to the identical field tail input; $A\mathsf Me_j=e_j$ for every odd $j\geq M+3$, and $A\Psi_{\ell,s}=e_{(\ell,s)}$ for every field tail input $(\ell,s)$.
\end{lemma}

\begin{proof}
As for Lemma~\ref{r3:lem-A}. For a normalised harmonic tail row $\hat y=e_m$, the shape part $T^{-1}e_m$ has norm at most $\sigma\tau_T\Upsilon(m)$, and $\widehat A\mathcal BT^{-1}e_m$ has norm at most $\rho_{\mathcal B}\tau_T\Upsilon(m)$.
\end{proof}

\subsection{The defect bounds}\label{r5:sec-Z}

We bound $Y=\nrm{A\Fc(x^\circ)}_{\mathcal X}$ and $Z=\nrm{I-AD\Fc(x^\circ)}_{\mathcal X\to\mathcal X}$ by normalised column norms (Lemma~\ref{lem:columns}). The input columns are partitioned into disjoint classes that exhaust all inputs; the class parameters are frozen in the file \texttt{params.json} and rechecked by the verifier.
\begin{itemize}
\item Field inputs: the $4617$ finite inputs; for even $\ell\leq160$, the radial class $s\geq S_\ell$, with $S_\ell=59$ for $\ell\leq10$ and $S_\ell=58$ for $12\leq\ell\leq160$, and the six exceptional inputs $(\ell,58)$, $\ell\leq10$; for even $162\leq\ell\leq254$, the angular gap class $s\geq S_\ell$, with $S_\ell=3$ for $\ell\leq178$, $S_\ell=2$ for $180\leq\ell\leq248$ and $S_\ell=1$ for $250\leq\ell\leq254$, and the $53$ exceptional inputs $(\ell,s)$ with $s<S_\ell$; and the uniform angular class $\ell\geq256$, $s\geq1$.
\item Shape inputs: the $81$ finite inputs $j\leq161$, the $219$ explicit inputs $163\leq j\leq599$, and the far class $j\geq601$.
\end{itemize}

\subsubsection*{The residual and the explicit columns}
The residual $\Fc(x^\circ)$, which has $70299$ nonzero coefficients, and the columns $D\Fc(x^\circ)e$ of all $4698$ finite inputs, of the $59$ exceptional field inputs and of the $219$ explicit shape inputs are polynomials. They are expanded with their complete support in ball arithmetic at the final centre. For the field columns, the products of the four multipliers $q^\circ$, $\partial_{x_1}q^\circ$, $y^{-1}\partial_yq^\circ$, $q^\circ|p^\circ|^2$ with $\Psi_{\ell,s}$ are obtained for all $s$ from the exact three-term expression of $t\Psi_{\ell,s}$ in $\Psi_{\ell,s-1},\Psi_{\ell,s},\Psi_{\ell,s+1}$ (Lemma~\ref{r3:lem-jacobi}(b)), at $1024$ bits; the shape columns are obtained from exact recurrences in $m$ for $r^mU_m$, $r^mC^{(2)}_m$ and $\re(\overline{p^\circ}w^m)$, at $1536$ bits. Every column is recorded by enclosures of its finite rows and of all its harmonic tail coordinates, together with an upper bound for the norm of its nonharmonic tail part; a relative-width test rejects any loss of precision. Let $y=y_f+y^H_t+y^N_t$ be such a column and $w_e$ the weight of the input $e$. By \eqref{r5:eq-A}:
\begin{itemize}
\item for a finite input, $\nrm{(I-AD\Fc(x^\circ))e}_{\mathcal X}/w_e\leq\delta_e+(A_t\nrm{y^H_t}_{\mathcal Y}+\nrm{y^N_t}_{\mathcal Y})/w_e$, where $\delta_e$ bounds the column sum of $|I-\widehat A\widetilde J|$ and is certified from the interval columns by a long-double matrix product with rigorous rounding-error bounds (unit roundoff $2^{-64}$);
\item for an exceptional field input, whose image under $A$ is the identical field tail input, the bound is $(\nrm{\widehat Ay_f}+A_t\nrm{y^H_t}+\nrm{y^N_t-e})/w_e$;
\item for an explicit shape input $e_j$, with $a$ the Neumann approximation of $T^{-1}y^H_t$ and $\mathrm{rem}$ its remainder bound, the bound is $\big(\nrm{\widehat A(y_f-\mathcal Ba)}+\sigma\nrm{a-e_j}+\nrm{y^N_t}+(\sigma+\rho_{\mathcal B})\mathrm{rem}\big)/w_j$; the same expression without $e_j$ and $w_j$, applied to $y=\Fc(x^\circ)$, bounds $Y$.
\end{itemize}
We denote the maxima over the three groups of columns by $Z_{\mathrm{fin}}$, $Z_{\mathrm{ex}}$ and $Z_{\mathrm{sh}}$.

\subsubsection*{Field tail classes}
For a field tail input, $A\Psi_{\ell,s}=e_{(\ell,s)}$, so that $(I-AD\Fc(x^\circ))e_{(\ell,s)}=-Av$ with
\[
v=\chi_1\Psi_{\ell,s}+3\chi_2\,\partial_{x_1}K_5\Psi_{\ell,s}+3\chi_3\,EK_5\Psi_{\ell,s}-3\chi_4\,\partial_{x_1}K_5\Psi_{\ell,s}+\rho^2\chi_5\,K_5\Psi_{\ell,s},
\]
where $\chi_1=q^\circ/b-1$, $\chi_2=\partial_{x_1}q^\circ/b$, $\chi_3=y^{-1}\partial_yq^\circ/b$, $\chi_4=x_1\chi_3$ and $\chi_5=q^\circ|p^\circ|^2/b$. For $1\leq i\leq5$ let $\tau_i(d)$ be the graded $\tau$-norm of the homogeneous part of degree $d$ of $\chi_i$, multiplied by $3$ for $i=2,3,4$ and by $\rho^2$ for $i=5$; it is computed exactly from the solid-monomial expansion. Define $\mathsf A(\ell,\varsigma)$, $\mathsf B(L)$ and $\mathfrak B_d(\ell,s)$ as before Lemma~\ref{r3:lem-envelope}, now with the column norms $\alpha_\iota$ of this section, with $\vartheta_d$ as in Theorem~\ref{r5:thm-mult}(b), and with
\begin{gather*}
\beta_v=\frac1{h_v\Upsilon(v)}\sum_nD_{nv}R^n\Upsilon(n)\alpha^H_n\quad(v\leq1400),\\
\beta^\infty=A_t+\sum_{n\leq M}(\alpha^H_n-A_t)_+\,\mathsf a_{k_n}\zeta^{k_n},\qquad k_n=\frac{1402-n}2,
\end{gather*}
in place of $\beta_\ell$ and $\beta^\infty$ (and $1400$ in place of $1200$ in the definition of $\mathsf B$). Let $\mathfrak B^{(i)}(\ell,s)=\sum_d\tau_i(d)\mathfrak B_d(\ell,s)$, $\bar d=\ell+2s+\frac32$, and
\begin{align*}
\mathfrak z(\ell,s)={}&\mathfrak B^{(1)}(\ell,s)+\frac{\Upsilon(1)}{2(\bar d+1)}\sum_{i=2,4}\Big[\lambda^+_\ell\frac{h_{\ell+1}}{h_\ell}\big(\mathfrak B^{(i)}(\ell+1,s)+\mathfrak B^{(i)}(\ell+1,s-1)\big)\\
&\hphantom{\mathfrak B^{(1)}(\ell,s)+\frac{\Upsilon(1)}{2(\bar d+1)}\sum_{i=2,4}\Big[}+\lambda^-_\ell\frac{h_{\ell-1}}{h_\ell}\big(\mathfrak B^{(i)}(\ell-1,s+1)+\mathfrak B^{(i)}(\ell-1,s)\big)\Big]\\
&+\frac{(\bar d+\frac32)\mathfrak B^{(3)}(\ell,s-1)}{4\bar d(\bar d+1)}+\frac{3\,\mathfrak B^{(3)}(\ell,s)}{4\bar d(\bar d+2)}+\frac{\Upsilon(2)(\bar d+\frac12)\mathfrak B^{(3)}(\ell,s+1)}{4(\bar d+1)(\bar d+2)}\\
&+\frac{\mathfrak B^{(5)}(\ell,s-1)}{4\bar d(\bar d+1)}+\frac{\mathfrak B^{(5)}(\ell,s)}{2\bar d(\bar d+2)}+\frac{\Upsilon(2)\,\mathfrak B^{(5)}(\ell,s+1)}{4(\bar d+1)(\bar d+2)},
\end{align*}
where terms with $\ell-1<0$ are omitted. Let $\mathfrak z_{\mathrm u}(\ell,s)$ be the same expression with the bracket replaced by $R\max\{\mathfrak B^{(i)}(\ell+1,s)+\mathfrak B^{(i)}(\ell+1,s-1),\ \mathfrak B^{(i)}(\ell-1,s+1)+\mathfrak B^{(i)}(\ell-1,s)\}$.

\begin{proposition}\label{r5:prop-gtail}
\begin{enumerate}[label=\textup{(\alph*)}]
\item For every field tail input, $\nrm{(I-AD\Fc(x^\circ))e_{(\ell,s)}}_{\mathcal X}/(h_\ell\Upsilon(\ell+2s))\leq\mathfrak z(\ell,s)\leq\mathfrak z_{\mathrm u}(\ell,s)$.
\item For fixed $\ell$, $\mathfrak z(\ell,s)$ is nonincreasing in $s$, and $\mathfrak z_{\mathrm u}(\ell,s)$ is nonincreasing in $\ell$ and in $s$. Hence the radial class and the angular gap class of $\ell$ are bounded by $\mathfrak z(\ell,S_\ell)$, and the uniform angular class by $\mathfrak z_{\mathrm u}(256,1)$.
\end{enumerate}
\end{proposition}

\begin{proof}
(a) As in Proposition~\ref{r3:prop-gtail}(a): expand $v$ with Lemma~\ref{r5:lem-K} and bound each product of a homogeneous part of $\chi_i$ with a $\Psi$ by Lemma~\ref{r3:lem-envelope}, which holds in the graded norms with $h_ah_\ell$ replaced by $h_a\Upsilon(a+2k)\,h_\ell\Upsilon(\ell+2s)$, by Theorems~\ref{r5:thm-mult}(b) and~\ref{r5:thm-transfer}(c). In its proof, Lemma~\ref{r3:lem-rows} is used; the bound for $v>1400$ now follows from $\sum_nD_{nv}R^n\Upsilon(n)\leq h_v\Upsilon(v)$, Lemma~\ref{r5:lem-conn}(c) and the monotonicity of $\mathsf a_k\zeta^k$ for $k\geq4$. The factors $\Upsilon(1)$ and $\Upsilon(2)$ bound the ratios of the weights of the outputs of the stencils to $\Upsilon(\ell+2s)$. The second inequality follows from $\lambda^+_\ell h_{\ell+1}+\lambda^-_\ell h_{\ell-1}\leq Rh_\ell$ (Lemma~\ref{r5:lem-conn}(c)).

(b) Each $\mathfrak B_d$ is nonincreasing in both arguments: $\mathsf A$, $\mathsf B$ and $\vartheta_d$ are nonincreasing (the step functions are suffix maxima), the first bound in $\mathfrak B_d$ is nondecreasing in $\mathsf A$, $\mathsf B$ and $\vartheta_d$, and the second is an expectation of a function that is nonincreasing in both arguments. The stencil coefficients decrease in $\bar d$, and in $\mathfrak z_{\mathrm u}$ no other dependence on $\ell$ remains.
\end{proof}

We write $\max_\ell Z_{\mathrm{rad}}(\ell)$ and $\max_\ell Z_{\mathrm{gap}}(\ell)$ for the maxima of $\mathfrak z(\ell,S_\ell)$ over the radial and angular gap classes, and $Z_{\mathrm{ang}}=\mathfrak z_{\mathrm u}(256,1)$.

\subsubsection*{Far shape columns}
Let $j=N+1$ with $N\geq N_*=600$ even. Let $W^\circ=K_5g^\circ$, $W^\circ_1=\partial_{x_1}W^\circ$, $Z^\circ=y^{-1}\partial_yW^\circ$ (a polynomial, since $W^\circ$ is a polynomial in $x_1$ and $y^2$; $y\partial_yW^\circ=t(1-\xi^2)Z^\circ$), and $H^\circ=g^\circ+\rho^2|p^\circ|^2(1+W^\circ)$. With $\delta q=r^NU_N$, $\delta\partial_{x_1}q=(N+1)r^{N-1}U_{N-1}$, $\delta(y^{-1}\partial_yq)=-2r^{N-2}C^{(2)}_{N-2}$ and $\delta p=(N+1)w^N$, Lemma~\ref{r3:lem-axis}(b) gives
\begin{equation}\label{r5:eq-farcol}
\begin{split}
b\,D\Fc(x^\circ)e_j={}&r^NU_N\,H^\circ+3(N+1)r^{N-1}U_{N-1}\,W^\circ_1+3\Big(Nr^NT_N-\frac{r^NU_N+t\,r^{N-2}U_{N-2}}2\Big)Z^\circ\\
&+2(N+1)\rho^2q^\circ\re(\overline{p^\circ}w^N)\,W^\circ+b\,\Pi_j .
\end{split}
\end{equation}
Since $A\mathsf Me_j=e_j$, $(I-AD\Fc(x^\circ))e_j=-A(D\Fc(x^\circ)e_j-\mathsf Me_j)$. From $\re(\bw^aw^N)=r^{N+a}T_{N-a}(\xi)$, $r^{i-1}U_{i-1}=\sum_k(2-\delta_{2k,i-1})t^kr^{i-1-2k}T_{i-1-2k}$ and $T_KT_n=\frac12(T_{K+n}+T_{K-n})$ we obtain
\[
q^\circ\re(\overline{p^\circ}w^N)=\sum_{(\eta,e)}\gamma_{\eta,e}\,t^er^{N+\eta}T_{N+\eta}(\xi),
\]
a finite sum over $-320\leq\eta\leq160$ and $e\geq0$ with coefficients independent of $N$; let $\gamma_{\eta,e}\geq0$ be the sum of the moduli of the individual contributions to the coefficient of $t^er^{N+\eta}T_{N+\eta}$. The element $\mathsf Me_j$ is obtained from $\Pi_j$ by replacing each $t^er^KT_K$ with $\mathcal T_K$.

\emph{Column profiles of $A$.} For $v\geq0$ let $\Gamma_F(v)\geq1$ and $\Gamma_H(v)\geq A_t$ be nonincreasing bounds, computed from the $\alpha_\iota$, for the normalised norms $\nrm{A\Psi_{v',\varsigma}}_{\mathcal X}/(h_{v'}\Upsilon(v'+2\varsigma))$ ($\varsigma\geq1$) and for $\beta_{v'}$, over all $v'\geq v$; let $\Gamma_*=\max(\Gamma_F,\Gamma_H)$, $\Gamma_{\max}=\Gamma_*(0)$, $\alpha^H(K)=\alpha^H_K$, and let $\mathcal H_v=\sum_nD_{nv}\mathcal T_n$ be the harmonic element with boundary values $Y_v$, so that $\nrm{A\mathcal H_v}\leq\Gamma_H(v)h_v\Upsilon(v)$. Let $I=140$, $S_p(n)=\sum_{k\geq0}\zeta^k(n+k)^p$ and, for $p\in\{0,1,2\}$ and $i_1=\lfloor(K-M)/2\rfloor+1$,
\[
G_p(K)=\begin{cases}\mathsf w_{I+1}S_p(I+1)\,\Gamma_{\max},&i_1\leq I+1,\\ \mathsf w_{I+1}\big(S_p(I+1)\,\Gamma_*(M+1)+\zeta^{i_1-I-1}S_p(i_1)\,\Gamma_{\max}\big),&i_1>I+1.\end{cases}
\]
Since $e_i$ is nonincreasing for $i\geq1$, $\mathsf w_i\leq\mathsf w_{I+1}\zeta^{i-I-1}$ for $i>I$, and $\sum_{I<i\leq K/2}\mathsf w_ii^p\Gamma_*(K-2i)\leq G_p(K)$, because $K-2i\geq M$ for $i<i_1$ and $\Gamma_*(M)=\Gamma_*(M+1)$ for the certified profiles.

\begin{lemma}\label{r5:lem-lap}
For $K\geq2$, $\sum_ii\,(K-i+\frac32)\,C_{K-2i,K}\,Y_{K-2i}=-\frac{3K}4U_{K-2}$.
\end{lemma}

\begin{proof}
In $\R^5$, $\Delta(t^iH_v)=4i(i+v+\frac32)t^{i-1}H_v$. The function $r^KT_K(\xi)=\re w^K$ is harmonic in the meridian variables, so $\Delta\re w^K=\frac3y\partial_y\re w^K=-\frac{3K}y\im w^{K-1}=-3Kr^{K-2}U_{K-2}(\xi)$. Apply $\Delta$ to $r^KT_K=\sum_iC_{K-2i,K}H_{K-2i}t^i$ and restrict to $t=1$.
\end{proof}

\begin{lemma}\label{r5:lem-raw}
Let $K\geq0$, $e\geq0$, $m\geq0$, $\beta=K+\frac32$, $c_m=\max(2,m+1)$, $\mathsf F_m(\beta')=m!/(\beta'/2+1)_m$, and
\[
\mathsf f(x)=\frac{\beta-2x+1}{\prod_{k=1}^{m+1}(\beta-x+e+k)},\qquad\varrho_2(i)=\mathsf f(i)-\mathsf f(0)-\mathsf f'(0)\,i .
\]
Let $u=t^er^KT_K(\xi)(1-t)^m$ if $m\geq1$, and $u=t^er^KT_K(\xi)-\mathcal T_K$ if $m=0$ and $K>M$. Then
\begin{multline*}
\frac{\nrm{Au}_{\mathcal X}}{R^K\Upsilon(K+2e+2m)}\leq\sum_i\mathsf w_i\,\mathsf g_i\,\Gamma_F(K-2i)\\
+m!\Big(\mathsf f_0\,\alpha^H(K)+\frac34|\mathsf f'(0)|\,\mathsf u(K)+\sum_{i\geq1}\mathsf w_i\Big(\frac{|\mathsf f'(0)|\,i^2}\beta+|\varrho_2(i)|\Big)\Gamma_H(K-2i)\Big),
\end{multline*}
where $\mathsf g_i=\mathsf F_m(\beta-2i)$ and $\mathsf f_0=\mathsf f(0)$ if $m\geq1$, $\mathsf g_i=(e+i)/(\beta-i+e+1)$ and $\mathsf f_0=1-\mathsf f(0)=e/(\beta+e+1)$ if $m=0$, and $\mathsf u(K)=2A_t\zeta/(1-\zeta)+2\sum_{n\leq\min(M,K-2),\,n\equiv K}R^{n-K}(\alpha^H_n-A_t)_+$. Moreover, for $i\leq K/2$,
\[
|\mathsf f'(0)|\leq\frac{c_m\mathsf f(0)}{\beta+1},\qquad|\varrho_2(i)|\leq\frac{i^2}2\cdot\frac{(\beta+1)(c_m^2+m+5)}{(\beta-2i+1)^2\prod_{k=1}^{m+1}(\beta-i+e+k)} .
\]
\end{lemma}

\begin{proof}
Let $v_i=K-2i$. By definition of $C_{\ell K}$, $u=\sum_iC_{v_i,K}H_{v_i}t^{e+i}(1-t)^m$, minus $\mathcal T_K=\sum_iC_{v_i,K}H_{v_i}$ if $m=0$, and all outputs have degree at most $K+2e+2m$; so it suffices to bound the norms with $\Upsilon\equiv1$ (Theorem~\ref{r5:thm-transfer}(a)). \emph{Field layer.} By Lemma~\ref{r3:lem-jacobi}(c),(d), the $J^{\beta_{v_i}}$-coefficients of $t^{e+i}(1-t)^m$ have $\ell^1$ norm at most $\mathsf F_m(\beta_{v_i})=\mathsf F_m(\beta-2i)$, and those of $t^{e+i}-1$ with $s\geq1$ have sum $(e+i)/(\beta_{v_i}+e+i+1)$; with $|C_{v_i,K}|h_{v_i}\leq\mathsf w_iR^K$ this gives the first sum. \emph{Harmonic layer.} By Lemma~\ref{r3:lem-jacobi}(a), the coefficient of $J^{\beta_v}_0$ in $t^\mu(1-t)^m$ is $(\beta_v+1)\int_0^1t^{\beta_v+\mu}(1-t)^mdt=m!(\beta_v+1)/\prod_{k=1}^{m+1}(\beta_v+\mu+k)$, which for $v=v_i$ and $\mu=e+i$ equals $m!\,\mathsf f(i)$. Hence the boundary trace of the harmonic layer of $u$ is $m!\sum_iC_{v_i,K}\mathsf f(i)Y_{v_i}$, minus $T_K$ if $m=0$. By $\sum_iC_{v_i,K}Y_{v_i}=T_K$ and Lemma~\ref{r5:lem-lap}, $\sum_iiC_{v_i,K}Y_{v_i}=\beta^{-1}\big(-\frac{3K}4U_{K-2}+\sum_ii^2C_{v_i,K}Y_{v_i}\big)$, and the trace equals
\[
m!\Big(\mathsf f(0)T_K-\frac{3K\mathsf f'(0)}{4\beta}U_{K-2}+\sum_iC_{v_i,K}\Big(\frac{\mathsf f'(0)i^2}\beta+\varrho_2(i)\Big)Y_{v_i}\Big),
\]
with $\mathsf f(0)-1$ in place of $\mathsf f(0)$ if $m=0$. We apply $A$ to the three parts: $\nrm{A\mathcal T_K}\leq\alpha^H(K)R^K\Upsilon(K)$; the Chebyshev coefficients of $U_{K-2}$ are $2$ (and $1$ for $T_0$), which gives $\mathsf u(K)$, since $K<\beta$; and $Y_{v}$ is the trace of $\mathcal H_v$. (Thus the harmonic layer is, up to terms of relative size $O(i^2/\beta)$, a single Chebyshev mode; the triangle inequality applied directly to $\sum_iC_{v_i,K}\mathsf f(i)Y_{v_i}$ would lose a factor close to $\kappa_5$.) For the bounds on $\mathsf f$ write $\mathsf f=g/P$ with $g(x)=\beta-2x+1$ and $P(x)=\prod_k(\beta-x+e+k)$, and let $S_1=\sum_k(\beta-x+e+k)^{-1}$, $S_2=\sum_k(\beta-x+e+k)^{-2}$. Then $\mathsf f'=\mathsf f\,(S_1-2/g)$ and $\mathsf f''=P^{-1}\big(g(S_1^2+S_2)-4S_1\big)$. At $x=0$, $0<S_1\leq(m+1)/(\beta+1)$ and $2/g=2/(\beta+1)$, which gives the bound for $\mathsf f'(0)$. For $0\leq x\leq i\leq K/2$ the two terms of $\mathsf f''$ have opposite signs, so $|\mathsf f''(x)|\leq P(x)^{-1}\max\{g(S_1^2+S_2),4S_1\}$; with $P(x)\geq P(i)$, $g\leq\beta+1$, $S_1\leq(m+1)/(\beta-i+1)$, $S_2\leq(m+1)/(\beta-i+1)^2$ and $(\beta-2i+1)^2\leq(\beta+1)(\beta-i+1)$ this is at most $(\beta+1)\max\{(m+1)(m+2),4(m+1)\}/((\beta-2i+1)^2P(i))$, and $\max\{(m+1)(m+2),4(m+1)\}\leq c_m^2+m+5$. Taylor's formula gives the bound for $\varrho_2$.
\end{proof}

Let $\Phi_m(K)$ ($m\geq1$) and $\Phi_P(K,e)$ ($m=0$) be the right-hand side of Lemma~\ref{r5:lem-raw}, evaluated as follows. For $m\geq1$ the bounds are made independent of $e$ by putting $e=0$ in $\mathsf f(0)$ and in the bound for $\varrho_2$. The sums over $i$ are evaluated exactly for $i\leq24$; for $24<i\leq\min(I,K/2)$ the factors $\mathsf g_i$ and $|\varrho_2(i)|/i^2$, which increase with $i$, are replaced by their values at the largest index; and the terms with $i>I$ are bounded by means of $G_0,G_1,G_2$, using $\mathsf g_i\leq1$, $|\mathsf f(i)|\leq\mathsf f_{\max}$ and $|\varrho_2(i)|\leq2\mathsf f_{\max}+|\mathsf f'(0)|i$, where $\mathsf f_{\max}=(\beta+1)/(K/2+\frac52)^{m+1}$ for $m\geq1$ and $\mathsf f_{\max}=1$ for $m=0$. Let $\varphi_m(K)=\sup_{K'\geq K}\Phi_m(K')$, and
\[
\bar\Gamma(K)=\sup_{K'\geq K}\sum_{i\leq I}\mathsf w_i\Gamma_*\big(\max(K'-2i,0)\big)+\mathsf w_{I+1}S_0(I+1)\,\Gamma_{\max} .
\]
For $f$ one of the polynomials $H^\circ$, $W^\circ_1$, $Z^\circ$, $W^\circ$, write $f=\sum_aH_af_a(t)$, $f_a=\sum_sf_{a,s}J^{\beta_a}_s=\sum_mf^\natural_{a,m}(1-t)^m$ and $\nrm{f_a}_J=\sum_s|f_{a,s}|$. By Theorem~\ref{r5:thm-C1}, every $f_a$ has a double zero at $t=1$ for $f=W^\circ$ and a zero for $f=W^\circ_1,Z^\circ$, so the corresponding $f^\natural_{a,m}$ vanish exactly. For $K'\geq0$ let
\[
\mathrm{HH}_a(K')=\sup_{K''\geq K'}\sum_{i\leq\min(100,K''/2)}\mathsf w_i\Gamma_*(K''-2i)\sum_{m>3}|f^\natural_{a,m}|\,\mathsf F_m(\beta_{K''-2i}),
\]
let $\mathrm{JT}(K')$ be the supremum over $K''\geq K'$ of $\sum_{100<i\leq\min(I,K''/2)}\mathsf w_i\Gamma_*(K''-2i)$ plus $G_0(K'')$ if $K''>2I$, and put, for $K\geq a$,
\begin{multline*}
\mathrm{Mix}_a(K)=|f^\natural_{a,0}|\,\bar\Gamma(K-a)+\sum_{m=1}^3|f^\natural_{a,m}|\,\varphi_m(K-a)+\mathrm{HH}_a(K-a)\\
+\Big(\nrm{f_a}_J+\sum_{m\leq3}|f^\natural_{a,m}|\,\mathsf F_m(\beta_a)\Big)\mathrm{JT}(K-a),
\end{multline*}
$\mathrm{Jac}_a(K)=\nrm{f_a}_J\,\bar\Gamma(\max(K-a,0))$, and finally
\[
V_f(K)=\sup_{K'\geq K}\sum_ah_a\min\{\mathrm{Jac}_a(K'),\mathrm{Mix}_a(K')\},
\]
where $\mathrm{Mix}_a(K')$ is omitted for $K'<a$. All these tables are nonincreasing by construction.

\begin{lemma}\label{r5:lem-centre}
For $f\in\{H^\circ,W^\circ_1,Z^\circ,W^\circ\}$, $K\geq0$ and $e\geq0$, $\nrm{A(t^er^KT_K(\xi)f)}_{\mathcal X}\leq R^K\,\Upsilon(K+2e+\deg f)\,V_f(K)$.
\end{lemma}

\begin{proof}
It suffices to treat one angular component $H_af_a$, with $\Upsilon\equiv1$, and to show that its contribution is at most $R^Kh_a\mathrm{Jac}_a(K)$ and, for $K\geq a$, at most $R^Kh_a\mathrm{Mix}_a(K)$. \emph{Jacobi route.} Expand $r^KT_K=\sum_iC_{K-2i,K}H_{K-2i}t^i$ and multiply by $H_af_a$ with Theorem~\ref{r5:thm-mult}(a),(b): the outputs have angular index at least $K-2i-a$, the lowering steps of the Jacobi parameter that are needed come with enough factors $t$ (as in the proof of Lemma~\ref{r3:lem-mult}(a)), all steps are stochastic, and $\sum_vG_{v'av}h_v\leq h_{v'}h_a$; with Lemma~\ref{r5:lem-conn}(b) this gives $\mathrm{Jac}_a$. \emph{Mixed route}, for $K\geq a$. By $H_a=\sum_nD_{na}t^{(a-n)/2}r^nT_n$ and $T_KT_n=\frac12(T_{K+n}+T_{K-n})$, $t^er^KT_KH_a$ is a nonnegative combination of raw monomials $t^{e'}r^{K'}T_{K'}$ with $K'\geq K-a$ and $R$-weighted mass at most $h_aR^K$. Split $f_a=P_a+(f_a-P_a)$ with $P_a=\sum_{m\leq3}f^\natural_{a,m}(1-t)^m$. The part $P_a$ is bounded by $\bar\Gamma$ for $m=0$ and by Lemma~\ref{r5:lem-raw} for $1\leq m\leq3$. For $f_a-P_a$, in the connection expansion of $t^{e'}r^{K'}T_{K'}$ the components with index $i\leq100$ are bounded with Lemma~\ref{r3:lem-jacobi}(c),(d) term by term in $m>3$, which gives $\mathrm{HH}_a$, and those with $i>100$ by the Jacobi norm, which gives $\mathrm{JT}$: here
\[
\nrm{t^\mu(f_a-P_a)}_{J^{\beta_v},1}\leq\nrm{f_a-P_a}_{J^{\beta_a},1}\leq\nrm{f_a}_J+\sum_{m\leq3}|f^\natural_{a,m}|\,\mathsf F_m(\beta_a),
\]
because every lowering of the Jacobi parameter from $\beta_a$ to $\beta_v<\beta_a$ comes with a factor $t$ (one checks $\mu\geq a-v$ whenever $K\geq a$), and the raising and lowering stencils of Lemma~\ref{r3:lem-jacobi}(b) are nonnegative with column sums one. Finally, $\bar\Gamma$, $\varphi_m$, $\mathrm{HH}_a$ and $\mathrm{JT}$ are nonincreasing and $K'\geq K-a$.
\end{proof}

\begin{proposition}\label{r5:prop-far}
Let $\mathsf p=\Upsilon(1000)/(b\sigma R)$, $N_2=1000$, $k_0=N_2/2$, and for even $N\geq N_*$
\begin{align*}
\mathfrak F(N)=\mathsf p\Big[&\frac1{N+1}\sum_{k=0}^{N/2}(2-\delta_{2k,N})\zeta^kV_{H^\circ}(N-2k)+\frac6R\sum_{k=0}^{N/2-1}\zeta^kV_{W^\circ_1}(N-1-2k)\\
&+3V_{Z^\circ}(N)
+\frac3{2(N+1)}\Big(\sum_{k=0}^{N/2}(2-\delta_{2k,N})\zeta^kV_{Z^\circ}(N-2k)\\
&\hphantom{{}+3V_{Z^\circ}(N)+\frac3{2(N+1)}\Big(}+\sum_{k=0}^{N/2-1}(2-\delta_{2k,N-2})\zeta^{k+1}V_{Z^\circ}(N-2-2k)\Big)\\
&+2\rho^2\sum_{(\eta,e)}\gamma_{\eta,e}R^\eta\big(\Phi_P(N_*+\eta,e)+V_{W^\circ}(N_*+\eta)\big)\Big],\\
\mathfrak F_\infty=\mathfrak F(N_2)&+\mathsf p\,\frac{\zeta^{k_0}}{1-\zeta}\Big(\frac{2\,V_{H^\circ}(0)}{N_2+1}+\frac{6\,V_{W^\circ_1}(0)}R+\frac{6\,V_{Z^\circ}(0)}{N_2+1}\Big).
\end{align*}
Then for every odd $j\geq N_*+1$,
\[
\frac{\nrm{(I-AD\Fc(x^\circ))e_j}_{\mathcal X}}{\sigma jR^j\Upsilon(j-1)}\leq Z_{\mathrm{far}}:=\max\big(\{\mathfrak F(N):N_*\leq N\leq N_2\ \text{even}\}\cup\{\mathfrak F_\infty\}\big)+10^{-17}.
\]
\end{proposition}

\begin{proof}
Let $j=N+1$. We divide the five parts of \eqref{r5:eq-farcol}, with $\mathsf Me_j$ subtracted from the last, by $b$ and by the input weight $\sigma(N+1)R^{N+1}\Upsilon(N)$. All polynomials involved have degree at most $N+596$ (the degree of $H^\circ$ is at most $596$), so the factor $\Upsilon(1000)$ in $\mathsf p$ covers the grading (Theorem~\ref{r5:thm-transfer}(c)). By $r^NU_N=\sum_k(2-\delta_{2k,N})t^kr^{N-2k}T_{N-2k}$ and Lemma~\ref{r5:lem-centre}, the first part contributes the first sum of $\mathfrak F(N)$; the second, where all Chebyshev coefficients of $U_{N-1}$ equal two, contributes the second sum; for the third, $3N/(N+1)\leq3$ and $r^NU_{N-2}=t\,r^{N-2}U_{N-2}$. For the last two parts, Lemmas~\ref{r5:lem-raw} and~\ref{r5:lem-centre} give the bound with $\Phi_P(N+\eta,e)+V_{W^\circ}(N+\eta)$ in place of the values at $N_*+\eta$.

\emph{Monotonicity in $N$.} The tables $V_f$ are nonincreasing. For $K\geq M+2I+2=442$ the function $K\mapsto\Phi_P(K,e)$ is nonincreasing: the index ranges $i\leq24$, $24<i\leq I$, $i>I$ no longer depend on $K$; $\Gamma_F$ and $\Gamma_H$ are nonincreasing; $\alpha^H(K)=A_t$; the terms of $\mathsf u(K)$ are multiplied by $\zeta$ when $K$ increases by two; the factors $\mathsf g_i=(e+i)/(\beta-i+e+1)$, $\mathsf f_0=e/(\beta+e+1)$, the bound $2/(\beta+e+1)$ for $|\mathsf f'(0)|$, $1/\beta$ and the bound for $|\varrho_2(i)|/i^2$ decrease in $\beta=K+\frac32$; and $G_p(K)$ is nonincreasing, because $i_1$ increases with $K$ and $S_p(i)-\zeta S_p(i+1)=i^p\geq0$. Since $N_*+\eta\geq442$ for $\eta\geq-158$, these terms of $\mathfrak F(N)$ bound the corresponding terms at every $N\geq N_*$. For the monomials with $\eta\leq-160$ and $N>N_*$ we use instead $\nrm{A(t^er^KT_K-\mathcal T_K)}_{\mathcal X}\leq\nrm A\,\nrm{t^er^KT_K-\mathcal T_K}_{\mathcal Y}\leq2\kappa_5\nrm AR^K\Upsilon(K+2e)$, which follows from $t^er^KT_K-\mathcal T_K=\sum_iC_{v_i,K}H_{v_i}(t^{e+i}-1)$ and $\nrm{H_v(t^\mu-1)}_{\mathcal Y}\leq2h_v\Upsilon(v+2\mu)$; the exact total weight $\mathsf p\,2\rho^2\sum_{\eta\leq-160}\gamma_{\eta,e}R^\eta$ of these monomials is below $1.2\cdot10^{-25}$, so their contribution is below $10^{-17}$ for every $N$.

\emph{The $k$-sums.} For $N_*\leq N\leq N_2$ the sums are evaluated exactly. Let $N>N_2$. In the first sum and in the first $Z$-sum the terms with $k<k_0$ have the same weight $2$ as at $N_2$; at $k=k_0$ the weight is $2$ instead of $1$; in the second sum and in the second $Z$-sum, which at $N_2$ end at $k=k_0-1$, the terms with $k\geq k_0$ are new, and in the second $Z$-sum the weight at $k=k_0-1$ is $2$ instead of $1$. All values $V_f(N-2k)$ are at most $V_f(N_2-2k)$ for $k\leq k_0$ and at most $V_f(0)$ in general, and $1/(N+1)\leq1/(N_2+1)$. Hence the excess over the sums evaluated at $N_2$ is at most $\zeta^{k_0}V_f(0)+2\sum_{k>k_0}\zeta^kV_f(0)\leq2\zeta^{k_0}V_f(0)/(1-\zeta)$ for the first sum and for the first $Z$-sum, and at most $2\zeta^{k_0}V_f(0)/(1-\zeta)$ for the second sum and for the second $Z$-sum (where the powers are $\zeta^{k+1}$ with $k\geq k_0-1$). With the prefactors $1/(N_2+1)$, $3/R$ and $3/(2(N_2+1))$ this gives $\mathfrak F_\infty$. Thus every $N>N_2$ is covered by $\mathfrak F_\infty+10^{-17}$. (The bound for $Z_{\mathrm{far}}$ in Table~\ref{r5:tab-bounds} includes the term $10^{-17}$.)
\end{proof}

Collecting all classes,
\[
Z=\max\big\{Z_{\mathrm{fin}},\ Z_{\mathrm{ex}},\ \max_\ell Z_{\mathrm{rad}}(\ell),\ \max_\ell Z_{\mathrm{gap}}(\ell),\ Z_{\mathrm{ang}},\ Z_{\mathrm{sh}},\ Z_{\mathrm{far}}\big\}
\]
bounds $\nrm{I-AD\Fc(x^\circ)}_{\mathcal X\to\mathcal X}$.

\subsection{Certified bounds, geometry and reconstruction}\label{r5:sec-zero}

The bounds in Table~\ref{r5:tab-bounds} are outward roundings of the interval enclosures in the receipt of the verifier (Section~\ref{r5:sec-cap}). The residual uses the entire polynomial support and the fixed dyadic $b$, and all finite columns are computed at the final centre, without transfer from another matrix. The symbol defect is $\eps_\Upsilon<0.025209$, the finite column sums $\delta_e$ are bounded by $Z_{\mathrm{fin}}<1$, and $\rho_{\mathcal B}<2896.07$.

\begin{table}[ht]
\centering\footnotesize
\begin{tabular}{@{}ll@{}}
\toprule
Quantity & Certified bound\\
\midrule
$\nrm A$, $A_t$ & $<104568.72$, $<16.60796$\\
$Y=\nrm{A\Fc(x^\circ)}_{\mathcal X}$ & $<3.19608\cdot10^{-12}$\\
$Z_{\mathrm{fin}}$ ($4698$ columns), $Z_{\mathrm{ex}}$ ($59$ columns) & $<0.0585324$, $<0.496612$\\
field classes: $81$ radial, $47$ angular gap, $Z_{\mathrm{ang}}$ ($\ell\geq256$) & all $<0.887235$; the maximum is $Z_{\mathrm{ang}}$\\
$Z_{\mathrm{sh}}$ ($163\leq j\leq599$), $Z_{\mathrm{far}}$ ($j\geq601$) & $<0.753050$, $<0.605269$\\
$Z=\nrm{I-AD\Fc(x^\circ)}_{\mathcal X\to\mathcal X}$ & $<0.887235$\\
$G_0$, $C_0$ & $<107535.85$, $<1.079742$\\
$c_2$, $c_3$, $c_4$ & $<1.540189$, $<6.608341\cdot10^{-6}$, $<1.582007\cdot10^{-11}$\\
$r=10^{-10}$: radii polynomial & $<-8.0788\cdot10^{-12}$\\
$r=10^{-10}$: contraction bound & $<0.887267$\\
\bottomrule
\end{tabular}
\caption{The five-dimensional certificate for $M=160$, $S=57$, $R=21/20$, $\Upsilon(D)=1+D/10^6$ and $\sigma=30\rho^2$. The final $Z$ is the maximum over every field and shape class; the full interval values are in the receipt.}
\label{r5:tab-bounds}
\end{table}

\begin{theorem}\label{r5:thm-zero}
There is exactly one zero $x^*=(g^*,c^*)$ of $\Fc$ in $\overline B_{\mathcal X}(x^\circ,10^{-10})$.
\end{theorem}

\begin{proof}
By Lemma~\ref{r5:lem-A} and the bounds $\eps_\Upsilon<1$, $\max_e\delta_e<1$, the operator $A$ is injective. The classes of Section~\ref{r5:sec-Z} cover all input columns, and $Z<0.887235$. The verifier evaluates both inequalities \eqref{r3:eq-radii} with the full enclosures of $Y$, $Z$, $\nrm A$, $c_2$, $c_3$, $c_4$ at the exact radius $r=10^{-10}$; their left-hand sides are below $-8.0788\cdot10^{-12}$ and $0.887267$, so Proposition~\ref{r3:prop-NK} applies.
\end{proof}

The existence ball projects onto the shape ball $\nrm{c-c^\circ}_{\mathcal E}\leq\eps_0$ with $\eps_0=1044359479949477\cdot10^{-30}\geq r/\sigma$. Let
\[
S_1=\sum_{j\geq3}j|c^\circ_j|,\qquad S_2=\sum_{j\geq3}j(j-1)|c^\circ_j|,\qquad H_{\mathcal E}=\max_{j\geq3\ \mathrm{odd}}\frac{j-1}{R^j\Upsilon(j-1)} ;
\]
the maximum is attained for $j\leq21$, since the ratio of consecutive odd terms is at most $(j+1)/((j-1)R^2)\leq1$ for $j\geq21$.

\begin{proposition}\label{r5:prop-geometry}
For every $c$ with $\nrm{c-c^\circ}_{\mathcal E}\leq\eps_0$: $\re\psi_c'>0.981143$ on $\{|w|\leq\frac{101}{100}\}$; $\re(1+w\psi_c''/\psi_c')>0.839460$ on $\overline\D$; $\re(w\psi_c'/\psi_c)>0.984794$ on $|w|=1$; and $|c_3|>4.45696\cdot10^{-4}$.
\end{proposition}

\begin{proof}
Let $\delta=c-c^\circ$. For $|w|\leq R$, $|\sum_jj\delta_jw^{j-1}|\leq\nrm\delta_{\mathcal E}/R$, $|\sum_j\delta_jw^j|\leq\nrm\delta_{\mathcal E}/R$ for $|w|\leq1$, and $\sum_jj(j-1)|\delta_j|\leq H_{\mathcal E}\nrm\delta_{\mathcal E}$. Hence $\re\psi_c'\geq c^\circ_1-\sum_{j\geq3}j|c^\circ_j|(\frac{101}{100})^{j-1}-\eps_0/R$ on the disc of radius $\frac{101}{100}$, and on $\overline\D$
\[
\re\Big(1+\frac{w\psi_c''}{\psi_c'}\Big)\geq1-\frac{S_2+\eps_0H_{\mathcal E}}{c^\circ_1-S_1-\eps_0/R},\qquad\re\frac{w\psi_c'}{\psi_c}\geq1-\frac{\sum_{j\geq3}(j-1)|c^\circ_j|+\eps_0/R}{c^\circ_1-\sum_{j\geq3}|c^\circ_j|-\eps_0/R}\quad(|w|=1),
\]
and $|c_3|\geq|c^\circ_3|-\eps_0/(3R^3\Upsilon(2))$. The right-hand sides are evaluated in exact rational arithmetic.
\end{proof}

\subsubsection*{Proof of Theorem~\ref{thm:main} for $n=5$}
Let $x^*=(g^*,c^*)$ be the zero of Theorem~\ref{r5:thm-zero}, and $\psi=\psi_{c^*}$. By Proposition~\ref{r5:prop-geometry}, $\re\psi'>0$ on the disc of radius $\frac{101}{100}$, so by Lemma~\ref{r5:lem-lift}, with $R_0=\frac{101}{100}$, $\Theta=\Theta_{c^*}$ is a real-analytic diffeomorphism of a neighbourhood of $\overline B$ onto its image, with Jacobian $q_{c^*}^3|p_{c^*}|^2>0$. Let $\Omega_0=\Theta(B)$, $U=1+\iota(K_5g^*)$ and $u_0=U\circ\Theta^{-1}$. By Lemma~\ref{r5:lem-F} and Lemma~\ref{r5:lem-lift}(b), $u_0\in C^1(\overline{\Omega_0})$ is a weak solution of $\Delta u_0+\rho^2u_0=0$ with $u_0=1$ and $\nabla u_0=0$ on $\partial\Omega_0$. Let $\Omega=\rho\,\Omega_0$ and $u(X)=u_0(X/\rho)$. Then $u$ is a weak solution of $\Delta u+u=0$ with the same boundary data, and Lemma~\ref{r3:lem-regularity} extends $u$ to a real-analytic solution on a neighbourhood of $\overline\Omega$. The boundary $\partial\Omega=\rho\Theta(\partial B)$ is a compact real-analytic hypersurface diffeomorphic to $S^4$, and $\Omega$ is invariant under $O(4)$ acting on $X'$ and under $X_1\mapsto-X_1$. The lower bound for $\re(w\psi'/\psi)$ makes the meridian domain $\psi(\D)$, and hence $\Omega$, strictly star-shaped with respect to the origin. The lower bound for $\re(1+w\psi''/\psi')$ makes $\psi(\D)$ strictly convex, and the solid-of-revolution argument in the proof of Proposition~\ref{r3:prop-geometry}, with $X'\in\R^4$, shows that $\Omega$ is convex. A constant solution of $\Delta u+u=0$ vanishes, but $u=1$ on $\partial\Omega$. If $\Omega$ were a ball, its symmetries would put its centre at the origin, $\psi(\D)$ would be a disc centred at the origin, and Schwarz's lemma would force $\psi$ to be linear, contrary to $|c^*_3|>4.45696\cdot10^{-4}$. This proves Theorem~\ref{thm:main} for $n=5$.

\begin{remark}\label{r5:rem-ball}
The number $\rho$ agrees with $j_{5/2,17}$ to within $3\cdot10^{-15}$ and $c^\circ_1\approx1.000225$, so $\Omega$ is close to the Schiffer ball of radius $j_{5/2,17}$, which corresponds to $\psi(w)=(j_{5/2,17}/\rho)w$. That ball is far from the existence ball: its coefficient $c_3$ vanishes, so its distance from $x^\circ$ in $\mathcal X$ is at least $3\sigma R^3\Upsilon(2)|c^\circ_3|>148$, while the existence ball has radius $10^{-10}$. The domain $\Omega$ lies at the near-resonance $(j_{5/2,17},12)$ of Table~\ref{tab:resonances}, and its dominant nonconstant zonal harmonic has degree twelve (Remark~\ref{rem:domain}).
\end{remark}

\section{Dimension seven}\label{r7:sec}

We prove Theorem~\ref{thm:main} for $n=7$ by the axisymmetric formulation of Section~\ref{r5:sec}, with the weight $\Upsilon(D)=(1+D/10^6)^2$, the square of the weight used for $n=5$.

\subsection{Coefficient spaces and reconstruction}\label{r7:sec-spaces}
Let $B\subset\R^7$ be the unit ball, $x=(x_1,x')\in\R\times\R^6$, $y=|x'|$, $w=x_1+iy$, $r=|x|$, $t=r^2$ and $\xi=x_1/r$. Throughout this section
\[
\lambda=\frac52,\qquad R=\frac65,\qquad\zeta=R^{-2}=\frac{25}{36},\qquad
\Upsilon(D)=\Big(1+\frac D{10^6}\Big)^2.
\]
Put $Y_\ell=C^{(\lambda)}_\ell/C^{(\lambda)}_\ell(1)$, $H_\ell=r^\ell Y_\ell(\xi)$, $\beta_\ell=\ell+\lambda$ and $\Psi_{\ell,s}=H_\ell J^{\beta_\ell}_s(t)$. Since $\lambda=(n-2)/2$, $H_\ell$ is the axisymmetric harmonic polynomial on $\R^7$, homogeneous of degree $\ell$, with $H_\ell(1,0)=1$ and $|H_\ell|\leq1$ on $\overline B$, and $\Delta(H_\ell v(t))=4H_\ell\big(tv''+(\beta_\ell+1)v'\big)$. Here
\[
C^{(\lambda)}_\ell(1)=\frac{(\ell+1)(\ell+2)(\ell+3)(\ell+4)}{24},\qquad
\lambda^+_\ell=\frac{\ell+5}{2\ell+5},\quad\lambda^-_\ell=\frac\ell{2\ell+5}.
\]
The identities $x_1H_\ell=\lambda^+_\ell H_{\ell+1}+\lambda^-_\ell tH_{\ell-1}$ and $(x\cdot\nabla)H_\ell=\ell H_\ell$ follow from the Gegenbauer recurrence and homogeneity. For $\ell\geq1$, $\partial_{x_1}H_\ell=\ell H_{\ell-1}$: the left side is axisymmetric, harmonic and homogeneous of degree $\ell-1$, hence a multiple of $H_{\ell-1}$, and the multiple is $\ell$ because $H_\ell(x_1,0)=x_1^\ell$. For $\ell=0$ the derivative is zero. Let $\mathsf a_k=(\lambda)_k/k!$ and
\[
Y_\ell=\sum_mD_{m\ell}T_m,\qquad
D_{\ell-2k,\ell}=\frac{(2-\delta_{\ell,2k})\mathsf a_k\mathsf a_{\ell-k}}{C^{(\lambda)}_\ell(1)},\qquad h_\ell=\sum_mD_{m\ell}R^m.
\]
$D_{m\ell}\geq0$ and $\sum_mD_{m\ell}=1$. Write $T_K=\sum_\ell C_{\ell K}Y_\ell$ and $\mathcal T_K=\sum_\ell C_{\ell K}H_\ell$. All field and harmonic indices in $\mathcal X$, $\mathcal Y$ below are even, and all shape indices are odd. Define
\begin{align*}
\nrm f_{\mathcal L}&=\sum_{\ell,s\geq0}h_\ell\Upsilon(\ell+2s)|f_{\ell,s}|,\qquad
\Gc=\{g\in\mathcal L^{\mathrm{ev}}:g_{\ell,0}=0\},\\
\nrm f_{\mathcal Y}&=\sum_mR^m\Upsilon(m)|f_m|+\sum_{\ell,\,s\geq1}h_\ell\Upsilon(\ell+2s)|f_{\ell,s}|,\\
\nrm c_{\mathcal E}&=\sum_{j\ \mathrm{odd}}jR^j\Upsilon(j-1)|c_j|,\qquad
\nrm{(g,c)}_{\mathcal X}=\nrm g_{\mathcal L}+\sigma\nrm c_{\mathcal E}.
\end{align*}
Here $f\in\mathcal Y$ represents $\sum_m f_m\mathcal T_m+\sum_{s\geq1}f_{\ell,s}\Psi_{\ell,s}$. For multipliers put $\tau(\sum m_{a,k}H_at^k)=\sum h_a\Upsilon(a+2k)|m_{a,k}|$ and $\nrm{\sum\tilde m_{n,k}r^{n+2k}T_n}_{\mathrm{ch}}=\sum R^n\Upsilon(n+2k)|\tilde m_{n,k}|$. $\mathcal L^{\mathrm{ev}}\hookrightarrow\mathcal Y$ has norm at most one.

\begin{lemma}[Connection coefficients]\label{r7:lem-conn}
$C_{00}=1$, and for $K\geq1$, $0\leq i\leq K/2$ and $\ell=K-2i$,
\[
C_{\ell K}=\frac{K(-\lambda)_i(K-i-1)!(\ell+\lambda)C^{(\lambda)}_\ell(1)}{2i!(\lambda)_{K-i+1}}.
\]
Let $e_i=|(-\lambda)_i|/i!$ and $\mathsf w_i=e_i\zeta^i(1-\zeta)^{-\lambda}$. Then
\begin{gather*}
|C_{K-2i,K}|h_{K-2i}\leq\mathsf w_iR^K,\qquad
\kappa_7:=\sum_i\mathsf w_i=(1-\zeta)^{-5/2}\Big(2+\frac{15}4\zeta^2\Big)-1,\\
\lambda^+_\ell h_{\ell+1}+\lambda^-_\ell h_{\ell-1}\leq Rh_\ell,\qquad
D_{\ell-2k,\ell}R^{\ell-2k}\leq\mathsf a_k\zeta^kh_\ell.
\end{gather*}
\end{lemma}
\begin{proof}
Take the connection formula \cite[18.18.16]{DLMF} for $C^{(\mu)}_K$ in terms of the $C^{(\lambda)}_{K-2i}$ with $\lambda=\frac52$, divide by $\mu$ and let $\mu\to0$; by \cite[18.7.25]{DLMF}, $C^{(\mu)}_K/\mu\to(2/K)T_K$ for $K\geq1$, and $(\mu)_{K-i}/\mu\to(K-i-1)!$ since $K-i\geq1$. As $\lambda(\lambda+1)_{K-i}=(\lambda)_{K-i+1}$, this gives the stated $C_{\ell K}$. For $\ell>0$ it gives $|C_{\ell K}|D_{\ell\ell}=e_i\varrho_i$, where $\varrho_0=1$ and
\[
\varrho_i=\frac K{K-i+\lambda}\prod_{q=1}^{i-1}\frac{\ell+q}{\ell+q+\lambda}
\leq\frac{K(\ell+1)}{(K-i+\lambda)(\ell+i)}\leq1\quad(i\geq1).
\]
Telescoping, using $(\ell+q)/(\ell+q+\lambda)\leq(\ell+q)/(\ell+q+1)$ since $\lambda\geq1$, gives the first inequality; the second uses
\[
(K-i+\lambda)(K-i)-K(K-2i+1)
=i^2-\lambda i+(\lambda-1)K\geq i^2+(\lambda-2)i\geq0.
\]
This uses $K\geq2i$ and $\lambda\geq2$. For $\ell=0$ one has $D_{00}=1$, hence $|C_{0K}|D_{00}=\frac12e_i\varrho_i\leq e_i$. Since $\mathsf a_k$ is nondecreasing ($(k+\lambda)/(k+1)\geq1$),
$h_\ell\leq D_{\ell\ell}R^\ell\sum_k\mathsf a_k\zeta^k=D_{\ell\ell}R^\ell(1-\zeta)^{-\lambda}$, proving the bound. In $(1-\zeta)^{5/2}$ only the coefficients of degrees $0$ and $2$ are positive; taking absolute values gives $\sum_ie_i\zeta^i=2+15\zeta^2/4-(1-\zeta)^{5/2}$. The last two inequalities follow by applying the positive functional $T_m\mapsto R^m$ to $\xi Y_\ell=\lambda^+_\ell Y_{\ell+1}+\lambda^-_\ell Y_{\ell-1}$, and by $D_{\ell-2k,\ell}\leq\mathsf a_kD_{\ell\ell}$.
\end{proof}

\begin{lemma}[Graded multiplication]\label{r7:lem-mult}
Theorem~\ref{r5:thm-mult} holds in the spaces above, with $\kappa_7$ in place of $\kappa_5$ and harmonic fraction
$\vartheta_{d_a}(\ell,s)=\mathbf1_{d_a\geq s}(d_a/(\ell+d_a+9/2))^s$. The expansion of $H_at^k\Psi_{\ell,s}$ is nonnegative, has angular index $v\geq|\ell-a|$, radial index $s'\geq s-d_a$, degree at most $\ell+2s+d_a$, and its radial lowering is dominated by $\mathrm{Binomial}(d_a,1/2)$, where $d_a=a+2k$. Theorem~\ref{r5:thm-transfer} holds with the present $\Upsilon$.
\end{lemma}
\begin{proof}
The linearisation formula \cite[18.18.22]{DLMF}, at $\lambda=5/2>0$, gives nonnegative coefficients $G_{a\ell v}$ of $Y_aY_\ell$, with sum one and $\sum_vG_{a\ell v}h_v\leq h_ah_\ell$. In the proof of Theorem~\ref{r5:thm-mult}(b), every Jacobi raising or lowering step remains a probability vector because $\beta_\ell>0$. The factors in its harmonic estimate become $\beta_v+1=v+7/2$ and $(d_a-i)/(\ell+d_a+9/2+i)$; the binomial domination is the same sequence of lowering steps of probability at most $1/2$ as in Lemma~\ref{r3:lem-radial}. The degree bound, the submultiplicativity of $\Upsilon$, and Lemma~\ref{r7:lem-conn} prove the norm bounds in part~(c). Part~(d) involves only the meridian expansions, whose degrees are unchanged. Finally, the proof of Theorem~\ref{r5:thm-transfer} uses only that $\Upsilon$ is nondecreasing and submultiplicative; these properties hold for its square as well.
\end{proof}

\begin{lemma}[Compatible inverse and stencils]\label{r7:lem-K}
For $s\geq1$, let $d=\ell+2s+\lambda$ and $E=x\cdot\nabla$. Then
\begin{align*}
K_7\Psi_{\ell,s}&=\frac{\Psi_{\ell,s-1}}{4d(d+1)}-\frac{\Psi_{\ell,s}}{2d(d+2)}+\frac{\Psi_{\ell,s+1}}{4(d+1)(d+2)},\\
\partial_{x_1}K_7\Psi_{\ell,s}&=\frac{\lambda^+_\ell(\Psi_{\ell+1,s}-\Psi_{\ell+1,s-1})+\lambda^-_\ell(\Psi_{\ell-1,s+1}-\Psi_{\ell-1,s})}{2(d+1)},\\
EK_7\Psi_{\ell,s}&=-\frac{d+\lambda}{4d(d+1)}\Psi_{\ell,s-1}+\frac\lambda{2d(d+2)}\Psi_{\ell,s}+\frac{d+2-\lambda}{4(d+1)(d+2)}\Psi_{\ell,s+1}.
\end{align*}
One has $\Delta K_7\Psi_{\ell,s}=\Psi_{\ell,s}$ and zero value and gradient on $\partial B$. The norms of $K_7$, $\partial_{x_1}K_7$, $EK_7$ and $y\partial_yK_7=EK_7-x_1\partial_{x_1}K_7$, from $\Gc$ to $\mathcal L$, are bounded by
\begin{gather*}
\mathsf k_K=\frac1{d_0(d_0+2)}+\frac{\Upsilon(2)-1}{4(d_0+1)(d_0+2)},\qquad
\mathsf k_D=\frac{R(1+\Upsilon(1))}{2(d_0+1)},\\
\mathsf k_E=\frac{d_0+\lambda}{2d_0(d_0+1)}+\frac{(\Upsilon(2)-1)(d_0+2-\lambda)}{4(d_0+1)(d_0+2)},\qquad
\mathsf k_Y=\mathsf k_E+R\Upsilon(1)\mathsf k_D,\quad d_0=\frac92.
\end{gather*}
\end{lemma}
\begin{proof}
Write $K_7\Psi_{\ell,s}=H_\ell k(t)$ and $\beta=\beta_\ell$. Then $\Delta K_7\Psi_{\ell,s}=\Psi_{\ell,s}$ reads $(t^{\beta+1}k')'=\frac14t^\beta J^\beta_s$. By the Rodrigues formula the polynomial solution is
\[
t^{\beta+1}k'=\frac1{4\,s!}\frac{d^{s-1}}{dt^{s-1}}\big[t^{\beta+s}(t-1)^s\big].
\]
Thus $\deg k=s+1$, $k'(1)=\frac14\int_0^1t^\beta J^\beta_s\,dt=0$ by Lemma~\ref{r3:lem-jacobi}(a), and we normalise $k(1)=0$. For a polynomial $\varphi$ of degree at most $s-1$, two integrations by parts with $k(1)=k'(1)=0$ give
\[
\int_0^1t^\beta k\,(t\varphi''+(\beta+1)\varphi')\,dt
=\frac14\int_0^1t^\beta J^\beta_s\varphi\,dt=0.
\]
The map $\varphi\mapsto t\varphi''+(\beta+1)\varphi'$ sends the polynomials of degree at most $s-1$ onto those of degree at most $s-2$. Hence $k$ is orthogonal for the weight $t^\beta$ to the polynomials of degree at most $s-2$, so it lies in the span of $J^\beta_{s-1},J^\beta_s,J^\beta_{s+1}$. Its leading coefficient and the two boundary conditions give the first stencil. The Jacobi identities give
\[
k'=\frac{J_s^{\beta_\ell+1}-J_{s-1}^{\beta_\ell+1}}{4(d+1)},\qquad
\beta_\ell k+tk'=\frac{J_{s+1}^{\beta_\ell-1}-J_s^{\beta_\ell-1}}{4(d+1)}.
\]
Substitute these into $\partial_{x_1}(H_\ell k)=2\lambda^+_\ell H_{\ell+1}k'+2\lambda^-_\ell H_{\ell-1}(\beta_\ell k+tk')$. Substitution into $E(H_\ell k)=H_\ell(2(\beta_\ell k+tk')-(\beta_\ell+\lambda)k)$ gives the third stencil. All formulas include $\ell=0$, with $\lambda^-_0=0$. For the norm bounds, the absolute column sums of the first and third stencils are $1/(d(d+2))$ and $(d+\lambda)/(2d(d+1))$. The degree increases are at most two; only the last entry requires $\Upsilon(2)$ instead of one. The derivative stencil has degree shifts $\pm1$ and uses $\lambda^+_\ell h_{\ell+1}+\lambda^-_\ell h_{\ell-1}\leq Rh_\ell$. These bounds decrease for $d\geq d_0$; multiplication by $x_1$ costs $R\Upsilon(1)$.
\end{proof}

\begin{proposition}[$C^1$ reconstruction]\label{r7:prop-C1}
The series $K_7g$ converges in $C^1(\overline B)$ for $g\in\Gc$, has zero value and gradient on $\partial B$, satisfies $\Delta K_7g=\iota(g)$ in distributions, and
\[
\nrm{K_7g}_{C^1(\overline B)}\leq3.73\cdot10^{10}\nrm g_{\mathcal L},
\]
where the $C^1$ norm is the maximum of the two suprema of value and Euclidean gradient.
\end{proposition}
\begin{proof}
We use the correspondence between polynomials on the ball and $O(2)$-invariant spherical harmonics on $S^8$ (cf.\ \cite[Ch.~4]{DX14}); its measure and degree argument is as follows. Let $V_m$ be the orthogonal complement of polynomials of degree at most $m-1$ in those of degree at most $m$, in $L^2(B)$. An $O(2)$-invariant homogeneous harmonic polynomial of degree $m$ on $\R^7\times\R^2$ restricts, by $|z|^2=1-|x|^2$, to $V_m$. Indeed, the parametrisation $(x,\theta)\mapsto(x,\sqrt{1-|x|^2}(\cos\theta,\sin\theta))$ has surface element $dx\,d\theta$, and
$\int_{S^8}FG=2\pi\int_Bfg$. Different harmonic degrees are orthogonal. Conversely, homogenise each polynomial of degree at most $m$ to degree $m$ or $m-1$, according to parity, using $|x|^2+|z|^2$, and apply the equivariant Fischer decomposition. The restrictions therefore exhaust $V_0\oplus\cdots\oplus V_m$, and the map is injective by the displayed norm identity.

As in Lemma~\ref{r3:lem-embed}(a), by Fischer decomposition and Lemma~\ref{r3:lem-jacobi}(a), the $\Psi_{\ell,s}$ are pairwise orthogonal in $L^2(B)$ and $\Psi_{\ell,s}$ is orthogonal to all polynomials of degree less than $\ell+2s$. The map $\iota\colon\mathcal L\to L^2(B)$ is continuous, since $\nrm{\Psi_{\ell,s}}^2_{L^2(B)}\leq|S^6|/7$ by the formula below. Let $N_m^{(q)}$ denote the dimension of degree-$m$ harmonics on $S^{q-1}$. The spherical addition theorem and $|S^8|=(2\pi/7)|S^6|$ give $|f(x)|^2\leq7N_m^{(9)}\nrm f_{L^2(B)}^2/|S^6|$ for $f\in V_m$. Jacobi orthogonality gives, with $D=\ell+2s$ and $d=D+5/2$,
\[
\nrm{\Psi_{\ell,s}}_{L^2(B)}^2=\frac{|S^6|}{2N_\ell^{(7)}(D+7/2)},\qquad
\nrm{\nabla K_7\Psi_{\ell,s}}_{L^2(B)}^2=\frac{\nrm{\Psi_{\ell,s}}_{L^2(B)}^2}{2d(d+2)}.
\]
Integration by parts and parity put every Cartesian derivative of $K_7\Psi_{\ell,s}$ in $V_{D-1}\oplus V_{D+1}$. Cauchy--Schwarz hence bounds its squared gradient supremum by
\[
\Gamma_7(\ell,s)^2=\frac{7(N_{D-1}^{(9)}+N_{D+1}^{(9)})}{4N_\ell^{(7)}(D+5/2)(D+7/2)(D+9/2)}
\leq B_7(s):=\frac{\prod_{i=2}^7(2s+i)}{720(2s+5/2)(2s+7/2)}.
\]
For the last inequality use $N_{D-1}^{(9)}\leq N_{D+1}^{(9)}$ and the harmonic dimension formula: with $D=\ell+2s$ the resulting factors $(D+1+i)/(\ell+i)$ ($1\leq i\leq4$), $(D+6)/(D+5/2)$, $(D+7)/(D+7/2)$ and $1/(2\ell+5)$ all decrease in $\ell$; their value at zero is $B_7(s)$. Integrating along a radius gives the same bound for the value. For $c=10^6$ and $x=2s$, $(x+2)(x+3)\leq(x+5/2)(x+7/2)$ and $x+i\leq x+c$ for $4\leq i\leq7$, so $B_7(s)/\Upsilon(2s)^2\leq c^4/720$. Since $373^2\cdot720>10^8$, $c^2/\sqrt{720}<3.73\cdot10^{10}$. Since $h_\ell\geq1$ and $\Upsilon(D)\geq\Upsilon(2s)$, absolute summation proves convergence and the traces; $L^2$ convergence of $g$ passes the Laplacian identity to distributions. This constant is used only in reconstruction.
\end{proof}

\subsection{The equation, Taylor bounds and inverse}\label{r7:sec-F}
The exact dyadic centre $x^\circ=(g^\circ,c^\circ)$ has $(M,S)=(210,84)$, $\rho=731659610854495\cdot2^{-43}$, $b=c^\circ_1\approx1.0000971073$ and $\sigma=25\rho^2$. Its $8904$ field coefficients and $106$ shape coefficients are fixed exact data. Put $\psi_c(w)=\sum_{j\ \mathrm{odd}}c_jw^j$, $p_c=\psi'_c$, $q_c=\im\psi_c/y$, $W=K_7g$, $W_1=\partial_{x_1}W$ and $W_2=y\partial_yW$. Define
\begin{equation}\label{r7:eq-F}
\Fc(g,c)=b^{-1}\big[q_cg+5\big((q_c)_{x_1}W_1+((q_c)_y/y)W_2\big)+\rho^2q_c|p_c|^2(1+W)\big].
\end{equation}

\begin{lemma}\label{r7:lem-F}
$\Fc\colon\mathcal X\to\mathcal Y$ is a bounded quartic polynomial. If $\Fc(g,c)=0$ and $\re\psi'_c>0$ on a disc of radius $R_0$, $1<R_0<R$, then $U=1+\iota(K_7g)\in C^1(\overline B)$ has $U=1$, $\nabla U=0$ on $\partial B$ and satisfies
\[
\int_Bq_c^5(\nabla U\cdot\nabla\varphi-\rho^2|p_c|^2U\varphi)=0\quad(\varphi\in C_c^\infty(B)).
\]
\end{lemma}
\begin{proof}
Lemma~\ref{r7:lem-mult} and the stencils prove boundedness, and the values are even in $x_1$, because $q_c$, $|p_c|^2$, $W$ and $y\partial_yW$ are even in $x_1$ while $\partial_{x_1}q_c\,\partial_{x_1}W$ is a product of two odd functions. The meridian series make $q_c$ analytic in $(x_1,|x'|^2)$, with $q_c=\int_0^1\re\psi'_c(x_1+iuy)\,du>0$. Proposition~\ref{r7:prop-C1} gives $\Delta W=\iota(g)$ and the boundary traces. For any test function, integrate the distributional identity against $q_c^5\varphi$ to obtain
\[
\int_Bq_c^5(\nabla U\cdot\nabla\varphi-\rho^2|p_c|^2U\varphi)
=-\int_B bq_c^4\varphi\,\iota(\Fc(g,c))=0.
\]
The factor five comes from $\nabla q_c^5=5q_c^4\nabla q_c$; the test function need not be axisymmetric.
\end{proof}

Let $G_0\geq\nrm{g^\circ}_{\mathcal L}$, $C_0\geq\nrm{c^\circ}_{\mathcal E}$ and
\begin{align*}
B'&=\frac1R+5\Big(\frac{2\mathsf k_D}{R^2-1}+\frac{4R\mathsf k_Y}{(R^2-1)^2}\Big),\\
c_2&=\frac{\kappa_7}b\Big[\frac{B'}\sigma+\frac{3\rho^2\mathsf k_KC_0^2}{R^3\sigma}+\frac{3\rho^2C_0(1+\mathsf k_KG_0)}{R^3\sigma^2}\Big],\\
c_3&=\frac{\kappa_7\rho^2}{bR^3}\Big[\frac{1+\mathsf k_KG_0}{\sigma^3}+\frac{3C_0\mathsf k_K}{\sigma^2}\Big],\qquad
c_4=\frac{\kappa_7\rho^2\mathsf k_K}{bR^3\sigma^3}.
\end{align*}
For $\nrm h_{\mathcal X}\leq r$ the Taylor remainder and derivative remainder of $\Fc$ are bounded by $\sum_{k=2}^4c_kr^k$ and $\sum_{k=2}^4kc_kr^{k-1}$. To prove this, write $h=(\delta g,\delta c)$ with norms at most $r,r/\sigma$. The raw multiplier bounds (Theorem~\ref{r5:thm-mult}(d), which does not depend on $n$) are $\nrm{\delta q}_{\mathrm{ch}}\leq\nrm{\delta c}_{\mathcal E}/R$, $\nrm{\delta q_{x_1}}_{\mathrm{ch}}\leq2\nrm{\delta c}_{\mathcal E}/(R^2-1)$ and $\nrm{\delta q_y/y}_{\mathrm{ch}}\leq4R\nrm{\delta c}_{\mathcal E}/(R^2-1)^2$. The first two terms of \eqref{r7:eq-F} give $\kappa_7B'r^2/(b\sigma)$. In the potential term $q^\circ$ has raw norm and $p^\circ$ analytic coefficient norm at most $a=C_0/R$, and their increments at most $e=r/(\sigma R)$. Expand $(a+\eta e)^3(1+\mathsf k_K(G_0+\eta r))$ in $\eta$; the coefficients of orders two, three and four give the remaining terms. Convert each raw product to solid multipliers once, costing $\kappa_7$. The same multilinear bounds, with one increment left free, prove the derivative estimate.

The finite inputs are $(\ell,s)$ with even $\ell\leq210$, $1\leq s\leq84$, and odd $j\leq211$. The finite rows are the same field rows and the harmonic rows $\mathcal T_m$, $m\leq210$ even. Thus the scaled finite Jacobian $\widetilde J$ has size $9010$. Let $\widehat A$ be an exact dyadic approximate inverse. In tail coordinates $\hat a_N=(N+1)R^{N+1}c_{N+1}$, $\hat y_m=R^my_m$ ($N,m$ even), the principal part of the shape column is as in Lemma~\ref{r3:lem-principal} and Section~\ref{r5:sec-quartic}. The input $e_{N+1}$ has $\hat a_N=(N+1)R^{N+1}$, and its principal polynomial $\frac{2(N+1)\rho^2}b\,q^\circ\re(\overline{p^\circ}w^N)$ has boundary values $\frac{2(N+1)\rho^2}b\sum_\delta s_\delta T_{|N+2\delta|}$; let $\mathsf Me_{N+1}$ be the harmonic element with these boundary values. For $N,m\geq M+2$, the output coordinate $\hat y_m$ in row $m=N+2\delta$, divided by the input coordinate $\hat a_N$, is $\frac{2\rho^2}{bR}s_\delta R^{2\delta}$. Thus the part $T$ of $\mathsf M$ from the shape tail to the harmonic tail rows is
\[
T=\frac{2\rho^2s_0}{bR}(I+E),\quad
(E\hat a)_m=\sum_{N\ne m}\frac{s_{(m-N)/2}}{s_0}R^{m-N}\hat a_N,
\quad m,N\geq212\text{ even},
\]
where $s_\delta=\sum_{k-a=2\delta}p_aq_k$, $p_a=(a+1)c^\circ_{a+1}$, $q_k=\sum_{j>|k|}c^\circ_j$, $s_0>0$ and $-210\leq\delta\leq105$. Since $N+2\delta\geq-208$ for $N\geq212$, negative frequencies fold only into finite rows. The finite fold $\mathcal B$ collects the modes $|N+2\delta|\leq210$. The bound
\[
\eps_\Upsilon=\sum_{\delta\ne0}|s_\delta/s_0|R^{2\delta}\Upsilon((2\delta)_+)<1
\]
gives $\nrm{T^{-1}}\leq\tau_T=bR/(2\rho^2s_0(1-\eps_\Upsilon))$, including a Neumann remainder of norm at most $\tau_T\eps_\Upsilon^{L+1}$ after powers through $L$. Define
\begin{equation}\label{r7:eq-A}
Ay=\widehat A(y_f-\mathcal BT^{-1}y_t^H)+T^{-1}y_t^H+y_t^N.
\end{equation}
Let $\alpha_\iota$ bound the normalised finite-row column norms of $\widehat A$, set $\alpha^H_m=\alpha_{\mathcal T_m}$ for $m\leq M$, and put
\[
\rho_{\mathcal B}=\sup_{N>M}\frac{\sum_{m\leq M}\alpha^H_m\Upsilon(m)|\hat{\mathcal B}_{mN}|}{\Upsilon(N)},\qquad
A_t=\tau_T(\sigma+\rho_{\mathcal B}),\qquad\alpha^H_m=A_t\ (m>M).
\]
Then $\nrm A\leq\max(\max_\iota\alpha_\iota,A_t,1)$, $A\Psi_{\ell,s}=e_{(\ell,s)}$ for field tail inputs, and $A\mathsf Me_j=e_j$ for shape tail inputs, where $\mathsf Me_j$ is the harmonic extension of the principal boundary polynomial. If $\nrm{I-\widehat A\widetilde J}<1$, $\widehat A$ is invertible; the upper triangular form \eqref{r7:eq-A}, with invertible $T$ and the identity on field tails, makes $A$ a bounded bijection.

\subsection{All field columns}\label{r7:sec-field}
There are $106$ radial classes $s\geq S_\ell$ for even $\ell\leq210$, where $S_\ell=90,89,88,87,86,85$ on the ranges $0$--$2$, $4$--$10$, $12$--$18$, $20$--$26$, $28$--$36$, $38$--$210$, respectively. The intervening $51$ field inputs are explicit. For even $212\leq\ell\leq386$, the $88$ gap classes have $S_\ell=3$ on $212$--$218$, $2$ on $220$--$366$ and $1$ on $368$--$386$, with $82$ explicit inputs below these thresholds. The uniform angular class is $\ell\geq388$, $s\geq1$.

For each finite input $e$, the complete column $y=D\Fc(x^\circ)e$ gives the normalised defect bound $\delta_e+(A_t\nrm{y_t^H}+\nrm{y_t^N})/w_e$, where $\delta_e$ bounds its finite inverse defect. For an exceptional field input the bound is $(\nrm{\widehat Ay_f}+A_t\nrm{y_t^H}+\nrm{y_t^N-e})/w_e$. For a shape input $e_j$ with $213\leq j\leq701$ odd, use
\[
\frac{\nrm{\widehat A(y_f-\mathcal Ba)}+\sigma\nrm{a-e_j}+\nrm{y_t^N}+(\sigma+\rho_{\mathcal B})\mathrm{rem}}{w_j},
\]
where $a$ approximates $T^{-1}y_t^H$ and $\mathrm{rem}$ bounds its remainder. The same numerator without $e_j$, applied to $y=\Fc(x^\circ)$, bounds $Y_0=\nrm{A\Fc(x^\circ)}$.

For the infinite field classes define $\chi_1=q^\circ/b-1$, $\chi_2=5q^\circ_{x_1}/b$, $\chi_3=5(q^\circ_y/y)/b$, $\chi_4=x_1\chi_3$, $\chi_5=\rho^2q^\circ|p^\circ|^2/b$. Let $\tau_i(d)$ be the graded solid norm of the degree-$d$ part of $\chi_i$. For each angular grid row $\nu$ let $g_\nu(t)$ be the maximum of one and the normalised finite field-row column norms of $A$ over $\ell'\geq2\nu$, $s'\geq t+1$, with $g_\nu(t)=1$ outside the finite grid, including $t\geq S$. Let $\beta^H(L)$ be a nonincreasing upper bound for
\[
\beta^H_v=\frac{\sum_mD_{mv}R^m\Upsilon(m)\alpha^H_m}{h_v\Upsilon(v)}\quad(v\geq L\text{ even}).
\]
Odd profile indices carry the maximum of their even neighbours. For even $v>1400$ use $\beta^H_v\leq A_t+\sum_{m\leq M}(\alpha^H_m-A_t)_+\mathsf a_{k_m}\zeta^{k_m}$, $k_m=(1402-m)/2$. This follows from Lemma~\ref{r7:lem-conn}; $\mathsf a_k\zeta^k$ decreases for $k\geq k_m$, since $(k+5/2)\leq(k+1)R^2$ at these starting indices.

For $\nu=\lceil(\ell-d)_+/2\rceil$, $a=g_\nu(\max(1,s-d)-1)$, $b_d=\beta^H((\ell-d)_+)$, $\theta=\vartheta_d(\ell,s)$ and $X_d\sim\mathrm{Binomial}(d,1/2)$, put
\begin{align*}
\mathfrak B_d(\ell,s)=\min\Big\{&a+(b_d-a)_+\theta,\\
&1+\sum_{t=0}^{S-1}(g_\nu(t)-g_\nu(t+1))\Pr(X_d\geq s-t-1)
 +(b_d-g_\nu(0))_+\Pr(X_d\geq s)\Big\}.
\end{align*}
Set $\mathfrak B^{(i)}(\ell,s)=\sum_d\tau_i(d)\mathfrak B_d(\ell,s)$. With $d_\ell=\ell+2s+\lambda$, define
\begin{align*}
\mathfrak z(\ell,s)={}&\mathfrak B^{(1)}(\ell,s)
+\frac{\Upsilon(1)}{2(d_\ell+1)}\sum_{i=2,4}\Big[\lambda^+_\ell\frac{h_{\ell+1}}{h_\ell}
 (\mathfrak B^{(i)}(\ell+1,s)+\mathfrak B^{(i)}(\ell+1,s-1))\\
&\hphantom{\mathfrak B^{(1)}(\ell,s)+\frac{\Upsilon(1)}{2(d_\ell+1)}\sum_{i=2,4}\Big[}
+\lambda^-_\ell\frac{h_{\ell-1}}{h_\ell}(\mathfrak B^{(i)}(\ell-1,s+1)+\mathfrak B^{(i)}(\ell-1,s))\Big]\\
&+\frac{(d_\ell+\lambda)\mathfrak B^{(3)}(\ell,s-1)}{4d_\ell(d_\ell+1)}
+\frac{\lambda\mathfrak B^{(3)}(\ell,s)}{2d_\ell(d_\ell+2)}
+\frac{\Upsilon(2)(d_\ell+2-\lambda)\mathfrak B^{(3)}(\ell,s+1)}{4(d_\ell+1)(d_\ell+2)}\\
&+\frac{\mathfrak B^{(5)}(\ell,s-1)}{4d_\ell(d_\ell+1)}
+\frac{\mathfrak B^{(5)}(\ell,s)}{2d_\ell(d_\ell+2)}
+\frac{\Upsilon(2)\mathfrak B^{(5)}(\ell,s+1)}{4(d_\ell+1)(d_\ell+2)}.
\end{align*}
Omit the $\ell-1$ terms when $\ell=0$. In $\mathfrak z_{\mathrm u}$ replace each bracket by $R$ times the maximum of its two parenthesised sums.

\begin{lemma}[Field envelope]\label{r7:lem-envelope}
For every field tail input, its normalised defect is at most $\mathfrak z(\ell,s)\leq\mathfrak z_{\mathrm u}(\ell,s)$. For fixed $\ell$, $\mathfrak z$ decreases in $s$; $\mathfrak z_{\mathrm u}$ decreases in both $\ell$ and $s$. Thus the infinite classes are bounded at $(\ell,S_\ell)$ and $(388,1)$.
\end{lemma}
\begin{proof}
The column defect is $-Av$, with $v=\chi_1\Psi+\chi_2\partial_{x_1}K_7\Psi+\chi_3EK_7\Psi-\chi_4\partial_{x_1}K_7\Psi+\chi_5K_7\Psi$. Lemma~\ref{r7:lem-mult} bounds each multiplier by its degree norm. Its support and harmonic fraction give the first branch of $\mathfrak B_d$; binomial domination of the radial lowering gives the second, by summation of the layers $g_\nu(t)-g_\nu(t+1)$. The index $\max(1,s-d)-1$ selects precisely $s'\geq\max(1,s-d)$. Applying the stencils proves the formula, and Lemma~\ref{r7:lem-conn} proves the uniform replacement.

The first branch is nondecreasing in $a,b_d,\theta$, each of which decreases with $\ell,s$. For the second, put $p_{-1}=\Pr(X_d\geq s)$ and $p_t=\Pr(X_d\geq s-t-1)$. Summation by parts rewrites it as
\[
1-p_{S-1}+\max(b_d,g_\nu(0))p_{-1}+\sum_{t=0}^{S-1}g_\nu(t)(p_t-p_{t-1}).
\]
All coefficients of the grid values are nonnegative, proving decrease in $\ell$; its original expression proves decrease in $s$. The stencil coefficients decrease for $d_\ell\geq9/2$. For the only less immediate numerator $d_\ell+2-\lambda$, the derivative has numerator $-d_\ell^2+d_\ell+7/2<0$; for $d_\ell+\lambda$ it is $-d_\ell^2-2\lambda d_\ell-\lambda<0$. The nonuniform $h$-ratios are fixed in a radial class; the uniform replacement removes them. This proves the stated monotonicity.
\end{proof}

\subsection{Far shape columns}\label{r7:sec-far}
Let $W^\circ=K_7g^\circ$, $W_1^\circ=\partial_{x_1}W^\circ$, $Z^\circ=y^{-1}\partial_yW^\circ$ and $H^\circ=g^\circ+\rho^2|p^\circ|^2(1+W^\circ)$. These are polynomials, with $W^\circ$ vanishing to order two and $W_1^\circ,Z^\circ$ to order one at $t=1$. The axis identity
\[
\frac{\partial_yH_\ell}y=-\sum_{k=\ell-2,\ell-4,\ldots\geq0}
\frac{2(k+\lambda)C_k^{(\lambda)}(1)}{C_\ell^{(\lambda)}(1)}H_kt^{(\ell-2-k)/2}
\]
follows from $\xi C_\ell'=\ell C_\ell+C_{\ell-1}'$ and $C_{\ell-1}'=\sum_k2(k+\lambda)C_k^{(\lambda)}$, where $0\leq k\leq\ell-2$ and $k\equiv\ell-2\pmod2$ (see \cite[18.9]{DLMF}). Differentiating \eqref{r7:eq-F} in $c_j$, $j=N+1$, gives
\begin{equation}\label{r7:eq-farcol}
\begin{split}
bD\Fc(x^\circ)e_j={}&r^NU_NH^\circ+5(N+1)r^{N-1}U_{N-1}W_1^\circ\\
&+5\Big(Nr^NT_N-\frac{r^NU_N+tr^{N-2}U_{N-2}}2\Big)Z^\circ\\
&+2(N+1)\rho^2q^\circ\re(\overline{p^\circ}w^N)W^\circ+b\Pi_j,
\end{split}
\end{equation}
where $b\Pi_j=2(N+1)\rho^2q^\circ\re(\overline{p^\circ}w^N)$. Indeed $\delta q=r^NU_N$, $\delta q_{x_1}=(N+1)r^{N-1}U_{N-1}$ and $\delta(q_y/y)=-2r^{N-2}C^{(2)}_{N-2}$; use $2(1-\xi^2)C^{(2)}_{N-2}=(U_N+U_{N-2})/2-NT_N$ for the third term.

Expand the principal polynomial as $q^\circ\re(\overline{p^\circ}w^N)=\sum_{(\eta,e)}\gamma_{\eta,e}t^er^{N+\eta}T_{N+\eta}$, with shifts $-420\leq\eta\leq210$. In estimates $\gamma_{\eta,e}$ denotes the sum of the moduli of its individual contributions. Its harmonic model replaces $t^er^KT_K$ by $\mathcal T_K$.

Let $\Gamma_F(v)\geq1$ and $\Gamma_H(v)\geq A_t$ be nonincreasing bounds for all normalised field-row columns and for $\beta^H_{v'}$, respectively, with $v'\geq v$. Put $\Gamma_* =\max(\Gamma_F,\Gamma_H)$ and $\Gamma_{\max}=\Gamma_*(0)$. Fix $I=140$, $I_1=24$, $i_c=100$, $I_+=I+1$, $i_1=\lceil(K-M)/2\rceil$ and $S_p(j)=\sum_{k\geq0}\zeta^k(j+k)^p$. Define
\[
G_p(K)=\mathsf w_{I_+}\begin{cases}
\Gamma_{\max}S_p(I_+),&i_1\leq I_+,\\
\Gamma_*(M+1)S_p(I_+)+\Gamma_{\max}\zeta^{i_1-I_+}S_p(i_1),&i_1>I_+,
\end{cases}\quad p=0,1,2.
\]
These bound $\sum_{I<i\leq K/2}\mathsf w_ii^p\Gamma_*(K-2i)$: $e_i$ decreases for $i\geq I_+$, and $K-2i>M$ for $i<i_1$.

\begin{lemma}[Geometric switch]\label{r7:lem-switch}
If $\Gamma_*(M+1)S_p(I_+)\leq\Gamma_{\max}I_+^p$ for $p=0,1,2$, each $G_p$ is nonincreasing in $i_1$, including its first switch.
\end{lemma}
\begin{proof}
The identity $S_p(I_+)-\zeta S_p(I_++1)=I_+^p$ gives the first switch inequality. In the second branch increasing $i_1$ drops a nonnegative summand from $\sum_{i\geq i_1}\zeta^{i-I_+}i^p$. For the present profiles the three assumptions hold with $\Gamma_*(M+1)S_p(I_+)/(\Gamma_{\max}I_+^p)\approx7.9\cdot10^{-4},\ 8.0\cdot10^{-4},\ 8.1\cdot10^{-4}$ for $p=0,1,2$.
\end{proof}
\begin{lemma}\label{r7:lem-lap}
For $K\geq2$, $\sum_ii\,(K-i+\frac52)\,C_{K-2i,K}\,Y_{K-2i}=-\frac{5K}4U_{K-2}$.
\end{lemma}

\begin{proof}
In $\R^7$, $\Delta(t^iH_v)=4i(i+v+\frac52)t^{i-1}H_v$. The function $r^KT_K(\xi)=\re w^K$ is harmonic in the meridian variables, so $\Delta\re w^K=\frac5y\partial_y\re w^K=-\frac{5K}y\im w^{K-1}=-5Kr^{K-2}U_{K-2}(\xi)$. Apply $\Delta$ to $r^KT_K=\sum_iC_{K-2i,K}H_{K-2i}t^i$ and restrict to $t=1$.
\end{proof}

\begin{lemma}\label{r7:lem-raw}
Let $K,e,m$ be nonnegative integers with $K$ even, $\beta=K+\frac52$, $c_m=\max(2,m+1)$, $\mathsf F_m(\beta')=m!/(\beta'/2+1)_m$, and
\[
\mathsf f(x)=\frac{\beta-2x+1}{\prod_{k=1}^{m+1}(\beta-x+e+k)},\qquad\varrho_2(i)=\mathsf f(i)-\mathsf f(0)-\mathsf f'(0)\,i .
\]
All connection sums below have $i\leq K/2$. Let $u=t^er^KT_K(\xi)(1-t)^m$ if $m\geq1$, and $u=t^er^KT_K(\xi)-\mathcal T_K$ if $m=0$ and $K>M$. Then
\begin{multline*}
\frac{\nrm{Au}_{\mathcal X}}{R^K\Upsilon(K+2e+2m)}\leq\sum_i\mathsf w_i\,\mathsf g_i\,\Gamma_F(K-2i)\\
+m!\Big(\mathsf f_0\,\alpha^H(K)+\frac54|\mathsf f'(0)|\,\mathsf u(K)+\sum_{i\geq1}\mathsf w_i\Big(\frac{|\mathsf f'(0)|\,i^2}\beta+|\varrho_2(i)|\Big)\Gamma_H(K-2i)\Big),
\end{multline*}
where $\mathsf g_i=\mathsf F_m(\beta-2i)$ and $\mathsf f_0=\mathsf f(0)$ if $m\geq1$, $\mathsf g_i=(e+i)/(\beta-i+e+1)$ and $\mathsf f_0=1-\mathsf f(0)=e/(\beta+e+1)$ if $m=0$, and $\mathsf u(K)=2A_t\zeta/(1-\zeta)+2\sum_{n\leq\min(M,K-2),\,n\equiv K}R^{n-K}(\alpha^H_n-A_t)_+$. Here $\alpha^H(K)=\max(A_t,\alpha^H_K)$ for $K\leq M$ and $\alpha^H(K)=A_t$ for $K>M$; $\mathcal H_v=\sum_nD_{nv}\mathcal T_n$ denotes the harmonic element with trace $Y_v$. Moreover, for $i\leq K/2$,
\[
|\mathsf f'(0)|\leq\frac{c_m\mathsf f(0)}{\beta+1},\qquad|\varrho_2(i)|\leq\frac{i^2}2\cdot\frac{(\beta+1)(c_m^2+m+5)}{(\beta-2i+1)^2\prod_{k=1}^{m+1}(\beta-i+e+k)} .
\]
\end{lemma}

\begin{proof}
Let $v_i=K-2i$, with every connection sum restricted to $0\leq i\leq\lfloor K/2\rfloor$. By definition of $C_{\ell K}$, $u=\sum_iC_{v_i,K}H_{v_i}t^{e+i}(1-t)^m$, minus $\mathcal T_K=\sum_iC_{v_i,K}H_{v_i}$ if $m=0$, and all outputs have degree at most $K+2e+2m$; so it suffices to bound the norms with $\Upsilon\equiv1$ (Lemma~\ref{r7:lem-mult}). \emph{Field layer.} By Lemma~\ref{r3:lem-jacobi}(c),(d), the $J^{\beta_{v_i}}$-coefficients of $t^{e+i}(1-t)^m$ have $\ell^1$ norm at most $\mathsf F_m(\beta_{v_i})=\mathsf F_m(\beta-2i)$, and those of $t^{e+i}-1$ with $s\geq1$ have sum $(e+i)/(\beta_{v_i}+e+i+1)$; with $|C_{v_i,K}|h_{v_i}\leq\mathsf w_iR^K$ this gives the first sum. \emph{Harmonic layer.} By Lemma~\ref{r3:lem-jacobi}(a), the coefficient of $J^{\beta_v}_0$ in $t^\mu(1-t)^m$ is $(\beta_v+1)\int_0^1t^{\beta_v+\mu}(1-t)^mdt=m!(\beta_v+1)/\prod_{k=1}^{m+1}(\beta_v+\mu+k)$, which for $v=v_i$ and $\mu=e+i$ equals $m!\,\mathsf f(i)$. Hence the boundary trace of the harmonic layer of $u$ is $m!\sum_iC_{v_i,K}\mathsf f(i)Y_{v_i}$, minus $T_K$ if $m=0$. By $\sum_iC_{v_i,K}Y_{v_i}=T_K$ and Lemma~\ref{r7:lem-lap}, $\sum_iiC_{v_i,K}Y_{v_i}=\beta^{-1}\big(-\frac{5K}4U_{K-2}+\sum_ii^2C_{v_i,K}Y_{v_i}\big)$, and the trace equals
\[
m!\Big(\mathsf f(0)T_K-\frac{5K\mathsf f'(0)}{4\beta}U_{K-2}+\sum_iC_{v_i,K}\Big(\frac{\mathsf f'(0)i^2}\beta+\varrho_2(i)\Big)Y_{v_i}\Big),
\]
with $\mathsf f(0)-1$ in place of $\mathsf f(0)$ if $m=0$. We apply $A$ to the three parts: $\nrm{A\mathcal T_K}\leq\alpha^H(K)R^K\Upsilon(K)$; the Chebyshev coefficients of $U_{K-2}$ are $2$ (and $1$ for $T_0$), which gives $\mathsf u(K)$, since $K<\beta$; and $Y_{v}$ is the trace of $\mathcal H_v$. For the bounds on $\mathsf f$ write $\mathsf f=g/P$ with $g(x)=\beta-2x+1$ and $P(x)=\prod_k(\beta-x+e+k)$, and let $S_1=\sum_k(\beta-x+e+k)^{-1}$, $S_2=\sum_k(\beta-x+e+k)^{-2}$. Then $\mathsf f'=\mathsf f\,(S_1-2/g)$ and $\mathsf f''=P^{-1}\big(g(S_1^2+S_2)-4S_1\big)$. At $x=0$, $0<S_1\leq(m+1)/(\beta+1)$ and $2/g=2/(\beta+1)$, which gives the bound for $\mathsf f'(0)$. For $0\leq x\leq i\leq K/2$ the two terms of $\mathsf f''$ have opposite signs, so $|\mathsf f''(x)|\leq P(x)^{-1}\max\{g(S_1^2+S_2),4S_1\}$; with $P(x)\geq P(i)$, $g\leq\beta+1$, $S_1\leq(m+1)/(\beta-i+1)$, $S_2\leq(m+1)/(\beta-i+1)^2$ and $(\beta-2i+1)^2\leq(\beta+1)(\beta-i+1)$ this is at most $(\beta+1)\max\{(m+1)(m+2),4(m+1)\}/((\beta-2i+1)^2P(i))$, and $\max\{(m+1)(m+2),4(m+1)\}\leq c_m^2+m+5$. Taylor's formula gives the bound for $\varrho_2$.
\end{proof}

Let $\Phi_m(K)$ ($m\geq1$) and $\Phi_P(K,e)$ ($m=0$) be the right-hand side of Lemma~\ref{r7:lem-raw}, evaluated as follows. For $m\geq1$ the bounds are made independent of $e$ by putting $e=0$ in $\mathsf f(0)$ and in the bound for $\varrho_2$. For $m=0$ use $|\mathsf f'(0)|\leq2/(\beta+e+1)$, $\mathsf f_{\max}=1$ and the second-derivative constant $9$. The sums over $i$ are evaluated exactly for $i\leq I_1$; for  $I_1<i\leq\min(I,K/2)$ the factors $\mathsf g_i$ and $|\varrho_2(i)|/i^2$, which increase with $i$, are replaced by their values at the largest index; and the terms with $i>I$ are bounded by means of $G_0,G_1,G_2$, using $\mathsf g_i\leq1$, $|\mathsf f(i)|\leq\mathsf f_{\max}$ and $|\varrho_2(i)|\leq2\mathsf f_{\max}+|\mathsf f'(0)|i$, where $\mathsf f_{\max}=(\beta+1)/(K/2+\frac72)^{m+1}$ for $m\geq1$ and $\mathsf f_{\max}=1$ for $m=0$. The same formulas define placeholder values at odd indices for suffix maximisation. For odd $K\leq M$, which enters only through suffix maxima, $\alpha^H(K)$ is the maximum of $A_t$ and of the two even neighbours, and the sum in $\mathsf u(K)$ is empty. Let $\varphi_m(K)=\sup_{K'\geq K}\Phi_m(K')$, and
\[
\bar\Gamma(K)=\sup_{K'\geq K}\sum_{i\leq I}\mathsf w_i\Gamma_*\big(\max(K'-2i,0)\big)+\mathsf w_{I_+}S_0(I_+)\,\Gamma_{\max} .
\]
For $f$ one of the polynomials $H^\circ$, $W^\circ_1$, $Z^\circ$, $W^\circ$, write $f=\sum_aH_af_a(t)$, $f_a=\sum_sf_{a,s}J^{\beta_a}_s=\sum_mf^\natural_{a,m}(1-t)^m$ and $\nrm{f_a}_J=\sum_s|f_{a,s}|$. By Proposition~\ref{r7:prop-C1}, every $f_a$ has a double zero at $t=1$ for $f=W^\circ$ and a zero for $f=W^\circ_1,Z^\circ$, so the corresponding $f^\natural_{a,m}$ vanish exactly. For $K'\geq0$ let
\[
\mathrm{HH}_a(K')=\sup_{K''\geq K'}\sum_{i\leq\min(100,K''/2)}\mathsf w_i\Gamma_*(K''-2i)\sum_{m>3}|f^\natural_{a,m}|\,\mathsf F_m(\beta_{K''-2i}),
\]
let $\mathrm{JT}(K')$ be the supremum over $K''\geq K'$ of $\sum_{100<i\leq\min(I,K''/2)}\mathsf w_i\Gamma_*(K''-2i)$ plus $G_0(K'')$ if $K''>2I$, and put, for $K\geq a$,
\begin{multline*}
\mathrm{Mix}_a(K)=|f^\natural_{a,0}|\,\bar\Gamma(K-a)+\sum_{m=1}^3|f^\natural_{a,m}|\,\varphi_m(K-a)+\mathrm{HH}_a(K-a)\\
+\Big(\nrm{f_a}_J+\sum_{m\leq3}|f^\natural_{a,m}|\,\mathsf F_m(\beta_a)\Big)\mathrm{JT}(K-a),
\end{multline*}
$\mathrm{Jac}_a(K)=\nrm{f_a}_J\,\bar\Gamma(\max(K-a,0))$, and finally
\[
V_f(K)=\sup_{K'\geq K}\sum_ah_a\min\{\mathrm{Jac}_a(K'),\mathrm{Mix}_a(K')\},
\]
where $\mathrm{Mix}_a(K')$ is omitted for $K'<a$. For $\mathrm{HH}_a$, replacing $\mathsf F_m(\beta_v)$ by $\mathsf F_m(\beta_{8\lfloor v/8\rfloor})$ is also permitted: this is a nonincreasing upper step bound and is the bound evaluated in Section~\ref{r7:sec-cap}.

\begin{lemma}\label{r7:lem-centre}
For $f\in\{H^\circ,W^\circ_1,Z^\circ,W^\circ\}$ and integers $K,e\geq0$, with $K$ odd for $W_1^\circ$ and even otherwise, $\nrm{A(t^er^KT_K(\xi)f)}_{\mathcal X}\leq R^K\,\Upsilon(K+2e+\deg f)\,V_f(K)$.
\end{lemma}

\begin{proof}
It suffices to treat one angular component $H_af_a$, with $\Upsilon\equiv1$, and to show that its contribution is at most $R^Kh_a\mathrm{Jac}_a(K)$ and, for $K\geq a$, at most $R^Kh_a\mathrm{Mix}_a(K)$. \emph{Jacobi route.} Expand $r^KT_K=\sum_iC_{K-2i,K}H_{K-2i}t^i$ and multiply by $H_af_a$ with Lemma~\ref{r7:lem-mult}: the outputs have angular index at least $K-2i-a$, the lowering steps of the Jacobi parameter that are needed come with enough factors $t$ (as in the proof of Lemma~\ref{r3:lem-mult}(a)), all steps are stochastic, and $\sum_vG_{v'av}h_v\leq h_{v'}h_a$; with Lemma~\ref{r7:lem-conn} this gives $\mathrm{Jac}_a$. \emph{Mixed route}, for $K\geq a$. By $H_a=\sum_nD_{na}t^{(a-n)/2}r^nT_n$ and $T_KT_n=\frac12(T_{K+n}+T_{K-n})$, $t^er^KT_KH_a$ is a nonnegative combination of raw monomials $t^{e'}r^{K'}T_{K'}$ with $K'\geq K-a$ and $R$-weighted mass at most $h_aR^K$. Split $f_a=P_a+(f_a-P_a)$ with $P_a=\sum_{m\leq3}f^\natural_{a,m}(1-t)^m$. The part $P_a$ is bounded by $\bar\Gamma$ for $m=0$ and by Lemma~\ref{r7:lem-raw} for $1\leq m\leq3$. For $f_a-P_a$, in the connection expansion of $t^{e'}r^{K'}T_{K'}$ the components with index $i\leq100$ are bounded with Lemma~\ref{r3:lem-jacobi}(c),(d) term by term in $m>3$, which gives $\mathrm{HH}_a$, and those with $i>100$ by the Jacobi norm, which gives $\mathrm{JT}$: here
\[
\nrm{t^\mu(f_a-P_a)}_{J^{\beta_v},1}\leq\nrm{f_a-P_a}_{J^{\beta_a},1}\leq\nrm{f_a}_J+\sum_{m\leq3}|f^\natural_{a,m}|\,\mathsf F_m(\beta_a),
\]
because every lowering of the Jacobi parameter from $\beta_a$ to $\beta_v<\beta_a$ comes with a factor $t$ (for the monomials $t^{e'}r^{K'}T_{K'}$ above one has $e'+K'/2=e+(a+K)/2\geq a$ when $K\geq a$, so that $\mu=e'+i\geq a-(K'-2i)=a-v$ for every $i\leq K'/2$), and the raising and lowering stencils of Lemma~\ref{r3:lem-jacobi}(b) are nonnegative with column sums one. Finally, $\bar\Gamma$, $\varphi_m$, $\mathrm{HH}_a$ and $\mathrm{JT}$ are nonincreasing and $K'\geq K-a$.
\end{proof}

\begin{lemma}[Principal difference and infinite suffix tables]\label{r7:lem-extension}
Assume the three inequalities of Lemma~\ref{r7:lem-switch}. Then $\Phi_P(K,e)$ is nonincreasing on the even integers $K\geq\max(M+2,2I+2)$. Let $K_{\max}$ be an even integer with
\[
K_{\max}\geq\max(2I+2,M+2,2i_c).
\]
Then the suffix maxima of the raw tables defining $\bar\Gamma$, $\varphi_m$, $\mathrm{HH}_a$ and $\mathrm{JT}$, computed through $K_{\max}$ and extended constantly afterwards, bound the suprema of these tables over all integers $K'\geq K$. The same construction gives nonincreasing upper bounds $V_f$ at all required indices.
\end{lemma}
\begin{proof}
For $K\geq2I+2$ all $i$-ranges in $\Phi_P$ are fixed, and the geometric tail is present. The weights $\mathsf w_i$ are independent of $K$, the profiles decrease, and $\alpha^H(K)=A_t$. The factors $(e+i)/(\beta-i+e+1)$, $e/(\beta+e+1)$, $2/(\beta+e+1)$ and $1/\beta$ decrease with $\beta=K+5/2$. So does the second-derivative majorant $9(\beta+1)/((\beta-2i+1)^2(\beta-i+e+1))$, since $(\beta+1)/(\beta-2i+1)^2$ has negative logarithmic derivative. On even indices the finite part of $\mathsf u(K)$ is multiplied by $\zeta$ at each step once $K>M$; on odd indices it is absent. Lemma~\ref{r7:lem-switch} controls every branch change of $G_p$.

For the other raw tables the terms of $\bar\Gamma$ decrease without a threshold. In $\Phi_m$, once $K\geq\max(2I+2,M+2)$, $\mathsf F_m(\beta-2i)$, $\mathsf f(0)$, its derivative and second-derivative bounds, and $\mathsf f_{\max}$ all decrease. For the last factor the logarithmic derivative is $1/(K+7/2)-(m+1)/(K+7)<0$ for $m\geq1$, $K>0$. In $\mathrm{JT}$ the range $(i_c,I]$ is fixed for $K\geq2I+2$; in $\mathrm{HH}_a$ the range $[0,i_c]$ is fixed for $K\geq2i_c$. The remaining factors are nonincreasing suffix profiles. Odd indices of $\mathsf u$ contribute only $2A_t\zeta/(1-\zeta)$ and are bounded by the preceding even value. Thus every raw value beyond the even $K_{\max}$ is at most its value there, for either parity.

Suffix maximisation now bounds the whole infinite sequence. For $V_f$ both routes are nonnegative combinations of these tables at $\max(K-a,0)$ or $K-a$; when $K=a$ the additional mixed route enters a minimum and can only lower it. Hence their summed minimum is nonincreasing, and its constant extension after $K_{\max}$ is an upper bound. Tables are formed for both parities. Only odd $K$ is used for $f=W_1^\circ$, and only even $K$ for $H^\circ,Z^\circ,W^\circ$; entries of the other parity enter solely through suffix maxima and can only enlarge the bounds.

For the frozen parameters $K_{\max}=N_2+2=1202$ is even and $\max(2I+2,M+2,2i_c)=282$. The extension condition involves no Taylor angular degrees.
\end{proof}

\begin{proposition}[All far shape columns]\label{r7:prop-far}
Let $N_*=702$, $N_2=1200$, $k_0=N_2/2$, $\mathsf p=\Upsilon(1000)/(b\sigma R)$, and for even $N\geq N_*$ put
\begin{align*}
\mathfrak F(N)=\mathsf p\Big[&\frac1{N+1}\sum_{k=0}^{N/2}(2-\delta_{2k,N})\zeta^kV_{H^\circ}(N-2k)
+\frac{10}R\sum_{k=0}^{N/2-1}\zeta^kV_{W_1^\circ}(N-1-2k)\\
&+5V_{Z^\circ}(N)+\frac5{2(N+1)}\Big(\sum_{k=0}^{N/2}(2-\delta_{2k,N})\zeta^kV_{Z^\circ}(N-2k)\\
&\hphantom{{}+5V_{Z^\circ}(N)+\frac5{2(N+1)}\Big(}
+\sum_{k=0}^{N/2-1}(2-\delta_{2k,N-2})\zeta^{k+1}V_{Z^\circ}(N-2-2k)\Big)\\
&+2\rho^2\sum_{(\eta,e)}\gamma_{\eta,e}R^\eta\big(\Phi_P(N_*+\eta,e)+V_{W^\circ}(N_*+\eta)\big)\Big],\\
\mathfrak F_\infty=\mathfrak F(N_2)&+\frac{\mathsf p\zeta^{k_0}}{1-\zeta}
\Big(\frac{2V_{H^\circ}(0)}{N_2+1}+\frac{10V_{W_1^\circ}(0)}R+\frac{10V_{Z^\circ}(0)}{N_2+1}\Big).
\end{align*}
Every odd $j\geq703$ has normalised defect at most
$Z_{\mathrm{far}}=\max(\{\mathfrak F(N):702\leq N\leq1200\text{ even}\}\cup\{\mathfrak F_\infty\})$.
\end{proposition}
\begin{proof}
Divide \eqref{r7:eq-farcol}, with $\mathsf Me_j$ subtracted, by $b\sigma(N+1)R^{N+1}\Upsilon(N)$. All output degrees are at most $N+800$, so $\Upsilon(1000)$ covers the grading. The expansions of $U_N,U_{N-1},U_{N-2}$ and Lemma~\ref{r7:lem-centre} give the first four terms. Lemma~\ref{r7:lem-raw} bounds the principal difference and Lemma~\ref{r7:lem-centre} the $W^\circ$ term. The smallest principal frequency is $N_*-420=282=2I+2>M$, so Lemma~\ref{r7:lem-extension} bounds these terms at $N_*$ for every later $N$.

For $N\leq N_2$ the finite $k$-sums are evaluated as written. For $N>N_2$, every old term decreases because $V_f$ does. In the $H^\circ$ sum and the first $Z^\circ$ sum, the endpoint $k=k_0$ has weight two instead of one. The $W_1^\circ$ sum has new terms from $k=k_0$. In the second $Z^\circ$ sum the old endpoint $k=k_0-1$ changes from weight one to two, with power $\zeta^{k_0}$, and new terms follow. Write the $W_1^\circ$ sum as $\frac5R\sum_k2\zeta^kV_{W_1^\circ}(N-1-2k)$, all Chebyshev coefficients of $U_{N-1}$ being two. Then each sum, with its weights, exceeds its value at $N_2$ by at most $2\zeta^{k_0}V_f(0)/(1-\zeta)$, using $1/(N+1)\leq1/(N_2+1)$. The prefactors are $1/(N_2+1)$, $5/R$ and $5/(2(N_2+1))$, and the two $Z^\circ$ sums together carry the weight four; this gives exactly $\mathfrak F_\infty$. Thus all infinitely many far columns are covered.
\end{proof}

\subsection{Certified zero and geometry}\label{r7:sec-zero}
Table~\ref{r7:tab-bounds} bounds all input columns. The finite inverse defect is below $5.94\cdot10^{-7}$, $\eps_\Upsilon<0.010791473$, $\rho_{\mathcal B}<846.350140$ and $\tau_T<8.758304\cdot10^{-5}$.
The value $Z_0\approx0.899$ is the bound for the uniform angular class $\ell\geq388$, $s\geq1$, given by the envelope $\mathfrak z_{\mathrm u}(388,1)$; every explicitly computed column has certified defect below $0.68$.

\begin{table}[ht]
\centering\footnotesize
\begin{tabular}{@{}ll@{}}
\toprule
Quantity & Certified upper bound\\
\midrule
$\nrm A$, $A_t$ & $342234.383037$, $15.223634$\\
$Y_0$, $Z_0$ & $5.950563\cdot10^{-11}$, $0.898934481$\\
finite columns ($9010$) & $0.058622927$\\
exceptional field columns ($133$) & $0.676824352$\\
field envelopes ($106$ radial, $88$ gap, angular) & $0.898934481$\\
explicit shape tail ($245$ columns), far shape tail & $0.651602561$, $0.305605856$\\
$G_0$, $C_0$ & $366106.776644$, $1.243142911$\\
$c_2$, $c_3$, $c_4$ & $0.745260675$, $1.946871609\cdot10^{-6}$, $1.925257056\cdot10^{-12}$\\
$p(10^{-9})$, contraction bound & $-4.13048\cdot10^{-11}$, $0.899444589$\\
\bottomrule
\end{tabular}
\caption{Outward roundings of the seven-dimensional certificate: $(M,S)=(210,84)$, $R=6/5$, $\sigma=25\rho^2$.}\label{r7:tab-bounds}
\end{table}

\begin{theorem}[Seven-dimensional certificate]\label{r7:thm-zero}
For $(M,S)=(210,84)$, $R=6/5$ and $\sigma=25\rho^2$, equation \eqref{r7:eq-F} has exactly one zero $x^*=(g^*,c^*)$ in $\overline B_{\mathcal X}(x^\circ,10^{-9})$. Its residual bound is $Y_0<5.950563\cdot10^{-11}$, its linear defect bound is $Z_0<0.898934481$, and the Newton map has contraction bound below $0.899444589$. For every shape in this existence ball,
\begin{gather*}
\re\psi'>0.993557571646\quad(|w|\leq101/100),\\
\re(1+w\psi''/\psi')>0.951916263887\quad(|w|\leq1),\qquad
\re(w\psi'/\psi)>0.994796343179\quad(|w|=1),\\
|c_3|>0.000189913656777,\qquad
\kappa_\perp>0.934750157145\quad\text{on the unscaled revolved boundary}.
\end{gather*}
In particular, the domain reconstructed from $x^*$ is a strictly convex non-ball domain in $\R^7$ satisfying Theorem~\ref{thm:main}.
\end{theorem}
\begin{proof}
The block inverse is injective. The finite, exceptional and envelope field classes are disjoint and exhaustive, as are the finite shape inputs $j\leq211$, explicit tail $213\leq j\leq701$ and far tail $j\geq703$. The preceding bounds therefore give $Y_0$ and $Z_0$ for the full operator by Lemma~\ref{lem:columns}. At $r=10^{-9}$,
\[
p(r)=Y_0+(Z_0-1)r+\nrm A\sum_{k=2}^4c_kr^k<-4.13048\cdot10^{-11},\qquad
Z_0+\nrm A\sum_{k=2}^4kc_kr^{k-1}<0.899444589.
\]
Proposition~\ref{r3:prop-NK} gives existence and uniqueness. The shape projection has radius at most
$\eps_0=2890623035776027/(5\cdot10^{29})\geq r/\sigma$.

Put $S_1=\sum_{j\geq3}j|c^\circ_j|$, $S_2=\sum_{j\geq3}j(j-1)|c^\circ_j|$, $S_{\mathrm{rad}}=\sum_{j\geq3}j^2|c^\circ_j|$ and
$H_{\mathcal E}=\max_{j\geq3\text{ odd}}(j-1)/(R^j\Upsilon(j-1))$. Its consecutive odd ratio is at most $(j+1)/((j-1)R^2)\leq1$ for $j\geq7$; thus the maximum uses only $j\leq7$. The triangle estimates of Proposition~\ref{r5:prop-geometry}, with these weights and $\eps_0$, prove the three real-part bounds and $|c_3|\geq|c^\circ_3|-\eps_0/(3R^3\Upsilon(2))$.

For $w=e^{i\theta}$ on the upper semicircle the outward meridian normal is $w\psi'/|\psi'|$. Each of the five transverse principal curvatures equals
\[
\kappa_\perp=\frac{\im(w\psi')}{|\psi'|\im\psi}\geq
\frac{c^\circ_1-S_{\mathrm{rad}}-\eps_0(H_{\mathcal E}+1/R)}{(c^\circ_1+S_1+\eps_0/R)^2}>0.934750157145.
\]
Indeed $|\sin(j\theta)/\sin\theta|\leq j$, and $\sum_jj^2|\delta c_j|\leq\eps_0(H_{\mathcal E}+1/R)$. Both factors of the denominator, after division of $\im\psi$ by $\sin\theta$, are positive and at most $c^\circ_1+S_1+\eps_0/R$. The formula extends continuously to the poles by real analyticity. The meridian curvature is $\re(1+w\psi''/\psi')/|\psi'|>0$; all six principal curvatures are positive throughout the shape ball. These are triangle-inequality lower bounds uniform over the shape ball; at the centre itself one finds numerically $\min_{|w|=1}\re(1+w\psi''/\psi')\approx0.9712$ and $\min\kappa_\perp\approx0.9988$.

Let $\Theta(x_1,x')=(\re\psi(x_1+i|x'|),q_cx')$. The proof of Lemma~\ref{r5:lem-lift} applies in $\R^7$: the meridian differential is $|p|$ times a rotation and the five transverse differentials equal $q_c$, so $\det D\Theta=q_c^5|p|^2>0$. Univalence and $q_c>0$ make $\Theta$ injective near $\overline B$; its meridian series give analyticity across the axis. Lemma~\ref{r7:lem-F} and change of variables give $u_0=(1+K_7g^*)\circ\Theta^{-1}\in C^1(\overline{\Theta(B)})$, solving $\Delta u_0+\rho^2u_0=0$ weakly with the prescribed traces. Set $\Omega=\rho\Theta(B)$ and $u(X)=u_0(X/\rho)$. Lemma~\ref{r3:lem-regularity} gives analyticity near the closed domain. Its boundary is diffeomorphic to $S^6$, with $O(6)\times\Z_2$ symmetry. The symmetric strictly convex meridian section gives a convex solid of revolution, and the curvature bounds give strict convexity (physical curvatures are divided by $\rho$). The lower bound for $\re(w\psi'/\psi)$ on $|w|=1$ makes $\psi(\D)$, and hence the solid of revolution $\Omega$, strictly star-shaped with respect to the origin. A ball would be centred at zero by symmetry and would force $\psi$ linear by Schwarz's lemma, contradicting $c_3\ne0$. A constant Helmholtz solution is zero, whereas $u=1$ on the boundary. This proves Theorem~\ref{thm:main} for $n=7$.
\end{proof}

\part{The exact inverse in general dimension}\label{gd:part}

\section{The general-dimensional certificate}\label{gd:sec:theorem}

This part proves the cases $n\in\gdDimensions\setminus\{5\}$ of
Theorem~\ref{thm:main} and gives a second proof of the case $n=5$.

\begin{theorem}\label{gd:thm:main}
For every $n\in\gdDimensions$ there is a bounded strictly convex domain
$\Omega\subset\R^n$, different from a ball and strictly star-shaped
with respect to the origin, with real-analytic boundary
diffeomorphic to $S^{n-1}$, invariant under $O(n-1)$ on the last $n-1$
coordinates and under reflection of the first coordinate. It carries a
nonconstant real-valued solution, real analytic on a neighbourhood of
$\overline\Omega$, of
\begin{equation}\label{gd:eq:schiffer}
 (\Delta+1)u=0\quad\hbox{in }\Omega,\qquad
 u=1,\quad \partial_{\mathbf n} u=0\quad\hbox{on }\partial\Omega.
\end{equation}
Here $\mathbf n$ is the outward unit normal.
In particular, $\Omega$ fails the Pompeiu property.
\end{theorem}

For $n=5$ this is a second proof of the five-dimensional case of
Theorem~\ref{thm:main}; it uses none of the bounds of Section~\ref{r5:sec}.
The argument below is parametrised by an integer $n\ge5$.
The angular conversion uses $\lambda\ge1$, and root coverage uses
$\nu=l+\lambda\ge3/2$.
Its exact parameters and constants are in
Tables~\ref{gd:tab:inputs}--\ref{gd:tab:geometry}; the finite centres are
the compressed files \path{anc/gd/frozen_centers/nN.json.gz} described in
Subsection~\ref{gd:sec:verification}.
The scale $\rho$, analytic weight $R>1$, shape weight $\sigma>0$ and
existence radius $r_*>0$ are exact, as are $\lambda$, $L_0$ and the mode
and state counts. Entries carrying an inequality sign are strict
one-sided bounds.

\begin{table}[htbp]
\centering\small
\def\gdRow#1#2#3#4#5#6#7#8{$#1$ & #2\\}
\begin{tabular}{@{}crrrrr@{}}
\toprule
$n$ & stored modes & $\rho$ & $R$ & $\sigma$ & $r_*$\\
\midrule
\gdData
\bottomrule
\end{tabular}
\caption{Exact inputs. The stored-mode count is the length of each finite
shape and field array, not a count of converged coefficients; the clamping
lift and the solution have infinite expansions.}\label{gd:tab:inputs}
\end{table}

\subsection{Meridian coordinates and the equation}\label{gd:sec:formulation}
Write reference points as $x=(x_1,x')\in\R\times\R^{n-1}$ and set
\[
 y=|x'|,\quad w=x_1+iy,\quad r=|x|,\quad \xi=x_1/r,\quad
 a=n-2,\quad\lambda=a/2,\quad\mu=\rho^2.
\]
As in Sections~\ref{r5:sec} and~\ref{r7:sec}, $\xi$ is the angular
variable. The radial variable here is $r$; below $t$ denotes a time or
a parameter.
All scalar unknowns are $O(n-1)$-invariant and even in $x_1$.
The regular lift to the full ball interprets the meridian operator
\begin{equation}\label{gd:eq:laplacian}
 \Delta_n=\partial_{x_1}^2+\partial_y^2+\frac a y\partial_y .
\end{equation}
There is no boundary condition on the axis.

Let $\psi$ be odd holomorphic with real coefficients, $p=\psi'$ and
$q=\im\psi/y$. As in Lemma~\ref{r5:lem-lift}, put
\begin{equation}\label{gd:eq:lift-map}
 \Theta_\psi(x_1,x')=
 \bigl(\re\psi(x_1+i|x'|),\,q(x_1,|x'|)x'\bigr).
\end{equation}
Hereafter $b$ is the positive coefficient of $w$ in the fixed centre
$\psi_0$; it is held constant throughout the equation.
The unscaled domain is $\Theta_\psi(B^n)$, with $u_0$ of frequency
$\rho$; the domain of the theorem is $\Omega=\rho\Theta_\psi(B^n)$.
On the meridian plane the differential is $|p|$ times a rotation;
on each of the $n-2$ transverse directions it is multiplication by $q$.
Consequently
\[
 \det D\Theta_\psi=|p|^2q^{n-2}.
\]
The proof of Lemma~\ref{r5:lem-lift}, with $n-2$ transverse vectors,
gives the change of variables in every dimension.

For $U=u_0\circ\Theta_\psi$, with frequency $\rho$, the
pull-back equation is
\[
 \operatorname{div}(q^a\nabla U)+\mu q^a|p|^2U=0.
\]
Divide by the positive factor $bq^{a-1}$ and define
\begin{equation}\label{gd:eq:B}
 \begin{split}
 BU&=\frac1b\left\{q\Delta_n U+
 a\left(q_{x_1}U_{x_1}+\frac{q_y}{y}\,yU_y\right)
 +\mu q|p|^2U\right\},\\
 \mathcal L U&=\Delta_n U+\frac a q\nabla q\cdot\nabla U
                  +\mu|p|^2U,\qquad B=(q/b)\mathcal L .
 \end{split}
\end{equation}
In particular the multiplier relating $B$ to the operator $\Delta+\mu$ is
\begin{equation}\label{gd:eq:physical-multiplier}
 B(u_0\circ\Theta_\psi)=
 \frac{q|p|^2}{b}\bigl((\Delta+\mu)u_0\bigr)\circ\Theta_\psi .
\end{equation}
The quotients in~\eqref{gd:eq:B} are regular analytic scalars:
$q=\int_0^1\re p(x_1+isy)\,ds$ is even in $y$, and so is $q_y/y$.
This is the quartic normalisation of~\eqref{r5:eq-div}, with the
fixed scalar factor $1/b$.

\section{Angular coefficient spaces and the compatible Poisson inverse}
\label{gd:sec:spaces}

\subsection{The angular algebra}\label{gd:sec:algebra}
Set
\[
 G_l(\xi)=\frac{C_l^\lambda(\xi)}{C_l^\lambda(1)},\qquad
 H_l(x)=r^lG_l(\xi).
\]
For $n=5,7$ these are the polynomials denoted $Y_\ell,H_\ell$ in
Sections~\ref{r5:sec-spaces} and~\ref{r7:sec-spaces}.
The generating function of the Gegenbauer polynomials gives
\begin{equation}\label{gd:eq:cosine}
 \begin{split}
 G_l(\cos\theta)&=\sum_mD_{ml}\cos(m\theta),\\
 D_{l-2k,l}&=
 \frac{(2-\delta_{l,2k})(\lambda)_k(\lambda)_{l-k}}
 {k!(l-k)!C_l^\lambda(1)}
 \quad(0\le k\le l/2),\qquad
 h_l=\sum_m D_{ml}R^m .
 \end{split}
\end{equation}
In particular $D_{ml}\ge0$, $\sum_mD_{ml}=1$ and
$h_l\ge\sum_mD_{ml}=1$.

Let $\mathcal Y_n$ consist of expansions
$f=\sum_{l\ge0}f_l(r)G_l(\xi)$, with $f_l\in C([0,1])$,
$f_l(0)=0$ for $l>0$, and norm
\begin{equation}\label{gd:eq:Y}
 \nrm f_{\mathcal Y_n}=\sum_{l\ge0}h_l\nrm{f_l}_\infty.
\end{equation}
Let $Y$ be its real even sector, so only even $l$ occur.
The radial norm is a supremum over all radii. It replaces the sum of
Jacobi radial coefficients, graded by polynomial degree, in
Sections~\ref{r5:sec-spaces} and~\ref{r7:sec-spaces}.
In particular a continuous radial profile is admitted without a Jacobi
expansion. The Poisson estimates below provide its $C^1$ reconstruction.

The raw cosine norm is
$\nrm f_{\rm raw}=\sum_mR^m\nrm{\tilde f_m}_\infty$ for the corresponding
cosine expansion. The boundary norms $W_0,W_1$ have exactly the Fourier
convention of~\eqref{r2:eq:Wone}, now with weight $R^{|m|}$.
For a real cosine series they are
\[
 \nrm h_{W_0}=\sum_{m\ge0}R^m|h_m|,\qquad
 \nrm h_{W_1}=\sum_{m\ge0}(1+m)R^m|h_m|.
\]
Thus $\nrm f_{\rm raw}\le\nrm f_{\mathcal Y_n}$ and boundary evaluation
has norm at most one into $W_0$.

\begin{lemma}\label{gd:lem:algebra}
The spaces $\mathcal Y_n$ and $Y$ are Banach algebras with
$\nrm{fg}_{\mathcal Y_n}\le\nrm f_{\mathcal Y_n}\nrm g_{\mathcal Y_n}$.
\end{lemma}
\begin{proof}
Completeness is the weighted $C([0,1])$ argument of
Lemma~\ref{r2:lem:algebra}; the origin conditions define a closed subspace.
Since $|G_l|\le1$, convergence in the norm gives uniform convergence on
the full closed ball. For the product, write
$G_lG_k=\sum_dL_{lk}^dG_d$. The addition theorem represents $G_l$
as a positive semidefinite zonal kernel on $S^{n-1}$.
The product kernel is positive semidefinite by tensoring its feature
vectors. Its scalar eigenvalues on spherical harmonics are nonnegative,
and orthogonality identifies them, up to positive normalisations, with
$L_{lk}^d$. Hence $L_{lk}^d\ge0$; evaluation at $\xi=1$ gives
$\sum_dL_{lk}^d=1$. This is the positivity used in
Theorem~\ref{r5:thm-mult}(a), also given by the Gegenbauer linearisation
formula~\cite[18.18.22]{DLMF}.
The nonnegative cosine expansions and
$T_mT_j=(T_{m+j}+T_{|m-j|})/2$ give
\[
 \sum_dL_{lk}^dh_d\le h_lh_k .
\]
Taking radial suprema proves the product estimate for finite sums.
Absolute convergence proves it for infinite sums. At the origin only
constant angular terms remain, so products preserve origin compatibility.
\end{proof}

\begin{lemma}[Conversion of the full cosine expansion]\label{gd:lem:connection}
Put $z=R^{-2}$ and $e_i=|(-\lambda)_i|/i!$. Then
\begin{equation}\label{gd:eq:kappa}
 \nrm f_{\mathcal Y_n}\le\kappa\nrm f_{\rm raw},\qquad
 \kappa=(1-z)^{-\lambda}\sum_{i\ge0}e_i z^i .
\end{equation}
This bound applies to arbitrary continuous radial profiles and includes
all angular degrees.
\end{lemma}
\begin{proof}
The connection formula of Lemma~\ref{r7:lem-conn}, whose derivation of
the coefficient formula holds for every $\lambda>0$, is
\begin{equation}\label{gd:eq:connection}
 T_N=\sum_{i=0}^{\lfloor N/2\rfloor} C_{N-2i,N}G_{N-2i},\qquad
 C_{l,N}=
 \frac{N(-\lambda)_i(N-i-1)!(l+\lambda)C_l^\lambda(1)}
 {2i!(\lambda)_{N-i+1}},
 \quad l=N-2i,
\end{equation}
for $N>0$; $T_0=G_0=1$. It follows by the same limit of the
Gegenbauer connection formula~\cite[18.18.16,18.7.25]{DLMF}.
For $l>0$ the top cosine coefficient is
$2(\lambda)_l/(l!C_l^\lambda(1))$.
Since $\lambda\ge1$, the remaining positive coefficients in
\eqref{gd:eq:cosine} give
\[
 h_l\le
 \frac{2(\lambda)_lR^l}{l!C_l^\lambda(1)}(1-z)^{-\lambda}.
\]
For $l=0$ use the top cosine coefficient one and
$h_0\le(1-z)^{-\lambda}$.
For $i\ge1$, multiplication by $|C_{l,N}|$ leaves
$e_iR^Nz^i(1-z)^{-\lambda}$ times
\[
 \frac{N}{N-i+\lambda}
 \prod_{s=1}^{i-1}\frac{l+s}{l+s+\lambda}\le1.
\]
For $i=0$, the product of $C_{N,N}$ and the top cosine
coefficient is exactly one. For $i\ge1$ the displayed product is at most
$(l+1)/(l+i)$, and the desired inequality reduces to
$i^2-\lambda i+N(\lambda-1)\ge i(i+\lambda-2)\ge0$,
using $N\ge2i$. Summation proves~\eqref{gd:eq:kappa}.
Conversion lowers degree and preserves the origin condition.

For integer $\lambda$, the infinite series is $(1+z)^\lambda$.
For $\lambda=m+1/2$, the Taylor coefficients $(-\lambda)_i/i!$
have sign $(-1)^i$ for $i\le m+1$ and fixed sign
$\epsilon_\lambda=(-1)^{m+1}$ for $i\ge m+1$. Hence
\[
 \sum_{i\ge0}e_i z^i=\epsilon_\lambda(1-z)^\lambda
 +2\sum_{\substack{0\le i\le m\\i\equiv m\ (\mathrm{mod}\ 2)}}e_i z^i.
\]
For $\lambda=3/2,5/2$ this is the expression in
Lemmas~\ref{r5:lem-conn}(b) and~\ref{r7:lem-conn}.
Thus~\eqref{gd:eq:kappa} has a finite closed expression in both cases.
\end{proof}

\begin{lemma}[Weights and polynomial moments]\label{gd:lem:moments}
One has $1=h_0\le h_1\le h_2\le\cdots$ and $h_l\le R^l$.
For a representation $f=\sum_{l,s}c_{ls}r^{2s}H_l$ define
$\mathfrak m_k(f)=\sum_{l,s}|c_{ls}|h_l(l+2s)^k$.
Then
\begin{equation}\label{gd:eq:poly-derivatives}
 \nrm{\partial_{x_1}f}_{\mathcal Y_n}\le C_x^{\rm pol}\mathfrak m_1(f),
 \qquad
 \nrm{y\partial_yf}_{\mathcal Y_n}\le C_y^{\rm pol}\mathfrak m_1(f),
 \quad C_x^{\rm pol}=1+R,\quad C_y^{\rm pol}=1+RC_x^{\rm pol}.
\end{equation}
\end{lemma}
\begin{proof}
Choose a random $\zeta$ with distribution $\operatorname{Beta}(\lambda,\lambda)$,
then independent signs of mean $2\zeta-1$, and let $S_l$ be their sum.
Integration of the binomial law gives
$\binom lk(\lambda)_k(\lambda)_{l-k}/(2\lambda)_l$.
Pairing $k$ and $l-k$ identifies~\eqref{gd:eq:cosine} with
$h_l=\mathbb E R^{|S_l|}$.
The conditional expectation is
$\mathbb E(S_{l+1}\mid S_l)=S_l(1+(2\lambda+l)^{-1})$.
Jensen's inequality for the even convex function $s\mapsto R^{|s|}$
proves monotonicity; $|S_l|\le l$ proves the upper bound.
The harmonic identity $\partial_{x_1}H_l=lH_{l-1}$ is proved in
Section~\ref{r7:sec-spaces} by an argument valid in every dimension.
Differentiating $r^{2s}H_l$ and using
$\nrm{x_1}_{\mathcal Y_n}=R$ gives the first estimate.
The Euler derivative multiplies that monomial by $l+2s$;
$y\partial_y=r\partial_r-x_1\partial_{x_1}$ gives the second.
Absolute convergence with the indicated moments justifies all derivatives.
\end{proof}

\subsection{The zero-Dirichlet inverse and its trace}\label{gd:sec:Poisson}
Let $K_n^D$ be the regular zero-Dirichlet inverse of $\Delta_n$.
For angular degree $l$ its radial equation is
\[
 \mathcal D_lv=v''+\frac{n-1}{r}v'-\frac{l(l+n-2)}{r^2}v=f_l,
 \qquad v(1)=0.
\]
Regularity at the origin is that of the solution on $B^n$.

\begin{lemma}\label{gd:lem:Poisson}
For $f\in Y$ the following bounds hold:
\begin{equation}\label{gd:eq:Poisson-constants}
 \begin{split}
 \nrm{K_n^Df}_Y&\le C_K\nrm f_Y,\qquad C_K=\frac1{2n},\\
 \nrm{\partial_{x_1}K_n^Df}_{\mathcal Y_n}&\le C_x^K\nrm f_Y,\qquad
 \nrm{y\partial_yK_n^Df}_Y\le C_y^K\nrm f_Y,\\
 C_x^K&=\frac Rn+\frac R{n+2}+\frac{2(2R+1)}{n+1},\\
 C_E^K&=\frac1n+\frac1{n+2}+\frac2{n+1},
 \qquad C_y^K=C_E^K+RC_x^K.
 \end{split}
\end{equation}
Moreover,
\begin{equation}\label{gd:eq:Poisson-trace}
 (K_n^Df)_l'(1)=\int_0^1r^{l+n-1}f_l(r)\,dr,\qquad
 \nrm{\partial_rK_n^Df}_{W_1}\le\nrm f_Y.
\end{equation}
\end{lemma}
\begin{proof}
For $-\mathcal D_l$ the function $(1-r^2)/(2n)$ is a supersolution
with right side at least one. The radial maximum principle therefore gives
$\nrm v_\infty\le\nrm{f_l}_\infty/(2n)$.
For $l\ge2$, $-\mathcal D_l r=(l-1)(l+n-1)/r$; comparison with
$r\nrm{f_l}_\infty/d_l$, where $d_l=(l-1)(l+n-1)$, gives
$|v(r)|/r\le\nrm{f_l}_\infty/d_l$.
For $l=0$, apply the maximum principle on $B^n$ to the radial functions
$(1-r^2)\nrm{f_0}_\infty/(2n)\pm v$. For $l\ge1$, apply it on
$[\epsilon,1]$ and let $\epsilon\to0$. Continuity gives $v(\epsilon)=o(1)$ for $l\ge1$.
Full-ball Dirichlet regularity for continuous data~\cite{GT01} gives a $C^1$
solution; its constant and linear Taylor terms have angular degrees
zero and one. Thus $v(\epsilon)=o(\epsilon)$ for $l\ge2$, as
required for the second barrier.
Writing $v=r^l\phi$ and integrating its equation gives
\begin{equation}\label{gd:eq:radial-derivative}
 v'=\frac{lv}{r}+\mathcal I_l(r),\qquad
 \mathcal I_l(r)=r^{-l-n+1}\int_0^r s^{l+n-1}f_l(s)\,ds,\qquad
 |\mathcal I_l(r)|\le\frac r{l+n}\nrm{f_l}_\infty .
\end{equation}
For $l=0$ this gives $|v'|\le\nrm{f_0}_\infty/n$.

The Gegenbauer identity
$(1-\xi^2)G_l'=l(G_{l-1}-\xi G_l)$, obtained by differentiating
$H_l$, gives the cancellation
\[
 \partial_{x_1}(vG_l)=\xi\mathcal I_lG_l+(lv/r)G_{l-1}.
\]
Multiplication by $\xi=G_1$ costs $R$.
For even $l\ge2$, $l/d_l\le2/(n+1)$, and $h_{l-1}\le h_l$.
Thus the coefficient bound for $\partial_{x_1}K_n^D$ is at most
$\max\{R/n,R/(n+2)+2/(n+1)\}$.
For the Euler derivative $r\partial_r$ the corresponding bound is
$\max\{1/n,1/(n+2)+2/(n+1)\}$.
The enlarged constants in~\eqref{gd:eq:Poisson-constants} dominate
these bounds, and $y\partial_y=r\partial_r-x_1\partial_{x_1}$
gives $C_y^K$. Inputs have even degree, so no $l=1$ radial estimate
is required; $\partial_{x_1}$ may have odd outputs in $\mathcal Y_n$.

At $r=1$, \eqref{gd:eq:radial-derivative} and $v(1)=0$
give the moment formula. The boundary cosine frequencies of $G_l$
are at most $l$, hence
\[
 \nrm{\partial_rK_n^Df}_{W_1}
 \le\sum_l(l+1)h_l|v_l'(1)|
 \le\sum_l\frac{l+1}{l+n}h_l\nrm{f_l}_\infty\le\nrm f_Y .
\]
For continuous data the full-ball Poisson solution belongs to
$W^{2,p}(B^n)$ for every finite $p$ and to $C^1(\overline{B^n})$
for $p>n$, by Dirichlet regularity~\cite{GT01}. Approximation of the data proves the moment formula
for its classical normal trace.
\end{proof}

Define the closed compatible subspace and product space
\begin{equation}\label{gd:eq:X}
 \begin{split}
 \Gc&=\left\{g\in Y:\int_0^1r^{l+n-1}g_l(r)\,dr=0
                            \text{ for every even }l\right\},\\
 K_n&=K_n^D|_{\Gc},\qquad
 W=\left\{p=\sum_{j\ge0}p_{2j}w^{2j}:
                  p_{2j}\in\R,\ \nrm p_W=\sum_{j\ge0}R^{2j}|p_{2j}|<\infty\right\},\\
 X&=\Gc\times W,\qquad
 \nrm{(g,p)}_X=\nrm g_Y+\sigma\nrm p_W .
 \end{split}
\end{equation}
The moments are bounded functionals, so $\Gc$ is complete.
Lemma~\ref{gd:lem:Poisson} gives both zero traces of $K_ng$.
The normalised primitive is $\psi(w)=\int_0^wp(z)\,dz$.
The nonlinear equation is
\begin{equation}\label{gd:eq:F}
 F(g,p)=B(1+K_ng)=0 .
\end{equation}
It is affine in $g$ and cubic in $p$, thus a polynomial of degree four
in $(g,p)$.
The constant $\sigma$ changes the product norm, not this equation.

\begin{table}[htbp]
\centering\small
\def\gdRow#1#2#3#4#5#6#7#8{$#1$ & #3\\}
\begin{tabular}{@{}crrrrr@{}}
\toprule
$n$ & $\lambda$ & $\kappa$ & $C_x^K$ & $C_y^K$ & $L_0$\\
\midrule
\gdData
\bottomrule
\end{tabular}
\caption{Angular conversion, Poisson bounds, and the starting degree
$L_0$ of the high angular complement. Here $C_K=1/(2n)$ exactly.}
\label{gd:tab:Poisson}
\end{table}

\section{The symmetric-sector Dirichlet inverse}\label{gd:sec:spectral}

All operators in this section are evaluated at the finite shape $\psi_0$.
Set $e=a/2=\lambda$. Direct differentiation of $q^{-e}$ yields
\begin{equation}\label{gd:eq:conjugation}
 q^e\mathcal L(q^{-e}Z)=(\Delta_n+\mathcal V)Z,\qquad
 \mathcal V=\mu|p|^2-\frac{a(a-2)}2\frac{q_y}{yq}
                 -\frac{a(a-2)}4\frac{|\nabla q|^2}{q^2}.
\end{equation}
Here $\Delta_2=\partial_{x_1}^2+\partial_y^2$. The identity
$\Delta_2(yq)=0$ implies $\Delta_nq=(a-2)q_y/y$;
substitution cancels the first derivatives of $Z$.
Let $L=\Delta_n+\mathcal V$, $c=\rho b$, and $v_0=c^2$.

\subsection{The complete window}\label{gd:sec:window}
With reduced measure
$r^{n-1}\,dr\,(1-\xi^2)^{(n-3)/2}\,d\xi$, an orthonormal Dirichlet basis
in the invariant even sector is
\begin{equation}\label{gd:eq:spectral-basis}
 b_{l,\alpha}(r,\xi)=
 \frac{\sqrt2\,r^{-\lambda}J_{l+\lambda}(\alpha r)}
 {|J_{l+\lambda+1}(\alpha)|}
 \frac{C_l^\lambda(\xi)}{\sqrt{h_l^{(2)}}},\qquad
 h_l^{(2)}=\frac{\pi2^{1-2\lambda}\Gamma(l+2\lambda)}
 {l!(l+\lambda)\Gamma(\lambda)^2},
\end{equation}
where $l$ is even and $\alpha$ runs through all positive zeros of
$J_{l+\lambda}$. The Bessel equation and differentiation in $\alpha$
give the radial norm, as in~\eqref{r2:eq:basis}.
Angular orthogonality and the self-adjoint ball Dirichlet spectral
decomposition give completeness. The analytic weight $h_l$ in
\eqref{gd:eq:Y} is distinct from the $L^2$ normalisation $h_l^{(2)}$.

Let $\PP$ include every state with
\begin{equation}\label{gd:eq:window}
 |\alpha-c|<12,\qquad \QQ=I-\PP .
\end{equation}
There are two complete coverage arguments for this window.

\begin{lemma}\label{gd:lem:coverage}
Opposite-sign root brackets satisfying the spacing and guard conditions
below enumerate all states of~\eqref{gd:eq:window}.
\end{lemma}
\begin{proof}
Put $\nu=l+\lambda$. The leading series of $J_\nu$ makes
$J_\nu,J_\nu'>0$ near zero. At a first stationary point $x<\nu$
with $J_\nu>0$, the Bessel equation would give
$J_\nu''=(\nu^2/x^2-1)J_\nu>0$, contrary to a first loss of
positive derivative. A first zero of $J_\nu$ is also impossible
while its derivative is positive. Thus both remain positive for $x<\nu$.
If $J_\nu'(\nu)=0$, its equation gives $J_\nu''(\nu)=0$ and
$J_\nu'''(\nu)=-2J_\nu(\nu)/\nu<0$; Taylor expansion would make
$J_\nu'$ negative immediately to the left. Hence
$J_\nu,J_\nu'>0$ also at $\nu$.

For $z_\nu(x)=\sqrt{x}J_\nu(x)$ one has
\[
 z_\nu''+\left(1-\frac{\nu^2-1/4}{x^2}\right)z_\nu=0.
\]
On an interval ending at $B>\nu$, its coefficient is at most
$q_B=1-(\nu^2-1/4)/B^2$. The integration-by-parts and interval
Poincar\'e argument of Lemma~\ref{r2:lem:zero-coverage} gives spacing
at least $\pi/\sqrt{q_B}$. Starting from
$z_\nu(\nu),z_\nu'(\nu)>0$, comparison with the constant-coefficient
equation gives first-zero distance at least $\pi/(2\sqrt{q_B})$.
Consequently the first bracket ending strictly before
$\nu+3\pi/(2\sqrt{q_B})$ cannot miss an earlier zero.
If the outer endpoints of successive brackets have separation less
than $2\pi/\sqrt{q_B}$, no zero is missed between them.
Each bracket has opposite nonzero signs and width below the spacing;
it therefore contains exactly one simple zero.
Enumeration from $\nu$ through a guard root above $c+12$ covers the
window in that angular sector. Positivity up to $\nu$ excludes
all orders $\nu\ge c+12$.

For the second argument evaluate $J_\nu$ on the exact grid
$c-12+k/4$, $0\le k\le96$, in each remaining sector.
Here $\nu\ge3/2$, so the differential-equation coefficient is at
most one on the positive axis. Distinct positive roots are separated
by at least $\pi$. A grid cell contains at most one root.
Simplicity, and nonzero endpoint signs, make equal signs equivalent
to no root and opposite signs equivalent to one root.
Exclude equality at the two window edges. The same positivity bound excludes all larger angular orders.
\end{proof}

\subsection{The polynomial potential and its full error}\label{gd:sec:potential}
Index shape coefficients throughout this part by their actual odd degree:
$\psi_0(w)=\sum_{j\ {\rm odd}}c_jw^j$, $c_1=b$.
For a full, retained or discarded list define
\[
 s_1=\sum_{j>1}j|c_j|,\qquad
 s_2=\sum_{j>1}j(j-1)|c_j|,\qquad
 s_3=\frac13\sum_{j>1}j(j-1)(j-2)|c_j|.
\]
The integral formula for $q$ gives
$|p-b|,|q-b|\le s_1$, $|q_{x_1}|,|q_y|\le s_2$ and
$|q_y/y|\le s_3$. For the last bound, $q_y(x_1,0)=0$ because $q$ is even in $y$, and
$q_{yy}=-\int_0^1s^2\re p''(x_1+isy)\,ds$ gives
$|q_{yy}|\le s_3$, hence $|q_y|\le s_3|y|$.
Let the subscripts $f,t,d$ denote full, retained and discarded data.
Retain degrees at most $61$ in the polynomial potential only.
Put $q_*=b-s_{1,f}>0$ and $A_n=a(a-2)$. Then
\begin{equation}\label{gd:eq:vfull}
 \nrm{\mathcal V-v_0}_\infty\le
 v_f:=\mu(2bs_{1,f}+s_{1,f}^2)
 +\frac{A_ns_{3,f}}{2q_*}+\frac{A_ns_{2,f}^2}{2q_*^2}.
\end{equation}
The difference caused by discarding shape degrees is at most
\begin{equation}\label{gd:eq:shape-error}
 \begin{split}
 \delta_{\rm shape}={}&
 \mu(2b+s_{1,f}+s_{1,t})s_{1,d}\\
 &+\frac{A_n}2\left(\frac{s_{3,d}}{q_*}
                    +\frac{s_{3,t}s_{1,d}}{q_*^2}\right)\\
 &+\frac{A_n}4\left\{
 \frac{2(s_{2,f}+s_{2,t})s_{2,d}}{q_*^2}
 +\frac{2s_{2,t}^2(2b+s_{1,f}+s_{1,t})s_{1,d}}{q_*^4}\right\}.
 \end{split}
\end{equation}
The first line bounds the difference of $|p|^2$.
The second uses $|q_f^{-1}-q_t^{-1}|\le s_{1,d}/q_*^2$.
For the last line split the gradient square and its squared reciprocal;
$|\nabla q|^2\le2s_2^2$ and
$|q_f^{-2}-q_t^{-2}|\le
(2b+s_{1,f}+s_{1,t})s_{1,d}/q_*^4$ give the displayed bound.

Replace $q_t^{-1}$ by the polynomial
$S_4=b^{-1}\sum_{k=0}^4[-(q_t-b)/b]^k$ and $q_t^{-2}$ by $S_4^2$.
Since $s_{1,t}\le s_{1,f}<b$, the geometric remainder is at most
$\delta_S=(s_{1,t}/b)^5/(b-s_{1,t})$.
Also $|S_4|\le(b-s_{1,t})^{-1}$.
The resulting polynomial potential $\mathcal V_t$ satisfies
\begin{equation}\label{gd:eq:potential-error}
 \nrm{\mathcal V-\mathcal V_t}_\infty\le
 \delta_V:=\delta_{\rm shape}
 +\frac{A_n}2s_{3,t}\delta_S
 +\frac{A_ns_{2,t}^2\delta_S}{b-s_{1,t}}.
\end{equation}
This includes the squared-reciprocal error
$2\delta_S/(b-s_{1,t})$. As a polynomial in $(x_1,y)$,
$\mathcal V_t$ has degree at most $598$: the term
$|\nabla q_t|^2S_4^2$ has degree $118+480$.
Set $V_t=\mathcal V_t-v_0$ and $v_*=v_f+\delta_V$.

\subsection{The Gram identity and Schur complement}\label{gd:sec:Schur}
Form the exact finite matrices
\begin{equation}\label{gd:eq:matrices}
 \mathsf B=\PP V_t\PP,\qquad \mathsf B_2=\PP V_t^2\PP,\qquad
 H_{\mathsf P}=\diag(\alpha^2-v_0)-\mathsf B,\qquad G=\mathsf B_2-\mathsf B^2.
\end{equation}
The sans-serif symbols distinguish these matrices from the differential
operator $B$ in~\eqref{gd:eq:B}. The Gram identity is
\begin{equation}\label{gd:eq:Gram}
 G=(\QQ V_t\PP)^*(\QQ V_t\PP)\succeq0.
\end{equation}
It is the identity~\eqref{r2:eq:Gram} on the entire orthogonal complement;
in particular it includes every omitted radial and angular mode.

On the complement the ball operator $-\Delta_n-v_0$ has absolute
spectral gap at least $12(2c-12)$, with $c>12$.
Therefore
\[
 M_Q=\QQ(-\Delta_n-v_0-V_t)\QQ,\qquad
 \nrm{M_Q^{-1}}\le\gamma^{-1},\qquad
 \gamma=12(2c-12)-v_*>0 .
\]
This is the Neumann factorisation~\eqref{r2:eq:unbounded-factor};
the operator is allowed to be indefinite.

\begin{proposition}\label{gd:prop:Schur}
Let $C$ be a square real matrix of the window size, $D$ a real
diagonal matrix with nonzero diagonal, and suppose
\begin{equation}\label{gd:eq:Schur-gates}
 \begin{gathered}
 \nrm{C^TC-I}_2\le\epsilon_C<1,\\
 \nrm{|D|^{-1/2}(C^TH_{\mathsf P}C-D)|D|^{-1/2}}_2\le\epsilon_{\rm S},\\
 \gamma^{-1}\tr(|D|^{-1}C^TGC)\le\eta_{\rm S},\qquad\epsilon_{\rm S}+\eta_{\rm S}<1.
 \end{gathered}
\end{equation}
Set
\begin{equation}\label{gd:eq:L2}
 \begin{split}
 S_*&=\frac{1+\epsilon_C}{\min_i|D_i|(1-\epsilon_{\rm S}-\eta_{\rm S})},\\
 B_{\rm pol}&=\bigl(1+(v_*/\gamma)^2\bigr)S_*+\gamma^{-1},\\
 B_{L^2}&=\frac{B_{\rm pol}}{1-B_{\rm pol}\delta_V}\qquad(B_{\rm pol}\delta_V<1).
 \end{split}
\end{equation}
Then $L$ has a symmetric-sector Dirichlet inverse with
$\nrm{L_D^{-1}}_{L^2\to L^2}\le B_{L^2}$.
\end{proposition}
\begin{proof}
The proof of Proposition~\ref{r2:prop:Schur} applies with this ball basis
and $V_t$, the two defects being bounded directly in the $\ell^2$
operator norm. It uses only the orthogonal splitting of the Dirichlet
operator domain, the complement gap and~\eqref{gd:eq:Gram}.
In particular, for $T=\QQ V_t\PP C|D|^{-1/2}$,
\[
 \nrm{T^*M_Q^{-1}T}\le\gamma^{-1}\nrm{T^*T}
 \le\gamma^{-1}\tr(T^*T)\le\eta_{\rm S}.
\]
After congruence the Schur matrix is $\operatorname{sgn}(D)$ plus a
perturbation of norm at most $\epsilon_{\rm S}+\eta_{\rm S}$, so its inverse is bounded
by $S_*$. The block inverse has norm at most
$(1+(v_*/\gamma)^2)S_*+\gamma^{-1}$, exactly as in
\eqref{r2:eq:BLref}. Finally the bounded perturbation
$\mathcal V-\mathcal V_t$ is transferred by the factorisation of
\eqref{r2:eq:transfer}, giving $B_{L^2}$.
\end{proof}

The reduced measure in~\eqref{gd:eq:spectral-basis} differs from
Lebesgue measure on invariant functions on $B^n$ by the factor
$|S^{n-2}|$. On the even invariant sector, the map
$f(r,\xi)\mapsto |S^{n-2}|^{-1/2}f(|x|,x_1/|x|)$ is a unitary
identification of the reduced space with the invariant subspace of
$L^2(B^n,dx)$. It intertwines the Dirichlet operators, so $B_{L^2}$
is also the inverse bound on this subspace with Lebesgue measure.

\subsection{Quadrature enclosures}\label{gd:sec:quadrature}
Here are the finite expressions defining the matrix enclosures.
For odd $n$, the angular rule is Gauss--Legendre with polynomial weight
$(1-\xi^2)^{(n-3)/2}$.
For even $n$, it is Gauss--Chebyshev of the second kind, with
\[
 \xi_k=\cos\frac{k\pi}{N_\theta+1},\qquad
 \omega_k=\frac{\pi}{N_\theta+1}
             \sin^2\frac{k\pi}{N_\theta+1}\quad(1\le k\le N_\theta),
\]
and remaining polynomial weight $(1-\xi^2)^{(n-4)/2}$.
Choose the angular order so that the entire polynomial degree of each
integrand is at most $2N_\theta-1$; this makes angular integration exact.
This covers both integral and half-integral $\lambda$.

For radial Gauss--Legendre order $N_r$, the Bernstein ellipse of parameter
two in $2r-1$ gives $|r|\le9/8$ and $|\im r|\le3/8$.
Poisson's Bessel integral~\cite[10.9.4]{DLMF} gives
\[
 |z^{-\nu}J_\nu(z)|\le
 \frac{e^{|\im z|}}{2^\nu\Gamma(\nu+1)}.
\]
For each window basis vector its radial integrand factor, including
$r^{(n-1)/2}$, is bounded by
\begin{equation}\label{gd:eq:ellipse}
 \hat e_i=\frac{\sqrt2C_l^\lambda(1)}
 {\sqrt{h_l^{(2)}}|J_{\nu+1}(\alpha)|}
 \frac{\alpha^\nu(9/8)^{l+(n-1)/2}e^{3\alpha/8}}
 {2^\nu\Gamma(\nu+1)},\qquad \nu=l+\lambda .
\end{equation}
The products are entire: after removing the Bessel leading powers,
the radial power is $l+l'+n-1$, an integer.
Let $\hat e_*=\max_i\hat e_i$ and let $V_e$ bound $|V_t(r,\xi)|$ for $r$ on the ellipse
and $-1\le\xi\le1$, using the $s_i$ sums at $9/8$ and the finite
polynomial $S_4$. The entrywise radial errors can be taken as
\begin{equation}\label{gd:eq:quadrature-errors}
 q_1=32\hat e_*^2(1+V_e)2^{-2N_r},\qquad
 q_2=32\hat e_*^2(1+V_e^2)2^{-2N_r}.
\end{equation}
Indeed the Chebyshev coefficients of an analytic integrand bounded by
$M$ are at most $2M2^{-k}$. The radial Gaussian rule is exact through
degree $2N_r-1$; bounding both its positive mass and the integration
mass and summing the remaining geometric series gives the factor
$16$. The angular integration has mass at most two, giving $32$.

Let $u_{\rm fl}=2^{-52}$,
$\gamma_k=ku_{\rm fl}/(1-ku_{\rm fl})$, and $\tau=10^{-280}$.
Use these bounds only when $ku_{\rm fl}<1$, with finite intermediate
values. Regard every binary64 export as its exact dyadic rational.
At a tensor node the radial and angular factors are
\[
 \frac{\sqrt{2\omega_rr}J_\nu(\alpha r)}{|J_{\nu+1}(\alpha)|},
 \qquad
 C_l^\lambda(\xi)
 \sqrt{\omega_\theta w_{\rm poly}(\xi)/h_l^{(2)}}.
\]
Here $\omega_r$ is the radial Gauss--Legendre weight on $[0,1]$,
$\omega_\theta$ is the angular weight, and
$w_{\rm poly}(\xi)=(1-\xi^2)^{(n-3)/2}$ for odd $n$,
$(1-\xi^2)^{(n-4)/2}$ for even $n$.
Let their exported vectors be $r_i,a_i$, with enclosed norm errors
$e_{r,i},e_{a,i}$. The column error is bounded by
\begin{equation}\label{gd:eq:tensor-error}
 e_i^{\rm col}=e_{r,i}\nrm{a_i}_2+\nrm{r_i}_2e_{a,i}
              +e_{r,i}e_{a,i}
              +u_{\rm fl}\nrm{r_i}_2\nrm{a_i}_2+\tau .
\end{equation}
Thus $e_{\rm bas}=(\sum_i(e_i^{\rm col})^2)^{1/2}$
bounds the basis matrix Frobenius error.
Bessel argument uncertainty is bounded by differentiating the Poisson
integral:
\[
 \sup_{0\le x\le X}|J_\nu'(x)|
 \le\frac{\nu X^{\nu-1}+X^\nu}{2^\nu\Gamma(\nu+1)}
 \quad(\nu\ge1).
\]

The potential arrays are given by the exact identities
\begin{equation}\label{gd:eq:potential-arrays}
 \begin{aligned}
 q-b&=\sum_{j>1}c_jr^{j-1}U_{j-1}(\xi),&
 \re p-b&=\sum_{j>1}jc_jr^{j-1}T_{j-1}(\xi),\\
 q_{x_1}&=\sum_{j>1}jc_jr^{j-2}U_{j-2}(\xi),&
 q_y/y&=-2\sum_{j>1}c_jr^{j-3}C_{j-3}^{(2)}(\xi).
 \end{aligned}
\end{equation}
Each array is expanded in $r^iT_m(\xi)$ using
\[
 \begin{split}
 U_m&=\sum_{k=0}^{\lfloor m/2\rfloor}(2-\delta_{m,2k})T_{m-2k},\\
 C_N^{(2)}&=\sum_{k=0}^{\lfloor N/2\rfloor}
       (2-\delta_{N,2k})(k+1)(N-k+1)T_{N-2k}.
 \end{split}
\]
Here $S_{\rm pol}$ is the absolute sum of the resulting coefficients.
These coefficients are exported, and only the $T_m$ are evaluated by
the recurrence below.
For degree at most $d$, radial power error costs $\gamma_{d+1}$.
The recurrence at an exported node $|\xi|<1$ is
$\widehat T_{j+1}=2\xi\widehat T_j-\widehat T_{j-1}$,
$\widehat T_0=1$, $\widehat T_1=\xi$.
Suppose $10u_{\rm fl}(d+1)^3<10^{-3}$. If
$|\widehat T_i|\le1.001$ for $i\le j$, the product and subtraction
have total rounding error $|\epsilon_j|\le6u_{\rm fl}$, apart from
the separate underflow allowance. Since $|U_m|\le m+1$,
\[
 \left|\widehat T_{j+1}-T_{j+1}(\xi)\right|
 =\left|\sum_{i=1}^j\epsilon_iU_{j-i}(\xi)\right|
 \le3u_{\rm fl}j(j+1)<10^{-3}.
\]
Thus $|\widehat T_{j+1}|\le1.001$ by induction, and the error at
every degree $k\le d$ is at most
$3u_{\rm fl}d(d-1)\le10u_{\rm fl}(d+1)^3$, the recurrence allowance
of the array error.
Node-export uncertainty is separate: $|T_m|\le1$ and
$|T_m'|\le m^2$ imply the array error
$S_{\rm pol}(d\,\delta r+d^2\delta\xi)$.
Add coefficient export errors and both dense-product allowances.
Propagate them through the seven-variable polynomial potential by
interval differentiation on a box containing both the exact and
exported inputs, so the mean value theorem applies; an additional $2\gamma_{400}$ times the
absolute arithmetic majorant covers its fewer than $100$ scalar
operations. Let $\delta_{\rm pt}$ be the resulting pointwise potential error,
and $v_{\rm exp}$ the exported absolute bound.

For an exported basis matrix $E_F$ with $\nrm{E_F}_F\le K_F$,
window size $m$ and total node count $N_q$, valid matrix errors are
\begin{equation}\label{gd:eq:matrix-errors}
 \begin{split}
 e_B={}&v_*(2K_Fe_{\rm bas}+e_{\rm bas}^2)+\delta_{\rm pt} K_F^2
              +\gamma_{N_q+4}v_{\rm exp}K_F^2+mq_1+\tau,\\
 e_{B_2}={}&v_*^2(2K_Fe_{\rm bas}+e_{\rm bas}^2)
       +\delta_{\rm pt}(2v_*+\delta_{\rm pt})K_F^2
       +\gamma_{N_q+6}v_{\rm exp}^2K_F^2+mq_2+\tau .
 \end{split}
\end{equation}
These follow by adding and subtracting the exact factors in
$E_F^T\diag(V_F)E_F$. Include root-diagonal export, subtraction
and symmetrisation errors in $e_H$.
Bound $\epsilon_C$ by the Frobenius norm of the computed $C^TC-I$
plus $\gamma_{m+3}\nrm C_F^2+\tau$.
Let $H_F$ be the exported symmetrised $H_{\mathsf P}$. Bound the numerator of
$\epsilon_{\rm S}$ by the Frobenius norm of the computed $C^TH_FC-D$ plus
\[
 \gamma_{2m+8}(\nrm C_F^2\nrm{H_F}_F+\nrm D_F)+\tau
                    +(1+\epsilon_C)e_H.
\]
Division by $\min|D_i|$ bounds the normalised defect in
\eqref{gd:eq:Schur-gates}.
For the exported Gram difference $G_F=\mathsf B_{2,\rm exp}-\mathsf B_{\rm exp}^2$ add
\begin{equation}\label{gd:eq:Gram-error}
 e_G=e_{B_2}+e_B(2v_*+e_B)+\gamma_{m+4}\nrm{\mathsf B_{\rm exp}}_F^2
                  +u_{\rm fl}(\nrm{\mathsf B_{2,\rm exp}}_F+\nrm{\mathsf B_{\rm exp}^2}_F)+\tau .
\end{equation}
The weighted trace allowance is
\[
 \frac{\gamma_{4m+20}\nrm C_F^2\nrm{G_F}_F
              +m(1+\epsilon_C)e_G}{\min|D_i|}+\tau .
\]
Divide the upper trace bound by $\gamma$ to obtain $\eta_{\rm S}$.
For an exported array $A_F$ with $N$ entries, let $s_F$ be its computed
sum of squares. The Frobenius reduction uses
\[
 \nrm{A_F}_F\le\sqrt{\frac{s_F}{1-\gamma_{2N+2}}+\tau}.
\]

\section{Transfer of the Dirichlet inverse to the working norm}
\label{gd:sec:transfer}

\subsection{Low angular degrees by heat smoothing}\label{gd:sec:heat}
Let $M\ge\nrm{\mathcal V}_Y$, so $|\mathcal V|\le M$, and let $Lv=f$
with zero Dirichlet data in the invariant sector.
For the Dirichlet heat semigroup with this bounded real potential,
\begin{equation}\label{gd:eq:heat}
 \nrm{e^{tL}}_{2\to\infty}\le e^{Mt}(8\pi t)^{-n/4},
 \qquad
 \nrm{e^{tL}}_{\infty\to\infty}\le e^{Mt}\quad(t>0).
\end{equation}
To obtain these bounds, the maximum principle first dominates the
Dirichlet heat kernel by the positive free Gaussian kernel, after
extension of data by zero. Iterating Duhamel's formula for multiplication
by $\mathcal V$, taking absolute values, and using $|\mathcal V|\le M$
sums to the factor $e^{Mt}$.
The $L^2$ norm of
$(4\pi t)^{-n/2}\exp(-|x|^2/(4t))$ is $(8\pi t)^{-n/4}$.
Its $L^1$ norm is one. This proves both estimates, also on invariant
inputs with the full Lebesgue measure.

The spectral inverse places $v$ in the Dirichlet operator domain.
Integrating the semigroup derivative gives
\[
 v=e^{tL}v-\int_0^t e^{sL}f\,ds .
\]
Since $\nrm f_2\le\sqrt{|B^n|}\nrm f_Y$ and
$\nrm v_2\le B_{L^2}\nrm f_2$, it follows that
\begin{equation}\label{gd:eq:Bsup}
 \nrm v_\infty\le B_\infty\nrm f_Y,\qquad
 B_\infty=e^{Mt}(8\pi t)^{-n/4}\sqrt{|B^n|}B_{L^2}
                    +\frac{e^{Mt}-1}{M},
 \quad |B^n|=\frac{\pi^{n/2}}{\Gamma(n/2+1)}.
\end{equation}
We take $t=1/M$. In particular $v$ is bounded before any analytic
coefficient norm of $v$ is used.
The equation $\Delta_n v=f-\mathcal Vv$ now has bounded data;
full-ball elliptic regularity gives $v\in W^{2,p}$ for every finite $p$,
and hence a continuous radial expansion~\cite{GT01}.

The dimension of degree-$l$ spherical harmonics on $S^{n-1}$ is
\[
 d_l^{\rm sph}=\frac{2l+n-2}{n-2}\binom{l+n-3}{l}.
\]
The addition theorem gives
$\int_{S^{n-1}}G_l^2\,d\mathcal H^{n-1}/|S^{n-1}|=1/d_l^{\rm sph}$.
At each $r$, angular projection and Cauchy--Schwarz imply
\begin{equation}\label{gd:eq:low-angular}
 |v_l(r)|\le\sqrt{d_l^{\rm sph}}\nrm{v(r,\cdot)}_\infty
                 \le\sqrt{d_l^{\rm sph}}B_\infty\nrm f_Y .
\end{equation}
Only finitely many angular degrees will use this estimate.

\subsection{The entire high angular complement}\label{gd:sec:high}
Let $v_Y\ge\nrm{\mathcal V-v_0}_Y$.
The tabulated $L_0$ is the least even degree with $d_h>2v_Y$.
For the argument, choose any even $L_0$ with
\begin{equation}\label{gd:eq:high}
 d_h=L_0(L_0+n-2)-v_0>v_Y,\qquad \beta=v_Y/d_h<1 .
\end{equation}
For every $l\ge L_0$, the positive zero-order part of $-\mathcal D_l-v_0$
is at least $d_h$. The radial maximum principle therefore bounds its
zero-Dirichlet inverse in supremum norm by $1/d_h$.
First define this high-mode inverse in $L^2$ by coercivity.
For a source in $Y$, its finite angular truncations converge in $Y$;
the radial bound makes their inverse images converge in $Y$ and in
$L^2$. Closedness of the Dirichlet operator places their limit in
$H^2\cap H^1_0$ and identifies it with the same inverse.
Let $\Pi_H$ be angular projection to these degrees, and let
$v_{\rm low}=(I-\Pi_H)v$. By~\eqref{gd:eq:low-angular} this low part
belongs to $Y$. Solve
\[
 (\Delta_n+v_0)\tilde v_H
   =\Pi_H\bigl(f-(\mathcal V-v_0)(v_{\rm low}+\tilde v_H)\bigr)
\]
with zero Dirichlet data by contraction on the high part of $Y$.
The algebra bound gives contraction factor at most $\beta$.

We next identify $\tilde v_H$ with $\Pi_Hv$, as in the proof of
Proposition~\ref{r2:prop:Yinverse}.
Its Poisson right side, including $-v_0\tilde v_H$, lies in $Y$.
Thus $\tilde v_H$ equals its zero-Dirichlet Poisson potential, and belongs
to $H^2\cap H^1_0$. The original high projection belongs to the same
operator domain. Their difference $z$ has only degrees $l\ge L_0$ and
satisfies
\[
 (\Delta_n+v_0)z=-\Pi_H((\mathcal V-v_0)z).
\]
Angular orthogonality and $r\le1$ give
$\int|\nabla z|^2\ge L_0(L_0+n-2)\int|z|^2$.
Testing the equation with $-z$, using self-adjointness of $\Pi_H$,
gives
\[
 d_h\int|z|^2
 \le\int|\nabla z|^2-v_0\int|z|^2
 =\int(\mathcal V-v_0)|z|^2\le v_Y\int|z|^2 .
\]
Hence $z=0$. This proves membership of the original solution in $Y$
before rearranging any inequality involving its $Y$ norm.

Add the low norm to the high contraction inequality, use $h_l\le R^l$,
and obtain
\begin{equation}\label{gd:eq:BY}
 \nrm{L_D^{-1}f}_Y\le B_Y\nrm f_Y,\qquad
 B_Y=\frac{B_\infty
       \displaystyle\sum_{\substack{0\le l<L_0\\l\ {\rm even}}}
                R^l\sqrt{d_l^{\rm sph}}+d_h^{-1}}{1-\beta}.
\end{equation}
The Poisson trace formula applied to $\Delta_n v=f-\mathcal Vv$
also gives
\begin{equation}\label{gd:eq:trace-conjugated}
 \nrm{\partial_rv}_{W_1}\le(1+MB_Y)\nrm f_Y .
\end{equation}

\subsection{Undoing the conjugation}\label{gd:sec:unconjugation}
Let $\nrm{q-b}_Y\le q_b<b$ and define
\begin{equation}\label{gd:eq:source-multipliers}
 \mathfrak I=(b-q_b)^{-1},\qquad
 S_q=b^e(1-q_b/b)^{-|e-1|},\qquad
 M_0=b^{-e}(1-q_b/b)^{-e}.
\end{equation}
The binomial series gives
$\nrm{bq^{e-1}}_Y\le S_q$ and $\nrm{q^{-e}}_Y\le M_0$.
For a positive exponent, absolute falling-factorial coefficients are
bounded by the corresponding rising-factorial coefficients.
This proves the stated enlarged bound for either integral or
half-integral $e$.

For the original Dirichlet equation $B\Phi=f$, set $Z=q^e\Phi$.
Then $LZ=bq^{e-1}f$. Write
\[
 \begin{gathered}
 X_*\ge\nrm{q_{x_1}}_{\mathcal Y_n},\quad Y_*\ge\nrm{q_y/y}_Y,\quad
 Y_2\ge\nrm{y^2}_Y,\quad Q^{[1]}\ge\mathfrak m_1(q-b),\\
 B_Z=S_qB_Y,\qquad J_Z=S_q+MB_Z .
 \end{gathered}
\]
Consequently $\nrm Z_Y\le B_Z\nrm f_Y$ and
$\nrm{\Delta_nZ}_Y\le J_Z\nrm f_Y$.
Define
\begin{equation}\label{gd:eq:JW}
 \begin{split}
 J_W=M_0\bigl[&
 J_Z+2e\mathfrak I(X_*C_x^K+Y_*C_y^K)J_Z\\
 &+\{e(n-4)\mathfrak I Y_*
       +e(e+1)\mathfrak I^2(X_*^2+Y_2Y_*^2)\}B_Z\bigr],\\
 N_W&=M_0(1+e\mathfrak I Q^{[1]})J_Z .
 \end{split}
\end{equation}

\begin{proposition}\label{gd:prop:original-inverse}
The original symmetric-sector zero-Dirichlet inverse satisfies
\[
 \nrm{\Delta_nB_D^{-1}f}_Y\le J_W\nrm f_Y,\qquad
 \nrm{\partial_rB_D^{-1}f}_{W_1}\le N_W\nrm f_Y .
\]
\end{proposition}
\begin{proof}
Expand
\[
 \begin{split}
 \Delta_n(q^{-e}Z)=q^{-e}\bigl[
 \Delta_nZ-2e q^{-1}\nabla q\cdot\nabla Z
 +\{-e q^{-1}\Delta_nq
       +e(e+1)q^{-2}|\nabla q|^2\}Z\bigr].
 \end{split}
\]
The zero-Dirichlet identity $Z=K_n^D\Delta_nZ$ and
Lemma~\ref{gd:lem:Poisson} control both gradient terms.
Use $\Delta_nq=(n-4)q_y/y$ and
$|\nabla q|^2=q_{x_1}^2+y^2(q_y/y)^2$ to obtain $J_W$.
At the boundary $Z=0$, so $\partial_r\Phi=q^{-e}\partial_rZ$.
Lemma~\ref{gd:lem:Poisson} bounds $\partial_rZ$ by $J_Z\nrm f_Y$
in $W_1$. Moreover
\[
 \nrm{q^{-e}|_\partial}_{W_1}
 \le M_0+eM_0\mathfrak I\nrm{\partial_\theta q|_\partial}_{W_0}
 \le M_0(1+e\mathfrak I Q^{[1]}).
\]
The boundary algebra of~\eqref{r2:eq:Wone} gives $N_W$.
\end{proof}

\begin{table}[htbp]
\centering\small
\def\gdRow#1#2#3#4#5#6#7#8{$#1$ & #4\\}
\begin{tabular}{@{}crrrrr@{}}
\toprule
$n$ & states & $B_{L^2}$ & $B_Y$ & $J_W$ & $N_W$\\
\midrule
\gdData
\bottomrule
\end{tabular}
\caption{Complete Dirichlet inverse bounds. $B_Y$ is for the conjugated
operator; $J_W,N_W$ are for the original operator $B$.}
\label{gd:tab:inverse}
\end{table}

\section{The exactly clamped centre and its residual}\label{gd:sec:centre}

\subsection{Finite field and infinite clamping lift}
\label{gd:sec:clamp}
The finite input arrays are exact decimals $c_{2j+1},a_j,s_j$,
$0\le j\le J$, with $s_j>0$: these are the keys \texttt{c},
\texttt{a}, \texttt{scales} of the frozen centre, and $J+1$ is the
mode count of Table~\ref{gd:tab:inputs}. Their shape convention is
$\psi_0(w)=\sum_{j\ {\rm odd}}c_jw^j$, $b=c_1$ and $p_0=\psi_0'$.
Put, in unscaled meridian coordinates $(X_1,Y_\psi)$,
\begin{equation}\label{gd:eq:physical-field}
 u_*=\sum_{j=0}^J\frac{a_j}{s_j}
 \mathcal R^{-\lambda}J_{2j+\lambda}(\rho\mathcal R)
 G_{2j}(X_1/\mathcal R),\qquad
 \mathcal R^2=X_1^2+Y_\psi^2,\qquad U_*=u_*\circ\Theta_{\psi_0}.
\end{equation}
The Bessel power series expresses each summand as a solid harmonic
times an entire series in $\mathcal R^2$, and separation gives
$(\Delta_n+\mu)u_*=0$. The $s_j$ are fixed divisors; no identity
between them and Bessel values is required. By~\eqref{gd:eq:physical-multiplier},
$BU_*=0$ exactly.

Let $D_l,N_l$ be the full Gegenbauer coefficients of
$D=U_*|_{r=1}-1$ and $N=\partial_rU_*|_{r=1}$.
The double-trace lifts are
\begin{equation}\label{gd:eq:lifts}
 h_l^D=r^l[1-l(r^2-1)/2]G_l,\qquad
 h_l^N=\tfrac12r^l(r^2-1)G_l .
\end{equation}
Their value and radial derivative traces are respectively $(G_l,0)$
and $(0,G_l)$.
Define the exact centre by
\begin{equation}\label{gd:eq:clamp}
 \ell=\sum_l(D_lh_l^D+N_lh_l^N),\qquad
 U_0=U_*-\ell,\qquad g_0=\Delta_nU_0,\qquad x_0=(g_0,p_0).
\end{equation}
The boundary decay proved below gives analytic convergence on a
neighbourhood of the closed ball. Thus $U_0=1$ and $\partial_rU_0=0$
on the sphere, $g_0\in\Gc$, and $U_0=1+K_ng_0$ by Dirichlet uniqueness.
This is the clamping construction of~\eqref{r2:eq:clamped}
with solid zonal harmonics in place of planar harmonics.

\subsection{Complex majorants for the finite field}\label{gd:sec:strip}
Choose $S>R>1$, $\delta>0$, and $r_0=1/\rho<1$.
For $0\le r\le1$, $|\im\theta|\le\log S$, allow independent
perturbations of size $\delta$ in $w=re^{i\theta}$ and
$\bar w=re^{-i\theta}$. Define
\begin{equation}\label{gd:eq:strip-parameters}
 \begin{split}
 s&=b\delta S/r_0+
       S\sum_{j>1}|c_j|(S+\delta/r_0)^j,\qquad d_s=2s+s^2/b,\\
 z_*&=\sum_{j\ {\rm odd}}|c_j|(r_0S+\delta)^j,\qquad
 H_s=S(b+s),\qquad A_s=H_s/(b-d_s).
 \end{split}
\end{equation}
Assume $d_s<b$.
For $r\ge r_0$ write the two perturbed coordinates as
$X_1\pm iY_\psi=re^{\pm i\theta}A_\pm$.
Then $|A_\pm-b|\le s$.
To verify the nonlinear part of this estimate, set
$\varsigma=|re^{\pm i\theta}|\in[r_0/S,S]$.
For odd $m\ge3$, $f(\varsigma)=(\varsigma+\delta)^m/\varsigma$ has derivative
$f'(\varsigma)=(\varsigma+\delta)^{m-1}((m-1)\varsigma-\delta)/\varsigma^2$, so $f$ decreases for
$\varsigma<\delta/(m-1)$ and increases afterwards; hence its supremum
on $[r_0/S,S]$ is attained at an endpoint. Since $r_0<1<S$,
\[
 \begin{split}
 f(S)&=(S+\delta)^m/S\le S(S+\delta/r_0)^m,\\
 f(r_0/S)&=Sr_0^{m-1}(1/S+\delta/r_0)^m
                       \le S(S+\delta/r_0)^m.
 \end{split}
\]
The linear perturbation is at most $b\delta S/r_0$.

The product $A_+A_-$ lies in the disc of radius $2bs+s^2=bd_s<b^2$
about $b^2$, so it admits a square root $B_s$ with positive real part.
It satisfies $|B_s-b|\le d_s$ because $|B_s+b|\ge b$.
The complex radius is $rB_s$, with positive real part
and $|\im(rB_s)|\le d_s$.
The positive cosine expansion of the solid harmonic gives
$|H_l(X_1,Y_\psi)|\le(rH_s)^l$ and
$|G_l(X_1/\mathcal R)|\le A_s^l$.

For $\nu>0$, $\re z>0$, Schl\"afli's integral~\cite[10.9.6]{DLMF} gives
\[
 |J_\nu(z)|\le e^{|\im z|}+\frac1{\pi\nu}.
\]
Indeed the finite cosine integral has modulus at most $e^{|\im z|}$,
and the remaining integral is bounded by
$\pi^{-1}\int_0^\infty e^{-\nu t}\,dt$.
The normalised Poisson integral also applies at the removable origin.
Writing $l=2j$, $\nu=l+\lambda$ and
$P_\nu=(\rho/2)^\nu/\Gamma(\nu+1)$, a uniform field majorant is
\begin{equation}\label{gd:eq:Mfield}
 \begin{split}
 M_*=\sum_j\frac{|a_j|}{s_j}\biggl[
 P_\nu z_*^l e^{\rho z_*}
 +\min\biggl\{&
 (r_0(b-d_s))^{-\lambda}A_s^l
                 \left(e^{\rho d_s}+\frac1{\pi\nu}\right),\\[-2pt]
 &P_\nu H_s^l e^{\rho d_s}\biggr\}\biggr].
 \end{split}
\end{equation}
The first term covers $r\le r_0$ by the entire solid-harmonic
representation; the minimum covers $r\ge r_0$ by the two Bessel bounds.
Adding these nonnegative bounds enlarges their maximum.

Angular Cauchy estimates and Lemma~\ref{gd:lem:connection} give, with
$C=\kappa(S+R)/(S-R)$,
\begin{equation}\label{gd:eq:fixed-field-bounds}
 \begin{gathered}
 U_b=CM_*,\quad X_b=2U_b/\delta,\quad Y_b=RX_b,\\
 XX_b=6U_b/\delta^2,\quad YX_b=RXX_b,\quad YY_b=R^2XX_b .
 \end{gathered}
\end{equation}
These bound $U_*,(U_*)_{x_1},y(U_*)_y,(U_*)_{x_1x_1},
y(U_*)_{x_1y},y^2(U_*)_{yy}$ in $Y$ or $\mathcal Y_n$
according to parity.
For the first derivative the two independent complex coordinates give
the factor two; their second derivatives and twice the mixed derivative
give six by the polydisc Cauchy estimates.
The factors $y,y^2$ are multiplied in the raw Fourier algebra before
conversion of the resulting regular scalar. Thus $y$ itself is not
treated as an analytic scalar on $B^n$.
For these fixed-field estimates use $S=(102/100)R$ and
$\delta=(1000\rho)^{-1}$.
The boundary moment estimates use their own larger strip radius $S$.

\subsection{Boundary coefficients and their entire tails}\label{gd:sec:boundary}
Let $K$ and $J_b<K$ be the boundary sample and retained coefficient
parameters. Reflection of the values at
$\theta_k=(k+1/2)\pi/(2K)$, $0\le k<K$, gives $2K$
equispaced samples in $\phi=2\theta$.
For $D$ and $N$ the complex strip majorants are
$M_D=M_*+1$, $M_N=2SM_*/\delta$.
The corresponding continuum cosine coefficient of frequency $2j$
has modulus at most $2M_{D,N}S^{-2j}$.
The exact discrete Fourier transform, with its half-grid phase,
differs at each retained index $0\le j\le J_b$ by at most
\begin{equation}\label{gd:eq:alias}
 e_Q=\frac{4M_{D,N}S^{-2(2K-J_b)}}{1-S^{-4K}} .
\end{equation}
This is the alias sum at indices $j+2Kq$, as in
\eqref{r2:eq:alias}; the constant mode is included with the enlarged
same bound. Above $J_b$ use the full continuum coefficient bound.

Define the actual Gegenbauer moments
\begin{equation}\label{gd:eq:boundary-moments}
 D^{[k]}=\sum_lh_l(l+1)^k|D_l|,\qquad
 N^{[k]}=\sum_lh_l(l+1)^k|N_l|,\qquad 0\le k\le3.
\end{equation}
The connection lowers degree and hence the moment factor. Below,
$D^{[k]},N^{[k]}$ denote the upper bounds obtained by multiplying
the retained cosine moments
$\sum_{j=0}^{J_b}R^{2j}(2j+1)^k(|d_j|+e_Q)$ and the full tail
of~\eqref{gd:eq:tail-sums} by $\kappa$, where $d_j$ denotes the
corresponding cosine coefficient of $D$ or $N$. All later constants
are monotone in these bounds. For $q_{\rm tail}=(R/S)^2$, expand
$(2(J_b+1+t)+1)^k$ and use
\begin{equation}\label{gd:eq:tail-sums}
 \begin{aligned}
 \sum_{t\ge0}q_{\rm tail}^t&=\frac1{1-q_{\rm tail}},&
 \sum_{t\ge0}tq_{\rm tail}^t&=\frac{q_{\rm tail}}{(1-q_{\rm tail})^2},\\
 \sum_{t\ge0}t^2q_{\rm tail}^t&=\frac{q_{\rm tail}(1+q_{\rm tail})}{(1-q_{\rm tail})^3},&
 \sum_{t\ge0}t^3q_{\rm tail}^t&=\frac{q_{\rm tail}(1+4q_{\rm tail}+q_{\rm tail}^2)}{(1-q_{\rm tail})^4}.
 \end{aligned}
\end{equation}
The geometric coefficient decay also shows
that the lift in~\eqref{gd:eq:clamp} converges with all its derivatives
on a neighbourhood of the closed ball.

Direct differentiation of~\eqref{gd:eq:lifts} gives
\[
 \Delta_n h_l^D=-l(2l+n)H_l,\qquad
 \Delta_n h_l^N=(2l+n)H_l .
\]
Consequently set
\begin{equation}\label{gd:eq:lift-moments}
 \Lambda_k=2^k(D^{[k+1]}+N^{[k]})\ (0\le k\le2),\qquad
 \Lambda^\Delta_k=(n+2)(D^{[k+2]}+N^{[k+1]})\ (0\le k\le1).
\end{equation}
These bound the degree moments $\mathfrak m_k(\ell)$ and
$\mathfrak m_k(\Delta_n\ell)$.
For the first use that a lift monomial has degree at most
$l+2\le2(l+1)$, with total absolute coefficient at most $l+1$
for $h_l^D$ and one for $h_l^N$.
For the second use $2l+n\le(n+2)(l+1)$.
The radial factors themselves also give
\[
 \nrm\ell_Y\le D^{[0]}+N^{[0]}/2:
\]
the Dirichlet lift increases to one on $[0,1]$, and the Neumann
lift has modulus at most $1/2$.

\subsection{Residual and centre derivatives}\label{gd:sec:residual}
Let $Q_*,X_*,Y_*,P_*$ bound $q,q_{x_1},q_y/y,|p_0|^2$
in $\mathcal Y_n$, and let $Q^{[1]},X^{[1]},Y^{[1]},P^{[1]}$ bound their
total-degree coefficient moments.
Here $X_*,Y_*$ are the bounds of Section~\ref{gd:sec:unconjugation},
and $Q_*=b+q_b$. They arise from the finite shape identities~\eqref{gd:eq:potential-arrays}
and the full product
\[
 |p_0|^2=\sum_{i,j}(2i+1)(2j+1)c_{2i+1}c_{2j+1}
                         r^{2(i+j)}\cos(2(i-j)\theta).
\]
Retained connection coefficients can be converted directly;
Lemma~\ref{gd:lem:connection} bounds every higher angular degree.
Products add degree moments with nonnegative majorants.

Since $BU_*=0$, the actual residual is $E=F(x_0)=-B\ell$.
The product and derivative bounds give
\begin{equation}\label{gd:eq:E0}
 E_0=\frac{Q_*\Lambda^\Delta_0+
       a(X_*C_x^{\rm pol}+Y_*C_y^{\rm pol})\Lambda_1+\mu Q_*P_*\Lambda_0}{b}
       \ge\nrm E_Y
\end{equation}
and the coefficient degree moment bound
\begin{equation}\label{gd:eq:Ed}
 \begin{split}
 E_d=\frac1b\bigl[&
 Q^{[1]}\Lambda^\Delta_0+Q_*\Lambda^\Delta_1\\
 &+a\{(X^{[1]}C_x^{\rm pol}+Y^{[1]}C_y^{\rm pol})\Lambda_1
       +(X_*C_x^{\rm pol}+Y_*C_y^{\rm pol})\Lambda_2\}\\
 &+\mu\{(Q^{[1]}P_*+Q_*P^{[1]})\Lambda_0+Q_*P_*\Lambda_1\}\bigr].
 \end{split}
\end{equation}
Thus
\[
 \nrm{E_{x_1}}_{\mathcal Y_n}\le C_x^{\rm pol}E_d,\quad
 \nrm{yE_y}_Y\le C_y^{\rm pol}E_d,\quad
 \nrm{\partial_\theta E|_\partial}_{W_0}\le E_d .
\]

The inverse below uses derivatives of the fixed centre only.
Define $A^q=q_{x_1}/q$, $B^q=(q_y/y)/q$, $C^q=\mu|p_0|^2$.
With $\mathfrak I$ of~\eqref{gd:eq:source-multipliers}, upper bounds for these
coefficients and their derivatives are
\begin{equation}\label{gd:eq:coefficient-derivatives}
 \begin{aligned}
 A_*^q&=X_*\mathfrak I,& B_*^q&=Y_*\mathfrak I,& C_*^q&=\mu P_*,\\
 A_x^q&=C_x^{\rm pol}X^{[1]}\mathfrak I+X_*^2\mathfrak I^2,&
 A_y^q&=C_y^{\rm pol}X^{[1]}\mathfrak I+X_*Y_2Y_*\mathfrak I^2,\\
 B_x^q&=C_x^{\rm pol}Y^{[1]}\mathfrak I+Y_*X_*\mathfrak I^2,&
 B_y^q&=C_y^{\rm pol}Y^{[1]}\mathfrak I+Y_*^2Y_2\mathfrak I^2,\\
 C_x^q&=\mu C_x^{\rm pol}P^{[1]},&
 C_y^q&=\mu C_y^{\rm pol}P^{[1]} .
 \end{aligned}
\end{equation}
Here subscripts $x$ and $y$ indicate $\partial_{x_1}$ and
$y\partial_y$ respectively.
Use $\Delta_nU_*=-a(A^q(U_*)_{x_1}+B^qy(U_*)_y)-C^qU_*$
and~\eqref{gd:eq:fixed-field-bounds} to define
\begin{equation}\label{gd:eq:gstar}
 \begin{split}
 \Gamma_b={}&a(A_*^qX_b+B_*^qY_b)+C_*^qU_b,\\
 \Gamma_x={}&a(A_x^qX_b+A_*^qXX_b+B_x^qY_b+B_*^qYX_b)
                      +C_x^qU_b+C_*^qX_b,\\
 \Gamma_y={}&a(A_y^qX_b+A_*^qYX_b+B_y^qY_b+B_*^q(Y_b+YY_b))
                      +C_y^qU_b+C_*^qY_b .
 \end{split}
\end{equation}
The term $Y_b+YY_b$ is the derivative of $y(U_*)_y$ by
$y\partial_y$. Finally the bounds for the clamped centre are
\begin{equation}\label{gd:eq:centre-bounds}
 \begin{gathered}
 U_c=U_b+D^{[0]}+N^{[0]}/2,\quad
 U_x=X_b+C_x^{\rm pol}\Lambda_1,\quad U_y=Y_b+C_y^{\rm pol}\Lambda_1,\\
 \Gamma_0=\Gamma_b+\Lambda^\Delta_0,\quad \Gamma_{0x}=\Gamma_x+C_x^{\rm pol}\Lambda^\Delta_1,\quad
 \Gamma_{0y}=\Gamma_y+C_y^{\rm pol}\Lambda^\Delta_1 .
 \end{gathered}
\end{equation}
They bound $U_0,(U_0)_{x_1},y(U_0)_y,g_0,
(g_0)_{x_1},y(g_0)_y$ in their respective parity spaces.

\section{The exact inverse of the corrected linearisation}
\label{gd:sec:coupled}

\subsection{The material identity, including the axis term}
\label{gd:sec:material}
All coefficients here are evaluated at $x_0$.
For $\delta p\in W$ define the odd real-coefficient holomorphic functions
\[
 \eta(w)=\int_0^w\delta p(z)\,dz,\qquad v=\eta/p_0 .
\]
The associated full-ball vector field is
$(\re v,(\im v/y)x')$, regular on the axis.
Its action on $U_0$ is $T=\re v\,(U_0)_{x_1}+\im v\,(U_0)_y$.

To derive the material identity, let $\Theta_t$ vary the shape,
$m_t=q_t|p_t|^2/b$, and keep the function
$u_0=U_0\circ\Theta_0^{-1}$ fixed.
The centre is analytic on a neighbourhood of the closed ball, so this
function is defined on a neighbourhood as well.
Equivalently the following differentiation is a local differential
identity. Differentiating
$B_t(u_0\circ\Theta_t)=m_t((\Delta+\mu)u_0)\circ\Theta_t$ gives
\begin{equation}\label{gd:eq:material-general}
 BT+(\delta B)U_0
 =v\cdot\nabla E+(\delta\log m-v\cdot\nabla\log m)E .
\end{equation}
For example the last transport term before expansion is
$m\,v\cdot\nabla(E/m)$, which explains the minus sign.

The cancellation for $|p|^2$ uses $\eta'=p'v+pv'$:
\[
 \delta\log|p|^2-v\cdot\nabla\log|p|^2=2\re v'.
\]
For $q$ one has, with $Y_\psi=yq=\im\psi$,
$v\cdot\nabla Y_\psi=\im(pv)=\im\eta$.
Hence $v\cdot\nabla q=\im\eta/y-q\im v/y$ and
\[
 \delta\log q-v\cdot\nabla\log q=\im v/y .
\]
Define
\begin{equation}\label{gd:eq:correction}
 \Ecal(\delta p)=\re v\,E_{x_1}+\im v\,E_y
                         +(2\re v'+\im v/y)E,\qquad
 \widetilde D=DF(x_0)-\Ecal .
\end{equation}
The coefficient of $\im v/y$ is one, independent of $n$, because
the multiplier in~\eqref{gd:eq:physical-multiplier} contains one factor
$q$. The dimension enters the operator $B$ through $a=n-2$.
For the divergence multiplier $|p|^2q^a$, the same calculation gives
coefficient $a=n-2$ instead.
Equations~\eqref{gd:eq:F} and~\eqref{gd:eq:material-general} now imply
\begin{equation}\label{gd:eq:corrected-factor}
 \widetilde D(\delta g,\delta p)=B(K_n\delta g-T).
\end{equation}
The material mechanism is the Hadamard shape derivative used in
Proposition~\ref{r2:prop:coupled}; see, for example,~\cite{SZ92,HP18}.

\subsection{Boundary reconstruction and both inverse identities}
\label{gd:sec:inverse-identities}
Exact clamping gives $\nabla U_0=0$ on the sphere.
Differentiating this identity in each tangential direction shows that
the Hessian annihilates every tangent vector. By symmetry of the
Hessian,
\[
 D^2U_0=(\Delta_nU_0)\mathbf n\otimes\mathbf n\quad\hbox{on }\partial B^n .
\]
Thus, with $w=e^{i\theta}$ on the meridian circle,
\begin{equation}\label{gd:eq:boundary-material}
  \begin{gathered}
 T=0,\qquad \partial_rT=-d_\partial\re(v\bar w),\\
 d_\partial=-\Delta_nU_0|_\partial=\mu|p_0|^2-bE/q .
 \end{gathered}
\end{equation}
The term from differentiating $v$ multiplies $\nabla U_0$ and vanishes.
At the poles these identities are full-ball identities, hence also
hold by continuity.

Assume that $B_D^{-1}$ exists as proved above, $p_0^{-1}$ belongs
to the holomorphic weighted algebra, and $d_\partial^{-1}\in W_1$.
For $f\in Y$ solve
\begin{equation}\label{gd:eq:Dirichlet-recipe}
 B\Phi=f,\quad \Phi|_\partial=0,\qquad N_\partial=\partial_r\Phi,\quad H_\partial=N_\partial/d_\partial.
\end{equation}
Write its cosine coefficients as
$H_\partial=H_{\partial,0}+\sum_{j>0}H_{\partial,2j}\cos(2j\theta)$ and set
\begin{equation}\label{gd:eq:realpart-recipe}
 \mathfrak a(w)=H_{\partial,0}+\sum_{j>0}H_{\partial,2j}w^{2j},\qquad v=w\mathfrak a.
\end{equation}
Then $\re(v\bar w)=H_\partial$ on the circle and
$\nrm{v'}_W=\nrm{\mathfrak a}_{W_1}=\nrm{H_\partial}_{W_1}$.
The holomorphic real-part problem is the one of
\eqref{r2:eq:realpart}, with cosine coefficients in place of paired
complex Fourier coefficients. Its real-coefficient normalisation
excludes the imaginary constant; the real constant $H_{\partial,0}$ determines
the dilation variation.

\begin{proposition}[Two-sided inverse]\label{gd:prop:coupled}
Under these hypotheses the corrected derivative
$\widetilde D:X\to Y$ has a bounded two-sided inverse $A:Y\to X$.
It is given by
\begin{equation}\label{gd:eq:coupled-recipe}
 \delta\psi=p_0v,\quad \delta p=(p_0v)',\quad
 \delta U=\Phi+T,\quad \delta g=\Delta_n\delta U,\qquad
 Af=(\delta g,\delta p).
\end{equation}
\end{proposition}
\begin{proof}
Equation~\eqref{gd:eq:boundary-material} makes both traces of
$\delta U$ zero, since $\partial_r\delta U=N_\partial-d_\partial H_\partial=0$.
The Laplacian estimate in the next subsection proves
$\delta g\in Y$; the moment formula therefore gives
$\delta g\in\Gc$ and Dirichlet uniqueness gives
$K_n\delta g=\delta U$. The shape variation is even holomorphic
with real coefficients.
Then~\eqref{gd:eq:corrected-factor} gives $\widetilde D Af=f$.

For the other identity, take any $(\delta g,\delta p)\in X$,
form its $\eta,v,T$, and put $\Phi=K_n\delta g-T$.
It has zero Dirichlet trace and satisfies
$B\Phi=\widetilde D(\delta g,\delta p)$.
Dirichlet uniqueness recovers precisely this $\Phi$ when the recipe is
applied to that right side.
Its normal trace is $N_\partial=d_\partial\,\re(v\bar w)$.
The function $\mathfrak a=v/w=\eta/(wp_0)$ has a removable singularity at zero.
It belongs to $W_1$, since
$\nrm{\eta/w}_{W_1}=\nrm{\delta p}_W$ and
$p_0^{-1}\in W_1$ by differentiation of its reciprocal.
The unique real-coefficient real-part solution therefore recovers
$\mathfrak a$ and then $v,\delta p,\delta U,\delta g$.
Hence $A\widetilde D=I_X$ as well as $\widetilde DA=I_Y$.
The bounds below prove boundedness on the entire spaces.
\end{proof}

The sector is essential to the Dirichlet inverse.
For a Schiffer solution, physical translation derivatives solve the
homogeneous Dirichlet equation. The axial derivative is odd in the
first coordinate and the transverse derivatives are not invariant
under $O(n-1)$, so they lie outside this sector.

\subsection{Cancellation of higher shape derivatives}
\label{gd:sec:no-loss}
Cauchy--Riemann cancellation gives
\begin{equation}\label{gd:eq:DeltaT}
 \begin{split}
 \Delta_nT={}&v\cdot\nabla g_0+2\re v'\,g_0
       -a\frac{\im v'}y(U_0)_{x_1}\\
       &+a\frac{\im v/y-\re v'}{y^2}\,y(U_0)_y .
 \end{split}
\end{equation}
To see the axis terms directly, write $A_v=\re v$, $C_v=\im v$.
The planar Laplacian gives
$\Delta_2(v\cdot\nabla U_0)
=v\cdot\nabla\Delta_2U_0+2\re v'\Delta_2U_0$.
For the remaining operator $ay^{-1}\partial_y$, subtracting its
transport derivative and $2\re v'$ times that operator leaves
\[
 a\frac{(A_v)_y}{y}(U_0)_{x_1}
 +a\left(\frac{(C_v)_y-2\re v'}y+\frac{C_v}{y^2}\right)(U_0)_y .
\]
Use $(A_v)_y=-\im v'$ and $(C_v)_y=\re v'$ to obtain
\eqref{gd:eq:DeltaT}. Thus no second derivative of the unknown
holomorphic variation appears.

For the odd monomial $v=w^j$, $j\ge3$, the regular quotient is
\[
 \frac{\im v/y-\re v'}{y^2}
                  =2r^{j-3}C_{j-3}^{(2)}(\xi);
\]
for $j=1$, both $\im v'/y$ and
$(\im v/y-\re v')/y^2$ vanish. The positive cosine expansions of $U_m$
and $C_m^{(2)}$, summed with weight $R$, give
\begin{equation}\label{gd:eq:axis-multipliers}
 \begin{gathered}
 c_a=\frac2{R(1-R^{-2})},\qquad
 c_v=\frac4{R^2(1-R^{-2})^2},\\
 \nrm{\im v'/y}_{\rm raw}\le c_a\nrm{v'}_W,\qquad
 \nrm{(\im v/y-\re v')/y^2}_{\rm raw}\le c_v\nrm{v'}_W,\\
 \nrm{\im v/y}_{\rm raw}\le\nrm{v'}_W,\qquad
 \nrm{\re v}_{\rm raw}\le R\nrm{v'}_W .
 \end{gathered}
\end{equation}
For the first inequality use the sum of the coefficients of
$U_m$ with decreasing powers $R^{-2k}$, bounded by
$2/(1-R^{-2})$ after the leading power.
For the second use the Gegenbauer generating coefficients
$(2)_k/k!=k+1$; their weighted sum is $(1-R^{-2})^{-2}$.
The leading powers and division by the derivative coefficient $j$
give the stated enlarged $c_a,c_v$.
For $\im v/y$ the unweighted sum of the coefficients of $U_{j-1}$
is $j$, and every frequency is at most $j-1$.
The last inequality follows coefficientwise from
$R^j/j\le R R^{j-1}$.
Lemma~\ref{gd:lem:connection} converts each regular scalar once.
Consequently
\begin{equation}\label{gd:eq:MT}
 M_T=\kappa\left[R\Gamma_{0x}+\Gamma_{0y}+2\Gamma_0+a c_aU_x+a c_vU_y\right]
\end{equation}
bounds $\nrm{\Delta_nT}_Y$ per unit $\nrm{v'}_W$.
This proves the field membership required in
Proposition~\ref{gd:prop:coupled} without changing the analytic radius.

\subsection{Reciprocals, inverse norm and left Newton defect}
\label{gd:sec:inverse-bounds}
Define the finite shape quantities
\[
 P=\nrm{p_0}_W,\qquad
 P'=\sum_{m>0}m|(p_0)_m|R^{m-1},\qquad q_0=2b-P>0 .
\]
The holomorphic reciprocal of $p_0$ and the raw scalar reciprocal
of $q$ have norm at most $1/q_0$, by Neumann series about $b$.
Set
\begin{equation}\label{gd:eq:boundary-reciprocal}
 \begin{gathered}
 d_*=\mu(2b^2-P^2)-bE_0/q_0>0,\\
 D_\theta=2\mu PRP'+b(E_d/q_0+E_0RP'/q_0^2),\\
 I_0=d_*^{-1},\qquad I_1=I_0+I_0^2D_\theta .
 \end{gathered}
\end{equation}
Indeed
$\nrm{d_\partial-\mu b^2}_{W_0}\le\mu(P^2-b^2)+bE_0/q_0<\mu b^2$,
so $\nrm{d_\partial^{-1}}_{W_0}\le I_0$.
Since $\nrm{\partial_\theta q}_{W_0}\le RP'$,
differentiating $d_\partial$ gives the bound $D_\theta$.
Differentiating $d_\partial d_\partial^{-1}=1$ then gives
$\nrm{d_\partial^{-1}}_{W_1}\le I_1$.
This is the weaker-norm reciprocal argument in
Proposition~\ref{r2:prop:coupled}; smallness in $W_1$ is unnecessary.

By~\eqref{gd:eq:Dirichlet-recipe} and~\eqref{gd:eq:JW},
$\nrm{\mathfrak a}_{W_1}\le I_1N_W\nrm f_Y$.
Equation~\eqref{gd:eq:coupled-recipe} therefore gives
\begin{equation}\label{gd:eq:A}
 B_g=J_W+M_TI_1N_W,\qquad
 B_p=(P+RP')I_1N_W,\qquad A_*=B_g+\sigma B_p,\qquad
 \nrm A_{Y\to X}\le A_* .
\end{equation}
Here $\delta p=(wp_0\mathfrak a)'$ has the norm of $p_0\mathfrak a$ in $W_1$.
For an arbitrary input variation, $v=\eta/p_0$ gives
\begin{equation}\label{gd:eq:Jv}
 \nrm{v'}_W\le J_v\nrm{\delta p}_W,\qquad
 J_v=q_0^{-1}+RP'q_0^{-2},
\end{equation}
using $\nrm\eta_W\le R\nrm{\delta p}_W$.
The material correction has the bound
\begin{equation}\label{gd:eq:epsilon}
 \nrm{\Ecal}_{X\to Y}\le\varepsilon_{\mathcal E}
 =\frac{\kappa}{\sigma}
       \left[(RC_x^{\rm pol}+C_y^{\rm pol})E_d+3E_0\right]J_v .
\end{equation}
The factor three is $2+1$ from the last two terms of
\eqref{gd:eq:correction}. The transport terms use
$\im v\,E_y=(\im v/y)yE_y$.
Since $A\widetilde D=I_X$, the left Newton defect is
\begin{equation}\label{gd:eq:Newton}
 \nrm{AF(x_0)}_X\le Y_0=A_*E_0,\qquad
 \nrm{I_X-ADF(x_0)}_{X\to X}\le Z=A_*\varepsilon_{\mathcal E} .
\end{equation}
The field right sides need only be continuous.
Their Dirichlet potentials have $C^1$ traces by the preceding
Poisson and elliptic bounds, whereas $T$ is analytic near the closed
ball. Thus the identities apply to every element of $X$ and $Y$.

\begin{table}[htbp]
\centering\small
\def\gdRow#1#2#3#4#5#6#7#8{$#1$ & #5\\}
\begin{tabular}{@{}crrrrr@{}}
\toprule
$n$ & $A_*$ & $Y_0$ & $Z$ & $4C_2Y_0$ & $\mathfrak c(r_*)$\\
\midrule
\gdData
\bottomrule
\end{tabular}
\caption{Corrected inverse and contraction bounds.
$\mathfrak c$ is defined in~\eqref{gd:eq:radii}.
The quantity $4C_2Y_0$ is defined by~\eqref{gd:eq:Newton}
and~\eqref{gd:eq:radii}.}
\label{gd:tab:Newton}
\end{table}

\section{Quartic bounds and the exact zero}\label{gd:sec:existence}

For a variation $(h,k)\in X$ write $u_h=K_nh$ and
$q_k=\im\int_0^wk(z)\,dz/y$.
Lemma~\ref{gd:lem:connection} and the axis multiplier estimates give
\begin{equation}\label{gd:eq:shape-increments}
 \begin{gathered}
 \nrm{q_k}_Y\le\kappa\nrm k_W,\quad
 \nrm{(q_k)_{x_1}}_{\mathcal Y_n}\le\kappa c_a\nrm k_W,\quad
 \nrm{(q_k)_y/y}_Y\le\kappa c_v\nrm k_W,\\
 \nrm{p_0\bar k+\bar p_0k}_Y\le2\kappa P\nrm k_W,\qquad
 \nrm{|k|^2}_Y\le\kappa\nrm k_W^2 .
 \end{gathered}
\end{equation}
The last two estimates multiply in the raw angular algebra before
conversion. These bounded linear operators, together with
Lemma~\ref{gd:lem:Poisson}, make $F:X\to Y$ a continuous quartic
polynomial on this fixed Banach space.

To bound its remainder, put $P_k=p_0\bar k+\bar p_0k$.
The terms of degree two, after removal of $E+DF(x_0)(h,k)$,
are
\[
 \begin{split}
 b\mathcal N_2={}&q_kh+
 a\{(q_k)_{x_1}(u_h)_{x_1}+((q_k)_y/y)y(u_h)_y\}\\
 &+\mu\{(qP_k+q_k|p_0|^2)u_h+(q|k|^2+q_kP_k)U_0\}.
 \end{split}
\]
The terms of degrees three and four are
\[
 b\mathcal N_3=\mu\{(q|k|^2+q_kP_k)u_h+q_k|k|^2U_0\},
 \qquad b\mathcal N_4=\mu q_k|k|^2u_h .
\]
Projection to the field has norm at most one and projection to the
shape has norm at most $1/\sigma$ in $X$.
The algebra, Poisson, and centre estimates therefore give
\begin{equation}\label{gd:eq:c234}
 \begin{split}
 \mathsf k_2={}&\frac1b\biggl[
 \frac{\kappa}{\sigma}
 +\frac{a\kappa(c_aC_x^K+c_vC_y^K)}{\sigma}
 +\frac{\mu(2\kappa Q_*P+\kappa P_*)C_K}{\sigma}\\
 &\hspace{24mm}
       +\frac{\mu(\kappa Q_*+2\kappa^2P)U_c}{\sigma^2}\biggr],\\
 \mathsf k_3={}&\frac\mu b\left[
       \frac{(\kappa Q_*+2\kappa^2P)C_K}{\sigma^2}
       +\frac{\kappa^2U_c}{\sigma^3}\right],\\
 \mathsf k_4={}&\frac{\mu\kappa^2C_K}{b\sigma^3}.
 \end{split}
\end{equation}
In particular $\nrm{\mathcal N_j(h,k)}_Y\le
\mathsf k_j\nrm{(h,k)}_X^j$.
Each displayed product also defines a multilinear form with this same
bound on the product of input norms. Symmetrisation averages over the
positions of the field and shape inputs, and introduces no larger
constant. Differentiating the diagonal form consequently bounds
its derivative by $j\mathsf k_js^{j-1}$ on $\nrm{(h,k)}_X\le s$,
as in the polarisation argument of Section~\ref{r2:sec:existence}.

Set $C_j=A_*\mathsf k_j$ and define
\begin{equation}\label{gd:eq:radii}
 \begin{split}
 \mathfrak p(s)&=Y_0+(Z-1)s+C_2s^2+C_3s^3+C_4s^4,\\
 \mathfrak c(s)&=Z+2C_2s+3C_3s^2+4C_4s^3 .
 \end{split}
\end{equation}

\begin{proposition}\label{gd:prop:zero}
If $\mathfrak p(r_*)<0$ and $\mathfrak c(r_*)<1$, then
$F$ has exactly one zero in
$\overline B_X(x_0,r_*)$.
\end{proposition}
\begin{proof}
The proof of Lemma~\ref{r2:lem:contraction} applies to the additional
quartic term. For $\mathcal T(x)=x-AF(x)$, the estimates give
\[
 \nrm{\mathcal T(x_0+h)-x_0}_X
 \le Y_0+Z\nrm h_X+\sum_{j=2}^4C_j\nrm h_X^j .
\]
The first inequality makes the closed ball invariant.
Its derivative is bounded there by $\mathfrak c(r_*)<1$,
so completeness gives a unique fixed point.
The two-sided inverse $A$ is injective, hence $AF(x)=0$ implies
$F(x)=0$. Every zero is a fixed point.
\end{proof}

\begin{table}[htbp]
\centering\small
\def\gdRow#1#2#3#4#5#6#7#8{$#1$ & #6\\}
\begin{tabular}{@{}crrrrr@{}}
\toprule
$n$ & $E_0$ & $C_2$ & $C_3$ & $C_4$ & $\mathfrak p(r_*)$\\
\midrule
\gdData
\bottomrule
\end{tabular}
\caption{The full residual, homogeneous remainder constants and
negative radii polynomial, at the exact radius of
Table~\ref{gd:tab:inputs}.}\label{gd:tab:nonlinear}
\end{table}

Tables~\ref{gd:tab:Newton} and~\ref{gd:tab:nonlinear} give
$\mathfrak p(r_*)<0$ and $\mathfrak c(r_*)<10^{-5}$,
so Proposition~\ref{gd:prop:zero} applies.

\section{Geometry on the whole existence ball}\label{gd:sec:geometry}

Write $\psi_0=\sum_{j\ {\rm odd}}c_jw^j$ and let
$\psi=\psi_0+\delta\psi$ be any shape in the existence ball.
Then $\nrm{\delta p}_W\le s_*=r_*/\sigma$.
Define the unweighted centre sums
\begin{equation}\label{gd:eq:geometry-sums}
 \begin{gathered}
 S_1=\sum_{j>1}j|c_j|,\qquad
 S_2=\sum_{j>1}j(j-1)|c_j|,\\
 S_T=\sum_{j>1}j^2|c_j|=S_1+S_2,\qquad
 H_R=\frac{R}{(R-1)^2},\\
 \beta_0=2b-P-s_*,\quad L_*=b-S_1-s_*,\quad
 \chi_*=b-S_1-S_2-(1+H_R)s_*,\\
 \vartheta_3=|c_3|-\frac{s_*}{3R^2}.
 \end{gathered}
\end{equation}
Table~\ref{gd:tab:geometry} gives strict positive lower bounds.
Since $P-b\ge S_1$, $\beta_0>0$ also implies $L_*>0$.
The two sums $S_T$ and $S_1+S_2$ are identical, so the meridian and
transverse estimates below use one common numerator $\chi_*$.

\begin{table}[htbp]
\centering
\def\gdRow#1#2#3#4#5#6#7#8{$#1$ & #7\\}
\begin{tabular}{@{}crrr@{}}
\toprule
$n$ & $\beta_0$ & $\chi_*$ & $\vartheta_3$\\
\midrule
\gdData
\bottomrule
\end{tabular}
\caption{Whole-ball geometry: analytic univalence, the common
curvature numerator and the absolute $w^3$ coefficient. The numerator
$\chi_*$ is the common numerator of
\eqref{gd:eq:meridian-curvature} and~\eqref{gd:eq:transverse-numerator}.}
\label{gd:tab:geometry}
\end{table}

\subsection{Univalence and analytic lift}\label{gd:sec:univalence}
On $|w|\le R$,
\[
 \re p(w)\ge b-(P-b)-s_*=\beta_0>0.
\]
On the closed unit disc the sharper coefficient estimate is
\begin{equation}\label{gd:eq:unit-disc-tangent}
 \re\psi'(w)\ge b-S_1-s_*=L_*>0\qquad(|w|\le1).
\end{equation}
The line-segment proof in Section~\ref{r2:sec:existence} applies verbatim:
for distinct points $w_1,w_2$ of the disc $|w|<R$,
\[
 \frac{\psi(w_2)-\psi(w_1)}{w_2-w_1}
 =\int_0^1p(w_1+t(w_2-w_1))\,dt
\]
has positive real part. Thus $\psi$ is univalent on $|w|<R$.
Also $q=\int_0^1\re p(x_1+isy)\,ds>0$ there, so $\psi$ preserves
the upper half-disc. The real-coefficient identities express
$\re\psi$ and $q$ as convergent power series in $x_1,y^2$.
Replacing $y^2$ by $|x'|^2$ makes $\Theta_\psi$ analytic across the axis.

Off the axis the meridian differential and the $n-2$ transverse
differentials give $\det D\Theta_\psi=|p|^2q^{n-2}>0$.
On the axis its differential is
$\diag(p(x_1),p(x_1)I_{n-1})$, also nonsingular.
If two lifted points have the same image, their meridian coordinates
have the same image under $\psi$; univalence gives equality of $x_1,y$,
and positivity of $q$ then gives equality of $x'$.
Thus the lift is an analytic diffeomorphism on a neighbourhood of
$\overline{B^n}$. Its restriction to the sphere gives a real-analytic
boundary diffeomorphic to $S^{n-1}$.
Oddness and reality give the required $O(n-1)\times\Z_2$ symmetry.

\subsection{Meridian and transverse curvatures}\label{gd:sec:curvatures}
For the perturbation,
\[
 \sup_{|w|\le1}|\delta p'|
 \le\sup_{m\ge1}mR^{-m}\nrm{\delta p}_W
 \le H_Rs_* .
\]
Therefore $|p|\ge L_*$ and $|wp'|\le S_2+H_Rs_*$ on the closed disc.
The curvature formula~\eqref{r2:eq:curvature-formula} gives
\begin{equation}\label{gd:eq:meridian-curvature}
 \re(1+wp'/p)\ge\frac{\chi_*}{L_*}>0,\qquad
 \kappa_{\rm mer}=\frac{\re(1+wp'/p)}{|p|}>0 .
\end{equation}
The simple meridian boundary consequently encloses a strictly convex
planar domain symmetric in both axes.

For the transverse curvatures and the graph property, write
$\delta\psi=\sum_{j\ {\rm odd}}\delta c_jw^j$ and parameterise the upper arc as
$\psi(e^{i\theta})=(X(\theta),Y_\psi(\theta))$, $0<\theta<\pi$.
Its outward normal is $(Y_\psi',-X')/|p|$ and speed is $|p|$.
Using $U_m(\cos\theta)=\sin((m+1)\theta)/\sin\theta$ gives
\begin{equation}\label{gd:eq:transverse-numerator}
 -\frac{X'}{\sin\theta}
 =\sum_{j\ {\rm odd}}j(c_j+\delta c_j)U_{j-1}(\cos\theta)
 \ge b-S_T-(1+H_R)s_*=\chi_* .
\end{equation}
Indeed $|U_{j-1}|\le j$, and
\[
 \sum_{j\ {\rm odd}}j^2|\delta c_j|
 =\sum_{m\ {\rm even}}(m+1)|\delta p_m|
 \le(1+H_R)\nrm{\delta p}_W.
\]
Since $Y_\psi/\sin\theta=q>0$, each of the $n-2$ transverse principal
curvatures of the unscaled body is
\begin{equation}\label{gd:eq:transverse-curvature}
 \kappa_\perp=\frac{-X'}{|p|Y_\psi}
 =\frac{-X'/\sin\theta}{|p|(Y_\psi/\sin\theta)}
 \ge\frac{\chi_*}{(b+S_1+s_*)^2}>0.
\end{equation}
Both denominators have the displayed upper bound by the coefficient
and integral estimates for $p,q$.
The quotients extend to the poles analytically.
At $\theta=0$, $Y_\psi'=p(1)>0$ and
$-X'/\sin\theta\to-X''(0)=p(1)+p'(1)$, so
$\kappa_\perp$ tends to
$(p(1)+p'(1))/p(1)^2$, the meridian curvature there.
The other pole has the same conclusion by symmetry.
Thus every principal curvature is positive on the whole boundary.
Scaling by $\rho$ divides all principal curvatures by $\rho$.

\subsection{Global convexity and the non-ball coefficient}
\label{gd:sec:global}
By~\eqref{gd:eq:transverse-numerator}, $X'(\theta)<0$ on the open
upper arc. It is therefore a graph $Y_\psi=f(X)$ between its two poles.
The curvature numerator of this oriented graph is
$X'Y_\psi''-Y_\psi'X''=(X')^3f''$; positive meridian curvature and $X'<0$
give $f''<0$ in the interior. Continuity extends its concavity to the
closed interval, where $f$ vanishes at the endpoints and is positive
inside. The lifted domain is exactly
\[
 \Theta_\psi(B^n)=\{(X_1,X_\perp):X_1\text{ between the poles},\
                 X_\perp\in\R^{n-1},\ |X_\perp|<f(X_1)\}.
\]
For two interior points $\mathbf X=(X_1,X_\perp)$ and
$\tilde{\mathbf X}=(\tilde X_1,\tilde X_\perp)$ and $0<t<1$,
the triangle inequality and concavity give
\[
 \begin{split}
 |(1-t)X_\perp+t\tilde X_\perp|
 &\le(1-t)|X_\perp|+t|\tilde X_\perp|\\
 &<(1-t)f(X_1)+tf(\tilde X_1)\\
 &\le f((1-t)X_1+t\tilde X_1).
 \end{split}
\]
Hence the body is convex globally.
Its smooth boundary has positive principal curvatures by
\eqref{gd:eq:meridian-curvature} and~\eqref{gd:eq:transverse-curvature},
so it contains no boundary line segment and is strictly convex.
This is the solid-of-revolution argument used in
Sections~\ref{r5:sec-zero} and~\ref{r7:sec-zero}.

Finally the coefficient estimate gives
$|c_3+\delta c_3|\ge\vartheta_3>0$.
A ball invariant under $O(n-1)$ and first-coordinate reflection has
centre zero. Its meridian section would be a centred disc.
The odd conformal map fixing zero with positive derivative there
would then be linear, by uniqueness of the normalised conformal map,
contradicting its nonzero $w^3$ coefficient.
This excludes balls throughout the shape projection of the existence
ball. The origin $\Theta_\psi(0)$ is an interior point of the convex
domain. Every tangent hyperplane supports the domain and therefore
does not contain the origin. Every ray from the origin meets the
boundary exactly once and transversally, so $\Omega$ is strictly
star-shaped with respect to the origin.

\section{The Schiffer field and the Pompeiu conclusion}
\label{gd:sec:conclusion}

\begin{proof}[Proof of Theorem~\ref{gd:thm:main}]
For each $n\in\gdDimensions$, the centre satisfies
$c=\rho b>20.69$, $q_*>0.9548$ and $\gamma>299.4$.
The window of Lemma~\ref{gd:lem:coverage} has the state count of
Table~\ref{gd:tab:inverse}. The hypotheses~\eqref{gd:eq:Schur-gates} hold with
$\epsilon_C<2.7\cdot10^{-10}$, $\epsilon_{\rm S}+\eta_{\rm S}<0.0242$
and $B_{\rm pol}\delta_V<1.1\cdot10^{-4}$, so
Proposition~\ref{gd:prop:Schur} gives $B_{L^2}$.
With the tabulated $L_0$, one has $\beta<0.4987$ and
$q_b<0.3412<b$. Thus~\eqref{gd:eq:BY} and
Proposition~\ref{gd:prop:original-inverse} give $B_Y,J_W,N_W$.
For both strip radii, $d_s<0.13b$ in~\eqref{gd:eq:strip-parameters};
also $q_0>0.9478$ and $d_*>387.0$ in~\eqref{gd:eq:boundary-reciprocal}.
Proposition~\ref{gd:prop:coupled} therefore applies with the bound
$A_*$ of Table~\ref{gd:tab:Newton}. The strict radii inequalities
in Tables~\ref{gd:tab:Newton} and~\ref{gd:tab:nonlinear} give the unique
zero $(g,p)$ in $\overline B_X(x_0,r_*)$ by
Proposition~\ref{gd:prop:zero}.
Put $U=1+K_ng$. Lemma~\ref{gd:lem:Poisson} gives $U\in W^{2,p}(B^n)$
for every finite $p$, with $U=1$ and $\partial_rU=0$ on the sphere.
Its tangential derivatives also vanish there, so $\nabla U=0$.
Since $U\in W^{2,p}$, $F(g,p)=0$ says $BU=0$ almost everywhere.
Multiplication by $bq^{a-1}$ and integration by parts give
\[
 \int_{B^n}q^a(\nabla U\cdot\nabla\varphi-\mu|p|^2U\varphi)\,dx=0
 \qquad(\varphi\in C_c^\infty(B^n)).
\]
The analytic diffeomorphism of Section~\ref{gd:sec:univalence} and
the change of variables of Lemma~\ref{r5:lem-lift}(b), with $n-2$
transverse directions, give
$u_0=U\circ\Theta_\psi^{-1}\in C^1(\overline{\Theta_\psi(B^n)})$,
a weak solution of $(\Delta+\mu)u_0=0$ with $u_0=1$ and
$\nabla u_0=0$ on the boundary. Set
\[
 \Omega=\rho\Theta_\psi(B^n),\qquad u(X)=u_0(X/\rho).
\]
Since $\mu=\rho^2$, $u\in C^1(\overline\Omega)$ solves
$\Delta u+u=0$ weakly with $u=1$, $\nabla u=0$ on the real-analytic
boundary. Lemma~\ref{r3:lem-regularity} gives a real-analytic extension
to a neighbourhood of $\overline\Omega$, solving~\eqref{gd:eq:schiffer}.
It cannot be constant, because its boundary value would make it the
constant one, which does not solve $(\Delta+1)u=0$.
Sections~\ref{gd:sec:univalence}--\ref{gd:sec:global} give all geometric
assertions and the symmetry.

The Green identity of Section~\ref{sec:pompeiu}
applies in every dimension. For a unit vector $\mathbf k$ put
$v(X)=e^{i\mathbf k\cdot X}$. Then
\[
 0=\int_\Omega(u\Delta v-v\Delta u)
   =\int_{\partial\Omega}\partial_{\mathbf n} v
   =-\int_\Omega e^{i\mathbf k\cdot X}\,dX .
\]
Thus $\widehat{\one_\Omega}$ vanishes on the unit sphere.
For a rigid motion $X\mapsto QX+A$,
\[
 \int_{Q\Omega+A}e^{iX_1}\,dX
   =e^{iA_1}\int_\Omega e^{i(Q^Te_1)\cdot X}\,dX=0 .
\]
Taking real parts gives zero integrals of the nonzero continuous
function $\cos X_1$ over every rigid motion of $\Omega$.
Hence $\Omega$ fails the Pompeiu property.
\end{proof}

\part{Consequences, computations and remarks}\label{part:further}

\section{The Pompeiu property}\label{sec:pompeiu}

\begin{proof}[Proof of Corollary~\ref{cor:pompeiu}]
Let $n$ satisfy $2\le n\le\topdim$ or $n\in\{20,21\}$, let $\Omega$ and $u$ be as in Theorem~\ref{thm:r3} or Theorem~\ref{thm:main} (proved in Sections~\ref{r2:sec:existence}, \ref{sec:r3}, \ref{sec:r5}, \ref{r7:sec}, \ref{sec:recon}, \ref{sec:r6}, \ref{lie:sec} and~\ref{gd:sec:conclusion}), let $\xi\in\R^n$ with $|\xi|=1$, and let $v(x)=e^{i\xi\cdot x}$, so that $\Delta v+v=0$. Since $u$ and $v$ are smooth up to the real-analytic boundary and solve the same equation, Green's identity and \eqref{eq:main} give
\[
0=\int_\Omega(u\Delta v-v\Delta u)\,dx=\int_{\partial\Omega}(u\,\partial_\nu v-v\,\partial_\nu u)\,dS=\int_{\partial\Omega}\partial_\nu v\,dS=\int_\Omega\Delta v\,dx=-\int_\Omega e^{i\xi\cdot x}\,dx .
\]
Thus $\widehat{\one_\Omega}$ vanishes on the unit sphere. If $\sigma(x)=Qx+a$ is a rigid motion, then
\[
\int_{\sigma(\Omega)}e^{ix_1}\,dx=e^{ia_1}\int_\Omega e^{i(Q^{\mathsf T}e_1)\cdot x}\,dx=0,
\]
and taking real parts, $\int_{\sigma(\Omega)}\cos x_1\,dx=0$. Since $\cos x_1$ is continuous and not identically zero, $\Omega$ fails the Pompeiu property. By Theorems~\ref{thm:r3} and~\ref{thm:main}, $\Omega$ is a bounded domain, not a ball, whose boundary is real analytic and diffeomorphic to $S^{n-1}$, so the Pompeiu conjecture fails in $\R^n$; the domain is convex for every $2\le n\le\topdim$ and for $n=20,21$, by Theorems~\ref{thm:r3} and~\ref{thm:main}.
\end{proof}

\section{Computer-assisted verification}\label{sec:cap}

This section describes, for each dimension, what the scripts check and how the computations are reproduced; Section~\ref{subsec:avail} describes the ancillary files. In all dimensions the inequalities are checked in the ball arithmetic of Arb \cite{Joh17}, through \texttt{python-flint} 0.9.0 \cite{flint}: an order relation between balls is true only if it holds for all points of the balls, every check raises an exception when it fails, and NaN or infinite enclosures are rejected. The analytic content of the proofs, namely all lemmas and all tail estimates, is contained in Sections~\ref{r2:sec:inverse}--\ref{sec:r5}, Section~\ref{r7:sec} and Parts~\ref{gd:part} and~\ref{br:part}; the scripts only evaluate the finitely many explicit expressions defined there.

\subsection{The plane}\label{r2:sec:computer}

\subsubsection{Inputs and arithmetic}
The ancillary file \texttt{verify\_planar\_convex\_v2.py} rebuilds the numerical hypotheses used in the planar proof sections. Its inputs are two finite sets of exact decimal strings: the eight-mode reference shape and the forty-mode centre, the latter consisting of the arrays \texttt{p}, \texttt{a} and \texttt{scales}. Every decimal string is interpreted as an exact rational and enclosed without a binary64 conversion. In particular the centre's \texttt{p} array contains the coefficients $c_j$ of $\psi_0$; differentiation to obtain $p_0$ is exact. The field divisors \texttt{scales} are positive exact decimals.

The single-file verifier embeds the stage sources, the reference and the centre in a compressed bundle. Hashes identify both the bundle and each embedded file. It embeds no computed Bessel roots, spectral matrices or boundary outputs. Stage functions construct floating-point proposals where useful, then validate them with enclosures. Opposite endpoint signs enclose Bessel roots; the spacing and sign-grid arguments certify coverage; enclosed Gauss rules and basis errors certify matrix entries; the congruence defects certify the inverse. The boundary computation encloses the values and adds the analytic aliases and infinite Fourier tail before forming the residual.

Special functions, quadrature data, matrix errors and the boundary traces are enclosed with FLINT/Arb through \texttt{python-flint}~\cite{flint,Joh17}. Root-sign evaluations increase through $256$, $1024$ and $4096$ bits when necessary; spectral Bessel values use $2048$ bits and boundary values use $4096$ bits. Binary64 matrix proposals are exact dyadics with separately bounded rounding errors. A second implementation in \texttt{mpmath.iv}, at $400$ bits, recombines the scalar chain from the spectral inverse bound and the eight boundary moment bounds. It independently reconstructs the lift and coupled constants, but does not independently certify the special functions, spectral matrices or boundary traces.

The acceptance checks use certain one-sided interval comparisons. Nonfinite values and comparisons that an interval cannot resolve are rejected. The final check recomputes the coupled constants in the driver process rather than accepting constants supplied by a stage-output file. It also compares the two spectral enumerations at the level of individual roots. A successful run writes the receipt and prints \texttt{NUMERICAL\_GATES\_PASS}, certifying the numerical hypotheses of the analytic lemmas. The receipt field \texttt{full\_proof=false} records that the analytic lemmas of this paper are not machine-checked. The proof consists of those lemmas together with the enclosed inequalities. Symbolic polynomial control examples test the material identity, while its general proof is Proposition~\ref{r2:prop:coupled}.

\subsubsection{Stages and recorded timings}
Table~\ref{r2:tab:timings} gives the stage timings in the recorded receipt. The leaf groups report sums of individual stage elapsed times, so they are not wall times for the parallel groups. The recorded run used five single-threaded workers and had wall time $242.21$ seconds. The number of workers affects only the scheduling of the stages: the two-worker command below reproduces the same numerical hypotheses with lower concurrency.
\begin{table}[htbp]
\centering
\begin{tabular}{lrr}
\toprule
Stage group & Stages & Recorded seconds\\
\midrule
Independent sign-grid coverage & 1 & 30.42\\
Spectral preparation & 1 & 11.95\\
Spectral basis leaves & 17 & 455.06\\
Spectral assembly and inverse & 1 & 13.68\\
Boundary trace leaves & 8 & 219.60\\
Boundary merge & 1 & 0.48\\
Coupled constants and radii & 1 & 0.28\\
Symbolic identities & 1 & 3.99\\
Second interval scalar chain & 1 & 0.33\\
\bottomrule
\end{tabular}
\caption{The $32$ stages of the recorded run. Each leaf is single-threaded. Group sums are displayed to two decimal places and may differ slightly from a sum of rounded individual times.}\label{r2:tab:timings}
\end{table}

The built-in failure controls reject $144$ deliberate perturbations, including nonfinite decimal inputs, altered coupled bounds, missing or duplicated spectral states, changed grid size and violated geometric thresholds. A nonzero stage exit or a timeout terminates the run without a receipt. The final artifact hash map is read back before the status line is printed. These controls exercise the numerical acceptance path; they do not replace the analytic estimates.

\subsubsection{Ancillary files and replay}
The directory \texttt{anc/r2/} contains the verifier \path{verify_planar_convex_v2.py}, the recorded receipt \path{VERIFICATION_RECEIPT.json}, the frozen centre \path{center_j40_frozen.json}, the eight-mode reference \path{reference.json} and the nine stage sources that the verifier embeds. It also contains the two files from which the single-file verifier is assembled, the list of software versions, the scalar robustness calculation \path{ROBUSTNESS_CHECK.json}, the script \path{exact_constants.py}, a description of the files, and the checksum file \texttt{SHA256SUMS}. The verifier needs only its embedded inputs to run; the separate centre and sources expose the exact input and implementation for inspection. The receipt records the enclosures in Table~\ref{r2:tab:constants}, the complete stage list, software versions, the rejected perturbations and hashes of replay artifacts. The script \path{exact_constants.py} evaluates in exact rational arithmetic the algebraic constants of the two inputs that are quoted in Subsection~\ref{r2:sec:spectrum} through Section~\ref{r2:sec:existence}. The directory \texttt{anc/r2/diagnostics/} contains the scripts and output logs of the non-interval computations quoted in Sections~\ref{r2:sec:intro} and~\ref{r2:sec:rigidity}.

The verifier's SHA-256 is
\begin{center}
\texttt{\footnotesize ed4f65a5d5a31ee86bfe6c8c742ea86ba76f4cdb888b3c8fb70b53b8334e7226}
\end{center}
The recorded environment was Python $3.10.19$, \texttt{python-flint} $0.9.0$, NumPy $1.26.4$, SciPy $1.15.2$, SymPy $1.14.0$ and mpmath $1.3.0$.

From the directory containing \texttt{anc/r2/}, with those packages available, choose a new or empty replay directory and an output receipt path that does not exist. A command with at most two single-threaded stage workers is
\begin{verbatim}
nice -n 19 env OMP_NUM_THREADS=1 OPENBLAS_NUM_THREADS=1 \
  MKL_NUM_THREADS=1 NUMEXPR_NUM_THREADS=1 \
  PYTHONDONTWRITEBYTECODE=1 PYTHONPATH= \
  python -O anc/r2/verify_planar_convex_v2.py \
  --workdir ./replay --jobs 2 \
  --output ./REPLAY_RECEIPT.json
\end{verbatim}
The driver refuses an occupied working directory and a pre-existing output receipt. The default timeout is $7200$ seconds for each stage and can be set with \texttt{--stage-timeout}. The files of \texttt{anc/r2/} can be checked by running \texttt{sha256sum -c SHA256SUMS} inside that directory.

The computational trust base consists of the Python runtime, the enclosure libraries and their arithmetic, and the correspondence between the verifier and the formulas of Sections~\ref{r2:sec:inverse}--\ref{r2:sec:existence}; the analytic arguments of those sections are not machine-checked. The frozen centre is a finite proposal with a nonzero residual. All existence and geometry statements concern the exact zero of $F$ in the certified infinite-dimensional ball.

\subsection{The convex planar Berenstein domain}\label{br:sec:verification}
The files in \texttt{anc/berenstein/} reproduce the numerical hypotheses of
Part~\ref{br:part}. The single-file verifier
\path{verify_berenstein_planar_convex_v3.py} embeds the exact-decimal centre
and ten stage sources. It contains no computed roots, spectral matrices,
boundary enclosures or existence receipts. The two coverage arguments
reconstruct the complete Robin window; enclosed radial and boundary
integration gives the form matrices and their complementary Gram estimates.
Boundary interval sums include the aliases and infinite tails of
Subsection~\ref{br:sec:residual}. The driver merges those sums and recomputes
the coupled, nonlinear and geometric quantities in its own process.

The parent process rechecks the quoted upper and lower bounds listed in the
receipt field \texttt{parent\_gates}, complete coverage, actual matrix and
basis dimensions, and the scalar-chain consistency before acceptance. An interval identity stage tests Cartesian translations at the
frozen centre and the residual-corrected material and boundary identities
on finite polynomial examples with nonzero residual. It also tests both
clamping traces and rejects wrong-sign and omitted-term variants.
The boundary coefficient $d=\re(wp_0'(w)/p_0(w))$ is independently
recovered from the exact centre by Taylor division in $z=w^m$.
The eleven Fourier coefficients of $d$ that enter the boundary form are
compared with the Taylor-division values, and all $90000$ entries of the
form with Laurent products of these values and the radial traces of the
basis functions; the assembled operator is checked against that form.

The spectral Bessel values use \brSpectralBits\ bits and the boundary
evaluations use \brTraceBits\ bits in Arb through
\texttt{python-flint}~\cite{flint,Joh17}. Exact decimal input is enclosed
without a binary64 conversion; floating-point proposals are read as exact
dyadics with explicit errors. The $403$ semantic controls and $31$ identity
negative controls test input syntax, coverage, centre binding, matrix signs,
published bounds and the radii gates. Nonfinite values, unresolved strict
comparisons, a failed subprocess or a timeout prevent a success receipt.
Before finalisation the driver reads back all artifact hashes. The displayed
bounds are rounded outward from the recorded enclosures.
The bound $E_0<\brEzero$ is the all-mode, weight-amplified clamped residual.

The recorded isolated run took \brRecordedSeconds\ seconds with at most
\brWorkers\ single-threaded workers. Its environment was Python $3.10.19$,
\texttt{python-flint} $0.9.0$, NumPy $1.26.4$, SciPy $1.15.2$ and mpmath
$1.3.0$. A second implementation recomputes $56$ scalar quantities in
$400$-bit \texttt{mpmath.iv} arithmetic from six Arb spectral primitives
and eight boundary moments, including the complement and Schur estimates,
norm transfer, lifts, corrected inverse, radii and geometry. It does not
reconstruct the special functions, spectral matrices or boundary moments.
The receipt field \texttt{full\_proof=false} records that the analytic
lemmas are not machine-checked. They are proved in the paper; the verifier
certifies their numerical hypotheses.

With these dependencies installed in an environment visible to isolated
Python, run
\begin{verbatim}
nice -n 19 env OMP_NUM_THREADS=1 OPENBLAS_NUM_THREADS=1 \
  MKL_NUM_THREADS=1 NUMEXPR_NUM_THREADS=1 \
  PYTHONDONTWRITEBYTECODE=1 \
  python -I -O verify_berenstein_planar_convex_v3.py \
  --workdir replay --jobs 4
\end{verbatim}
The working directory \texttt{replay} must not exist. The verifier installs
no packages. A successful run exits with status zero, writes
\path{replay/VERIFICATION_RECEIPT.json} with status
\texttt{NUMERICAL\_GATES\_PASS}, and ends with a line beginning
\texttt{NUMERICAL\_GATES\_PASS} followed by the receipt path and a line
beginning \texttt{RADIUS} that also prints \texttt{CONTRACTION}.
The recorded verifier SHA-256 is
\begin{center}\ttfamily\footnotesize
\brVerifierHash
\end{center}
The frozen centre's SHA-256 is
\begin{center}\ttfamily\footnotesize
\brCentreHash
\end{center}
The ancillary directory contains the verifier, separate frozen centre,
recorded isolated receipt, requirements, \texttt{README.md} and
\texttt{SHA256SUMS}. Only the machine-specific Python executable path in
the receipt is replaced by \texttt{python}. A fresh run regenerates its
stage sources and outputs. The scripts and records in
\path{anc/diagnostics/berenstein/} describe the numerical curvature,
eigenvalue-index and stationary-phase comparisons of
Subsection~\ref{br:sec:literature}. The script \path{csb_nonconvex.py}
encloses in Arb the triangle witness for the non-convexity of the domain
of~\cite{CSS26}. These records are not inputs to the contraction.

\subsection{Dimension four: what the script checks}
All finite computations for $n=4$ are carried out by a single script, \texttt{verify\_r4.py} (569 lines), which uses only the Python standard library and \texttt{python-flint} 0.9.0 \cite{flint}, the Python interface to the ball arithmetic of FLINT/Arb \cite{Joh17}. Every real number is represented by a ball $[m\pm\varepsilon]$ whose arithmetic is rigorous, including all rounding errors. The $525$ centre coefficients are embedded in the script as hexadecimal binary64 strings, which are converted to exact dyadic rationals (exactness is checked); $\rho=21/20$, $r=1/50000$ and the rational majorants of Theorem~\ref{thm:zero} are exact. The script performs the following steps.
\begin{enumerate}[label=(\arabic*)]
\item It expands $\Fc(x^\circ)$ and the $525$ finite columns of $D\Fc(x^\circ)$ in ball arithmetic, using \eqref{eq:rec}, \eqref{eq:K} and \eqref{eq:DF}; no coefficient is discarded.
\item It assembles $M_f$, computes a ball enclosure of $M_f^{-1}$, and defines $\widehat A$ as the matrix of the (exact, dyadic) midpoints of this enclosure. It then computes $I-\widehat AM_f$ in ball arithmetic and verifies \eqref{eq:Ahat} for the weighted column norm. The matrix $\widehat A$ is never treated as an interval-valued operator.
\item It computes $\nrm R$, $H$, the numbers $h_n$ for $M+2\leq n\leq N_H=443$ and the remainder bound of Lemma~\ref{lem:normA}, hence $A_t$ and a bound for $\nrm A$.
\item It computes $Y$ and $Z_f$ by applying $A$, in the form \eqref{eq:A}, \eqref{eq:B}, \eqref{eq:Tinv}, to the residual and to the finite columns.
\item It computes the $3650$ tail $g$-columns with $n\leq123$, $s\leq67$ with the short polynomial $p_{\leq16}$ and adds the error term of \eqref{eq:gshort}; it evaluates $Z_{g,\infty}$.
\item It computes the $480$ shape columns $43\leq j\leq1001$ without truncation, checks the hypotheses of Proposition~\ref{prop:farshape}, and evaluates $\sU_1,\ldots,\sU_5,\sU_W$ and $Z_{\mathrm{sh},\infty}$.
\item It evaluates $Z$, $c_2$, $c_3$, $C_2$, $C_3$, checks $Y<7\cdot10^{-6}$, $Z<37/100$, $C_2<4$, $C_3<10^{-3}$, and then the two inequalities \eqref{eq:radii} at $r=1/50000$.
\item It checks the geometric inequalities \eqref{eq:uni} and \eqref{eq:star}.
\end{enumerate}
In \texttt{python-flint}, an order relation between balls evaluates to true only if it holds for all points of the balls. Every check is an explicit test that raises an exception when it fails (also under \texttt{python -O}); maxima over columns are taken over rigorous upper endpoints, and nonfinite enclosures are rejected. After step (4) the script writes an intermediate checkpoint file (\texttt{*.base.json}), which carries no status field; only when all checks succeed does it write the certificate file (JSON) with the enclosures of all quantities of Table~\ref{tab:bounds} and the status \texttt{PROVED}, and print \texttt{PROVED}. A diagnostic mode (\texttt{--stage base}) stops after step (4) and never prints \texttt{PROVED}.

A second script, \texttt{verify\_convex.py}, checks that the centre data file coincides with the data embedded in \texttt{verify\_r4.py}, that the main certificate has status \texttt{PROVED} for the radius $r=1/50000$, and then verifies the bounds of Proposition~\ref{prop:convex}, which involve only the finite sums $A_1,A_2$ and the explicit tail constants $E_1,E_2$.

\subsection{Dimension four: reproduction}
\begin{sloppypar}
From a directory containing the files \texttt{verify\_r4.py}, \texttt{verify\_convex.py} and \texttt{data/centre.json}, the certificate is reproduced by
\end{sloppypar}
\begin{verbatim}
mkdir -p logs
nice -n 19 env OMP_NUM_THREADS=1 OPENBLAS_NUM_THREADS=1 \
  MKL_NUM_THREADS=1 python3 verify_r4.py --stage all --bits 128 --jstar 1001 \
  --output logs/certificate_r4_128.json
python3 verify_convex.py --bits 128 \
  --main-certificate logs/certificate_r4_128.json \
  --output logs/convex_128.json
\end{verbatim}
The second command prints \texttt{PROVED} as its last line, the third prints \texttt{PROVED\_CONVEX}. The runs use one thread. The distributed certificates are in the subdirectory \texttt{certificates/} of \texttt{anc/} (Section~\ref{subsec:avail}); the files written to \texttt{logs/} are byte-identical to them. On a single core of a Linux server the $128$-bit run took between $163$ and $167$ seconds of wall-clock time in several independent runs, and the same run at $192$ bits (\texttt{--bits 192}) took about $198$ seconds; the convexity check takes a fraction of a second. Runs under Python 3.10.19 and Python 3.12.3, both with \texttt{python-flint} 0.9.0, produced byte-identical certificate files. At $192$ bits all inequalities hold with the same displayed bounds, and the finite inverse defect \eqref{eq:Ahat} is below $7.5\cdot10^{-54}$.
\subsection{Dimension six}\label{subsec:cap6}
{\sloppy For $n=6$ the finite computations are split into five stage scripts, which evaluate the quantities of Section~\ref{sec:r6} in ball arithmetic at $128$ bits, in batches, and write interval receipts (JSON files), and a final script that combines the receipts. All scripts use Python, \texttt{python-flint} 0.9.0 and, only to read the matrix $\widehat A$, NumPy.
\begin{enumerate}[label=(\arabic*)]
\item \texttt{finite\_interval\_r6.py} expands $\Fc_6(x^\circ)$ without truncation and computes $A\Fc_6(x^\circ)$ and $Y$ (receipt \texttt{residual\_bound\_r6.json}); in eight batches covering all $465$ finite columns it computes the inverse defect \eqref{eq:Ahat6} and $Z_f$ (\texttt{finite\_batches\_r6.json}). The matrix $\widehat A$ is read from \texttt{inverse\_r6\_M58\_S30.npy}, and its binary64 entries are converted to exact dyadic rationals.
\item \texttt{tail\_inverse\_r6.py} computes $H$, the numbers $h_n$ for $M+4\leq n\leq N_H=458$ and the remainder bound of Lemma~\ref{lem:normA6}, hence $A_t$ (\texttt{tail\_inverse\_bound\_r6.json}).
\item \texttt{g\_tail\_interval\_r6.py} computes the $705$ near tail $g$-columns in seven batches with the short polynomial $p_{\leq24}$ and adds the error term of \eqref{eq:gshort6} (\texttt{g\_tail\_batches\_r6.json}); it also evaluates the two far bounds of Lemma~\ref{lem:gtail6} (\texttt{g\_far\_bound\_r6.json}).
\item \texttt{shape\_tail\_interval\_r6.py} computes the $85$ near shape columns $61\leq j\leq397$ without truncation, in five batches (\texttt{shape\_tail\_batches\_r6.json}).
\item \texttt{far\_shape\_r6.py} computes $\Gamma_1$, $\Gamma_0$, the expansion of $Q^\circ$ by the connection recurrences, the coefficients $\gamma^W_{m,t}$, $\sU_\Gamma$, $\sU_W$ and $Z_{\mathrm{sh},\infty}$ for $\jst=401$, and checks the support condition of Proposition~\ref{prop:farshape6} (\texttt{far\_shape\_bound\_r6.json}).
\item \texttt{verify\_r6.py --stage final} checks the SHA-256 checksums of the centre, of $\widehat A$, of the seven receipts and of the five stage scripts against values frozen in its source; checks that the batches of each receipt are contiguous and cover all columns of their class; recomputes $\nrm R$ from $\widehat A$; forms $Z$ and the constants $c_2,c_3,c_4$ of Lemma~\ref{lem:nonlinear6} with $\nrm A\leq1402.314$; checks the rational majorants of Theorem~\ref{thm:zero6} and the inequalities \eqref{eq:radii6} at $r=1/5000$; and checks the geometric inequalities \eqref{eq:geom6}. Only then does it write the certificate \texttt{certificate\_r6.json}, with the status \texttt{PROVED}, all enclosures, the frozen checksums and its own checksum, and print \texttt{PROVED}.
\item \texttt{verify\_convex\_r6.py} checks the checksums of \texttt{verify\_r6.py} and of the centre, and the status, radius, weight and checksums recorded in the certificate, and then verifies the bounds of Proposition~\ref{prop:convex6} (\texttt{convex\_r6.json}, status \texttt{PROVED\_CONVEX}).
\end{enumerate}
\par}
As for $n=4$, an order relation between balls is true only if it holds for all points of the balls, every check raises an exception when it fails, and maxima are taken over rigorous upper endpoints; a NaN or infinite enclosure, whether computed or read back from a receipt, stops the verification. The script \texttt{test\_fail\_closed\_r6.py} tests this behaviour on artificial inputs. A further script, \texttt{check\_identities\_r6.py}, checks in exact rational arithmetic the identities of Lemma~\ref{lem:K}(a)--(d) for all $n\equiv2\pmod4$, $n\leq202$, and $1\leq s\leq80$ ($4080$ pairs), reruns the connection recurrences used for $Q^\circ$ in rational arithmetic ($450$ pairs, all coefficients positive), checks $\nrm{Kg^\circ-(1-|w|^2)^2Q^\circ}_\rho\leq2.02\cdot10^{-31}$ in ball arithmetic, and compares \eqref{eq:Tj6} with the column algebra of the verifier, exactly for $5\leq j\leq2001$ and in ball arithmetic for $5\leq j\leq1001$. These identities are proved in Section~\ref{sec:r6}; the script is a consistency check.

The complete six-dimensional verification is reproduced, from the directory \texttt{anc/r6/} of the ancillary files, by
\begin{verbatim}
sha256sum -c SHA256SUMS
bash reproduce_r6.sh
\end{verbatim}
{\sloppy The script \texttt{reproduce\_r6.sh} uses the interpreter given by the environment variable \texttt{PYTHON}, if it is set, checks the \texttt{python-flint} version, runs \texttt{test\_fail\_closed\_r6.py}, regenerates all seven receipts from the centre and $\widehat A$ with the stage scripts, each batch in a separate single-core process under \texttt{nice -n 19} with a time limit of $600$ seconds, and compares them byte for byte with the distributed receipts. It then reruns steps (6) and (7) and \texttt{check\_identities\_r6.py} in a scratch directory, compares the resulting \texttt{certificate\_r6.json}, \texttt{convex\_r6.json} and \texttt{identities\_r6.json} byte for byte with the distributed files, and prints \texttt{REPRODUCTION COMPLETE}; any mismatch stops it with a nonzero exit code. On one core of a Linux server, with Python 3.10.19 and \texttt{python-flint} 0.9.0, the complete reproduction took about three minutes of wall-clock time (3~min~2~s and 3~min~5~s in two independent runs). The individual final steps are\par}
\begin{verbatim}
python3 verify_r6.py --stage final \
  --output certificate_r6.json
python3 verify_convex_r6.py
python3 check_identities_r6.py --output identities_r6.json
\end{verbatim}
which print \texttt{PROVED}, \texttt{PROVED\_CONVEX} and \texttt{IDENTITIES VERIFIED} and rewrite the three files byte-identically; they take about two seconds, less than a second and about fifteen seconds, respectively.

\subsection{Consistency checks for $n=4$ and $n=6$}
The following checks are not part of the proof. The algebraic identities of Sections~\ref{sec:disk}, \ref{sec:K} and Lemma~\ref{lem:Tform} were verified symbolically for small indices. An independently written floating-point implementation, which computes projections by quadrature instead of the recurrences \eqref{eq:rec} and uses its own approximate inverse, reproduced $\nrm R$, $H$, $Z_f$ and the largest shape column of Table~\ref{tab:bounds} to the displayed digits. The function $u$ was also evaluated directly in $\R^4$ from a Fourier--Bessel representation $F=\sum_{m\ \mathrm{odd}}a_mJ_m(r)\sin m\phi$ fitted to the boundary data of the centre: at random points of $\partial\Omega$ the defects $|u-1|$ and $|\nabla u|$ were below $10^{-10}$, and the Fourier transform of $\one_\Omega$, computed by quadrature, was below $10^{-10}$ in modulus on the unit sphere and between $18$ and $28$ in modulus on the sphere of radius $0.999$.

For $n=6$, Lemma~\ref{lem:laplace6} and the averages $\langle Y_\ell^3\rangle/\langle Y_\ell^2\rangle$ of Section~\ref{subsec:mechanism}, for $\ell=4,8,12,16$, were verified symbolically, and seven exact finite columns of $D\Fc_6(x^\circ)$ agreed with an independently written quadrature implementation to relative accuracy $1.7\cdot10^{-13}$. The Fourier--Bessel solution from which the centre was obtained (Remark~\ref{rem:centre6}) has Cauchy defects $8.2\cdot10^{-14}$ and $2.2\cdot10^{-13}$ at $2001$ boundary points off the collocation grid, with $20$ modes, and a solution with $28$ modes agrees with it. Using only its boundary, the first $14$ volume integrals over $\Omega$ of $O(3)\times O(3)$-invariant regular solutions of $\Delta v+v=0$, which vanish when $\widehat{\one_\Omega}$ vanishes on the unit sphere, were computed by quadrature; relative to the corresponding integrals in which the angular integrand is replaced by its modulus, they were between $0$ and $3.3\cdot10^{-13}$, compared with $4.7\cdot10^{-3}$ for a comparison domain that retains only the dominant degree-eight term of the boundary.

\subsection{Dimensions eight, ten and fourteen}\label{lie:subsec:cap}

The computations for $n=8,10,14$ are distributed in the directories \texttt{anc/r8}, \texttt{anc/r10} and \texttt{anc/r14}. They use Python, \texttt{python-flint} 0.9.0 \cite{flint,Joh17} and, only to read $\widehat A$, NumPy. Each dimension has stage scripts that evaluate the quantities of Sections~\ref{lie:subsec:conformal}--\ref{lie:subsec:Z} in ball arithmetic at $128$ bits and write interval receipts (JSON files), a final script that combines the receipts, and a separate convexity script.
{\sloppy
\begin{enumerate}[label=(\arabic*)]
\item For $n=8$ the algebra of disk polynomials is contained in \texttt{verify\_r8.py}; the stage scripts are \texttt{finite\_interval\_r8.py} ($Y$, $Z_f$, \eqref{lie:eq:Ahat}, $\nrm R$), \texttt{tail\_inverse\_r8.py} ($H$, $h_n$ for $M+2m\leq n\leq N_H$, $A_t$), \texttt{g\_tail\_r8.py} (the $396$ near tail $g$-columns in $20$ batches, and the far bounds of Lemma~\ref{lie:lem:gtail}), \texttt{shape\_tail\_r8.py} (the $29$ near shape columns in two batches) and \texttt{far\_shape\_r8.py} ($\Gamma_1$, $\Gamma_0$, $Q^\circ$, $\gamma^W_{k,t}$ and Proposition~\ref{lie:prop:farshape}).
\item For $n=10,14$ the algebra is in \texttt{rank2\_cap\_core.py}, shared by the stage scripts \texttt{finite\_rank2.py} (for $n=14$: \texttt{finite\_batch\_rank2.py}, in $26$ batches, and \texttt{aggregate\_finite\_r14.py}), \texttt{tail\_inverse\_rank2.py}, \texttt{g\_tail\_rank2.py} ($358$ columns in $18$ batches, resp.\ $1215$ columns in $13$ batches), \texttt{shape\_tail\_rank2.py} ($31$ columns in $4$ batches, resp.\ $10$ columns in $8$ batches), and \texttt{far\_shape\_rank2.py} for $n=10$, \texttt{far\_shape\_split\_r14.py} for $n=14$ (the first bound of Proposition~\ref{lie:prop:farshape}).
\item The final scripts \texttt{verify\_r8.py} (with the option \texttt{--stage final}), \texttt{verify\_r10.py} and \texttt{verify\_r14.py} check a SHA-256 digest, frozen in their source, of a bundle consisting of the centre, $\widehat A$, all stage scripts and all receipts; check that the batches of each receipt class are contiguous and cover all columns of the class; check the support conditions of Lemma~\ref{lie:lem:gtail} and Proposition~\ref{lie:prop:farshape}; form $Z$ and $c_2,\ldots,c_{m+2}$ by \eqref{lie:eq:cq} with the rational ceiling for $\nrm A$; check the rational majorants and \eqref{lie:eq:radii}; and check the bounds of Proposition~\ref{lie:prop:geometry}. Only then do they write a certificate with status \texttt{PROVED} and print \texttt{PROVED}.
\item \texttt{verify\_convex\_r8.py}, \texttt{verify\_convex\_r10.py} and \texttt{verify\_convex\_r14.py} check the status of the main certificate and the checksum of the main verifier, and recompute the bounds of Table~\ref{lie:tab:geom} (status \texttt{PROVED\_CONVEX}).
\item \texttt{check\_identities\_r8.py} and \texttt{check\_identities\_rank2.py} verify the identities of Lemma~\ref{lem:K}(a),(b) in exact symbolic arithmetic for $12$ index pairs, compare \eqref{lie:eq:Tj} with the column algebra at $256$ bits for three shape indices, check positivity and total mass one of $q_s(n)$, and check in ball arithmetic that all coefficients of $Kg^\circ-(1-|w|^2)^2Q^\circ$ are below $10^{-69}$ in absolute value. These identities are proved in Section~\ref{lie:sec}; the scripts are consistency checks.
\end{enumerate}
\par}
As for $n=4,6$, the centre coefficients, $B$ and the entries of $\widehat A$ are converted exactly to dyadic rationals (exactness is checked), and $\rho$, $\kappa_m$, $r$ and all thresholds are exact rationals. An order relation between balls is true only if it holds for all points of the balls, maxima are taken over rigorous upper endpoints, and a NaN or infinite enclosure, computed or read from a receipt, stops the verification. (In a preliminary version of the scripts for $n=10,14$, $\kappa_m$ was a binary64 number, which for $m=4$ lies $1.2\cdot10^{-18}$ below $1/48$; the distributed scripts use $\kappa_m$ exactly, and the displayed bounds are unaffected.)

{\sloppy In each directory, \texttt{bash reproduce\_r8.sh} (resp.\ \texttt{reproduce\_r10.sh}, \texttt{reproduce\_r14.sh}) copies the scripts, the centre and $\widehat A$ into a scratch directory, recomputes every receipt there, each job in a separate single-core process under \texttt{nice -n 19}, reruns the final and convexity scripts, and compares every distributed JSON file byte for byte with its recomputed counterpart. The final steps alone are\par}
\begin{verbatim}
python3 verify_r8.py --stage final \
  --centre center_r8_M45_S24.json --output certificate_r8.json
python3 verify_convex_r8.py
python3 verify_r10.py && python3 verify_convex_r10.py
python3 verify_r14.py && python3 verify_convex_r14.py
\end{verbatim}
which print \texttt{PROVED} and \texttt{PROVED\_CONVEX}; each takes less than four seconds on one core. In an independent rerun of all stage scripts, the $106$ interval receipts of the three dimensions were reproduced byte for byte, with a total of about $930$ seconds of single-core CPU time.

{\sloppy\emph{Consistency checks.} The following checks are not part of the proof. The script \texttt{check\_radial\_rank2.py} verifies Proposition~\ref{lie:prop:radial} for $\mathfrak{su}(3)$ in exact symbolic arithmetic in explicit matrix coordinates, and for the root data of $B_2$ and $G_2$ at $100$ random regular points each. An independently written check with explicit matrix realisations of $\mathfrak{so}(5)$ and of $\mathfrak g_2\subset\mathfrak{so}(7)$, as the stabiliser of a generic $3$-form, confirmed \eqref{lie:eq:radial} and the harmonicity of $\varpi$ at random regular points to relative accuracy $5.1\cdot10^{-13}$. A pointwise evaluation of $\Hc_m(x^\circ)$ confirmed the sign conventions (relative residuals between $10^{-15}$ and $10^{-12}$, against $2$ with the opposite sign). Centred differences at $256$ bits reproduced the shape column $j=2m+1$ of $D\Hc_m(x^\circ)$ to weighted accuracy $3.4\cdot10^{-12}$, $7.4\cdot10^{-16}$ and $2.9\cdot10^{-15}$ for $n=8,10,14$; the tail inverse \eqref{lie:eq:Tinv} was checked against $T$ on four tail shape columns in each dimension (errors below $10^{-62}$); the minimal output frequencies and radial indices of the first far tail $g$-columns were computed directly (for $n=14$: $126>M$ and $61>S$); and $13$ artificially corrupted receipts (NaN or infinite values, coverage gaps, bounds above the thresholds) were all rejected by the final scripts.\par}

\subsection{Dimension three}\label{r3:sec-cap}
The convex $n=3$ files are in \url{anc/r3_convex/}. The single-file verifier \url{verify_r3_convex.py} is frozen; its full SHA-256 is recorded in the ancillary README. It embeds the exact dyadic centre, the stage programs and the C ball-arithmetic accelerator. A fresh isolated replay completed all $17$ stages, printed \texttt{PROVED}, and wrote the full receipt after $6515.9$ seconds with five single-thread workers. A second from-scratch run of the same unmodified verifier, in a separate fresh directory, also printed \texttt{PROVED}; its receipt agrees with the first in every mathematical field. The distributed receipts record the centre and matrix hashes, the complete inverse and derivative bounds, the radii inequalities, the shape-ball geometry and the exact coverage counts.

The stages construct the $7371$-dimensional finite matrix and certify its dyadic product defect with a fixed-order dot-product error bound; certify the full-centre boundary-symbol inverse; compute the full-support residual; enclose all $7280$ finite field columns; compute $457$ exceptional field columns with the $j\leq121$ evaluation truncation and rigorous omitted-tail error; certify $91$ radial classes, the fourteen near angular classes and $\ell\geq210$; check $310$ explicit shape columns and every $j\geq621$; and assemble the nonlinear, geometric and coverage gates. Every strict inequality is checked with Arb balls. Missing columns, stale inputs, nonfinite enclosures or a failed gate prevent \texttt{PROVED}.

From an empty scratch directory, reproduction is requested by
\begin{verbatim}
python3 -O verify_r3_convex.py --workdir <empty dir> \
  --output receipt.json --jobs 5
\end{verbatim}
with Python 3, NumPy, SciPy, \texttt{python-flint} 0.9.0, \texttt{gcc} and 64-bit Linux. The run takes about $1.8$ hours on five cores. The numbers reported here are those of the frozen receipt; the SHA-256 of the verifier is recorded in the ancillary README.

The independent nonconvex example has its own verifier, two completed certificates and a separate exact rational nonconvexity check in \texttt{anc/r3/}. Its computation uses $M=180$, $S=64$ and the different centre and shape scale of Proposition~\ref{r3:prop-AX3}. Its previously certified values are Table~\ref{r3:tab-old}. They do not enter the convex certificate.

\subsection{Dimension five}\label{r5:sec-cap}
{\sloppy The five-dimensional files are in \texttt{anc/r5/}. The frozen single-file verifier is \texttt{verify\_r5.py}; its SHA-256 begins with \texttt{b49dc914} and is recorded in full in the files \texttt{README.md} and \texttt{SHA256SUMS} of that directory. It embeds, in a compressed payload that is checked against its own SHA-256, the exact dyadic centre (SHA-256 \texttt{c672b3fb\ldots}), the frozen class parameters of Section~\ref{r5:sec-Z} (also distributed as \texttt{params.json}), the stage programs and a C library of ball-arithmetic and long-double kernels; nothing computed is read from disk. In a new empty work directory it compiles the kernels and runs a self-test, which includes the identities $\sum_nD_{n\ell}=1$ and $\sum_\ell C_{\ell N}D_{m\ell}=\delta_{mN}$ for $N\leq40$ and Lemma~\ref{r5:lem-lap} for $K\leq30$ in exact arithmetic. In $23$ stages, run by at most six single-threaded worker processes, it then computes the boundary Taylor and Jacobi data of $H^\circ$, $W^\circ_1$, $Z^\circ$ and $W^\circ$; the residual; all $4698$ finite columns and the $59$ exceptional and $219$ explicit shape columns; the matrix $\widehat A$, its certified defect and its column norms; $\eps_\Upsilon$, $\tau_T$, $\mathcal B$, $\rho_{\mathcal B}$, $A_t$ and the profiles $\Gamma_F$, $\Gamma_H$; the defects of all explicit columns and $Y$; the field envelopes for the frozen classes; the far shape bound of Proposition~\ref{r5:prop-far}; and the constants $c_2,c_3,c_4$ and the bounds of Proposition~\ref{r5:prop-geometry}. Finally it checks that the classes cover every input exactly once and that all gates of Theorem~\ref{r5:thm-zero} hold. Only then does it print \texttt{PROVED} and write the receipt; otherwise it prints \texttt{NOT\_CERTIFIED} and exits with status~$2$. Every inequality is checked with Arb balls; binary64 and long-double numbers enter the bounds only through exact dyadic conversion or with rigorous rounding-error terms, maxima are taken over rigorous upper endpoints, and nonfinite enclosures are rejected.\par}

{\sloppy A fresh run of the distributed verifier in an empty directory printed \texttt{PROVED} after $2364$ seconds of wall-clock time with six single-threaded workers. Its receipt is distributed as \texttt{VERIFICATION\_RECEIPT.json}, in which only the machine-specific work-directory string was replaced, and Table~\ref{r5:tab-bounds} is taken from it. A second from-scratch run of the same verifier, in a separate empty directory, also printed \texttt{PROVED} (after $2355$ seconds); its receipt \texttt{VERIFICATION\_RECEIPT\_replay2.json} agrees with the first in every mathematical field. The computation is reproduced, from the directory \texttt{anc/r5/} and with an empty scratch directory, by\par}
\begin{verbatim}
python3 -O verify_r5.py --workdir <empty dir> \
  --output receipt.json --jobs 6
\end{verbatim}
{\sloppy\noindent with 64-bit Linux (the long-double kernels check the x87 format), Python~3, NumPy, \texttt{python-flint} 0.9.0 and \texttt{gcc}; it takes about $40$ minutes on six cores. The uniform bound for $N>N_2$ in Proposition~\ref{r5:prop-far} includes the edge terms at $k=k_0$ described in its proof; their total size is below $2\cdot10^{-16}$.\par}

\emph{Consistency checks.} The following checks are not part of the proof. An independently written floating-point implementation reproduced the uniform angular bound $Z_{\mathrm{ang}}$ and its five parts to ten digits and $Y$ and the column defects to eight digits. Eight field columns computed by direct evaluation, and seven shape columns computed from a formula without recurrences, agreed with those of the verifier. An independent multiprecision evaluation of the approximate solution defined by the centre gave a Helmholtz residual below $7\cdot10^{-19}$ at $250$ sample points, compared with values of $|u|$ up to about $10^3$, and $|u-1|<3\cdot10^{-65}$, $|\nabla u|<3\cdot10^{-33}$ on the boundary; for the centre, the principal curvatures of $\partial\Theta_{c^\circ}(B^5)$ are at least $0.912$ in the meridian direction and at least $0.993$ in the three transverse directions.

\subsection{Dimension seven}\label{r7:sec-cap}
The files in \url{anc/r7/} contain the frozen single-file verifier \url{verify_r7.py}, its two full-run receipts and checksums. Its SHA-256 is
\begin{center}\ttfamily\small
f8fd0ae17f4de35e7031b17ae61bde49e7d6c1e98705456e82cd1cca070989ed
\end{center}
The compressed payload has SHA-256 \texttt{a9335ae7\ldots}; it embeds the exact dyadic centre (SHA-256 \texttt{31de5cfd\ldots}), the class parameters, all stage sources and the C kernels. The field and shape numbers are the frozen centre data; no development file is used as a proof input.

In $28$ stages the script rebuilds the kernels and runs exact self-tests, computes the boundary Taylor data and the full-support residual, encloses all $9010$ finite columns, $133$ exceptional field columns and $245$ explicit shape tail columns, constructs the scaled finite inverse and certifies its defect, computes the boundary Toeplitz inverse and its Neumann remainders, the column profiles, every infinite field envelope and the all-index shape bound, and checks the quartic Taylor, radii, contraction and whole-ball geometry inequalities of Theorem~\ref{r7:thm-zero}. It also checks that the input classes cover every coordinate exactly once. Arb balls supply the enclosures; dyadic conversions and explicit dot-product rounding bounds cover the floating-point kernels. Missing data, nonfinite enclosures, a failed worker or a failed inequality prevent \texttt{PROVED}. The integer gates of the far-shape stage ($N_*$ and $N_2$ even with $N_2\geq N_*\geq500$, and $N_*+\eta>M$, $N_*+\eta\geq2I+2$ for every monomial $(\eta,e)$ of the principal polynomial) imply the hypotheses of Lemma~\ref{r7:lem-extension}: the monomial $(0,0)$ has the single contribution $(c^\circ_1)^2>0$, so $N_*\geq2I+2$ and $K_{\max}=N_2+2>N_*$, and $2i_c=200\leq500$. The inequalities of Lemma~\ref{r7:lem-switch} are checked in ball arithmetic.

From \texttt{anc/r7/}, with a new or empty scratch directory, run
\begin{verbatim}
python3 -O verify_r7.py --workdir <empty dir> \
  --output receipt.json --jobs 2
\end{verbatim}
It requires 64-bit Linux with x87 extended precision, Python~3, NumPy, version~0.9.0 of \texttt{python-flint}, and \texttt{gcc}. The recorded complete run printed \texttt{PROVED} after $6175$ seconds (about $103$ minutes) with eight single-threaded workers; \texttt{--jobs} accepts $1$ through $8$, and the two-worker command uses less concurrency. On success the receipt is written; a failure prints \texttt{NOT\_CERTIFIED} and exits with status~$2$. In the distributed receipt only the machine-specific work-directory string was replaced. A second from-scratch run of the same unmodified verifier, in a separate empty directory, also printed \texttt{PROVED} (after $10629$ seconds with six single-threaded workers); its receipt \texttt{VERIFICATION\_RECEIPT\_replay2.json} agrees with the first in every mathematical field. Its mathematical fields supply Table~\ref{r7:tab-bounds}; the full hashes and the integrity command are in the README.

\subsection{The exact inverse in general dimension}\label{gd:sec:verification}
The directory \texttt{anc/gd/} contains the single-file verifier
\path{verify_axisym_exact_v2.py} for Theorem~\ref{gd:thm:main}.
It embeds the finite exact-decimal centres, the dimension parameters
and the stage sources in a hash-checked compressed bundle.
It embeds no computed spectral matrix, boundary coefficient or
existence receipt. A fresh run extracts the sources and inputs,
checks their hashes and the module bindings, and recomputes the
numerical chain in one process for each dimension.
Relative to the single-file verifier that was first run for $n\in\{5,9,11,12,13\}$,
this file changes the dimension lists, the embedded centres and parameters
of the six further dimensions, and the two spectral settings for $n=16$
stated below. For $n\in\{5,9,11,12,13\}$ it reproduces every certified
bound of the first verifier.

The spectral stage encloses all roots in~\eqref{gd:eq:window},
compares the two enumerations root by root, encloses the tensor
Gaussian data and polynomial potential, adds the errors of
Section~\ref{gd:sec:quadrature}, and evaluates the full Schur bound.
The spectral arithmetic uses $768$ bits and $384$ radial Gauss nodes,
except for $n=16$, where it uses $1536$ bits and $640$ radial nodes.
The boundary stage encloses the trigonometric samples and the
Acb Fourier transform, including the half-grid phase, and adds
every retained alias allowance and the entire analytic tail.
The angular conversion, heat and maximum-principle bounds supply
the norm transfer. The coupled stage forms the residual of the
exactly clamped centre, the corrected inverse bound, the left
Newton defect, the quartic remainder and the whole-ball geometry.
Every retained decimal is an exact rational input.
Special functions and enclosure arithmetic use Arb/Acb through
\texttt{python-flint}~\cite{flint,Joh17}; binary64 matrix proposals
use exact dyadic interpretation and the explicit error budget.

The acceptance path accepts a strict interval comparison only when it holds for every
pair of points of the two enclosures,
rejects nonfinite or unresolved bounds, and is run under
\texttt{python -E -O}. A successful run performs $125$ exact symbolic connection-formula and
lift examples, $25$ in each of dimensions $5,9,11,12,13$, as well as the
symbolic-$n$ material identity, and rejects $52$ semantic failure controls.
The connection examples test the closed formula; they do not test
the implementation's connection recurrence.
After the controls, the final scalar chain is recomputed and all
receipt artifacts are read back against their hashes.

A second implementation in \texttt{mpmath.iv}, at $120$ decimal digits,
recombines the final residual, coupled, quartic, radii and geometric
expressions from the Arb primitives. It compares overlap of the nine
final scalar enclosures. It does not independently enclose the
special functions, spectral matrices, boundary samples, shape-norm constants, norm-transfer
constants or fixed-field primitives.
The scalar gates also include $4C_2Y_0<10^{-10}$; existence uses
the invariant-ball and contraction inequalities of~\eqref{gd:eq:radii}.
Every printed inequality bound is rounded outward from its own enclosure;
products of rounded entries need not recover the sharper product bounds.
The analytic lemmas of this paper and the stated numerical arithmetic
together give the certificate.

From \texttt{anc/gd/}, with an absent or empty output directory, run
\begin{verbatim}
nice -n 19 env OMP_NUM_THREADS=1 OPENBLAS_NUM_THREADS=1 \
  MKL_NUM_THREADS=1 NUMEXPR_NUM_THREADS=1 \
  python -E -O verify_axisym_exact_v2.py \
  --dimension 12 --workdir replay12
\end{verbatim}
The successful final line is \texttt{PROVED dimensions=12}, with exit
status zero. The overall and per-dimension receipts must both have
status \texttt{PROVED}. A search for the substring
\texttt{PROVED} is insufficient: it occurs in the failure status
\texttt{NOT\_PROVED}, the overall failure receipt status
\texttt{NOT\_ALL\_PROVED}, the spectral stage status
\texttt{PROVED\_L2\_DIRICHLET\_INVERSE\_ONLY} and the per-dimension lines.
The option \texttt{--dimension all} executes all supported dimensions
sequentially; its final line lists those dimensions.
The spectral module pins to the lowest allowed CPU; an explicit
\texttt{taskset} restriction to one allowed CPU fixes that choice.
Table~\ref{gd:tab:replay} records, for each dimension, the strip radius $S$
and the parameters $K$, $J_b$ of Subsection~\ref{gd:sec:boundary}, the
boundary precision in bits, and the times of the final isolated runs.

\begin{table}[htbp]
\centering\small
\def\gdReplayRow#1#2#3#4#5#6#7{$#1$ & $#2$ & $#3$ & $#4$ & $#5$ & $#6$ & $#7$\\}
\begin{tabular}{@{}crrrrrr@{}}
\toprule
$n$ & boundary $S$ & $K$ & $J_b$ & bits & chain (s) & process (s)\\
\midrule
\gdReplayData
\bottomrule
\end{tabular}
\caption{Boundary-stage parameters and times of the final isolated runs of
the verifier, rounded to three decimals. The process time also
includes extraction and symbolic checks. Each used one allowed CPU,
one Arb/BLAS thread and priority \texttt{nice 19}.}
\label{gd:tab:replay}
\end{table}

The recorded environment is Python $3.10.19$,
\texttt{python-flint} $0.9.0$, NumPy $1.26.4$, SciPy $1.15.2$,
SymPy $1.14.0$ and mpmath $1.3.0$.
The verifier's SHA-256 is
\begin{center}\ttfamily\footnotesize
b33fd6c7e9d72a283c9f752ec55b8e1d4dec48b132872759ac4a32553d1e3940
\end{center}
The separate frozen centres \path{frozen_centers/nN.json.gz} are the
gzip-compressed exact-decimal inputs (keys \texttt{n}, \texttt{rho},
\texttt{c}, \texttt{a}, \texttt{scales}, \texttt{status}); \texttt{gzip -dc} returns the JSON inputs
of the replays byte for byte.
The ancillary directory retains the overall
\path{receipts/nN/VERIFICATION_RECEIPT.json} of the final
isolated run for each dimension; for dimension $21$ this is the only
completed isolated replay of the verifier. These receipts bind the verifier,
payload and dimension and record the wall time. A fresh run regenerates
the extracted sources, stage outputs, identity receipt, both final
receipts and artifact-hash readbacks in its output directory; compare
the numerical enclosures with the tables in this paper.
The checksum command is \texttt{sha256sum -c SHA256SUMS}, run inside
\texttt{anc/gd/}. The ancillary requirements record the versions
actually used by these receipts.

\subsection{Availability}\label{subsec:avail}
All scripts, data and certificates are distributed as ancillary files with the arXiv version of this paper, in the directory \texttt{anc/}. The file \texttt{anc/README.md} describes the layout, the requirements, the commands and the running times for all dimensions. The planar verifier, recorded receipt, frozen forty-mode centre, reference shape, stage sources and diagnostic scripts are in \texttt{r2/}; the four-dimensional files are \texttt{verify\_r4.py}, \texttt{verify\_convex.py}, \texttt{data/centre.json} and the four certificates in \texttt{certificates/}; the six-dimensional files, with instructions in \texttt{RELEASE\_README.md}, are in \texttt{r6/}; the files for $n=8,10,14$, each directory with its own \texttt{README.md}, are in \texttt{r8/}, \texttt{r10/} and \texttt{r14/}; the convex three-dimensional verifier and its receipts are in \texttt{r3\_convex/}, and the independent nonconvex example with its two certificates and exact nonconvexity check is in \texttt{r3/}; the five-dimensional verifier, its receipt and the frozen class parameters are in \texttt{r5/}; the seven-dimensional verifier and its two full-run receipts are in \texttt{r7/}; the general-dimension verifier, frozen centres and overall receipts are in \texttt{gd/}; the convex planar Berenstein files are in \texttt{berenstein/}; their diagnostic comparisons and the interval triangle witness are in \path{diagnostics/berenstein/}; and \texttt{check\_radial\_rank2.py} is the consistency check of Proposition~\ref{lie:prop:radial} described in Section~\ref{lie:subsec:cap}. The file \texttt{anc/SHA256SUMS} lists the SHA-256 checksums of all other files under \texttt{anc/}, and each of the subdirectories \texttt{r2/}, \texttt{r3/}, \texttt{r3\_convex/}, \texttt{r5/}, \texttt{r6/}, \texttt{r7/}, \texttt{r8/}, \texttt{r10/}, \texttt{r14/}, \texttt{gd/}, \texttt{berenstein/}, \path{diagnostics/berenstein/} carries its own file \texttt{SHA256SUMS}; the integrity of the files is checked by \texttt{sha256sum -c SHA256SUMS}, run in \texttt{anc/} and in each of these subdirectories. The files for $n=4$ and $n=6$ are also available at \url{https://github.com/aster2024/schiffer-pompeiu-r4} (in \texttt{verification/} and \texttt{verification/r6/}) and archived on Zenodo in~\cite{Guo26}.

\section{Remarks on the mechanism, other dimensions and curved spaces}\label{sec:remarks}

The statements in this section are not used in the proofs of Theorems~\ref{thm:r3} and~\ref{thm:main}. The uncertified branches listed in Section~\ref{subsec:numerics} are based on numerical computations only.

\subsection{Symmetry classes for which the reduction is exact}
Remark~\ref{rem:why4} shows that, among splittings $\R^n=\R^{d_1}\times\R^{d_2}$ with symmetry group $O(d_1)\times O(d_2)$, a reduction to the planar Laplacian without potential exists only for $n\in\{2,4,6\}$. The classes used for $n=4$ and $n=6$ are the adjoint classes of $\mathfrak u(2)$ and $\mathfrak{su}(2)\oplus\mathfrak{su}(2)$ (Remark~\ref{lie:rem:m12}), and Section~\ref{lie:sec} carries out the reduction for all nonabelian compact Lie algebras of rank two; by Remark~\ref{lie:rem:complete}, the dimensions $4$, $6$, $8$, $10$ and $14$ are the only ones reached by exact planar reductions of adjoint type. The three-, five- and seven-dimensional problems admit no such reduction; Sections~\ref{sec:r3}, \ref{sec:r5} and~\ref{r7:sec} use fixed-domain formulations in the corresponding unit balls. The axisymmetric exact-inverse construction of Part~\ref{gd:part} uses the full ball Laplacian and has no rank-two restriction.

\subsection{A transcritical mechanism}\label{subsec:mechanism}
Let $n\geq3$ and let $\rho$ be a positive zero of $J_{n/2}$, so that the unit ball of $\R^n$ carries the radial solution of \eqref{eq:schiffer} with $\mu=\rho^2$. Consider domains $\{x:|x|<1+sY_\ell(x_1/|x|)+O(s^2)\}$ invariant under $O(n-1)\times\Z_2$, where $Y_\ell$ is the zonal spherical harmonic of even degree $\ell$ on $S^{n-1}$, normalised by $Y_\ell(1)=1$. For $n=6$ we use instead the class of Section~\ref{sec:r6}, of domains invariant under $O(3)\times O(3)$ and the exchange of the factors; there the invariant spherical harmonics of degree $\ell=2k$ are the functions $Y_\ell(\omega)=U_k(t)/(k+1)$ of $t=|\omega'|^2-|\omega''|^2$, and the exchange symmetry requires $k$ to be even. The linearisation of the Schiffer problem at the ball in the direction $Y_\ell$ is degenerate exactly when $J_{\ell+n/2-1}(\rho)=0$, and this never happens for integer $n$ and $\ell\geq2$: for even $n$ this is Bourget's hypothesis, proved by Siegel \cite{Sie29}, that Bessel functions of distinct nonnegative integer orders have no common positive zero, and for odd $n$ it follows from a similar transcendence argument based on the Lindemann--Weierstrass theorem. So there is no bifurcation from balls in this class. However, $J_{\ell+n/2-1}(\rho)$ can be small. A formal Lyapunov--Schmidt expansion in the direction $Y_\ell$ then gives, at second order, a reduced equation of the form
\[
-\sigma s+\frac{\rho^2}2\,\frac{\langle Y_\ell^3\rangle}{\langle Y_\ell^2\rangle}\,s^2+O(s^3,\sigma s^2)=0,
\]
where $\langle\cdot\rangle$ is the average over $S^{n-1}$ and $\sigma$ is a small quantity proportional to $J_{\ell+n/2-1}(\rho)$. This suggests a non-ball solution with boundary amplitude $s_*\approx2\sigma\langle Y_\ell^2\rangle/(\rho^2\langle Y_\ell^3\rangle)$.

In the plane the corresponding cubic average $\langle\cos^3N\phi\rangle$ vanishes and the reduced equation is odd in $s$. The cubic amplitude balance is described in Section~\ref{sec:intro}; the exact planar solution is proved in Section~\ref{r2:sec:existence}. For $n\geq3$ and even $\ell$ the cubic average does not vanish: for $n=3$, $\langle P_\ell^3\rangle/\langle P_\ell^2\rangle=(2\ell+1)\left(\begin{smallmatrix}\ell&\ell&\ell\\0&0&0\end{smallmatrix}\right)^2>0$; for $n=4$ the zonal harmonics are $U_\ell(\cos\phi)/(\ell+1)$, where $U_\ell(\cos\phi)$ is the character of the irreducible representation of $SU(2)\cong S^3$ of dimension $\ell+1$, and the Clebsch--Gordan rule gives $\langle U_\ell^3\rangle=\langle U_\ell^2\rangle=1$ for even $\ell$; for $n=5$ and $n=7$, $\langle Y_\ell^3\rangle/\langle Y_\ell^2\rangle$ is the linearisation coefficient $G_{\ell\ell\ell}$ of Theorem~\ref{r5:thm-mult}(a), respectively Lemma~\ref{r7:lem-mult}, which is positive for even $\ell$ by \cite[18.18.22]{DLMF}; for $n\in\gdDimensions\setminus\{5\}$ it is the coefficient $L_{\ell\ell}^{\ell}$ in the proof of Lemma~\ref{gd:lem:algebra}, positive for even $\ell$ by the same formula. For $n=6$ in the $O(3)\times O(3)$-invariant class, the invariant measure of $S^5$ is carried by $t=|\omega'|^2-|\omega''|^2$ to a multiple of $\sqrt{1-t^2}\,dt$, for which the $U_k$ are orthonormal after normalisation, and $U_k^2=\sum_{i=0}^kU_{2i}$ gives $\langle U_k^3\rangle=\langle U_k^2\rangle=1$ for even $k$, hence $\langle Y_\ell^3\rangle/\langle Y_\ell^2\rangle=1/(k+1)$. That quadratic terms make axisymmetric bifurcations in the even-degree representations of $O(3)$ generically transcritical is a classical fact of equivariant bifurcation theory \cite{IG84,GSS88}. This classical mechanism applies to the Schiffer problem near a near-resonance, with the quadratic coefficient given by the projection of the squared harmonic onto the original harmonic.

For $(n,\rho,\ell)=(4,j_{2,4},6)$, the zero $j_{2,4}\approx14.79595$ of $J_2$ is close to the zero $j_{7,2}\approx14.82127$ of $J_7=J_{\ell+n/2-1}$, and the formal prediction for the relative $Y_6$-amplitude is about $2.4\cdot10^{-2}$. The domain of Theorem~\ref{thm:main} has $Y_6$-amplitude about $2.2\cdot10^{-2}$ (Remark~\ref{rem:domain}); in the meridian plane the $Y_6$ direction corresponds to the mode $J_7(r)\sin7\phi$ (Section~\ref{sec:reduction}). This is how the example was located. Numerically, Newton iterations started from the ball perturbed by a small multiple of $Y_6$ converge either back to the ball or to our domain, depending on the size of the perturbation, as expected from a transcritical picture.

For $(n,\rho,\ell)=(6,j_{3,8},8)$, the zero $j_{3,8}\approx28.90835$ of $J_3$ is close to the zero $j_{10,5}\approx28.88738$ of $J_{10}=J_{\ell+n/2-1}$, and $J_{10}(j_{3,8})\approx-3.0\cdot10^{-3}$. The formal prediction for the relative amplitude of the degree-eight harmonic is about $7.3\cdot10^{-3}$, and the six-dimensional domain of Theorem~\ref{thm:main} has amplitude about $7.2\cdot10^{-3}$ (Remark~\ref{rem:domain}). In the meridian plane this direction corresponds to the mode $J_{10}(r)\sin10\phi$ (Section~\ref{sec:r6}). The centre of Section~\ref{sec:r6} was located in this way (Remark~\ref{rem:centre6}).

For $(n,\rho,\ell)=(3,j_{3/2,15},12)$ one has $J_{25/2}(j_{3/2,15})\approx2.3\cdot10^{-3}$, and the additional nonconvex domain of Proposition~\ref{r3:prop-AX3} has a dominant degree-twelve harmonic of relative amplitude about $1.5\cdot10^{-2}$. The convex domain of Theorem~\ref{thm:r3} instead lies near $(j_{3/2,34},18)$, as recorded in Table~\ref{tab:resonances}. For $(n,\rho,\ell)=(5,j_{5/2,17},12)$ one has $J_{27/2}(j_{5/2,17})\approx-6.1\cdot10^{-4}$, and the five-dimensional domain of Section~\ref{sec:r5} has a dominant degree-twelve zonal harmonic of relative amplitude about $2.5\cdot10^{-3}$ (Remark~\ref{rem:domain}). For $(n,\rho,\ell)=(7,j_{7/2,25},14)$ one has $J_{33/2}(\rho)\approx2.3\cdot10^{-4}$, and Section~\ref{r7:sec} gives the corresponding convex domain. Numerically, the centres of Part~\ref{gd:part} for $n\in\gdDimensions\setminus\{5\}$ lie at the near-resonances listed in Table~\ref{tab:resonances}; their dominant zonal harmonics have relative polar amplitudes between about $1.6\cdot10^{-5}$ (for $n=16$) and $2.0\cdot10^{-2}$ (for $n=18$) in absolute value (Remark~\ref{rem:domain}). For the adjoint classes of Section~\ref{lie:sec} the invariant harmonics and the near-resonances are described in Section~\ref{lie:subsec:reduced}, and the centres were located at the near-resonances of Table~\ref{tab:resonances} (Remark~\ref{lie:rem:centres}). In these classes the cubic average does not vanish either: by the proof of Proposition~\ref{lie:prop:radial}, the average over $S^{n-1}$ of an $\Ad(G)$-invariant function is proportional to the integral of its restriction to the unit circle of $\tf$ against $\sin^2(m\theta)\,d\theta$; the invariant harmonic of degree $\ell$ restricts to a multiple of $U_{\ell/m}(\cos m\theta)$, and in the variable $m\theta$ this is again the character computation of $SU(2)$, which gives $\langle U_k^3\rangle=\langle U_k^2\rangle=1$ for even $k=\ell/m$, as the enlarged symmetry requires.

\subsection{Further solutions in dimensions three and five}\label{subsec:numerics}
With a spectral collocation solver (expansions in exact Helmholtz solutions $r^{1-n/2}J_{\ell+n/2-1}(\sqrt\mu\,r)Y_\ell$, boundary given as a radial graph, Gauss--Newton iteration), we found numerical solutions of \eqref{eq:schiffer} on non-ball, star-shaped, $O(n-1)\times\Z_2$-invariant domains near the near-resonances predicted in Section~\ref{subsec:mechanism}. For $n=3$, $\sqrt\mu=j_{3/2,15}\approx48.6741$ (unit ball normalisation) and $\ell=12$, with amplitude about $1.5\cdot10^{-2}$ and off-grid residual about $2\cdot10^{-11}$ with $150$ modes, such a solution was the starting point for the additional nonconvex example of Proposition~\ref{r3:prop-AX3}, which certifies an exact solution nearby. The following further numerical branches have not been certified:
\begin{itemize}
\item $n=3$, $\sqrt\mu=j_{3/2,42}\approx133.5102$, $\ell=20$, boundary amplitude about $9.6\cdot10^{-4}$. In multiprecision the collocation residual is about $10^{-36}$, and the residual off the collocation grid decreases spectrally with the truncation (about $10^{-11}$, $10^{-14}$, $10^{-17}$ and $10^{-20}$ for $120$, $160$, $200$ and $240$ modes). A test of the Pompeiu condition that uses only the computed boundary gives $2\cdot10^{-12}$, compared with $1.4\cdot10^{-2}$ for a comparison domain.
\item $n=5$, $\sqrt\mu=j_{5/2,30}\approx97.3586$, $\ell=16$, amplitude about $7.1\cdot10^{-3}$, and $\sqrt\mu=j_{5/2,24}\approx78.5016$, $\ell=14$, amplitude about $2.1\cdot10^{-2}$; residuals about $10^{-11}$, stable under changes of the truncation.
\end{itemize}
By Remark~\ref{rem:why4}, a reduction to the planar Laplacian of the type used in Sections~\ref{sec:reduction} and~\ref{sec:r6} is not available for $n=3$ and $n=5$; the formulations of Sections~\ref{sec:r3} and~\ref{sec:r5} give the corresponding equations for these numerical branches. The certified five-dimensional domain of Section~\ref{sec:r5} lies at a different near-resonance, $(j_{5/2,17},12)$. Part~\ref{gd:part} gives a second proof of the five-dimensional case, using different spaces and a different inverse. The seven-dimensional branch at $(j_{7/2,25},14)$ is certified in Section~\ref{r7:sec}; its dominant zonal harmonic has relative polar amplitude about $9.5\cdot10^{-4}$.

\subsection{Curved spaces}
In spheres and hyperbolic spaces the curvature provides a genuine one-parameter family in which radial Neumann eigenvalues and non-radial Dirichlet eigenvalues of geodesic balls can cross, which allows local bifurcation from geodesic balls. Cao-Labora and Fern\'andez \cite{CLF25} used this to construct contractible Schiffer domains on the half-sphere $S^2$, and they pointed out \cite[Remark~1.2]{CLF25} that their method generalises to other non-flat spaces, with numerically observed crossings for $S^3$, $S^4$ and $H^2$. Earlier examples in spheres are bounded by isoparametric hypersurfaces \cite{Shk00,PS21}, whose boundaries are not topological spheres in general. The mechanism of Section~\ref{subsec:mechanism} is also relevant in spheres and hyperbolic spaces of dimension $n\geq3$.

\section{Relation with rigidity results and other claims}\label{sec:rigidity}\label{r2:sec:rigidity}
The planar geometric theorems below concern the domain of Theorem~\ref{r2:thm:main}; the spectral-sequence theorem concerns all the domains of Theorems~\ref{thm:main} and~\ref{br:thm:main}, and the higher-dimensional spectral and perturbation comparisons concern the corresponding domains of Theorem~\ref{thm:main}. Apart from the certified bounds already stated, the planar numerical values are binary64 evaluations for the forty-mode centre $\psi_0$, from \path{anc/r2/diagnostics/}, and are not interval enclosures. The planar domain is uniformly convex and centrally symmetric, with real-analytic boundary.

The Berenstein domain of Theorem~\ref{br:thm:main} is treated separately in Subsection~\ref{br:sec:literature}: its boundary arc-length measure replaces the indicator measure, its eigenfunction has zero Dirichlet data and constant nonzero flux, and the comparison uses the Dirichlet version of the rigidity results.

\subsection{Dai--Sun--Wei--Zhang: high frequency in the plane}
Dai, Sun, Wei and Zhang~\cite[Theorem 1.2, version 2]{DSWZ25} prove that a bounded uniformly convex planar domain with Lipschitz connected boundary carrying a Schiffer solution at an eigenvalue $\alpha>\alpha_0(\Omega)$ must be a disc. The threshold depends on the domain. Their quantitative argument uses support coordinates in Steiner position. Write $h(\theta)$ for the support function and $\varrho(\theta)=h+h''$ for the radius of curvature. The relevant amplitude defect is
\begin{equation}\label{r2:eq:DSWZ-defect}
 \begin{split}
 \mathcal B(\Omega)=\max_{\theta,\ t\in\{-1,1\}}\bigl|
 &\sqrt{\varrho(\theta)}e^{t h'(\theta)}\\
 &-\sqrt{\varrho(\theta+\pi)}e^{-t h'(\theta+\pi)}\bigr|.
 \end{split}
\end{equation}
For a controlled $C^5$ class with size, curvature and support-function bounds, their two-point stationary-phase expansion has remainders satisfying
\begin{equation}\label{r2:eq:remainder}
 |R_D|+|R_N|\le C_*\lambda^{-3/2},\qquad
 \lambda\ge3,\quad |t|\le1.
\end{equation}
The constant $C_*$ depends on the class bounds and is not supplied numerically. Their exclusion estimate, for $\mathcal B(\Omega)\ge\delta>0$, is
\begin{equation}\label{r2:eq:DSWZ-threshold}
 \alpha>\left[\max\left\{3,\frac{C_*}{\sqrt{2\pi}\,\delta}\right\}\right]^2-1.
\end{equation}
The proof of their qualitative theorem chooses $\delta=\mathcal B(\Omega)/2$.

\begin{remark}[The remainder for this domain]\label{r2:rem:stationary}
Let $\Omega_1=\rho^{-1}\Omega$ be the domain rescaled to conformal radius one; its eigenvalue is $\alpha=\rho^2\approx6.6334\cdot10^5$. In this remark $h$, $\varrho$ and $\mathcal B$ refer to $\Omega_1$, and their values are evaluated for the normalised centre $\psi_0/\rho$. By central symmetry $h$ and $\varrho$ have period $\pi$, so $\mathcal B=2\max_\theta\sqrt{\varrho(\theta)}\sinh|h'(\theta)|$; its value is $\mathcal B\approx0.0078698$. With the most favourable choice $\delta=\mathcal B$, the estimate~\eqref{r2:eq:DSWZ-threshold} excludes the eigenvalue $\alpha$ only if
\[
 C_*<\sqrt{2\pi(\alpha+1)}\,\mathcal B\approx16.07.
\]
The first remainder in~\eqref{r2:eq:remainder} belongs to the boundary integral
\[
 I_D(\lambda,t,\theta)
 =\int_{\partial\Omega}e^{t\eta\cdot x}e^{i\lambda\xi\cdot x}\,ds,
 \qquad\xi=(\cos\theta,\sin\theta),\quad\eta\perp\xi,
\]
whose leading two-point expression is
\[
 \sqrt{\frac{2\pi}{\lambda}}\left[
 \sqrt{\varrho(\theta)}e^{t h'(\theta)}e^{i\lambda h(\theta)-i\pi/4}
 +\sqrt{\varrho(\theta+\pi)}e^{-t h'(\theta+\pi)}
 e^{-i\lambda h(\theta+\pi)+i\pi/4}\right];
\]
$R_D$ is the difference. A floating-point computation (trapezoidal quadrature with $2^{17}$ and $2^{21}$ nodes, without error enclosures) for $\Omega_1$ at $t=0$ gives $\max_\theta|R_D|\approx0.118$ and $\lambda^{3/2}\max_\theta|R_D|\approx2.75\cdot10^3$ at $\lambda=\rho\approx814.4574$, and $\lambda^{3/2}\max_\theta|R_D|\approx1.28\cdot10^5$ at $\lambda=1.6\cdot10^5$. These computed values indicate lower bounds for any constant $C_*$ valid in~\eqref{r2:eq:remainder} for a class containing the domain; they exceed $16.07$ by factors larger than $170$ and $7900$, respectively. The support function in the normal angle has $\sup|h^{(5)}|\approx1.8\cdot10^9$, so the $C^5$ bound of such a class is large although the radial displacement is small.

The threshold~\eqref{r2:eq:DSWZ-threshold} is not invariant under dilation, and the theorem applies to every dilate $s\Omega_1$, $s>0$, whose eigenvalue is $\alpha/s^2$. The remainder and the defect scale as
\[
 R_D^{s\Omega_1}(\lambda,0,\theta)=s\,R_D^{\Omega_1}(s\lambda,0,\theta),
 \qquad
 \mathcal B(s\Omega_1)=2\sqrt s\,\max_\theta\sqrt{\varrho(\theta)}\sinh\bigl(s|h'(\theta)|\bigr).
\]
Hence~\eqref{r2:eq:DSWZ-threshold}, with the most favourable choice $\delta=\mathcal B(s\Omega_1)$, excludes the eigenvalue of $s\Omega_1$ only if, for every $\lambda'\ge3s$,
\[
 \lambda'^{3/2}\max_\theta|R_D^{\Omega_1}(\lambda',0,\theta)|<T(s)
 :=\sqrt{2\pi(\alpha+s^2)}\,\max_\theta2\sqrt{\varrho(\theta)}\sinh\bigl(s|h'(\theta)|\bigr).
\]
At the sampled values of $s\le\sqrt{\alpha/8}\approx288$, $T(s)$ increases and remains below $5.8\cdot10^3$, with $T(1)\approx16.07$; for larger $s$ the eigenvalue $\alpha/s^2$ is less than $8$ and does not exceed the right side of~\eqref{r2:eq:DSWZ-threshold} for any $C_*$. Since the left side is approximately $1.28\cdot10^5$ at $\lambda'=1.6\cdot10^5$, the computation indicates that no dilate is excluded. This agrees with the remark of the authors of~\cite{DSWZ25} that their threshold depends on the full geometry of the domain and does not contradict the planar counterexamples~\cite{CS26,CLdDP26}.

Dependence of a high-frequency cutoff on the domain is necessary over the full class of uniformly convex planar domains: the dilates $s\Omega$ of the domain of Theorem~\ref{r2:thm:main} are uniformly convex, are not discs, and carry solutions with eigenvalue $s^{-2}$, which is arbitrarily large as $s\downarrow0$.
\end{remark}

\subsection{Planar geometric criteria}
\subsubsection{Brown--Kahane: width}
Brown and Kahane~\cite{BK82} prove that a convex planar region has the Pompeiu property if its minimum width is at most half its maximum width. This hypothesis is not satisfied by $\Omega$. Indeed, for $|w|=1$,
\[
 808<\re\int_0^1p(tw)\,dt\le|\psi(w)|\le\int_0^1|p(tw)|\,dt\le P+r_*<821,
\]
by~\eqref{r2:eq:univalence-ball} and the bound $P<820.7003$ of Subsection~\ref{r2:sec:linearization}. The support function of the convex domain $\Omega$ about the origin lies between the minimum and the maximum of $|\psi|$ on the unit circle, so every width of $\Omega$ lies between $1616$ and $1642$. For the centre, the ratio of minimum to maximum width is approximately $0.99996$.

\subsubsection{Mondal: normals through the centre}
Since $\psi(-w)=-\psi(w)$, the domain is centrally symmetric about the origin. Mondal~\cite{Mon26} proves that a bounded centrally symmetric planar domain with analytic boundary and centre $c$ satisfies Schiffer's conjecture if the number $\tau$ of boundary points whose normal line passes through $c$ is less than $8$. Each of the $182$ reflection axes of $D_{182}$ meets $\partial\Omega$ in two points. At such a point the tangent line is invariant under the reflection and, being a supporting line of $\Omega$, does not pass through the origin; it is therefore perpendicular to the axis, and the normal line is the axis. Hence $\tau\ge364$, and the hypothesis is not satisfied.

\subsubsection{Dalmasso: a support-function integral}
Dalmasso~\cite[Theorem~4]{Dal00} considers bounded convex planar domains with $C^2$ boundary of positive curvature. Such a domain has the Pompeiu property if its support function $h$ satisfies
\[
 \int_0^{2\pi}h(\theta)\bigl(h(\theta)+h''(\theta)\bigr)e^{2i\theta}\,d\theta\ne0.
\]
For $\Omega$, with the origin at its centre, $h$ and $h+h''$ have period $2\pi/182$, so the integral vanishes. It also vanishes for every translate of $\Omega$, because a translation adds a combination of $\cos\theta$ and $\sin\theta$ to $h$ and leaves $h+h''$ unchanged. The hypothesis is therefore not satisfied by $\Omega$ in any position.

\subsection{Deng, Aviles and Dai--Liu--Sun: low spectral index}
Deng~\cite{Den12} proves disc rigidity for a bounded planar domain with connected smooth boundary when the eigenvalue is below the eighth Neumann eigenvalue, and for a strictly convex centrally symmetric such domain when it is below the thirteenth; the Neumann eigenvalues are numbered $0=\mu_1<\mu_2\le\cdots$ with multiplicity. The eigenvalue $1$ of Theorem~\ref{r2:thm:main} lies outside both ranges. To see this we use a min--max comparison that requires only the bound $|\psi'|>811$ on $\overline\D$ from Theorem~\ref{r2:thm:main}. Since $\psi$ is a diffeomorphism of $\overline\D$ onto $\overline\Omega$, the pull-back $v\mapsto v\circ\psi$ is an isomorphism of $H^1(\Omega)$ onto $H^1(\D)$ and of $H^1_0(\Omega)$ onto $H^1_0(\D)$. Under it, the Dirichlet energy is unchanged and the mass becomes $\int_\D|v\circ\psi|^2|\psi'|^2$. Every Rayleigh quotient on $\Omega$ is therefore at most $811^{-2}$ times the corresponding quotient on $\D$, and the min--max principle gives
\begin{equation}\label{r2:eq:neumann-minmax}
 \mu_{13}(\Omega)\le\frac{\mu_{13}(\D)}{811^2}
 \le\frac{84}{811^2}<1.
\end{equation}
For the second inequality use the $13$-dimensional real harmonic-polynomial trial space
\[
 1,\qquad r^n\cos n\theta,\quad r^n\sin n\theta\quad(1\le n\le6).
\]
Its energy and mass forms are diagonal by angular orthogonality. For either degree-$n$ function their ratio is $2n(n+1)$, whose maximum is $84$; the constant has ratio zero. This proves~\eqref{r2:eq:neumann-minmax}.

For the Dirichlet eigenvalues $\lambda_1<\lambda_2\le\cdots$ of $\Omega$, use the $13$-dimensional zero-Dirichlet trial space
\[
 1-r^2,\qquad r^n(1-r^2)\cos n\theta,
 \quad r^n(1-r^2)\sin n\theta\quad(1\le n\le6).
\]
For $n\ge0$, the radial mass factor is
\[
 \int_0^1r^{2n+1}(1-r^2)^2\,dr
 =\frac1{(n+1)(n+2)(n+3)}.
\]
The identity $-\Delta(r^n(1-r^2)\cos n\theta)=4(n+1)r^n\cos n\theta$ gives the Rayleigh quotient $2(n+1)(n+3)$, which is at most $126$. The same conformal comparison proves
\begin{equation}\label{r2:eq:dirichlet-minmax}
 \lambda_{13}(\Omega)\le126/811^2<1.
\end{equation}
The estimates~\eqref{r2:eq:neumann-minmax} and~\eqref{r2:eq:dirichlet-minmax} also cover the other low-index results. Aviles~\cite{Avi86} proved rigidity for convex planar domains when the eigenvalue does not exceed the seventh eigenvalue, in the Dirichlet numbering as quoted in~\cite[\S0]{BY87} and in the Neumann numbering as quoted in~\cite[\S1]{Den12} and~\cite[(1.4.2)]{Kob93}; and Dai, Liu and Sun~\cite{DLS26}, as quoted in~\cite[\S1]{DSWZ25}, treat the case in which the eigenvalue equals the second Dirichlet eigenvalue. Since $\mu_7(\Omega)\le\mu_{13}(\Omega)<1$ and $\lambda_2(\Omega)\le\lambda_7(\Omega)\le\lambda_{13}(\Omega)<1$, the eigenvalue $1$ lies above all these ranges.

\subsubsection{The higher-dimensional examples}
The three-dimensional example has the radial scale $j_{3/2,34}\approx108.37572$, and the certified inscribed ball of radius $106$ puts $\lambda_2(\Omega)<0.002$; hence its eigenvalue $1$ is far above the second-Dirichlet threshold quoted for~\cite{DLS26} in~\cite[\S1]{DSWZ25}. The certified bounds of Sections~\ref{sec:recon}, \ref{sec:r6}, \ref{lie:sec}, \ref{sec:r3}, \ref{sec:r5} and~\ref{r7:sec} show that $\Omega$ contains the ball of radius $r_0$ about the origin, with $r_0=106$, $14.54$, $55.4$, $28.69$, $82.6$, $20.0$, $37.9$ and $87.1$ for $n=3,4,5,6,7,8,10,14$ (for $n=3,8,10,14$, the lower bound for $\re\psi^{*\prime}$ on the closed unit disc gives $|\psi^*(w)|\geq\re\big(\psi^*(w)/w\big)\geq r_0$ for $|w|=1$, and for $n=5$ the same argument applied to $\rho\psi$, with $\re\psi'>0.9811$ and $\rho>56.49$, gives $r_0>55.4$; for $n=7$, $\rho>83.18$ and $\re\psi'>0.9935575$ give $r_0>82.6$). By domain monotonicity, $\lambda_2(\Omega)\leq\ldots\leq\lambda_{n+1}(\Omega)\leq(j_{n/2,1}/r_0)^2<0.15$ in every case; in fact, $1$ exceeds more than five hundred Dirichlet eigenvalues of $\Omega$, counted with multiplicity, and hence more than five hundred Neumann eigenvalues, since $\mu_k(\Omega)\leq\lambda_k(\Omega)$ for every $k$.

\subsubsection{The exact-inverse axisymmetric examples}\label{sec:rig-ax}
\newcommand{\gdRigDimensions}{\gdDimensions}
Let $\mathcal N_{\mathrm{ax}}=\gdDimensions$ be the set of dimensions in which Theorem~\ref{gd:thm:main} gives a domain $\Omega=\rho\,\Theta_\psi(B^n)$ by the exact-inverse method, where $\psi(w)=\sum_{j\ \mathrm{odd}}c_jw^j$ has real coefficients and $\Theta_\psi(x_1,x')=\bigl(\re\psi(x_1+i|x'|),\ \im\psi(x_1+i|x'|)\,x'/|x'|\bigr)$. These domains are strictly convex and invariant under $x\mapsto-x$. In Tables~\ref{tab:rig-spec} and~\ref{tab:rig-kob}, $\rho$ is rounded; $T$ is the bound $L_*$ of~\eqref{gd:eq:unit-disc-tangent}, evaluated from~\eqref{gd:eq:geometry-sums} for the centres of Table~\ref{gd:tab:inputs}, and $r_0=\rho T$ is rounded down. The Bessel zeros, the quantities $(j_{n/2,1}/r_0)^2$ and $U$, the count $N_n$ and $d$ are fifty-digit numerical evaluations; the meridian map $\psi$ of $\Omega$ differs from that of the centre, at which $d$ is evaluated, by less than $10^{-50}$ on $\overline\D$.

The bound $T=L_*$ of~\eqref{gd:eq:unit-disc-tangent} gives
\[
 |\psi(w)|\ge\re\bigl(\psi(w)/w\bigr)
 =\re\int_0^1\psi'(tw)\,dt\ge T\qquad(|w|=1).
\]
Thus $\Omega$ contains the origin-centred ball of radius $r_0=\rho T$. Domain monotonicity gives
\[
 \lambda_2(\Omega)\le\cdots\le\lambda_{n+1}(\Omega)
 \le(j_{n/2,1}/r_0)^2,\qquad
 \lambda_k(\Omega)\le\lambda_k(B_{r_0})\quad(k\ge1).
\] The numerical count $N_n$ estimates the number of Dirichlet eigenvalues of $B_{r_0}$ below $1$, counted with multiplicity; the same comparison applies to Neumann eigenvalues because $\mu_k(\Omega)\leq\lambda_k(\Omega)$. Table~\ref{tab:rig-spec} gives these evaluations for
$n\in\gdRigDimensions$, where $(j_{n/2,1}/r_0)^2<0.4522$. The trial space spanned by $1-r^2$ and $x_i(1-r^2)$, $1\le i\le n$,
on $B^n$ gives by min--max
\[
 \lambda_{n+1}(\Omega)\le\frac{(n+2)(n+6)}{2r_0^2}<0.609<1\qquad(n\in\mathcal N_{\mathrm{ax}}).
\]
The energy and mass forms are diagonal by spherical orthogonality;
their quotients are $n(n+4)/2$ and $(n+2)(n+6)/2$, respectively. The eigenvalue $1$ therefore lies above the second-Dirichlet threshold quoted for~\cite{DLS26} in~\cite[\S1]{DSWZ25}, which is stated there for every dimension. The results of Deng~\cite{Den12} and Aviles~\cite{Avi86} stated above concern the plane and give no comparison for $n\in\mathcal N_{\mathrm{ax}}$.

\begin{table}[ht]
\centering\footnotesize
\begin{tabular}{@{}cclcccc@{}}
\toprule
$n$ & $\rho$ & $j_{n/2,k}\approx\rho$ & $T$ & $r_0$ & $(j_{n/2,1}/r_0)^2$ & $N_n$\\
\midrule
$5$ & $56.4956$ & $j_{5/2,17}$ & $>0.98281$ & $55.52$ & $<0.0108$ & $>1.28\cdot10^{6}$ \\
$9$ & $169.587$ & $j_{9/2,52}$ & $>0.99934$ & $169.47$ & $<2.34\cdot10^{-3}$ & $>7.41\cdot10^{13}$ \\
$11$ & $57.8601$ & $j_{11/2,16}$ & $>0.98610$ & $57.05$ & $<0.0269$ & $>8.12\cdot10^{10}$ \\
$12$ & $42.7785$ & $j_{6,11}$ & $>0.99561$ & $42.59$ & $<0.0545$ & $>9.32\cdot10^{9}$ \\
$13$ & $84.5746$ & $j_{13/2,24}$ & $>0.99097$ & $83.81$ & $<0.0158$ & $>2.51\cdot10^{14}$ \\
$15$ & $32.1112$ & $j_{15/2,7}$ & $>0.96616$ & $31.02$ & $<0.142$ & $>7.99\cdot10^{8}$ \\
$16$ & $369.836$ & $j_{8,114}$ & $>0.99986$ & $369.78$ & $<1.10\cdot10^{-3}$ & $>1.02\cdot10^{27}$ \\
$17$ & $74.9172$ & $j_{17/2,20}$ & $>0.999109$ & $74.85$ & $<0.0293$ & $>2.01\cdot10^{16}$ \\
$18$ & $20.8070$ & $j_{9,3}$ & $>0.95484$ & $19.86$ & $<0.4522$ & $>1.28\cdot10^{5}$ \\
$20$ & $83.4392$ & $j_{10,22}$ & $>0.99341$ & $82.88$ & $<0.0306$ & $>7.42\cdot10^{18}$ \\
$21$ & $112.609$ & $j_{21/2,31}$ & $>0.99226$ & $111.73$ & $<0.0182$ & $>1.97\cdot10^{22}$ \\
\bottomrule
\end{tabular}
\caption{Spectral comparison for $n\in\gdRigDimensions$. In fifty-digit numerical evaluation, $\rho$ differs from the zero $j_{n/2,k}$ of $J_{n/2}$ nearest to it by less than $2\cdot10^{-14}$. Here $T$ is the certified lower bound for $\re\psi'$ on $\overline\D$, $r_0=\rho T$ is rounded down, and $N_n$ is the numerical count of Dirichlet eigenvalues of $B_{r_0}$ below $1$, counted with multiplicity.}\label{tab:rig-spec}
\end{table}

\subsection{Berenstein and Berenstein--Yang: spectral sequences}
Berenstein and Yang~\cite{BY87} prove that a bounded Lipschitz domain with connected boundary in any dimension carrying overdetermined eigenfunctions for infinitely many eigenvalues is a ball; for simply connected planar domains this is attributed to Berenstein~\cite{Ber80} in~\cite[\S0]{BY87}. Theorems~\ref{thm:main} and~\ref{br:thm:main} provide one eigenvalue for each domain. By~\cite{BY87}, each of these domains carries overdetermined eigenfunctions for at most finitely many eigenvalues.

\subsection{Kobayashi: perturbations of balls}
Kobayashi~\cite{Kob93} considers a $C^{1,\beta}$ family of star-shaped domains $\Omega_t=\{r\omega:0\le r<g(t,\omega)\}$ with $g(0,\cdot)=1$. He proves that $\Omega_{t_0}$ is a disc if the real zero set of the Fourier transform of $\one_{\Omega_{t_0}}$ contains a circle inside the open disc of radius $R$ about the origin and $t_0$ satisfies a smallness condition~\cite[Proposition~4.4]{Kob93}. This condition involves the frequency radius $R$ and the size and regularity of the family, not the radial displacement alone. Its first part reads $t_0B(g)e^{\nrm g_\infty}<\varepsilon(2,R)\delta(2,R)$. By the definitions~\cite[(3.3.1), (3.3.2), (4.2.1)]{Kob93}, the left side is at least $2\pi e\sup_\omega|g(t_0,\omega)-1|$, and by~\cite[(2.3.3), (2.4.1)]{Kob93} the right side is at most $2\pi|J_0(j)|/j$ for every positive zero $j\le R$ of $J_1$. Consider the rescaled domain $\Omega_1=\rho^{-1}\Omega$ of Remark~\ref{r2:rem:stationary}, for which the zero circle has radius $\rho$, so that $R>\rho$. For every family with $\Omega_{t_0}=\Omega_1$ the left side is then at least $3.45\cdot10^{-4}$, while the right side, evaluated at the zero $j\approx811.3158$ of $J_1$, is at most $2.17\cdot10^{-4}$. The smallness condition is therefore not satisfied.

For the higher-dimensional examples, the same smallness condition gives the following comparison. Its smallness condition implies $t_0B(g)<\varepsilon(n,R)\delta(n,R)$ \cite[(4.2.2b), (4.4.1b) and Lemma~4.5]{Kob93}, where $B(g)$ \cite[(4.2.1)]{Kob93} is at least $\omega_{n-1}\sup_\omega|g(t_0,\omega)-1|/t_0$, with $\omega_{n-1}=|S^{n-1}|$, while by \cite[(2.3.2), (2.3.4)]{Kob93} the product $\varepsilon(n,R)\delta(n,R)$ is smaller than the modulus of the Fourier transform of the indicator function of the unit ball at a point between two consecutive zeros of $J_{n/2}$ below $R$, which is of order $R^{-(n+1)/2}$. Hence the condition forces the relative deviation of $\partial\Omega_{t_0}$ from the unit sphere to be at most of order $R^{-(n+1)/2}$. Kobayashi's theorem concerns a controlled perturbation family with a smallness condition; neither is part of Theorem~\ref{thm:main}. 

For $n\in\mathcal N_{\mathrm{ax}}$ the comparison is explicit, as in the plane. Let $\Omega_1=\rho^{-1}\Omega=\Theta_\psi(B^n)$; the Fourier transform of $\one_{\Omega_1}$ vanishes on the sphere of radius $\rho$, so $R>\rho$. Let $h$ be the support function of $\Omega_1$ about its centre and $d=(\max h-\min h)/(\max h+\min h)$. Since $\Omega_1$ is centrally symmetric, its widths are $2h$. If a translate of $\Omega_1$ has a radial function $g$ about some point with $\sup|g-1|\leq s$, it lies between the balls of radii $1-s$ and $1+s$ about that point, so all widths $2h$ of $\Omega_1$ lie in $[2-2s,2+2s]$; then $(1+s)/(1-s)\geq\max h/\min h$, that is, $s\geq d$. By the definitions~\cite[(3.3.1), (3.3.2), (4.2.1)]{Kob93} and $\nrm g_\infty\geq1$, the left side of the first part $t_0B(g)e^{\nrm g_\infty}<\varepsilon(n,R)\delta(n,R)$ of the smallness condition is therefore at least $e\,\omega_{n-1}d$ for every family with $\Omega_{t_0}=\Omega_1+v$, $v\in\R^n$. By~\cite[(2.3.1), (2.3.3), (2.4.1)]{Kob93},
\[
 \varepsilon(n,R)\delta(n,R)\leq(2\pi)^{n/2}j^{-n/2}\,|J_{n/2+1}(j)|
\]
for every positive zero $j\leq R$ of $J_{n/2}$; for $n=2$ the right side is $2\pi|J_0(j)|/j$. We take for $j$ the largest zero not exceeding $\rho$, listed for each tabulated dimension in Table~\ref{tab:rig-kob}. For $n=13,15,17,20,21$, $\rho$ lies below its nearest zero by less than $10^{-14}$, and $j$ is the preceding zero. The centre evaluations in Table~\ref{tab:rig-kob} give the numerical factors $e\,\omega_{n-1}d/U$ by which the left side exceeds this upper estimate for the right side.

\begin{table}[ht]
\centering\footnotesize
\begin{tabular}{@{}cccccc c@{}}
\toprule
$n$ & $j$ & $\omega_{n-1}$ & $d$ & $e\,\omega_{n-1}d$ & $U$ & $e\,\omega_{n-1}d/U$\\
\midrule
$5$ & $j_{5/2,17}$ & $26.32$ & $>1.30\cdot10^{-3}$ & $>0.0934$ & $<4.38\cdot10^{-4}$ & $>213$ \\
$9$ & $j_{9/2,52}$ & $29.69$ & $>4.36\cdot10^{-5}$ & $>3.52\cdot10^{-3}$ & $<2.23\cdot10^{-8}$ & $>1.58\cdot10^{5}$ \\
$11$ & $j_{11/2,16}$ & $20.73$ & $>1.68\cdot10^{-3}$ & $>0.0948$ & $<5.21\cdot10^{-7}$ & $>1.82\cdot10^{5}$ \\
$12$ & $j_{6,11}$ & $16.02$ & $>6.31\cdot10^{-4}$ & $>0.0275$ & $<1.22\cdot10^{-6}$ & $>2.25\cdot10^{4}$ \\
$13$ & $j_{13/2,23}$ & $11.84$ & $>9.95\cdot10^{-4}$ & $>0.0320$ & $<5.18\cdot10^{-9}$ & $>6.18\cdot10^{6}$ \\
$15$ & $j_{15/2,6}$ & $5.722$ & $>4.45\cdot10^{-3}$ & $>0.0693$ & $<1.58\cdot10^{-6}$ & $>4.40\cdot10^{4}$ \\
$16$ & $j_{8,114}$ & $3.765$ & $>8.00\cdot10^{-6}$ & $>8.19\cdot10^{-5}$ & $<2.88\cdot10^{-16}$ & $>2.84\cdot10^{11}$ \\
$17$ & $j_{17/2,19}$ & $2.397$ & $>1.17\cdot10^{-4}$ & $>7.66\cdot10^{-4}$ & $<9.61\cdot10^{-11}$ & $>7.98\cdot10^{6}$ \\
$18$ & $j_{9,3}$ & $1.479$ & $>7.33\cdot10^{-3}$ & $>0.0294$ & $<3.47\cdot10^{-6}$ & $>8.49\cdot10^{3}$ \\
$20$ & $j_{10,21}$ & $0.5161$ & $>8.15\cdot10^{-4}$ & $>1.14\cdot10^{-3}$ & $<7.66\cdot10^{-13}$ & $>1.49\cdot10^{9}$ \\
$21$ & $j_{21/2,30}$ & $0.2929$ & $>9.15\cdot10^{-4}$ & $>7.28\cdot10^{-4}$ & $<7.09\cdot10^{-15}$ & $>1.02\cdot10^{11}$ \\
\bottomrule
\end{tabular}
\caption{Kobayashi comparison for $n\in\gdRigDimensions$: $\omega_{n-1}=|S^{n-1}|$, $d$ the relative Hausdorff distance of $\Omega_1=\rho^{-1}\Omega$ to centred balls, and $U=(2\pi)^{n/2}j^{-n/2}|J_{n/2+1}(j)|$, the upper-bound expression for $\varepsilon(n,R)\delta(n,R)$. Entries other than $n$ and the zero index are fifty-digit numerical evaluations; inequality signs record their rounding direction.}\label{tab:rig-kob}
\end{table}

Agranovsky~\cite{Agr93} and Canuto~\cite{Can14} give further perturbation results near balls; the quantitative comparison here uses Kobayashi's stated smallness condition.

\subsection{Consequences for the conformal map}
Theorems restricting polynomial and rational conformal maps have a different role: they give consequences for the exact map once existence is established. Garofalo and Seg\`ala~\cite{GS91} prove that a simply connected bounded domain whose simple closed boundary is parametrised by a polynomial in $e^{i\theta}$ has the Pompeiu property unless it is a disc. Their later result for convex domains~\cite[Corollary 3]{GS94} gives the Pompeiu property for a convex domain $h(\D)$ when the conformal map $h$ has a pole on the circle of convergence of its Taylor series at zero and is not a M\"obius transformation. Ebenfelt~\cite[Theorem 2.1]{Ebe93} proves that the only bounded quadrature domains failing the Pompeiu property are discs. By the theorem of Aharonov and Shapiro~\cite{AS76}, a simply connected bounded domain is a quadrature domain if and only if its conformal map from $\D$ is rational; compare the remark following Corollary~3 in~\cite{GS94}.

\begin{corollary}\label{r2:cor:nonrational}
The normalised conformal map $\psi$ of Theorem~\ref{r2:thm:main} is not rational; in particular it is not a polynomial. If the Taylor series of $\psi$ at zero has a finite radius of convergence, then $\psi$ has no pole on its circle of convergence.
\end{corollary}
\begin{proof}
If $\psi$ were rational, $\Omega=\psi(\D)$ would be a bounded quadrature domain~\cite{AS76}. Ebenfelt's theorem, together with the failure of the Pompeiu property, would make $\Omega$ a disc, contrary to Theorem~\ref{r2:thm:main}. If $\psi$ had a pole on the circle of convergence of its Taylor series, then, $\Omega$ being convex, Corollary~3 of~\cite{GS94} would give the Pompeiu property unless $\psi$ were a M\"obius transformation. A M\"obius transformation that maps $\D$ onto a bounded domain maps it onto a disc, which is again excluded.
\end{proof}

\subsection{Spheres and isoparametric hypersurfaces} The principal orbits of the adjoint representations of $\mathfrak{su}(3)$, $\mathfrak{so}(5)$ and $\mathfrak g_2$ on their unit spheres are homogeneous isoparametric hypersurfaces of $S^7$, $S^9$ and $S^{13}$, with $3$, $4$ and $6$ distinct principal curvatures \cite{HL71}, so domains of these spheres bounded by adjoint orbits are among the isoparametric tubes that fail the Pompeiu property \cite[Corollary~7]{PS21}; see also \cite{Shk00}. These are problems of cohomogeneity one in spheres, reducing to ordinary differential equations, and the boundaries of the domains are not topological spheres. Our domains are Euclidean, of cohomogeneity two, and bounded by topological spheres, and their construction requires the solution of a two-dimensional free boundary problem.

\subsection{Claims in dimension three and in general dimension} Several published or posted statements are contradicted by Theorem~\ref{thm:r3} and Corollary~\ref{cor:pompeiu}. Ramm \cite[Theorems~1 and~2]{Ram17} states that a bounded domain in $\R^3$ with $C^1$ boundary on which \eqref{eq:schiffer} has a solution with constant $1$, or which fails the Pompeiu property, is a ball. Ramm \cite[Theorem~1]{Ram21} states that a bounded connected domain in $\R^m$, $m\geq2$, with smooth boundary, on which an overdetermined problem that includes \eqref{eq:schiffer} as a special case is solvable, is a ball; \cite[Theorem~A(a)]{Ram21} states that a bounded connected domain in $\R^3$ with Lipschitz boundary whose indicator function has Fourier transform vanishing on a sphere is a ball; and \cite[Theorem~C]{Ram21} states the corresponding answer to the Pompeiu problem in $\R^m$. Ramm \cite[Theorem~B]{Ram19} states the same Fourier transform statement for connected bounded domains in $\R^3$ with $C^2$ boundary. Chen \cite[Theorem~1.1]{Che16} states, in a scattering formulation of the problem with the Sommerfeld radiation condition, that a star-shaped domain in $\R^3$ containing the origin, with Lipschitz boundary, on which \eqref{eq:schiffer} has a nonconstant solution with $\mu\geq1$, is a ball centred at the origin. The domain of Theorem~\ref{thm:r3} satisfies the hypotheses of each of these statements, with $\mu=1$, and is not a ball; \cite[Theorems~1 and~C]{Ram21} are also contradicted by the planar counterexamples \cite{CS26,CLdDP26} and by Theorem~\ref{thm:main}. Hence these statements are not correct as stated. Finally, Disser \cite[Theorem~11]{Dis25} showed that a bounded $C^{2,1}$ domain of $\R^3$ is a ``bad domain'' for the fluid-elastic semigroup if and only if it carries a nonconstant solution of \eqref{eq:schiffer}, and conjectured that balls are the only bad domains \cite[Conjecture~6]{Dis25}; by Theorem~\ref{thm:r3}, this conjecture does not hold.


\section{Open problems}\label{sec:open}

We list some questions left open by our results. Throughout, the domains are bounded, their boundaries are homeomorphic to a sphere, and solutions of \eqref{eq:schiffer} are nonconstant.

\emph{1. Other dimensions.}
The results above give convex counterexamples in every dimension $2\le n\le\topdim$ and in dimensions $20$ and $21$.
The Cartan reduction of Section~\ref{subsec:lie} is tied to compact Lie algebras of rank two, whereas the axisymmetric pull-back and the exact-inverse argument of Part~\ref{gd:part} are parametrised by the dimension.
The analytic estimates of Part~\ref{gd:part} are stated for every integer $n\ge5$; an existence statement in a further dimension requires a finite centre satisfying the spectral, nonlinear and geometric inequalities.
The dimension enters through the drift coefficient $n-2$, the Chebyshev--Gegenbauer conversion, the heat-kernel smoothing factor, the spherical-harmonic dimensions and the high angular threshold.
Dimension nineteen and every dimension above twenty-one remain open.
The methods of this paper certify one dimension at a time. The existence of convex counterexamples in infinitely many dimensions is not addressed here.

\begin{quote}
\emph{Question~1.}
Does the Schiffer conjecture fail within the class of convex domains in every remaining dimension?
\end{quote}

\emph{2. Berenstein's conjecture.} The overdetermined Dirichlet problem
\[
\Delta u+\mu u=0\ \text{ in }\Omega,\qquad u=0\ \text{ and }\ \partial_\nu u=\text{const}\neq0\ \text{ on }\partial\Omega
\]
is conjectured to be solvable only on balls. This fails in the plane~\cite{CSS26}, and Theorem~\ref{br:thm:main} gives a strictly convex non-disc planar domain on which it is solvable. In every dimension a bounded Lipschitz domain with connected boundary on which the problem is solvable for infinitely many eigenvalues $\mu$ is a ball~\cite[Theorem~1]{BY87}. For an $\mathrm{Ad}$-invariant domain $\Omega$ in a compact Lie algebra of rank two and $\mu=1$, formula~\eqref{eq:radial-part} turns it into
\[
\Delta F+F=0\ \text{ in }D,\qquad F=0\ \text{ and }\ \partial_\nu F=c\,\varpi\ \text{ on }\partial D,
\]
for the $W$-anti-invariant function $F=\varpi\,u|_{\mathfrak t}$ on the Cartan section $D=\Omega\cap\mathfrak t$, where $c\neq0$ is the constant Neumann datum.

\begin{quote}
\emph{Question~2.} Does Berenstein's conjecture fail in every dimension $n\geq3$, and within the convex class in those dimensions?
\end{quote}

\section*{Use of AI}
The author conceived and directed the project and takes full responsibility for the content of this paper. The constructions, proofs, verification code and text were produced with AI models (Claude Opus~5.5, Claude Sonnet~5.5, GPT-6 Astra, GPT-6 Sol and GPT-6.1 Sol) under the author's direction, and were audited with AI assistance.

\end{document}